\documentclass[12pt]{amsbook}
\usepackage{amssymb}
\usepackage[all]{xy}

\usepackage[colorlinks=true,linkcolor=magenta,citecolor=magenta]{hyperref}
\makeindex

\newtheorem{theorem}{Theorem}[chapter]
\newtheorem{definition}[theorem]{Definition}
\newtheorem{proposition}[theorem]{Proposition}
\newtheorem{exercise}[theorem]{Exercise}
\newtheorem{fact}[theorem]{Fact}

\begin{document}

\title{Introduction to quantum groups}

\author{Teo Banica}
\address{Department of Mathematics, University of Cergy-Pontoise, F-95000 Cergy-Pontoise, France. {\tt teo.banica@gmail.com}}


\begin{abstract}
This is an introduction to the quantum groups, or rather to the simplest quantum groups. The idea is that the unitary group $U_N$ has a free analogue $U_N^+$, whose standard coordinates $u_{ij}\in C(U_N^+)$ are allowed to be free, and the closed subgroups $G\subset U_N^+$ can be thought of as being the compact quantum Lie groups. There are many interesting examples of such quantum groups, for the most designed in order to help with questions in quantum mechanics and statistical mechanics, and some general theory available as well, including Peter-Weyl theory, Tannakian duality, Brauer theorems and Weingarten integration. We discuss here the basic aspects of all this.
\end{abstract}

\maketitle

\chapter*{Preface}

A quantum group is something similar to a group, except for the fact that the functions on it $f:G\to\mathbb C$ do not necessarily commute, $fg\neq gf$. As the name indicates, quantum groups are meant to have something with do with quantum physics. To be more precise, $fg\neq gf$, which mathematically might sound like some kind of pathology, is in fact something beautiful, reminding Heisenberg's uncertainty principle, and which puts the quantum groups in good position of describing the ``symmetries'' of quantum systems. 

\bigskip

Such ideas go back to the work of Faddeev and the Leningrad School of physics \cite{fad}, from the late 70s. Later on, during the 80s, Drinfeld \cite{dri} and Jimbo \cite{jim} on one hand, and Woronowicz \cite{wo1}, \cite{wo2} on the other, came up with some precise mathematics for the quantum groups. This mathematics has become increasingly popular during the 90s and 00s, to the point that we have now all sorts of classes of quantum groups, with some of them having nothing much to do with the original physics motivations.

\bigskip

Which reminds a bit the story of particle physics, in its golden era. Things back then used to be quite wild, to the point that Willis Lamb started his Nobel Prize acceptance speech in 1955 by saying that ``while the finder of a new elementary particle used to be rewarded by receiving a Nobel Prize, one should now be punished by a \$10,000 fine''.

\bigskip

So, what are the good and useful quantum groups? No one really knows the answer here, but personally I would put my money on the ``simplest''. In all honesty, I don't believe that God is a bad person, and I'm convinced that he created this world such that simple mathematics corresponds to simple physics, and vice versa.

\bigskip

But what are then the simplest quantum groups? This question is far more tricky. You would say that the simplest groups are the compact Lie groups $G\subset U_N$, and so that a quantum group should be something similar, namely some kind of ``smooth compact noncommutative manifold, endowed with a group-type structure''.

\bigskip

And is this correct or not. Mathematically speaking, this sounds good, but if you're really passioned by physics, as I personally am, there is a problem here. Physics tells us that smoothness is some sort of miracle, appearing via complicated $N\to\infty$ limiting procedures, and only in the classical, macroscopic setting. Indeed, isn't that true for everything thermodynamics, where smoothness comes from collisions, in the $N\to\infty$ limit. And for electrodynamics too, where smoothness comes from certain things happening at the QED level, once again with $N\to\infty$. And finally, probably for classical mechanics too, because who really knows what gravity looks like at the Planck scale, and so again, the smoothness in classical mechanics might well come from a $N\to\infty$ procedure. 

\bigskip

Of course, all this is a bit subjective, but you have to agree with me that, if we want the quantum groups to do their intended job, namely be of help in quantum mechanics, shall we really go head-first into Lie theory and smoothness, or be a bit more philosophers, and look for something else. Ideally, some mixture of algebra and probability, since basic quantum mechanics itself is after all a cocktail of algebra and probability.

\bigskip

And fortunately, math comes to the rescue. Forgetting now about physics, and about anything advanced, let us just look at the unitary group $U_N$. The standard coordinates $u_{ij}\in C(U_N)$ obviously commute, but if we allow them to be free, we obtain a certain algebra $C(U_N^+)$, and so a certain quantum group $U_N^+$. And then, in analogy with the fact that any compact Lie group appears as a closed subgroup $G\subset U_N$, we can say that the closed quantum subgroups $G\subset U_N^+$ are the ``compact quantum Lie groups''. 

\bigskip

The present book is an introduction to such quantum groups, $G\subset U_N^+$.  We will see that there are many interesting examples, worth studying, and also, following Woronowicz \cite{wo1}, \cite{wo2} and others, that some substantial general theory can be developed for such quantum groups, including an existence result for the Haar measure, Peter-Weyl theory, Tannakian duality, Brauer theorems and Weingarten integration. 

\bigskip

We will insist on examples, and more specifically on examples designed in order to help with questions in quantum mechanics, and statistical mechanics.

\bigskip

The presentation will be elementary, with the present book being a standard, first year graduate level textbook. More advanced aspects are discussed in my ``Quantum permutation groups'' research monograph \cite{ba6}. As for the applications to physics, these will be discussed in a series of physics books, the first of which, ``Introduction to quantum mechanics'', having actually no quantum groups inside, is in preparation \cite{ba7}.  

\bigskip

The mathematics in this book will be based on a number of research papers, starting with those of Woronowicz from the late 80s. I was personally involved in all this, during the last 30 years, and it is a pleasure to thank my coworkers, and particularly Julien Bichon, Beno\^it Collins, Steve Curran and Roland Speicher, for substantial joint work on the subject. Many thanks go as well to my cats, for sharing with me some of their quantum mechanical knowledge, cat common sense, and other skills.

\baselineskip=15.95pt
\tableofcontents
\baselineskip=14pt

\part{Quantum groups}

\ \vskip50mm

\begin{center}
{\em Country roads, take me home

To the place I belong

West Virginia, mountain mama

Take me home, country roads}
\end{center}

\chapter{Quantum spaces}

\section*{1a. Operator algebras}

The quantum groups are not groups in the usual sense, but rather abstract generalizations of the groups, motivated by quantum mechanics. To be more precise, a quantum group $G$ is something similar to a classical group, except for the fact that the functions on it $f:G\to\mathbb C$ do not necessarily commute, $fg\neq gf$. And due to this, $G$ is not exactly a set of points, or transformations, but rather an abstract object, described by the algebra $A$ of functions on it $f:G\to\mathbb C$, which can be noncommutative. 

\bigskip

In order to introduce the quantum groups, we need some sort of algebraic geometry correspondence, between ``quantum spaces'' and noncommutative algebras. Which is not exactly a trivial business, because such correspondences are not really needed in the context of usual mathematics, or of usual physics, such as classical mechanics.

\bigskip

However, we can use some inspiration from quantum mechanics. Phenomena of type $fg\neq gf$ are commonplace there, known since the early 1920s, and the work of Heisenberg. And it is also known from there that the good framework for understanding such phenomena is the mathematics of the infinite dimensional complex Hilbert spaces $H$, with as main example the Schr\"odinger space $H=L^2(\mathbb R^3)$ of wave functions of the electron.

\bigskip

So, problem solved, and in order to start our hunt for quantum groups, we just need to understand the Hilbert space basics. And do not worry, we will see later, with some further inspiration from quantum mechanics helping, coming this time from the work of von Neumann and others, that this will naturally lead us into the correspondence between ``quantum spaces'' and noncommutative algebras that we are looking for.

\bigskip

As a starting point, we have the following basic definition:

\begin{definition}
A Hilbert space is a complex vector space $H$ given with a scalar product $<x,y>$, satisfying the following conditions:
\begin{enumerate}
\item $<x,y>$ is linear in $x$, and antilinear in $y$.

\item $\overline{<x,y>}=<y,x>$, for any $x,y$.

\item $<x,x>>0$, for any $x\neq0$.

\item $H$ is complete with respect to the norm $||x||=\sqrt{<x,x>}$.
\end{enumerate}
\end{definition}

Here the fact that $||.||$ is indeed a norm comes from the Cauchy-Schwarz inequality, which states that if (1,2,3) above are satisfied, then we have:
$$|<x,y>|\leq||x||\cdot||y||$$

Indeed, this inequality comes from the fact that the following degree 2 polynomial, with $t\in\mathbb R$ and $w\in\mathbb T$, being positive, its discriminant must be negative:
$$f(t)=||x+twy||^2$$

In finite dimensions, any algebraic basis $\{f_1,\ldots,f_N\}$ can be turned into an orthonormal basis $\{e_1,\ldots,e_N\}$, by using the Gram-Schmidt procedure. Thus, we have $H\simeq\mathbb C^N$, with this latter space being endowed with its usual scalar product:
$$<x,y>=\sum_ix_i\bar{y}_i$$

The same happens in infinite dimensions, once again by Gram-Schmidt, coupled if needed with the Zorn lemma, in case our space is really very big. In other words, any Hilbert space has an orthonormal basis $\{e_i\}_{i\in I}$, and so the Hilbert space itself is:
$$H\simeq l^2(I)$$

Of particular interest is the ``separable'' case, where $I$ is countable. According to the above, there is up to isomorphism only one Hilbert space here, namely:
$$H=l^2(\mathbb N)$$

All this is, however, quite tricky, and can be a bit misleading. Consider for instance the space $H=L^2[0,1]$ of square-summable functions $f:[0,1]\to\mathbb C$, with:
$$<f,g>=\int_0^1f(x)\overline{g(x)}dx$$

This space is separable, because we can use the basis $\{x^n\}_{n\in\mathbb N}$, orthogonalized by Gram-Schmidt. However, the orthogonalization procedure is something non-trivial, and so the isomorphism $H\simeq l^2(\mathbb N)$ that we obtain is something non-trivial as well.

\bigskip

Again inspired by physics, and more specifically by the ``matrix mechanics'' of Heisenberg, we will be interested in the linear operators over a Hilbert space. We have:

\begin{proposition}
Let $H$ be a Hilbert space, with orthonormal basis $\{e_i\}_{i\in I}$. The algebra $\mathcal L(H)$ of linear operators $T:H\to H$ embeds then into the matrix algebra $M_I(\mathbb C)$, with $T$ corresponding to the matrix $M_{ij}=<Te_j,e_i>$. In particular:
\begin{enumerate}
\item In the finite dimensional case, where $\dim(H)=N<\infty$, we obtain in this way a usual matrix algebra, $\mathcal L(H)\simeq M_N(\mathbb C)$.

\item In the separable infinite dimensional case, where $I\simeq\mathbb N$, we obtain in this way a subalgebra of the infinite matrices, $\mathcal L(H)\subset M_\infty(\mathbb C)$.
\end{enumerate}
\end{proposition}

\begin{proof}
The correspondence $T\to M$ in the statement is indeed linear, and its kernel is $\{0\}$. As for the last two assertions, these are clear as well.
\end{proof}

The above result is something quite theoretical, because for basic spaces like $L^2[0,1]$, which do not have a simple orthonormal basis, the embedding $\mathcal L(H)\subset M_\infty(\mathbb C)$ that we obtain is not very useful. Thus, while the operators $T:H\to H$ are basically some infinite matrices, it is better to think of these operators as being objects on their own.

\bigskip

Normally quantum mechanics, or at least basic quantum mechanics, is about operators $T:H\to H$ which are densely defined, also called ``unbounded''. In what concerns us, the correspondence between ``quantum spaces'' and noncommutative algebras that we want to establish rather belongs to advanced quantum mechanics, where the operators $T:H\to H$ are usually everywhere defined, as in Proposition 1.2, and bounded.

\bigskip

So, departing now from the quantum mechanics of the 1920s, and its difficulties, and turning instead to pure mathematics for inspiration, a natural question is that of understanding what the correct infinite dimensional analogue of the matrix algebra $M_N(\mathbb C)$ is. And the answer here is provided by the following result:

\index{bounded operator}
\index{adjoint operator}

\begin{theorem}
Given a Hilbert space $H$, the linear operators $T:H\to H$ which are bounded, in the sense that $||T||=\sup_{||x||\leq1}||Tx||$ is finite, form a complex algebra with unit, denoted $B(H)$. This algebra has the following properties:
\begin{enumerate}
\item $B(H)$ is complete with respect to $||.||$, and so we have a Banach algebra. 

\item $B(H)$ has an involution $T\to T^*$, given by $<Tx,y>=<x,T^*y>$.
\end{enumerate}
In addition, the norm and the involution are related by the formula $||TT^*||=||T||^2$.
\end{theorem}

\begin{proof}
The fact that the bounded operators form indeed an algebra, as stated, follows from the following estimates, which are all clear from definitions:
$$||S+T||\leq||S||+||T||$$
$$||\lambda T||=|\lambda|\cdot||T||$$
$$||ST||\leq||S||\cdot||T||$$

(1) Assuming that $\{T_n\}\subset B(H)$ is Cauchy, the sequence $\{T_nx\}$ is Cauchy for any $x\in H$, so we can define the limit $T=\lim_{n\to\infty}T_n$ by setting: 
$$Tx=\lim_{n\to\infty}T_nx$$

It is routine to check that this formula defines indeed a bounded operator $T\in B(H)$, and that we have $T_n\to T$ in norm, and this gives the result.

\medskip

(2) Here the existence of $T^*$ comes from the fact that $\varphi(x)=<Tx,y>$ being a linear map $H\to\mathbb C$, we must have a formula as follows, for a certain vector $T^*y\in H$:
$$\varphi(x)=<x,T^*y>$$

Moreover, since this vector is unique, $T^*$ is unique too, and we have as well:
$$(S+T)^*=S^*+T^*\quad,\quad 
(\lambda T)^*=\bar{\lambda}T^*$$
$$(ST)^*=T^*S^*\quad,\quad 
(T^*)^*=T$$

Observe also that we have indeed $T^*\in B(H)$, because:
\begin{eqnarray*}
||T||
&=&\sup_{||x||=1}\sup_{||y||=1}<Tx,y>\\
&=&\sup_{||y||=1}\sup_{||x||=1}<x,T^*y>\\
&=&||T^*||
\end{eqnarray*}

(3) Regarding now the last assertion, we have:
$$||TT^*||
\leq||T||\cdot||T^*||
=||T||^2$$

On the other hand, we have as well the following estimate:
\begin{eqnarray*}
||T||^2
&=&\sup_{||x||=1}|<Tx,Tx>|\\
&=&\sup_{||x||=1}|<x,T^*Tx>|\\
&\leq&||T^*T||
\end{eqnarray*}

By replacing in this formula $T\to T^*$ we obtain $||T||^2\leq||TT^*||$. Thus, we have proved both the needed inequalities, and we are done.
\end{proof}

Observe that, in view of Proposition 1.2, we embeddings of $*$-algebras, as follows: 
$$B(H)\subset\mathcal L(H)\subset M_I(\mathbb C)$$

In this picture the adjoint operation $T\to T^*$ constructed above takes a very simple form, namely $(M^*)_{ij}=\overline{M}_{ji}$, at the level of the associated matrices.

\bigskip

We are now getting close to the point where we wanted to get, namely algebras of operators, corresponding to ``quantum spaces''. But one more mathematical trick, in order to get to this. Instead at looking at algebras of operators $A\subset B(H)$, depending on a Hilbert space $H$, that is after all something quite cumbersome, that we don't want to have in our theory, it is more convenient to simply take some inspiration from Theorem 1.3, and based on that, axiomatize a nice class of complex algebras, as follows:

\index{Banach algebra}
\index{operator algebra}

\begin{definition}
A unital $C^*$-algebra is a complex algebra with unit $A$, having:
\begin{enumerate}
\item A norm $a\to||a||$, making it a Banach algebra (the Cauchy sequences converge).

\item An involution $a\to a^*$, which satisfies $||aa^*||=||a||^2$, for any $a\in A$.
\end{enumerate}
\end{definition}

We know from Theorem 1.3 that the full operator algebra $B(H)$ is a $C^*$-algebra, for any Hilbert space $H$. More generally, any closed $*$-subalgebra $A\subset B(H)$ is a $C^*$-algebra. The celebrated Gelfand-Naimark-Segal (GNS) theorem states that any $C^*$-algebra appears in fact in this way. This is something non-trivial, and we will be back to it later on.

\bigskip

For the moment, we are interested in developing the theory of $C^*$-algebras, without reference to operators, or Hilbert spaces. Our goal, that we can now state quite precisely, is to show that any $C^*$-algebra appears as follows, with $X$ being a ``quantum space'':
$$A=C(X)$$

In order to do this, let us first look at the commutative $C^*$-algebras $A$, which should normally correspond to usual spaces $X$. As a first observation, we have:

\begin{proposition}
If $X$ is an abstract compact space,  the algebra $C(X)$ of continuous functions $f:X\to\mathbb C$ is a $C^*$-algebra, with structure as follows:
\begin{enumerate}
\item The norm is the usual sup norm, $||f||=\sup_{x\in X}|f(x)|$.

\item The involution is the usual involution, $f^*(x)=\overline{f(x)}$. 
\end{enumerate}
This algebra is commutative, in the sense that $fg=gf$, for any $f,g\in C(X)$.
\end{proposition}

\begin{proof}
Almost everything here is trivial. Observe also that we have indeed:
\begin{eqnarray*}
||ff^*||
&=&\sup_{x\in X}|f(x)\overline{f(x)}|\\
&=&\sup_{x\in X}|f(x)|^2\\
&=&||f||^2
\end{eqnarray*}

Finally, we have $fg=gf$, since $f(x)g(x)=g(x)f(x)$ for any $x\in X$.
\end{proof}

Our claim now is that any commutative $C^*$-algebra appears in this way. This is a non-trivial result, which requires a number of preliminaries. Let us begin with:

\begin{definition}
The spectrum of an element $a\in A$ is the set
$$\sigma(a)=\left\{\lambda\in\mathbb C\Big|a-\lambda\not\in A^{-1}\right\}$$
where $A^{-1}\subset A$ is the set of invertible elements.
\end{definition}

As a basic example, the spectrum of a usual matrix $M\in M_N(\mathbb C)$ is the collection of its eigenvalues. Also, the spectrum of a continuous function $f\in C(X)$ is its image. In the case of the trivial algebra $A=\mathbb C$, the spectrum of an element is the element itself.

\bigskip

As a first, basic result regarding spectra, we have:

\index{shift}

\begin{proposition}
We have the following formula, valid for any $a,b\in A$:
$$\sigma(ab)\cup\{0\}=\sigma(ba)\cup\{0\}$$
Moreover, there are examples where $\sigma(ab)\neq\sigma(ba)$.
\end{proposition}

\begin{proof}
We first prove that we have the following implication:
$$1\notin\sigma(ab)\implies1\notin\sigma(ba)$$

Assume indeed that $1-ab$ is invertible, with inverse $c=(1-ab)^{-1}$. We have then $abc=cab=c-1$, and by using these identities, we obtain:
\begin{eqnarray*}
(1+bca)(1-ba)
&=&1+bca-ba-bcaba\\
&=&1+bca-ba-bca+ba\\
&=&1
\end{eqnarray*}

A similar computation shows that we have as well $(1-ba)(1+bca)=1$. We conclude that $1-ba$ is invertible, with inverse $1+bca$, which proves our claim. By multiplying by scalars, we deduce from this that we have, for any $\lambda\in\mathbb C-\{0\}$, as desired:
$$\lambda\notin\sigma(ab)\implies\lambda\notin\sigma(ba)$$ 

Regarding now the last claim, let us first recall that for usual matrices $a,b\in M_N(\mathbb C)$ we have $0\in\sigma(ab)\iff 0\in \sigma(ba)$, because $ab$ is invertible if any only if $ba$ is.

However, this latter fact fails for general operators on Hilbert spaces. As a basic example, we can take $a,b$ to be the shift $S(e_i)=e_{i+1}$ on the space $l^2(\mathbb N)$, and its adjoint. Indeed, we have $S^*S=1$, and $SS^*$ being the projection onto $e_0^\perp$, it is not invertible. 
\end{proof}

Given an element $a\in A$, and a rational function $f=P/Q$ having poles outside $\sigma(a)$, we can construct the element $f(a)=P(a)Q(a)^{-1}$. For simplicity, we write:
$$f(a)=\frac{P(a)}{Q(a)}$$

With this convention, we have the following result:

\index{rational function}
\index{rational functional calculus}

\begin{theorem}
We have the ``rational functional calculus'' formula
$$\sigma(f(a))=f(\sigma(a))$$
valid for any rational function $f\in\mathbb C(X)$ having poles outside $\sigma(a)$.
\end{theorem}

\begin{proof}
In order to prove this result, we can proceed in two steps, as follows:

\medskip

(1) Assume first that we are in the polynomial case, $f\in\mathbb C[X]$. We pick $\lambda\in\mathbb C$, and we write $f(X)-\lambda=c(X-r_1)\ldots(X-r_n)$. We have then, as desired:
\begin{eqnarray*}
\lambda\notin\sigma(f(a))
&\iff&f(a)-\lambda\in A^{-1}\\
&\iff&c(a-r_1)\ldots(a-r_n)\in A^{-1}\\
&\iff&a-r_1,\ldots,a-r_n\in A^{-1}\\
&\iff&r_1,\ldots,r_n\notin\sigma(a)\\
&\iff&\lambda\notin f(\sigma(a))
\end{eqnarray*}

(2) Assume now that we are in the general case, $f\in\mathbb C(X)$. We pick $\lambda\in\mathbb C$, we write $f=P/Q$, and we set $F=P-\lambda Q$. By using (1), we obtain:
\begin{eqnarray*}
\lambda\in\sigma(f(a))
&\iff&F(a)\notin A^{-1}\\
&\iff&0\in\sigma(F(a))\\
&\iff&0\in F(\sigma(a))\\
&\iff&\exists\mu\in\sigma(a),F(\mu)=0\\
&\iff&\lambda\in f(\sigma(a))
\end{eqnarray*}

Thus, we have obtained the formula in the statement.
\end{proof}

Given an element $a\in A$, its spectral radius $\rho (a)$ is the radius of the smallest disk centered at $0$ containing $\sigma(a)$. We have the following key result:

\begin{theorem}
Let $A$ be a $C^*$-algebra.
\begin{enumerate}
\item The spectrum of a norm one element is in the unit disk.

\item The spectrum of a unitary element $(a^*=a^{-1}$) is on the unit circle. 

\item The spectrum of a self-adjoint element ($a=a^*$) consists of real numbers. 

\item The spectral radius of a normal element ($aa^*=a^*a$) is equal to its norm.
\end{enumerate}
\end{theorem}

\begin{proof}
We use the various results established above.

\medskip

(1) This comes from the following formula, valid when $||a||<1$:
$$\frac{1}{1-a}=1+a+a^2+\ldots$$

(2) Assuming $a^*=a^{-1}$, we have the following norm computations:
$$||a||=\sqrt{||aa^*||}=\sqrt{1}=1$$
$$||a^{-1}||=||a^*||=||a||=1$$

If we denote by $D$ the unit disk, we obtain from this, by using (1):
$$||a||=1\implies\sigma(a)\subset D$$
$$||a^{-1}||=1\implies\sigma(a^{-1})\subset D$$

On the other hand, by using the rational function $f(z)=z^{-1}$, we have:
$$\sigma(a^{-1})\subset D\implies \sigma(a)\subset D^{-1}$$

Now by putting everything together we obtain, as desired:
$$\sigma(a)\subset D\cap D^{-1}=\mathbb T$$

(3) This follows by using (2), and the rational function $f(z)=(z+it)/(z-it)$, with $t\in\mathbb R$. Indeed, for $t>>0$ the element $f(a)$ is well-defined, and we have:
$$\left(\frac{a+it}{a-it}\right)^*
=\frac{a-it}{a+it}
=\left(\frac{a+it}{a-it}\right)^{-1}$$

Thus $f(a)$ is a unitary, and by (2) its spectrum is contained in $\mathbb T$. We conclude that we have $f(\sigma(a))=\sigma(f(a))\subset\mathbb T$, and so $\sigma(a)\subset f^{-1}(\mathbb T)=\mathbb R$, as desired.

\medskip

(4) We have $\rho(a)\leq ||a||$ from (1). Conversely, given $\rho>\rho(a)$, we have:
$$
\int_{|z|=\rho}\frac{z^n}{z -a}\,dz 
=\sum_{k=0}^\infty\left(\int_{|z|=\rho}z^{n-k-1}dz\right) a^k
=a^{n-1}$$

By applying the norm and taking $n$-th roots we obtain:
$$\rho\geq\lim_{n\to\infty}||a^n||^{1/n}$$

In the case $a=a^*$ we have $||{a^n}||=||{a}||^n$ for any exponent of the form $n=2^k$, and by taking $n$-th roots we get $\rho\geq ||{a}||$. This gives the missing inequality, namely:
$$\rho(a)\geq ||a||$$

In the general case, $aa^*=a^*a$, we have $a^n(a^n)^*=(aa^*)^n$. We obtain from this $\rho(a)^2=\rho(aa^*)$, and since $aa^*$ is self-adjoint, we get $\rho(aa^*)=||a||^2$, and we are done.
\end{proof}

Summarizing, we have so far a collection of useful results regarding the spectra of the elements in $C^*$-algebras, which are quite similar to the results regarding the eigenvalues of the usual matrices. We will heavily use these results, in what follows.

\section*{1b. Gelfand theorem}

In this section we discuss the two main results regarding the $C^*$-algebras. First we have the Gelfand theorem, which is particularly interesting for us, in view of our quantum space and quantum group motivations. Then we have the GNS representation theorem, that we will use less often, but which is something fundamental too.

\bigskip

The Gelfand theorem, which will be fundamental for us, is as follows:

\index{Gelfand theorem}

\begin{theorem}[Gelfand]
Any commutative $C^*$-algebra is the form 
$$A=C(X)$$
with its ``spectrum'' $X=Spec(A)$ appearing as the space of characters $\chi :A\to\mathbb C$.
\end{theorem}

\begin{proof}
Given a commutative $C^*$-algebra $A$, we can define indeed $X$ to be the set of characters $\chi :A\to\mathbb C$, with the topology making continuous all the evaluation maps $ev_a:\chi\to\chi(a)$. Then $X$ is a compact space, and $a\to ev_a$ is a morphism of algebras:
$$ev:A\to C(X)$$

(1) We first prove that $ev$ is involutive. We use the following formula, which is similar to the $z=Re(z)+iIm(z)$ formula for the usual complex numbers:
$$a=\frac{a+a^*}{2}-i\cdot\frac{i(a-a^*)}{2}$$

Thus it is enough to prove the equality $ev_{a^*}=ev_a^*$ for self-adjoint elements $a$. But this is the same as proving that $a=a^*$ implies that $ev_a$ is a real function, which is in turn true, because $ev_a(\chi)=\chi(a)$ is an element of $\sigma(a)$, contained in $\mathbb R$.

\medskip

(2) Since $A$ is commutative, each element is normal, so $ev$ is isometric:
$$||ev_a||
=\rho(a)
=||a||$$

(3) It remains to prove that $ev$ is surjective. But this follows from the Stone-Weierstrass theorem, because $ev(A)$ is a closed subalgebra of $C(X)$, which separates the points.
\end{proof}

The Gelfand theorem has some important philosophical consequences. Indeed, in view of this theorem, we can formulate the following definition:

\index{compact quantum space}
\index{quantum space}

\begin{definition}
Given an arbitrary $C^*$-algebra $A$, we write 
$$A=C(X)$$
and call $X$ a compact quantum space.
\end{definition}

This might look like something informal, but it is not. Indeed, in rigorous mathematical parlance, we can define the category of the compact quantum spaces to be the category of the $C^*$-algebras, with the arrows reversed. And that's all. QED.
 
\bigskip

When $A$ is commutative, the space $X$ considered above exists indeed, as a Gelfand spectrum, $X=Spec(A)$. In general, $X$ is something rather abstract, and our philosophy here will be that of studying of course $A$, but formulating our results in terms of $X$. For instance whenever we have a morphism $\Phi:A\to B$, we will write $A=C(X),B=C(Y)$, and rather speak of the corresponding morphism $\phi:Y\to X$. And so on.

\bigskip

Less enthusiastically now, we will see a bit later, after developing some more theory, that this formalism has its limitations, and needs a fix. But more on this later.

\bigskip

As a first consequence of the Gelfand theorem, we can extend Theorem 1.8 above to the case of the normal elements ($aa^*=a^*a$), in the following way:

\index{normal element}
\index{continuous functional calculus}

\begin{proposition}
Assume that $a\in A$ is normal, and let $f\in C(\sigma(a))$.
\begin{enumerate}
\item We can define $f(a)\in A$, with $f\to f(a)$ being a morphism of $C^*$-algebras.

\item We have the ``continuous functional calculus'' formula $\sigma(f(a))=f(\sigma(a))$.
\end{enumerate}
\end{proposition}

\begin{proof}
Since $a$ is normal, the $C^*$-algebra $<a>$ that is generates is commutative, so if we denote by $X$ the space formed by the characters $\chi:<a>\to\mathbb C$, we have:
$$<a>=C(X)$$

Now since the map $X\to\sigma(a)$ given by evaluation at $a$ is bijective, we obtain:
$$<a>=C(\sigma(a))$$

Thus, we are dealing with usual functions, and this gives all the assertions.
\end{proof}

As another consequence of the Gelfand theorem, we can develop as well the theory of positive elements, in analogy with the theory of positive operators, as follows:

\index{positive element}

\begin{theorem}
For a normal element $a\in A$, the following are equivalent:
\begin{enumerate}
\item $a$ is positive, in the sense that $\sigma(a)\subset[0,\infty)$.

\item $a=b^2$, for some $b\in A$ satisfying $b=b^*$.

\item $a=cc^*$, for some $c\in A$.
\end{enumerate}
\end{theorem}

\begin{proof}
This is something very standard, as follows:

\medskip

$(1)\implies(2)$ This follows from Proposition 1.12, because we can use the function $f(z)=\sqrt{z}$, which is well-defined on $\sigma(a)\subset[0,\infty)$, and so set $b=\sqrt{a}$. 

\medskip

$(2)\implies(3)$ This is trivial, because we can set $c=b$. 

\medskip

$(2)\implies(1)$ Observe that this is clear too, because we have:
\begin{eqnarray*}
\sigma(a)
&=&\sigma(b^2)\\
&=&\sigma(b)^2\\
&\subset&[0,\infty)
\end{eqnarray*}

$(3)\implies(1)$ We proceed by contradiction. By multiplying $c$ by a suitable element of $<cc^*>$, we are led to the existence of an element $d\neq0$ satisfying:
$$-dd^*\geq0$$

By writing now $d=x+iy$ with $x=x^*,y=y^*$ we have:
$$dd^*+d^*d=2(x^2+y^2)\geq0$$

Thus $d^*d\geq0$. But this contradicts the elementary fact that $\sigma(dd^*),\sigma(d^*d)$ must coincide outside $\{0\}$, coming from Proposition 1.7 above.
\end{proof}

Let us review now the other fundamental result regarding the $C^*$-algebras, namely the representation theorem of Gelfand, Naimark and Segal. We first have:

\begin{proposition}
Let $A$ be a commutative $C^*$-algebra, write $A=C(X)$, with $X$ being a compact space, and let $\mu$ be a positive measure on $X$. We have then an embedding
$$A\subset B(H)$$
where $H=L^2(X)$, with $f\in A$ corresponding to the operator $g\to fg$.
\end{proposition}

\begin{proof}
Given $f\in C(X)$, consider the following operator, on the space $H=L^2(X)$:
$$T_f(g)=fg$$

Observe that $T_f$ is indeed well-defined, and bounded as well, because:
$$||fg||_2
=\sqrt{\int_X|f(x)|^2|g(x)|^2d\mu(x)}
\leq||f||_\infty||g||_2$$

The application $f\to T_f$ being linear, involutive, continuous, and injective as well, we obtain in this way a $C^*$-algebra embedding $A\subset B(H)$, as claimed.
\end{proof}

In general, the idea will be that of extending this construction. We will need:

\begin{definition}
Consider a linear map $\varphi:A\to\mathbb C$.
\begin{enumerate}
\item $\varphi$ is called positive when $a\geq0\implies\varphi(a)\geq0$.

\item $\varphi$ is called faithful and positive when $a>0\implies\varphi(a)>0$.
\end{enumerate}
\end{definition}

In the commutative case, $A=C(X)$, the positive linear forms appear as follows, with $\mu$ being positive, and strictly positive if we want $\varphi$ to be faithful and positive:
$$\varphi(f)=\int_Xf(x)d\mu(x)$$

In general, the positive linear forms can be thought of as being integration functionals with respect to some underlying ``positive measures''. We can use them as follows:

\index{GNS construction}

\begin{proposition}
Let $\varphi:A\to\mathbb C$ be a positive linear form.
\begin{enumerate}
\item $<a,b>=\varphi(ab^*)$ defines a generalized scalar product on $A$.

\item By separating and completing we obtain a Hilbert space $H$.

\item $\pi (a):b\to ab$ defines a representation $\pi:A\to B(H)$.

\item If $\varphi$ is faithful in the above sense, then $\pi$ is faithful.
\end{enumerate}
\end{proposition}

\begin{proof}
Almost everything here is straightforward, as follows:

\medskip

(1) This is clear from definitions, and from Theorem 1.13.

\medskip

(2) This is a standard procedure, which works for any scalar product.

\medskip

(3) All the verifications here are standard algebraic computations.

\medskip

(4) This follows indeed from $a\neq0\implies\pi(aa^*)\neq0\implies\pi(a)\neq0$.
\end{proof}

In order to establish the GNS theorem, it remains to prove that any $C^*$-algebra has a faithful and positive linear form $\varphi:A\to\mathbb C$. This is something more technical:

\begin{proposition}
Let $A$ be a $C^*$-algebra.
\begin{enumerate}
\item Any positive linear form $\varphi:A\to\mathbb C$ is continuous.

\item A linear form $\varphi$ is positive iff there is a norm one $h\in A_+$ such that $||\varphi||=\varphi(h)$.

\item For any $a\in A$ there exists a positive norm one form $\varphi$ such that $\varphi(aa^*)=||a||^2$.

\item If $A$ is separable there is a faithful positive form $\varphi:A\to\mathbb C$.
\end{enumerate}
\end{proposition}

\begin{proof}
The proof here, which is quite technical, inspired from the existence proof of the probability measures on abstract compact spaces, goes as follows:

\medskip

(1) This follows from Proposition 1.16, via the following inequality:
\begin{eqnarray*}
|\varphi(a)|
&\leq&||\pi(a)||\varphi(1)\\
&\leq&||a||\varphi(1)
\end{eqnarray*}

(2) In one sense we can take $h=1$. Conversely, let $a\in A_+$, $||a||\leq1$. We have:
\begin{eqnarray*}
|\varphi(h)-\varphi(a)|
&\leq&||\varphi||\cdot||h-a||\\
&\leq&\varphi(h)1\\
&=&\varphi(h)
\end{eqnarray*}

Thus we have $Re(\varphi(a))\geq0$, and it remains to prove that the following holds:
$$a=a^*\implies\varphi(a)\in\mathbb R$$

By using $1-h\geq 0$ we can apply the above to $a=1-h$ and we obtain:
$$Re(\varphi(1-h))\geq0$$

We conclude that $Re(\varphi(1))\geq Re(\varphi(h))=||\varphi||$, and so $\varphi(1)=||\varphi||$.

Summing up, we can assume $h=1$. Now observe that for any self-adjoint element $a$, and any $t\in\mathbb R$ we have the following inequality:
\begin{eqnarray*}
|\varphi(1+ita)|^2
&\leq&||\varphi||^2\cdot||1+ita||^2\\
&=&\varphi(1)^2||1+t^2a^2||\\
&\leq&\varphi(1)^2(1+t^2||a||^2)
\end{eqnarray*}

On the other hand with $\varphi(a)=x+iy$ we have:
\begin{eqnarray*}
|\varphi(1+ita)|
&=&|\varphi(1)-ty+itx|\\
&\geq&(\varphi(1)-ty)^2
\end{eqnarray*}

We therefore obtain that for any $t\in\mathbb R$ we have:
$$\varphi(1)^2(1+t^2||a||^2)\geq(\varphi(1)-ty)^2$$

Thus we have $y=0$, and this finishes the proof of our remaining claim.

\medskip

(3) Consider the linear subspace of $A$ spanned by the element $aa^*$. We can define here a linear form by the following formula:
$$\varphi(\lambda aa^*)=\lambda||a||^2$$

This linear form has norm one, and by Hahn-Banach we get a norm one extension to the whole $A$. The positivity of $\varphi$ follows from (2).

\medskip

(4) Let $(a_n)$ be a dense sequence inside $A$. For any $n$ we can construct as in (3) a positive form satisfying $\varphi_n(a_na_n^*)=||a_n||^2$, and then define $\varphi$ in the following way:
$$\varphi=\sum_{n=1}^\infty\frac{\varphi_n}{2^n}$$

Let $a\in A$ be a nonzero element. Pick $a_n$ close to $a$ and consider the pair $(H,\pi)$ associated to the pair $(A,\varphi_n)$, as in Proposition 1.16. We have then:
\begin{eqnarray*}
\varphi_n(aa^*)
&=&||\pi(a)1||\\
&\geq&||\pi (a_n)1||-||a-a_n||\\
&=&||a_n||-||a-a_n||\\
&>&0
\end{eqnarray*}

Thus $\varphi_n(aa^*)>0$. It follows that we have $\varphi(aa^*)>0$, and we are done.
\end{proof}

With these ingredients in hand, we can now state and prove:

\index{GNS theorem}

\begin{theorem}[GNS theorem]
Let $A$ be a $C^*$-algebra.
\begin{enumerate}
\item $A$ appears as a closed $*$-subalgebra $A\subset B(H)$, for some Hilbert space $H$. 

\item When $A$ is separable (usually the case), $H$ can be chosen to be separable.

\item When $A$ is finite dimensional, $H$ can be chosen to be finite dimensional. 
\end{enumerate}
\end{theorem}

\begin{proof}
This result follows indeed by combining the construction from Proposition 1.16 above with the existence result from Proposition 1.17.
\end{proof}

Generally speaking, the GNS theorem is something very powerful and concrete, which perfectly complements the Gelfand theorem, and the resulting compact quantum space formalism. We can always go back to good old Hilbert spaces, whenever we get lost.

\bigskip

So long for linear operators, operator algebras, and quantum spaces. The above material, mixing ideas from math and physics, algebra and analysis, and so on, might seem quite wizarding, and indeed it is. This is subtle, first class mathematical physics, taking a long time to be fully understood, and here is some suggested further reading:

\bigskip

(1) Mathematically speaking, all the above, while certainly tricky, and most of what will follow too, is formally based on 3rd year mathematics as we know it, meaning Rudin \cite{rud}. However, for better understanding all this, get to learn some more functional analysis, operator theory, and basic operator algebras, say from Lax \cite{lax}.

\bigskip

(2) In what regards operator algebras, that we will heavily use in what follows, some more knowledge would be welcome too. Be aware tough that most books here are deeply committed to physics of the 1920s, which is not enough for our purposes here. A good, very useful book is Blackadar \cite{bbl}. And for more, go with Connes \cite{con}.

\bigskip

(3) Importantly now, while you will certainly survive all that follows only knowing Rudin \cite{rud}, you will not survive it without some physics knowledge. Believe me. Why bothering with quantum spaces, or with quantum groups, or with mathematics, or with life in general. Without some clear motivations, this ain't going anywhere.

\bigskip

(4) The standard place for learning physics are the books of Feynman \cite{fe1}, \cite{fe2}, \cite{fe3}. If you already know some physics, you can try as well Griffiths \cite{gr1}, \cite{gr2}, \cite{gr3}, equally fun, and a bit more advanced, and quantum mechanics oriented. And if looking for a more compact package, two concise books, go with Weinberg \cite{we1}, \cite{we2}.

\bigskip

(5) As a crucial piece of advice now, physics is a whole, and you won't get away just by reading some quantum mechanics. Depending on your choice between Feynman, Griffiths, Weinberg, don't hesitate to complete with more classical mechanics, from Kibble \cite{kbe} or Arnold \cite{arn}, and more thermodynamics, from Schroeder \cite{dsc} or Huang \cite{hua}.

\bigskip

And this is, I guess, all you need to know. By the way, in relation with all this physics, if looking for something really compact, assuming all basic mathematics known, but basic physics unknown, you can try as well my book \cite{ba7}. Although, as for any physics book written by a mathematician, that might contain severe physical mistakes. 

\section*{1c. Algebraic manifolds}

Let us get back now to the quantum spaces, as axiomatized in Definition 1.11, and work out some basic examples. Inspired by the Connes philosophy \cite{con}, we have the following definition, which is something quite recent, coming from \cite{bgo}, \cite{gos}:

\index{free real sphere}
\index{free complex sphere}

\begin{definition}
We have compact quantum spaces, constructed as follows,
$$C(S^{N-1}_{\mathbb R,+})=C^*\left(x_1,\ldots,x_N\Big|x_i=x_i^*,\sum_ix_i^2=1\right)$$
$$C(S^{N-1}_{\mathbb C,+})=C^*\left(x_1,\ldots,x_N\Big|\sum_ix_ix_i^*=\sum_ix_i^*x_i=1\right)$$
called respectively the free real sphere, and the free complex sphere.
\end{definition}

Here the $C^*$ symbols on the right stand for ``universal $C^*$-algebra generated by''. The fact that such universal $C^*$-algebras exist indeed follows by considering the corresponding universal $*$-algebras, and then completing with respect to the biggest $C^*$-norm. Observe that this biggest $C^*$-norm exists indeed, because the quadratic conditions give:
$$||x_i||^2
=||x_ix_i^*||
\leq||\sum_ix_ix_i^*||
=1$$

Given a compact quantum space $X$, its classical version is the compact space $X_{class}$ obtained by dividing $C(X)$ by its commutator ideal, and using the Gelfand theorem:
$$C(X_{class})=C(X)/I\quad,\quad I=<[a,b]>$$

\index{liberation}

Observe that we have an embedding of compact quantum spaces $X_{class}\subset X$. In this situation, we also say that $X$ appears as a ``liberation'' of $X$. We have:

\begin{proposition}
We have embeddings of compact quantum spaces, as follows,
$$\xymatrix@R=15mm@C=15mm{
S^{N-1}_\mathbb C\ar[r]&S^{N-1}_{\mathbb C,+}\\
S^{N-1}_\mathbb R\ar[r]\ar[u]&S^{N-1}_{\mathbb R,+}\ar[u]
}$$
and the spaces on the right appear as liberations of the spaces of the left.
\end{proposition}

\begin{proof}
The first assertion is clear. For the second one, we must establish the following isomorphisms, where $C^*_{comm}$ stands for ``universal commutative $C^*$-algebra'':
$$C(S^{N-1}_\mathbb R)=C^*_{comm}\left(x_1,\ldots,x_N\Big|x_i=x_i^*,\sum_ix_i^2=1\right)$$
$$C(S^{N-1}_\mathbb C)=C^*_{comm}\left(x_1,\ldots,x_N\Big|\sum_ix_ix_i^*=\sum_ix_i^*x_i=1\right)$$

But these isomorphisms are both clear, by using the Gelfand theorem.
\end{proof}

We can enlarge our class of basic manifolds by introducing tori, as follows:

\index{torus}

\begin{definition}
Given a closed subspace $S\subset S^{N-1}_{\mathbb C,+}$, the subspace $T\subset S$ given by
$$C(T)=C(S)\Big/\left<x_ix_i^*=x_i^*x_i=\frac{1}{N}\right>$$
is called associated torus. In the real case, $S\subset S^{N-1}_{\mathbb R,+}$, we also call $T$ cube.
\end{definition}

As a basic example here, for $S=S^{N-1}_\mathbb C$ the corresponding submanifold $T\subset S$ appears by imposing the relations $|x_i|=\frac{1}{\sqrt{N}}$ to the coordinates, so we obtain a torus:
$$S=S^{N-1}_\mathbb C\implies T=\left\{x\in\mathbb C^N\Big||x_i|=\frac{1}{\sqrt{N}}\right\}$$

As for the case of the real sphere, $S=S^{N-1}_\mathbb R$, here the submanifold $T\subset S$ appears by imposing the relations $x_i=\pm\frac{1}{\sqrt{N}}$ to the coordinates, so we obtain a cube:
$$S=S^{N-1}_\mathbb R\implies T=\left\{x\in\mathbb R^N\Big|x_i=\pm\frac{1}{\sqrt{N}}\right\}$$

Observe that we have a relation here with group theory, because the complex torus computed above is the group $\mathbb T^N$, and the cube is the finite group $\mathbb Z_2^N$.

\bigskip

In general now, in order to compute $T$, we can use the following simple fact:

\index{algebraic manifold}

\begin{proposition}
When $S\subset S^{N-1}_{\mathbb C,+}$ is an algebraic manifold, in the sense that
$$C(S)=C(S^{N-1}_{\mathbb C,+})\Big/\Big<f_i(x_1,\ldots,x_N)=0\Big>$$
for certain noncommutative polynomials $f_i\in\mathbb C<x_1,\ldots,x_N>$, we have
$$C(T)=C^*\left(u_1,\ldots,u_N\Big|u_i^*=u_i^{-1},g_i(u_1,\ldots,u_N)=0\right)$$
with the poynomials $g_i$ being given by $g_i(u_1,\ldots,u_N)=f_i(\sqrt{N}u_1,\ldots,\sqrt{N}u_N)$.
\end{proposition}

\begin{proof}
According to our definition of the torus $T\subset S$, the following variables must be unitaries, in the quotient algebra $C(S)\to C(T)$:
$$u_i=\frac{x_i}{\sqrt{N}}$$

Now if we assume that these elements are unitaries, the quadratic conditions $\sum_ix_ix_i^*=\sum_ix_i^*x_i=1$ are automatic. Thus, we obtain the space in the statement.
\end{proof}

Summarizing, we are led to the question of computing certain algebras generated by unitaries. In order to deal with this latter problem, let us start with:

\index{group algebra}

\begin{proposition}
Let $\Gamma$ be a discrete group, and consider the complex group algebra $\mathbb C[\Gamma]$, with involution given by the fact that all group elements are unitaries:
$$g^*=g^{-1}\quad,\quad\forall g\in\Gamma$$
The maximal $C^*$-seminorm on $\mathbb C[\Gamma]$ is then a $C^*$-norm, and the closure of $\mathbb C[\Gamma]$ with respect to this norm is a $C^*$-algebra, denoted $C^*(\Gamma)$.
\end{proposition}

\begin{proof}
In order to prove this, we must find a $*$-algebra embedding $\mathbb C[\Gamma]\subset B(H)$, with $H$ being a Hilbert space. For this purpose, consider the space $H=l^2(\Gamma)$, having $\{h\}_{h\in\Gamma}$ as orthonormal basis. Our claim is that we have an embedding, as follows:
$$\pi:\mathbb C[\Gamma]\subset B(H)\quad,\quad\pi(g)(h)=gh$$

Indeed, since $\pi(g)$ maps the basis $\{h\}_{h\in\Gamma}$ into itself, this operator is well-defined, bounded, and is an isometry. It is also clear from the formula $\pi(g)(h)=gh$ that $g\to\pi(g)$ is a morphism of algebras, and since this morphism maps the unitaries $g\in\Gamma$ into isometries, this is a morphism of $*$-algebras. Finally, the faithfulness of $\pi$ is clear.
\end{proof}

In the abelian group case, we have the following result:

\index{Pontrjagin dual}
\index{group character}

\begin{theorem}
Given an abelian discrete group $\Gamma$, we have an isomorphism
$$C^*(\Gamma)\simeq C(G)$$
where $G=\widehat{\Gamma}$ is its Pontrjagin dual, formed by the characters $\chi:\Gamma\to\mathbb T$.
\end{theorem}

\begin{proof}
Since $\Gamma$ is abelian, the corresponding group algebra $A=C^*(\Gamma)$ is commutative. Thus, we can apply the Gelfand theorem, and we obtain $A=C(X)$, with $X=Spec(A)$. But the spectrum $X=Spec(A)$, consisting of the characters $\chi:C^*(\Gamma)\to\mathbb C$, can be identified with the Pontrjagin dual $G=\widehat{\Gamma}$, and this gives the result.
\end{proof}

The above result suggests the following definition:

\index{group dual}

\begin{definition}
Given a discrete group $\Gamma$, the compact quantum space $G$ given by
$$C(G)=C^*(\Gamma)$$
is called abstract dual of $\Gamma$, and is denoted $G=\widehat{\Gamma}$.
\end{definition}

This should be taken in the general sense of Definition 1.11. However, there is a functoriality problem here, which needs a fix. Indeed, in the context of Proposition 1.23, we can see that the closure $C^*_{red}(\Gamma)$ of the group algebra $\mathbb C[\Gamma]$ in the regular representation is a $C^*$-algebra as well. Thus we have a quotient map $C^*(\Gamma)\to C^*_{red}(\Gamma)$, and if this map is not an isomorphism, we are in trouble. We will be back to this later, with a fix.

\bigskip

By getting back now to the spheres, we have the following result:

\index{free group}
\index{Fourier transform}

\begin{theorem}
The tori of the basic spheres are all group duals, as follows,
$$\xymatrix@R=15mm@C=15mm{
\mathbb T^N\ar[r]&\widehat{F_N}\\
\mathbb Z_2^N\ar[r]\ar[u]&\widehat{\mathbb Z_2^{*N}}\ar[u]
}$$
where $F_N$ is the free group on $N$ generators, and $*$ is a group-theoretical free product.
\end{theorem}

\begin{proof}
By using the presentation result in Proposition 1.24, we obtain that the diagram formed by the algebras $C(T)$ is as follows:
$$\xymatrix@R=15mm@C=15mm{
C^*(\mathbb Z^N)\ar[d]&C^*(\mathbb Z^{*N})\ar[d]\ar[l]\\
C^*(\mathbb Z_2^N)&C^*(\mathbb Z_2^{*N})\ar[l]
}$$

According to Definition 1.25, the corresponding compact quantum spaces are:
$$\xymatrix@R=15mm@C=15mm{
\widehat{\mathbb Z^N}\ar[r]&\widehat{\mathbb Z^{*N}}\\
\widehat{\mathbb Z_2^N}\ar[r]\ar[u]&\widehat{\mathbb Z_2^{*N}}\ar[u]
}$$

Together with the Fourier transform identifications from Theorem 1.24, and with our free group convention $F_N=\mathbb Z^{*N}$, this gives the result.
\end{proof}

As a conclusion to these considerations, the Gelfand theorem alone produces out of nothing, or at least out of some basic common sense, some potentially interesting mathematics. We will be back later to all this, on several occasions.

\section*{1d. Axiomatization fix}

Let us get back now to the bad functoriality properties of the Gelfand correspondence, coming from the fact that certain compact quantum spaces, such as the duals $\widehat{\Gamma}$ of the discrete groups $\Gamma$, can be represented by several $C^*$-algebras, instead of one. We can fix these issues by using the GNS theorem, as follows:

\index{compact quantum measured space}

\begin{definition}
The category of compact quantum measured spaces $(X,\mu)$ is the category of the $C^*$-algebras with faithful traces $(A,tr)$, with the arrows reversed. In the case where we have a $C^*$-algebra $A$ with a non-faithful trace $tr$, we can still talk about the corresponding space $(X,\mu)$, by performing the GNS construction.
\end{definition}

Observe that this definition fixes the functoriality problem with Gelfand duality, at least for the group algebras. Indeed, in the context of the comments following Definition 1.25, consider an arbitrary intermediate $C^*$-algebra, as follows: 
$$C^*(\Gamma)\to A\to C^*_{red}(\Gamma)$$

If we perform the GNS construction with respect to the canonical trace, we obtain the reduced algebra $C^*_{red}(\Gamma)$. Thus, all these algebras $A$ correspond to a unique compact quantum measured space in the above sense, which is the abstract group dual $\widehat{\Gamma}$. Let us record a statement about this finding, as follows:

\begin{proposition}
The category of group duals $\widehat{\Gamma}$ is a well-defined subcategory of the category of compact quantum measured spaces, with each $\widehat{\Gamma}$ corresponding to the full group algebra $C^*(\Gamma)$, or the reduced group algebra $C^*_{red}(\Gamma)$, or any algebra in between.
\end{proposition}

\begin{proof}
This is more of an empty statement, coming from the above discussion.
\end{proof}

With this in hand, it is tempting to go even further, namely forgetting about the $C^*$-algebras, and trying to axiomatize instead the operator algebras of type $L^\infty(X)$. Such an axiomatization is possible, and the resulting class of operator algebras consists of a certain special type of $C^*$-algebras, called ``finite von Neumann algebras''. 

\bigskip

However, and here comes our point, doing so would be bad, and would lead to a weak theory, because many spaces such as the compact groups, or the compact homogeneous spaces, do not come with a measure by definition, but rather by theorem.

\bigskip

In short, our ``fix'' is not a very good fix, and if we want a really strong theory, we must invent something else. In order to do so, our idea will be that of restricting the attention to certain special classes of quantum algebraic manifolds, as follows:

\index{real algebraic manifold}

\begin{definition}
A real algebraic submanifold $X\subset S^{N-1}_{\mathbb C,+}$ is a closed quantum subspace defined, at the level of the corresponding $C^*$-algebra, by a formula of type
$$C(X)=C(S^{N-1}_{\mathbb C,+})\Big/\Big<f_i(x_1,\ldots,x_N)=0\Big>$$
for certain noncommutative polynomials $f_i\in\mathbb C<x_1,\ldots,x_N>$. We denote by $\mathcal C(X)$ the $*$-subalgebra of $C(X)$ generated by the coordinate functions $x_1,\ldots,x_N$. 
\end{definition}

Observe that any family $f_i\in\mathbb C<x_1,\ldots,x_N>$ produces such a manifold $X$, simply by defining an algebra $C(X)$ as above. Observe also that the use of $S^{N-1}_{\mathbb C,+}$ is essential in all this, because the quadratic condition $\sum_ix_ix_i^*=\sum_ix_i^*x_i=1$ gives by positivity $||x_i||\leq1$ for any $i$, and so guarantees the fact that the universal $C^*$-norm is bounded.

\bigskip

We have already met such manifolds, in the context of the free spheres, free tori, and more generally in Proposition 1.22 above. Here is a list of examples:

\begin{proposition}
The following are algebraic submanifolds $X\subset S^{N-1}_{\mathbb C,+}$:
\begin{enumerate}
\item The spheres $S^{N-1}_\mathbb R\subset S^{N-1}_\mathbb C,S^{N-1}_{\mathbb R,+}\subset S^{N-1}_{\mathbb C,+}$.

\item Any compact Lie group, $G\subset U_n$, when $N=n^2$.

\item The duals $\widehat{\Gamma}$ of finitely generated groups, $\Gamma=<g_1,\ldots,g_N>$.
\end{enumerate}
\end{proposition}

\begin{proof}
These facts are all well-known, the proof being as follows:

\medskip

(1) This is true by definition of our various spheres.

\medskip

(2) Given a closed subgroup $G\subset U_n$, we have indeed an embedding $G\subset S^{N-1}_\mathbb C$, with $N=n^2$, given in double indices by: 
$$x_{ij}=\frac{u_{ij}}{\sqrt{n}}$$

We can further compose this embedding with the standard embedding $S^{N-1}_\mathbb C\subset S^{N-1}_{\mathbb C,+}$, and we obtain an embedding as desired. As for the fact that we obtain indeed a real algebraic manifold, this is well-known, coming either from Lie theory or from Tannakian duality. We will be back to this later on, in a more general context.

\medskip

(3) This follows from the fact that the variables $x_i=\frac{g_i}{\sqrt{N}}$ satisfy the quadratic relations  $\sum_ix_ix_i^*=\sum_ix_i^*x_i=1$, with the algebricity claim of the manifold being clear.
\end{proof}

At the level of the general theory, we have the following version of the Gelfand theorem, which is something very useful, and that we will use many times in what follows:

\begin{theorem}
When $X\subset S^{N-1}_{\mathbb C,+}$ is an algebraic manifold, given by
$$C(X)=C(S^{N-1}_{\mathbb C,+})\Big/\Big<f_i(x_1,\ldots,x_N)=0\Big>$$
for certain noncommutative polynomials $f_i\in\mathbb C<x_1,\ldots,x_N>$, we have
$$X_{class}=\left\{x\in S^{N-1}_\mathbb C\Big|f_i(x_1,\ldots,x_N)=0\right\}$$
and $X$ appears as a liberation of $X_{class}$.
\end{theorem}

\begin{proof}
This is something that already met, in the context of the free spheres. In general, the proof is similar, by using the Gelfand theorem. Indeed, if we denote by $X_{class}'$ the manifold constructed in the statement, then we have a quotient map of $C^*$-algebras as follows, mapping standard coordinates to standard coordinates:
$$C(X_{class})\to C(X_{class}')$$

Conversely now, from $X\subset S^{N-1}_{\mathbb C,+}$ we obtain $X_{class}\subset S^{N-1}_\mathbb C$, and since the relations defining $X_{class}'$ are satisfied by $X_{class}$, we obtain an inclusion of subspaces $X_{class}\subset X_{class}'$. Thus, at the level of algebras of continuous functions, we have a quotient map of $C^*$-algebras as follows, mapping standard coordinates to standard coordinates:
$$C(X_{class}')\to C(X_{class})$$

Thus, we have constructed a pair of inverse morphisms, and we are done.
\end{proof}

With these results in hand, we are now ready for formulating our second ``fix'' for the functoriality issues of the Gelfand correspondence, as follows:

\index{liberation}

\begin{definition}
The category of the real algebraic submanifolds $X\subset S^{N-1}_{\mathbb C,+}$ is the category of the universal $C^*$-algebras of type
$$C(X)=C(S^{N-1}_{\mathbb C,+})\Big/\Big<f_i(x_1,\ldots,x_N)=0\Big>$$
with $f_i\in\mathbb C<x_1,\ldots,x_N>$ being noncommutative polynomials, with the arrows $X\to Y$ being the $*$-algebra morphisms between $*$-algebras of coordinates
$$\mathcal C(Y)\to\mathcal C(X)$$
mapping standard coordinates to standard coordinates.
\end{definition}

In other words, what we are doing here is that of proposing a definition for the morphisms between the compact quantum spaces, in the particular case where these compact quantum spaces are algebraic submanifolds of the free complex sphere $S^{N-1}_{\mathbb C,+}$. And the point is that this ``fix'' perfectly works for the group duals, as follows:

\begin{theorem}
The category of finitely generated groups $\Gamma=<g_1,\ldots,g_N>$, with the morphisms being the group morphisms mapping generators to generators, embeds contravariantly via $\Gamma\to\widehat{\Gamma}$ into the category of real algebraic submanifolds $X\subset S^{N-1}_{\mathbb C,+}$. 
\end{theorem}

\begin{proof}
We know from Proposition 1.30 above that, given a finitely generated group $\Gamma=<g_1,\ldots,g_N>$, we have an embedding of algebraic manifolds $\widehat{\Gamma}\subset S^{N-1}_{\mathbb C,+}$, given by $x_i=\frac{g_i}{\sqrt{N}}$. Now since a morphism $C[\Gamma]\to C[\Lambda]$ mapping coordinates to coordinates means a morphism of groups $\Gamma\to\Lambda$ mapping generators to generators, our notion of isomorphism is indeed the correct one, as claimed.
\end{proof}

We will see later on that Theorem 1.33 has various extensions to the quantum groups and quantum homogeneous spaces that we will be interested in, which are all algebraic submanifolds $X\subset S^{N-1}_{\mathbb C,+}$. We will also see that all these manifolds have Haar integration functionals, which are traces, and so that for these manifolds, our functoriality fix from Definition 1.32 coincides with the ``von Neumann'' fix from Definition 1.27.

\bigskip

So, this will be our formalism, and operator algebra knowledge required. We should mention that our approach heavily relies on Woronowicz's philosophy in \cite{wo1}. Also, part of the above has been folklore for a long time, with the details worked out in \cite{bb4}.

\section*{1e. Exercises}

Generally speaking, the best complement to the material presented in this chapter is some further reading on operator theory, operator algebras and quantum mechanics. Here are some exercises in direct relation with what has been said above:

\begin{exercise}
Find an explicit orthonormal basis of the separable Hilbert space
$$H=L^2[0,1]$$
by applying the Gram-Schmidt procedure to the polynomials $f_n=x^n$, with $n\in\mathbb N$.
\end{exercise}

This is something quite tricky, and in case you get stuck, the answer can be found by doing an internet search with the keyword ``orthogonal polynomials''.

\begin{exercise}
Given a Hilbert space $H$, prove that we have embeddings of $*$-algebras as follows, which are both proper, unless $H$ is finite dimensional: 
$$B(H)\subset\mathcal L(H)\subset M_I(\mathbb C)$$
Also, prove that in this picture the adjoint operation $T\to T^*$ takes a very simple form, namely $(M^*)_{ij}=\overline{M}_{ji}$ at the level of the associated matrices.
\end{exercise}

This is something that we already discussed in the above, and the problem now is that of working out all the details. The counterexample on the left is tricky.

\begin{exercise}
Prove that for the usual matrices $A,B\in M_N(\mathbb C)$ we have
$$\sigma^+(AB)=\sigma^+(BA)$$
where $\sigma^+$ denotes the set of eigenvalues, taken with multiplicities.
\end{exercise}

As a remark here, we have seen that $\sigma(AB)=\sigma(BA)$ holds outside $\{0\}$. And the equality on $\{0\}$ holds as well, because $AB$ is invertible when $BA$ is invertible.

\begin{exercise}
Prove that an operator $T\in B(H)$ satisfies the condition
$$<Tx,x>\geq0$$
for any $x\in H$ precisely when it is positive in our sense, $\sigma(T)\in[0,\infty)$.
\end{exercise}

Working out first the case of the usual matrices, $M\in M_N(\mathbb C)$, with not much advanced linear algebra involved, is actually a very good preliminary exercise.

\begin{exercise}
Prove that the Pontrjagin dual of $\mathbb Z_N$ is this group itself
$$\widehat{\mathbb Z}_N=\mathbb Z_N$$
and work out the details of the subsequent isomorphism $C^*(\mathbb Z_N)\simeq C(\mathbb Z_N)$.
\end{exercise}

Here some knowledge of the roots of unity is needed. And as a bonus exercise, do a similar study for an arbitrary finite abelian group $G$.

\begin{exercise}
Find a discrete group $\Gamma$ such that the quotient map
$$C^*(\Gamma)\to C^*_{red}(\Gamma)$$
is not an isomorphism.
\end{exercise} 

This is actually something rather undoable. But hey, working on undoable exercises makes you strong. And do not worry, we will come back to this, later in this book.

\chapter{Quantum groups}

\section*{2a. Hopf algebras}

In this chapter we introduce the compact quantum groups. Let us start with the finite case, which is elementary, and easy to explain. The idea will be that of calling ``finite quantum groups'' the compact quantum spaces $G$ appearing via a formula of type $A=C(G)$, with the algebra $A$ being finite dimensional, and having some suitable extra structure. In order to simplify the presentation, we use the following terminology:

\index{comultiplication}
\index{counit}
\index{antipode}
\index{opposite algebra}

\begin{definition}
Given a finite dimensional $C^*$-algebra $A$, any morphisms of type
$$\Delta:A\to A\otimes A$$
$$\varepsilon:A\to\mathbb C$$
$$S:A\to A^{opp}$$
will be called comultiplication, counit and antipode. 
\end{definition}

The terminology comes from the fact that in the commutative case, $A=C(X)$, the morphisms $\Delta,\varepsilon,S$ are transpose to group-type operations, as follows:
$$m:X\times X\to X$$
$$u:\{.\}\to X$$
$$i:X\to X$$

The reasons for using the opposite algebra $A^{opp}$ instead of $A$ will become clear in a moment. Now with these conventions in hand, we can formulate:

\index{Hopf algebra}
\index{finite quantum group}
\index{quantum group}

\begin{definition}
A finite dimensional Hopf algebra is a finite dimensional $C^*$-algebra $A$, with a comultiplication, counit and antipode, satisfying
$$(\Delta\otimes id)\Delta=(id\otimes \Delta)\Delta$$
$$(\varepsilon\otimes id)\Delta=id$$
$$(id\otimes\varepsilon)\Delta=id$$
$$m(S\otimes id)\Delta=\varepsilon(.)1$$
$$m(id\otimes S)\Delta=\varepsilon(.)1$$
along with the condition $S^2=id$. Given such an algebra we write $A=C(G)=C^*(H)$, and call $G,H$ finite quantum groups, dual to each other.
\end{definition}

In this definition everything is standard, except for our choice to use $C^*$-algebras in all that we are doing, and also for the last axiom, $S^2=id$. This axiom corresponds to the fact that, in the corresponding quantum group, we have:
$$(g^{-1})^{-1}=g$$

It is possible to prove that this condition is automatic, in the present $C^*$-algebra setting. However, this is something non-trivial, and since all this is just a preliminary discussion, not needed later, we have opted for including $S^2=id$ in our axioms. 

\bigskip

\index{cocommutative}

For reasons that will become clear in a moment, we say that a Hopf algebra $A$ as above is cocommutative if, with $\Sigma(a\otimes b)=b\otimes a$ being the flip, we have:
$$\Sigma\Delta=\Delta$$

With this convention made, we have the following result, which summarizes the basic theory of finite quantum groups, and justifies the terminology and axioms:

\index{commutative Hopf algebra}
\index{cocommutative Hopf algebra}

\begin{theorem}
The following happen:
\begin{enumerate}
\item If $G$ is a finite group then $C(G)$ is a commutative Hopf algebra, with
$$\Delta(\varphi)=(g,h)\to \varphi(gh)$$
$$\varepsilon(\varphi)=\varphi(1)$$
$$S(\varphi)=g\to\varphi(g^{-1})$$
as structural maps. Any commutative Hopf algebra is of this form. 

\item If $H$ is a finite group then $C^*(H)$ is a cocommutative Hopf algebra, with
$$\Delta(g)=g\otimes g$$
$$\varepsilon(g)=1$$
$$S(g)=g^{-1}$$
as structural maps. Any cocommutative Hopf algebra is of this form.

\item If $G,H$ are finite abelian groups, dual to each other via Pontrjagin duality, then we have an identification of Hopf algebras $C(G)=C^*(H)$.
\end{enumerate}
\end{theorem}

\begin{proof}
These results are all elementary, the idea being as follows:

\medskip

(1) The fact that $\Delta,\varepsilon,S$ satisfy the axioms is clear from definitions, and the converse follows from the Gelfand theorem, by working out the details, regarding $\Delta,\varepsilon,S$. 

\medskip

(2) Once again, the fact that $\Delta,\varepsilon,S$ satisfy the axioms is clear from definitions, with the remark that the use of the opposite multiplication $(a,b)\to a\cdot b$ in really needed here, in order for the antipode $S$ to be an algebra morphism, as shown by:
$$S(gh)
=(gh)^{-1}
=h^{-1}g^{-1}
=g^{-1}\cdot h^{-1}
=S(g)\cdot S(h)$$

For the converse, we use a trick. Let $A$ be an arbitrary Hopf algebra, as in Definition 2.2, and consider its comultiplication, counit, multiplication, unit and antipode maps. The transposes of these maps are then linear maps as follows:
$$\Delta^t:A^*\otimes A^*\to A^*$$
$$\varepsilon^t:\mathbb C\to A^*$$
$$m^t:A^*\to A^*\otimes A^*$$
$$u^t:A^*\to\mathbb C$$
$$S^t:A^*\to A^*$$

It is routine to check that these maps make $A^*$ into a Hopf algebra. Now assuming that  $A$ is cocommutative, it follows that $A^*$ is commutative, so by (1) we obtain $A^*=C(G)$ for a certain finite group $G$, which in turn gives $A=C^*(G)$, as desired.

\medskip

(3) This follows from the discussion in the proof of (2) above.
\end{proof}

This was for the basics of finite quantum groups, under the strongest possible axioms. It is possible to further build on this, but we will discuss this directly in the compact setting. For more on Hopf algebras, over $\mathbb C$ as above, or over $\mathbb C$ with weaker axioms, or over other fields $k$, we refer to Abe \cite{abe}, Chari-Pressley \cite{cpr} and Majid \cite{maj}.

\section*{2b. Axioms, theory}

\index{Leningrad school}
\index{Faddeev}
\index{Woronowicz}

Let us get now into the compact quantum group case. Thinks are quite tricky here, with the origin of the modern theory going back to the work from the 70s of the Leningrad School of physics, by Faddeev and others \cite{fad}. From that work emerged a mathematical formalism, explained and developed in the papers of Drinfeld \cite{dri} and Jimbo \cite{jim} on one hand, and in the papers of Woronowicz \cite{wo1}, \cite{wo2}, on the other.

\bigskip

For our purposes here, which are rather post-modern, with the aim on focusing on what is beautiful and essential, from the point of view of present and future mathematics, and also potentially useful, from the point of view of present and future physics, we will only need a light version of all this theory, somewhat in the spirit of Definition 2.2, and of old-style mathematics, such as that of Brauer \cite{bra} and Weyl \cite{wey}. 

\bigskip

Let us begin with an exploration of the subject, leading quite often to negative results, and with all this being of course a bit subjective, and advanced too. The discussion will be harsh mathematical physics, so hang on. First, we have:

\begin{fact}
Lifting the finite dimensionality assumption on the Hopf algebra $A$ from Definition 2.2 does not work.
\end{fact}

To be more precise, the problem comes from the axioms to be satisfied by $S$, which do not make sense in the infinite dimensional setting, due to problems with $\otimes$. Of course, you might say why not ditching then operator algebras, and the topological tensor products $\otimes$ coming with them. But then if we do so, this basically means ditching our quantum mechanics motivations too, and isn't this the worse thing that can happen.

\bigskip

In view of this, an idea, which looks rather viable, is that of forgetting about the antipode $S$, which most likely brings troubles. But this does not work either:

\begin{fact}
Reformulating things as to make dissapear the antipode $S$, and then lifting the finite dimensionality assumption on the algebra $A$, does not work either.
\end{fact}

To be more precise, the antipode $S$ is about inversion in the corresponding quantum group $G$, and for making it dissapear, we can use the simple group theory fact that a semigroup $G$ is a group precisely when it has cancellation, $gh=gk\implies h=k$ and $hg=kg\implies h=k$. But the problem is that the axioms for compact quantum groups that we obtain in this way are quite ugly, and hard to verify, and there is better.

\bigskip

By the way, no offense to anyone here, because the above-mentioned ugly axioms, as well as the beautiful ones to be discussed below, are both due to Woronowicz. In short, we have two papers of Woronowicz to choose from, and we will choose one, \cite{wo1}.

\bigskip

Moving ahead now, what to do. Obviously, and a bit surprisingly, Definition 2.2 is not a good starting point, so we must come up with something else. You would say why not looking into compact Lie groups, because the finite groups are after all some kind of degenerate compact Lie groups too. But here, we get into another trouble:

\begin{fact}
The free complex torus $\widehat{F_N}$, that we love, and that we want to be a compact quantum Lie group, has no interesting differential geometry.
\end{fact}

To be more precise, we know from chapter 1 that $\widehat{F_N}$ is the free analogue of the complex torus $\mathbb T^N$, and this is why we would like to have it as a compact quantum Lie group, and even more, as a central example of such quantum groups. On the other hand, for various mathematical and physical reasons, differential geometry and smoothness are phenomena which are in close relation with the classical world, and the known types of noncommutative differential geometry cannot cover wild, free objects like $\widehat{F_N}$.

\bigskip

We should mention here that, contrary to the comments on Fact 2.4 and Fact 2.5, which are technical but hopefully understandable, this is something really advanced. Smoothness in mathematics and physics is certainly easy to formally define, and you leaned that in Calculus 1. But the reasons behind smoothness, that you need to understand if you want to talk about noncommutative smoothness, without talking nonsense, are amazingly deep, requiring on the bottom line reading say Feynman \cite{fe1}, \cite{fe2}, \cite{fe3} for physics, and Connes \cite{con} for mathematics. So, just trust me here, on Fact 2.6.

\bigskip

Now, in view of all the above, what to do. We are in a bit of an impasse here, but fortunately, pure mathematics comes to the rescue, in the following way:

\begin{fact}
The compact Lie groups are exactly the closed subgroups $G\subset U_N$, and for such a closed subgroup the multiplication, unit and inverse operation are given by
$$(UV)_{ij}=\sum_kU_{ik}V_{kj}$$
$$(1_N)_{ij}=\delta_{ij}$$
$$(U^{-1})_{ij}=U_{ji}^*$$
that is, the usual formulae for unitary matrices.
\end{fact}

And isn't this exactly what we need. Assuming that we are a bit familiar with Gelfand duality, and so are we, it shouldn't be hard from this to axiomatize the algebras of type $A=C(G)$, with $G\subset U_N$ being a closed subgroup, and then lift the commutativity assumption on $A$, as to have our axioms for the compact quantum Lie groups.

\bigskip

Getting directly to the answer, and with the Gelfand duality details in the classical case, $G\subset U_N$, to be explained in a moment, in the proof of Proposition 2.9 below, we are led in this way to the following definition, due to Woronowicz \cite{wo1}:

\index{Woronowicz algebra}
\index{quantum group}
\index{compact quantum group}
\index{compact matrix quantum group}
\index{opposite algebra}
\index{comultiplication}
\index{counit}
\index{antipode}

\begin{definition}
A Woronowicz algebra is a $C^*$-algebra $A$, given with a unitary matrix $u\in M_N(A)$ whose coefficients generate $A$, such that we have morphisms of $C^*$-algebras
$$\Delta:A\to A\otimes A$$
$$\varepsilon:A\to\mathbb C$$
$$S:A\to A^{opp}$$
given by the following formulae, on the standard generators $u_{ij}$:
$$\Delta(u_{ij})=\sum_ku_{ik}\otimes u_{kj}$$
$$\varepsilon(u_{ij})=\delta_{ij}$$
$$S(u_{ij})=u_{ji}^*$$
In this case, we write $A=C(G)$, and call $G$ a compact quantum Lie group.
\end{definition}

All this is quite subtle, and there are countless comments to be made here. Generally speaking, we will defer these comments for a bit later, once we'll know at least the basic examples, and also some basic theory. As some quick comments, however:

\bigskip

(1) In the above definition $A\otimes A$ can be any topological tensor product of $A$ with itself, meaning $C^*$-algebraic completion of the usual algebraic tensor product, and with the choice of the exact $\otimes$ operation being irrelevant, because we will divide later the class of Woronowicz algebras by a certain equivalence relation, making the choice of $\otimes$ to be irrelevant. In short, good news, no troubles with $\otimes$, and more on this later.

\bigskip

(2) Generally speaking, the above definition is motivated by Fact 2.7, and a bit of Gelfand duality thinking, and we will see details in a moment, in the proof of Proposition 2.9 below. The morphisms $\Delta,\varepsilon,S$ are called comultiplication, counit and antipode. Observe that if these morphisms exist, they are unique. This is analogous to the fact that a closed set of unitary matrices $G\subset U_N$ is either a compact group, or not.

\bigskip

(3) For everything else regarding Definition 2.8, and there are so many things to be said here, as a continuation of Facts 2.4, 2.5, 2.6, 2.7, allow me please to state and prove Propositions 2.9, 2.10, 2.11, 2.12 below, matter of exploring a bit Definition 2.8 and its consequences, then Definition 2.13 below, coming as a complement to Definition 2.8. And then we will resume our philosophical examination of possible rival definitions.

\bigskip

So, getting started now, and taking Definition 2.8 as it is, mysterious new thing, that we will have to explore, we first have the following result:

\index{commutative Woronowicz algebra}

\begin{proposition}
Given a closed subgroup $G\subset U_N$, the algebra $A=C(G)$, with the matrix formed by the standard coordinates $u_{ij}(g)=g_{ij}$, is a Woronowicz algebra, and:
\begin{enumerate}
\item For this algebra, the morphisms $\Delta,\varepsilon,S$ appear as functional analytic transposes of the multiplication, unit and inverse maps $m,u,i$ of the group $G$.

\item This Woronowicz algebra is commutative, and conversely, any Woronowicz algebra which is commutative appears in this way.
\end{enumerate}
\end{proposition}

\begin{proof}
Since we have $G\subset U_N$, the matrix $u=(u_{ij})$ is unitary. Also, since the coordinates $u_{ij}$ separate the points of $G$, by the Stone-Weierstrass theorem we obtain that the $*$-subalgebra $\mathcal A\subset C(G)$ generated by them is dense. Finally, the fact that we have morphisms $\Delta,\varepsilon,S$ as in Definition 2.8 follows from the proof of (1) below.

\medskip

(1) We use the formulae for $U_N$ from Fact 2.7. The fact that the transpose of the multiplication $m^t$ satisfies the condition in Definition 2.8  follows from:
\begin{eqnarray*}
m^t(u_{ij})(U\otimes V)
&=&(UV)_{ij}\\
&=&\sum_kU_{ik}V_{kj}\\
&=&\sum_k(u_{ik}\otimes u_{kj})(U\otimes V)
\end{eqnarray*}

Regarding now the transpose of the unit map $u^t$, the verification of the condition in Definition 2.8 is trivial, coming from the following equalities:
$$u^t(u_{ij})
=1_{ij}
=\delta_{ij}$$

Finally, the transpose of the inversion map $i^t$ verifies the condition in Definition 2.8, because we have the following computation, valid for any $U\in G$:
$$i^t(u_{ij})(U)
=(U^{-1})_{ij}
=\bar{U}_{ji}
=u_{ji}^*(U)$$

(2) Assume that $A$ is commutative. By using the Gelfand theorem, we can write $A=C(G)$, with $G$ being a certain compact space. By using now the coordinates $u_{ij}$, we obtain an embedding $G\subset U_N$. Finally, by using $\Delta,\varepsilon,S$, it follows that the subspace $G\subset U_N$ that we have obtained is in fact a closed subgroup, and we are done.
\end{proof}

Let us go back now to the general setting of Definition 2.8. According to Proposition 2.9, and to the general $C^*$-algebra philosophy, the morphisms $\Delta,\varepsilon,S$ can be thought of as coming from a multiplication, unit map and inverse map, as follows:
$$m:G\times G\to G$$
$$u:\{.\}\to G$$
$$i:G\to G$$

Here is a first result of this type, expressing in terms of $\Delta,\varepsilon,S$ the fact that the underlying maps $m,u,i$ should satisfy the usual group theory axioms:

\index{coassociativity}
\index{square of antipode}

\begin{proposition}
The comultiplication, counit and antipode have the following properties, on the dense $*$-subalgebra $\mathcal A\subset A$ generated by the variables $u_{ij}$:
\begin{enumerate}
\item Coassociativity: $(\Delta\otimes id)\Delta=(id\otimes\Delta)\Delta$.

\item Counitality: $(id\otimes\varepsilon)\Delta=(\varepsilon\otimes id)\Delta=id$.

\item Coinversality: $m(id\otimes S)\Delta=m(S\otimes id)\Delta=\varepsilon(.)1$.\end{enumerate}
In addition, the square of the antipode is the identity, $S^2=id$.
\end{proposition}

\begin{proof}
Observe first that the result holds in the case where $A$ is commutative. Indeed, by using Proposition 2.9 we can write:
$$\Delta=m^t\quad,\quad
\varepsilon=u^t\quad,\quad
S=i^t$$

The above 3 conditions come then by transposition from the basic 3 group theory conditions satisfied by $m,u,i$, which are as follows, with $\delta(g)=(g,g)$:
$$m(m\times id)=m(id\times m)$$
$$m(id\times u)=m(u\times id)=id$$
$$m(id\times i)\delta=m(i\times id)\delta=1$$

Observe that $S^2=id$ is satisfied as well, coming from $i^2=id$, which is a consequence of the group axioms. In general now, the proof goes as follows:

\medskip

(1) We have indeed the following computation:
$$(\Delta\otimes id)\Delta(u_{ij})
=\sum_l\Delta(u_{il})\otimes u_{lj}
=\sum_{kl}u_{ik}\otimes u_{kl}\otimes u_{lj}$$

On the other hand, we have as well the following computation:
$$(id\otimes\Delta)\Delta(u_{ij})
=\sum_ku_{ik}\otimes\Delta(u_{kj})
=\sum_{kl}u_{ik}\otimes u_{kl}\otimes u_{lj}$$

(2) The proof here is quite similar. We first have the following computation:
$$(id\otimes\varepsilon)\Delta(u_{ij})
=\sum_ku_{ik}\otimes\varepsilon(u_{kj})
=u_{ij}$$

On the other hand, we have as well the following computation:
$$(\varepsilon\otimes id)\Delta(u_{ij})
=\sum_k\varepsilon(u_{ik})\otimes u_{kj}
=u_{ij}$$

(3) By using the fact that the matrix $u=(u_{ij})$ is unitary, we obtain:
\begin{eqnarray*}
m(id\otimes S)\Delta(u_{ij})
&=&\sum_ku_{ik}S(u_{kj})\\
&=&\sum_ku_{ik}u_{jk}^*\\
&=&(uu^*)_{ij}\\
&=&\delta_{ij}
\end{eqnarray*}

Similarly, we have the following computation:
\begin{eqnarray*}
m(S\otimes id)\Delta(u_{ij})
&=&\sum_kS(u_{ik})u_{kj}\\
&=&\sum_ku_{ki}^*u_{kj}\\
&=&(u^*u)_{ij}\\
&=&\delta_{ij}
\end{eqnarray*}

Finally, the formula $S^2=id$ holds as well on the generators, and we are done.
\end{proof}

Let us discuss now another class of basic examples, namely the group duals:

\index{group dual}
\index{cocommutative}
\index{cocommutative Woronowicz algebra}
\index{finitely generated group}

\begin{proposition}
Given a finitely generated discrete group $\Gamma=<g_1,\ldots,g_N>$, the group algebra $A=C^*(\Gamma)$, together with the diagonal matrix formed by the standard generators, $u=diag(g_1,\ldots,g_N)$, is a Woronowicz algebra, with $\Delta,\varepsilon,S$ given by: 
$$\Delta(g)=g\otimes g$$
$$\varepsilon(g)=1$$
$$S(g)=g^{-1}$$
This Woronowicz algebra is cocommutative, in the sense that $\Sigma\Delta=\Delta$.
\end{proposition}

\begin{proof}
Since the involution on $C^*(\Gamma)$ is given by $g^*=g^{-1}$, the standard generators $g_1,\ldots,g_N$ are unitaries, and so must be the diagonal matrix $u=diag(g_1,\ldots,g_N)$ formed by them. Also, since $g_1,\ldots,g_N$ generate $\Gamma$, these elements generate the group algebra $C^*(\Gamma)$ as well, in the algebraic sense. Let us verify now the axioms in Definition 2.8:

\medskip

(1) Consider the following map, which is a unitary representation:
$$\Gamma\to C^*(\Gamma)\otimes C^*(\Gamma)\quad,\quad 
g\to g\otimes g$$

This representation extends, as desired, into a morphism of algebras, as follows:
$$\Delta:C^*(\Gamma)\to C^*(\Gamma)\otimes C^*(\Gamma)\quad,\quad 
\Delta(g)=g\otimes g$$

(2) The situation for $\varepsilon$ is similar, because this comes from the trivial representation:
$$\Gamma\to\{1\}\quad,\quad 
g\to1$$

(3) Finally, the antipode $S$ comes from the following unitary representation:
$$\Gamma\to C^*(\Gamma)^{opp}\quad,\quad 
g\to g^{-1}$$

Summarizing, we have shown that we have a Woronowicz algebra, with $\Delta,\varepsilon,S$ being as in the statement. Regarding now the last assertion, observe that we have:
$$\Sigma\Delta(g)
=\Sigma(g\otimes g)
=g\otimes g
=\Delta(g)$$

Thus $\Sigma\Delta=\Delta$ holds on the group elements $g\in\Gamma$, and by linearity and continuity, this formula must hold on the whole algebra $C^*(\Gamma)$, as desired.
\end{proof}

We will see later that any cocommutative Woronowicz algebra appears as above, up to a standard equivalence relation for such algebras, and with this being something non-trivial. In the abelian group case now, we have a more precise result, as follows:

\index{Pontrjagin dual}
\index{group dual}

\begin{proposition}
Assume that $\Gamma$ as above is abelian, and let $G=\widehat{\Gamma}$ be its Pontrjagin dual, formed by the characters $\chi:\Gamma\to\mathbb T$. The canonical isomorphism
$$C^*(\Gamma)\simeq C(G)$$
transforms the comultiplication, counit and antipode of $C^*(\Gamma)$ into the comultiplication, counit and antipode of $C(G)$, and so is a compact quantum group isomorphism.
\end{proposition}

\begin{proof}
Assume indeed that $\Gamma=<g_1,\ldots,g_N>$ is abelian. Then with $G=\widehat{\Gamma}$ we have a group embedding $G\subset U_N$, constructed as follows:
$$\chi\to
\begin{pmatrix}
\chi(g_1)\\
&\ddots\\
&&\chi(g_N)\end{pmatrix}$$

Thus, we have two Woronowicz algebras to be compared, namely $C(G)$, constructed as in Proposition 2.9, and $C^*(\Gamma)$, constructed as in Proposition 2.11. We already know from chapter 1 that the underlying $C^*$-algebras are isomorphic. Now since $\Delta,\varepsilon,S$ agree on $g_1,\ldots,g_N$, they agree everywhere, and we are led to the above conclusions.
\end{proof}

As a conclusion to all this, we can supplement Definition 2.8 with:

\index{compact quantum group}
\index{discrete quantum group}

\begin{definition}
Given a Woronowicz algebra $A=C(G)$, we write as well
$$A=C^*(\Gamma)$$
and call $\Gamma=\widehat{G}$ a finitely generated discrete quantum group.
\end{definition}

As usual with this type of definition, this comes with a warning, because we still have to divide the Woronowicz algebras by a certain equivalence relation, in order for our quantum spaces to be well-defined. We will be back to this in a moment, with the fix.

\bigskip

This being said, time perhaps to resume our discussion about other axiomatizations, started at the beginning of the present section. With our present knowledge of Definition 2.8 and its consequences, we can now say a few more things, as follows:

\bigskip

(1) First of all, we have as examples all compact Lie groups, and also dually, all the finitely generated discrete groups. It is also possible, by using a regular representation construction, to show that any finite quantum group in the sense of Definition 2.2 is a compact quantum group in our sense. All this is very nice, to the point that we can ask ourselves if there is something missing, at the level of the basic examples.

\bigskip

(2) And here, we must talk about the quantum groups of Drinfeld \cite{dri} and Jimbo \cite{jim}. These are deformations of type $G^q$, with $G\subset U_N$ being a compact Lie group, and with $q\in\mathbb C$ being a parameter, which at 1 gives the group itself, $G^1=G$. These quantum groups are not covered by our formalism, due to the condition $S^2=id$ implied by our axioms, which is not satisfied by these objects $G^q$, except in the cases $q=\pm1$.

\bigskip

(3) Regarding $q=\pm1$, these are certainly interesting values of $q$, corresponding to commutation and anticommutation, so we could say that, at least, we have some serious common ground with Drinfeld-Jimbo. However, we will see later, in chapter 7 below, that that Drinfeld-Jimbo construction is wrong at $q=-1$, in the sense that better, semisimple quantum groups $G^{-1}$ can be constructed by using our formalism.

\bigskip

(4) In short, we disagree with Drinfeld-Jimbo on mathematical grounds, and there is actually a physical discussion to be made too, but let us not get here into that. Now, if you disagree with Drinfeld-Jimbo, you have to disagree too with the very idea of $S^2\neq id$, and so with the general formalism of Woronowicz in \cite{wo1}, designed for covering the case $q>0$. Which is exactly what we did when formulating Definition 2.8 above.

\bigskip

(5) And the story is not over here, because, ironically, the case $q\in\mathbb T$, and more specifically the case where $q$ is a root of unity, was the one that Drinfeld-Jimbo were mainly interested in, due to some beautiful ties with arithmetics, and some potential applications to physics too, and $q>0$ has nothing to do with all this. So, even when trusting Drinfeld-Jimbo, there would be no way to include it in our formalism.

\bigskip

(6) And for ending with something advanced, the correct framework for these quantum group disputes is Jones' subfactor theory \cite{jo1}, \cite{jo2}, \cite{jo3}. Both the Woronowicz and the Drinfeld-Jimbo quantum groups have their place there, and can be compared at wish, better understood and generalized, thanks notably to work by Kirillov Jr., Wenzl and Xu for the Drinfeld-Jimbo quantum groups. But all this would take us too far.

\bigskip

Back to work now, let us develop some further general theory. We first have:

\index{biunitary}

\begin{proposition}
Given a Woronowicz algebra $(A,u)$, we have
$$u^t=\bar{u}^{-1}$$
so the matrix $u=(u_{ij})$ is a biunitary, meaning unitary, with unitary transpose.
\end{proposition}

\begin{proof}
The idea is that $u^t=\bar{u}^{-1}$ comes from $u^*=u^{-1}$, by applying the antipode. Indeed, by denoting $(a,b)\to a\cdot b$ the multiplication of $A^{opp}$, we have:
\begin{eqnarray*}
(uu^*)_{ij}=\delta_{ij}
&\implies&\sum_ku_{ik}u_{jk}^*=\delta_{ij}\\
&\implies&\sum_kS(u_{ik})\cdot S(u_{jk}^*)=\delta_{ij}\\
&\implies&\sum_ku_{kj}u_{ki}^*=\delta_{ij}\\
&\implies&(u^t\bar{u})_{ji}=\delta_{ij}
\end{eqnarray*}

Similarly, we have the following computation:
\begin{eqnarray*}
(u^*u)_{ij}=\delta_{ij}
&\implies&\sum_ku_{ki}^*u_{kj}=\delta_{ij}\\
&\implies&\sum_kS(u_{ki}^*)\cdot S(u_{kj})=\delta_{ij}\\
&\implies&\sum_ku_{jk}^*u_{ik}=\delta_{ij}\\
&\implies&(\bar{u}u^t)_{ji}=\delta_{ij}
\end{eqnarray*}

Thus, we are led to the conclusion in the statement.
\end{proof}

By using Proposition 2.14 we obtain the following theoretical result, which makes the link with the algebraic manifold considerations from chapter 1:

\begin{proposition}
Given a Woronowicz algebra $A=C(G)$, we have an embedding
$$G\subset S^{N^2-1}_{\mathbb C,+}$$
given in double indices by $x_{ij}=\frac{u_{ij}}{\sqrt{N}}$, where $u_{ij}$ are the standard coordinates of $G$.
\end{proposition}

\begin{proof}
This is something that we already know for the classical groups, and for the group duals as well, from chapter 1. In general, the proof is similar, coming from the fact that the matrices $u,\bar{u}$ are both unitaries, that we know from Proposition 2.14.
\end{proof}

In view of the above result, we can take some inspiration from the Gelfand correspondence ``fix'' presented in chapter 1, and formulate:

\begin{definition}
Given two Woronowicz algebras $(A,u)$ and $(B,v)$, we write 
$$A\simeq B$$
and identify the corresponding quantum groups, when we have an isomorphism 
$$<u_{ij}>\simeq<v_{ij}>$$
of $*$-algebras, mapping standard coordinates to standard coordinates.
\end{definition}

With this convention, the functoriality problem is fixed, any compact or discrete quantum group corresponding to a unique Woronowicz algebra, up to equivalence. 

\bigskip

As another comment, we can now see why in Definition 2.8 the choice of the exact topological tensor product $\otimes$ is irrelevant. Indeed, no matter what tensor product $\otimes$ we use there, we end up with the same Woronowicz algebra, and the same compact and discrete quantum groups, up to equivalence. In practice, we will use in what follows the simplest such tensor product $\otimes$, which is the so-called maximal one, obtained as completion of the usual algebraic tensor product with respect to the biggest $C^*$-norm. With the remark that this maximal tensor product is something rather algebraic and abstract, and so can be treated, in practice, as a usual algebraic tensor product.

\bigskip

We will be back to this later, with a number of supplementary comments, and some further results on the subject, when talking about amenability.

\section*{2c. Product operations}

We have seen so far that the compact quantum Lie groups can be axiomatized, and that as a bonus, we obtain in this way a definition as well for the finitely generated discrete quantum groups. Let us get now into a more exciting question, namely the construction of examples. We first have the following construction, due to Wang \cite{wa1}:

\index{product of quantum groups}

\begin{proposition}
Given two compact quantum groups $G,H$, so is their product $G\times H$, constructed according to the following formula:
$$C(G\times H)=C(G)\otimes C(H)$$ 
Equivalently, at the level of the associated discrete duals $\Gamma,\Lambda$, we can set
$$C^*(\Gamma\times\Lambda)=C^*(\Gamma)\otimes C^*(\Lambda)$$
and we obtain the same equality of Woronowicz algebras as above.
\end{proposition}

\begin{proof}
Assume indeed that we have two Woronowicz algebras, $(A,u)$ and $(B,v)$. Our claim is that the following construction produces a Woronowicz algebra:
$$C=A\otimes B\quad,\quad w=diag(u,v)$$

Indeed, the matrix $w$ is unitary, and its coefficients generate $C$. As for the existence of the maps $\Delta,\varepsilon,S$, this follows from the functoriality properties of $\otimes$, which is here, as usual, the universal $C^*$-algebraic completion of the algebraic tensor product. 

With this claim in hand, the first assertion is clear. As for the second assertion, let us recall that when $G,H$ are classical and abelian, we have the following formula:
$$\widehat{G\times H}=\widehat{G}\times\widehat{H}$$

Thus, our second assertion is simply a reformulation of the first assertion, with the $\times$ symbol used there being justified by this well-known group theory formula. 
\end{proof}

Here is now a more subtle construction, once again due to Wang \cite{wa1}:

\index{free product}
\index{dual free product}

\begin{proposition}
Given two compact quantum groups $G,H$, so is their dual free product $G\,\hat{*}\,H$, constructed according to the following formula:
$$C(G\,\hat{*}\,H)=C(G)*C(H)$$ 
Equivalently, at the level of the associated discrete duals $\Gamma,\Lambda$, we can set
$$C^*(\Gamma*\Lambda)=C^*(\Gamma)*C^*(\Lambda)$$
and we obtain the same equality of Woronowicz algebras as above.
\end{proposition}

\begin{proof}
The proof here is identical with the proof of Proposition 2.17, by replacing everywhere the tensor product $\otimes$ with the free product $*$, with this latter product being by definition the universal $C^*$-algebraic completion of the algebraic free product. 
\end{proof}

Here is another construction, which once again, has no classical counterpart:

\index{free complexification}

\begin{proposition}
Given a compact quantum group $G$, so is its free complexification $\widetilde{G}$, constructed according to the following formula, where $z=id\in C(\mathbb T)$:
$$C(\widetilde{G})\subset C(\mathbb T)*C(G)\quad,\quad\tilde{u}=zu$$
Equivalently, at the level of the associated discrete dual $\Gamma$, we can set
$$C^*(\widetilde{\Gamma})\subset C^*(\mathbb Z)*C^*(\Gamma)\quad,\quad\tilde{u}=zu$$
where $z=1\in\mathbb Z$, and we obtain the same Woronowicz algebra as above.
\end{proposition}

\begin{proof}
This follows from Proposition 2.17. Indeed, we know that $C(\mathbb T)*C(G)$ is a Woronowicz algebra, with matrix of coordinates $w=diag(z,u)$. Now, let us try to replace this matrix with the matrix $\tilde{u}=zu$. This matrix is unitary, and we have:
$$\Delta(\tilde{u}_{ij})
=(z\otimes z)\sum_ku_{ik}\otimes u_{kj}
=\sum_k\tilde{u}_{ik}\otimes\tilde{u}_{kj}$$

Similarly, in what regards the counit, we have the following formula:
$$\varepsilon(\tilde{u}_{ij})
=1\cdot\delta_{ij}
=\delta_{ij}$$

Finally, recalling that $S$ takes values in the opposite algebra, we have as well:
$$S(\tilde{u}_{ij})
=u_{ji}^*\cdot\bar{z}
=\tilde{u}_{ji}^*$$

Summarizing, the conditions in Definition 2.8 are satisfied, except for the fact that the entries of $\tilde{u}=zu$ do not generate the whole algebra $C(\mathbb T)*C(G)$. We conclude that if we let $C(\widetilde{G})\subset C(\mathbb T)*C(G)$ be the subalgebra generated by the entries of $\tilde{u}=zu$, as in the statement, then the conditions in Definition 2.8 are satisfied, as desired.
\end{proof}

Another standard operation is that of taking subgroups:

\index{quantum subgroup}
\index{quotient quantum group}
\index{Hopf ideal}

\begin{proposition}
Let $G$ be compact quantum group, and let $I\subset C(G)$ be a closed $*$-ideal satisfying the following condition: 
$$\Delta(I)\subset C(G)\otimes I+I\otimes C(G)$$
We have then a closed quantum subgroup $H\subset G$, constructed as follows:
$$C(H)=C(G)/I$$
At the dual level we obtain a quotient of discrete quantum groups, $\widehat{\Gamma}\to\widehat{\Lambda}$.
\end{proposition}

\begin{proof}
This follows indeed from the above conditions on $I$, which are designed precisely as for $\Delta,\varepsilon,S$ to factorize through the quotient. As for the last assertion, this is just a reformulation, coming from the functoriality properties of the Pontrjagin duality.
\end{proof}

\index{corepresentation}

In order to discuss now the quotient operation, let us agree to call ``corepresentation'' of a Woronowicz algebra $A$ any unitary matrix $v\in M_n(\mathcal A)$ satisfying: 
$$\Delta(v_{ij})=\sum_kv_{ik}\otimes v_{kj}\quad,\quad 
\varepsilon(v_{ij})=\delta_{ij}\quad,\quad
S(v_{ij})=v_{ji}^*$$

We will study in detail such corepresentations in chapter 3 below. For the moment, we just need their definition, in order to formulate the following result:

\index{quantum subgroup}
\index{quotient quantum group}

\begin{proposition}
Let $G$ be a compact quantum group, and $v=(v_{ij})$ be a corepresentation of $C(G)$. We have then a quotient quantum group $G\to H$, given by:
$$C(H)=<v_{ij}>$$
At the dual level we obtain a discrete quantum subgroup, $\widehat{\Lambda}\subset\widehat{\Gamma}$.
\end{proposition}

\begin{proof}
Here the first assertion follows from the above definition of the corepresentations, and the second assertion is just a reformulation of it, coming from the basic functoriality properties of the Pontrjagin duality.
\end{proof}

Finally, here is one more construction, which is something more tricky, and which will be of importance in what follows:

\index{projective version}

\begin{theorem}
Given a compact quantum group $G$, with fundamental corepresentation denoted $u=(u_{ij})$, the $N^2\times N^2$ matrix given in double index notation by
$$v_{ia,jb}=u_{ij}u_{ab}^*$$
is a corepresentation in the above sense, and we have the following results:
\begin{enumerate}
\item The corresponding quotient $G\to PG$ is a compact quantum group.

\item Via the standard embedding $G\subset S^{N^2-1}_{\mathbb C,+}$, this is the projective version.

\item In the classical group case, $G\subset U_N$, we have $PG=G/(G\cap\mathbb T^N)$.

\item In the group dual case, with $\Gamma=<g_i>$, we have $\widehat{P\Gamma}=<g_ig_j^{-1}>$.
\end{enumerate}
\end{theorem}

\begin{proof}
The fact that $v$ is indeed a corepresentation is routine, and follows as well from the general properties of such corepresentations, to be discussed in chapter 3 below. Regarding now other assertions, the proofs go as follows:

\medskip

(1) This follows from Proposition 2.21 above.

\medskip

(2) Observe first that, since the matrix $v=(v_{ia,jb})$ is biunitary, we have indeed an embedding $G\subset S^{N^2-1}_{\mathbb C,+}$ as in the statement, given in double index notation by $x_{ia,jb}=\frac{v_{ia,jb}}{N}$. Now with this formula in hand, the assertion is clear from definitions.

\medskip

(3) This follows from the elementary fact that, via Gelfand duality, $w$ is the matrix of coefficients of the adjoint representation of $G$, whose kernel is the subgroup $G\cap\mathbb T^N$, where $\mathbb T^N\subset U_N$ denotes the subgroup formed by the diagonal matrices.

\medskip

(4) This is something trivial, which follows from definitions.
\end{proof}

\section*{2d. Free constructions}

At the level of the really ``new'' examples now, we have basic liberation constructions, going back to the pioneering work of Wang \cite{wa1}, \cite{wa2}, and to the subsequent papers \cite{ba1}, \cite{ba2}, as well as several more recent constructions. We first have, following Wang \cite{wa1}:

\index{free quantum group}
\index{free orthogonal quantum group}
\index{free unitary quantum group}
\index{orthogonal quantum group}
\index{unitary quantum group}

\begin{theorem}
The following universal algebras are Woronowicz algebras,
\begin{eqnarray*}
C(O_N^+)&=&C^*\left((u_{ij})_{i,j=1,\ldots,N}\Big|u=\bar{u},u^t=u^{-1}\right)\\
C(U_N^+)&=&C^*\left((u_{ij})_{i,j=1,\ldots,N}\Big|u^*=u^{-1},u^t=\bar{u}^{-1}\right)
\end{eqnarray*}
so the underlying compact quantum spaces $O_N^+,U_N^+$ are compact quantum groups.
\end{theorem}

\begin{proof}
This follows from the elementary fact that if a matrix $u=(u_{ij})$ is orthogonal or biunitary, as above, then so must be the following matrices:
$$u^\Delta_{ij}=\sum_ku_{ik}\otimes u_{kj}\quad,\quad 
u^\varepsilon_{ij}=\delta_{ij}\quad,\quad
u^S_{ij}=u_{ji}^*$$

Consider indeed the matrix $U=u^\Delta$. We have then:
$$(UU^*)_{ij}
=\sum_{klm}u_{il}u_{jm}^*\otimes u_{lk}u_{mk}^*
=\sum_{lm}u_{il}u_{jm}^*\otimes\delta_{lm}
=\delta_{ij}$$

In the other sense the computation is similar, as follows:
$$(U^*U)_{ij}
=\sum_{klm}u_{kl}^*u_{km}\otimes u_{li}^*u_{mj}
=\sum_{lm}\delta_{lm}\otimes u_{li}^*u_{mj}
=\delta_{ij}$$

The verification of the unitarity of $\bar{U}$ is similar. We first have:
$$(\bar{U}U^t)_{ij}
=\sum_{klm}u_{il}^*u_{jm}\otimes u_{lk}^*u_{mk}
=\sum_{lm}u_{il}^*u_{jm}\otimes\delta_{lm}
=\delta_{ij}$$

In the other sense the computation is similar, as follows:
$$(U^t\bar{U})_{ij}
=\sum_{klm}u_{kl}u_{km}^*\otimes u_{li}u_{mj}^*
=\sum_{lm}\delta_{lm}\otimes u_{li}u_{mj}^*
=\delta_{ij}$$

Regarding now the matrix $u^\varepsilon=1_N$, and also the matrix $u^S$, their biunitarity its clear. Thus, we can indeed define morphisms $\Delta,\varepsilon,S$ as in Definition 2.8, by using the universal properties of $C(O_N^+)$, $C(U_N^+)$, and this gives the result.
\end{proof}

Let us study now the above quantum groups, with the techniques that we have. As a first observation, we have embeddings of compact quantum groups, as follows:
$$\xymatrix@R=15mm@C=15mm{
U_N\ar[r]&U_N^+\\
O_N\ar[r]\ar[u]&O_N^+\ar[u]
}$$

The basic properties of $O_N^+,U_N^+$ can be summarized as follows:

\index{free group}

\begin{theorem}
The quantum groups $O_N^+,U_N^+$ have the following properties:
\begin{enumerate}
\item The closed subgroups $G\subset U_N^+$ are exactly the $N\times N$ compact quantum groups. As for the closed subgroups $G\subset O_N^+$, these are those satisfying $u=\bar{u}$.

\item We have liberation embeddings $O_N\subset O_N^+$ and $U_N\subset U_N^+$, obtained by dividing the algebras $C(O_N^+),C(U_N^+)$ by their respective commutator ideals.

\item We have as well embeddings $\widehat{L}_N\subset O_N^+$ and $\widehat{F}_N\subset U_N^+$, where $L_N$ is the free product of $N$ copies of $\mathbb Z_2$, and where $F_N$ is the free group on $N$ generators.
\end{enumerate}
\end{theorem}

\begin{proof}
All these assertions are elementary, as follows:

\medskip

(1) This is clear from definitions, and from Proposition 2.14. 

\medskip

(2) This follows from the Gelfand theorem, which shows that we have presentation results for $C(O_N),C(U_N)$ as follows, similar to those in Theorem 2.23:
\begin{eqnarray*}
C(O_N)&=&C^*_{comm}\left((u_{ij})_{i,j=1,\ldots,N}\Big|u=\bar{u},u^t=u^{-1}\right)\\
C(U_N)&=&C^*_{comm}\left((u_{ij})_{i,j=1,\ldots,N}\Big|u^*=u^{-1},u^t=\bar{u}^{-1}\right)
\end{eqnarray*}

(3) This follows from (1) and from Proposition 2.11 above, with the remark that with $u=diag(g_1,\ldots,g_N)$, the condition $u=\bar{u}$ is equivalent to $g_i^2=1$, for any $i$.
\end{proof}

As an interesting philosophical conclusion, if we denote by $L_N^+,F_N^+$ the discrete quantum groups which are dual to $O_N^+,U_N^+$, then we have embeddings as follows:
$$L_N\subset L_N^+\quad,\quad
F_N\subset F_N^+$$

\index{free free group}

Thus $F_N^+$ is some kind of ``free free group'', and $L_N^+$ is its real counterpart. This is not surprising, since $F_N,L_N$ are not ``fully free'', their group algebras being cocommutative.

\bigskip

The last assertion in Theorem 2.24 suggests the following construction, from \cite{bv2}:

\index{diagonal torus}

\begin{proposition}
Given a closed subgroup $G\subset U_N^+$, consider its ``diagonal torus'', which is the closed subgroup $T\subset G$ constructed as follows:
$$C(T)=C(G)\Big/\left<u_{ij}=0\Big|\forall i\neq j\right>$$
This torus is then a group dual, $T=\widehat{\Lambda}$, where $\Lambda=<g_1,\ldots,g_N>$ is the discrete group generated by the elements $g_i=u_{ii}$, which are unitaries inside $C(T)$.
\end{proposition}

\begin{proof}
Since $u$ is unitary, its diagonal entries $g_i=u_{ii}$ are unitaries inside $C(T)$. Moreover, from $\Delta(u_{ij})=\sum_ku_{ik}\otimes u_{kj}$ we obtain, when passing inside the quotient:
$$\Delta(g_i)=g_i\otimes g_i$$

It follows that we have $C(T)=C^*(\Lambda)$, modulo identifying as usual the $C^*$-completions of the various group algebras, and so that we have $T=\widehat{\Lambda}$, as claimed.
\end{proof}

With this notion in hand, Theorem 2.24 (3) tells us that the diagonal tori of $O_N^+,U_N^+$ are the group duals $\widehat{L}_N,\widehat{F}_N$. We will be back to this later.

\bigskip 

Here is now a more subtle result on $O_N^+,U_N^+$, having no classical counterpart:

\begin{proposition}
Consider the quantum groups $O_N^+,U_N^+$, with the corresponding fundamental corepresentations denoted $v,u$, and let $z=id\in C(\mathbb T)$.
\begin{enumerate}
\item We have a morphism $C(U_N^+)\to C(\mathbb T)*C(O_N^+)$, given by $u=zv$.

\item In other words, we have a quantum group embedding $\widetilde{O_N^+}\subset U_N^+$.

\item This embedding is an isomorphism at the level of the diagonal tori.
\end{enumerate}
\end{proposition}

\begin{proof}
The first two assertions follow from Proposition 2.19, or simply from the fact that $u=zv$ is biunitary. As for the third assertion, the idea here is that we have a similar model for the free group $F_N$, which is well-known to be faithful, $F_N\subset\mathbb Z*L_N$.
\end{proof}

We will be back to the above morphism later on, with a proof of its faithfulness, after performing a suitable GNS construction, with respect to the Haar functionals. 

\bigskip

Let us construct now some more examples of compact quantum groups. Following  \cite{ez1}, \cite{bsp}, \cite{bv2}, \cite{bdu}, we can introduce some intermediate liberations, as follows:

\index{half-liberation}
\index{half-classical orthogonal group}
\index{half-classical unitary group}

\begin{proposition}
We have intermediate quantum groups as follows,
$$\xymatrix@R=15mm@C=15mm{
U_N\ar[r]&U_N^*\ar[r]&U_N^+\\
O_N\ar[r]\ar[u]&O_N^*\ar[r]\ar[u]&O_N^+\ar[u]}$$
with $*$ standing for the fact that $u_{ij},u_{ij}^*$ must satisfy the relations $abc=cba$. 
\end{proposition}

\begin{proof}
This is similar to the proof of Theorem 2.23, by using the elementary fact that if the entries of $u=(u_{ij})$ half-commute, then so do the entries of $u^\Delta$, $u^\varepsilon$, $u^S$.
\end{proof}

\index{half-classical sphere}

In the same spirit, we have as well intermediate spheres as follows, with the symbol $*$ standing for the fact that $x_i,x_i^*$ must satisfy the relations $abc=cba$:
$$\xymatrix@R=15mm@C=15mm{
S^{N-1}_\mathbb C\ar[r]&S^{N-1}_{\mathbb C,*}\ar[r]&S^{N-1}_{\mathbb C,+}\\
S^{N-1}_\mathbb R\ar[r]\ar[u]&S^{N-1}_{\mathbb R,*}\ar[r]\ar[u]&S^{N-1}_{\mathbb R,+}\ar[u]
}$$

At the level of the diagonal tori, we have the following result:

\begin{theorem}
The tori of the basic spheres and quantum groups are as follows,
$$\xymatrix@R=15mm@C=15mm{
\widehat{\mathbb Z^N}\ar[r]&\widehat{\mathbb Z^{\circ N}}\ar[r]&\widehat{\mathbb Z^{*N}}\\
\widehat{\mathbb Z_2^N}\ar[r]\ar[u]&\widehat{\mathbb Z_2^{\circ N}}\ar[r]\ar[u]&\widehat{\mathbb Z_2^{*N}}\ar[u]}$$
with $\circ$ standing for the half-classical product operation for groups.
\end{theorem}

\begin{proof}
The idea here is as follows:

\medskip

(1) The result on the left is well-known.

\medskip

(2) The result on the right follows from Theorem 2.24 (3).

\medskip

(3) The middle result follows as well, by imposing the relations $abc=cba$.
\end{proof}

Let us discuss now the relation with the noncommutative spheres. Having the things started here is a bit tricky, and as a main source of inspiration, we have:

\index{algebraic manifold}
\index{affine isometry}
\index{isometry}

\begin{proposition}
Given an algebraic manifold $X\subset S^{N-1}_\mathbb C$, the formula
$$G(X)=\left\{U\in U_N\Big|U(X)=X\right\}$$
defines a compact group of unitary matrices, or isometries, called affine isometry group of $X$. For the spheres $S^{N-1}_\mathbb R,S^{N-1}_\mathbb C$ we obtain in this way the groups $O_N,U_N$.
\end{proposition}

\begin{proof}
The fact that $G(X)$ as defined above is indeed a group is clear, its compactness is clear as well, and finally the last assertion is clear as well. In fact, all this works for any closed subset $X\subset\mathbb C^N$, but we are not interested here in such general spaces.
\end{proof}

We have the following quantum analogue of the above construction:

\index{quantum isometry}
\index{quantum affine isometry}
\index{quantum isometry group}
 
\begin{proposition}
Given an algebraic manifold $X\subset S^{N-1}_{\mathbb C,+}$, the category of the closed subgroups $G\subset U_N^+$ acting affinely on $X$, in the sense that the formula
$$\Phi(x_i)=\sum_jx_j\otimes u_{ji}$$ 
defines a morphism of $C^*$-algebras as follows,
$$\Phi:C(X)\to C(X)\otimes C(G)$$
has a universal object, denoted $G^+(X)$, and called affine quantum isometry group of $X$.
\end{proposition}

\begin{proof}
Observe first that in the case where $\Phi$ as above exists, this morphism is automatically a coaction, in the sense that it satisfies the following conditions:
$$(\Phi\otimes id)\Phi=(id\otimes\Delta)\Phi$$
$$(id\otimes\varepsilon)\Phi=id$$

In order to prove now the result, assume that $X\subset S^{N-1}_{\mathbb C,+}$ comes as follows:
$$C(X)=C(S^{N-1}_{\mathbb C,+})\Big/\Big<f_\alpha(x_1,\ldots,x_N)=0\Big>$$

Consider now the following variables:
$$X_i=\sum_jx_j\otimes u_{ji}\in C(X)\otimes C(U_N^+)$$

Our claim is that $G=G^+(X)$ in the statement appears as follows:
$$C(G)=C(U_N^+)\Big/\Big<f_\alpha(X_1,\ldots,X_N)=0\Big>$$

In order to prove this claim, we have to clarify how the relations $f_\alpha(X_1,\ldots,X_N)=0$ are interpreted inside $C(U_N^+)$, and then show that $G$ is indeed a quantum group. So, pick one of the defining polynomials, $f=f_\alpha$, and write it as follows:
$$f(x_1,\ldots,x_N)=\sum_r\sum_{i_1^r\ldots i_{s_r}^r}\lambda_r\cdot x_{i_1^r}\ldots x_{i_{s_r}^r}$$

With $X_i=\sum_jx_j\otimes u_{ji}$ as above, we have the following formula:
$$f(X_1,\ldots,X_N)
=\sum_r\sum_{i_1^r\ldots i_{s_r}^r}\lambda_r\sum_{j_1^r\ldots j_{s_r}^r}x_{j_1^r}\ldots x_{j_{s_r}^r}\otimes u_{j_1^ri_1^r}\ldots u_{j_{s_r}^ri_{s_r}^r}$$

Since the variables on the right span a certain finite dimensional space, the relations $f(X_1,\ldots,X_N)=0$ correspond to certain relations between the variables $u_{ij}$. Thus, we have indeed a closed subspace $G\subset U_N^+$, coming with a universal map:
$$\Phi:C(X)\to C(X)\otimes C(G)$$

In order to show now that $G$ is a quantum group, consider the following elements:
$$u_{ij}^\Delta=\sum_ku_{ik}\otimes u_{kj}\quad,\quad 
u_{ij}^\varepsilon=\delta_{ij}\quad,\quad 
u_{ij}^S=u_{ji}^*$$

Consider as well the following associated elements, with $\gamma\in\{\Delta,\varepsilon,S\}$:
$$X_i^\gamma=\sum_jx_j\otimes u_{ji}^\gamma$$

From the relations $f(X_1,\ldots,X_N)=0$ we deduce that we have:
$$f(X_1^\gamma,\ldots,X_N^\gamma)
=(id\otimes\gamma)f(X_1,\ldots,X_N)
=0$$

But this shows that for any exponent $\gamma\in\{\Delta,\varepsilon,S\}$ we can map $u_{ij}\to u_{ij}^\gamma$, and it follows that $G$ is indeed a compact quantum group, and we are done.
\end{proof}

Following \cite{bgo} and related papers, we can now formulate:

\begin{theorem}
The quantum isometry groups of the basic spheres are
$$\xymatrix@R=15mm@C=17mm{
U_N\ar[r]&U_N^*\ar[r]&U_N^+\\
O_N\ar[r]\ar[u]&O_N^*\ar[r]\ar[u]&O_N^+\ar[u]}$$
modulo identifying, as usual, the various $C^*$-algebraic completions.
\end{theorem}

\begin{proof}
Let us first construct an action $U_N^+\curvearrowright S^{N-1}_{\mathbb C,+}$. We must prove here that the variables $X_i=\sum_jx_j\otimes u_{ji}$ satisfy the defining relations for $S^{N-1}_{\mathbb C,+}$, namely:
$$\sum_ix_ix_i^*=\sum_ix_i^*x_i=1$$

But this follows from the biunitarity of $u$. We have indeed:
\begin{eqnarray*}
\sum_iX_iX_i^*
&=&\sum_{ijk}x_jx_k^*\otimes u_{ji}u_{ki}^*\\
&=&\sum_jx_jx_j^*\otimes1\\
&=&1\otimes1
\end{eqnarray*}

In the other sense the computation is similar, as follows:
\begin{eqnarray*}
\sum_iX_i^*X_i
&=&\sum_{ijk}x_j^*x_k\otimes u_{ji}^*u_{ki}\\
&=&\sum_jx_j^*x_j\otimes1\\
&=&1\otimes1
\end{eqnarray*}

Regarding now $O_N^+\curvearrowright S^{N-1}_{\mathbb R,+}$, here we must check the extra relations $X_i=X_i^*$, and these are clear from $u_{ia}=u_{ia}^*$. Finally, regarding the remaining actions, the verifications are clear as well, because if the coordinates $u_{ia}$ and $x_a$ are subject to commutation relations of type $ab=ba$, or of type $abc=cba$, then so are the variables $X_i=\sum_jx_j\otimes u_{ji}$.

\medskip

We must prove now that all these actions are universal: 

\medskip

\underline{$S^{N-1}_{\mathbb R,+},S^{N-1}_{\mathbb C,+}$.} The universality of $U_N^+\curvearrowright S^{N-1}_{\mathbb C,+}$ is trivial by definition. As for the universality of $O_N^+\curvearrowright S^{N-1}_{\mathbb R,+}$, this comes from the fact that $X_i=X_i^*$, with $X_i=\sum_jx_j\otimes u_{ji}$ as above, gives $u_{ia}=u_{ia}^*$. Thus $G\curvearrowright S^{N-1}_{\mathbb R,+}$ implies $G\subset O_N^+$, as desired.

\medskip

\underline{$S^{N-1}_\mathbb R,S^{N-1}_\mathbb C$.} We use here a trick from Bhowmick-Goswami \cite{bhg}. Assuming first that we have an action $G\curvearrowright S^{N-1}_\mathbb R$, consider the following variables:
$$w_{kl,ij}=u_{ki}u_{lj}$$
$$p_{ij}=x_ix_j$$

In terms of these variables, which can be thought of as being projective coordinates, the corresponding projective coaction map is given by:
$$\Phi(p_{ij})=\sum_{kl}p_{kl}\otimes w_{kl,ij}$$

We have the following formulae:
\begin{eqnarray*}
\Phi(p_{ij})&=&\sum_{k<l}p_{kl}\otimes(w_{kl,ij}+w_{lk,ij})+\sum_kp_{kk}\otimes w_{kk,ij}\\
\Phi(p_{ji})&=&\sum_{k<l}p_{kl}\otimes(w_{kl,ji}+w_{lk,ji})+\sum_kp_{kk}\otimes w_{kk,ji}
\end{eqnarray*}

By comparing these two formulae, and then by using the linear independence of the variables $p_{kl}=x_kx_l$ with $k\leq l$, we conclude that we must have:
$$w_{kl,ij}+w_{lk,ij}=w_{kl,ji}+w_{lk,ji}$$

Let us apply the antipode to this formula. For this purpose, observe that we have:
$$S(w_{kl,ij})
=S(u_{ki}u_{lj})
=S(u_{lj})S(u_{ki})
=u_{jl}u_{ik}
=w_{ji,lk}$$

Thus by applying the antipode we obtain:
$$w_{ji,lk}+w_{ji,kl}=w_{ij,lk}+w_{ij,kl}$$

By relabelling the indices, we obtain from this: 
$$w_{kl,ij}+w_{kl,ji}=w_{lk,ij}+w_{lk,ji}$$

Now by comparing with the original relation, we obtain:
$$w_{lk,ij}=w_{kl,ji}$$

But, recalling that we have $w_{kl,ij}=u_{ki}u_{lj}$, this formula reads:
$$u_{li}u_{kj}=u_{kj}u_{li}$$

We therefore conclude we have $G\subset O_N$, as claimed. The proof of the universality of the action $U_N\curvearrowright S^{N-1}_\mathbb C$ is similar.

\medskip

\underline{$S^{N-1}_{\mathbb R,*},S^{N-1}_{\mathbb C,*}$.} Assume that we have an action $G\curvearrowright S^{N-1}_{\mathbb C,*}$. From $\Phi(x_a)=\sum_ix_i\otimes u_{ia}$ we obtain then that, with $p_{ab}=z_a\bar{z}_b$, we have:
$$\Phi(p_{ab})=\sum_{ij}p_{ij}\otimes u_{ia}u_{jb}^*$$

By multiplying these two formulae, we obtain:
\begin{eqnarray*}
\Phi(p_{ab}p_{cd})&=&\sum_{ijkl}p_{ij}p_{kl}\otimes u_{ia}u_{jb}^*u_{kc}u_{ld}^*\\
\Phi(p_{ad}p_{cb})&=&\sum_{ijkl}p_{il}p_{kj}\otimes u_{ia}u_{ld}^*u_{kc}u_{jb}^*
\end{eqnarray*}

The left terms being equal, and the first terms on the right being equal too, we deduce that, with $[a,b,c]=abc-cba$, we must have the following equality:
$$\sum_{ijkl}p_{ij}p_{kl}\otimes u_{ia}[u_{jb}^*,u_{kc},u_{ld}^*]=0$$

Since the variables $p_{ij}p_{kl}=z_i\bar{z}_jz_k\bar{z}_l$ depend only on $|\{i,k\}|,|\{j,l\}|\in\{1,2\}$, and this dependence produces the only relations between them, we are led to $4$ equations:

\medskip

(1) $u_{ia}[u_{jb}^*,u_{ka},u_{lb}^*]=0$, $\forall a,b$.

\medskip

(2) $u_{ia}[u_{jb}^*,u_{ka},u_{ld}^*]+u_{ia}[u_{jd}^*,u_{ka},u_{lb}^*]=0$, $\forall a$, $\forall b\neq d$.

\medskip

(3) $u_{ia}[u_{jb}^*,u_{kc},u_{lb}^*]+u_{ic}[u_{jb}^*,u_{ka},u_{lb}^*]=0$, $\forall a\neq c$, $\forall b$.

\medskip

(4) $u_{ia}([u_{jb}^*,u_{kc},u_{ld}^*]+[u_{jd}^*,u_{kc},u_{lb}^*])+u_{ic}([u_{jb}^*,u_{ka},u_{ld}^*]+[u_{jd}^*,u_{ka},u_{lb}^*])=0,\forall a\neq c,b\neq d$.

\medskip

From (1,2) we conclude that (2) holds with no restriction on the indices. By multiplying now this formula to the left by $u_{ia}^*$, and then summing over $i$, we obtain:
$$[u_{jb}^*,u_{ka},u_{ld}^*]+[u_{jd}^*,u_{ka},u_{lb}^*]=0$$

By applying now the antipode, then the involution, and finally by suitably relabelling all the indices, we successively obtain from this formula:
\begin{eqnarray*}
&&[u_{dl},u_{ak}^*,u_{bj}]+[u_{bl},u_{ak}^*,u_{dj}]=0\\
&\implies&[u_{dl}^*,u_{ak},u_{bj}^*]+[u_{bl}^*,u_{ak},u_{dj}^*]=0\\
&\implies&[u_{ld}^*,u_{ka},u_{jb}^*]+[u_{jd}^*,u_{ka},u_{lb}^*]=0
\end{eqnarray*}

Now by comparing with the original relation, above, we conclude that we have:
$$[u_{jb}^*,u_{ka},u_{ld}^*]=[u_{jd}^*,u_{ka},u_{lb}^*]=0$$

Thus we have reached to the formulae defining $U_N^*$, and we are done. Finally, in what regards the universality of the action $O_N^*\curvearrowright S^{N-1}_{\mathbb R,*}$, this follows from the universality of the actions $U_N^*\curvearrowright S^{N-1}_{\mathbb C,*}$ and of $O_N^+\curvearrowright S^{N-1}_{\mathbb R,+}$, and from $U_N^*\cap O_N^+=O_N^*$.
\end{proof}

As a conclusion to all this, we have now a simple and reliable definition for the compact quantum groups, in the Lie case, namely $G\subset U_N^+$, covering all the compact Lie groups, $G\subset U_N$, covering as well all the duals $\widehat{\Gamma}$ of the finitely generated groups, $F_N\to\Gamma$, and allowing the construction of several interesting examples, such as $O_N^+,U_N^+$. 

\bigskip

With respect to the noncommutative geometry questions raised in chapter 1 above, we certainly have here some advances. In order to further advance, however, we would need now representation theory results, in the spirit of Weyl \cite{wey}, for our quantum isometry groups. We will develop all this in what follows, in the next few chapters.

\section*{2e. Exercises} 

In connection with quantum groups, a good familiarity with the Hopf algebra $\Delta,\varepsilon,S$ calculus is one of the main needed things. Here is a first exercise:  

\index{Hopf algebra}
\index{dual Hopf algebra}
\index{finite quantum group}

\begin{exercise}
Given a finite dimensional Hopf algebra $A$, prove that its dual $A^*$ is a Hopf algebra too, with structural maps as follows:
$$\Delta^t:A^*\otimes A^*\to A^*$$
$$\varepsilon^t:\mathbb C\to A^*$$
$$m^t:A^*\to A^*\otimes A^*$$
$$u^t:A^*\to\mathbb C$$
$$S^t:A^*\to A^*$$
Also, check that $A$ is commutative if and only if $A^*$ is cocommutative, and also discuss what happens in the cases $A=C(G)$ and $A=C^*(H)$, with $G,H$ being finite groups.
\end{exercise}

This is something that we already discussed, but a bit in a hurry, as a preliminary to the compact quantum groups, and the problem now is that of filling all the details.

\begin{exercise}
Prove that the compact quantum groups $G$ which are finite, in the sense that $\dim C(G)<\infty$, coincide with the discrete quantum groups $\Gamma$ which are finite, in the sense that $\dim C^*(\Gamma)<\infty$, and coincide as well with the finite quantum groups.
\end{exercise}

This might sound obvious, but in practice, all this needs a proof. The first step is that of clearly formulating, in terms of algebras, what exactly you want to prove.

\begin{exercise}
Clarify the discrete quantum group formulation of the various compact quantum group product operations, namely taking subgroups, quotients, dual free products, free complexifications and projective versions.
\end{exercise}

This is something that was discussed in the above, but rather quickly. The problem is that of working out all the details, in dual formulation.

\begin{exercise}
Prove that the free complexification embedding
$$\widetilde{O_N^+}\subset U_N^+$$
is an isomorphism at the level of the associated diagonal tori.
\end{exercise}

As before, this is something that we talked about, but rather quickly, and this because we will prove anyway, later on, that the embedding $\widetilde{O_N^+}\subset U_N^+$ itself is an isomorphism.

\begin{exercise}
Find the complexification operation $O_N\to U_N$, by using linear algebra, or Lie algebras, or whatever other means.
\end{exercise}

As a comment, we will be mainly interested here in free quantum groups, where the complexification operation $O_N^+\to U_N^+$ is something much simpler, namely $U_N^+=\widetilde{O_N^+}$.

\chapter{Representation theory}

\section*{3a. Representations}

In order to reach to some more advanced insight into the structure of the compact quantum groups, we can use representation theory. We follow Woronowicz's paper \cite{wo1}, with a few simplifications coming from our $S^2=id$ formalism. We first have:

\index{representation}
\index{corepresentation}

\begin{definition}
A corepresentation of a Woronowicz algebra $(A,u)$ is a unitary matrix $v\in M_n(\mathcal A)$ over the dense $*$-algebra $\mathcal A=<u_{ij}>$, satisfying:
$$\Delta(v_{ij})=\sum_kv_{ik}\otimes v_{kj}$$
$$\varepsilon(v_{ij})=\delta_{ij}$$
$$S(v_{ij})=v_{ji}^*$$
That is, $v$ must satisfy the same conditions as $u$.
\end{definition}

As basic examples here, we have the trivial corepresentation, having dimension 1, as well as the fundamental corepresentation, and its adjoint:
$$1=(1)\quad,\quad
u=(u_{ij})\quad,\quad 
\bar{u}=(u_{ij}^*)$$

In the classical case, we recover in this way the usual representations of $G$:

\index{smooth representation}

\begin{proposition}
Given a closed subgroup $G\subset U_N$, the corepresentations of the associated Woronowicz algebra $C(G)$ are in one-to-one correspondence, given by
$$\pi(g)=\begin{pmatrix}v_{11}(g)&\ldots&v_{1n}(g)\\
\vdots&&\vdots\\
v_{n1}(g)&\ldots&v_{nn}(g)
\end{pmatrix}$$
with the finite dimensional unitary smooth representations of $G$.
\end{proposition}

\begin{proof}
With $A=C(G)$, consider the unitary matrices $v\in M_n(A)$ satisfying the equations in Definition 3.1. By using the computations from chapter 2, performed when proving that any closed subgroup $G\subset U_N$ is indeed a compact quantum group, we conclude that we have a correspondence $v\leftrightarrow\pi$ as in the statement, between such matrices, and the finite dimensional unitary representations of $G$.

\medskip

Regarding now the smoothness part, this is something more subtle, which requires some knowledge of Lie theory. The point is that any closed subgroup $G\subset U_N$ is a Lie group, and since the coefficient functions $u_{ij}:G\to\mathbb C$ are smooth, we have:
$$\mathcal A\subset C^\infty(G)$$

Thus, when assuming $v\in M_n(\mathcal A)$, the corresponding representation $\pi:G\to U_n$ is smooth, and the converse of this fact is known to hold as well.
\end{proof}

In general now, we have the following operations on the corepresentations:

\index{sum of corepresentations}
\index{tensor product of corepresentation}
\index{conjugate corepresentation}

\begin{proposition}
The corepresentations are subject to the following operations:
\begin{enumerate}
\item Making sums, $v+w=diag(v,w)$.

\item Making tensor products, $(v\otimes w)_{ia,jb}=v_{ij}w_{ab}$.

\item Taking conjugates, $(\bar{v})_{ij}=v_{ij}^*$.
\end{enumerate}
\end{proposition}

\begin{proof}
Observe that the result holds in the commutative case, where we obtain the usual operations on the representations of the corresponding group. In general now:

\medskip

(1) Everything here is clear, as already mentioned in chapter 2 above, when using such corepresentations in order to construct quantum group quotients.

\medskip

(2) First of all, the matrix $v\otimes w$ is unitary. Indeed, we have:
\begin{eqnarray*}
\sum_{jb}(v\otimes w)_{ia,jb}(v\otimes w)_{kc,jb}^*
&=&\sum_{jb}v_{ij}w_{ab}w_{cb}^*v_{kj}^*\\
&=&\delta_{ac}\sum_jv_{ij}v_{kj}^*\\
&=&\delta_{ik}\delta_{ac}
\end{eqnarray*}

In the other sense, the computation is similar, as follows:
\begin{eqnarray*}
\sum_{jb}(v\otimes w)^*_{jb,ia}(v\otimes w)_{jb,kc}
&=&\sum_{jb}w_{ba}^*v_{ji}^*v_{jk}w_{bc}\\
&=&\delta_{ik}\sum_bw_{ba}^*w_{bc}\\
&=&\delta_{ik}\delta_{ac}
\end{eqnarray*}

The comultiplicativity condition follows from the following computation: 
\begin{eqnarray*}
\Delta((v\otimes w)_{ia,jb})
&=&\sum_{kc}v_{ik}w_{ac}\otimes v_{kj}w_{cb}\\
&=&\sum_{kc}(v\otimes w)_{ia,kc}\otimes(v\otimes w)_{kc,jb}
\end{eqnarray*}

The proof of the counitality condition is similar, as follows:
$$\varepsilon((v\otimes w)_{ia,jb})
=\delta_{ij}\delta_{ab}
=\delta_{ia,jb}$$

As for the condition involving the antipode, this can be checked as follows:
$$S((v\otimes w)_{ia,jb})
=w_{ba}^*v_{ji}^*
=(v\otimes w)_{jb,ia}^*$$

(3) In order to check that $\bar{v}$ is unitary, we can use the antipode, exactly as we did in chapter 2 above, for $\bar{u}$. As for the comultiplicativity axioms, these are all clear.
\end{proof}

We have as well the following supplementary operation:

\index{spinned corepresentation}

\begin{proposition}
Given a corepresentation $v\in M_n(A)$, its spinned version
$$w= UvU^*$$
is a corepresentation as well, for any unitary matrix $U\in U_n$.
\end{proposition}

\begin{proof}
The matrix $w$ is unitary, and its comultiplicativity properties can be checked by doing some computations. Here is however another proof of this fact, using a useful trick. In the context of Definition 3.1, if we write $v\in M_n(\mathbb C)\otimes A$, the axioms read:
$$(id\otimes\Delta)v=v_{12}v_{13}$$
$$(id\otimes\varepsilon)v=1$$
$$(id\otimes S)v=v^*$$

Here we use standard tensor calculus conventions. Now when spinning by a unitary the matrix that we obtain, with these conventions, is $w=U_1vU_1^*$, and we have:
\begin{eqnarray*}
(id\otimes\Delta)w
&=&U_1v_{12}v_{13}U_1^*\\
&=&U_1v_{12}U_1^*\cdot U_1v_{13}U_1^*\\
&=&w_{12}w_{13}
\end{eqnarray*}

The proof of the counitality condition is similar, as follows:
$$(id\otimes\varepsilon)w
=U\cdot 1\cdot U
=1$$

Finally, the last condition, involving the antipode, can be checked as follows:
$$(id\otimes S)w
=U_1v^*U_1^*
=w^*$$

Thus, with usual notations, $w=UvU^*$ is a corepresentation, as claimed.
\end{proof}

As a philosophical comment here, the above proof might suggest that the more abstract our notations and formalism, and our methods in order to deal with mathematical questions, the easier our problems will become. But this is wrong. Bases and indices are a blessing: they can be understood by undergraduate students, computers, fellow scientists, engineers, and of course also by yourself, when you're tired or so. 

\bigskip

In addition, in the quantum group context, we will see later on, starting from chapter 4 below, that bases and indices can be turned into something very beautiful and powerful, allowing us to do some serious theory, well beyond the level of abstractions.

\bigskip

Back to work now, in the group dual case, we have the following result:

\index{group algebra}
\index{group dual}

\begin{proposition}
Assume $A=C^*(\Gamma)$, with $\Gamma=<g_1,\ldots,g_N>$ being a discrete group.
\begin{enumerate}
\item Any group element $h\in\Gamma$ is a $1$-dimensional corepresentation of $A$, and the operations on corepresentations are the usual ones on group elements.

\item Any diagonal matrix of type $v=diag(h_1,\ldots,h_n)$, with $n\in\mathbb N$ arbitrary, and with $h_1,\ldots,h_n\in\Gamma$, is a corepresentation of $A$.

\item More generally, any matrix of type $w=Udiag(h_1,\ldots,h_n)U^*$ with $h_1,\ldots,h_n\in\Gamma$ and with $U\in U_n$, is a corepresentation of $A$.
\end{enumerate}
\end{proposition}

\begin{proof}
These assertions are all elementary, as follows:

\medskip

(1) The first assertion is clear from definitions and from the comultiplication, counit and antipode formulae for the discrete group algebras, namely:
$$\Delta(h)=h\otimes h$$
$$\varepsilon(h)=1$$
$$S(h)=h^{-1}$$

The assertion on the operations is clear too, because we have:
$$(g)\otimes(h)=(gh)$$ 
$$\overline{(g)}=(g^{-1})$$

(2) This follows from (1) by performing sums, as in Proposition 3.3 above.

\medskip

(3) This follows from (2) and from the fact that we can conjugate any corepresentation by a unitary matrix, as explained in Proposition 3.4 above.
\end{proof}

Observe that the class of corepresentations in (3) is stable under all the operations from Propositions 3.3, and under the spinning operation from Proposition 3.4 too. When $\Gamma$ is abelian we can apply Proposition 3.2 with $G=\widehat{\Gamma}$, and after performing a number of identifications, we conclude that these are all the corepresentations of $C^*(\Gamma)$. 

\bigskip

We will see later that this latter fact holds in fact for any discrete group $\Gamma$. To be more precise, this is something non-trivial, which will follow from Peter-Weyl theory.

\bigskip

Summarizing, the representations of a compact quantum group can be defined as in the classical case, but by using coefficients, and in the group dual case we obtain something which is a priori quite simple too, namely formal direct sums of group elements.

\section*{3b. Peter-Weyl theory}

In the remainder of this chapter we develop the Peter-Weyl theory for the representations of the compact quantum groups, following the paper of Woronowicz \cite{wo1}. There is quite some work to be done here, and we will do it in two parts, first with some basic algebraic results, which are quite elementary, and then with more advanced results, mixing algebra and analysis. Let us start with the following definition:

\index{intertwiner}
\index{Hom space}
\index{End space}
\index{Fix space}

\begin{definition}
Given two corepresentations $v\in M_n(A),w\in M_m(A)$, we set 
$$Hom(v,w)=\left\{T\in M_{m\times n}(\mathbb C)\Big|Tv=wT\right\}$$
and we use the following conventions:
\begin{enumerate}
\item We use the notations $Fix(v)=Hom(1,v)$, and $End(v)=Hom(v,v)$.

\item We write $v\sim w$ when $Hom(v,w)$ contains an invertible element.

\item We say that $v$ is irreducible, and write $v\in Irr(G)$, when $End(v)=\mathbb C1$.
\end{enumerate}
\end{definition}

In the classical case $A=C(G)$ we obtain the usual notions concerning the representations. Observe also that in the group dual case we have:
$$g\sim h\iff g=h$$

Finally, observe that $v\sim w$ means that $v,w$ are conjugated by an invertible matrix. Here are a few basic results, regarding the above Hom spaces:

\index{tensor category}

\begin{proposition}
We have the following results:
\begin{enumerate}
\item $T\in Hom(u,v),S\in Hom(v,w)\implies ST\in Hom(u,w)$.

\item $S\in Hom(p,q),T\in Hom(v,w)\implies S\otimes T\in Hom(p\otimes v,q\otimes w)$.

\item $T\in Hom(v,w)\implies T^*\in Hom(w,v)$.
\end{enumerate}
In other words, the Hom spaces form a tensor $*$-category.
\end{proposition}

\begin{proof}
These assertions are all elementary, as follows:

\medskip

(1) By using our assumptions $Tu=vT$ and $Sv=Ws$ we obtain, as desired:
$$STu
=SvT
=wST$$

(2) Assume indeed that we have $Sp=qS$ and $Tv=wT$. With tensor product notations, as in the proof of Proposition 3.4 above, we have:
\begin{eqnarray*}
(S\otimes T)(p\otimes v)
&=&S_1T_2p_{13}v_{23}\\
&=&(Sp)_{13}(Tv)_{23}
\end{eqnarray*}

On the other hand, we have as well the following computation:
\begin{eqnarray*}
(q\otimes w)(S\otimes T)
&=&q_{13}w_{23}S_1T_2\\
&=&(qS)_{13}(wT)_{23}
\end{eqnarray*}

The quantities on the right being equal, this gives the result.

\medskip

(3) By conjugating, and then using the unitarity of $v,w$, we obtain, as desired:
\begin{eqnarray*}
Tv=wT
&\implies&v^*T^*=T^*w^*\\
&\implies&vv^*T^*w=vT^*w^*w\\
&\implies&T^*w=vT^*
\end{eqnarray*}

Finally, the last assertion follows from definitions, and from the obvious fact that, in addition to (1,2,3) above, the Hom spaces are linear spaces, and contain the units. In short, this is just a theoretical remark, that will be used only later on.
\end{proof}

As a main consequence, the spaces $End(v)\subset M_n(\mathbb C)$ are subalgebras which are stable under $*$, and so are $C^*$-algebras. In order to exploit this fact, we will need a basic result, complementing the operator algebra theory presented in chapter 1 above, namely:

\index{finite dimensional algebra}

\begin{theorem}
Let $B\subset M_n(\mathbb C)$ be a $C^*$-algebra.
\begin{enumerate}
\item We can write $1=p_1+\ldots+p_k$, with $p_i\in B$ central minimal projections.

\item Each of the linear spaces $B_i=p_iBp_i$ is a non-unital $*$-subalgebra of $B$.

\item We have a non-unital $*$-algebra sum decomposition $B=B_1\oplus\ldots\oplus B_k$.

\item We have unital $*$-algebra isomorphisms $B_i\simeq M_{r_i}(\mathbb C)$, where $r_i=rank(p_i)$.

\item Thus, we have a $C^*$-algebra isomorphism $B\simeq M_{r_1}(\mathbb C)\oplus\ldots\oplus M_{r_k}(\mathbb C)$.
\end{enumerate}
In addition, the final conclusion holds for any finite dimensional $C^*$-algebra.
\end{theorem}

\begin{proof}
This is something well-known, with the proof of the various assertions in the statement being something elementary, and routine:

\medskip

(1) This is more of a definition.

\medskip

(2) This is elementary, coming from $p_i^2=p_i=p_i^*$.

\medskip

(3) The verification of the direct sum conditions is indeed elementary.

\medskip

(4) This follows from the fact that each $p_i$ was assumed to be central and minimal.

\medskip

(5) This follows by putting everything together.

\medskip

As for the last assertion, this follows from (5) by using the GNS representation theorem, which provides us with an embedding $B\subset M_n(\mathbb C)$, for some $n\in\mathbb N$.
\end{proof}

Following Woronowicz's paper \cite{wo1}, we can now formulate a first Peter-Weyl theorem, and to be more precise a first such theorem from a 4-series, as follows:

\index{Peter-Weyl theorem}

\begin{theorem}[PW1]
Let $v\in M_n(A)$ be a corepresentation, consider the $C^*$-algebra $B=End(v)$, and write its unit as $1=p_1+\ldots+p_k$, as above. We have then
$$v=v_1+\ldots+v_k$$
with each $v_i$ being an irreducible corepresentation, obtained by restricting $v$ to $Im(p_i)$.
\end{theorem}

\begin{proof}
This is something very classical, well-known to hold for the compact groups, and the proof in general can be deduced from Theorem 3.8, as follows:

\medskip

(1) We first associate to our corepresentation $v\in M_n(A)$ the corresponding coaction map $\Phi:\mathbb C^n\to\mathbb C^n\otimes A$, given by the following formula:
$$\Phi(e_i)=\sum_je_j\otimes v_{ji}$$

We say that a linear subspace $V\subset\mathbb C^n$ is invariant under $v$ if:
$$\Phi(V)\subset V\otimes A$$

In this case, we can consider the following restriction map:
$$\Phi_{|V}:V\to V\otimes A$$

This is a coaction map too, which must come from a subcorepresentation $w\subset v$.

\medskip

(2) Consider now a projection $p\in End(v)$. From $pv=vp$ we obtain that the linear space $V=Im(p)$ is invariant under $v$, and so this space must come from a subcorepresentation $w\subset v$. It is routine to check that the operation $p\to w$ maps subprojections to subcorepresentations, and minimal projections to irreducible corepresentations.

\medskip

(3) With these preliminaries in hand, let us decompose the algebra $End(v)$ as in Theorem 3.8, by using the decomposition of 1 into minimal projections there:
$$1=p_1+\ldots+p_k$$

Consider now the following vector spaces, obtained as images of these projections:
$$V_i=Im(p_i)$$

If we denote by $v_i\subset v$ the subcorepresentations coming from these vector spaces, then we obtain in this way a decomposition $v=v_1+\ldots+v_k$, as in the statement.
\end{proof}

In order to formulate our second Peter-Weyl type theorem, we will need:

\index{Peter-Weyl representation}
\index{Peter-Weyl corepresentation}

\begin{definition}
We denote by $u^{\otimes k}$, with $k=\circ\bullet\bullet\circ\ldots$ being a colored integer, the various tensor products between $u,\bar{u}$, indexed according to the rules 
$$u^{\otimes\emptyset}=1\quad,\quad 
u^{\otimes\circ}=u\quad,\quad 
u^{\otimes\bullet}=\bar{u}$$
and multiplicativity, $u^{\otimes kl}=u^{\otimes k}\otimes u^{\otimes l}$, and call them Peter-Weyl corepresentations. 
\end{definition}

Here are a few examples of such corepresentations, namely those coming from the colored integers of length 2, to be often used in what follows:
$$u^{\otimes\circ\circ}=u\otimes u\quad,\quad 
u^{\otimes\circ\bullet}=u\otimes\bar{u}$$
$$u^{\otimes\bullet\circ}=\bar{u}\otimes u\quad,\quad 
u^{\otimes\bullet\bullet}=\bar{u}\otimes\bar{u}$$

There are several particular cases of interest of the above construction, where some considerable simplifications appear, as follows:

\begin{proposition}
The Peter-Weyl corepresentations $u^{\otimes k}$ are as follows:
\begin{enumerate}
\item In the real case, $u=\bar{u}$, we can assume $k\in\mathbb N$.

\item In the classical case, we can assume, up to equivalence, $k\in\mathbb N\times\mathbb N$.
\end{enumerate}
\end{proposition}

\begin{proof}
These assertions are both elementary, as follows:

\medskip

(1) Here we have indeed $u^{\otimes k}=u^{\otimes|k|}$, where $|k|\in\mathbb N$ is the length. Thus the Peter-Weyl corepresentations are indexed by $\mathbb N$, as claimed.

\medskip

(2) In the classical case, our claim is that we have equivalences $v\otimes w\sim w\otimes v$, implemented by the flip operator $\Sigma(a\otimes b)=b\otimes a$. Indeed, we have:
\begin{eqnarray*}
v\otimes w
&=&v_{13}w_{23}\\
&=&w_{23}v_{13}\\
&=&\Sigma w_{13}v_{23}\Sigma\\
&=&\Sigma(w\otimes v)\Sigma
\end{eqnarray*}

In particular we have an equivalence $u\otimes\bar{u}\sim\bar{u}\otimes u$. We conclude that the Peter-Weyl corepresentations are the corepresentations of type $u^{\otimes k}\otimes\bar{u}^{\otimes l}$, with $k,l\in\mathbb N$. 
\end{proof}

Here is now our second Peter-Weyl theorem, from a series of a total 4 Peter-Weyl theorems, also from Woronowicz \cite{wo1}, complementing Theorem 3.9 above:

\index{Peter-Weyl theorem}

\begin{theorem}[PW2]
Each irreducible corepresentation of $A$ appears as:
$$v\subset u^{\otimes k}$$
That is, $v$ appears inside a certain Peter-Weyl corepresentation.
\end{theorem}

\begin{proof}
Given an arbitrary corepresentation $v\in M_n(A)$, consider its space of coefficients, $C(v)=span(v_{ij})$. It is routine to check that the construction $v\to C(v)$ is functorial, in the sense that it maps subcorepresentations into subspaces.

By definition of the Peter-Weyl corepresentations, we have:
$$\mathcal A=\sum_{k\in\mathbb N*\mathbb N}C(u^{\otimes k})$$

Now given a corepresentation $v\in M_n(A)$, the corresponding coefficient space is a finite dimensional subspace $C(v)\subset\mathcal A$, and so we must have, for certain $k_1,\ldots,k_p$:
$$C(v)\subset C(u^{\otimes k_1}\oplus\ldots\oplus u^{\otimes k_p})$$

We deduce from this that we have an inclusion of corepresentations, as follows:
$$v\subset u^{\otimes k_1}\oplus\ldots\oplus u^{\otimes k_p}$$

Together with Theorem 3.9, this leads to the conclusion in the statement.
\end{proof}

\section*{3c. The Haar measure}

In order to further advance, with some finer results, we need to integrate over $G$. In the classical case the existence of such an integration is well-known, as follows:

\index{Haar measure}
\index{Haar integration}
\index{Ces\`aro limit}

\begin{proposition}
Any commutative Woronowicz algebra, $A=C(G)$ with $G\subset U_N$, has a unique faithful positive unital linear form $\int_G:A\to\mathbb C$ satisfying
$$\int_Gf(xy)dx=\int_Gf(yx)dx=\int_Gf(x)dx$$ 
called Haar integration. This Haar integration functional can be constructed by starting with any faithful positive unital form $\varphi\in A^*$, and taking the Ces\`aro limit
$$\int_G=\lim_{n\to\infty}\frac{1}{n}\sum_{k=1}^n\varphi^{*k}$$
where the convolution operation for linear forms is given by $\phi*\psi=(\phi\otimes\psi)\Delta$.
\end{proposition}

\begin{proof}
This is the existence theorem for the Haar measure of $G$, in functional analytic formulation. Observe first that the invariance conditions in the statement read:
$$d(xy)=d(yx)=dx\quad,\quad\forall y\in G$$

Thus, we are looking indeed for the integration with respect to the Haar measure on $G$. Now recall that this Haar measure exists, is unique, and can be constructed by starting with any probability measure $\mu$, and performing the following Ces\`aro limit:
$$dx=\lim_{n\to\infty}\frac{1}{n}\sum_{k=1}^nd\mu^{*k}(x)$$

In functional analysis terms, this corresponds precisely to the second assertion.
\end{proof}

The above statement and proof, which are quite brief, are of course more of a reminder, with all the technical details missing. However, we will reprove all this later on, as a particular case of a general Haar integration existence result, in the general Woronowicz algebra setting. In general now, let us start with a definition, as follows:

\begin{definition}
Given an arbitrary Woronowicz algebra $A=C(G)$, any positive unital tracial state $\int_G:A\to\mathbb C$ subject to the invariance conditions
$$\left(\int_G\otimes id\right)\Delta=\left(id\otimes\int_G\right)\Delta=\int_G(.)1$$
is called Haar integration over $G$.
\end{definition}

As a first observation, in the commutative case, this notion agrees with the one in Proposition 3.13. To be more precise, Proposition 3.13 tells us that any commutative Woronowicz algebra has a Haar integration in the above sense, which is unique, and which can be constructed by performing the Ces\`aro limiting procedure there. 

\bigskip

Before getting into the general case, let us discuss the group dual case. Here things are quite elementary, and we have the following result:

\begin{proposition}
Given a discrete group $\Gamma=<g_1,\ldots,g_N>$, the Woronowicz algebra $A=C^*(\Gamma)$ has a Haar functional, given on the standard generators $g\in\Gamma$ by:
$$\int_{\widehat{\Gamma}}g=\delta_{g,1}$$
This functional is faithful on the image on $C^*(\Gamma)$ in the regular representation. Also, in the abelian case, we obtain in this way the counit of $C(\widehat{\Gamma})$.
\end{proposition}

\begin{proof}
Consider indeed the left regular representation $\pi:C^*(\Gamma)\to B(l^2(\Gamma))$, given by $\pi(g)(h)=gh$, that we already met in chapter 1. By composing it with the functional $T\to<T1,1>$, the functional $\int_{\widehat{\Gamma}}$ that we obtain is given by:
$$\int_{\widehat{\Gamma}}g=<g1,1>=\delta_{g,1}$$

But this gives all the assertions in the statement, namely the existence, traciality, left and right invariance properties, and faithfulness on the reduced algebra. As for the last assertion, this is clear from the Pontrjagin duality isomorphism.
\end{proof}

With a bit of functional analysis knowledge, one can improve the above result, with a proof of the fact that the Haar integration is unique, and appears via a Ces\`aro limiting procedure, as in Proposition 3.13. We will do this directly, in the general case.

\bigskip

In order to discuss now the general case, that of the arbitrary Woronowicz algebras, let us define the convolution operation for linear forms by:
$$\phi*\psi=(\phi\otimes\psi)\Delta$$

We have then the following technical result, from Woronowicz's paper \cite{wo1}:

\begin{proposition}
Given an arbitrary unital linear form $\varphi\in A^*$, the limit
$$\int_\varphi a=\lim_{n\to\infty}\frac{1}{n}\sum_{k=1}^n\varphi^{*k}(a)$$
exists, and for a coefficient of a corepresentation $a=(\tau\otimes id)v$, we have
$$\int_\varphi a=\tau(P)$$
where $P$ is the orthogonal projection onto the $1$-eigenspace of $(id\otimes\varphi)v$.
\end{proposition}

\begin{proof}
By linearity, it is enough to prove the first assertion for elements of the following type, where $v$ is a Peter-Weyl corepresentation, and $\tau$ is a linear form:
$$a=(\tau\otimes id)v$$

Thus we are led into the second assertion, and more precisely we can have the whole result proved if we can establish the following formula, with $a=(\tau\otimes id)v$:
$$\lim_{n\to\infty}\frac{1}{n}\sum_{k=1}^n\varphi^{*k}(a)=\tau(P)$$

In order to prove this latter formula, observe that we have:
$$\varphi^{*k}(a)
=(\tau\otimes\varphi^{*k})v
=\tau((id\otimes\varphi^{*k})v)$$

Consider now the following matrix, which is a usual complex matrix:
$$M=(id\otimes\varphi)v$$

In terms of this matrix, we have the following formula:
\begin{eqnarray*}
((id\otimes\varphi^{*k})v)_{i_0i_{k+1}}
&=&\sum_{i_1\ldots i_k}M_{i_0i_1}\ldots M_{i_ki_{k+1}}\\
&=&(M^k)_{i_0i_{k+1}}
\end{eqnarray*}

Thus for any $k\in\mathbb N$ we have the following formula:
$$(id\otimes\varphi^{*k})v=M^k$$

It follows that our Ces\`aro limit is given by the following formula:
\begin{eqnarray*}
\lim_{n\to\infty}\frac{1}{n}\sum_{k=1}^n\varphi^{*k}(a)
&=&\lim_{n\to\infty}\frac{1}{n}\sum_{k=1}^n\tau(M^k)\\
&=&\tau\left(\lim_{n\to\infty}\frac{1}{n}\sum_{k=1}^nM^k\right)
\end{eqnarray*}

Now since $v$ is unitary we have $||v||=1$, and we conclude that we have:
$$||M||\leq1$$

Thus, by standard calculus, the above Ces\`aro limit on the right exists, and equals the orthogonal projection onto the $1$-eigenspace of $M$:
$$\lim_{n\to\infty}\frac{1}{n}\sum_{k=1}^nM^k=P$$

Thus our initial Ces\`aro limit converges as well, to $\tau(P)$, as desired. 
\end{proof}

When $\varphi$ is faithful, we have the following finer result, also from Woronowicz \cite{wo1}:

\begin{proposition}
Given a faithful unital linear form $\varphi\in A^*$, the limit
$$\int_\varphi a=\lim_{n\to\infty}\frac{1}{n}\sum_{k=1}^n\varphi^{*k}(a)$$
exists, and is independent of $\varphi$, given on coefficients of corepresentations by
$$\left(id\otimes\int_\varphi\right)v=P$$
where $P$ is the orthogonal projection onto $Fix(v)=\{\xi\in\mathbb C^n|v\xi=\xi\}$.
\end{proposition}

\begin{proof}
In view of Proposition 3.16, it remains to prove that when $\varphi$ is faithful, the $1$-eigenspace of $M=(id\otimes\varphi)v$ equals $Fix(v)$.

\medskip

``$\supset$'' This is clear, and for any $\varphi$, because we have:
$$v\xi=\xi\implies M\xi=\xi$$

``$\subset$'' Here we must prove that, when $\varphi$ is faithful, we have:
$$M\xi=\xi\implies v\xi=\xi$$

For this purpose, we use a positivity trick. Consider the following element:
$$a=\sum_i\left(\sum_jv_{ij}\xi_j-\xi_i\right)\left(\sum_kv_{ik}\xi_k-\xi_i\right)^*$$

We want to prove that we have $a=0$. Since $v$ is biunitary, we have:
\begin{eqnarray*}
a
&=&\sum_i\left(\sum_j\left(v_{ij}\xi_j-\frac{1}{N}\xi_i\right)\right)\left(\sum_k\left(v_{ik}^*\bar{\xi}_k-\frac{1}{N}\bar{\xi}_i\right)\right)\\
&=&\sum_{ijk}v_{ij}v_{ik}^*\xi_j\bar{\xi}_k-\frac{1}{N}v_{ij}\xi_j\bar{\xi}_i-\frac{1}{N}v_{ik}^*\xi_i\bar{\xi}_k+\frac{1}{N^2}\xi_i\bar{\xi}_i\\
&=&\sum_j|\xi_j|^2-\sum_{ij}v_{ij}\xi_j\bar{\xi}_i-\sum_{ik}v_{ik}^*\xi_i\bar{\xi}_k+\sum_i|\xi_i|^2\\
&=&||\xi||^2-<v\xi,\xi>-\overline{<v\xi,\xi>}+||\xi||^2\\
&=&2(||\xi||^2-Re(<v\xi,\xi>))
\end{eqnarray*}

By using now our assumption $M\xi=\xi$, we obtain from this:
\begin{eqnarray*}
\varphi(a)
&=&2\varphi(||\xi||^2-Re(<v\xi,\xi>))\\
&=&2(||\xi||^2-Re(<M\xi,\xi>))\\
&=&2(||\xi||^2-||\xi||^2)\\
&=&0
\end{eqnarray*}

Thus $a=0$, and by positivity we obtain $v\xi=\xi$, as desired.
\end{proof}

We can now formulate the general Haar measure result, due to Woronowicz \cite{wo1}:

\index{Haar measure}
\index{Haar integration}
\index{Ces\`aro limit}

\begin{theorem}
Any Woronowicz algebra has a unique Haar integration, which can be constructed by starting with any faithful positive unital state $\varphi\in A^*$, and setting
$$\int_G=\lim_{n\to\infty}\frac{1}{n}\sum_{k=1}^n\varphi^{*k}$$
where $\phi*\psi=(\phi\otimes\psi)\Delta$. Moreover, for any corepresentation $v$ we have
$$\left(id\otimes\int_G\right)v=P$$
where $P$ is the orthogonal projection onto $Fix(v)=\{\xi\in\mathbb C^n|v\xi=\xi\}$.
\end{theorem}

\begin{proof}
Let us first go back to the general context of Proposition 3.16 above. Since convolving one more time with $\varphi$ will not change the Ces\`aro limit appearing there, the functional $\int_\varphi\in A^*$ constructed there has the following invariance property:
$$\int_\varphi*\varphi=\varphi*\int_\varphi=\int_\varphi$$

In the case where $\varphi$ is assumed to be faithful, as in Proposition 3.17 above, our claim is that we have the following formula, valid this time for any $\psi\in A^*$:
$$\int_\varphi*\psi=\psi*\int_\varphi=\psi(1)\int_\varphi$$

It is enough to prove this formula on a coefficient of a corepresentation, $a=(\tau\otimes id)v$. In order to do so, consider the following matrices:
$$P=\left(id\otimes\int_\varphi\right)v\quad,\quad 
Q=(id\otimes\psi)v$$

In terms of these matrices, we have:
$$\left(\int_\varphi*\psi\right)a
=\left(\tau\otimes\int_\varphi\otimes\psi\right)(v_{12}v_{13})
=\tau(PQ)$$

Similarly, we have the following computation:
$$\left(\psi*\int_\varphi\right)a
=\left(\tau\otimes\psi\otimes\int_\varphi\right)(v_{12}v_{13})
=\tau(QP)$$

Finally, regarding the term on the right, this is given by:
$$\psi(1)\int_\varphi a=\psi(1)\tau(P)$$

Thus, our claim is equivalent to the following equality:
$$PQ=QP=\psi(1)P$$

But this latter equality follows from the fact, coming from Proposition 3.17 above, that $P=(id\otimes\int_\varphi)v$ equals the orthogonal projection onto $Fix(v)$. Thus, we have proved our claim. Now observe that our formula can be written as:
$$\psi\left(\int_\varphi\otimes id\right)\Delta=\psi\left(id\otimes\int_\varphi\right)\Delta=\psi\int_\varphi(.)1$$

This formula being true for any $\psi\in A^*$, we can simply delete $\psi$, and we conclude that the invariance formula in Definition 3.14 holds indeed, with $\int_G=\int_\varphi$. Finally, assuming that we have two invariant integrals $\int_G,\int_G'$, we have:
\begin{eqnarray*}
\left(\int_G\otimes\int_G'\right)\Delta
&=&\left(\int_G'\otimes\int_G\right)\Delta\\
&=&\int_G(.)1\\
&=&\int_G'(.)1
\end{eqnarray*}

Thus we have $\int_G=\int_G'$, and this finishes the proof.
\end{proof}

As a first observation, in the case of the classical groups, and of the group duals, we recover the various Haar measure results mentioned before. As another illustration, for the basic product operations, we have the following result, due to Wang \cite{wa1}:

\index{product of quantum groups}
\index{dual free product}
\index{quotient quantum group}
\index{projective version}

\begin{proposition}
We have the following results:
\begin{enumerate}
\item For a product $G\times H$, we have $\int_{G\times H}=\int_G\otimes\int_H$.

\item For a dual free product $G\,\hat{*}\,H$, we have $\int_{G\,\hat{*}\,H}=\int_G*\int_H$.

\item For a quotient $G\to H$, we have $\int_H=\left(\int_G\right)_{|C(H)}$.

\item For a projective version $G\to PG$, we have $\int_{PG}=\left(\int_G\right)_{|C(PG)}$.
\end{enumerate}
\end{proposition}

\begin{proof}
These formulae all follow from the invariance property, as follows:

\medskip

(1) Here the tensor product form $\int_G\otimes\int_H$ satisfies the left and right invariance properties of the Haar functional $\int_{G\times H}$, and so by uniqueness, it is equal to it.

\medskip

(2) Here the situation is similar, with the free product of linear forms being defined with some inspiration from the discrete group case, where $\int_{\widehat{\Gamma}}g=\delta_{g,1}$.

\medskip

(3) Here the restriction $\left(\int_G\right)_{|C(H)}$ satisfies by definition the required left and right invariance properties, so once again we can conclude by uniqueness.

\medskip

(4) Here we simply have a particular case of (3) above.
\end{proof}

In practice, the last assertion in Theorem 3.18 is the most useful one. By applying it to the Peter-Weyl corepresentations, we obtain the following alternative statement:

\index{Weingarten formula}
\index{Weingarten integration}

\begin{theorem}
The Haar integration of a Woronowicz algebra is given, on the coefficients of the Peter-Weyl corepresentations, by the Weingarten formula
$$\int_Gu_{i_1j_1}^{e_1}\ldots u_{i_kj_k}^{e_k}=\sum_{\pi,\sigma\in D_k}\delta_\pi(i)\delta_\sigma(j)W_k(\pi,\sigma)$$
valid for any colored integer $k=e_1\ldots e_k$ and any multi-indices $i,j$, where:
\begin{enumerate}
\item  $D_k$ is a linear basis of $Fix(u^{\otimes k})$.

\item $\delta_\pi(i)=<\pi,e_{i_1}\otimes\ldots\otimes e_{i_k}>$. 

\item $W_k=G_k^{-1}$, with $G_k(\pi,\sigma)=<\pi,\sigma>$.
\end{enumerate}
\end{theorem}

\begin{proof}
As a first observation, the above formula computes indeed the Haar integral, because the coefficients of the Peter-Weyl corepresentations span a dense subalgebra:
$$A=\overline{span\left(u_{i_1j_1}^{e_1}\ldots u_{i_kj_k}^{e_k}\Big| e,i,j,k\right)}$$

Regarding now the proof, we know from Theorem 3.18 that the integrals in the statement form altogether the orthogonal projection $P$ onto the following space:
$$Fix(u^{\otimes k})=span(D_k)$$

Consider now the following linear map:
$$E(x)=\sum_{\pi\in D_k}<x,\pi>\pi$$

By a standard linear algebra computation, it follows that we have $P=WE$, where $W$ is the inverse on $span(D_k)$ of the restriction of $E$. But this restriction is the linear map given by $G_k$, and so $W$ is the linear map given by $W_k$, and this gives the result.
\end{proof}

We will be back to the above two Haar measure theorems, which are both fundamental, with versions, illustrations and applications, on several occasions, later on.

\section*{3d. More Peter-Weyl}

Let us go back now to algebra, and establish two more Peter-Weyl theorems. We will need the following result, which is very useful, and is of independent interest:

\index{Frobenius isomorphism}

\begin{theorem}
We have a Frobenius type isomorphism
$$Hom(v,w)\simeq Fix(\bar{v}\otimes w)$$
valid for any two corepresentations $v,w$.
\end{theorem}

\begin{proof}
According to the definitions, we have the following equivalence:
\begin{eqnarray*}
T\in Hom(v,w)
&\iff&Tv=wT\\
&\iff&\sum_jT_{aj}v_{ji}=\sum_bw_{ab}T_{bi}
\end{eqnarray*}

On the other hand, we have as well the following equivalence:
\begin{eqnarray*}
T\in Fix(\bar{v}\otimes w)
&\iff&(\bar{v}\otimes w)T=T\\
&\iff&\sum_{kb}v_{ik}^*w_{ab}T_{bk}=T_{ai}
\end{eqnarray*}

With these formulae in hand, we must prove that we have:
$$\sum_jT_{aj}v_{ji}=\sum_bw_{ab}T_{bi}\iff \sum_{kb}v_{ik}^*w_{ab}T_{bk}=T_{ai}$$

(1) In one sense, the computation is as follows, using the unitarity of $v^t$:
\begin{eqnarray*}
\sum_{kb}v_{ik}^*w_{ab}T_{bk}
&=&\sum_kv_{ik}^*\sum_bw_{ab}T_{bk}\\
&=&\sum_kv_{ik}^*\sum_jT_{aj}v_{jk}\\
&=&\sum_j(\bar{v}v^t)_{ij}T_{aj}\\
&=&T_{ai}
\end{eqnarray*}

(2) In the other sense we have, once again by using the unitarity of $v^t$:
\begin{eqnarray*}
\sum_jT_{aj}v_{ji}
&=&\sum_jv_{ji}\sum_{kb}v_{jk}^*w_{ab}T_{bk}\\
&=&\sum_{kb}(v^t\bar{v})_{ik}w_{ab}T_{bk}\\
&=&\sum_bw_{ab}T_{bi}
\end{eqnarray*}

Thus, we are led to the conclusion in the statement.
\end{proof}

With these ingredients, namely two Peter-Weyl theorems, Haar measure and Frobenius duality, we can establish a third Peter-Weyl theorem, also from Woronowicz \cite{wo1}:

\index{Peter-Weyl theorem}

\begin{theorem}[PW3]
The dense subalgebra $\mathcal A\subset A$ decomposes as a direct sum 
$$\mathcal A=\bigoplus_{v\in Irr(A)}M_{\dim(v)}(\mathbb C)$$
with this being an isomorphism of $*$-coalgebras, and with the summands being pairwise orthogonal with respect to the scalar product given by
$$<a,b>=\int_Gab^*$$
where $\int_G$ is the Haar integration over $G$.
\end{theorem}

\begin{proof}
By combining the previous Peter-Weyl results, from Theorem 3.9 and Theorem 3.12 above, we deduce that we have a linear space decomposition as follows:
$$\mathcal A
=\sum_{v\in Irr(A)}C(v)
=\sum_{v\in Irr(A)}M_{\dim(v)}(\mathbb C)$$

Thus, in order to conclude, it is enough to prove that for any two irreducible corepresentations $v,w\in Irr(A)$, the corresponding spaces of coefficients are orthogonal:
$$v\not\sim w\implies C(v)\perp C(w)$$ 

But this follows from Theorem 3.18, via Theorem 3.21. Let us set indeed:
$$P_{ia,jb}=\int_Gv_{ij}w_{ab}^*$$

Then $P$ is the orthogonal projection onto the following vector space:
$$Fix(v\otimes\bar{w})
\simeq Hom(\bar{v},\bar{w})
=\{0\}$$

Thus we have $P=0$, and this gives the result.
\end{proof}

We can obtain further results by using characters, which are defined as follows:

\index{character}

\begin{proposition}
The characters of the corepresentations, given by 
$$\chi_v=\sum_iv_{ii}$$ 
behave as follows, in respect to the various operations:
$$\chi_{v+w}=\chi_v+\chi_w$$
$$\chi_{v\otimes w}=\chi_v\chi_w$$ 
$$\chi_{\bar{v}}=\chi_v^*$$
In addition, given two equivalent corepresentations, $v\sim w$, we have $\chi_v=\chi_w$.
\end{proposition}

\begin{proof}
The three formulae in the statement are all clear from definitions. Regarding now the last assertion, assuming that we have $v=T^{-1}wT$, we obtain:
\begin{eqnarray*}
\chi_v
&=&Tr(v)\\
&=&Tr(T^{-1}wT)\\
&=&Tr(w)\\
&=&\chi_w
\end{eqnarray*}

We conclude that $v\sim w$ implies $\chi_v=\chi_w$, as claimed.
\end{proof}

We have the following more advanced result, regarding the characters, also from Woronowicz \cite{wo1}, completing the Peter-Weyl theory:

\index{Peter-Weyl theorem}
\index{central function}

\begin{theorem}[PW4]
The characters of the irreducible corepresentations belong to the $*$-algebra
$$\mathcal A_{central}=\left\{a\in\mathcal A\Big|\Sigma\Delta(a)=\Delta(a)\right\}$$
of ``smooth central functions'' on $G$, and form an orthonormal basis of it.
\end{theorem}

\begin{proof}
As a first remark, the linear space $\mathcal A_{central}$ defined above is indeed an algebra. In the classical case, we obtain in this way the usual algebra of smooth central functions. Also, in the group dual case, where we have $\Sigma\Delta=\Delta$, we obtain the whole convolution algebra. Regarding now the proof, in general, this goes as follows:

\medskip

(1) The algebra $\mathcal A_{central}$ contains indeed all the characters, because we have:
\begin{eqnarray*}
\Sigma\Delta(\chi_v)
&=&\Sigma\left(\sum_{ij}v_{ij}\otimes v_{ji}\right)\\
&=&\sum_{ij}v_{ji}\otimes v_{ij}\\
&=&\Delta(\chi_v)
\end{eqnarray*}

(2) Conversely, consider an element $a\in\mathcal A$, written as follows:
$$a=\sum_{v\in Irr(A)}a_v$$

The condition $a\in\mathcal A_{central}$ is then equivalent to the following conditions:
$$a_v\in\mathcal A_{central}\quad,\forall v\in Irr(A)$$

But each condition $a_v\in\mathcal A_{central}$ means that $a_v$ must be a scalar multiple of the corresponding character $\chi_v$, and so the characters form a basis of $\mathcal A_{central}$, as stated.

\medskip

(3) The fact that we have an orthogonal basis follows from Theorem 3.22. 

\medskip

(4) Finally, regarding the norm 1 assertion, consider the following integrals:
$$P_{ik,jl}=\int_Gv_{ij}v_{kl}^*$$

We know from Theorem 3.18 that these integrals form the orthogonal projection onto the following vector space, computed via Theorem 3.21:
$$Fix(v\otimes\bar{v})
\simeq End(\bar{v})
=\mathbb C1$$

By using this fact, we obtain the following formula:
\begin{eqnarray*}
\int_G\chi_v\chi_v^*
&=&\sum_{ij}\int_Gv_{ii}v_{jj}^*\\
&=&\sum_i\frac{1}{N}\\
&=&1
\end{eqnarray*}

Thus the characters have indeed norm 1, and we are done.
\end{proof}

As a first application of the Peter-Weyl theory, and more specifically of Theorem 3.24, we can now clarify a question that we left open in chapter 2, regarding the cocommutative case. To be more precise, once again following Woronowicz \cite{wo1}, we have:

\index{group algebra}
\index{cocommutative}
\index{cocommutative Woronowicz algebra}

\begin{theorem}
For a Woronowicz algebra $A$, the following are equivalent:
\begin{enumerate}
\item $A$ is cocommutative, $\Sigma\Delta=\Delta$.

\item The irreducible corepresentations of $A$ are all $1$-dimensional.

\item $A=C^*(\Gamma)$, for some group $\Gamma=<g_1,\ldots,g_N>$, up to equivalence.
\end{enumerate}
\end{theorem}

\begin{proof}
This follows from the Peter-Weyl theory, as follows:

\medskip

$(1)\implies(2)$ The assumption $\Sigma\Delta=\Delta$ tells us that the inclusion $\mathcal A_{central}\subset\mathcal A$ is an isomorphism, and by using Theorem 3.24 we conclude that any irreducible corepresentation of $A$ must be equal to its character, and so must be 1-dimensional.

\medskip

$(2)\implies(3)$ This follows once again from Peter-Weyl, because if we denote by $\Gamma$ the group formed by the 1-dimensional corepresentations, then we have $\mathcal A=\mathbb C[\Gamma]$, and so $A=C^*(\Gamma)$ up to the standard equivalence relation for Woronowicz algebras.

\medskip

$(3)\implies(1)$ This is something trivial, that we already know from chapter 2.
\end{proof}

The above result is not the end of the story, because one can still ask what happens, without reference to the equivalence relation. We will be back to this later.

\bigskip

At the level of the product operations, we have, following Wang \cite{wa1}:

\index{product of quantum groups}
\index{dual free product}
\index{quantum subgroup}
\index{projective version}

\begin{proposition}
We have the following results:
\begin{enumerate}
\item The irreducible corepresentations of $C(G\times H)$ are the tensor products of the form $v\otimes w$, with $v,w$ being irreducible corepresentations of $C(G),C(H)$.

\item The irreducible corepresentations of $C(G\,\hat{*}\,H)$ appear as alternating tensor products of irreducible corepresentations of $C(G)$ and of $C(H)$.

\item The irreducible corepresentations of $C(H)\subset C(G)$ are the irreducible corepresentations of $C(G)$ whose coefficients belong to $C(H)$.

\item The irreducible corepresentations of $C(PG)\subset C(G)$ are the irreducible corepresentations of $C(G)$ which appear by decomposing the tensor powers of $u\otimes\bar{u}$.
\end{enumerate}
\end{proposition}

\begin{proof}
This is something routine, the idea being as follows:

\medskip

(1) Here we can integrate characters, by using Proposition 3.19 (1), and we conclude that if $v,w$ are irreducible corepresentations of $C(G),C(H)$, then $v\otimes w$ is an irreducible corepresentation of $C(G\times H)$. Now since the coefficients of these latter corepresentations span $\mathcal C(G\times H)$, by Peter-Weyl these are all the irreducible corepresentations.

\medskip

(2) Here we can use a similar method. By using Proposition 3.19 (2) we conclude that if $v_1,v_2,\ldots$ are irreducible corepresentations of $C(G)$ and $w_1,w_2,\ldots$ are irreducible corepresentations of $C(H)$, then $v_1\otimes w_1\otimes v_2\otimes w_2\otimes\ldots$ is an irreducible corepresentation of $C(G\,\hat{*}\,H)$, and then we can conclude by using the Peter-Weyl theory.

\medskip

(3) This is clear from definitions, and from the Peter-Weyl theory.

\medskip

(4) This is a particular case of the result (3) above.
\end{proof}

Let us go back now to Theorem 3.25, and try to understand what happens in general, without reference to the equivalence relation. We know from chapter 1 that associated to any discrete group $\Gamma$ are at least two group algebras, which are as follows:
$$C^*(\Gamma)\to C^*_{red}(\Gamma)\subset B(l^2(\Gamma))$$

For the finite, or abelian, or more generally amenable groups $\Gamma$, these two algebras are known to coincide, but in the non-amenable case, the opposite happens. Thus, we are led into the question on whether $C^*_{red}(\Gamma)$, and other possible group algebras of $\Gamma$, are Woronowicz algebras in our sense, having morphisms as follows:
$$\Delta:A\to A\otimes A$$
$$\varepsilon:A\to\mathbb C$$
$$S:A\to A^{opp}$$

Generally speaking, the answer here is ``no'', and the subject is quite technical, requiring a good knowledge of advanced functional analysis. In order to have $C^*_{red}(\Gamma)$ among our examples, if we really want to, we must change a bit our axioms, as follows:

\index{minimal tensor product}
\index{spatial tensor product}

\begin{proposition}
Given a discrete group $\Gamma=<g_1,\ldots,g_N>$, its reduced algebra $A=C^*_{red}(\Gamma)$ has morphisms as follows, given on generators by the usual formulae,
$$\Delta:A\to A\otimes_{min}A$$
$$\varepsilon:\mathcal A\to\mathbb C$$
$$S:A\to A^{opp}$$
where $\otimes_{min}$ is the spatial tensor product of $C^*$-algebras, and where $\mathcal A=\mathbb C[\Gamma]$.
\end{proposition}

\begin{proof}
This is something quite technical, and philosophical, related to $*$-algebras vs $C^*$-algebras, and to $\otimes_{min}$ vs $\otimes_{max}$, that we will not really need in what follows. In what regards the comultiplication, consider the following diagonal embedding:
$$\Gamma\subset\Gamma\times\Gamma\quad,\quad g\to(g,g)$$

This embedding induces a $*$-algebra representation, as follows:
$$\mathbb C[\Gamma]\to B(l^2(\Gamma))\otimes_{min}B(l^2(\Gamma))\quad,\quad
g\to g\otimes g$$

We can extend then this representation into a morphism $\Delta$, as in the statement. As for the existence of morphisms $\varepsilon,S$ as in the statement, this is clear.
\end{proof}

Going ahead with some philosophy, the above result might suggest to modify our quantum group axioms, in a somewhat obvious way, with a densely defined counit, as to include algebras of type $C^*_{red}(\Gamma)$ in our formalism. But do we really want to do that. Remember, we are interested here in quantum spaces and quantum groups, which are well-defined up to equivalence, and so Theorem 3.25 above is all we need.

\bigskip

What does make sense, however, is to do such modifications in case you are interested in more general quantum groups, such as the Drinfeld-Jimbo deformations at $q>0$, which are not covered by our formalism, and this is what Woronowicz did in \cite{wo1}. But, as explained in chapter 2, from a modern perspective at least, these deformations with parameter $q>0$ have only theoretical interest, and we will surely not follow this way.

\bigskip 

Let us discuss now, however, in relation with all this, the notion of amenability, which is something important and useful. The basic result here, due to Blanchard \cite{bla}, once again requiring a good knowledge of functional analysis, is as follows:

\index{amenability}
\index{amenable quantum group}
\index{amenable Woronowicz algebra}
\index{Kesten amenability}

\begin{theorem}
Let $A_{full}$ be the enveloping $C^*$-algebra of $\mathcal A$, and let $A_{red}$ be the quotient of $A$ by the null ideal of the Haar integration. The following are then equivalent:
\begin{enumerate}
\item The Haar functional of $A_{full}$ is faithful.

\item The projection map $A_{full}\to A_{red}$ is an isomorphism.

\item The counit map $\varepsilon:A\to\mathbb C$ factorizes through $A_{red}$.

\item We have $N\in\sigma(Re(\chi_u))$, the spectrum being taken inside $A_{red}$.
\end{enumerate}
If this is the case, we say that the underlying discrete quantum group $\Gamma$ is amenable.
\end{theorem}

\begin{proof}
This is well-known in the group dual case, $A=C^*(\Gamma)$, with $\Gamma$ being a usual discrete group. In general, the result follows by adapting the group dual case proof:

\medskip

$(1)\implies(2)$ This follows from the fact that the GNS construction for the algebra $A_{full}$ with respect to the Haar functional produces the algebra $A_{red}$.

\medskip

$(2)\implies(3)$ This is trivial, because we have quotient maps $A_{full}\to A\to A_{red}$, and so our assumption $A_{full}=A_{red}$ implies that we have $A=A_{red}$. 

\medskip

$(3)\implies(4)$ This implication is clear too, because we have:
\begin{eqnarray*}
\varepsilon(Re(\chi_u))
&=&\frac{1}{2}\left(\sum_{i=1}^N\varepsilon(u_{ii})+\sum_{i=1}^N\varepsilon(u_{ii}^*)\right)\\
&=&\frac{1}{2}(N+N)\\
&=&N
\end{eqnarray*}

Thus the element $N-Re(\chi_u)$ is not invertible in $A_{red}$, as claimed. 

\medskip

$(4)\implies(1)$ In terms of the corepresentation $v=u+\bar{u}$, whose dimension is $2N$ and whose character is $2Re(\chi_u)$, our assumption $N\in\sigma(Re(\chi_u))$ reads:
$$\dim v\in\sigma(\chi_v)$$

By functional calculus the same must hold for $w=v+1$, and then once again by functional calculus, the same must hold for any tensor power of $w$:
$$w_k=w^{\otimes k}$$ 

Now choose for each $k\in\mathbb N$ a state $\varepsilon_k\in A_{red}^*$ having the following property:
$$\varepsilon_k(w_k)=\dim w_k$$

By Peter-Weyl we must have $\varepsilon_k(r)=\dim r$ for any $r\leq w_k$, and since any irreducible corepresentation appears in this way, the sequence $\varepsilon_k$ converges to a counit map: 
$$\varepsilon:A_{red}\to\mathbb C$$

In order to finish, we can use the right regular corepresentation. Indeed, we can define such a corepresentation by the following formula:
$$W(a\otimes x)=\Delta(a)(1\otimes x)$$

This corepresentation is unitary, so we can define a morphism as follows: 
$$\Delta':A_{red}\to A_{red}\otimes A_{full}$$
$$a\to W(a\otimes1)W^*$$

Now by composing with $\varepsilon\otimes id$, we obtain a morphism as follows:
$$(\varepsilon\otimes id)\Delta':A_{red}\to A_{full}$$
$$u_{ij}\to u_{ij}$$

Thus, we have our inverse map for the projection $A_{full}\to A_{red}$, as desired.
\end{proof}

All the above was of course quite short, but we will be back to this, with full details, and with a systematic study of the notion of amenability, in chapter 14 below. In particular, we will discuss in detail the case of the usual discrete group algebras $A=C^*(\Gamma)$, by further building on the findings in Theorem 3.25 and Proposition 3.27.

\bigskip

Here are now some basic applications of the above amenability result:

\index{coamenable quantum group}
\index{group dual}

\begin{proposition}
We have the following results:
\begin{enumerate}
\item The compact Lie groups $G\subset U_N$ are all coamenable.

\item A group dual $G=\widehat{\Gamma}$ is coamenable precisely when $\Gamma$ is amenable.

\item A product $G\times H$ of coamenable compact quantum groups is coamenable.
\end{enumerate}
\end{proposition}

\begin{proof}
This follows indeed from the results that we have:

\medskip

(1) This is clear by using any of the criteria in Theorem 3.28 above, because for an algebra of type $A=C(G)$, we have $A_{full}=A_{red}$.

\medskip

(2) Here the various criteria in Theorem 3.28 above correspond to the various equivalent definitions of the amenability of a discrete group.

\medskip

(3) This follows from the description of the Haar functional of $C(G\times H)$, from Proposition 3.19 (1) above. Indeed, if $\int_G,\int_H$ are both faithful, then so is $\int_G\otimes\int_H$.
\end{proof}

As already mentioned, we will be back to this, in chapter 14 below. But that is in a long time from now, so perhaps time for some philosophy, and advice, in relation with various functional analysis issues, including tensor products, and amenability:

\bigskip

(1) Sorry for having to start with this, but I'm sure that there might be a graduate student or postdoc around you, or perhaps even young researcher, struggling with tensor products and amenability, and telling you something of type ``operator algebras are all about tensor products and amenability, you ain't understand anything, you have to spend time and learn tensor products and amenability first, before anything else''.

\bigskip

(2) Which is utterly wrong, and believe me, senior researcher talking here. Operator algebras are about quantum mechanics, and more specifically about modern quantum mechanics, from the 1950s onwards. And here, what's needed are quantum groups and quantum spaces, axiomatized as we did it, via equivalence relation between the corresponding operator algebras, and with minimal fuss about tensor products and amenability.

\bigskip

(3) Of course tensor products and amenability will come into play, at some point. But later. That's advanced. More precisely, amenability comes into play, be that in modern mathematics, or modern physics, via Theorem 3.28 (4), called Kesten criterion, and some further work that can be done on that, involving random walks, spectral measures and so on. But that's advanced level, conformal field theory (CFT), or higher.

\bigskip

(4) So stay with me, and of course we'll talk about such things, in due time, meaning end of this book, chapters 13-16 below. And to that graduate student friend of yours, please tell him that he's on the good track of revolutionizing quantum mechanics from the 1920s. And more specifically, from the early 1920s. Which is of course a decent business, papers about old quantum mechanics being always welcome for publication.

\bigskip

Of course, if the advice about learning tensor products and amenability comes from your PhD advisor, guess we'll have to do that. Which is not a problem, that will take you 1 month or so, and you might learn some interesting things there. That you can effectively use later on, once you'll have the black belt in mathematics and physics.

\section*{3e. Exercises}

Generally speaking, the best complement to the material presented in this chapter is some further reading, either in the classical case, for the finite groups, or for the compact Lie groups, or in the quantum group case, say for the finite quantum groups. Indeed, in all these situations some interesting simplifications, worth knowing, might appear. In relation with what has been said above, here is a first exercise:

\begin{exercise}
Prove that the finite dimensional $C^*$-algebras are exactly the direct sums of matrix algebras
$$A=M_{N_1}(\mathbb C)\oplus\ldots\oplus M_{N_k}(\mathbb C)$$
by decomposing first the unit into a sum of central minimal projections.
\end{exercise}

This is self-explanatory, and we have talked about this in the above, the problem now being that of clarifying all this, by doing all the needed computations.

\begin{exercise}
Given a matrix $M\in M_N(\mathbb C)$ having norm $||M||\leq1$, prove that
$$P=\lim_{n\to\infty}\sum_{k=1}^nM^k$$
exists, and equals the projection onto the $1$-eigenspace of $M$.
\end{exercise}

This is something which was at the core of the proof of the existence of the Haar measure. Normally the proof is not very complicated, based on linear algebra.

\begin{exercise}
Work out the details of the abstract Weingarten integration formula in the group dual case, where $A=C^*(\Gamma)$ with $\Gamma=<g_1,\ldots,g_N>$.
\end{exercise}

The first problem here is that of choosing a suitable basis for $Fix(u^{\otimes k})$, which normally should not cause any problems, and then writing down the integration formula.

\begin{exercise}
Work out in detail the representation theory for the basic operations, namely products, dual free products, quotients, projective versions. 
\end{exercise}

As before, this is something that we already discussed, but a bit in a hurry, just as an illustration, and the problem is now that of working out all the details.

\chapter{Tannakian duality}

\section*{4a. Tensor categories}

In order to have more insight into the structure of the compact quantum groups, in general and for the concrete examples too, and to effectively compute their representations, we can use algebraic geometry methods, and more precisely Tannakian duality. 

\bigskip

Tannakian duality rests on the basic principle in any kind of mathematics, algebra, geometry or analysis, ``linearize''. In the present setting, where we do not have a Lie algebra, this will be in fact our only possible linearization method. 

\bigskip

In practice, this duality is something quite broad, and there are many formulations of it, sometimes not obviously equivalent. In what follows we will present Woronowicz's original Tannakian duality result from \cite{wo2}, in its ``soft'' form, worked out by Malacarne in \cite{mal}. This is something which is very efficient, in what regards the applications.

\bigskip

Finally, let us mention that there will be a lot of algebra going on here, in this chapter, and if you're more of an analyst, this might disturb you. To which I have to say two things. First is that this chapter is definitely not to be skipped, and especially by you, analyst, because if you want to do advanced probability theory over quantum groups, you need Tannakian duality. And second is that algebra and analysis are both part of mathematics, along by the way with geometry, PDE and many other things, and a mathematician's job is normally to: (1) know mathematics, (2) develop mathematics.

\bigskip

Getting started now, the idea will be that of further building on the Peter-Weyl theory, from  chapter 3. Let us start with the following result, that we know from there:

\index{tensor category}
\index{Hom space}

\begin{theorem}
Given a Woronowicz algebra $(A,u)$, the Hom spaces for its corepresentations form a tensor $*$-category, in the sense that:
\begin{enumerate}
\item $T\in Hom(u,v),S\in Hom(v,w)\implies ST\in Hom(u,w)$.

\item $S\in Hom(p,q),T\in Hom(v,w)\implies S\otimes T\in Hom(p\otimes v,q\otimes w)$.

\item $T\in Hom(v,w)\implies T^*\in Hom(w,v)$.
\end{enumerate} 
\end{theorem}

\begin{proof}
This is something that we already know, from chapter 3 above, the proofs of all the assertions being elementary, as follows:

\medskip

(1) By using our assumptions $Tu=vT$ and $Sv=Ws$ we obtain, as desired:
$$STu=SvT=wST$$

(2) Assume indeed that we have $Sp=qS$ and $Tv=wT$. With standard tensor product notations, we have the following computation:
$$(S\otimes T)(p\otimes v)
=S_1T_2p_{13}v_{23}
=(Sp)_{13}(Tv)_{23}$$

We have as well the following computation, which gives the result:
$$(q\otimes w)(S\otimes T)
=q_{13}w_{23}S_1T_2
=(qS)_{13}(wT)_{23}$$

(3) By conjugating, and then using the unitarity of $v,w$, we obtain, as desired:
\begin{eqnarray*}
Tv=wT
&\implies&v^*T^*=T^*w^*\\
&\implies&vv^*T^*w=vT^*w^*w\\
&\implies&T^*w=vT^*
\end{eqnarray*}

Thus, we are led to the conclusion in the statement.
\end{proof}

Generally speaking, Tannakian duality amounts in recovering $(A,u)$ from the tensor category constructed in Theorem 4.1. In what follows we will present a ``soft form'' of this duality, coming from \cite{mal}, \cite{wo2}, which uses the following smaller category:

\index{Tannakian category}
\index{Peter-Weyl representation}
\index{Peter-Weyl corepresentation}

\begin{definition}
The Tannakian category associated to a Woronowicz algebra $(A,u)$ is the collection $C=(C(k,l))$ of vector spaces
$$C(k,l)=Hom(u^{\otimes k},u^{\otimes l})$$
where the corepresentations $u^{\otimes k}$ with $k=\circ\bullet\bullet\circ\ldots$ colored integer, defined by
$$u^{\otimes\emptyset}=1\quad,\quad 
u^{\otimes\circ}=u\quad,\quad 
u^{\otimes\bullet}=\bar{u}$$
and multiplicativity, $u^{\otimes kl}=u^{\otimes k}\otimes u^{\otimes l}$, are the Peter-Weyl corepresentations.
\end{definition}

We know from Theorem 4.1 above that $C$ is a tensor $*$-category. To be more precise, if we denote by $H=\mathbb C^N$ the Hilbert space where $u\in M_N(A)$ coacts, then $C$ is a tensor $*$-subcategory of the tensor $*$-category formed by the following linear spaces: 
$$E(k,l)=\mathcal L(H^{\otimes k},H^{\otimes l})$$

Here the tensor powers $H^{\otimes k}$ with $k=\circ\bullet\bullet\circ\ldots$ colored integer are those where the corepresentations $u^{\otimes k}$ act, defined by the following formulae, and multiplicativity:
$$H^{\otimes\emptyset}=\mathbb C\quad,\quad 
H^{\otimes\circ}=H\quad,\quad 
H^{\otimes\bullet}=\bar{H}\simeq H$$

Our purpose in what follows will be that of reconstructing $(A,u)$ in terms of the category $C=(C(k,l))$. We will see afterwards that this method has many applications.

\bigskip

As a first, elementary result on the subject, we have:

\begin{proposition}
Given a morphism $\pi:(A,u)\to(B,v)$ we have inclusions
$$Hom(u^{\otimes k},u^{\otimes l})\subset Hom(v^{\otimes k},v^{\otimes l})$$
for any $k,l$, and if these inclusions are all equalities, $\pi$ is an isomorphism.
\end{proposition}

\begin{proof}
The fact that we have indeed inclusions as in the statement is clear from definitions. As for the last assertion, this follows from the Peter-Weyl theory. 

Indeed, if we assume that $\pi$ is not an isomorphism, then one of the irreducible corepresentations of $A$ must become reducible as a corepresentation of $B$. 

But the irreducible corepresentations being subcorepresentations of the Peter-Weyl corepresentations $u^{\otimes k}$, one of the spaces $End(u^{\otimes k})$ must therefore increase strictly, and this gives the desired contradiction.
\end{proof}

The Tannakian duality result that we want to prove states, in a simplified form, that in what concerns the last conclusion in the above statement, the assumption that we have a morphism $\pi:(A,u)\to(B,v)$ is not needed. In other words, if we know that the Tannakian categories of $A,B$ are different, then $A,B$ themselves must be different.

\bigskip

In order to get started, our first goal will be that of gaining some familiarity with the notion of Tannakian category. And here, we have to use the only general fact that we know about $u$, namely that this matrix is biunitary. We have:

\index{biunitary}

\begin{proposition}
Consider the operator $R:\mathbb C\to\mathbb C^N\otimes\mathbb C^N$ given by: 
$$R(1)=\sum_ie_i\otimes e_i$$
An abstract matrix $u\in M_N(A)$ is then a biunitary precisely when the conditions
$$R\in Hom(1,u\otimes\bar{u})\quad,\quad 
R\in Hom(1,\bar{u}\otimes u)$$
$$R^*\in Hom(u\otimes\bar{u},1)\quad,\quad 
R^*\in Hom(\bar{u}\otimes u,1)$$
are all satisfied, in a formal sense, as suitable commutation relations.
\end{proposition}

\begin{proof}
Let us first recall that, in the Woronowicz algebra setting, the definition of the Hom space between two corepresentations $v\in M_n(A)$, $w\in M_m(A)$ is as follows: 
$$Hom(v,w)=\left\{T\in M_{m\times n}(\mathbb C)\Big|Tv=wT\right\}$$

But this is something that makes no reference to the Woronowicz algebra structure of $A$, or to the fact that $v,w$ are indeed corepresentations. Thus, this notation can be formally formally used for any two matrices $v\in M_n(A)$, $w\in M_m(A)$, over an arbitrary $C^*$-algebra $A$, and so our statement, as formulated, makes sense indeed.

With $R$ being as in the statement, we have the following computation:
\begin{eqnarray*}
(u\otimes\bar{u})(R(1)\otimes1)
&=&\sum_{ijk}e_i\otimes e_k\otimes u_{ij}u_{kj}^*\\
&=&\sum_{ik}e_i\otimes e_k\otimes(uu^*)_{ik}
\end{eqnarray*}

We conclude from this that we have the following equivalence:
$$R\in Hom(1,u\otimes\bar{u})\iff uu^*=1$$

Consider now the adjoint operator $R^*:\mathbb C^N\otimes\mathbb C^N\to\mathbb C$, which is given by:
$$R^*(e_i\otimes e_j)=\delta_{ij}$$

We have then the following computation:
\begin{eqnarray*}
(R^*\otimes id)(u\otimes\bar{u})(e_j\otimes e_l\otimes1)
&=&\sum_iu_{ij}u_{il}^*\\
&=&(u^t\bar{u})_{jl}
\end{eqnarray*}

We conclude from this that we have the following equivalence:
$$R^*\in Hom(u\otimes\bar{u},1)\iff u^t\bar{u}=1$$

Similarly, or simply by replacing $u$ in the above two conclusions with its conjugate $\bar{u}$, which is a corepresentation too, we have as well the following two equivalences:
$$R\in Hom(1,\bar{u}\otimes u)\iff\bar{u}u^t=1$$
$$R^*\in Hom(\bar{u}\otimes u,1)\iff u^*u=1$$

Thus, we are led to the biunitarity conditions, and we are done.
\end{proof}

As a consequence of this computation, we have the following result:

\begin{proposition}
The Tannakian category $C=(C(k,l))$ associated to a Woronowicz algebra $(A,u)$ must contain the operators
$$R:1\to\sum_ie_i\otimes e_i$$
$$R^*(e_i\otimes e_j)=\delta_{ij}$$
in the sense that we must have: 
$$R\in C(\emptyset,\circ\bullet)\quad,\quad 
R\in C(\emptyset,\bullet\circ)$$
$$R^*\in C(\circ\bullet,\emptyset)\quad,\quad 
R^*\in C(\bullet\circ,\emptyset)$$
In fact, $C$ must contain the whole tensor category $<R,R^*>$ generated by $R,R^*$.
\end{proposition}

\begin{proof}
The first assertion is clear from the above result. As for the second assertion, this is clear from definitions, because $C=(C(k,l))$ is indeed a tensor category.
\end{proof}

Let us formulate now the following key definition:

\index{tensor category}

\begin{definition}
Let $H$ be a finite dimensional Hilbert space. A tensor category over $H$ is a collection $C=(C(k,l))$ of subspaces 
$$C(k,l)\subset\mathcal L(H^{\otimes k},H^{\otimes l})$$
satisfying the following conditions:
\begin{enumerate}
\item $S,T\in C$ implies $S\otimes T\in C$.

\item If $S,T\in C$ are composable, then $ST\in C$.

\item $T\in C$ implies $T^*\in C$.

\item Each $C(k,k)$ contains the identity operator.

\item $C(\emptyset,\circ\bullet)$ and $C(\emptyset,\bullet\circ)$ contain the operator $R:1\to\sum_ie_i\otimes e_i$.
\end{enumerate}
\end{definition}

As a basic example here, the collection of the vector spaces $\mathcal L(H^{\otimes k},H^{\otimes l})$ is of course a tensor category over $H$. There are many other concrete examples, which can be constructed by using various combinatorial methods, and we will discuss this later on.

\bigskip

In relation with the quantum groups, this formalism generalizes the Tannakian category formalism from Definition 4.2 above, because we have the following result:

\begin{proposition}
Let $(A,u)$ be a Woronowicz algebra, with fundamental corepresentation $u\in M_N(A)$. The associated Tannakian category $C=(C(k,l))$, given by
$$C(k,l)=Hom(u^{\otimes k},u^{\otimes l})$$
is then a tensor category over the Hilbert space $H=\mathbb C^N$.
\end{proposition}

\begin{proof}
The fact that the above axioms (1-5) are indeed satisfied is clear, as follows:

\medskip

(1) This follows from Theorem 4.1.

\medskip

(2) Once again, this follows from Theorem 4.1.

\medskip

(3) This once again follows from Theorem 4.1.

\medskip

(4) This is clear from definitions.

\medskip

(5) This follows from Proposition 4.5 above.
\end{proof}

Our purpose in what follows will be that of proving that the converse of the above statement holds. That is, we would like to prove that any tensor category in the sense of Definition 4.6 must appear as a Tannakian category. And with this being obviously a powerful ``linearization'' result, as advertised in the beginning of this chapter.

\bigskip

As a first result on this subject, providing us with a correspondence $C\to A_C$, which is complementary to the correspondence $A\to C_A$ from Proposition 4.7, we have:

\begin{proposition}
Given a tensor category $C=(C(k,l))$, the following algebra, with $u$ being the fundamental corepresentation of $C(U_N^+)$, is a Woronowicz algebra:
$$A_C=C(U_N^+)\big/\left<T\in Hom(u^{\otimes k},u^{\otimes l})\Big|\forall k,l,\forall T\in C(k,l)\right>$$
In the case where $C$ comes from a Woronowicz algebra $(A,v)$, we have a quotient map:
$$A_C\to A$$
Moreover, this map is an isomorphism in the discrete group algebra case.
\end{proposition}

\begin{proof}
Given two colored integers $k,l$ and a linear operator $T\in\mathcal L(H^{\otimes k},H^{\otimes l})$, consider the following $*$-ideal of the algebra $C(U_N^+)$:
$$I=\Big<T\in Hom(u^{\otimes k},u^{\otimes l})\Big>$$

Our claim is that $I$ is a Hopf ideal. Indeed, let us set:
$$U=\sum_ku_{ik}\otimes u_{kj}$$

We have then the following implication, which is something elementary, coming from a standard algebraic computation with indices, and which proves our claim:
$$T\in Hom(u^{\otimes k},u^{\otimes l})\implies T\in Hom(U^{\otimes k},U^{\otimes l})$$

With this claim in hand, the algebra $A_C$ appears from $C(U_N^+)$ by dividing by a certain collection of Hopf ideals, and is therefore a Woronowicz algebra. Since the relations defining $A_C$ are satisfied in $A$, we have a quotient map as in the statement, namely:
$$A_C\to A$$

Regarding now the last assertion, assume that we are in the case $A=C^*(\Gamma)$, with $\Gamma=<g_1,\ldots,g_N>$ being a finitely generated discrete group. If we denote by $\mathcal R$ the complete collection of relations between the generators, then we have:
$$\Gamma=F_N/\mathcal R$$

By using now the basic functoriality properties of the group algebra construction, we deduce from this that we have an identification as follows:
$$A_C=C^*\left(F_N\Big/\Big<\mathcal R\Big>\right)$$

Thus the quotient map $A_C\to A$ is indeed an isomorphism, as claimed.
\end{proof}

With the above two constructions in hand, from Proposition 4.7 and Proposition 4.8, we are now in position of formulating a clear objective. To be more precise, the theorem that we want to prove states that the following operations are inverse to each other:
$$A\to A_C\quad,\quad C\to C_A$$

We have the following result, to start with, which simplifies our work:

\begin{proposition}
Consider the following conditions:
\begin{enumerate}
\item $C=C_{A_C}$, for any Tannakian category $C$.

\item $A=A_{C_A}$, for any Woronowicz algebra $(A,u)$.
\end{enumerate}
We have then $(1)\implies(2)$. Also, $C\subset C_{A_C}$ is automatic.
\end{proposition}

\begin{proof}
Given a Woronowicz algebra $(A,u)$, let us set:
$$C=C_A$$

By using (1) we have then an equality as follows: 
$$C_A=C_{A_{C_A}}$$

On the other hand, by Proposition 4.8 we have an arrow as follows:
$$A_{C_A}\to A$$

Thus, we are in the general situation from Proposition 4.3 above, with a surjective arrow of Woronowicz algebras, which becomes an isomorphism at the level of the associated Tannakian categories. We conclude that Proposition 4.3 can be applied, and this gives the isomorphism of the associated Woronowicz algebras, $A_{C_A}=A$, as desired. Finally, the fact that we have an inclusion $C\subset C_{A_C}$ is clear from definitions.
\end{proof}

Summarizing, in order to establish the Tannakian duality correspondence, it is enough to prove that we have $C_{A_C}\subset C$, for any Tannakian category $C$. 

\section*{4b. Abstract algebra}

In order to prove that we have $C_{A_C}\subset C$, for any Tannakian category $C$, let us begin with some abstract constructions. Following Malacarne \cite{mal}, we have:

\begin{proposition}
Given a tensor category $C=C((k,l))$ over a Hilbert space $H$,
$$E_C^{(s)}
=\bigoplus_{|k|,|l|\leq s}C(k,l)
\subset\bigoplus_{|k|,|l|\leq s}B(H^{\otimes k},H^{\otimes l})
= B\left(\bigoplus_{|k|\leq s}H^{\otimes k}\right)$$
is a finite dimensional $C^*$-subalgebra. Also, 
$$E_C
=\bigoplus_{k,l}C(k,l)
\subset\bigoplus_{k,l}B(H^{\otimes k},H^{\otimes l})
\subset B\left(\bigoplus_kH^{\otimes k}\right)$$
is a closed $*$-subalgebra.
\end{proposition}

\begin{proof}
This is clear indeed from the categorical axioms from Definition 4.6, via the standard embeddings and isomorphisms in the statement.
\end{proof}

Now back to our reconstruction question, given a tensor category $C=(C(k,l))$, we want to prove that we have $C=C_{A_C}$, which is the same as proving that we have:
$$E_C=E_{C_{A_C}}$$

Equivalently, we want to prove that we have equalities as follows, for any $s\in\mathbb N$:
$$E_C^{(s)}=E_{C_{A_C}}^{(s)}$$

The problem, however, is that these equalities are not easy to establish directly. In order to solve this question, we will use a standard commutant trick, as follows:

\index{bicommutant}

\begin{theorem}
For any $C^*$-algebra $B\subset M_n(\mathbb C)$ we have the formula
$$B=B''$$
where prime denotes the commutant, given by:
$$A'=\left\{T\in M_n(\mathbb C)\Big|Tx=xT,\forall x\in A\right\}$$
\end{theorem}

\begin{proof}
This is a particular case of von Neumann's bicommutant theorem, which follows as well from the explicit description of $B$ given in chapter 3 above. To be more precise, let us decompose $B$ as there, as a direct sum of matrix algebras:
$$B=M_{r_1}(\mathbb C)\oplus\ldots\oplus M_{r_k}(\mathbb C)$$

The center of each matrix algebra being reduced to the scalars, the commutant of this algebra is then as follows, with each copy of $\mathbb C$ corresponding to a matrix block:
$$B'=\mathbb C\oplus\ldots\oplus\mathbb C$$

By taking once again the commutant, and computing over the matrix blocks, we obtain the algebra $B$ itself, and this gives the formula in the statement.
\end{proof}

Now back to our questions, we recall that we want to prove that we have $C=C_{A_C}$, for any Tannakian category $C$. By using the bicommutant theorem, we have:

\begin{proposition}
Given a Tannakian category $C$, the following are equivalent:
\begin{enumerate}
\item $C=C_{A_C}$.

\item $E_C=E_{C_{A_C}}$.

\item $E_C^{(s)}=E_{C_{A_C}}^{(s)}$, for any $s\in\mathbb N$.

\item $E_C^{(s)'}=E_{C_{A_C}}^{(s)'}$, for any $s\in\mathbb N$.
\end{enumerate}
In addition, the inclusions $\subset$, $\subset$, $\subset$, $\supset$ are automatically satisfied.
\end{proposition}

\begin{proof}
This follows from the above results, as follows:

\medskip

$(1)\iff(2)$ This is clear from definitions.

\medskip

$(2)\iff(3)$ This is clear from definitions as well.

\medskip

$(3)\iff(4)$ This comes from the bicommutant theorem. As for the last assertion, we have indeed $C\subset C_{A_C}$ from Proposition 4.9, and this shows that we have as well: 
$$E_C\subset E_{C_{A_C}}$$

We therefore obtain the following inclusion: 
$$E_C^{(s)}\subset E_{C_{A_C}}^{(s)}$$

By taking now the commutants, this gives:
$$E_C^{(s)}\supset E_{C_{A_C}}^{(s)}$$

Thus, we are led to the conclusion in the statement.
\end{proof}

Summarizing, in order to finish, given a tensor category $C=(C(k,l))$, we would like to prove that we have inclusions as follows, for any $s\in\mathbb N$:
$$E_C^{(s)'}\subset E_{C_{A_C}}^{(s)'}$$ 

Let us first study the commutant on the right. As a first observation, we have:

\begin{proposition}
Given a Woronowicz algebra $(A,u)$, we have
$$E_{C_A}^{(s)}=End\left(\bigoplus_{|k|\leq s}u^{\otimes k}\right)$$
as subalgebras of the following algebra:
$$B\left(\bigoplus_{|k|\leq s}H^{\otimes k}\right)$$
\end{proposition}

\begin{proof}
The category $C_A$ is by definition given by:
$$C_A(k,l)=Hom(u^{\otimes k},u^{\otimes l})$$

Thus, according to the various identifications in Proposition 4.10 above, the corresponding algebra $E_{C_A}^{(s)}$ appears as follows:
\begin{eqnarray*}
E_{C_A}^{(s)}
&=&\bigoplus_{|k|,|l|\leq s}Hom(u^{\otimes k},u^{\otimes l})\\
&\subset&\bigoplus_{|k|,|l|\leq s}B(H^{\otimes k},H^{\otimes l})\\
&=&B\left(\bigoplus_{|k|\leq s}H^{\otimes k}\right)
\end{eqnarray*}

On the other hand, the algebra of intertwiners of $\bigoplus_{|k|\leq s}u^{\otimes k}$ is given by:
\begin{eqnarray*}
End\left(\bigoplus_{|k|\leq s}u^{\otimes k}\right)
&=&\bigoplus_{|k|,|l|\leq s}Hom(u^{\otimes k},u^{\otimes l})\\
&\subset&\bigoplus_{|k|,|l|\leq s}B(H^{\otimes k},H^{\otimes l})\\
&=&B\left(\bigoplus_{|k|\leq s}H^{\otimes k}\right)
\end{eqnarray*}

Thus we have indeed the same algebra, and we are done.
\end{proof}

In practice now, we have to compute the commutant of the above algebra. And for this purpose, we can use the following general result:

\begin{proposition}
Given a corepresentation $v\in M_n(A)$, we have a representation
$$\pi_v:A^*\to M_n(\mathbb C)$$
$$\varphi\to (\varphi(v_{ij}))_{ij}$$
whose image is given by the following formula:
$$Im(\pi_v)=End(v)'$$
\end{proposition}

\begin{proof}
The first assertion is clear, with the multiplicativity claim coming from:
\begin{eqnarray*}
(\pi_v(\varphi*\psi))_{ij}
&=&(\varphi\otimes\psi)\Delta(v_{ij})\\
&=&\sum_k\varphi(v_{ik})\psi(v_{kj})\\
&=&\sum_k(\pi_v(\varphi))_{ik}(\pi_v(\psi))_{kj}\\
&=&(\pi_v(\varphi)\pi_v(\psi))_{ij}
\end{eqnarray*} 

Let us first prove the inclusion $\subset$. Given $\varphi\in A^*$ and $T\in End(v)$, we have:
\begin{eqnarray*}
[\pi_v(\varphi),T]=0
&\iff&\sum_k\varphi(v_{ik})T_{kj}=\sum_kT_{ik}\varphi(v_{kj}),\forall i,j\\
&\iff&\varphi\left(\sum_kv_{ik}T_{kj}\right)=\varphi\left(\sum_kT_{ik}v_{kj}\right),\forall i,j\\
&\iff&\varphi((vT)_{ij})=\varphi((Tv)_{ij}),\forall i,j
\end{eqnarray*}

But this latter formula is true, because $T\in End(v)$ means that we have:
$$vT=Tv$$

As for the converse inclusion $\supset$, the proof is quite similar. Indeed, by using the bicommutant theorem, this is the same as proving that we have:
$$Im(\pi_v)'\subset End(v)$$

But, by using the above equivalences, we have the following computation:
\begin{eqnarray*}
T\in Im(\pi_v)'
&\iff&[\pi_v(\varphi),T]=0,\forall\varphi\\
&\iff&\varphi((vT)_{ij})=\varphi((Tv)_{ij}),\forall\varphi,i,j\\
&\iff&vT=Tv
\end{eqnarray*}

Thus, we have obtained the desired inclusion, and we are done.
\end{proof}

By combining now the above results, we obtain:

\begin{theorem}
Given a Woronowicz algebra $(A,u)$, we have
$$E_{C_A}^{(s)'}=Im(\pi_v)$$
as subalgebras of the following algebra,
$$B\left(\bigoplus_{|k|\leq s}H^{\otimes k}\right)$$
where the corepresentation $v$ is the sum
$$v=\bigoplus_{|k|\leq s}u^{\otimes k}$$
and where $\pi_v:A^*\to M_n(\mathbb C)$ is given by $\varphi\to(\varphi(v_{ij}))_{ij}$.
\end{theorem}

\begin{proof}
This follows indeed from Proposition 4.13 and Proposition 4.14.
\end{proof}

Summarizing, we have some advances on the duality question, with the whole problem tending to become something quite concrete, which can be effectively solved.

\section*{4c. The correspondence}

We recall that we want to prove that we have $E_C^{(s)'}\subset E_{C_{A_C}}^{(s)'}$, for any $s\in\mathbb N$. For this purpose, we must first refine Theorem 4.15, in the case $A=A_C$. In order to do so, we will use an explicit model for $A_C$. In order to construct such a model, let $<u_{ij}>$ be the free $*$-algebra over $\dim(H)^2$ variables, with comultiplication and counit as follows:
$$\Delta(u_{ij})=\sum_ku_{ik}\otimes u_{kj}$$
$$\varepsilon(u_{ij})=\delta_{ij}$$

Following Malacarne \cite{mal}, we can model this $*$-bialgebra, in the following way:

\begin{proposition}
Consider the following pair of dual vector spaces,
$$F=\bigoplus_kB\left(H^{\otimes k}\right)$$
$$F^*=\bigoplus_kB\left(H^{\otimes k}\right)^*$$
and let $f_{ij},f_{ij}^*\in F^*$ be the standard generators of $B(H)^*,B(\bar{H})^*$.
\begin{enumerate}
\item $F^*$ is a $*$-algebra, with multiplication $\otimes$ and involution $f_{ij}\leftrightarrow f_{ij}^*$.

\item $F^*$ is a $*$-bialgebra, with $\Delta(f_{ij})=\sum_kf_{ik}\otimes f_{kj}$ and $\varepsilon(f_{ij})=\delta_{ij}$.

\item We have a $*$-bialgebra isomorphism $<u_{ij}>\simeq F^*$, given by $u_{ij}\to f_{ij}$.
\end{enumerate}
\end{proposition}

\begin{proof}
Since $F^*$ is spanned by the various tensor products between the variables $f_{ij},f_{ij}^*$, we have a vector space isomorphism as follows, given by $u_{ij}\to f_{ij},u_{ij}^*\to f_{ij}^*$:
$$<u_{ij}>\simeq F^*$$

The corresponding $*$-bialgebra structure induced on $F^*$ is the one in the statement.
\end{proof}

Now back to our algebra $A_C$, we have the following modelling result for it:

\begin{proposition}
The smooth part of the algebra $A_C$ is given by 
$$\mathcal A_C\simeq F^*/J$$
where $J\subset F^*$ is the ideal coming from the following relations,
\begin{eqnarray*}
&&\sum_{p_1,\ldots,p_k}T_{i_1\ldots i_l,p_1\ldots p_k}f_{p_1j_1}\otimes\ldots\otimes  f_{p_kj_k}\\
&=&\sum_{q_1,\ldots,q_l}T_{q_1\ldots q_l,j_1\ldots j_k}f_{i_1q_1}\otimes\ldots\otimes f_{i_lq_l}\quad,\quad\forall i,j
\end{eqnarray*}
one for each pair of colored integers $k,l$, and each $T\in C(k,l)$.
\end{proposition}

\begin{proof}
Our first claim is that $A_C$ appears as enveloping $C^*$-algebra of the following universal $*$-algebra, where $u=(u_{ij})$ is regarded as a formal corepresentation:
$$\mathcal A_C=\left<(u_{ij})_{i,j=1,\ldots,N}\Big|T\in Hom(u^{\otimes k},u^{\otimes l}),\forall k,l,\forall T\in C(k,l)\right>$$

Indeed, this follows from Proposition 4.4 above, because according to the result there, the relations defining $C(U_N^+)$ are included into those that we impose.

With this claim in hand, the conclusion is that we have a formula as follows, where $I$ is the ideal coming from the relations $T\in Hom(u^{\otimes k},u^{\otimes l})$, with $T\in C(k,l)$:
$$\mathcal A_C=<u_{ij}>/I$$

Now if we denote by $J\subset F^*$ the image of the ideal $I$ via the $*$-algebra isomorphism $<u_{ij}>\simeq F^*$ from Proposition 4.16, we obtain an identification as follows:
$$\mathcal A_C\simeq F^*/J$$

In order to compute $J$, let us go back to $I$. With standard multi-index notations, and by assuming that $k,l\in\mathbb N$ are usual integers, for simplifying, a relation of type $T\in Hom(u^{\otimes k},u^{\otimes l})$ inside $<u_{ij}>$ is equivalent to the following conditions:
\begin{eqnarray*}
&&\sum_{p_1,\ldots,p_k}T_{i_1\ldots i_l,p_1\ldots p_k}u_{p_1j_1}\ldots u_{p_kj_k}\\
&=&\sum_{q_1,\ldots,q_l}T_{q_1\ldots q_l,j_1\ldots j_k}u_{i_1q_1}\ldots u_{i_lq_l}\quad,\quad\forall i,j
\end{eqnarray*}

Now by recalling that the isomorphism of $*$-algebras $<u_{ij}>\to F^*$ is given by $u_{ij}\to f_{ij}$, and that the multiplication operation of $F^*$ corresponds to the tensor product operation $\otimes$, we conclude that $J\subset F^*$ is the ideal from the statement.
\end{proof}

With the above result in hand, let us go back to Theorem 4.15. We have:

\begin{proposition}
The linear space $\mathcal A_C^*$ is given by the formula
$$\mathcal A_C^*=\left\{a\in F\Big|Ta_k=a_lT,\forall T\in C(k,l)\right\}$$
and the representation
$$\pi_v:\mathcal A_C^*\to B\left(\bigoplus_{|k|\leq s}H^{\otimes k}\right)$$
appears diagonally, by truncating, via the following formula:
$$\pi_v:a\to (a_k)_{kk}$$
\end{proposition}

\begin{proof}
We know from Proposition 4.17 that we have:
$$\mathcal A_C\simeq F^*/J$$

But this gives a quotient map $F^*\to\mathcal A_C$, and so an inclusion as follows:
$$\mathcal A_C^*\subset F$$

To be more precise, we have the following formula:
$$\mathcal A_C^*=\left\{a\in\ F\Big|f(a)=0,\forall f\in J\right\}$$

Now since $J=<f_T>$, where $f_T$ are the relations in Proposition 4.17, we obtain:
$$\mathcal A_C^*=\left\{a\in F\Big|f_T(a)=0,\forall T\in C\right\}$$

Given $T\in C(k,l)$, for an arbitrary element $a=(a_k)$, we have:
\begin{eqnarray*}
&&f_T(a)=0\\
&\iff&\sum_{p_1,\ldots,p_k}T_{i_1\ldots i_l,p_1\ldots p_k}(a_k)_{p_1\ldots p_k,j_1\ldots j_k}=\sum_{q_1,\ldots,q_l}T_{q_1\ldots q_l,j_1\ldots j_k}(a_l)_{i_1\ldots i_l,q_1\ldots q_l},\forall i,j\\
&\iff&(Ta_k)_{i_1\ldots i_l,j_1\ldots j_k}=(a_lT)_{i_1\ldots i_l,j_1\ldots j_k},\forall i,j\\
&\iff&Ta_k=a_lT
\end{eqnarray*}

Thus, the dual space $\mathcal A_C^*$ is given by the formula in the statement. It remains to compute the representation $\pi_v$, which appears as follows:
$$\pi_v:\mathcal A_C^*\to B\left(\bigoplus_{|k|\leq s}H^{\otimes k}\right)$$

With $a=(a_k)$, we have the following computation:
\begin{eqnarray*}
\pi_v(a)_{i_1\ldots i_k,j_1\ldots j_k}
&=&a(v_{i_1\ldots i_k,j_1\ldots j_k})\\
&=&(f_{i_1j_1}\otimes\ldots\otimes f_{i_kj_k})(a)\\
&=&(a_k)_{i_1\ldots i_k,j_1\ldots j_k}
\end{eqnarray*}

Thus, our representation $\pi_v$ appears diagonally, by truncating, as claimed.
\end{proof}

In order to further advance, consider the following vector spaces:
$$F_s=\bigoplus_{|k|\leq s}B\left(H^{\otimes k}\right)$$ 
$$F^*_s=\bigoplus_{|k|\leq s}B\left(H^{\otimes k}\right)^*$$

We denote by $a\to a_s$ the truncation operation $F\to F_s$. We have:

\begin{proposition}
The following hold:
\begin{enumerate}
\item $E_C^{(s)'}\subset F_s$.

\item $E_C'\subset F$.

\item $\mathcal A_C^*=E_C'$.

\item $Im(\pi_v)=(E_C')_s$.
\end{enumerate}
\end{proposition}

\begin{proof}
These results basically follow from what we have, as follows:

\medskip

(1) We have an inclusion as follows, as a diagonal subalgebra:
$$F_s\subset B\left(\bigoplus_{|k|\leq s}H^{\otimes k}\right)$$

The commutant of this algebra is given by:
$$F_s'=\left\{b\in F_s\Big|b=(b_k),b_k\in\mathbb C,\forall k\right\}$$

On the other hand, we know from the identity axiom for the catrgory $C$ that this algebra is contained inside $E_C^{(s)}$:
$$F_s'\subset E_C^{(s)}$$

Thus, our result follows from the bicommutant theorem, as follows:
$$F_s'\subset E_C^{(s)}\implies F_s\supset E_C^{(s)'}$$

(2) This follows from (1), by taking inductive limits.

\medskip

(3) With the present notations, the formula of $\mathcal A_C^*$ from Proposition 4.18 reads:
$$\mathcal A_C^*=F\cap E_C'$$

Now since by (2) we have $E_C'\subset F$, we obtain from this $\mathcal A_C^*=E_C'$.

\medskip

(4) This follows from (3), and from the formula of $\pi_v$ in Proposition 4.18.
\end{proof}

Still following the paper of Malacarne \cite{mal}, we can now state and prove our main result, originally due to Woronowicz \cite{wo2}, as follows:

\index{Tannakian duality}

\begin{theorem}
The Tannakian duality constructions 
$$C\to A_C\quad,\quad 
A\to C_A$$
are inverse to each other, modulo identifying full and reduced versions.
\end{theorem}

\begin{proof}
According to Proposition 4.9, Proposition 4.12, Theorem 5.15 and Proposition 4.19, we have to prove that, for any Tannakian category $C$, and any $s\in\mathbb N$:
$$E_C^{(s)'}\subset(E_C')_s$$

By taking duals, this is the same as proving that we have:
$$\left\{f\in F_s^*\Big|f_{|(E_C')_s}=0\right\}\subset\left\{f\in F_s^*\Big|f_{|E_C^{(s)'}}=0\right\}$$

For this purpose, we use the following formula, coming from Proposition 4.19: 
$$\mathcal A_C^*=E_C'$$

We know as well that we have the following formula:
$$\mathcal A_C=F^*/J$$

We conclude that the ideal $J$ is given by:
$$J=\left\{f\in F^*\Big|f_{|E_C'}=0\right\}$$

Our claim is that we have the following formula, for any $s\in\mathbb N$:
$$J\cap F_s^*=\left\{f\in F_s^*\Big|f_{|E_C^{(s)'}}=0\right\}$$

Indeed, let us denote by $X_s$ the spaces on the right. The categorical axioms for $C$ show that these spaces are increasing, that their union $X=\cup_sX_s$ is an ideal, and that:
$$X_s=X\cap F_s^*$$

We must prove that we have $J=X$, and this can be done as follows:

\medskip

``$\subset$'' This follows from the following fact, for any $T\in C(k,l)$ with $|k|,|l|\leq s$:
\begin{eqnarray*}
(f_T)_{|\{T\}'}=0
&\implies&(f_T)_{|E_C^{(s)'}}=0\\
&\implies&f_T\in X_s
\end{eqnarray*}

``$\supset$'' This follows from our description of $J$, because from $E_C^{(s)}\subset E_C$ we obtain:
$$f_{|E_C^{(s)'}}=0\implies f_{|E_C'}=0$$

Summarizing, we have proved our claim. On the other hand, we have:
\begin{eqnarray*}
J\cap F_s^*
&=&\left\{f\in F^*\Big|f_{|E_C'}=0\right\}\cap F_s^*\\
&=&\left\{f\in F_s^*\Big|f_{|E_C'}=0\right\}\\
&=&\left\{f\in F_s^*\Big|f_{|(E_C')_s}=0\right\}
\end{eqnarray*}

Thus, our claim is exactly the inclusion that we wanted to prove, and we are done.
\end{proof}

Summarizing, we have proved Tannakian duality. As already mentioned in the beginning of this chapter, there are many other forms of Tannakian duality for the compact quantum groups, and we refer here to Woronowicz's original paper \cite{wo2}, which contains a full discussion of the subject, and to the subsequent literature.

\bigskip

As we will see in a moment, Tannakian duality in the above form is something quite powerful, enabling us to recover the Brauer theorems for $O_N,U_N$, going back to Brauer's paper \cite{bra}, and for their free versions $O_N^+,U_N^+$ as well, following \cite{ba1}, \cite{bc1}. 

\bigskip

Later on, in chapter 7 below and afterwards, we will further build on Tannakian duality, with a subsequent notion of ``easiness'' coming from it. Let us also mention, for the concerned reader, that all this escalation of algebraic methods will eventually lead into very concrete applications, of analytic and probabilistic nature. 

\bigskip

As a first application now, let us record the following theoretical fact, from \cite{bb4}:

\index{algebraic manifold}
\index{real algebraic manifold}

\begin{theorem}
Each closed subgroup $G\subset U_N^+$ appears as an algebraic manifold of the free complex sphere, 
$$G\subset S^{N^2-1}_{\mathbb C,+}$$
the embedding being given by the following formula, in double indices:
$$x_{ij}=\frac{u_{ij}}{\sqrt{N}}$$
\end{theorem}

\begin{proof}
This follows from Theorem 4.20, by using the following inclusions:
$$G\subset U_N^+\subset S^{N^2-1}_{\mathbb C,+}$$

Indeed, both these inclusions are algebraic, and this gives the result.
\end{proof}

\section*{4d. Brauer theorems}

\index{orthogonal quantum group}
\index{unitary quantum group}

As a second application of Tannakian duality, let us study now the representation theory of $O_N^+,U_N^+$. In order to get started, let us get back to the operators $R,R^*$, from the beginning of this chapter. We know that these two operators must be present in any Tannakian category, and in what concerns $U_N^+$, which is the biggest $N\times N$ compact quantum group, a converse of this fact holds, by contravariant functoriality, as follows:

\begin{proposition}
The tensor category $<R,R^*>$ generated by the operators
$$R:1\to\sum_ie_i\otimes e_i$$
$$R^*(e_i\otimes e_j)=\delta_{ij}$$
produces via Tannakian duality the algebra $C(U_N^+)$.
\end{proposition}

\begin{proof}
This follows from the results from the beginning of this chapter, via the Tannakian duality established above. To be more precise, we know from Proposition 4.5 that the intertwining relations coming from the operators $R,R^*$, and so from any element of the tensor category $<R,R^*>$, hold automatically. Thus the quotient operation in Proposition 4.8 is trivial, and we obtain the algebra $C(U_N^+)$ itself, as stated.
\end{proof}

As a conclusion, in order to compute the Tannakian category of $U_N^+$, we must simply solve a linear algebra question, namely computing the category $<R,R^*>$. Regarding now $O_N^+$, the result here is similar, as follows:

\begin{proposition}
The tensor category $<R,R^*>$ generated by the operators
$$R:1\to\sum_ie_i\otimes e_i$$
$$R^*(e_i\otimes e_j)=\delta_{ij}$$
with identifying the colors, $\circ=\bullet$, produces via Tannakian duality the algebra $C(O_N^+)$.
\end{proposition}

\begin{proof}
By Proposition 4.5 the intertwining relations coming from $R,R^*$, and so from any element of the tensor category $<R,R^*>$, hold automatically, so the quotient operation in Proposition 4.8 is trivial, and we obtain $C(O_N^+)$ itself, as stated.
\end{proof}

Our goal now will be that of reaching to a better understanding of $R,R^*$. In order to do so, we use a diagrammatic formalism, as follows:

\index{pairing}
\index{matching pairing}
\index{noncrossing pairing}

\begin{definition}
Let $k,l$ be two colored integers, having lengths $|k|,|l|\in\mathbb N$.
\begin{enumerate}
\item $P_2(k,l)$ is the set of pairings between an upper row of $|k|$ points, and a lower row of $|l|$ points, with these two rows of points colored by $k,l$.

\item $\mathcal{P}_2(k,l)\subset P_2(k,l)$ is the set of matching pairings, whose horizontal strings connect $\circ-\circ$ or $\bullet-\bullet$, and whose vertical strings connect $\circ-\bullet$.

\item $NC_2(k,l)\subset P_2(k,l)$ is the set of pairings which are noncrossing, in the sense that we can draw the pairing as for the strings to be noncrossing. 

\item $\mathcal{NC}_2(k,l)\subset P_2(k,l)$ is the subset of noncrossing matching pairings, obtained as an intersection, $\mathcal{NC}_2(k,l)=NC_2(k,l)\cap\mathcal P_2(k,l)$.  
\end{enumerate}
\end{definition}

The relation with the Tannakian categories of linear maps comes from the fact that we can associate linear maps to the pairings, as in \cite{bsp}, as follows:

\index{Kronecker symbol}

\begin{definition}
Associated to any pairing $\pi\in P_2(k,l)$ and any integer $N\in\mathbb N$ is the linear map 
$$T_\pi:(\mathbb C^N)^{\otimes k}\to(\mathbb C^N)^{\otimes l}$$
given by the following formula, with $\{e_1,\ldots,e_N\}$ being the standard basis of $\mathbb C^N$,
$$T_\pi(e_{i_1}\otimes\ldots\otimes e_{i_k})=\sum_{j_1\ldots j_l}\delta_\pi\begin{pmatrix}i_1&\ldots&i_k\\ j_1&\ldots&j_l\end{pmatrix}e_{j_1}\otimes\ldots\otimes e_{j_l}$$
and with the Kronecker symbols $\delta_\pi\in\{0,1\}$ depending on whether the indices fit or not. 
\end{definition}

To be more precise here, in the definition of the Kronecker symbols, we agree to put the two multi-indices on the two rows of points of the pairing, in the obvious way. The Kronecker symbols are then defined by $\delta_\pi=1$ when all the strings of $\pi$ join equal indices, and by $\delta_\pi=0$ otherwise. Observe that all this is independent of the coloring.

\bigskip

Here are a few basic examples of such linear maps:

\index{semicircle partition}
\index{basic crossing}

\begin{proposition}
The correspondence $\pi\to T_\pi$ has the following properties:
\begin{enumerate}
\item $T_\cap=R$.

\item $T_\cup=R^*$.

\item $T_{||\ldots||}=id$.

\item $T_{\slash\hskip-1.5mm\backslash}=\Sigma$.
\end{enumerate}
\end{proposition}

\begin{proof}
We can assume if we want that all the upper and lower legs of $\pi$ are colored $\circ$. With this assumption made, the proof goes as follows:

\medskip

(1) We have $\cap\in P_2(\emptyset,\circ\circ)$, and so the corresponding operator is a certain linear map $T_\cap:\mathbb C\to\mathbb C^N\otimes\mathbb C^N$. The formula of this map is as follows:
\begin{eqnarray*}
T_\cap(1)
&=&\sum_{ij}\delta_\cap(i\ j)e_i\otimes e_j\\
&=&\sum_{ij}\delta_{ij}e_i\otimes e_j\\
&=&\sum_ie_i\otimes e_i
\end{eqnarray*}

We recognize here the formula of $R(1)$, and so we have $T_\cap=R$, as claimed.

\medskip

(2) Here we have $\cup\in P_2(\circ\circ,\emptyset)$, and so the corresponding operator is a certain linear form $T_\cap:\mathbb C^N\otimes\mathbb C^N\to\mathbb C$. The formula of this linear form is as follows:
\begin{eqnarray*}
T_\cap(e_i\otimes e_j)
&=&\delta_\cap(i\ j)\\
&=&\delta_{ij}
\end{eqnarray*}

Since this is the same as $R^*(e_i\otimes e_j)$, we have $T_\cup=R^*$, as claimed.

\medskip

(3) Consider indeed the ``identity'' pairing $||\ldots||\in P_2(k,k)$, with $k=\circ\circ\ldots\circ\circ$. The corresponding linear map is then the identity, because we have:
\begin{eqnarray*}
T_{||\ldots||}(e_{i_1}\otimes\ldots\otimes e_{i_k})
&=&\sum_{j_1\ldots j_k}\delta_{||\ldots||}\begin{pmatrix}i_1&\ldots&i_k\\ j_1&\ldots&j_k\end{pmatrix}e_{j_1}\otimes\ldots\otimes e_{j_k}\\
&=&\sum_{j_1\ldots j_k}\delta_{i_1j_1}\ldots\delta_{i_kj_k}e_{j_1}\otimes\ldots\otimes e_{j_k}\\
&=&e_{i_1}\otimes\ldots\otimes e_{i_k}
\end{eqnarray*}

(4) In the case of the basic crossing $\slash\hskip-2.0mm\backslash\in P_2(\circ\circ,\circ\circ)$, the corresponding linear map $T_{\slash\hskip-1.5mm\backslash}:\mathbb C^N\otimes\mathbb C^N\to\mathbb C^N\otimes\mathbb C^N$ can be computed as follows:
\begin{eqnarray*}
T_{\slash\hskip-1.5mm\backslash}(e_i\otimes e_j)
&=&\sum_{kl}\delta_{\slash\hskip-1.5mm\backslash}\begin{pmatrix}i&j\\ k&l\end{pmatrix}e_k\otimes e_l\\
&=&\sum_{kl}\delta_{il}\delta_{jk}e_k\otimes e_l\\
&=&e_j\otimes e_i
\end{eqnarray*}

Thus we obtain the flip operator $\Sigma(a\otimes b)=b\otimes a$, as claimed.
\end{proof}

Summarizing, the correspondence $\pi\to T_\pi$ provides us with some simple formulae for the operators $R,R^*$ that we are interested in, and for other important operators, such as the flip $\Sigma(a\otimes b)=b\otimes a$, and has as well some interesting categorical properties. 

\bigskip

Let us further explore these properties, and make the link with the Tannakian categories. We have the following key result, from \cite{bsp}:

\begin{proposition}
The assignement $\pi\to T_\pi$ is categorical, in the sense that we have
$$T_\pi\otimes T_\sigma=T_{[\pi\sigma]}$$
$$T_\pi T_\sigma=N^{c(\pi,\sigma)}T_{[^\sigma_\pi]}$$
$$T_\pi^*=T_{\pi^*}$$
where $c(\pi,\sigma)$ are certain integers, coming from the erased components in the middle.
\end{proposition}

\begin{proof}
The formulae in the statement are all elementary, as follows:

\medskip

(1) The concatenation axiom follows from the following computation:
\begin{eqnarray*}
&&(T_\pi\otimes T_\sigma)(e_{i_1}\otimes\ldots\otimes e_{i_p}\otimes e_{k_1}\otimes\ldots\otimes e_{k_r})\\
&=&\sum_{j_1\ldots j_q}\sum_{l_1\ldots l_s}\delta_\pi\begin{pmatrix}i_1&\ldots&i_p\\j_1&\ldots&j_q\end{pmatrix}\delta_\sigma\begin{pmatrix}k_1&\ldots&k_r\\l_1&\ldots&l_s\end{pmatrix}e_{j_1}\otimes\ldots\otimes e_{j_q}\otimes e_{l_1}\otimes\ldots\otimes e_{l_s}\\
&=&\sum_{j_1\ldots j_q}\sum_{l_1\ldots l_s}\delta_{[\pi\sigma]}\begin{pmatrix}i_1&\ldots&i_p&k_1&\ldots&k_r\\j_1&\ldots&j_q&l_1&\ldots&l_s\end{pmatrix}e_{j_1}\otimes\ldots\otimes e_{j_q}\otimes e_{l_1}\otimes\ldots\otimes e_{l_s}\\
&=&T_{[\pi\sigma]}(e_{i_1}\otimes\ldots\otimes e_{i_p}\otimes e_{k_1}\otimes\ldots\otimes e_{k_r})
\end{eqnarray*}

(2) The composition axiom follows from the following computation:
\begin{eqnarray*}
&&T_\pi T_\sigma(e_{i_1}\otimes\ldots\otimes e_{i_p})\\
&=&\sum_{j_1\ldots j_q}\delta_\sigma\begin{pmatrix}i_1&\ldots&i_p\\j_1&\ldots&j_q\end{pmatrix}
\sum_{k_1\ldots k_r}\delta_\pi\begin{pmatrix}j_1&\ldots&j_q\\k_1&\ldots&k_r\end{pmatrix}e_{k_1}\otimes\ldots\otimes e_{k_r}\\
&=&\sum_{k_1\ldots k_r}N^{c(\pi,\sigma)}\delta_{[^\sigma_\pi]}\begin{pmatrix}i_1&\ldots&i_p\\k_1&\ldots&k_r\end{pmatrix}e_{k_1}\otimes\ldots\otimes e_{k_r}\\
&=&N^{c(\pi,\sigma)}T_{[^\sigma_\pi]}(e_{i_1}\otimes\ldots\otimes e_{i_p})
\end{eqnarray*}

(3) Finally, the involution axiom follows from the following computation:
\begin{eqnarray*}
&&T_\pi^*(e_{j_1}\otimes\ldots\otimes e_{j_q})\\
&=&\sum_{i_1\ldots i_p}<T_\pi^*(e_{j_1}\otimes\ldots\otimes e_{j_q}),e_{i_1}\otimes\ldots\otimes e_{i_p}>e_{i_1}\otimes\ldots\otimes e_{i_p}\\
&=&\sum_{i_1\ldots i_p}\delta_\pi\begin{pmatrix}i_1&\ldots&i_p\\ j_1&\ldots& j_q\end{pmatrix}e_{i_1}\otimes\ldots\otimes e_{i_p}\\
&=&T_{\pi^*}(e_{j_1}\otimes\ldots\otimes e_{j_q})
\end{eqnarray*}

Summarizing, our correspondence is indeed categorical.
\end{proof}

We can now formulate a first non-trivial result regarding $O_N^+,U_N^+$, which is a Brauer type theorem for these quantum groups, as follows:

\index{Brauer theorem}
\index{orthogonal quantum group}
\index{unitary quantum group}
\index{noncrossing pairing}

\begin{theorem}
For the quantum groups $O_N^+,U_N^+$ we have
$$Hom(u^{\otimes k},u^{\otimes l})=span\left(T_\pi\Big|\pi\in D(k,l)\right)$$
with the sets on the right being respectively as follows, 
$$D=NC_2,\mathcal{NC}_2$$
and with the correspondence $\pi\to T_\pi$ being constructed as above.
\end{theorem}

\begin{proof}
We know from Proposition 4.22 above that the quantum group $U_N^+$ corresponds via Tannakian duality to the following category:
$$C=<R,R^*>$$

On the other hand, it follows from the above categorical considerations that this latter category is given by the following formula:
$$C=span\left(T_\pi\Big|\pi\in\mathcal{NC}_2\right)$$

To be more precise, consider the following collection of vector spaces:
$$C'=span\left(T_\pi\Big|\pi\in\mathcal{NC}_2\right)$$

According to the various formulae in Proposition 4.27, these vector spaces form a tensor category. But since the two matching semicircles generate the whole collection of matching pairings, via the operations in Proposition 4.27, we obtain from this $C=C'$.

\medskip

As for the result from $O_N^+$, this follows by adding to the picture the self-adjointness condition $u=\bar{u}$, which corresponds, at the level of pairings, to removing the colors.
\end{proof}

The above result is very useful, and virtually solves any question about $O_N^+,U_N^+$. We will be back to it in the next chapter, and afterwards, with applications, both of algebraic and analytic nature. As an example here, just by counting the dimensions of the spaces in Theorem 4.28, we will be able to compute the laws of the main characters.

\bigskip

By using the same methods, namely the general Tannakian duality result established above, we can recover as well the classical Brauer theorem \cite{bra}, as follows:

\index{Brauer theorem}
\index{orthogonal group}
\index{unitary group}
\index{pairing}
\index{matching pairing}

\begin{theorem}
For the groups $O_N,U_N$ we have
$$Hom(u^{\otimes k},u^{\otimes l})=span\left(T_\pi\Big|\pi\in D(k,l)\right)$$
with $D=P_2,\mathcal P_2$ respectively, and with $\pi\to T_\pi$ being constructed as above.
\end{theorem}

\begin{proof}
As already mentioned, this result is due to Brauer \cite{bra}, and is closely related to the Schur-Weyl duality \cite{wey}. There are several proofs of this result, one classical proof being via classical Tannakian duality, for the usual closed subgroups $G\subset U_N$. 

\medskip

In the present context, we can deduce this result from the one that we already have, for $O_N^+,U_N^+$. The idea is very simple, namely that of ``adding crossings'', as follows:

\medskip

(1) The group $U_N\subset U_N^+$ is defined via the following relations:
$$[u_{ij},u_{kl}]=0$$
$$[u_{ij},\bar{u}_{kl}]=0$$

But these relations which tell us that the following operators must be in the associated Tannakian category $C$:
$$T_\pi\quad,\quad \pi={\slash\hskip-2.1mm\backslash}^{\hskip-2.5mm\circ\circ}_{\hskip-2.5mm\circ\circ}$$
$$T_\pi\quad,\quad \pi={\slash\hskip-2.1mm\backslash}^{\hskip-2.5mm\circ\bullet}_{\hskip-2.5mm\bullet\circ}$$

Thus the associated Tannakian category is $C=span(T_\pi|\pi\in D)$, with:
$$D
=<\mathcal{NC}_2,{\slash\hskip-2.1mm\backslash}^{\hskip-2.5mm\circ\circ}_{\hskip-2.5mm\circ\circ},{\slash\hskip-2.1mm\backslash}^{\hskip-2.5mm\circ\bullet}_{\hskip-2.5mm\bullet\circ}>
=\mathcal P_2$$

Thus, we are led to the conclusion in the statement.

\medskip

(2) In order to deal now with $O_N$, we can simply use the following formula: 
$$O_N=O_N^+\cap U_N$$

At the categorical level, this tells us that the associated Tannakian category is given by $C=span(T_\pi|\pi\in D)$, with:
$$D
=<NC_2,\mathcal P_2>
=P_2$$

Thus, we are led to the conclusion in the statement.
\end{proof}

Summarizing, the orthogonal and unitary groups $O_N,U_N$ and their free analogues $O_N^+,U_N^+$ appear to be ``easy'', in the sense that their associated Tannakian categories appear in the simplest possible way, namely from certain categories of pairings.   

\bigskip

We will be exploit this phenomenon in chapters 5-6 below, with a detailed algebraic and analytic study of these quantum groups, based on their ``easiness'' property. Then, we will be back to this in chapter 7 below, with an axiomatization of the notion of category of pairings, or more generally of a category of partitions, a definition for easiness, some theory, and an exploration of the main examples of easy quantum groups.

\section*{4e. Exercises}

Generally speaking, the best complement to the material presented in this chapter is more reading on Tannakian duality, in its various versions, which are all useful. With the technology presented above, however, we can work out a few interesting particular cases of the Tannakian duality, and this will be the purpose of the first few exercises that we have here. Let us start with something quite elementary:

\begin{exercise}
Work out the Tannakian duality for the closed subgroups
$$G\subset O_N^+$$
first as a consequence of the general results that we have, regarding the closed subgroups 
$$G\subset U_N^+$$
and then independently, by pointing out the simplifications that appear in the real case.
\end{exercise}

Regarding the first question, this is normally something quite quick, obtained by adding the assumption $u=\bar{u}$ to the Tannakian statement that we have, and then working out the details. Regarding the second question, the idea here is basically that the colored exponents $k,l=\circ\bullet\bullet\circ\ldots$ will become in this way usual exponents, $k,l\in\mathbb N$, and this brings a number of simplifications in the proof, which are to be found.

\begin{exercise}
Work out the Tannakian duality for the closed subgroups
$$G\subset U_N^+$$
whose fundamental corepresentation is self-adjoint, up to equivalence,
$$u\sim\bar{u}$$
first as a consequence of the results that we have, and then independently.
\end{exercise}

Here are there are several possible paths, either by proceeding a bit as for the previous exercise, but with the condition $u=\bar{u}$ there replaced by the more general condition $u\sim\bar{u}$, or by using what was done in the previous exercise, and generalizing, from $u=\bar{u}$ to $u\sim\bar{u}$. In any case, regardless of the method which is chosen, the problem is that understanding what the condition $u\sim\bar{u}$ really means, categorially speaking.

\begin{exercise}
Work out the Tannakian duality for the closed subgroups
$$G\subset U_N$$
first as a consequence of the results that we have, and then independently.
\end{exercise}

The same comments as before apply. Some supplementary questions appear along these lines, regarding the closed subgroups $G\subset O_N$, or more generally the closed subgroups $G\subset U_N$ satisfying $u\sim\bar{u}$. Thus, there are in fact many questions here. In addition to this, looking a bit at the Tannakian duality literature for the compact Lie groups is definitely a very good idea, and the best possible exercise on the subject.

\begin{exercise}
Work out the Tannakian duality for the group duals
$$\widehat{\Gamma}\subset U_N$$
first as a consequence of the results that we have, and then independently.
\end{exercise}

This is actually the simplest exercise in the whole series, and the problem here is that of writing down a clear statement, along with a full, independent proof.

\begin{exercise}
Work out the Tannakian duality for the arbitrary group duals
$$\widehat{\Gamma}\subset U_N$$
first as a consequence of the results that we have, and then independently.
\end{exercise}

The key word here, which distinguishes this exercise from the previous one, is the word ``arbitrary''. Thus, in practice, we must go back here to the Peter-Weyl theory developed in chapter 3 above, see what happens exactly for the arbitrary group duals, and then go ahead and solve the above Tannakian question, a bit as before. 

\begin{exercise}
Check the Brauer theorems for $O_N,U_N$, which are both of type
$$Hom(u^{\otimes k},u^{\otimes l})=span\left(T_\pi\Big|\pi\in D(k,l)\right)$$
for small values of the global length parameter, $k+l\in\{1,2,3\}$.
\end{exercise}

The idea here is to prove these results that we already know directly, by double inclusion, with the inclusion in one sense being normally something quite elementary, and with the inclusion in the other sense being probably somehing quite tricky.

\begin{exercise}
Write down Brauer theorems for the quantum groups $O_N^*,U_N^*$, by identifying first the pairing which produces them, as subgroups of $O_N^+,U_N^+$.
\end{exercise}

This is actually something that will be discussed later on in this book, but without too much details, so the answer ``done in the book'' will not do.

\part{Quantum rotations}

\ \vskip50mm

\begin{center}
{\em And there's nothing short of dying

Half as lonesome as the sound

On the sleeping city sidewalks

Sunday morning coming down}
\end{center}

\chapter{Free rotations}

\section*{5a. Gram determinants}

We have seen that Tannakian duality allows us to get some substantial insight into the representation theory of $O_N^+,U_N^+$, with a free analogue of the classical Brauer theorem for $O_N,U_N$. In this second part of the present book we discuss some concrete applications of this result, and some generalizations, following \cite{ba1}, then \cite{bc1}, and then \cite{bsp}.

\bigskip

Let us begin with a summary of the Brauer type results established in the previous chapter. The statement here, collecting what we have so far, is as follows:

\index{Brauer theorem}

\begin{theorem}
For the basic unitary quantum groups, namely
$$\xymatrix@R=15mm@C=15mm{
U_N\ar[r]&U_N^+\\
O_N\ar[r]\ar[u]&O_N^+\ar[u]}$$
the intertwiners between the Peter-Weyl representations are given by
$$Hom(u^{\otimes k},u^{\otimes l})=span\left(T_\pi\Big|\pi\in D(k,l)\right)$$
with the linear maps $T_\pi$ associated to the pairings $\pi$ being given by
$$T_\pi(e_{i_1}\otimes\ldots\otimes e_{i_k})=\sum_{j_1\ldots j_l}\delta_\pi\begin{pmatrix}i_1&\ldots&i_k\\ j_1&\ldots&j_l\end{pmatrix}e_{j_1}\otimes\ldots\otimes e_{j_l}$$
and with the corresponding sets of pairings $D$ being as follows,
$$\xymatrix@R=16mm@C=15mm{
\mathcal P_2\ar[d]&\mathcal{NC}_2\ar[l]\ar[d]\\
P_2&NC_2\ar[l]}$$
with calligraphic standing for matching, and with NC standing for noncrossing.
\end{theorem}

\begin{proof}
This is indeed a summary of the results that we have, established in the previous chapter, and coming from Tannakian duality, via some combinatorics.
\end{proof}

\index{Gram matrix}
\index{Gram determinant}

In order to work out now some concrete applications, such as the classification of the irreducible representations of $O_N^+,U_N^+$, we must do some combinatorics.  The problem indeed is that we do not know whether the linear maps $T_\pi$ in Theorem 5.1 are linearly independent or not, so we must solve this problem first. Things are quite tricky here, technically speaking, and we will solve this question as follows:

\medskip
 
\begin{enumerate}
\item By Frobenius duality, it is enough to examine the vectors $\xi_\pi=T_\pi$ associated to the pairings $\pi\in P_2(0,l)$, having no upper points.

\medskip

\item In order to decide whether these vectors $\xi_\pi$ are linearly independent or not, we will compute the determinant of their Gram matrix.

\medskip

\item We will actually compute the determinant of a bigger Gram matrix, that of the vectors $\xi_\pi=T_\pi$ coming from arbitrary partitions $\pi\in P(0,l)$, which is simpler.
\end{enumerate}

\medskip

What we have here is an accumulation of tricks, and some changes in notations too. By replacing $l\to k$ as well, for making things look better, we are led in this way to:

\begin{proposition}
To any partition $\pi\in P(k)$ we associate the vector
$$\xi_\pi=\sum_{i_1\ldots i_k}\delta_\pi(i_1,\ldots,i_k)\,e_{i_1}\otimes\ldots\otimes e_{i_k}$$
with the Kronecker symbols being defined as usual, according to whether the indices fit or not. The Gram matrix of these vectors is then given by
$$G_k(\pi,\sigma)=N^{|\pi\vee\sigma|}$$
where $\pi\vee\sigma\in P(k)$ is obtained by superposing $\pi,\sigma$, and $|\,.\,|$ is the number of blocks.
\end{proposition}

\begin{proof}
According to the formula of the vectors $\xi_\pi$, we have:
\begin{eqnarray*}
<\xi_\pi,\xi_\sigma>
&=&\sum_{i_1\ldots i_k}\delta_\pi(i_1,\ldots,i_k)\delta_\sigma(i_1,\ldots,i_k)\\
&=&\sum_{i_1\ldots i_k}\delta_{\pi\vee\sigma}(i_1,\ldots,i_k)\\
&=&N^{|\pi\vee\sigma|}
\end{eqnarray*}

Thus, we have obtained the formula in the statement.
\end{proof}

Our goal in what follows will be that of computing $\det(G_k)$. Which will actually take some time, to the point that you might start wondering, sometimes soon, if this is really the right thing to do, right now. To which I would say that yes, this is the right thing to do. Our goal is to understand the closed subgroups $G\subset U_N^+$, and 0 chances with that, until we know what this $U_N^+$ beast is. And for this, we need to compute $\det(G_k)$.

\bigskip

As an illustration now, at $k=2$ we have $P(2)=\{||,\sqcap\}$, and the Gram matrix is:
$$G_2=\begin{pmatrix}N^2&N\\ N&N\end{pmatrix}$$

At $k=3$, we have $P(3)=\{|||,\sqcap|,\sqcap\hskip-3.2mm{\ }_|\,,|\sqcap,\sqcap\hskip-0.7mm\sqcap\}$, and the Gram matrix is:
$$G_3=\begin{pmatrix}
N^3&N^2&N^2&N^2&N\\
N^2&N^2&N&N&N\\
N^2&N&N^2&N&N\\
N^2&N&N&N^2&N\\
N&N&N&N&N
\end{pmatrix}$$

\index{Young tableaux}

These matrices might not look that bad, to the untrained eye, but in practice, their combinatorics can be fairly complicated. As an example here, the submatrix of $G_k$ coming from the usual pairings, that we are really interested in, according to Theorem 5.1, has as determinant a product of  terms indexed by Young tableaux. This is actually why we use $G_k$, because, as we will soon discover, this matrix is something quite simple. 

\bigskip

In order to compute the determinant of $G_k$, we will use a standard combinatorial trick, related to the M\"obius inversion formula. Let us start with:

\index{order of partitions}
\index{lattice of partitions}

\begin{definition}
Given two partitions $\pi,\sigma\in P(k)$, we write 
$$\pi\leq\sigma$$
if each block of $\pi$ is contained in a block of $\sigma$.
\end{definition}

Observe that this order is compatible with the previous convention for $\pi\vee\sigma$, in the sense that the $\vee$ operation is the supremum operation with respect to $\leq$. At the level of examples, at $k=2$ we have $P(2)=\{||,\sqcap\}$, and the order relation is as follows:
$$||\leq\sqcap$$

At $k=3$ now, we have $P(3)=\{|||,\sqcap|,\sqcap\hskip-3.2mm{\ }_|\,,|\sqcap,\sqcap\hskip-0.7mm\sqcap\}$, and the order relation is:
$$|||\leq\sqcap|,\sqcap\hskip-3.2mm{\ }_|\,,|\sqcap\leq\sqcap\hskip-0.7mm\sqcap$$

Summarizing, this order is very intuitive, and simple to compute. By using now this order, we can talk about the M\"obius function of $P(k)$, as follows:

\index{M\"obius function}
\index{M\"obius matrix}

\begin{definition}
The M\"obius function of any lattice, and so of $P(k)$, is given by
$$\mu(\pi,\sigma)=\begin{cases}
1&{\rm if}\ \pi=\sigma\\
-\sum_{\pi\leq\tau<\sigma}\mu(\pi,\tau)&{\rm if}\ \pi<\sigma\\
0&{\rm if}\ \pi\not\leq\sigma
\end{cases}$$
with this construction being performed by recurrence.
\end{definition}

This is something standard in combinatorics. As an illustration here, let us go back to the set of 2-point partitions, $P(2)=\{||,\sqcap\}$. We have by definition:
$$\mu(||,||)=\mu(\sqcap,\sqcap)=1$$

Next in line, we know that we have $||<\sqcap$, with no intermediate partition in between, and so the above recurrence procedure gives:
$$\mu(||,\sqcap)=-\mu(||,||)=-1$$

Finally, we have $\sqcap\not\leq||$, and so the last value of the M\"obius function is:
$$\mu(\sqcap,||)=0$$

Thus, as a conclusion, we have computed the M\"obius matrix $M_2(\pi,\sigma)=\mu(\pi,\sigma)$ of the lattice $P(2)=\{||,\sqcap\}$, the formula of this matrix being as follows:
$$M_2=\begin{pmatrix}1&-1\\ 0&1\end{pmatrix}$$

The computation for $P(3)=\{|||,\sqcap|,\sqcap\hskip-3.2mm{\ }_|\,,|\sqcap,\sqcap\hskip-0.7mm\sqcap\}$ is similar, and leads to the following formula for the associated M\"obius matrix:
$$M_3=\begin{pmatrix}
1&-1&-1&-1&2\\
0&1&0&0&-1\\
0&0&1&0&-1\\
0&0&0&1&-1\\
0&0&0&0&1
\end{pmatrix}$$

In general, the M\"obius matrix of $P(k)$ looks a bit like the above matrices at $k=2,3$, being upper triangular, with 1 on the diagonal, and so on. We will be back to this.

\bigskip

\index{M\"obius inversion}

Back to the general case now, the main interest in the M\"obius function comes from the M\"obius inversion formula, which states that the following happens:
$$f(\sigma)=\sum_{\pi\leq\sigma}g(\pi)
\quad\implies\quad g(\sigma)=\sum_{\pi\leq\sigma}\mu(\pi,\sigma)f(\pi)$$

In linear algebra terms, the statement and proof of this formula are as follows:

\index{adjacency matrix}

\begin{theorem}
The inverse of the adjacency matrix of $P(k)$, given by
$$A_k(\pi,\sigma)=\begin{cases}
1&{\rm if}\ \pi\leq\sigma\\
0&{\rm if}\ \pi\not\leq\sigma
\end{cases}$$
is the M\"obius matrix of $P$, given by $M_k(\pi,\sigma)=\mu(\pi,\sigma)$.
\end{theorem}

\begin{proof}
This is well-known, coming for instance from the fact that $A_k$ is upper triangular. Indeed, when inverting, we are led into the recurrence from Definition 5.4.
\end{proof}

As an illustration, for $P(2)=\{||,\sqcap\}$ the formula $M_2=A_2^{-1}$ appears as follows:
$$\begin{pmatrix}1&-1\\ 0&1\end{pmatrix}=
\begin{pmatrix}1&1\\ 0&1\end{pmatrix}^{-1}$$

Also, for $P(3)=\{|||,\sqcap|,\sqcap\hskip-3.2mm{\ }_|\,,|\sqcap,\sqcap\hskip-0.7mm\sqcap\}$ the formula $M_3=A_3^{-1}$ reads:
$$\begin{pmatrix}
1&-1&-1&-1&2\\
0&1&0&0&-1\\
0&0&1&0&-1\\
0&0&0&1&-1\\
0&0&0&0&1
\end{pmatrix}=
\begin{pmatrix}
1&1&1&1&1\\
0&1&0&0&1\\
0&0&1&0&1\\
0&0&0&1&1\\
0&0&0&0&1
\end{pmatrix}^{-1}$$

Now back to our Gram matrix considerations, we have the following key result, based on this technology, which basically solves our determinant question:

\begin{proposition}
The Gram matrix is given by $G_k=A_kL_k$, where
$$L_k(\pi,\sigma)=
\begin{cases}
N(N-1)\ldots(N-|\pi|+1)&{\rm if}\ \sigma\leq\pi\\
0&{\rm otherwise}
\end{cases}$$
and where $A_k=M_k^{-1}$ is the adjacency matrix of $P(k)$.
\end{proposition}

\begin{proof}
We have the following computation, using Proposition 5.2:
\begin{eqnarray*}
G_k(\pi,\sigma)
&=&N^{|\pi\vee\sigma|}\\
&=&\#\left\{i_1,\ldots,i_k\in\{1,\ldots,N\}\Big|\ker i\geq\pi\vee\sigma\right\}\\
&=&\sum_{\tau\geq\pi\vee\sigma}\#\left\{i_1,\ldots,i_k\in\{1,\ldots,N\}\Big|\ker i=\tau\right\}\\
&=&\sum_{\tau\geq\pi\vee\sigma}N(N-1)\ldots(N-|\tau|+1)
\end{eqnarray*}

According now to the definition of $A_k,L_k$, this formula reads:
\begin{eqnarray*}
G_k(\pi,\sigma)
&=&\sum_{\tau\geq\pi}L_k(\tau,\sigma)\\
&=&\sum_\tau A_k(\pi,\tau)L_k(\tau,\sigma)\\
&=&(A_kL_k)(\pi,\sigma)
\end{eqnarray*}

Thus, we are led to the formula in the statement.
\end{proof}

As an illustration for the above result, at $k=2$ we have $P(2)=\{||,\sqcap\}$, and the above decomposition $G_2=A_2L_2$ appears as follows:
$$\begin{pmatrix}N^2&N\\ N&N\end{pmatrix}
=\begin{pmatrix}1&1\\ 0&1\end{pmatrix}
\begin{pmatrix}N^2-N&0\\N&N\end{pmatrix}$$

At $k=3$ now, we have $P(3)=\{|||,\sqcap|,\sqcap\hskip-3.2mm{\ }_|\,,|\sqcap,\sqcap\hskip-0.7mm\sqcap\}$, and the Gram matrix is:
$$G_3=\begin{pmatrix}
N^3&N^2&N^2&N^2&N\\
N^2&N^2&N&N&N\\
N^2&N&N^2&N&N\\
N^2&N&N&N^2&N\\
N&N&N&N&N
\end{pmatrix}$$

Regarding $L_3$, this can be computed by writing down the matrix $E_3(\pi,\sigma)=\delta_{\sigma\leq\pi}|\pi|$, and then replacing each entry by the corresponding polynomial in $N$. We reach to the conclusion that the product $A_3L_3$ is as follows, producing the above matrix $G_3$:
$$A_3L_3=\begin{pmatrix}
1&1&1&1&1\\
0&1&0&0&1\\
0&0&1&0&1\\
0&0&0&1&1\\
0&0&0&0&1
\end{pmatrix}
\begin{pmatrix}
N^3-3N^2+2N&0&0&0&0\\
N^2-N&N^2-N&0&0&0\\
N^2-N&0&N^2-N&0&0\\
N^2-N&0&0&N^2-N&0\\
N&N&N&N&N
\end{pmatrix}$$

In general, the formula $G_k=A_kL_k$ appears a bit in the same way, with $A_k$ being binary and upper triangular, and with $L_k$ depending on $N$, and being lower triangular.

\bigskip

We are led in this way to the following formula, due to Lindst\"om \cite{lin}:

\index{Gram matrix}
\index{Gram determinant}
\index{Lindst\"om formula}

\begin{theorem}
The determinant of the Gram matrix $G_k$ is given by
$$\det(G_k)=\prod_{\pi\in P(k)}\frac{N!}{(N-|\pi|)!}$$
with the convention that in the case $N<k$ we obtain $0$.
\end{theorem}

\begin{proof}
If we order $P(k)$ as usual, with respect to the number of blocks, and then lexicographically, then $A_k$ is upper triangular, and $L_k$ is lower triangular. Thus, we have:
\begin{eqnarray*}
\det(G_k)
&=&\det(A_k)\det(L_k)\\
&=&\det(L_k)\\
&=&\prod_\pi L_k(\pi,\pi)\\
&=&\prod_\pi N(N-1)\ldots(N-|\pi|+1)
\end{eqnarray*}

Thus, we are led to the formula in the statement.
\end{proof}

We refer to \cite{bcu} for more on Gram determinants, and their conceptual meaning, from a modern perspective. Getting back now to quantum groups, or rather to the corresponding Tannakian categories, written as spans of diagrams, we have the following result:

\index{linear independence}

\begin{theorem}
The vectors associated to the partitions, namely
$$\left\{\xi_\pi\in(\mathbb C^N)^{\otimes k}\Big|\pi\in P(k)\right\}$$
and in particular the vectors associated to the pairings, namely
$$\left\{\xi_\pi\in(\mathbb C^N)^{\otimes k}\Big|\pi\in P_2(k)\right\}$$
are linearly independent for $N\geq k$. 
\end{theorem}

\begin{proof}
Here the first assertion follows from Theorem 5.7, the Gram determinant computed there being nonzero for $N\geq k$, and the second assertion follows from it.
\end{proof}

In what follows, the above result will be all that we need, for deducing a number of interesting consequences regarding $O_N,U_N,O_N^+,U_N^+$. Once these corollaries exhausted, we will have to go back to this, and work out some finer linear independence results.

\section*{5b. The Wigner law}

We discuss here some applications of the above linear independence results, following \cite{ba1}, \cite{bc1}. As a first application, we can study the laws of characters. First, we have:

\begin{proposition}
For the basic unitary quantum groups, namely
$$\xymatrix@R=15mm@C=15mm{
U_N\ar[r]&U_N^+\\
O_N\ar[r]\ar[u]&O_N^+\ar[u]}$$
the moments of the main character, which are the numbers $M_k=\int_G\chi^k$, depending on a colored integer $k$, are smaller than the following numbers,
$$\xymatrix@R=16mm@C=15mm{
|\mathcal P_2(k)|\ar@{-}[d]&|\mathcal{NC}_2(k)|\ar@{-}[l]\ar@{-}[d]\\
|P_2(k)|&|NC_2(k)|\ar@{-}[l]}$$
and with equality happening in each case at $N\geq k$.
\end{proposition}

\begin{proof}
We have the following computation, based on Theorem 5.1, and on the character formulae from Peter-Weyl theory, for each of our quantum groups:
\begin{eqnarray*}
\int_G\chi^k
&=&\dim(Fix(u^{\otimes k}))\\
&=&\dim\left(span\left(\xi_\pi\Big|\pi\in D(k)\right)\right)\\
&\leq&|D(k)|
\end{eqnarray*}

Thus, we have the inequalities in the statement, coming from easiness and Peter-Weyl. As for the last assertion, this follows from Theorem 5.8.
\end{proof}

In order to advance now, we must do some combinatorics and probability, first by counting the numbers in Proposition 5.9, and then by recovering the measures having these numbers as moments. We will restrict the attention to the orthogonal case, which is simpler, and leave the unitary case, which is more complicated, for later.

\bigskip

Since there are no pairings when $k$ is odd, we can assume that $k$ is even, and with the change $k\to 2k$, the partition count in the orthogonal case is as follows:

\index{double factorial}
\index{Catalan numbers}

\begin{proposition}
We have the following formulae for pairings,
$$|P_2(2k)|=(2k)!!$$
$$|NC_2(2k)|=C_k$$
with the numbers involved, double factorials and Catalan numbers, being as follows:
$$(2k)!!=(2k-1)(2k-3)(2k-5)\ldots$$
$$C_k=\frac{1}{k+1}\binom{2k}{k}$$
\end{proposition}

\begin{proof}
We have two assertions here, the idea being as follows:

\medskip

(1) We must count the pairings of $\{1,\ldots,2k\}$. Now observe that such a pairing appears by pairing 1 to a certain number, and there are $2k-1$ choices here, then pairing the next number, 2 if free or 3 if 2 was taken, to another number, and there are $2k-3$ choices here, and so on. Thus, we are led to the formula in the statement, namely:
$$|P_2(2k)|=(2k-1)(2k-3)(2k-5)\ldots$$

(2) We must count the noncrossing pairings of $\{1,\ldots,2k\}$. Now observe that such a pairing appears by pairing 1 to an odd number, $2a+1$, and then inserting a noncrossing pairing of $\{2,\ldots,2a\}$, and a noncrossing pairing of $\{2a+2,\ldots,2k\}$. We conclude from this that we have the following recurrence for the numbers $C_k=|NC_2(2k)|$:
$$C_k=\sum_{a+b=k-1}C_aC_b$$ 

Consider now the generating series of these numbers:
$$f(z)=\sum_{k\geq0}C_kz^k$$

In terms of this generating series, the recurrence that we found gives:
\begin{eqnarray*}
zf^2
&=&\sum_{a,b\geq0}C_aC_bz^{a+b+1}\\
&=&\sum_{k\geq1}\sum_{a+b=k-1}C_aC_bz^k\\
&=&\sum_{k\geq1}C_kz^k\\
&=&f-1
\end{eqnarray*}

Thus the generating series satisfies the following degree 2 equation:
$$zf^2-f+1=0$$

Now by solving this equation, using the usual degree 2 formula, and choosing the solution which is bounded at $z=0$, we obtain:
$$f(z)=\frac{1-\sqrt{1-4z}}{2z}$$ 

By using now the Taylor formula for $\sqrt{x}$, we obtain the following formula:
$$f(z)=\sum_{k\geq0}\frac{1}{k+1}\binom{2k}{k}z^k$$

Thus, we are led to the conclusion in the statement.
\end{proof}

Let us do now the second computation, which is probabilistic. We must find the real probability measures having the above numbers as moments, and we have here:

\index{normal law}
\index{Gaussian law}
\index{Wigner law}
\index{semicircle law}

\begin{theorem}
The standard Gaussian law, and standard Wigner semicircle law
$$g_1=\frac{1}{\sqrt{2\pi}}e^{-x^2/2}dx$$
$$\gamma_1=\frac{1}{2\pi}\sqrt{4-x^2}dx$$
have as $2k$-th moments the numbers $(2k)!!$ and $C_k$, and their odd moments vanish.
\end{theorem}

\begin{proof}
There are several proofs here, depending on your calculus and probability knowledge. Normally the ``honest'', white belt proof would be by trying to find centered measures $g_1,\gamma_1$ having as even moments the numbers $(2k)!!$ and $C_k$. But this is something quite complicated, requiring the usage of the Stieltjes inversion formula, namely:
$$d\mu (x)=\lim_{t\searrow 0}-\frac{1}{\pi}\,Im\left(G(x+it)\right)\cdot dx$$

\index{Stieltjes inversion}

Now the problem is that, assuming that you master this formula, you have certainly learned enough probability as to know about the solutions $g_1,\gamma_1$ to our problem. So, we will just cheat, assume that the laws $g_1,\gamma_1$ are found, and proceed as follows:

\medskip

(1) The moments of the normal law $g_1$ in the statement are given by:
\begin{eqnarray*}
M_k
&=&\frac{1}{\sqrt{2\pi}}\int_\mathbb Rx^ke^{-x^2/2}dx\\
&=&\frac{1}{\sqrt{2\pi}}\int_\mathbb R(x^{k-1})\left(-e^{-x^2/2}\right)'dx\\
&=&\frac{1}{\sqrt{2\pi}}\int_\mathbb R(k-1)x^{k-2}e^{-x^2/2}dx\\
&=&(k-1)\times\frac{1}{\sqrt{2\pi}}\int_\mathbb Rx^{k-2}e^{-x^2/2}dx\\
&=&(k-1)M_{k-2}
\end{eqnarray*}

Thus by recurrence we have $M_{2k}=(2k)!!$, and we are done.

\medskip

(2) The moments of the Wigner law $\gamma_1$ in the statement are given by:
\begin{eqnarray*}
N_k
&=&\frac{1}{2\pi}\int_{-2}^2\sqrt{4-x^2}\,x^{2k}dx\\
&=&\frac{1}{2\pi}\int_0^\pi\sqrt{4-4\cos^2t}\,(2\cos t)^{2k}(2\sin t)dt\\
&=&\frac{2^{2k+1}}{\pi}\int_0^\pi\cos^{2k}t\sin^2tdt\\
&=&\frac{2^{2k+1}}{\pi}\cdot\frac{(2k)!!2!!}{(2k+3)!!}\cdot\pi\\
&=&2^{2k+1}\cdot\frac{3\cdot5\cdot7\ldots(2k-1)}{2\cdot 4\cdot6\ldots(2k+2)}\\
&=&2^{2k+1}\cdot\frac{(2k)!}{2^kk!2^{k+1}(k+1)!}\\
&=&\frac{(2k)!}{k!(k+1)!}
\end{eqnarray*}

Here we have used an advanced calculus formula, but a routine computation based on partial integration works as well. Thus we have $N_k=C_k$, and we are done.
\end{proof}

As a comment here, the advanced calculus formula used in (2) above is as follows, with $\varepsilon(p)=1$ if $p$ is even and $\varepsilon(p)=0$ if $p$ is odd, and with $m!!=(m-1)(m-3)(m-5)\ldots$, with the product ending at $2$ if $m$ is odd, and ending at $1$ if $m$ is even:
$$\int_0^{\pi/2}\cos^pt\sin^qt\,dt=\left(\frac{\pi}{2}\right)^{\varepsilon(p)\varepsilon(q)}\frac{p!!q!!}{(p+q+1)!!}$$

This formula is something extremely useful, in everyday life, with the proof being by partial integration, and then a double recurrence on $p,q$. With spherical coordinates and Fubini it is possible to generalize this into an integration formula over the arbitrary real spheres $S^{N-1}_\mathbb R$, in arbitrary dimension $N\in\mathbb N$, but more on this later.

\bigskip

Now back to our orthogonal quantum groups, by using the above we can formulate a clear and concrete result regarding them, as follows:

\index{asymptotic character}
\index{normal law}
\index{Wigner law}
\index{Gaussian law}
\index{semicircle law}

\begin{theorem}
For the quantum groups $O_N,O_N^+$, the main character
$$\chi=\sum_iu_{ii}$$
follows respectively the standard Gaussian, and the Wigner semicircle law
$$g_1=\frac{1}{\sqrt{2\pi}}e^{-x^2/2}dx\quad,\quad 
\gamma_1=\frac{1}{2\pi}\sqrt{4-x^2}dx$$
in the $N\to\infty$ limit.
\end{theorem}

\begin{proof}
This follows by putting together the results that we have, namely Proposition 5.9 applied with $N>k$, and then Proposition 5.10 and Theorem 5.11.
\end{proof}

The above result is quite interesting, making the link with the law of Wigner \cite{wig}. Note also that this is the first application of our Tannakian duality methods, developed in chapter 4. We will see in what follows countless versions and generalizations of it, basically obtained by using the same method, Tannakian duality and easiness first, then combinatorics for linear independence, and then more combinatorics and probability.

\section*{5c. Clebsch-Gordan rules}

\index{Young tableaux}

Let us try now to work out some finer results, at fixed values of $N\in\mathbb N$. In the case of $O_N$ the above result cannot really be improved, the fixed $N\in\mathbb N$ laws being fairly complicated objects, related to Young tableaux and their combinatorics. 

\bigskip

In the case of $O_N^+$, however, we will see that some miracles happen, and the convergence in the above result is in fact stationary, starting from $N=2$. Following \cite{ba1}, we have:

\index{Wigner law}
\index{semicircle law}
\index{Clebsch-Gordan rules}
\index{fusion rules}
\index{orthogonal quantum group}

\begin{theorem}
For the quantum group $O_N^+$, the main character follows the standard Wigner semicircle law, and this regardless of the value of $N\geq 2$:
$$\chi\sim\frac{1}{2\pi}\sqrt{4-x^2}dx$$
The irreducible representations of $O_N^+$ are all self-adjoint, and can be labelled by positive integers, with their fusion rules being the Clebsch-Gordan ones,
$$r_k\otimes r_l=r_{|k-l|}+r_{|k-l|+2}+\ldots+r_{k+l}$$
as for the group $SU_2$. The dimensions of these representations are given by
$$\dim r_k=\frac{q^{k+1}-q^{-k-1}}{q-q^{-1}}$$
where $q,q^{-1}$ are the solutions of $X^2-NX+1=0$.
\end{theorem}

\begin{proof}
There are several proofs for this fact, the simplest one being via purely algebraic methods, based on the easiness property of $O_N^+$ from Theorem 5.1:

\medskip

(1) In order to get started, let us first work out the first few values of the representations $r_k$ that we want to construct, computed by recurrence, according to the Clebsch-Gordan rules in the statement, which will be useful for various illustrations:
$$r_0=1$$
$$r_1=u$$
$$r_2=u^{\otimes 2}-1$$
$$r_3=u^{\otimes 3}-2u$$
$$r_4=u^{\otimes 4}-3u^{\otimes 2}+1$$
$$r_5=u^{\otimes 5}-4u^{\otimes 3}+3u$$
$$\vdots$$

(2) We can see that what we want to do is to split the Peter-Weyl representations $u^{\otimes k}$ into irreducibles, because the above formulae can be written as well as follows:
$$u^{\otimes0}=r_0$$
$$u^{\otimes1}=r_1$$
$$u^{\otimes2}=r_2+r_0$$
$$u^{\otimes3}=r_3+2r_1$$
$$u^{\otimes4}=r_4+3r_2+2r_0$$
$$u^{\otimes5}=r_5+4r_3+5r_1$$
$$\vdots$$

(3) In order to get fully started now, our claim, which will basically prove the theorem, is that we can define, by recurrence on $k\in\mathbb N$, a sequence $r_0,r_1,r_2,\ldots$ of irreducible, self-adjoint and distinct representations of $O_N^+$, satisfying:
$$r_0=1$$
$$r_1=u$$
$$r_k+r_{k-2}=r_{k-1}\otimes r_1$$

(4) Indeed, at $k=0$ this is clear, and at $k=1$ this is clear as well, with the irreducibility of $r_1=u$ coming from the embedding $O_N\subset O_N^+$. So assume now that $r_0,\ldots,r_{k-1}$ as above are constructed, and let us construct $r_k$. We have, by recurrence:
$$r_{k-1}+r_{k-3}=r_{k-2}\otimes r_1$$

In particular we have an inclusion of representations, as follows:
$$r_{k-1}\subset r_{k-2}\otimes r_1$$

Now since $r_{k-2}$ is irreducible, by Frobenius reciprocity we have:
$$r_{k-2}\subset r_{k-1}\otimes r_1$$

Thus, there exists a certain representation $r_k$ such that:
$$r_k+r_{k-2}=r_{k-1}\otimes r_1$$

(5) As a first observation, this representation $r_k$ is self-adjoint. Indeed, our recurrence formula $r_k+r_{k-2}=r_{k-1}\otimes r_1$ for the representations $r_0,r_1,r_2,\ldots$ shows that the characters of these representations are polynomials in $\chi_u$. Now since $\chi_u$ is self-adjoint, all the characters that we can obtain via our recurrence are self-adjoint as well.

\medskip

(6) It remains to prove that $r_k$ is irreducible, and non-equivalent to $r_0,\ldots,r_{k-1}$. For this purpose, observe that according to our recurrence formula, $r_k+r_{k-2}=r_{k-1}\otimes r_1$, we can now split $u^{\otimes k}$, as a sum of the following type, with positive coefficients:  
$$u^{\otimes k}=c_kr_k+c_{k-2}r_{k-2}+\ldots$$

We conclude by Peter-Weyl that we have an inequality as follows, with equality precisely when $r_k$ is irreducible, and non-equivalent to the other summands $r_i$:
$$\sum_ic_i^2\leq\dim(End(u^{\otimes k}))$$

(7) Now let us use the easiness property of $O_N^+$. This gives us an upper bound for the number on the right, that we can add to our inequality, as follows:
$$\sum_ic_i^2
\leq\dim(End(u^{\otimes k}))
\leq C_k$$

The point now is that the coefficients $c_i$ come straight from the Clebsch-Gordan rules, and their combinatorics shows that $\sum_ic_i^2$ equals the Catalan number $C_k$, with the remark that this follows as well from the known theory of $SU_2$. Thus, we have global equality in the above estimate, and in particular we have equality at left, as desired.

\medskip

(8) In order to finish the proof of our claim, it still remains to prove that $r_k$ is non-equivalent to $r_{k-1},r_{k-3},\ldots$ But these latter representations appear inside $u^{\otimes k-1}$, and the result follows by using the embedding $O_N\subset O_N^+$, which shows that the even and odd tensor powers of $u$ cannot have common irreducible components.

\bigskip

(9) Summarizing, we have proved our claim, made in step (3) above.

\bigskip

(10) In order now to finish, since by the Peter-Weyl theory any irreducible representation of $O_N^+$ must appear in some tensor power of $u$, and we have a formula for decomposing each $u^{\otimes k}$ into sums of representations $r_k$, as explained above, we conclude that these representations $r_k$ are all the irreducible representations of $O_N^+$.

\medskip

(11) In what regards now the law of the main character, we obtain here the Wigner law $\gamma_1$, as stated, due to the fact that the equality in (7) gives us the even moments of this law, and that the observation in (8) tells us that the odd moments vanish.

\medskip

(12) Finally, from the Clebsch-Gordan rules we have in particular:
$$r_kr_1=r_{k-1}+r_{k+1}$$

We obtain from this, by recurrence, with $q^2-Nq+1=0$:
$$\dim r_k=q^k+q^{k-2}+\ldots+q^{-k+2}+q^{-k}$$

But this gives the dimension formula in the statement, and we are done.
\end{proof}

Let us discuss now the relation with $SU_2$. This group is the most well-known group in mathematics, and there is an enormous quantity of things known about it. For our purposes, we need a functional analytic approach to it. This can be done as follows:

\index{special unitary group}
\index{super-identity}

\begin{theorem}
The algebra of continuous functions on $SU_2$ appears as
$$C(SU_2)=C^*\left((u_{ij})_{i,j=1,2}\Big|u=F\bar{u}F^{-1}={\rm unitary}\right)$$
where $F$ is the following matrix,
$$F=\begin{pmatrix}0&1\\ -1&0\end{pmatrix}$$
called super-identity matrix. 
\end{theorem}

\begin{proof}
This can be done in several steps, as follows:

\medskip

(1) Let us first compute $SU_2$. Consider an arbitrary $2\times2$ complex matrix:
$$U=\begin{pmatrix}a&b\\c&d\end{pmatrix}$$

Assuming $\det U=1$, the unitarity condition $U^{-1}=U^*$ reads:
$$\begin{pmatrix}d&-b\\-c&a\end{pmatrix}
=\begin{pmatrix}\bar{a}&\bar{c}\\\bar{b}&\bar{d}\end{pmatrix}$$

Thus we must have $d=\bar{a}$, $c=-\bar{b}$, and we obtain the following formula:
$$SU_2=\left\{\begin{pmatrix}a&b\\-\bar{b}&\bar{a}\end{pmatrix}\Big|\ |a|^2+|b|^2=1\right\}$$

(2) With the above formula in hand, the fundamental corepresentation of $SU_2$ is:
$$u=\begin{pmatrix}a&b\\-\bar{b}&\bar{a}\end{pmatrix}$$

Now observe that we have the following equality:
$$\begin{pmatrix}a&b\\ -\bar{b}&\bar{a}\end{pmatrix}
\begin{pmatrix}0&1\\ -1&0\end{pmatrix}
=\begin{pmatrix}-b&a\\-\bar{a}&-\bar{b}\end{pmatrix}
=\begin{pmatrix}0&1\\ -1&0\end{pmatrix}
\begin{pmatrix}\bar{a}&\bar{b}\\ -b&a\end{pmatrix}$$

Thus, with $F$ being as in the statement, we have $uF=F\bar{u}$, and so:
$$u=F\bar{u}F^{-1}$$

We conclude that, if $A$ is the universal algebra in the statement, we have:
$$A\to C(SU_2)$$

(3) Conversely now, let us compute the universal algebra $A$ in the statement. For this purpose, let us write its fundamental corepresentation as follows:
$$u=\begin{pmatrix}a&b\\c&d\end{pmatrix}$$

We have $uF=F\bar{u}$, with these quantities being respectively given by:
$$uF=\begin{pmatrix}a&b\\c&d\end{pmatrix}
\begin{pmatrix}0&1\\ -1&0\end{pmatrix}
=\begin{pmatrix}-b&a\\-d&c\end{pmatrix}$$
$$F\bar{u}=\begin{pmatrix}0&1\\ -1&0\end{pmatrix}
\begin{pmatrix}a^*&b^*\\c^*&d^*\end{pmatrix}
=\begin{pmatrix}c^*&d^*\\-a^*&-b^*\end{pmatrix}$$

Thus we must have $d=a^*$, $c=-b^*$, and we obtain the following formula:
$$u=\begin{pmatrix}a&b\\-b^*&a^*\end{pmatrix}$$

We also know that this matrix must be unitary, and we have:
$$uu^*=\begin{pmatrix}a&b\\-b^*&a^*\end{pmatrix}
\begin{pmatrix}a^*&-b\\b^*&a\end{pmatrix}
=\begin{pmatrix}aa^*+bb^*&ba-ab\\a^*b^*-b^*a^*&a^*a+b^*b\end{pmatrix}$$
$$u^*u=\begin{pmatrix}a^*&-b\\b^*&a\end{pmatrix}
\begin{pmatrix}a&b\\-b^*&a^*\end{pmatrix}
=\begin{pmatrix}a^*a+bb^*&a^*b-ba^*\\b^*a-ab^*&aa^*+b^*b\end{pmatrix}$$

Thus, the unitarity equations for $u$ are as follows:
$$aa^*=a^*a=1-bb^*=1-b^*b$$
$$ab=ba,a^*b=ba^*,ab^*=a^*b,a^*b^*=b^*a^*$$

It follows that $a,b,a^*,b^*$ commute, so our algebra is commutative. Now since this algebra is commutative, the involution $*$ becomes the usual conjugation $-$, and so:
$$u=\begin{pmatrix}a&b\\-\bar{b}&\bar{a}\end{pmatrix}$$

But this tells us that we have $A=C(X)$ with $X\subset SU_2$, and so we have a quotient map $C(SU_2)\to A$, which is inverse to the map constructed in (2), as desired.
\end{proof}

Now with the above result in hand, we can see right away the relation with $O_N^+$, and more specifically with $O_2^+$. Indeed, this latter quantum group appears as follows:
$$C(O_2^+)=C^*\left((u_{ij})_{i,j=1,2}\Big|u=\bar{u}={\rm unitary}\right)$$

Thus, $SU_2$ appears from $O_2^+$ by replacing the identity with the super-identity, or perhaps vice versa. In any case, these two quantum groups are related by some ``twisting'' operation, so they should have similar representation theory. This is indeed the case:

\index{twisting}
\index{Wigner law}
\index{semicircle law}
\index{Clebsch-Gordan rules}

\begin{theorem}
For the group $SU_2$, the main character follows the Wigner law:
$$\chi\sim\frac{1}{2\pi}\sqrt{4-x^2}dx$$
The irreducible representations of $SU_2$ are all self-adjoint, and can be labelled by positive integers, with their fusion rules being the Clebsch-Gordan ones,
$$r_k\otimes r_l=r_{|k-l|}+r_{|k-l|+2}+\ldots+r_{k+l}$$
as for the quantum group $O_N^+$. The dimensions of these representations are given by
$$\dim r_k=k+1$$
exactly as for the quantum group $O_2^+$.
\end{theorem}

\begin{proof}
This result is as old as modern mathematics, with many proofs available, all instructive. Here is our take on the subject, in connection with what we do here:

\medskip

(1) A first proof, which is straightforward but rather long, is by taking everything that has been said so far about $O_N^+$, starting from the middle of chapter 4, setting $N=2$, and then twisting everything with the help of the super-identity matrix:
$$F=\begin{pmatrix}0&1\\ -1&0\end{pmatrix}$$

What happens then is that a Brauer theorem for $SU_2$ holds, involving the set $D=NC_2$ as before, but with the implementation of the partitions $\pi\to T_\pi$ being twisted by $F$. In particular, we obtain in this way, as before, inequalities as follows:
$$\dim(End(u^{\otimes k}))\leq C_k$$

But with such inequalities in hand, the proof of Theorem 5.13 applies virtually unchanged, and gives the result, with of course $q=1$ in the dimension formula.  

\medskip

(2) Here is as well a second proof, which is quite instructive. With $a=x+iy$, $b=z+it$, the formula for $SU_2$ that we found in the proof of Theorem 5.14 reads:
$$SU_2=\left\{\begin{pmatrix}x+iy&z+it\\-z+it&x-iy\end{pmatrix}\Big|\ x^2+y^2+z^2+t^2=1\right\}$$

Thus, $SU_2$ is isomorphic to the real unit sphere $S^3_\mathbb R\subset\mathbb R^4$. The point now is that the uniform measure on $SU_2$ corresponds in this way to the uniform measure on $S^3_\mathbb R$, and so in this picture, the moments of the main character of $SU_2$ are given by:
$$M_k=\int_{S^3_\mathbb R}(2x)^kd(x,y,z,t)$$

In order to compute now such integrals, we can use the following advanced calculus formula, valid for any exponents $k_i\in2\mathbb N$, which at $N=2$ corresponds to the advanced calculus formula mentioned after Theorem 5.11, and at $N\geq3$ comes as well from that advanced calculus formula, via spherical coordinates and Fubini:
$$\int_{S^{N-1}_\mathbb R}x_1^{k_1}\ldots x_N^{k_N}\,dx=\frac{(N-1)!!k_1!!\ldots k_N!!}{(N+\Sigma k_i-1)!!}$$

\index{spherical coordinates}

Indeed, by using this formula at $N=4$, we obtain:
\begin{eqnarray*}
\int_{S^3_\mathbb R}x_1^{2k}\,dx
&=&\frac{3!!(2k)!!}{(2k+3)!!}\\
&=&2\cdot\frac{3\cdot5\cdot7\ldots (2k-1)}{2\cdot4\cdot6\ldots (2k+2)}\\
&=&2\cdot\frac{(2k)!}{2^kk!2^{k+1}(k+1)!}\\
&=&\frac{C_k}{4^k}
\end{eqnarray*}

Thus the even moments of our character $\chi=2x_1$ are the Catalan numbers, $M_{2k}=C_k$, and since the odd moments vanish via $x\to-x$, we conclude that we have $\chi\sim\gamma_1$. But this formula, or rather the moment formula $M_{2k}=C_k$ it comes from, gives:
$$\dim(End(u^{\otimes k}))=C_k$$

Thus we can conclude as in the above first proof (1), by arguing that the recurrence construction of $r_k$ from the proof of Theorem 5.13 applies virtually unchanged, and gives the result, with of course $q=1$ in the dimension formula.  
\end{proof}

As a conclusion, we have two fringe proofs for the $SU_2$ result, one by crazy algebraists, and one by crazy probabilists. We recommend, as a complement, any of the proofs by geometers or physicists, which can be found in any good mathematical book. In fact $SU_2$, and also $SO_3$, are cult objects in geometry and physics, and there are countless things that can be said about them, and the more such things you know, the better your mathematics will be, no matter what precise mathematics you are interested in.

\section*{5d. Symplectic groups} 

Let us discuss now the unification of the $O_N^+$ and $SU_2$ results. In view of Theorem 5.14, and of the comments made afterwards, the idea is clear, namely that of looking at compact quantum groups appearing via relations of the following type: 
$$u=F\bar{u}F^{-1}={\rm unitary}$$

In order to clarify what exact matrices $F\in GL_N(\mathbb C)$ we can use, we must do some computations. Following \cite{ba1}, \cite{bsk}, \cite{bdv}, we first have the following result:

\begin{proposition}
Given a closed subgroup $G\subset U_N^+$, with irreducible fundamental corepresentation $u=(u_{ij})$, this corepresentation is self-adjoint, $u\sim\bar{u}$, precisely when 
$$u=F\bar{u}F^{-1}$$
for some unitary matrix $F\in U_N$, satisfying the following condition:
$$F\bar{F}=\pm 1$$
Moreover, when $N$ is odd we must have $F\bar{F}=1$. 
\end{proposition}

\begin{proof}
Since $u$ is self-adjoint, $u\sim\bar{u}$, we must have $u=F\bar{u}F^{-1}$, for a certain matrix $F\in GL_N(\mathbb C)$. We obtain from this, by using our assumption that $u$ is irreducible:
\begin{eqnarray*}
u=F\bar{u}F^{-1}
&\implies&\bar{u}=\bar{F}u\bar{F}^{-1}\\
&\implies&u=(F\bar{F})u(F\bar{F})^{-1}\\
&\implies&F\bar{F}=c1\\
&\implies&\bar{F}F=\bar{c}1\\
&\implies&c\in\mathbb R
\end{eqnarray*}

Now by rescaling we can assume $c=\pm1$, so we have proved so far that:
$$F\bar{F}=\pm 1$$

In order to establish now the formula $FF^*=1$, we can proceed as follows:
\begin{eqnarray*}
(id\otimes S)u=u^*
&\implies&(id\otimes S)\bar{u}=u^t\\
&\implies&(id\otimes S)(F\bar{u}F^{-1})=Fu^tF^{-1}\\
&\implies&u^*=Fu^tF^{-1}\\
&\implies&u=(F^*)^{-1}\bar{u}F^*\\
&\implies&\bar{u}=F^*u(F^*)^{-1}\\
&\implies&\bar{u}=F^*F\bar{u}F^{-1}(F^*)^{-1}\\
&\implies&FF^*=d1
\end{eqnarray*}

We have $FF^*>0$, so $d>0$. On the other hand, from $F\bar{F}=\pm 1$, $FF^*=d1$ we get:
$$|\det F|^2=\det(F\bar{F})=(\pm1)^N$$
$$|\det F|^2=\det(FF^*)=d^N$$

Since $d>0$ we obtain from this $d=1$, and so $FF^*=1$ as claimed. We obtain as well that when $N$ is odd the sign must be 1, and so $F\bar{F}=1$, as claimed.
\end{proof}

It is convenient to diagonalize $F$. Once again following Bichon-De Rijdt-Vaes \cite{bdv}, up to an orthogonal base change, we can assume that our matrix is as follows, where $N=2p+q$ and $\varepsilon=\pm 1$, with the $1_q$ block at right disappearing if $\varepsilon=-1$:
$$F=\begin{pmatrix}
0&1\ \ \ \\
\varepsilon 1&0_{(0)}\\
&&\ddots\\
&&&0&1\ \ \ \\
&&&\varepsilon 1&0_{(p)}\\
&&&&&1_{(1)}\\
&&&&&&\ddots\\
&&&&&&&1_{(q)}
\end{pmatrix}$$

We are therefore led into the following definition, from \cite{bsk}:

\index{super-space}
\index{super-identity}

\begin{definition}
The ``super-space'' $\mathbb C^N_F$ is the usual space $\mathbb C^N$, with its standard basis $\{e_1,\ldots,e_N\}$, with a chosen sign $\varepsilon=\pm 1$, and a chosen involution on the set of indices,
$$i\to\bar{i}$$
with $F$ being the ``super-identity'' matrix, $F_{ij}=\delta_{i\bar{j}}$ for $i\leq j$ and $F_{ij}=\varepsilon\delta_{i\bar{j}}$ for $i\geq j$.
\end{definition}

In what follows we will usually assume that $F$ is the explicit matrix appearing above. Indeed, up to a permutation of the indices, we have a decomposition $n=2p+q$ such that the involution is, in standard permutation notation:
$$(12)\ldots (2p-1,2p)(2p+1)\ldots (q)$$

Let us construct now some basic compact quantum groups, in our ``super'' setting. Once again following \cite{bsk}, let us formulate:

\index{super-orthogonal group}
\index{super-orthogonal quantum group}
\index{symplectic group}
\index{free symplectic group}

\begin{definition}
Associated to the super-space $\mathbb C^N_F$ are the following objects:
\begin{enumerate}
\item The super-orthogonal group, given by:
$$O_F=\left\{U\in U_N\Big|U=F\bar{U}F^{-1}\right\}$$

\item The super-orthogonal quantum group, given by:
$$C(O_F^+)=C^*\left((u_{ij})_{i,j=1,\ldots,n}\Big|u=F\bar{u}F^{-1}={\rm unitary}\right)$$
\end{enumerate}
\end{definition}

As explained in \cite{bsk}, it it possible to considerably extend this list, but for our purposes here, this is what we need for the moment. We have indeed the following result, from \cite{bsk}, making the connection with our unification problem for $O_N^+$ and $SU_2$:

\begin{theorem}
The basic orthogonal groups and quantum groups are as follows:
\begin{enumerate}
\item At $\varepsilon=-1$ we have $O_F=Sp_N$ and $O_F^+=Sp_N^+$.

\item At $\varepsilon=-1$ and $N=2$ we have $O_F=O_F^+=SU_2$.

\item At $\varepsilon=1$ we have $O_F=O_N$ and $O_F^+=O_N^+$.
\end{enumerate}
\end{theorem}

\begin{proof}
These results are all elementary, as follows:

\medskip

(1) At $\varepsilon=-1$ this follows from definitions, because the symplectic group $Sp_N\subset U_N$ is by definition the following group:
$$Sp_N=\left\{U\in U_N\Big|U=F\bar{U}F^{-1}\right\}$$

(2) Still at $\varepsilon=-1$, the equation $U=F\bar{U}F^{-1}$ tells us that the symplectic matrices $U\in Sp_N$ are exactly the unitaries $U\in U_N$ which are patterned as follows:
$$U=\begin{pmatrix}
a&b&\ldots\\
-\bar{b}&\bar{a}\\
\vdots&&\ddots
\end{pmatrix}$$

In particular, the symplectic matrices at $N=2$ are as follows:
$$U=\begin{pmatrix}
a&b\\
-\bar{b}&\bar{a}
\end{pmatrix}$$

Thus we have $Sp_2=U_2$, and the formula $Sp_2^+=Sp_2$ is elementary as well, via an analysis similar to the one in the proof of Theorem 5.14 above.

\medskip

(3) At $\varepsilon=1$ now, consider the root of unity $\rho=e^{\pi i/4}$, and set:
$$J=\frac{1}{\sqrt{2}}\begin{pmatrix}\rho&\rho^7\\ \rho^3&\rho^5\end{pmatrix}$$

This matrix $J$ is then unitary, and we have:
$$J\begin{pmatrix}0&1\\1&0\end{pmatrix}J^t=1$$

Thus the following matrix is unitary as well, and satisfies $KFK^t=1$:
$$K=\begin{pmatrix}J^{(1)}\\&\ddots\\&&J^{(p)}\\&&&1_q\end{pmatrix}$$

Thus in terms of the matrix $V=KUK^*$ we have:
$$U=F\bar{U}F^{-1}={\rm unitary}
\quad\iff\quad V=\bar{V}={\rm unitary}$$

We obtain in this way an isomorphism $O_F^+=O_N^+$ as in the statement, and by passing to classical versions, we obtain as well $O_F=O_N$, as desired.
\end{proof}

With the above formalism and results in hand, we can now formulate the unification result for $O_N^+$ and $SU_2$, which in complete form is as follows:

\index{free orthogonal group}
\index{Wigner law}
\index{semicircle law}
\index{Clebsch-Gordan rules}

\begin{theorem}
For the quantum group $O_F^+\in\{O_N^+,Sp_N^+\}$ with $N\geq2$, the main character follows the standard Wigner semicircle law,
$$\chi\sim\frac{1}{2\pi}\sqrt{4-x^2}dx$$
the irreducible representations are all self-adjoint, and can be labelled by positive integers, with their fusion rules being the Clebsch-Gordan ones,
$$r_k\otimes r_l=r_{|k-l|}+r_{|k-l|+2}+\ldots+r_{k+l}$$
and the dimensions of these representations are given by
$$\dim r_k=\frac{q^{k+1}-q^{-k-1}}{q-q^{-1}}$$
where $q,q^{-1}$ are the solutions of $X^2-NX+1=0$. Also, we have $Sp_2^+=SU_2$.
\end{theorem}

\begin{proof}
This is a straightforward unification of the results that we already have for $O_N^+$ and $SU_2$, the technical details being all standard. See \cite{ba1}.
\end{proof}

We will be back to $O_N^+$ and $O_F^+$ later on, first in chapter 7 below, with a number of more advanced algebraic considerations, in relation with super-structures and twists, and then in chapter 8 below, with a number of advanced probabilistic computations. 

\bigskip

Finally, as the saying in geometry and physics goes, there is no $SU_2$ without $SO_3$. We will construct in chapter 9 below a kind of ``$SO_3$ companion'' for $O_N^+$. This companion will be something quite unexpected, namely the quantum permutation group $S_N^+$.

\section*{5e. Exercises}

There has been a lot of combinatorics and calculus in this above, and doing some more combinatorics and calculus will be the goal of the exercises here. First, we have:

\begin{exercise}
Verify the Gram determinant formula for $P(3)$ explicitly, without any trick, just by computing the $5\times5$ determinant.
\end{exercise}

This might sound not very serious, because we have explained in the above a trick for dealing with such things. But finding such tricks always requires a lot of efforts and sweat, with computing $5\times5$ determinants being a daily occupation for researchers.

\begin{exercise}
Establish the following formula,
$$\int_0^{\pi/2}\cos^pt\,dt=\int_0^{\pi/2}\sin^pt\,dt=\left(\frac{\pi}{2}\right)^{\varepsilon(p)}\frac{p!!}{(p+1)!!}$$
where $\varepsilon(p)=1$ if $p$ is even, and $\varepsilon(p)=0$ if $p$ is odd.
\end{exercise}

Partial integration, enjoy.

\begin{exercise}
Establish the following formula, that we used in the above,
$$\int_0^{\pi/2}\cos^pt\sin^qt\,dt=\left(\frac{\pi}{2}\right)^{\varepsilon(p)\varepsilon(q)}\frac{p!!q!!}{(p+q+1)!!}$$
where $\varepsilon(p)=1$ if $p$ is even, and $\varepsilon(p)=0$ if $p$ is odd, as before.
\end{exercise}

More partial integration, enjoy.

\begin{exercise}
Establish the following integration formula over the sphere,
$$\int_{S^{N-1}_\mathbb R}x_1^{k_1}\ldots x_N^{k_N}\,dx=\frac{(N-1)!!k_1!!\ldots k_N!!}{(N+\Sigma k_i-1)!!}$$
that we used in the above, by using spherical coordinates and Fubini.
\end{exercise}

As before, no special comments, just enjoy. This is first-class mathematics.

\begin{exercise}
Learn and use the Stieltjes inversion formula, namely
$$d\mu (x)=\lim_{t\searrow 0}-\frac{1}{\pi}\,Im\left(G(x+it)\right)\cdot dx$$
in order to find the centered laws having as $2k$-th moments the numbers $(2k)!!$ and $C_k$.
\end{exercise}

No comments here either. As before, this is first-class mathematics.

\begin{exercise}
Write down a complete proof, using a method of your choice, found here or somewhere else, for the classification of the irreducible representations of $SU_2$.
\end{exercise}

This is the most important exercise of them all, because the relation between $SU_2$ and $O_N^+$ will be something that will appear regularly, in what follows.

\chapter{Unitary groups}

\section*{6a. Gaussian laws}

We have seen that the Brauer type results for $O_N,O_N^+$ lead to some concrete and interesting consequences. In this chapter we discuss similar results for $U_N,U_N^+$. The situation here is a bit more complicated than for $O_N,O_N^+$, and we will only do a part of the work, namely algebra and basic probability, following \cite{ba1}, with the other part, advanced probability, following \cite{bc1}, being left for later, in chapter 8 below.

\bigskip

Let us also mention that, probabilistically speaking, the basic probability theory that we used for $O_N,O_N^+$, while still applying to $U_N$, after some changes, will not apply to $U_N^+$, due to the fact that the main character here is not normal, $\chi\chi^*\neq\chi^*\chi$. So, in order to deal with $U_N^+$ we will have to use something more advanced, namely Voiculescu's free probability theory \cite{vdn}. Which is something very beautiful. But more on this later.

\bigskip

Let us start with a summary of what we know about $U_N,U_N^+$, coming from the Brauer type results from chapter 4, and the partition computations from chapter 5:

\begin{theorem}
For the basic unitary quantum groups, namely
$$U_N\subset U_N^+$$
the intertwiners between the Peter-Weyl representations are given by
$$Hom(u^{\otimes k},u^{\otimes l})=span\left(T_\pi\Big|\pi\in D(k,l)\right)$$
with the linear maps $T_\pi$ associated to the pairings $\pi$ being given by
$$T_\pi(e_{i_1}\otimes\ldots\otimes e_{i_k})=\sum_{j_1\ldots j_l}\delta_\pi\begin{pmatrix}i_1&\ldots&i_k\\ j_1&\ldots&j_l\end{pmatrix}e_{j_1}\otimes\ldots\otimes e_{j_l}$$
and with the pairings $D$ being as follows, with calligraphic standing for matching:
$$\mathcal P_2\supset\mathcal{NC}_2$$
At the level of the moments of the main character, we have in both cases
$$\int\chi^k\leq|D(k)|$$
with $D$ being the above sets of pairings, with equality happening at $N\geq k$.
\end{theorem}

\begin{proof}
This is a summary of the results that we have, established in the previous chapters, and coming from Tannakian duality, via some combinatorics. To be more precise, the Brauer type results are from chapter 4, the estimates for the moments follows from this and from Peter-Weyl, as explained in chapter 5, and finally the last assertion, regarding the equality at $N\geq k$, is something more subtle, explained in chapter 5. 
\end{proof}

Let us first investigate the unitary group $U_N$. As it was the case for the orthogonal group $O_N$, in chapter 5, the representation theory here is something quite complicated, related to Young tableaux, and we will not get into this subject. However, once again in analogy with $O_N$, there is one straightforward thing to be done, namely the computation of the law of the main character, in the $N\to\infty$ limit. 

\bigskip

In order to do this, we will need a basic probability result, as follows:

\index{complex Gaussian law}
\index{complex normal law}
\index{matching pairing}

\begin{theorem}
The moments of the complex Gaussian law, given by
$$G_1\sim\frac{1}{\sqrt{2}}(a+ib)$$
with $a,b$ being independent, each following the real Gaussian law $g_1$, are given by
$$M_k=|\mathcal P_2(k)|$$
for any colored integer $k=\circ\bullet\bullet\circ\ldots$ 
\end{theorem}

\begin{proof}
This is something well-known, which can be done in several steps, as follows:

\medskip

(1) We recall from chapter 5 that the moments of the real Gaussian law $g_1$, with respect to integer exponents $k\in\mathbb N$, are the following numbers:
$$m_k=|P_2(k)|$$

Numerically, we have the following formula, explained as well in chapter 5:
$$m_k=\begin{cases}
k!!&(k\ {\rm even})\\
0&(k\ {\rm odd})
\end{cases}$$

(2) We will show here that in what concerns the complex Gaussian law $G_1$, similar results hold. Numerically, we will prove that we have the following formula, where a colored integer $k=\circ\bullet\bullet\circ\ldots$ is called uniform when it contains the same number of $\circ$ and $\bullet$\,, and where $|k|\in\mathbb N$ is the length of such a colored integer:
$$M_k=\begin{cases}
(|k|/2)!&(k\ {\rm uniform})\\
0&(k\ {\rm not\ uniform})
\end{cases}$$

Now since the matching partitions $\pi\in\mathcal P_2(k)$ are counted by exactly the same numbers, and this for trivial reasons, we will obtain the formula in the statement, namely:
$$M_k=|\mathcal P_2(k)|$$

(3) This was for the plan. In practice now, we must compute the moments, with respect to colored integer exponents $k=\circ\bullet\bullet\circ\ldots$\,, of the variable in the statement:
$$c=\frac{1}{\sqrt{2}}(a+ib)$$

As a first observation, in the case where such an exponent $k=\circ\bullet\bullet\circ\ldots$ is not uniform in $\circ,\bullet$\,, a rotation argument shows that the corresponding moment of $c$ vanishes. To be more precise, the variable $c'=wc$ can be shown to be complex Gaussian too, for any $w\in\mathbb C$, and from $M_k(c)=M_k(c')$ we obtain $M_k(c)=0$, in this case.

\medskip

(4) In the uniform case now, where $k=\circ\bullet\bullet\circ\ldots$ consists of $p$ copies of $\circ$ and $p$ copies of $\bullet$\,, the corresponding moment can be computed as follows:
\begin{eqnarray*}
M_k
&=&\int(c\bar{c})^p\\
&=&\frac{1}{2^p}\int(a^2+b^2)^p\\
&=&\frac{1}{2^p}\sum_s\binom{p}{s}\int a^{2s}\int b^{2p-2s}\\
&=&\frac{1}{2^p}\sum_s\binom{p}{s}(2s)!!(2p-2s)!!\\
&=&\frac{1}{2^p}\sum_s\frac{p!}{s!(p-s)!}\cdot\frac{(2s)!}{2^ss!}\cdot\frac{(2p-2s)!}{2^{p-s}(p-s)!}\\
&=&\frac{p!}{4^p}\sum_s\binom{2s}{s}\binom{2p-2s}{p-s}
\end{eqnarray*}

(5) In order to finish now the computation, let us recall that we have the following formula, coming from the generalized binomial formula, or from the Taylor formula:
$$\frac{1}{\sqrt{1+t}}=\sum_{k=0}^\infty\binom{2k}{k}\left(\frac{-t}{4}\right)^k$$

By taking the square of this series, we obtain the following formula:
\begin{eqnarray*}
\frac{1}{1+t}
&=&\sum_{ks}\binom{2k}{k}\binom{2s}{s}\left(\frac{-t}{4}\right)^{k+s}\\
&=&\sum_p\left(\frac{-t}{4}\right)^p\sum_s\binom{2s}{s}\binom{2p-2s}{p-s}
\end{eqnarray*}

Now by looking at the coefficient of $t^p$ on both sides, we conclude that the sum on the right equals $4^p$. Thus, we can finish the moment computation in (4), as follows:
$$M_p=\frac{p!}{4^p}\times 4^p=p!$$

(6) As a conclusion, if we denote by $|k|$ the length of a colored integer $k=\circ\bullet\bullet\circ\ldots$\,, the moments of the variable $c$ in the statement are given by:
$$M_k=\begin{cases}
(|k|/2)!&(k\ {\rm uniform})\\
0&(k\ {\rm not\ uniform})
\end{cases}$$

On the other hand, the numbers $|\mathcal P_2(k)|$ are given by exactly the same formula. Indeed, in order to have matching pairings of $k$, our exponent $k=\circ\bullet\bullet\circ\ldots$ must be uniform, consisting of $p$ copies of $\circ$ and $p$ copies of $\bullet$, with $p=|k|/2$. But then the matching pairings of $k$ correspond to the permutations of the $\bullet$ symbols, as to be matched with $\circ$ symbols, and so we have $p!$ such matching pairings. Thus, we have the same formula as for the moments of $c$, and we are led to the conclusion in the statement.
\end{proof}

We should mention that the above proof is just one proof among others. There is a lot of interesting mathematics behind the complex Gaussian variables, whose knowledge can avoid some of the above computations, and we recommend some reading here.

\bigskip 

By getting back now to the unitary group $U_N$, with the above results in hand we can formulate our first concrete result about it, as follows:

\begin{theorem}
For the unitary group $U_N$, the main character
$$\chi=\sum_iu_{ii}$$
follows the standard complex Gaussian law
$$\chi\sim G_1$$
in the $N\to\infty$ limit.
\end{theorem}

\begin{proof}
This follows by putting together the results that we have, namely Theorem 6.1 applied with $N>k$, and then Theorem 6.2.
\end{proof}

As already mentioned, as it was the case for the orthogonal group $O_N$, in chapter 5, the representation theory for $U_N$ at fixed $N\in\mathbb N$ is something quite complicated, related to the combinatorics of Young tableaux, and we will not get into this subject here. 

\bigskip

There is, however, one more interesting topic regarding $U_N$ to be discussed, namely its precise relation with $O_N$, and more specifically the passage $O_N\to U_N$.

\bigskip 

Contrary to the passage $\mathbb R^N\to\mathbb C^N$, or to the passage $S^{N-1}_\mathbb R\to S^{N-1}_\mathbb C$, which are both elementary, the passage $O_N\to U_N$ cannot be understood directly. In order to understand this passage we must pass through the corresponding Lie algebras, a follows:

\index{complexification}

\begin{theorem}
The passage $O_N\to U_N$ appears via a Lie algebra complexification,
$$O_N\to\mathfrak o_N\to\mathfrak u_n\to U_N$$
with the Lie algebra $\mathfrak u_N$ being a complexification of the Lie algebra $\mathfrak o_N$.
\end{theorem}

\begin{proof}
This is something rather philosophical, and advanced as well, that we will not really need here, the idea being as follows:

\medskip

(1) The orthogonal and unitary groups $O_N,N_N$ are both Lie groups, in the sense that they are smooth manifolds, and the corresponding Lie algebras $\mathfrak o_N,\mathfrak u_N$, which are by definition the respective tangent spaces at 1, can be computed by differentiating the equations defining $O_N,U_N$, with the conclusion being as follows:
$$\mathfrak o_N=\left\{ A\in M_N(\mathbb R)\Big|A^t=-A\right\}$$
$$\mathfrak u_N=\left\{ B\in M_N(\mathbb C)\Big|B^*=-B\right\}$$

(2) This was for the correspondences $O_N\to\mathfrak o_N$ and $U_N\to\mathfrak u_N$. In the other sense, the correspondences $\mathfrak o_N\to O_N$ and $\mathfrak u_N\to U_N$ appear by exponentiation, the result here stating that, around 1, the orthogonal matrices can be written as $U=e^A$, with $A\in\mathfrak o_N$, and the unitary matrices can be written as $U=e^B$, with $B\in\mathfrak u_N$. 

\bigskip

(3) In view of all this, in order to understand the passage $O_N\to U_N$ it is enough to understand the passage $\mathfrak o_N\to\mathfrak u_N$. But, in view of the above explicit formulae for $\mathfrak o_N,\mathfrak u_N$, this is basically an elementary linear algebra problem. Indeed, let us pick an arbitrary matrix $B\in M_N(\mathbb C)$, and write it as follows, with $A,C\in M_N(\mathbb R)$:
$$B=A+iC$$

In terms of $A,C$, the equation $B^*=-B$ defining the Lie algebra $\mathfrak u_N$ reads:
$$A^t=-A$$
$$C^t=C$$

(4) As a first observation, we must have $A\in\mathfrak o_N$. Regarding now $C$, let us decompose it as follows, with $D$ being its diagonal, and $C'$ being the remainder:
$$C=D+C'$$

The remainder $C'$ being symmetric with 0 on the diagonal, by switching all the signs below the main diagonal we obtain a certain matrix $C'_-\in\mathfrak o_N$. Thus, we have decomposed $B\in\mathfrak u_N$ as follows, with $A,C'\in\mathfrak o_N$, and with $D\in M_N(\mathbb R)$ being diagonal:
$$B=A+iD+iC'_-$$

(5) As a conclusion now, we have shown that we have a direct sum decomposition of real linear spaces as follows, with $\Delta\subset M_N(\mathbb R)$ being the diagonal matrices:
$$\mathfrak u_N\simeq\mathfrak o_N\oplus\Delta\oplus\mathfrak o_N$$

Thus, we can stop our study here, and say that we have reached the conclusion in the statement, namely that $\mathfrak u_N$ appears as a ``complexification'' of $\mathfrak o_N$.
\end{proof}

As before with many other things, that we will not really need in what follows, this was just an introduction to the subject. More can be found in any Lie group book.

\section*{6b. Circular variables}

Let us discuss now the unitary quantum group $U_N^+$. We have 3 main topics to be discussed, namely the character law with $N\to\infty$, the representation theory at fixed $N\in\mathbb N$, and complexification, and the situation with respect to $U_N$ is as follows:

\medskip

\begin{enumerate}

\item The asymptotic character law appears as a ``free complexification'' of the Wigner law, the combinatorics being similar to the one in the classical case.

\medskip

\item The representation theory is definitely simpler, with the fusion rules being given by a ``free complexification'' of the Clebsch-Gordan rules, at any $N\geq2$.

\medskip

\item As for the complexification aspects, here the situation is extremely simple, with the passage $O_N^+\to U_N^+$ being a free complexification.
\end{enumerate}

\medskip

Obviously, some magic is going on here. Who would have imagined that, passed a few abstract things that can be learned, and we will learn indeed all this, the free quantum group $U_N^+$ is simpler than its classical counterpart $U_N$. This might suggest for instance that quantum mechanics might be simpler than classical mechanics. And isn't this crazy. But hey, read Arnold \cite{arn} first, and let me know if you find all that stuff simple.

\bigskip

Back to work now, let us first discuss the character problematics for $U_N^+$, or rather the difficulties that appear here. We have the following theoretical result, to start with, coming from the general $C^*$-algebra theory developed in chapter 1 above:

\index{law}
\index{distribution}

\begin{theorem}
Given a $C^*$-algebra with a faithful trace $(A,tr)$, any normal variable, 
$$aa^*=a^*a$$
has a ``law'', which is by definition a complex probability measure $\mu\in\mathcal P(\mathbb C)$ satisfying:
$$tr(a^k)=\int_\mathbb Cz^kd\mu(z)$$
This law is unique, and is supported by the spectrum $\sigma(a)\subset\mathbb C$. In the non-normal case, $aa^*\neq a^*a$, such a law does not exist.
\end{theorem}

\begin{proof}
We have two assertions here, the idea being as follows:

\medskip

(1) In the normal case, $aa^*=a^*a$, the Gelfand theorem, or rather the subsequent continuous functional calculus theorem, tells us that we have: 
$$<a>=C(\sigma(a))$$

Thus the functional $f(a)\to tr(f(a))$ can be regarded as an integration functional on the algebra $C(\sigma(a))$, and by the Riesz theorem, this latter functional must come from a probability measure $\mu$ on the spectrum $\sigma(a)$, in the sense that we must have:
$$tr(f(a))=\int_{\sigma(a)}f(z)d\mu(z)$$

We are therefore led to the conclusions in the statement, with the uniqueness assertion coming from the fact that the elements $a^k$, taken as usual with respect to colored integer exponents, $k=\circ\bullet\bullet\circ\ldots$\,, generate the whole $C^*$-algebra $C(\sigma(a))$.

\medskip

(2) In the non-normal case now, $aa^*\neq a^*a$, we must show that such a law does not exist. For this purpose, we can use a positivity trick, as follows:
\begin{eqnarray*}
aa^*-a^*a\neq0
&\implies&(aa^*-a^*a)^2>0\\
&\implies&aa^*aa^*-aa^*a^*a-a^*aaa^*+a^*aa^*a>0\\
&\implies&tr(aa^*aa^*-aa^*a^*a-a^*aaa^*+a^*aa^*a)>0\\
&\implies&tr(aa^*aa^*+a^*aa^*a)>tr(aa^*a^*a+a^*aaa^*)\\
&\implies&tr(aa^*aa^*)>tr(aaa^*a^*)
\end{eqnarray*}

Now assuming that $a$ has a law $\mu\in\mathcal P(\mathbb C)$, in the sense that the moment formula in the statement holds, the above two different numbers would have to both appear by integrating $|z|^2$ with respect to this law $\mu$, which is contradictory, as desired.
\end{proof}

All the above might look a bit abstract, so as an illustration here, consider the following matrix, which is the simplest example of a non-normal matrix:
$$Z=\begin{pmatrix}0&1\\0&0\end{pmatrix}$$

We have then the following formulae, which show that $Z$ has no law, indeed:
$$tr(ZZZ^*Z^*)=tr\begin{pmatrix}0&0\\0&0\end{pmatrix}=0$$
$$tr(ZZ^*ZZ^*)=tr\begin{pmatrix}1&0\\0&0\end{pmatrix}=\frac{1}{2}$$

Getting back now to $U_N^+$, its main character is not normal, so it does not have a law $\mu\in\mathcal P(\mathbb C)$. Here is a concrete illustration for this phenomenon:

\begin{proposition}
The main character of $U_N^+$ satisfies, at $N\geq4$,
$$\int_{U_N^+}\chi\chi\chi^*\chi^*=1$$
$$\int_{U_N^+}\chi\chi^*\chi\chi^*=2$$
and so this main character $\chi$ does not have a law $\mu\in\mathcal P(\mathbb C)$.
\end{proposition}

\begin{proof}
This follows from the last assertion in Theorem 6.1, which tells us that the moments of $\chi$ are given by the following formula, valid at any $N\geq k$:
$$\int_{U_N^+}\chi^k=|\mathcal{NC}_2(k)|$$

Indeed, we obtain from this the following formula, valid at any $N\geq4$:
\begin{eqnarray*}
\int_{U_N^+}\chi\chi\chi^*\chi^*
&=&|\mathcal NC_2(\circ\circ\bullet\,\bullet)|\\
&=&|\Cap|\\
&=&1
\end{eqnarray*}

On the other hand, we obtain as well the following formula, once again at $N\geq4$:
\begin{eqnarray*}
\int_{U_N^+}\chi\chi^*\chi\chi^*
&=&|\mathcal NC_2(\circ\bullet\circ\,\bullet)|\\
&=&|\cap\cap\,,\Cap\,|\\
&=&2
\end{eqnarray*}

Thus, we have the formulae in the statement. Now since we cannot obtain both 1 and 2 by integrating $|z|^2$ with respect to a measure, our variable has no law $\mu\in\mathcal P(\mathbb C)$.
\end{proof}

Summarizing, we are a bit in trouble here, and we must proceed as follows:

\index{law}
\index{distribution}

\begin{definition}
Given a $C^*$-algebra with a faithful trace $(A,tr)$, the law of a variable $a\in A$ is the following abstract functional:
$$\mu:\mathbb C<X,X^*>\to\mathbb C$$
$$P\to tr(P(a))$$
In particular two variables $a,b\in A$ have the same law, and we write in this case $a\sim b$, when all their moments coincide,
$$tr(a^k)=tr(b^k)$$
with these moments being taken with respect to colored integers, $k=\circ\bullet\bullet\circ\ldots$
\end{definition}

Here the compatibility between the first and the second above conventions comes from the fact that, by linearity, the functional $\mu$ is uniquely determined by its values on the monomials $P(z)=z^k$, with $k=\circ\bullet\bullet\circ\ldots$ being a colored integer.

\bigskip

In the normal case, $aa^*=a^*a$, it follows from Theorem 6.5 that the law, as defined above, comes from a probability measure $\mu\in\mathcal P(\mathbb C)$, via the following formula:
$$tr(P(a))=\int_\mathbb CP(z)d\mu(z)$$

In particular, in the case where we have two normal variables $a,b$, the equality $a\sim b$ tells us that the laws of $a,b$, taken in the complex measure sense, must coincide. In the general case, $aa^*\neq a^*a$, there is no such simple interpretation of the law, with this coming from the last assertion in Theorem 6.5, and also from the concrete example worked out in Proposition 6.6, and we must therefore use Definition 6.7 as it is.

\bigskip

Next in line, we have to talk about freeness. For this purpose, let us recall that the independence of two subalgebras $B,C\subset A$ can be defined in the following way:
$$tr(b)=tr(c)=0\implies tr(bc)=0$$

In analogy with this, we have the following definition, due to Voiculescu \cite{vdn}:

\index{freeness}

\begin{definition}
Two subalgebras $B,C\subset A$ are called free when the following condition is satisfied, for any $b_i\in B$ and $c_i\in C$:
$$tr(b_i)=tr(c_i)=0\implies tr(b_1c_1b_2c_2\ldots)=0$$
Also, two variables $b,c\in A$ are called free when the algebras that they generate,
$$B=<b>\quad,\quad C=<c>$$
are free inside $A$, in the above sense.
\end{definition}

In short, freeness appears as a kind of ``free analogue'' of independence, taking into account the fact that the variables do not necessarily commute. As a first result regarding this notion, in analogy with the basic theory of the independence, we have:

\begin{proposition}
Assuming that $B,C\subset A$ are free, the restriction of $tr$ to $<B,C>$ can be computed in terms of the restrictions of $tr$ to $B,C$. To be more precise,
$$tr(b_1c_1b_2c_2\ldots)=P\Big(\{tr(b_{i_1}b_{i_2}\ldots)\}_i,\{tr(c_{j_1}c_{j_2}\ldots)\}_j\Big)$$
where $P$ is certain polynomial in several variables, depending on the length of the word $b_1c_1b_2c_2\ldots$, and having as variables the traces of products of type $b_{i_1}b_{i_2}\ldots$ and $c_{j_1}c_{j_2}\ldots$\,, with the indices being chosen increasing, $i_1<i_2<\ldots$ and $j_1<j_2<\ldots$
\end{proposition}

\begin{proof}
This is something quite theoretical, so let us begin with an example. Our claim is that if $b,c$ are free then, exactly as in the case where we have independence:
$$tr(bc)=tr(b)tr(c)$$

Indeed, we have the following computation, with the convention $a'=a-tr(a)$:
\begin{eqnarray*}
tr(bc)
&=&tr[(b'+tr(b))(c'+tr(c))]\\
&=&tr(b'c')+t(b')tr(c)+tr(b)tr(c')+tr(b)tr(c)\\
&=&tr(b'c')+tr(b)tr(c)\\
&=&tr(b)tr(c)
\end{eqnarray*}

In general now, the situation is a bit more complicated, but the same trick applies. To be more precise, we can start our computation as follows:
\begin{eqnarray*}
tr(b_1c_1b_2c_2\ldots)
&=&tr\big[(b_1'+tr(b_1))(c_1'+tr(c_1))(b_2'+tr(b_2))(c_2'+tr(c_2))\ldots\ldots\big]\\
&=&tr(b_1'c_1'b_2'c_2'\ldots)+{\rm other\ terms}\\
&=&{\rm other\ terms}
\end{eqnarray*}

Observe that we have used here the freeness condition, in the following form:
$$tr(b_i')=tr(c_i')=0\implies tr(b_1'c_1'b_2'c_2'\ldots)=0$$

Now regarding the ``other terms'', those which are left, each of them will consist of a product of traces of type $tr(b_i)$ and $tr(c_i)$, and then a trace of a product still remaining to be computed, which is of the following form, with $\beta_i\in B$ and $\gamma_i\in C$:
$$tr(\beta_1\gamma_1\beta_2\gamma_2\ldots)$$

To be more precise, the variables $\beta_i\in B$ appear as ordered products of those $b_i\in B$ not getting into individual traces $tr(b_i)$, and the variables $\gamma_i\in C$ appear as ordered products of those $c_i\in C$ not getting into individual traces $tr(c_i)$. Now since the length of each such alternating product $\beta_1\gamma_1\beta_2\gamma_2\ldots$ is smaller than the length of the original alternating product $b_1c_1b_2c_2\ldots$, we are led into of recurrence, and this gives the result.
\end{proof}

As an illustration now, given two discrete groups $\Gamma,\Lambda$, the algebras $C^*(\Gamma),C^*(\Lambda)$ are independent inside $C^*(\Gamma\times\Lambda)$, are free inside $C^*(\Gamma*\Lambda)$. More on this later.

\bigskip

With the above definitions in hand, we can now advance, in connection with our questions, in the following rather formal way:

\index{circular law}
\index{Voiculescu circular law}

\begin{definition}
The Voiculescu circular law $\Gamma_1$ is defined by
$$\Gamma_1\sim\frac{1}{\sqrt{2}}(\alpha+i\beta)$$
with $\alpha,\beta$ being self-adjoint and free, each following the Wigner semicircle law $\gamma_1$.
\end{definition}

Our goal in what follows will be that of proving that the main character law of $U_N^+$ becomes circular with $N\to\infty$, and in fact, more generally, with $N\geq2$. In order to prove these results, we need first to study the Voiculescu circular law, a bit in the same way as we did with the Wigner semicircle law, in chapter 5. Let us start with:

\index{shift}
\index{semicircle law}
\index{Wigner semicircle}

\begin{proposition}
Consider the shift operator $S\in B(l^2(\mathbb N))$. We have then
$$S+S^*\sim\gamma_1$$
with respect to the state $\varphi(T)=<T\delta_0,\delta_0>$.
\end{proposition}

\begin{proof}
We must compute the moments of the variable $S+S^*$ with respect to the state $\varphi(T)=<T\delta_0,\delta_0>$. Our claim is that these moments are given by:
$$<(S+S^*)^k\delta_0,\delta_0>=|NC_2(k)|$$

Indeed, when expanding $(S+S^*)^k$ and computing the value of $\varphi:T\to<T\delta_0,\delta_0>$, the only contributions will come via the formula $S^*S=1$, which must succesively apply, as to collapse the whole product of $S,S^*$ variables into a 1 quantity. But these applications of $S^*S=1$ must appear in a non-crossing manner, and so the contributions, which are each worth 1, are parametrized by the partitions $\pi\in NC_2(k)$. Thus, we obtain the above moment formula, which shows that we have $S+S^*\sim\gamma_1$, as claimed.
\end{proof}

The next step is that of taking a free product of the model found in Proposition 6.11 with itself. For this purpose, we can use the following construction:

\index{Fock space}
\index{free Fock space}

\begin{definition}
Given a real Hilbert space $H$, we define the associated free Fock space as being the infinite Hilbert space sum
$$F(H)=\mathbb C\Omega\oplus H\oplus H^{\otimes2}\oplus\ldots$$
and then we define the algebra $A(H)$ generated by the creation operators
$$S_x:v\to x\otimes v$$
on this free Fock space.
\end{definition}

At the level of examples, with $H=\mathbb R$ we recover the shift algebra $A=<S>$ on the Hilbert space $H=l^2(\mathbb N)$. Also, with $H=\mathbb R^2$, we obtain the algebra $A=<S_1,S_2>$ generated by the two shifts on the Hilbert space $H=l^2(\mathbb N*\mathbb N)$.

\bigskip

With the above notions in hand, we have the following freeness result, from \cite{vdn}:

\begin{proposition}
Given a real Hilbert space $H$, and two orthogonal vectors $x,y\in H$,
$$x\perp y$$
the corresponding creation operators $S_x$ and $S_y$ are free with respect to
$$tr(T)=<T\Omega,\Omega>$$
called trace associated to the vacuum vector.
\end{proposition}

\index{vacuum vector}

\begin{proof}
In standard tensor notation for the elements of the free Fock space $F(H)$, the formula of a creation operator associated to a vector $x\in H$ is as follows:
$$S_x(y_1\otimes\ldots\otimes y_n)=x\otimes y_1\otimes\ldots\otimes y_n$$

As for the formula of the adjoint of this creation operator, this is as follows: 
$$S_x^*(y_1\otimes\ldots\otimes y_n)=<x,y_1>\otimes y_2\otimes\ldots\otimes y_n$$

We obtain from this the following formula, valid for any two vectors $x,y\in H$:
$$S_x^*S_y=<x,y>id$$

With these formulae in hand, the result follows by doing some elementary computations, a bit similar to those in the proof of Proposition 6.11.
\end{proof}

In order now to model the circular variables, still following Voiculescu \cite{vdn}, we can use the following key observation, coming from Proposition 6.13 via a rotation trick:

\begin{proposition}
Given two polynomials $f,g\in\mathbb C[X]$, consider the variables 
$$R^*+f(R)\quad,\quad 
S^*+g(S)$$
where $R,S$ are two creation operators, or shifts, associated to a pair of  orthogonal norm $1$ vectors. These variables are then free, and their sum has the same law as
$$T^*+(f+g)(T)$$
with $T$ being the usual shift on $l^2(\mathbb N)$.
\end{proposition}

\begin{proof}
We have two assertions here, the idea being as follows:

\medskip

(1) The freeness assertion comes from the general freeness result from Proposition 6.13, via the various identifications coming from the previous results.

\medskip

(2) Regarding now the second assertion, the idea is that this comes from a $45^\circ$ rotation trick. Let us write indeed the two variables in the statement as follows:
$$X=R^*+a_0+a_1R+a_2R^2+\ldots$$
$$Y=S^*+b_0+b_1S+a_2S^2+\ldots$$

Now let us perform the following $45^\circ$ base change, on the real span of the vectors $r,s\in H$ producing our two shifts $R,S$:
$$t=\frac{r+s}{\sqrt{2}}\quad,\quad u=\frac{r-s}{\sqrt{2}}$$

The new shifts, associated to these vectors $t,u\in H$, are then given by:
$$T=\frac{R+S}{\sqrt{2}}\quad,\quad U=\frac{R-S}{\sqrt{2}}$$

By using now these new shifts, which are free as well according to Proposition 6.13, we obtain the following equality of distributions:
\begin{eqnarray*}
X+Y
&=&R^*+S^*+\sum_ka_kR^k+b_kS^k\\
&=&\sqrt{2}T^*+\sum_ka_k\left(\frac{T+U}{\sqrt{2}}\right)^k+b_k\left(\frac{T-U}{\sqrt{2}}\right)^k\\
&\sim&\sqrt{2}T^*+\sum_ka_k\left(\frac{T}{\sqrt{2}}\right)^k+b_k\left(\frac{T}{\sqrt{2}}\right)^k\\
&\sim&T^*+\sum_ka_kT^k+b_kT^k
\end{eqnarray*}

To be more precise, here in the last two lines we have used the freeness property of $T,U$ in order to cut $U$ from the computation, as it cannot bring anything, and then we did a rescaling at the end. Thus, we are led to the conclusion in the statement.
\end{proof}

Still following Voiculescu \cite{vdn}, we can now formulate an explicit and very useful modelling result for the semicircular and circular variables, as follows:

\index{circular variable}
\index{semicircular variable}

\begin{theorem}
Let $H$ be the Hilbert space having as basis the colored integers $k=\circ\bullet\bullet\circ\ldots$\,, and consider the shift operators $S:k\to\circ k$ and $T:k\to\bullet k$. We have then
$$S+S^*\sim\gamma_1$$
$$S+T^*\sim\Gamma_1$$
with respect to the state $\varphi(T)=<Te,e>$, where $e$ is the empty word.
\end{theorem}

\begin{proof}
This is standard free probability, the idea being as follows:

\medskip

(1) The formula $S+S^*\sim\gamma_1$ is something that we already know, in a slightly different formulation, from Proposition 6.11 above.

\medskip

(2) The formula $S+T^*\sim\Gamma_1$ follows from this, by using the freeness result in Proposition 6.13, and the rotation trick in Proposition 6.14.
\end{proof}

At the combinatorial level now, we have the following result, which is in analogy with the moment theory of the Wigner semicircle law, developed in chapter 5:

\begin{theorem}
A variable $a\in A$ is circular when its moments are given by
$$tr(a^k)=|\mathcal{NC}_2(k)|$$
for any colored integer $k=\circ\bullet\bullet\circ\ldots$
\end{theorem}

\begin{proof}
By using Theorem 6.15, it is enough to do the computation in the model there. With $S:k\to\circ k$ and $T:k\to\bullet k$, our claim is that we have:
$$<(S+T^*)^ke,e>=|\mathcal{NC}_2(k)|$$

In order to prove this formula, we can proceed as in the proof of Proposition 6.11. Indeed, let us expand the quantity $(S+T^*)^k$, and then apply the state $\varphi$. With respect to the previous computation, from Proposition 6.11, what happens is that the contributions will come this time via the formulae $S^*S=1$, $T^*T=1$, which must succesively apply, as to collapse the whole product of $S,S^*,T,T^*$ variables into a 1 quantity. 

\medskip

As before, in the proof of Proposition 6.11, these applications of the rules $S^*S=1$, $T^*T=1$ must appear in a noncrossing manner, but what happens now, in contrast with the computation from the proof of Proposition 6.11, where $S+S^*$ was self-adjoint, is that at each point where the exponent $k$ has a $\circ$ entry we must use $T^*T=1$, and at each point where the exponent $k$ has a $\bullet$ entry we must use $S^*S=1$. Thus the contributions, which are each worth 1, are parametrized by the partitions $\pi\in\mathcal{NC}_2(k)$, and we are done.
\end{proof}

We will be back to this in chapter 8 below. For our purposes here, the above theory is all we need. Getting back now to $U_N^+$, following \cite{ba1}, we can reformulate the main result that we have so far about it, by using the above notions, as follows:

\begin{theorem}
For the quantum group $U_N^+$ with $N\geq2$ we have
$$Hom(u^{\otimes k},u^{\otimes l})=span\left(T_\pi\Big|\pi\in D(k,l)\right)$$
and at the level of the moments of the main character we have
$$\int_{U_N^+}\chi^k\leq|\mathcal{NC}_2(k)|$$
with equality at $N\geq k$, the numbers on the right being the moments of $\Gamma_1$.
\end{theorem}

\begin{proof}
This is something that we already know. To be more precise, the Brauer type result is from chapter 4, the estimate for the moments follows from this and from Peter-Weyl, as explained in chapter 5, the equality at $N\geq k$ is something more subtle, explained in chapter 5, and the last statement comes from the above discussion. 
\end{proof}

Summarizing, with a bit of abstract probability theory, of free type, we are now on our way into the study of $U_N^+$, paralleling the previous study of $O_N^+$.

\section*{6c. Fusion rules}

With the above result in hand, we can now go ahead and do with $U_N^+$ exactly what we did with $O_N^+$ in chapter 5, with modifications where needed, namely constructing the irreducible representations by recurrence, using a Frobenius duality trick, computing the fusion rules, and concluding as well that we have $\chi\sim\Gamma_1$, at any $N\geq2$.

\bigskip

In practice, all this will be more complicated than for $O_N^+$, mainly because the fusion rules will be something new, in need of some preliminary combinatorial study. These fusion rules will be a kind of ``free Clebsch-Gordan rules'', as follows:
$$r_k\otimes r_l=\sum_{k=xy,l=\bar{y}z}r_{xz}$$

Following \cite{ba1}, let $W$ be the set of colored integers $k=\circ\bullet\bullet\circ\ldots$\,, and  consider the complex algebra $E$ spanned by $W$. We have then an isomorphism, as follows:
$$(\mathbb C<X,X^*>,+,\cdot )\simeq(E,+,\cdot)$$
$$X\to\circ\quad,\quad X^*\to\bullet$$

We define an involution on our algebra $E$, by antilinearity and antimultiplicativity, according to the following formulae, with $e$ being as usual the empty word:
$$\bar{e}=e\quad,\quad\bar{\circ}=\bullet\quad,\quad \bar{\bullet}=\circ$$

With these conventions, we have the following result:

\begin{proposition} 
The map $\times:W\times W\to E$ given by 
$$x\times y=\sum_{x=ag,y=\bar{g}b}ab$$
extends by linearity into an associative multiplication of $E$.
\end{proposition}

\begin{proof}
Observe first that $\times$ is well-defined, the sum being finite. Let us prove now that $\times$ is associative. Let $x,y,z\in W$. Then:
\begin{eqnarray*}
(x\times y)\times z
&=&\sum_{x=a\bar{g},y=gb} ab\times z\\
&=&\sum_{x=a\bar{g},y=gb,ab=ch,z=\bar{h}d}cd
\end{eqnarray*}

Now observe that for $a,b,c,h\in W$ the equality $ab=ch$ is equivalent to $b=uh,c=au$ with $u\in W$, or to $a=cv,h=vb$ with $v\in W$. Thus, we have:
\begin{eqnarray*}
(x\times y)\times z
&=&\sum_{x=a\bar{g},y=guh,z=\bar{h}d}aud\\
&+&\sum_{x=cv\bar{g},y=gb,z=\overline{b}\bar{v}d}cd
\end{eqnarray*}

A similar computation shows that $x\times (y\times z)$ is given by the same formula.
\end{proof}

Still following \cite{ba1}, we have the following result:

\begin{proposition} 
Consider the following morphism, with $S,T$ being the shifts,
$$P:(E,+,\cdot )\to (B(l^2(W)),+,\circ )$$
$$\alpha\to S+T^*$$
and let $E_n\subset E$ be the linear space generated by the words of $W$ having length $\leq n$.
\begin{enumerate}
\item If  $J:E\to E$ is the map $f\to P(f)e$, then $(J-Id)E_n\subset E_{n-1}$ for any $n$.

\item $J$ is an isomorphism of $*$-algebras $(E,+,\cdot)\simeq(E,+,\times)$.
\end{enumerate}
\end{proposition}

\begin{proof}
We have several assertions here, the idea being as follows:

\medskip

(1) Let $f\in E$. We have then the following formula:
$$P(\alpha )f=(S+T^*)f=\circ\times f$$

Thus, for any $g\in E$, we have the following formula:
\begin{eqnarray*}
J(\circ g)
&=&P(\circ)J(g)\\
&=&\circ\times J(g)\\
&=&J(\circ)\times J(g)
\end{eqnarray*}

The same argument shows that we have, for any $g\in E$:
$$J(\bullet g)=J(\bullet)\times J(g)$$

Now the algebra $(E,+,\cdot )$ being generated by $\circ$ and $\bullet$, we conclude that $J$ is a morphism of algebras, as follows:
$$J:(E,+,\cdot )\to (E,+,\times )$$

We prove now by recurrence on $n\geq1$ that we have:
$$(J-Id)E_n\subset E_{n-1}$$

At $n=1$ we have $J(\circ)=\circ$, $J(\bullet)=\bullet$ and $J(e)=e$, and since $E_1$ is generated by $e,\circ,\bullet$, we have $J=Id$ on $E_1$, as desired. Now assume that the above formula is true for $n$, and let $k\in E_{n+1}$. We write, with $f,g,h\in E_n$:
$$k=\circ f+\bullet g+h$$

We have then the following computation:
\begin{eqnarray*}
(J-Id)k
&=&J(\circ f+\bullet g+h)-(\circ f+\bullet g+h)\\
&=&[(S+T^*)J(f)+(S^*+T)J(g)+J(h)]-[Sf+Tg+h]\\
&=&S(J(f)-f)+T(J(g)-g)+T^*J(f)+S^*J(g)+(J(h)-h)
\end{eqnarray*}

By using the recurrence assumption, applied to $f,g,h$ we find that $E_n$ contains all the terms of the above sum, and so contains $(J-Id)k$, and we are done.

\medskip

(2) Here we have to prove that $J$ preserves the involution $*$, and that it is bijective. We have $J*=*J$ on the generators $\{e,\circ,\bullet\}$ of $E$, so $J$ preserves the involution. Also, by (1), the restriction of $J-Id$ to $E_n$ is nilpotent, so $J$ is bijective. 
\end{proof}

Following \cite{ba1}, we can now formulate a main result about $U_N^+$, which is quite similar to the result for $O_N^+$ from chapter 5 above, as follows:

\index{unitary quantum group}
\index{free unitary quantum group}
\index{circular law}
\index{Voiculescu circular law}
\index{free Clebsch-Gordan rules}

\begin{theorem}
For the quantum group $U_N^+$, with $N\geq2$, the main character follows the Voiculescu circular law,
$$\chi\sim\Gamma_1$$
and the irreducible representations can be labelled by the colored integers, $k=\circ\bullet\bullet\circ\ldots$\,, with $r_e=1$, $r_\circ=u$, $r_\bullet=\bar{u}$, and with the involution and the fusion rules being
$$\bar{r}_k=r_{\bar{k}}$$
$$r_k\otimes r_l=\sum_{k=xy,l=\bar{y}z}r_{xz}$$
where $k\to\bar{k}$ is obtained by reversing the word, and switching the colors.
\end{theorem}

\begin{proof}
This is similar to the proof for $O_N^+$, as follows:

\medskip
(1) In order to get familiar with the fusion rules, let us first work out a few values of the representations $r_k$, computed according to the formula in the statement:
$$r_e=1$$
$$r_\circ=u$$
$$r_\bullet=\bar{u}$$
$$r_{\circ\circ}=u\otimes u$$
$$r_{\circ\bullet}=u\otimes\bar{u}-1$$
$$r_{\bullet\circ}=\bar{u}\otimes u-1$$
$$r_{\bullet\bullet}=\bar{u}\otimes\bar{u}$$
$$\vdots$$

(2) Equivalently, we want to decompose into irreducibles the Peter-Weyl representations, because the above formulae can be written as follows:
$$u^{\otimes e}=r_e$$
$$u^{\otimes\circ}=r_\circ$$
$$u^{\otimes\bullet}=r_\bullet$$
$$u^{\otimes\circ\circ}=r_{\circ\circ}$$
$$u^{\otimes\circ\bullet}=r_{\circ\bullet}+r_e$$
$$u^{\bullet\circ}=r_{\bullet\circ}+r_e$$
$$u^{\bullet\bullet}=r_{\bullet\bullet}$$
$$\vdots$$

(3) In order to prove the fusion rule assertion, let us construct a morphism as follows, by using the polynomiality of the algebra on the left:
$$\Psi: (E,+,\times )\to C(U_N^+)$$
$$\circ\to\chi (u)$$
$$\bullet\to\chi(\bar{u})$$

Our claim is that, given an integer $n\geq1$, assuming that $\Psi(x)$ is the character of an irreducible representation $r_x$ of $U_N^+$, for any $x\in W$ having length $\leq n$, then $\Psi(x)$ is the character of a non-null representation of $U_N^+$, for any $x\in W$ of length $n+1$.

\medskip

(4) At $n=1$ this is clear. Assume $n\geq2$, and let $x\in W$ of length $n+1$. If $x$ contains a $\geq 2$ power of $\circ$ or of $\bullet$, for instance if $x=z\circ\circ\,y$, then we can set:
$$r_x=r_{z\circ}\otimes r_{\circ y}$$

Assume now that $x$ is an alternating product of $\circ$ and $\bullet$. We can assume that $x$ begins with $\circ$. Then $x=\circ\bullet\circ\,y$, with $y\in W$ being of length $n-2$. Observe that $\Psi(\bar{z})=\Psi(z)^*$ holds on the generators $\{e,\circ,\bullet\}$ of $W$, so it holds for any $z\in W$. Thus, we have: 
\begin{eqnarray*}
<\chi (r_\circ\otimes r_{\bullet\circ y}),\chi (r_{\circ y})>
&=&<\chi (r_{\bullet\circ y}),\chi (r_\bullet\otimes r_{\circ y})>\\
&=&<\chi (r_{\bullet\circ y}),\Psi(\bullet\times\circ y )>\\
&=&<\chi (r_{\bullet\circ y}),\Psi(\bullet\circ y) +\Psi(y)>\\
&=&<\chi (r_{\bullet\circ y}),\chi (r_{\bullet\circ y})+\chi(r_y)>\\
&\geq&1
\end{eqnarray*}

Now since the corepresentation $r_{\circ y}$ is by assumption irreducible, we have $r_{\circ y}\subset r_\circ\otimes r_{\bullet\circ y}$. Consider now the following quantity:
\begin{eqnarray*}
\chi (r_\circ\otimes r_{\bullet\circ y})-\chi (r_{\circ y})
&=&\Psi(\circ\times \bullet\circ y-\circ y)\\
&=&\Psi(x)
\end{eqnarray*}

This is then the character of a representation, as desired.

\medskip

(5) We know from easiness that we have the following estimate:
$$\dim(Fix(u^{\otimes k}))\leq|\mathcal{NC}_2(k)|$$

By identifying as usual $(\mathbb C<X,X^*>,+,\cdot )=(E,+,\cdot )$, the noncommutative monomials in $X,X^*$ correspond to the elements of $W\subset E$. Thus, we have, on $W$:
$$h\Psi J\leq\tau J$$

(6) We prove now by recurrence on $n\geq 0$ that for any $z\in W$ having length $n$, $\Psi(z)$ is the character of an irreducible representation $r_z$. 

\medskip

(7) At $n=0$ we have $\Psi_G(e)=1$. So, assume that our claim holds at $n\geq 0$, and let $x\in W$ having length $n+1$. By Proposition 6.19 (1) we have, with  $z\in E_n$:
$$J(x)=x+z$$

Let $E^N\subset E$ be the set of functions $f$ such that $f(x)\in\mathbb N$ for any $x\in W$. Then $J(\alpha ),J(\beta )\in E^N$, so by multiplicativity $J(W)\subset E^N$. In particular, $J(x)\in E^N$. Thus there exist numbers $m(z)\in\mathbb N$ such that:
$$J(x)=x+\sum_{l(z)\leq n}m(z)z$$

(8) It is clear that for $a,b\in W$ we have $\tau(a\times\bar{b})=\delta_{a,b}$. Thus:
\begin{eqnarray*}
\tau J(x\bar{x})
&=&\tau\left(\left(x+\sum m(z)z\right)\times\left(\bar{x}+\sum m(z)\bar{z}\right)\right)\\
&=&1+\sum m(z)^2
\end{eqnarray*}

(9) By recurrence and by (3), $\Psi(x)$ is the character of a representation $r_x$. Thus  $\Psi J(x)$ is the character of $r_x+\sum_{l(z)\leq n}m(z)r_z$, and we obtain from this:
$$h\Psi J(x\bar{x})\geq h(\chi (r_x)\chi (r_x)^*)+\sum m(z)^2$$

(10) By using (5), (8), (9) we conclude that $r_x$ is irreducible, which proves (6). 

\medskip

(11) The fact that the $r_x$ are distinct comes from (5). Indeed, $W$ being an orthonormal basis of $((E,+,\times ),\tau )$, for any $x,y\in W$, $x\neq y$ we have $\tau (x\times\bar{y})=0$, and so: 
\begin{eqnarray*}
h(\chi (r_x\otimes\bar{r_y}))
&=&h\Psi J(x\bar{y})\\
&\leq&\tau J(x\bar{y})\\
&=&\tau (x\times\bar{y})\\
&=&0
\end{eqnarray*}

(12) The fact that we obtain all the irreducible representations is clear too, because we can now decompose all the tensor powers $u^{\otimes k}$ into irreducibles.

\medskip

(13) Finally, since $W$ is an orthonormal system in $((E,+,\times ),\tau )$, the set $\Psi(W)=\{\chi (r_x)|x\in W\}$ is an orthonormal system in $C(U_N^+)$, and so we have:
$$h\Psi J=\tau_0P$$

Now since the distribution of $\chi (u)\in (C(G),h)$ is the functional $h\Psi_GJ$, and the distribution of $S+T^*\in(B(l^2(\mathbb N*\mathbb N)),\tau_0)$ is the functional $\tau_0P$, we have $\chi\sim\Gamma_1$, as claimed.
\end{proof}

The above proof, from \cite{ba1}, is the original proof, still doing well after all these years, but there are some alternative proofs as well, to be discussed in the next section.

\section*{6d. Further results}

Let us discuss now the relation with $O_N^+$. As explained earlier in this chapter, in the classical case the passage $O_N\to U_N$ is something not trivial, requiring a passage via the associated Lie algebras. In the free case the situation is very simple, as follows:

\index{free complexification}

\begin{theorem}
We have an identification as follows,
$$U_N^+=\widetilde{O_N^+}$$
modulo the usual equivalence relation for compact quantum groups.
\end{theorem}

\begin{proof}
We recall from chapter 2 that the free complexification operation $G\to\widetilde{G}$ is obtained by multiplying the coefficients of the fundamental representation by a unitary free from them. We have embeddings as follows, with the first one coming by using the counit, and with the second one coming from the universality property of $U_N^+$:
$$O_N^+
\subset\widetilde{O_N^+}
\subset U_N^+$$

We must prove that the embedding on the right is an isomorphism, and there are several ways of doing this, all instructive, as follows:

\medskip

(1) The original argument, from \cite{ba1}, is something quick and advanced, based on the standard free probability fact that when freely multiplying a semicircular variable by a Haar unitary we obtain a circular variable \cite{vdn}. Thus, the main character of $\widetilde{O_N^+}$ is circular, exactly as for $U_N^+$, and by Peter-Weyl we obtain that the inclusion $\widetilde{O_N^+}\subset U_N^+$ must be an isomorphism, modulo the usual equivalence relation for quantum groups.

\medskip

(2) A version of this proof, not using any prior free probability knowledge, is by using fusion rules. Indeed, as explained in chapter 2 above, the representations of the dual free products, and in particular of the free complexifications, can be explicitely computed. Thus the fusion rules for $\widetilde{O_N^+}$ appear as a ``free complexification'' of the Clebsch-Gordan rules for $O_N^+$, and in practice this leads to the same fusion rules as for $U_N^+$. As before, by Peter-Weyl we obtain from this that the inclusion $\widetilde{O_N^+}\subset U_N^+$ must be an isomorphism, modulo the usual equivalence relation for the compact quantum groups.

\medskip 

(3) A third proof of the result, based on the same idea, and which is perhaps the simplest, makes use of the easiness property of $O_N^+,U_N^+$ only. Indeed, let us denote by $v,zv,u$ the fundamental representations of the following quantum groups:
$$O_N^+\subset\widetilde{O_N^+}\subset U_N^+$$

At the level of the associated Hom spaces we obtain reverse inclusions, as follows:
$$Hom(v^{\otimes k},v^{\otimes l})
\supset Hom((zv)^{\otimes k},(zv)^{\otimes l})
\supset Hom(u^{\otimes k},u^{\otimes l})$$

The spaces on the left and on the right are known from chapter 4 above, the result there stating that these spaces are as follows:
$$span\left(T_\pi\Big|\pi\in NC_2(k,l)\right)\supset span\left(T_\pi\Big|\pi\in\mathcal{NC}_2(k,l)\right)$$

Regarding the spaces in the middle, these are obtained from those on the left by ``coloring'', so we obtain the same spaces as those on the right. Thus, by Tannakian duality, our embedding $\widetilde{O_N^+}\subset U_N^+$ is an isomorphism, modulo the usual equivalence relation.
\end{proof}

As a comment here, the proof (3) above, when properly worked out, provides as well an alternative proof for Theorem 6.20. Indeed, once we know that we have $U_N^+=\widetilde{O_N^+}$, it follows that the fusion rules for $U_N^+$ appear as a ``free complexification'' of the Clebsch-Gordan rules for $O_N^+$, and in practice this leads to the formulae in Theorem 6.20.

\bigskip

However, this is nowhere done in the literature, and if you prefer this kind of proof, which is purely algebraic, you will have to work it out by yourself. The problem is that, with this proof, you still have to show afterwards that $\chi$ is circular, and this is best done starting from $U_N^+=\widetilde{O_N^+}$, and using the polar decomposition of circular variables, which is a free probability result due to Voiculescu \cite{vdn}, which is not exactly trivial.

\bigskip

Let us summarize this discussion by recording the following fact:

\begin{fact}
It is possible to establish the main results regarding $U_N^+$, namely
\begin{enumerate}
\item Free compexification, $U_N^+=\widetilde{O_N^+}$

\item Fusion rules, $r_k\otimes r_l=\sum_{k=xy,l=\bar{y}z}r_{xz}$

\item Character law, $\chi\sim\Gamma_1$
\end{enumerate}
by using diagrams for $(1)$, and then proving $(1)\implies(2),(3)$.
\end{fact}

Which leads us into the question on why \cite{ba1} was not written in this way, because that was a research paper, where the use of anything from \cite{vdn} was allowed anyway. Well, the story here is that \cite{ba1} was my PhD thesis, and my advisor Georges Skandalis, as one of the main architects, with his colleague Saad Baaj, of the theory of locally compact quantum groups with $S^2\neq id$, was insisting for me to do the work for the $S^2\neq id$ analogues of $U_N^+$ too, and for certain technical reasons, this cannot be done as in Fact 6.22.

\bigskip

In short, and as an advice now if you are a PhD student, just shut up and do what your advisor is saying, as a perfect mercenary. Discipline first, and learn to kill anything upon request, that's always a good skill to have. And plenty of time later to fully express yourself, during a long career. As I am actually doing myself now, when writing this book, with $S^2=id$ as an axiom, contrary to what Georges and Saad have taught me.

\bigskip

Back to work now, as an interesting consequence of the above result, we have:

\index{projective version}
\index{projective quantum group}

\begin{theorem}
We have an identification as follows,
$$PO_N^+=PU_N^+$$
modulo the usual equivalence relation for compact quantum groups.
\end{theorem}

\begin{proof}
As before, we have several proofs for this result, as follows:

\medskip

(1) This follows from Theorem 6.21, because we have:
$$PU_N^+=P\widetilde{O_N^+}=PO_N^+$$

(2) We can deduce this as well directly. With notations as before, we have:
$$Hom\left((v\otimes v)^k,(v\otimes v)^l\right)=span\left(T_\pi\Big|\pi\in NC_2((\circ\bullet)^k,(\circ\bullet)^l)\right)$$
$$Hom\left((u\otimes\bar{u})^k,(u\otimes\bar{u})^l\right)=span\left(T_\pi\Big|\pi\in \mathcal{NC}_2((\circ\bullet)^k,(\circ\bullet)^l)\right)$$

The sets on the right being equal, we conclude that the inclusion $PO_N^+\subset PU_N^+$ preserves the corresponding Tannakian categories, and so must be an isomorphism.
\end{proof}

As a conclusion, the passage $O_N^+\to U_N^+$ is something much simpler than the passage $O_N\to U_N$, with this ultimately coming from the fact that the combinatorics of $O_N^+,U_N^+$ is something much simpler than the combinatorics of $O_N,U_N$. In addition, all this leads as well to the interesting conclusion that the free projective geometry does not fall into real and complex, but is rather unique and ``scalarless''. We will be back to this.

\bigskip

More generally, once again by following \cite{ba1}, we have similar results obtained by replacing $O_N^+$ with the more general super-orthogonal quantum groups $O_F^+$ from the previous chapter, which include as well the free symplectic groups $Sp_N^+$. Let us start with:

\index{super-orthogonal group}
\index{free symplectic group}

\begin{theorem}
We have an identification as follows,
$$U_N^+=\widetilde{O_F^+}$$
valid for any super-orthogonal quantum group $O_F^+$.
\end{theorem}

\begin{proof}
This is a straightforward extension of Theorem 6.21 above, with any of the proofs there extending to the case of the quantum groups $O_F^+$. See \cite{ba1}.
\end{proof}

We have as well a projective version of the above result, as follows:

\begin{theorem}
We have an identification as follows,
$$PU_N^+=PO_F^+$$
valid for any super-orthogonal quantum group $O_F^+$.
\end{theorem}

\begin{proof}
This is a straightforward extension of Theorem 6.23, with any of the proofs there extending to the case of the quantum groups $O_F^+$. Alternatively, the result follows from Theorem 6.24, by taking the projective versions of the quantum groups there.
\end{proof}

The free symplectic result at $N=2$ is particularly interesting, because here we have $Sp_2^+=SU_2$, and so we obtain that $U_2^+$ is the free complexification of $SU_2$:

\begin{theorem}
We have an identification as follows,
$$U_2^+=\widetilde{SU_2}$$
modulo the usual equivalence relation for compact quantum groups.
\end{theorem}

\begin{proof}
As explained above, this follows from Theorem 6.24, and from $Sp_2^+=SU_2$, via the material explained in chapter 5 above. See \cite{ba1}, \cite{bsk}.
\end{proof}

Finally, we have a projective version of the above result, as follows:

\index{projective version}

\begin{theorem}
We have an identification as follows, and this even without using the standard equivalence relation for the compact quantum groups:
$$PU_2^+=SO_3$$
A similar result holds for the ``left'' projective version of $U_2^+$, constructed by using the corepresentation $\bar{u}\otimes u$ instead of $u\otimes\bar{u}$.
\end{theorem}

\begin{proof}
We have several assertions here, the idea being as follows:

\medskip

(1) By using Theorem 6.26 we obtain, modulo the equivalence relation:
$$PU_2^+=P\widetilde{SU_2}=PSU_2=SO_3$$

(2) Now since $SO_3$ is coamenable, the above formula must hold in fact in a plain way, meaning without using the equivalence relation. This can be checked as well directly, by verifying that the coefficients of $u\otimes\bar{u}$ commute indeed.

\medskip

(3) Finally, the last assertion can be either deduced from the first one, or proved directly, by using ``left'' free complexification operations, in all the above.
\end{proof}

We refer to \cite{ba1} for further applications of the above $N=2$ results, for instance with structure results regarding the von Neumann algebra $L^\infty(U_2^+)$. We will be back to $U_N^+$ in chapter 8 below, with a number of more advanced probabilistic results about it.

\section*{6e. Exercises}

As with the exercices from the previous chapter, regarding the quantum group $O_N^+$, we will mainly focus here on combinatorics and probability. Let us start with:

\begin{exercise}
Given two $C^*$-algebras with traces $A,B$, prove that these algebras are independent inside $A\otimes B$, and free inside $A*B$.
\end{exercise}

Here the independence assertion is quite straightforward, and the freeness assertion requires some preliminary work, in order to construct a trace on $A*B$. For this latter construction, the general formulae for freeness discussed in this chapter can be used.

\begin{exercise}
Given two discrete groups $\Gamma,\Lambda$, prove that the algebras $C^*(\Gamma),C^*(\Lambda)$ are independent inside $C^*(\Gamma\times\Lambda)$, and free inside $C^*(\Gamma*\Lambda)$.
\end{exercise}

The results here can be deduced either directly, by verifying the defining formulae for independence and freeness, or via the result from the previous exercise.

\begin{exercise}
Prove that the quantum group inclusion
$$PO_N^+\subset PU_N^+$$
is an isomorphism, by showing that the corresponding tensor categories coincide.
\end{exercise}

This is something that we already discussed in the above, the problem now being that of finding an explicit, complete proof for this, by using that method.

\begin{exercise}
Work out the details of the identification
$$U_2^+=\widetilde{SU_2}$$
and of the corresponding isomorphism at the level of diagonal tori.
\end{exercise}

To be more precise, the above identification is something that we already know, coming from abstract results, and the problem now is that of doing all this explicitely.

\begin{exercise}
Work out a theory of left and right projective versions for the compact quantum groups, and prove that
$$PU_2^+=SO_3$$
happens, independently of the projective version theory which is used.
\end{exercise}

Aa before with some other exercises, this is something that we already discussed in the above, but just briefly, and the problem now is that of clarifying all this, with full theory and details, examples and counterexamples, and so on.

\chapter{Easiness, twisting}

\section*{7a. Partitions, easiness}

We have seen that the Brauer theorems for $O_N,U_N$ and $O_N^+,U_N^+$, stating that these quantum groups are ``easy'', have a number of interesting algebraic and probabilistic consequences. All this was non-trivial, and natural too, going straight to the various problems about $O_N,U_N$ and $O_N^+,U_N^+$ that we wanted to solve. And so obviously, easiness and its extensions are the right tool for dealing with the closed subgroups $G\subset U_N^+$.

\bigskip

Our purpose now will be that of systematically developing the theory of easiness. There are plenty of things that can be done, as follows:

\bigskip

(1) First we have the question of working out Brauer theorems for $O_N^*,U_N^*$ too, and also for other suitably chosen quantum groups, which are of the same type as $O_N,U_N$, such as the bistochastic groups $B_N,C_N$, and their free analogues $B_N^+,C_N^+$. This is something quite straightforward, and we will do this in the beginning of this chapter.

\bigskip

(2) We also have the question of investigating various operations, and in particular the twisting, which can bring us into more complicated quantum groups, which are not exactly easy, such as the symplectic groups $Sp_N$ and their free analogues $Sp_N^+$, or the $q=-1$ twists of all groups that we know. We will discuss this too, in this chapter.

\bigskip

(3) There are also all sorts of further probabilistic things that can be done with $O_N,U_N$ and $O_N^+,U_N^+$, as well as with their various versions mentioned above. And with the theory here potentially going quite far, into things like de Finetti theorems, or all sorts of random matrix computations. Chapter 8 below will be an introduction to all this.

\bigskip

(4) Finally, we can take a look as well at the possible easiness property and liberation theory of various finite groups, such as the symmetric group $S_N\subset O_N$, or the hyperoctahedral group $H_N\subset O_N$, or other complex reflection groups $G\subset U_N$. And why not, start classifying all these beasts. We will discuss this in chapters 9-12 below.

\bigskip

In short, many things to be done. In addition, there are countless ways of presenting this material, with (1,2,3,4) being more of less interchangeable. Our ordering here comes from the principles ``algebra first, analysis after'' and ``continuous first, discrete after'', that we use as a philosophy for this book. Not necessarily that we adhere to this philosophy, but hey, a book is a linear presentation of something, and go find that linearity.

\bigskip

Finally, and above everything, let us mention that all this will be an introduction to easiness, and even a modest one. Such things are known since Weyl \cite{wey} and Brauer \cite{bra} in the classical case, meaning 1930s, and regarding probabilistic aspects, these go back to Weingarten \cite{wei}, meaning 1970s. As for quantum extensions, these go back algebrically to work of Wang \cite{wa1} and mine \cite{ba1} from the 1990s, and then analytically to my paper with Collins \cite{bc1}, from the mid 2000s. So, as you can imagine, plenty of things that are known, on all these topics, dozens or even hundreds of papers written.

\bigskip

Be said in passing, now that you might ask yourself this question: is actually everything known? Not at all. There are all sorts of open problems regarding classification, operations, twisting, probability aspects. Easiness over a field $k$. No one knows how to liberate the exceptional Lie groups. What is a Lie-Brauer algebra. Open questions in relation with Jones' planar algebras. And many more. And finally, all this was about pure mathematics, but when it comes to going into physics, and looking for applications, there are just some beautiful, viable ideas there, needing work. Lots of work.

\bigskip

But enough talking, and more on this later. Getting started now, we will follow my paper with Speicher \cite{bsp}, where the easy quantum groups were axiomatized in the orthogonal case, and with that paper being actually a pleasant, accessible read, with what you learned so far from this book. We have the following definition from there, slightly extended afterwards by Tarrago-Weber \cite{twe}, as to fit with the general unitary case:

\index{category of partitions}

\begin{definition}
Let $P(k,l)$ be the set of partitions between an upper colored integer $k$, and a lower colored integer $l$. A collection of subsets 
$$D=\bigsqcup_{k,l}D(k,l)$$
with $D(k,l)\subset P(k,l)$ is called a category of partitions when it has the following properties:
\begin{enumerate}
\item Stability under the horizontal concatenation, $(\pi,\sigma)\to[\pi\sigma]$.

\item Stability under vertical concatenation $(\pi,\sigma)\to[^\sigma_\pi]$, with matching middle symbols.

\item Stability under the upside-down turning $*$, with switching of colors, $\circ\leftrightarrow\bullet$.

\item Each set $P(k,k)$ contains the identity partition $||\ldots||$.

\item The sets $P(\emptyset,\circ\bullet)$ and $P(\emptyset,\bullet\circ)$ both contain the semicircle $\cap$.
\end{enumerate}
\end{definition} 

Observe the similarity with the axioms of Tannakian categories, from chapter 4. We will see in a moment that this similarity can be made into a precise theorem, stating that any category of partitions produces a family $G=(G_N)$ of easy quantum groups.

\bigskip

We have already met a number of such categories, in chapter 4. Indeed, the sets of Brauer pairings for $O_N,U_N$ and for $O_N^+,U_N^+$ are all categories of partitions in the above abstract sense, with inclusions between them as follows:
$$\xymatrix@R=16mm@C=17mm{
\mathcal P_2\ar[d]&\mathcal{NC}_2\ar[l]\ar[d]\\
P_2&NC_2\ar[l]
}$$

There are many other examples, as for instance $P$ itself, or the category $NC\subset P$ of noncrossing partitions. We have as well the category $P_{even}$ of partitions having even blocks, and its subcategory $NC_{even}$. These categories form a diagram as follows:
$$\xymatrix@R=15mm@C=15mm{
\mathcal P_{even}\ar[d]&\mathcal{NC}_{even}\ar[l]\ar[d]\\
P_{even}&NC_{even}\ar[l]
}$$

And there are many other examples, to be gradually explored in what follows. And with the comment, coming a bit in advance, that now that we talked about $P,NC$ and $P_{even},NC_{even}$, we must say what the corresponding quantum groups should be, right. Well, these are $S_N,S_N^+$ and $H_N,H_N^+$, to be discussed in chapters 9-12 below.

\bigskip

Getting back now to Definition 7.1 as it is, namely something abstract, of categorical flavor, the relation with the Tannakian categories comes from:

\index{Kronecker symbol}

\begin{proposition}
Each partition $\pi\in P(k,l)$ produces a linear map 
$$T_\pi:(\mathbb C^N)^{\otimes k}\to(\mathbb C^N)^{\otimes l}$$
given by the following formula, where $e_1,\ldots,e_N$ is the standard basis of $\mathbb C^N$, 
$$T_\pi(e_{i_1}\otimes\ldots\otimes e_{i_k})=\sum_{j_1\ldots j_l}\delta_\pi\begin{pmatrix}i_1&\ldots&i_k\\ j_1&\ldots&j_l\end{pmatrix}e_{j_1}\otimes\ldots\otimes e_{j_l}$$
and with the Kronecker type symbols $\delta_\pi\in\{0,1\}$ depending on whether the indices fit or not. The assignement $\pi\to T_\pi$ is categorical, in the sense that we have
$$T_\pi\otimes T_\sigma=T_{[\pi\sigma]}\quad,\quad 
T_\pi T_\sigma=N^{c(\pi,\sigma)}T_{[^\sigma_\pi]}\quad,\quad 
T_\pi^*=T_{\pi^*}$$
where $c(\pi,\sigma)$ are certain integers, coming from the erased components in the middle.
\end{proposition}

\begin{proof}
This is something that we already know for the pairings, from chapter 4 above. In general, the proof is identical, via exactly the same computations.
\end{proof}

In relation with the quantum groups, we have the following result, from \cite{bsp}:

\index{Tannakian duality}

\begin{theorem}
Each category of partitions $D=(D(k,l))$ produces a family of compact quantum groups $G=(G_N)$, one for each $N\in\mathbb N$, via the formula
$$Hom(u^{\otimes k},u^{\otimes l})=span\left(T_\pi\Big|\pi\in D(k,l)\right)$$
which produces a Tannakian category, and the Tannakian duality correspondence.
\end{theorem}

\begin{proof}
This follows indeed from Woronowicz's Tannakian duality \cite{wo2}, best formulated in its ``soft'' form from \cite{mal}, as explained in chapter 4. Indeed, let us set:
$$C(k,l)=span\left(T_\pi\Big|\pi\in D(k,l)\right)$$

By using the axioms in Definition 7.1, and the categorical properties of the operation $\pi\to T_\pi$, from Proposition 7.2, we deduce that $C=(C(k,l))$ is a Tannakian category, in the sense of chapter 4. Thus Tannakian duality applies, and gives the result.
\end{proof}

We already know, from chapter 4, that $O_N,O_N^+$ and $U_N,U_N^+$ appear in this way, with $D$ being respectively $P_2,NC_2$ and $\mathcal P_2,\mathcal{NC}_2$. In general now, let us formulate:

\index{easiness}
\index{easy quantum group}

\begin{definition}
A closed subgroup $G\subset U_N^+$ is called easy when we have
$$Hom(u^{\otimes k},u^{\otimes l})=span\left(T_\pi\Big|\pi\in D(k,l)\right)$$
for any colored integers $k,l$, for a certain category of partitions $D\subset P$.
\end{definition}

In other words, a compact quantum group is called easy when its Tannakian category appears in the simplest possible way: from a category of partitions. The terminology is quite natural, because Tannakian duality is basically our only serious tool.

\bigskip

Observe that the category $D$ is not unique, for instance because at $N=1$ all the categories of partitions produce the same easy quantum group, namely $G=\{1\}$. We will be back to this issue on several occasions, with various results about it. 

\bigskip

In practice now, what we know so far, from chapter 4 above, is that $O_N,O_N^+$ and $U_N,U_N^+$ are easy.  Regarding now the half-liberations, we have here:

\index{Brauer theorem}
\index{half-classical quantum group}
\index{half-classical orthogonal group}
\index{half-classical unitary group}
\index{half-classical crossing}

\begin{theorem}
We have the following results:
\begin{enumerate}
\item $U_N^*$ is easy, coming from the category $\mathcal P_2^*\subset\mathcal P_2$ of pairings having the property that, when the legs are relabelled clockwise $\circ\bullet\circ\bullet\ldots$, each string connects $\circ-\bullet$.

\item $O_N^*$ is easy too, coming from the category $P_2^*\subset P_2$ of pairings having the same property: when legs are labelled clockwise $\circ\bullet\circ\bullet\ldots$, each string connects $\circ-\bullet$.
\end{enumerate}
\end{theorem}

\begin{proof}
We can proceed here as in the proof for $U_N,O_N$, from chapter 4, by replacing the basic crossing by the half-commutation crossing, as follows:

\medskip

(1) Regarding $U_N^*\subset U_N^+$, it is elementary to check, via some computations that we will skip, that the half-commutation relations $abc=cba$ are implemented by the half-commutation crossing ${\slash\hskip-2.1mm\backslash\hskip-1.65mm|}$\,. Thus the corresponding Tannakian category is generated by the operators $T_\pi$, with $\pi={\slash\hskip-2.1mm\backslash\hskip-1.65mm|}$\,, taken with all the possible $2^3=8$ matching colorings. Since these latter 8 partitions generate the category $\mathcal P_2^*$, we obtain the result.

\medskip

(2) For $O_N^*$ we can either proceed similarly, or by using the following formula: 
$$O_N^*=O_N^+\cap U_N^*$$

Indeed, at the categorical level, this latter formula tells us that the associated Tannakian category is given by $C=span(T_\pi|\pi\in D)$, with:
$$D
=<NC_2,\mathcal P_2^*>
=P_2^*$$

Thus, we are led to the conclusion in the statement.
\end{proof}

Let us collect now the results that we have so far in a single theorem, as follows:

\index{unitary quantum group}
\index{category of pairings}

\begin{theorem}
The basic unitary quantum groups, namely
$$\xymatrix@R=15mm@C=15mm{
U_N\ar[r]&U_N^*\ar[r]&U_N^+\\
O_N\ar[r]\ar[u]&O_N^*\ar[r]\ar[u]&O_N^+\ar[u]}$$
are all easy, the corresponding categories of partitions being:
$$\xymatrix@R=16mm@C=15mm{
\mathcal P_2\ar[d]&\mathcal P_2^*\ar[l]\ar[d]&\mathcal{NC}_2\ar[l]\ar[d]\\
P_2&P_2^*\ar[l]&NC_2\ar[l]}$$
\end{theorem}

\begin{proof}
This follows indeed from our various Brauer type results.
\end{proof}

We have seen in chapters 5-6 that the easiness property of $O_N^+,U_N^+$ leads to some interesting consequences. Regarding $O_N^*,U_N^*$, as a main consequence, we can now compute their projective versions, as part of the following general result:

\index{projective version}
\index{projective unitary group}
\index{projective unitary quantum group}

\begin{theorem}
The projective versions of the basic quantum groups are as follows,
$$\xymatrix@R=15mm@C=15mm{
PU_N\ar[r]&PU_N\ar[r]&PU_N^+\\
PO_N\ar[r]\ar[u]&PU_N\ar[r]\ar[u]&PU_N^+\ar[u]}$$
when identifying, in the free case, full and reduced version algebras.
\end{theorem}

\begin{proof}
In the classical case, there is nothing to prove. Regarding the half-classical versions, consider the inclusions $O_N^*,U_N\subset U_N^*$. These induce inclusions as follows:
$$PO_N^*,PU_N\subset PU_N^*$$

Our claim is that these inclusions are isomorphisms. Let indeed $u,v,w$ be the fundamental corepresentations of $O_N^*,U_N,U_N^*$. According to Theorem 7.5, we have:
$$Hom\left((u\otimes\bar{u})^k,(u\otimes\bar{u})^l\right)=span\left(T_\pi\Big|\pi\in P_2^*((\circ\bullet)^k,(\circ\bullet)^l)\right)$$
$$Hom\left((u\otimes\bar{u})^k,(u\otimes\bar{u})^l\right)=span\left(T_\pi\Big|\pi\in \mathcal P_2((\circ\bullet)^k,(\circ\bullet)^l)\right)$$
$$Hom\left((u\otimes\bar{u})^k,(u\otimes\bar{u})^l\right)=span\left(T_\pi\Big|\pi\in \mathcal P_2^*((\circ\bullet)^k,(\circ\bullet)^l)\right)$$

The sets on the right being equal, we conclude that the inclusions $O_N^*,U_N\subset U_N^*$ preserve the corresponding Tannakian categories, and so must be isomorphisms.

Finally, in the free case the result follows either from the free complexification result from chapter 6, or from Theorem 7.6, by using the same method as above.
\end{proof}

The above result is quite interesting, philosophically, because it shows that, in the nocommutative setting, the distinction between $\mathbb R$ and $\mathbb C$ becomes ``blurred''. We will be back to this later, with some related noncommutative geometry considerations.

\section*{7b. Basic operations} 

Let us discuss now composition operations. There are many questions to be discussed here, some of them being even open. To start with, we will be interested in:

\index{intersection of quantum groups}
\index{generation operation}
\index{topological generation}

\begin{proposition}
The closed subgroups of $U_N^+$ are subject to operations as follows:
\begin{enumerate}
\item Intersection: $H\cap K$ is the biggest quantum subgroup of $H,K$.

\item Generation: $<H,K>$ is the smallest quantum group containing $H,K$.
\end{enumerate}
\end{proposition}

\begin{proof}
We must prove that the universal quantum groups in the statement exist indeed. For this purpose, let us pick writings as follows, with $I,J$ being Hopf ideals:
$$C(H)=C(U_N^+)/I$$
$$C(K)=C(U_N^+)/J$$

We can then construct our two universal quantum groups, as follows:
$$C(H\cap K)=C(U_N^+)/<I,J>$$
$$C(<H,K>)=C(U_N^+)/(I\cap J)$$

Thus, we are led to the conclusions in the statement.
\end{proof}

In practice now, the operation $\cap$ can be usually computed by using:

\begin{proposition}
Assuming $H,K\subset G$, appearing at the algebra level as follows, with $\mathcal R,\mathcal P$ being certain sets of polynomial $*$-relations between the coordinates $u_{ij}$,
$$C(H)=C(G)/\mathcal R\quad,\quad 
C(K)=C(G)/\mathcal P$$
the intersection $H\cap K$ is given by the following formula,
$$C(H\cap K)=C(G)/\{\mathcal R,\mathcal P\}$$
again at the algebraic level.
\end{proposition}

\begin{proof}
This follows from Proposition 7.8, or rather from its proof, and from the following trivial fact, regarding relations and ideals:
$$I=<\mathcal R>,J=<\mathcal P>
\quad\implies\quad <I,J>=<\mathcal R,\mathcal P>$$

Thus, we are led to the conclusion in the statement.
\end{proof}

\index{Hopf image}

In order to discuss now $<\,,>$, let us call Hopf image of a representation $C(G)\to A$ the smallest Hopf algebra quotient $C(L)$ producing a factorization as follows:
$$C(G)\to C(L)\to A$$

The fact that this quotient exists indeed is routine, by dividing by a suitable ideal, and we will be back to this in chapter 16. This notion can be generalized as follows:

\begin{proposition}
Assuming $H,K\subset G$, the quantum group $<H,K>$ is such that
$$C(G)\to C(H\cap K)\to C(H),C(K)$$
is the joint Hopf image of the following quotient maps:
$$C(G)\to C(H),C(K)$$
\end{proposition}

\begin{proof}
In the particular case from the statement, the joint Hopf image appears as the smallest Hopf algebra quotient $C(L)$ producing factorizations as follows:
$$C(G)\to C(L)\to C(H),C(K)$$

We conclude from this that we have $L=<H,K>$, as desired.
\end{proof}

In the Tannakian setting now, we have the following result:

\begin{theorem}
The intersection and generation operations $\cap$ and $<\,,>$ can be constructed via the Tannakian correspondence $G\to C_G$, as follows:
\begin{enumerate}
\item Intersection: defined via $C_{G\cap H}=<C_G,C_H>$.

\item Generation: defined via $C_{<G,H>}=C_G\cap C_H$.
\end{enumerate}
\end{theorem}

\begin{proof}
This follows from Proposition 7.8, or rather from its proof, by taking $I,J$ to be the ideals coming from Tannakian duality, in its soft form, from chapter 4.
\end{proof}

In relation now with our easiness questions, we first have the following result:

\begin{proposition}
Assuming that $H,K$ are easy, then so is $H\cap K$, and we have
$$D_{H\cap K}=<D_H,D_K>$$
at the level of the corresponding categories of partitions.
\end{proposition}

\begin{proof}
We have indeed the following computation:
\begin{eqnarray*}
C_{H\cap K}
&=&<C_H,C_K>\\
&=&<span(D_H),span(D_K)>\\
&=&span(<D_H,D_K>)
\end{eqnarray*}

Thus, by Tannakian duality we obtain the result.
\end{proof}

Regarding the generation operation, the situation is more complicated, as follows:

\begin{proposition}
Assuming that $H,K$ are easy, we have an inclusion 
$$<H,K>\subset\{H,K\}$$
coming from an inclusion of Tannakian categories as follows,
$$C_H\cap C_K\supset span(D_H\cap D_K)$$
where $\{H,K\}$ is the easy quantum group having as category of partitions $D_H\cap D_K$.
\end{proposition}

\begin{proof}
This follows from the following computation:
\begin{eqnarray*}
C_{<H,K>}
&=&C_H\cap C_K\\
&=&span(D_H)\cap span(D_K)\\
&\supset&span(D_H\cap D_K)
\end{eqnarray*}

Indeed, by Tannakian duality we obtain from this all the assertions.
\end{proof}

It is not clear if the inclusions in Proposition 7.13 are isomorphisms or not, and this not even under a supplementary $N>>0$ assumption. Technically speaking, the problem comes from the fact that the operation $\pi\to T_\pi$ does not produce linearly independent maps. Summarizing, we have some problems here, and we must proceed as follows:

\index{easy generation operation}

\begin{theorem}
The intersection and easy generation operations $\cap$ and $\{\,,\}$ can be constructed via the Tannakian correspondence $G\to D_G$, as follows:
\begin{enumerate}
\item Intersection: defined via $D_{G\cap H}=<D_G,D_H>$.

\item Easy generation: defined via $D_{\{G,H\}}=D_G\cap D_H$.
\end{enumerate}
\end{theorem}

\begin{proof}
Here the situation is as follows:

\medskip

(1) This is a true result, coming from Proposition 7.12.

\medskip

(2) This is more of an empty statement, coming from Proposition 7.13.
\end{proof}

With the above notions in hand, we can now formulate a nice result, which improves our main result so far, namely Theorem 7.6 above, as follows:

\index{intersection diagram}
\index{generation diagram}
\index{intersection and generation diagram}

\begin{theorem}
The basic unitary quantum groups, namely
$$\xymatrix@R=15mm@C=15mm{
U_N\ar[r]&U_N^*\ar[r]&U_N^+\\
O_N\ar[r]\ar[u]&O_N^*\ar[r]\ar[u]&O_N^+\ar[u]}$$
are all easy, and they form an intersection and easy generation diagram, in the sense that any rectangular subdiagram
$$P\subset Q,R\subset S$$
of the above diagram satisfies the condition $P=Q\cap R,\{Q,R\}=S$.
\end{theorem}

\begin{proof}
We know from Theorem 7.6 that the quantum groups in the statement are indeed easy, the corresponding categories of partitions being as follows:
$$\xymatrix@R=16mm@C=15mm{
\mathcal P_2\ar[d]&\mathcal P_2^*\ar[l]\ar[d]&\mathcal{NC}_2\ar[l]\ar[d]\\
P_2&P_2^*\ar[l]&NC_2\ar[l]}$$

Now observe that this latter diagram is an intersection and generation diagram, in the sense that any rectangular subdiagram $P\subset Q,R\subset S$ satisfies the following condition:
$$P=Q\cap R\quad,\quad 
\{Q,R\}=S$$

By using Theorem 7.14, this reformulates into the fact that the diagram of quantum groups is an intersection and easy generation diagram, as claimed.
\end{proof}

It is possible to further improve the above result, by proving that the diagram there is actually a plain generation diagram. However, this is something quite technical, requiring advanced quantum group techniques, and we will comment on this later.

\bigskip

Let us explore now a number of further examples of easy quantum groups, which appear as ``versions'' of the basic unitary groups. With the convention that a matrix is called bistochastic when its entries sum up to 1, on each row and column, we have:

\index{bistochastic group}
\index{bistochastic quantum group}

\begin{proposition}
We have the following groups and quantum groups:
\begin{enumerate}
\item $B_N\subset O_N$, consisting of the orthogonal matrices which are bistochastic.

\item $C_N\subset U_N$, consisting of the unitary matrices which are bistochastic.

\item $B_N^+\subset O_N^+$, coming via $u\xi=\xi$, where $\xi$ is the all-one vector.

\item $C_N^+\subset U_N^+$, coming via $u\xi=\xi$, where $\xi$ is the all-one vector.
\end{enumerate}
Also, we have inclusions $B_N\subset B_N^+$ and $C_N\subset C_N^+$, which are both liberations.
\end{proposition}

\begin{proof}
Here the fact that $B_N,C_N$ are groups is clear, and $B_N^+,C_N^+$ are quantum groups too, because $\xi\in Fix(u)$ is categorical. Now observe that for $U\in U_N$ we have:
$$U\xi=\xi\iff U^*\xi=\xi$$

By conjugating, these conditions are equivalent as well to $\bar{U}\xi=\xi$, $U^t\xi=\xi$. Thus $U\in U_N$ is bistochastic precisely when $U\xi=\xi$, and this gives the last assertion.
\end{proof}

The above quantum groups are all easy, and following \cite{bsp}, \cite{twe}, we have:

\begin{theorem}
The basic orthogonal and unitary quantum groups and their bistochastic versions are all easy, and they form a diagram as follows,
$$
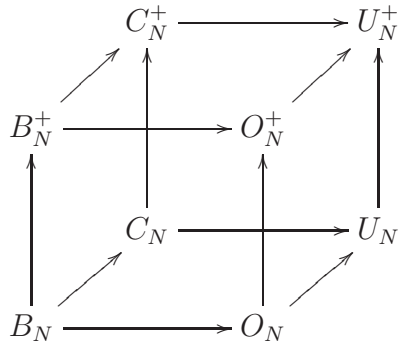
\xymatrix@R=18pt@C=18pt{
&C_N^+\ar[rr]&&U_N^+\\
B_N^+\ar[rr]\ar[ur]&&O_N^+\ar[ur]\\
&C_N\ar[rr]\ar[uu]&&U_N\ar[uu]\\
B_N\ar[uu]\ar[ur]\ar[rr]&&O_N\ar[uu]\ar[ur]
}$$
which is an intersection and easy generation diagram, in the sense of Theorem 7.15.
\end{theorem}

\begin{proof}
The first assertion comes from the fact that the all-one vector $\xi$ used in Proposition 7.16 is the vector associated to the singleton partition:
$$\xi=T_|$$

Indeed, we obtain that $B_N,C_N,B_N^+,C_N^+$ are easy, appearing from the categories of partitions for $O_N,U_N,O_N^+,U_N^+$, by adding singletons. In practice now, the categories of partitions for the quantum groups in the statement are as follows, with $12$ standing for ``singletons and pairings'', in the same way as the $2$ is standing for ``pairings'':
$$\xymatrix@R=20pt@C8pt{
&\mathcal{NC}_{12}\ar[dl]\ar[dd]&&\mathcal {NC}_2\ar[dl]\ar[ll]\ar[dd]\\
NC_{12}\ar[dd]&&NC_2\ar[dd]\ar[ll]\\
&\mathcal P_{12}\ar[dl]&&\mathcal P_2\ar[dl]\ar[ll]\\
P_{12}&&P_2\ar[ll]
}$$

\index{singletons and pairings}

Now since both this diagram and the one the statement are intersection diagrams, the quantum groups form an intersection and easy generation diagram, as stated.
\end{proof}

The above result is quite nice, among others because we are now exiting the world of pairings. However, there are a few problems with it. First, we cannot really merge it with Theorem 7.15, as to obtain as a final result a nice cubic diagram, containing all the quantum groups considered so far.  Indeed, the half-classical versions of the bistochastic quantum groups collapse, and so cannot be inserted into the cube, as shown by: 

\begin{proposition}
The half-classical versions of $B_N^+,C_N^+$ are given by:
$$B_N^+\cap O_N^*=B_N\quad,\quad 
C_N^+\cap U_N^*=C_N$$
In other words, the half-classical versions collapse to the classical versions.
\end{proposition}

\begin{proof}
This follows indeed from Tannakian duality, by using the fact that when capping the half-classical crossing with 2 singletons, we obtain the classical crossing.
\end{proof}

Yet another problem with the bistochastic quantum groups comes from the fact that these objects are not really ``new'', because, following Raum \cite{rau}, we have:

\begin{proposition}
We have isomorphisms as follows:
\begin{enumerate}
\item $B_N\simeq O_{N-1}$.

\item $B_N^+\simeq O_{N-1}^+$.

\item $C_N\simeq U_{N-1}$.

\item $C_N^+\simeq U_{N-1}^+$.
\end{enumerate}
\end{proposition}

\begin{proof}
Let us pick indeed a matrix $F\in U_N$ satisfying the following condition, where $\xi$ is the all-one vector:
$$Fe_0=\frac{1}{\sqrt{N}}\xi$$

Such matrices exist of course, the basic example being the Fourier matrix:
$$F_N=\frac{1}{\sqrt{N}}(w^{ij})_{ij}\quad,\quad w=e^{2\pi i/N}$$

We have then the following computation:
\begin{eqnarray*}
u\xi=\xi
&\iff&uFe_0=Fe_0\\
&\iff&F^*uFe_0=e_0\\
&\iff&F^*uF=diag(1,w)
\end{eqnarray*}

Thus we have an isomorphism given by $w_{ij}\to(F^*uF)_{ij}$, as desired.
\end{proof}

This being said, the bistochastic quantum groups $B_N,C_N$ and $B_N^+,C_N^+$ remain fundamental objects for us, and will appear several times in what follows. And let us mention too that the bistochastic matrices, and especially the unitary ones, $U\in C_N$, are cult objects in advanced matrix analysis, so all this was certainly worth developing.

\section*{7c. Ad-hoc twisting}

\index{ad-hoc twisting}
\index{Drinfeld-Jimbo}
\index{cocycle twisting}
\index{Schur-Weyl twisting}
\index{twisting}

Back to generalities now, the easy quantum groups are not the only ones ``coming from partitions'', but are rather the simplest ones having this property. An interesting and important class of compact quantum groups, which appear in relation with many questions, are the $q=-1$ twists of the compact Lie groups. Given a compact Lie group $G\subset U_N$, there are several methods for twisting it, as follows:

\medskip

\begin{enumerate}
\item Ad-hoc twisting. This basically amounts in replacing the commutation relations between the coordinates $u_{ij}\in C(G)$ by anticommutation. In practice, this is quite tricky, because some of these commutation relations must be kept as such.

\medskip

\item Cocycle twisting. This is something more conceptual, and more far-reaching, both at the level of the general theory and of the examples which can be obtained, the idea being that of twisting the multiplication of $C(G)$ by a cocycle.

\medskip

\item Schur-Weyl twisting. This is a method which works only in the easy case, and is the most powerful in this case, the idea here being that of using Tannakian duality, and twisting the construction $\pi\to T_\pi$, by using a signature map.

\medskip

\item Wrong twisting. This is a famous method, which out of a compact Lie group $G\subset U_N$, which is of course semisimple, produces, via a fairly complicated and advanced procedure, a certain non-semisimple beast, denoted $G^{-1}$.
\end{enumerate}

\medskip

We will discuss here this material, first by working out the main examples, by using the ad-hoc strategy explained in (1), and then by getting into more advanced aspects, of algebraic and representation theory flavor, following the ideas in (2) and (3). 

\bigskip

As for (4), we have already made comments throughout this book about this, and related topics, with flavor ranging from vinegar to ${\rm HNO}_3$. This being said, quantum groups could have not existed without the work of Drinfeld \cite{dri} and Jimbo \cite{jim}, and then Kazhdan, Lusztig and many others, and then all sorts of wonderful books on the subject, including those of Chari-Pressley \cite{cpr} and Majid \cite{maj}, and so the credit for the twisting operation, that we will explain here with our own sauce, goes to all these people.

\bigskip

For the story, I enormously benefited myself from all this, as a young PhD student, back in the mid 90s. I was back then part of the operator algebra team in Jussieu, of Connes and Skandalis, doing quantum groups of Woronowicz type, and of course, working hard on that. But, in parallel to this, I was secretly reading Drinfeld's paper \cite{dri}, which is one of the deepest and most beautiful papers ever, and also the book of Chari-Pressley \cite{cpr}, in order to get some details about what Drinfeld was saying. Good old times.

\bigskip

In order to get started now, the best is to deform first the simplest objects that we have, namely the noncommutative spheres. This can be done as follows:

\index{twisted sphere}

\begin{theorem}
We have noncommutative spheres as follows, obtained via the twisted commutation relations $ab=\pm ba$, and twisted half-commutation relations $abc=\pm cba$,
$$\xymatrix@R=15mm@C=14mm{
\bar{S}^{N-1}_\mathbb C\ar[r]&\bar{S}^{N-1}_{\mathbb C,*}\ar[r]&S^{N-1}_{\mathbb C,+}\\
\bar{S}^{N-1}_\mathbb R\ar[r]\ar[u]&\bar{S}^{N-1}_{\mathbb R,*}\ar[r]\ar[u]&S^{N-1}_{\mathbb R,+}\ar[u]
}$$
where the signs at left correspond to the anticommutation of distinct coordinates, and their adjoints, and the other signs come from functoriality.
\end{theorem}

\begin{proof}
For the spheres on the left, if we want to replace some of the commutation relations $z_iz_j=z_jz_i$ by anticommutation relations $z_iz_j=-z_jz_i$, a bit of thinking tells us that the one and only natural choice is as follows:
$$z_iz_j=-z_jz_i\quad,\quad\forall i\neq j$$

In other words, with the notation $\varepsilon_{ij}=1-\delta_{ij}$, we must have:
$$z_iz_j=(-1)^{\varepsilon_{ij}}z_jz_i$$

Regarding now the spheres in the middle, the situation is a priori a bit more tricky, because we have to take into account the various possible collapsings of $\{i,j,k\}$. However, if we want to have embeddings as above, there is only one choice, namely:
$$z_iz_jz_k=(-1)^{\varepsilon_{ij}+\varepsilon_{jk}+\varepsilon_{ik}}z_kz_jz_i$$

Thus, we have constructed our spheres, and embeddings, as needed.
\end{proof}

As already mentioned, the above is something quite ad-hoc, but we will be back later to this, with some more conceptual twisting methods as well. To be more precise, the alternative idea will be that of twisting the quantum groups first, by using advanced easiness theory, and then deducing from this the twisting formulae for the spheres.

\bigskip

Let us discuss now the quantum group case. The situation here is considerably more complicated, because the coordinates $u_{ij}$ depend on double indices, and finding for instance the correct signs for $u_{ij}u_{kl}u_{mn}=\pm u_{mn}u_{kl}u_{ij}$ looks nearly impossible. 

\bigskip

However, we can solve this problem by taking some inspiration from the sphere case, which was already solved. We first have the following result:

\index{twisted orthogonal group}
\index{twisted unitary group}

\begin{proposition}
We have quantum groups as follows,
$$\xymatrix@R=15mm@C=15mm{
\bar{U}_N\ar[r]&U_N^+\\
\bar{O}_N\ar[r]\ar[u]&O_N^+\ar[u]}$$
defined via the following relations,
$$\alpha\beta=\begin{cases}
-\beta\alpha&{\rm for}\ a,b\in\{u_{ij}\}\ {\rm distinct,\ on\ the\ same\ row\ or\ column}\\
\beta\alpha&{\rm otherwise}
\end{cases}$$
with the convention $\alpha=a,a^*$ and $\beta=b,b^*$.
\end{proposition}

\begin{proof}
These quantum groups are well-known, see \cite{bbc}. The idea indeed is that the existence of $\varepsilon,S$ is clear. Regarding now $\Delta$, set $U_{ij}=\sum_ku_{ik}\otimes u_{kj}$. For $j\neq k$ we have:
\begin{eqnarray*}
U_{ij}U_{ik}
&=&\sum_{s\neq t}u_{is}u_{it}\otimes u_{sj}u_{tk}+\sum_su_{is}u_{is}\otimes u_{sj}u_{sk}\\
&=&\sum_{s\neq t}-u_{it}u_{is}\otimes u_{tk}u_{sj}+\sum_su_{is}u_{is}\otimes(-u_{sk}u_{sj})\\
&=&-U_{ik}U_{ij}
\end{eqnarray*}

Also, for $i\neq k,j\neq l$ we have:
\begin{eqnarray*}
U_{ij}U_{kl}
&=&\sum_{s\neq t}u_{is}u_{kt}\otimes u_{sj}u_{tl}+\sum_su_{is}u_{ks}\otimes u_{sj}u_{sl}\\
&=&\sum_{s\neq t}u_{kt}u_{is}\otimes u_{tl}u_{sj}+\sum_s(-u_{ks}u_{is})\otimes(-u_{sl}u_{sj})\\
&=&U_{kl}U_{ij}
\end{eqnarray*}

This finishes the proof in the real case. In the complex case the remaining relations can be checked in a similar way, by putting $*$ exponents in the middle.
\end{proof}

It remains now to twist $O_N^*,U_N^*$. In order to do so, given three coordinates $a,b,c\in\{u_{ij}\}$, let us set $span(a,b,c)=(r,c)$, where $r,c\in\{1,2,3\}$ are the number of rows and columns spanned by $a,b,c$. In other words, if we write $a=u_{ij},b=u_{kl},c=u_{pq}$ then $r=\#\{i,k,p\}$ and $l=\#\{j,l,q\}$. With these conventions, we have:

\begin{proposition}
We have intermediate quantum groups as follows,
$$\xymatrix@R=15mm@C=15mm{
\bar{U}_N\ar[r]&\bar{U}_N^*\ar[r]&U_N^+\\
\bar{O}_N\ar[r]\ar[u]&\bar{O}_N^*\ar[r]\ar[u]&O_N^+\ar[u]}$$
defined via the following relations,
$$\alpha\beta\gamma=\begin{cases}
-\gamma\beta\alpha&{\rm for}\ a,b,c\in\{u_{ij}\}\ {\rm with}\ span(a,b,c)=(\leq 2,3)\ {\rm or}\ (3,\leq 2)\\
\gamma\beta\alpha&{\rm otherwise}
\end{cases}$$
with the conventions $\alpha=a,a^*$, $\beta=b,b^*$ and $\gamma=c,c^*$.
\end{proposition}

\begin{proof}
The rules for the various commutation/anticommutation signs are:
$$\begin{matrix}
r\backslash c&1&2&3\\
1&+&+&-\\
2&+&+&-\\
3&-&-&+
\end{matrix}$$

We first prove the result for $\bar{O}_N^*$. The construction of the counit, $\varepsilon(u_{ij})=\delta_{ij}$, requires the Kronecker symbols $\delta_{ij}$ to commute/anticommute according to the above table. Equivalently, we must prove that the situation $\delta_{ij}\delta_{kl}\delta_{pq}=1$ can appear only in a case where the above table indicates ``+''. But this is clear, because $\delta_{ij}\delta_{kl}\delta_{pq}=1$ implies $r=c$.

The construction of the antipode $S$ is clear too, because this requires the choice of our $\pm$ signs to be invariant under transposition, and this is true, the table being symmetric. 

With $U_{ij}=\sum_ku_{ik}\otimes u_{kj}$, we have the following computation:
\begin{eqnarray*}
U_{ia}U_{jb}U_{kc}
&=&\sum_{xyz}u_{ix}u_{jy}u_{kz}\otimes u_{xa}u_{yb}u_{zc}\\
&=&\sum_{xyz}\pm u_{kz}u_{jy}u_{ix}\otimes\pm u_{zc}u_{yb}u_{xa}\\
&=&\pm U_{kc}U_{jb}U_{ia}
\end{eqnarray*}

We must prove that, when examining the precise two $\pm$ signs in the middle formula, their product produces the correct $\pm$ sign at the end. The point now is that both these signs depend only on $s=span(x,y,z)$, and for $s=1,2,3$ respectively:

\medskip

-- For a $(3,1)$ span we obtain $+-$, $+-$, $-+$, so a product $-$ as needed.

-- For a $(2,1)$ span we obtain $++$, $++$, $--$, so a product $+$ as needed.

-- For a $(3,3)$ span we obtain $--$, $--$, $++$, so a product $+$ as needed.

-- For a $(3,2)$ span we obtain $+-$, $+-$, $-+$, so a product $-$ as needed.

-- For a $(2,2)$ span we obtain $++$, $++$, $--$, so a product $+$ as needed.

\medskip

Together with the fact that our problem is invariant under $(r,c)\to(c,r)$, and with the fact that for a $(1,1)$ span there is nothing to prove, this finishes the proof. For $\bar{U}_N^*$ the proof is similar, by putting $*$ exponents in the middle.
\end{proof}

The above results can be summarized as follows:

\index{twisted orthogonal group}
\index{twisted unitary group}

\begin{theorem}
We have quantum groups as follows, obtained via the twisted commutation relations $ab=\pm ba$, and twisted half-commutation relations $abc=\pm cba$,
$$\xymatrix@R=15mm@C=15mm{
\bar{U}_N\ar[r]&\bar{U}_N^*\ar[r]&U_N^+\\
\bar{O}_N\ar[r]\ar[u]&\bar{O}_N^*\ar[r]\ar[u]&O_N^+\ar[u]}$$
where the signs at left correspond to anticommutation for distinct entries on rows and columns, and commutation otherwise, and the other signs come from functoriality.
\end{theorem}

\begin{proof}
As explained above, there is only one reasonable way of arranging the signs, as for everything to work fine. So let us go ahead now, and present the solution.

Given abstract coordinates $a,b,c,\ldots\in\{u_{ij}\}$, let us set $span(a,b,c,\ldots)=(r,c)$, where $r,c\in\{1,2,3,\ldots\}$ are the numbers of rows and columns spanned by $a,b,c,\ldots$, inside the matrix $u=(u_{ij})$. Also, we make the conventions $\alpha=a,a^*$, $\beta=b,b^*$, and so on.

With these conventions, the relations for the quantum groups on the left, which are the only possible ones, as for having a good compatibility with the spheres, are:
$$\alpha\beta=\begin{cases}
-\beta\alpha&{\rm for}\ a,b\in\{u_{ij}\}\ {\rm with}\ span(a,b)=(1,2)\ {\rm or}\ (2,1)\\
\beta\alpha&{\rm otherwise}
\end{cases}$$

As for the relations for the quantum groups in the middle, once again these are uniquely determined by various functoriality considerations, and must be as follows:
$$\alpha\beta\gamma=\begin{cases}
-\gamma\beta\alpha&{\rm for}\ a,b,c\in\{u_{ij}\}\ {\rm with}\ span(a,b,c)=(\leq 2,3)\ {\rm or}\ (3,\leq 2)\\
\gamma\beta\alpha&{\rm otherwise}
\end{cases}$$

Summarizing, we are done with the difficult part, namely guessing the signs. What is left is to prove that the above relations produce indeed quantum groups, with inclusions between them, as in the statement. But this follows from the computations from the proof of Proposition 7.21 and Proposition 7.22 above.
\end{proof}

\section*{7d. Schur-Weyl twisting}

Our purpose now will be that of showing that the quantum groups constructed above can be defined in a more conceptual way, as ``Schur-Weyl twists". Let $P_{even}(k,l)\subset P(k,l)$ be the set of  partitions with blocks having even size, and $NC_{even}(k,l)\subset P_{even}(k,l)$ be the subset of noncrossing partitions.  Also, we use the standard embedding $S_k\subset P_2(k,k)$, via the pairings having only up-to-down strings. Given a partition $\tau\in P(k,l)$, we call ``switch'' the operation which consists in switching two neighbors, belonging to different blocks, in the upper row, or in the lower row. With these conventions, we have:

\index{signature map}
\index{partitions with even blocks}

\begin{proposition}
There is a signature map $\varepsilon:P_{even}\to\{-1,1\}$, given by 
$$\varepsilon(\tau)=(-1)^c$$
where $c$ is the number of switches needed to make $\tau$ noncrossing. In addition:
\begin{enumerate}
\item For $\tau\in S_k$, this is the usual signature.

\item For $\tau\in P_2$ we have $(-1)^c$, where $c$ is the number of crossings.

\item For $\tau\leq\pi\in NC_{even}$, the signature is $1$.
\end{enumerate}
\end{proposition}

\begin{proof}
We must prove that the number $c$ in the statement is well-defined modulo 2. And it is enough to perform the verification for the noncrossing partitions.

\medskip

In order to do so, observe that any partition $\tau\in P(k,l)$ can be put in ``standard form'', by ordering its blocks according to the appearence of the first leg in each block, counting clockwise from top left, and then by performing the switches as for block 1 to be at left, then for block 2 to be at left, and so on. Here the required switches are also uniquely determined, by the order coming from counting clockwise from top left. 

\medskip

Here is an example of such an algorithmic switching operation:
$$
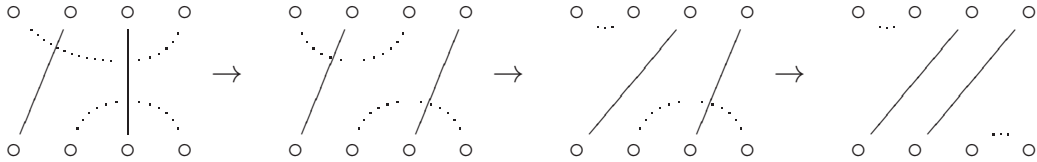
\xymatrix@R=3mm@C=3mm{\circ\ar@/_/@{.}[drr]&\circ\ar@{-}[dddl]&\circ\ar@{-}[ddd]&\circ\\
&&\ar@/_/@{.}[ur]&\\
&&\ar@/^/@{.}[dr]&\\
\circ&\circ\ar@/^/@{.}[ur]&\circ&\circ}
\xymatrix@R=5mm@C=1mm{&\\\to\\&\\& }
\xymatrix@R=3mm@C=3mm{\circ\ar@/_/@{.}[dr]&\circ\ar@{-}[dddl]&\circ&\circ\ar@{-}[dddl]\\
&\ar@/_/@{.}[ur]&&\\
&&\ar@/^/@{.}[dr]&\\
\circ&\circ\ar@/^/@{.}[ur]&\circ&\circ}
\xymatrix@R=5mm@C=1mm{&\\\to\\&\\&}
\xymatrix@R=3mm@C=3mm{\circ\ar@/_/@{.}[r]&\circ&\circ\ar@{-}[dddll]&\circ\ar@{-}[dddl]\\
&&&\\
&&\ar@/^/@{.}[dr]&\\
\circ&\circ\ar@/^/@{.}[ur]&\circ&\circ}
\xymatrix@R=5mm@C=1mm{&\\\to\\&\\& }
\xymatrix@R=3mm@C=3mm{\circ\ar@/_/@{.}[r]&\circ&\circ\ar@{-}[dddll]&\circ\ar@{-}[dddll]\\
&&&\\
&&&\\
\circ&\circ&\circ\ar@/^/@{.}[r]&\circ}$$

\vskip-3mm

The point now is that, under the assumption $\tau\in NC_{even}(k,l)$, each of the moves required for putting a leg at left, and hence for putting a whole block at left, requires an even number of switches. Thus, putting $\tau$ is standard form requires an even number of switches. Now given $\tau,\tau'\in NC_{even}$ having the same block structure, the standard form coincides, so the number of switches $c$ required for the passage $\tau\to\tau'$ is indeed even.

\medskip

Regarding now the remaining assertions, these are all elementary:

\medskip

(1) For $\tau\in S_k$ the standard form is $\tau'=id$, and the passage $\tau\to id$ comes by composing with a number of transpositions, which gives the signature. 

\medskip

(2) For a general $\tau\in P_2$, the standard form is of type $\tau'=|\ldots|^{\cup\ldots\cup}_{\cap\ldots\cap}$, and the passage $\tau\to\tau'$ requires $c$ mod 2 switches, where $c$ is the number of crossings. 

\medskip

(3) Assuming that $\tau\in P_{even}$ comes from $\pi\in NC_{even}$ by merging a certain number of blocks, we can prove that the signature is 1 by proceeding by recurrence.
\end{proof}

We can use the above signature map, as follows:

\index{twisted Kronecker symbol}

\begin{definition}
Associated to a partition $\pi\in P_{even}(k,l)$ is the linear map
$$\bar{T}_\pi:(\mathbb C^N)^{\otimes k}\to(\mathbb C^N)^{\otimes l}$$
given by the following formula, with $e_1,\ldots,e_N$ being the standard basis of $\mathbb C^N$,
$$\bar{T}_\pi(e_{i_1}\otimes\ldots\otimes e_{i_k})=\sum_{j_1\ldots j_l}\bar{\delta}_\pi\begin{pmatrix}i_1&\ldots&i_k\\ j_1&\ldots&j_l\end{pmatrix}e_{j_1}\otimes\ldots\otimes e_{j_l}$$
and where $\bar{\delta}_\pi\in\{-1,0,1\}$ is $\bar{\delta}_\pi=\varepsilon(\tau)$ if $\tau\geq\pi$, and $\bar{\delta}_\pi=0$ otherwise, with $\tau=\ker\binom{i}{j}$.
\end{definition}

In other words, what we are doing here is to add signatures to the usual formula of $T_\pi$. Indeed, observe that the usual formula for $T_\pi$ can be written as folllows:
$$T_\pi(e_{i_1}\otimes\ldots\otimes e_{i_k})=\sum_{j:\ker(^i_j)\geq\pi}e_{j_1}\otimes\ldots\otimes e_{j_l}$$

Now by inserting signs, coming from the signature map $\varepsilon:P_{even}\to\{\pm1\}$, we are led to the following formula, which coincides with the one given above:
$$\bar{T}_\pi(e_{i_1}\otimes\ldots\otimes e_{i_k})=\sum_{\tau\geq\pi}\varepsilon(\tau)\sum_{j:\ker(^i_j)=\tau}e_{j_1}\otimes\ldots\otimes e_{j_l}$$

We must first prove a key categorical result, as follows:

\begin{proposition}
The assignement $\pi\to\bar{T}_\pi$ is categorical, in the sense that
$$\bar{T}_\pi\otimes\bar{T}_\sigma=\bar{T}_{[\pi\sigma]}\quad,\quad
\bar{T}_\pi \bar{T}_\sigma=N^{c(\pi,\sigma)}\bar{T}_{[^\sigma_\pi]}\quad,\quad
\bar{T}_\pi^*=\bar{T}_{\pi^*}$$
where $c(\pi,\sigma)$ are certain positive integers.
\end{proposition}

\begin{proof}
We have to go back to the proof from the untwisted case, from chapter 4 above, and insert signs. We have to check three conditions, as follows:

\medskip

\underline{1. Concatenation}. In the untwisted case, this was based on the following formula:
$$\delta_\pi\begin{pmatrix}i_1\ldots i_p\\ j_1\ldots j_q\end{pmatrix}
\delta_\sigma\begin{pmatrix}k_1\ldots k_r\\ l_1\ldots l_s\end{pmatrix}
=\delta_{[\pi\sigma]}\begin{pmatrix}i_1\ldots i_p&k_1\ldots k_r\\ j_1\ldots j_q&l_1\ldots l_s\end{pmatrix}$$

In the twisted case, it is enough to check the following formula:
$$\varepsilon\left(\ker\begin{pmatrix}i_1\ldots i_p\\ j_1\ldots j_q\end{pmatrix}\right)
\varepsilon\left(\ker\begin{pmatrix}k_1\ldots k_r\\ l_1\ldots l_s\end{pmatrix}\right)=
\varepsilon\left(\ker\begin{pmatrix}i_1\ldots i_p&k_1\ldots k_r\\ j_1\ldots j_q&l_1\ldots l_s\end{pmatrix}\right)$$

Let us denote by $\tau,\nu$ the two partitions on the left, so that the partition on the right is of the form $\rho\leq[\tau\nu]$. Now by switching to the noncrossing form, $\tau\to\tau'$ and $\nu\to\nu'$, the partition on the right transforms into:
$$\rho\to\rho'\leq[\tau'\nu']$$

Since $[\tau'\nu']$ is noncrossing, by Proposition 7.24 (3) we obtain the result.

\medskip

\underline{2. Composition}. In the untwisted case, this was based on the following formula:
$$\sum_{j_1\ldots j_q}\delta_\pi\begin{pmatrix}i_1\ldots i_p\\ j_1\ldots j_q\end{pmatrix}
\delta_\sigma\begin{pmatrix}j_1\ldots j_q\\ k_1\ldots k_r\end{pmatrix}
=N^{c(\pi,\sigma)}\delta_{[^\pi_\sigma]}\begin{pmatrix}i_1\ldots i_p\\ k_1\ldots k_r\end{pmatrix}$$

In order to prove now the result in the twisted case, it is enough to check that the signs match. More precisely, we must establish the following formula:
$$\varepsilon\left(\ker\begin{pmatrix}i_1\ldots i_p\\ j_1\ldots j_q\end{pmatrix}\right)
\varepsilon\left(\ker\begin{pmatrix}j_1\ldots j_q\\ k_1\ldots k_r\end{pmatrix}\right)
=\varepsilon\left(\ker\begin{pmatrix}i_1\ldots i_p\\ k_1\ldots k_r\end{pmatrix}\right)$$

Let $\tau,\nu$ be the partitions on the left, so that the partition on the right is of the form $\rho\leq[^\tau_\nu]$. Our claim is that we can jointly switch $\tau,\nu$ to the noncrossing form. Indeed, we can first switch as for $\ker(j_1\ldots j_q)$ to become noncrossing, and then switch the upper legs of $\tau$, and the lower legs of $\nu$, as for both these partitions to become noncrossing. Now observe that when switching in this way to the noncrossing form, $\tau\to\tau'$ and $\nu\to\nu'$, the partition on the right transforms into:
$$\rho\to\rho'\leq[^{\tau'}_{\nu'}]$$

Since $[^{\tau'}_{\nu'}]$ is noncrossing, by Proposition 7.24 (3) we obtain the result.

\medskip

\underline{3. Involution}. Here we must prove the following formula:
$$\bar{\delta}_\pi\begin{pmatrix}i_1\ldots i_p\\ j_1\ldots j_q\end{pmatrix}=\bar{\delta}_{\pi^*}\begin{pmatrix}j_1\ldots j_q\\ i_1\ldots i_p\end{pmatrix}$$

But this is clear from the definition of $\bar{\delta}_\pi$, and we are done.
\end{proof}

As a conclusion, our twisted construction $\pi\to\bar{T}_\pi$ has all the needed properties for producing quantum groups, via Tannakian duality. Thus, we can formulate:

\begin{theorem}
Given a category of partitions $D\subset P_{even}$, the construction
$$Hom(u^{\otimes k},u^{\otimes l})=span\left(\bar{T}_\pi\Big|\pi\in D(k,l)\right)$$
produces via Tannakian duality a quantum group $\bar{G}_N\subset U_N^+$, for any $N\in\mathbb N$.
\end{theorem}

\begin{proof}
This follows indeed from the Tannakian results from chapter 4 above, exactly as in the easy case, by using this time Proposition 7.26 as technical ingredient. To be more precise, Proposition 7.26 shows that the linear spaces on the right form a Tannakian category, and so the results in chapter 4 apply, and give the result.
\end{proof}

We can unify the easy quantum groups, or at least the examples coming from categories $D\subset P_{even}$, with the quantum groups constructed above, as follows:

\index{quizzy quantum group}
\index{Schur-Weyl twisting}

\begin{definition}
A closed subgroup $G\subset U_N^+$ is called $q$-easy, or quizzy, with deformation parameter $q=\pm1$, when its tensor category appears as follows,
$$Hom(u^{\otimes k},u^{\otimes l})=span\left(\dot{T}_\pi\Big|\pi\in D(k,l)\right)$$
for a certain category of partitions $D\subset P_{even}$, where, for $q=-1,1$:
$$\dot{T}=\bar{T},T$$
The Schur-Weyl twist of $G$ is the quizzy quantum group $\bar{G}\subset U_N^+$ obtained via $q\to-q$.
\end{definition}

We will see later on that the easy quantum group associated to $P_{even}$ itself is the hyperochahedral group $H_N$, and so that our assumption $D\subset P_{even}$, replacing $D\subset P$, simply corresponds to $H_N\subset G$, replacing the usual condition $S_N\subset G$. In relation now with the basic quantum groups, we first have the following result:

\index{basic crossing}
\index{half-classical crossing}

\begin{proposition}
The linear map associated to the basic crossing is:
$$\bar{T}_{\slash\!\!\!\backslash}(e_i\otimes e_j)
=\begin{cases}
-e_j\otimes e_i&{\rm for}\ i\neq j\\
e_j\otimes e_i&{\rm otherwise}
\end{cases}$$
The linear map associated to the half-liberating permutation is:
$$\bar{T}_{\slash\hskip-1.6mm\backslash\hskip-1.1mm|\hskip0.5mm}(e_i\otimes e_j\otimes e_k)
=\begin{cases}
-e_k\otimes e_j\otimes e_i&{\rm for}\ i,j,k\ {\rm distinct}\\
e_k\otimes e_j\otimes e_i&{\rm otherwise}
\end{cases}$$
Also, for any noncrossing pairing $\pi\in NC_2$, we have $\bar{T}_\pi=T_\pi$.
\end{proposition}

\begin{proof}
The first formula in the statement is clear. Regarding the second formula, this follows from the following signature computations, obtained by counting the crossings, in the first case, by switching twice as to put the partition in noncrossing form, in the next 3 cases, and by observing that the partition is noncrossing, in the last case:
$$\xymatrix@R=10mm@C=5mm{\circ\ar@/_/@{-}[drr]&\circ\ar@/^/@{.}[d]&\circ\ar@/_/@{~}[dll]\\
\circ&\circ&\circ}
\xymatrix@R=4mm@C=1mm{&\\\to -1\\&\\& }\qquad\quad
\xymatrix@R=10mm@C=5mm{\circ\ar@/_/@{-}[drr]&\circ\ar@/^/@{.}[d]&\circ\ar@/_/@{-}[dll]\\
\circ&\circ&\circ}
\xymatrix@R=4mm@C=1mm{&\\\to 1\\&\\& }$$
$$\xymatrix@R=10mm@C=5mm{\circ\ar@/_/@{-}[drr]&\circ\ar@/^/@{-}[d]&\circ\ar@/_/@{~}[dll]\\
\circ&\circ&\circ}
\xymatrix@R=4mm@C=1mm{&\\\to 1\\&\\& }\qquad\quad
\xymatrix@R=10mm@C=5mm{\circ\ar@/_/@{-}[drr]&\circ\ar@/^/@{.}[d]&\circ\ar@/_/@{.}[dll]\\
\circ&\circ&\circ}
\xymatrix@R=4mm@C=1mm{&\\\to 1\\&\\& }\qquad\quad
\xymatrix@R=10mm@C=5mm{\circ\ar@/_/@{-}[drr]&\circ\ar@/^/@{-}[d]&\circ\ar@/_/@{-}[dll]\\
\circ&\circ&\circ}
\xymatrix@R=4mm@C=1mm{&\\\to 1\\&\\& }$$

Finally, the last assertion follows from Proposition 7.24 (3).
\end{proof}

The relation with the basic quantum groups comes from:

\begin{proposition}
For an orthogonal quantum group $G$, the following hold:
\begin{enumerate}
\item $\dot{T}_{\slash\!\!\!\backslash}\in End(u^{\otimes 2})$ precisely when $G\subset\dot{O}_N$.

\item $\dot{T}_{\slash\hskip-1.6mm\backslash\hskip-1.1mm|\hskip0.5mm}\in End(u^{\otimes 3})$ precisely when $G\subset\dot{O}_N^*$.
\end{enumerate}
\end{proposition}

\begin{proof}
We already know this in the untwisted case, coming from our various easiness considerations. In the twisted case, the proof is as follows:

\medskip

(1) By using the formula of $\bar{T}_{\slash\!\!\!\backslash}$ in Proposition 7.29, we obtain:
\begin{eqnarray*}
(\bar{T}_{\slash\!\!\!\backslash}\otimes1)u^{\otimes 2}(e_i\otimes e_j\otimes1)
&=&\sum_ke_k\otimes e_k\otimes u_{ki}u_{kj}\\
&-&\sum_{k\neq l}e_l\otimes e_k\otimes u_{ki}u_{lj}
\end{eqnarray*}

On the other hand, we have as well the following formula:
\begin{eqnarray*}
u^{\otimes 2}(\bar{T}_{\slash\!\!\!\backslash}\otimes1)(e_i\otimes e_j\otimes1)
&=&\begin{cases}
\sum_{kl}e_l\otimes e_k\otimes u_{li}u_{ki}&{\rm if}\ i=j\\
-\sum_{kl}e_l\otimes e_k\otimes u_{lj}u_{ki}&{\rm if}\ i\neq j
\end{cases}
\end{eqnarray*}

For $i=j$ the conditions are $u_{ki}^2=u_{ki}^2$ for any $k$, and $u_{ki}u_{li}=-u_{li}u_{ki}$ for any $k\neq l$. For $i\neq j$ the conditions are $u_{ki}u_{kj}=-u_{kj}u_{ki}$ for any $k$, and $u_{ki}u_{lj}=u_{lj}u_{ki}$ for any $k\neq l$. Thus we have exactly the relations between the coordinates of $\bar{O}_N$, and we are done.

\medskip

(2) By using the formula of $\bar{T}_{\slash\hskip-1.6mm\backslash\hskip-1.1mm|\hskip0.5mm}$ in Proposition 7.29, we obtain:
\begin{eqnarray*}
(\bar{T}_{\slash\hskip-1.6mm\backslash\hskip-1.1mm|\hskip0.5mm}\otimes1)u^{\otimes 2}(e_i\otimes e_j\otimes e_k\otimes1)
&=&\sum_{abc\ not\ distinct}e_c\otimes e_b\otimes e_a\otimes u_{ai}u_{bj}u_{ck}\\
&-&\sum_{a,b,c\ distinct}e_c\otimes e_b\otimes e_a\otimes u_{ai}u_{bj}u_{ck}
\end{eqnarray*}

On the other hand, we have as well the following formula:
\begin{eqnarray*}
&&u^{\otimes 2}(\bar{T}_{\slash\hskip-1.6mm\backslash\hskip-1.1mm|\hskip0.5mm}\otimes1)(e_i\otimes e_j\otimes e_k\otimes1)\\
&&=\begin{cases}
\sum_{abc}e_c\otimes e_b\otimes e_a\otimes u_{ck}u_{bj}u_{ai}&{\rm for}\ i,j,k\ {\rm not\ distinct}\\
-\sum_{abc}e_c\otimes e_b\otimes e_a\otimes u_{ck}u_{bj}u_{ai}&{\rm for}\ i,j,k\ {\rm distinct}
\end{cases}
\end{eqnarray*}

For $i,j,k$ not distinct the conditions are $u_{ai}u_{bj}u_{ck}=u_{ck}u_{bj}u_{ai}$ for $a,b,c$ not distinct, and $u_{ai}u_{bj}u_{ck}=-u_{ck}u_{bj}u_{ai}$ for $a,b,c$ distinct. For $i,j,k$ distinct the conditions are $u_{ai}u_{bj}u_{ck}=-u_{ck}u_{bj}u_{ai}$ for $a,b,c$ not distinct, and $u_{ai}u_{bj}u_{ck}=u_{ck}u_{bj}u_{ai}$ for $a,b,c$ distinct. Thus we have the relations between the coordinates of $\bar{O}_N^*$, and we are done.
\end{proof}

We can now formulate our first Schur-Weyl twisting result, as follows:

\index{Schur-Weyl twisting}

\begin{theorem}
The twisted quantum groups introduced before,
$$\xymatrix@R=15mm@C=15mm{
\bar{U}_N\ar[r]&\bar{U}_N^*\ar[r]&U_N^+\\
\bar{O}_N\ar[r]\ar[u]&\bar{O}_N^*\ar[r]\ar[u]&O_N^+\ar[u]}$$
appear as twists of the basic quantum groups, namely
$$\xymatrix@R=15mm@C=15mm{
U_N\ar[r]&U_N^*\ar[r]&U_N^+\\
\bar{O}_N\ar[r]\ar[u]&\bar{O}_N^*\ar[r]\ar[u]&O_N^+\ar[u]}$$
via the Schur-Weyl twisting procedure described above.
\end{theorem}

\begin{proof}
This follows indeed from Proposition 7.30 above.
\end{proof}

Summarizing, we have now a conceptual approach to the twisting of the basic unitary quantum groups. In order for our theory to be complete, let us discuss as well the computation of the quantum isometry groups of the twisted spheres. We have here:

\index{quantum isometry group}

\begin{theorem}
The quantum isometry groups of the twisted spheres,
$$\xymatrix@R=15mm@C=14mm{
\bar{S}^{N-1}_\mathbb C\ar[r]&\bar{S}^{N-1}_{\mathbb C,*}\ar[r]&S^{N-1}_{\mathbb C,+}\\
\bar{S}^{N-1}_\mathbb R\ar[r]\ar[u]&\bar{S}^{N-1}_{\mathbb R,*}\ar[r]\ar[u]&S^{N-1}_{\mathbb R,+}\ar[u]
}$$
are the above twisted orthogonal and unitary groups.
\end{theorem}

\begin{proof}
The proof in the classical twisted case is similar to the proof in the classical untwisted case, by adding signs. Indeed, for the twisted real sphere $\bar{S}^{N-1}_\mathbb R$ we have:
$$\Phi(z_iz_j)
=\sum_kz_k^2\otimes u_{ki}u_{kj}
+\sum_{k<l}z_kz_l\otimes(u_{ki}u_{lj}-u_{li}u_{kj})$$

We deduce that with $[[a,b]]=ab+ba$ we have the following formula:
$$\Phi([[z_i,z_j]])
=\sum_kz_k^2\otimes[[u_{ki},u_{kj}]]
+\sum_{k<l}z_kz_l\otimes([u_{ki},u_{lj}]-[u_{li},u_{kj}])$$

Now assuming $i\neq j$, we have $[[z_i,z_j]]=0$, and we therefore obtain:
$$[[u_{ki},u_{kj}]]=0\quad,\quad\forall k$$
$$[u_{ki},u_{lj}]=[u_{li},u_{kj}]\quad,\quad\forall k<l$$

By using now the standard trick of Bhowmick-Goswami \cite{bhg}, namely applying the antipode and then relabelling, the latter relation gives:
$$[u_{ki},u_{lj}]=0$$

Thus, we obtain the result. The proof for $\bar{S}^{N-1}_\mathbb C$ is similar, by using the above-mentioned categorical trick, in order to deduce from the relations $ab=\pm ba$ the remaining relations $ab^*=\pm b^*a$. Finally, the proof in the half-classical twisted cases is similar to the proof in the half-classical untwisted cases, by adding signs where needed.
\end{proof}

All the above was of course quite technical, and for further details, and generalizations, we refer to the literature, namely the follow-ups of \cite{bbc}, \cite{bgo}. These follow-ups are actually mostly dealing with noncommutative geometry, with the quantum group twisting results being scattered, here and there, stated as technical results, where needed. Good places for looking for references are the reference lists of the papers \cite{ba4}, \cite{bb4}.

\bigskip

In what concerns us, we will be back to twisting on several occasions, and notably in chapter 10 below, when looking at the complex reflection groups and their liberations. We will see there that the easy complex reflection groups equal their own Schur-Weyl twists, and with this suggesting that the twisting operation, at least in the easy case, is something of rather continuous nature, with the discrete objects being rigid.

\bigskip

On the other hand, we will see as well in chapter 9 that we have a key isomorphism of type $S_4^+=SO_3^{-1}$, with this time $SO_3$ being not easy, and its twist being a cocycle one. Thus, the axiomatization of the $q=-1$ twisting operation remains a subtle question, not solved yet. In fact, this is an open question since Drinfeld \cite{dri} and Jimbo \cite{jim}.

\section*{7e. Exercises} 

Here is a first instructive exercise, of rather algebraic and abstract nature, in relation with the notion of easiness that we developed in the above:

\begin{exercise}
Prove that any closed subgroup $G\subset U_N^+$ has an ``easy envelope''
$$G\subset\widetilde{G}\subset U_N^+$$
which is the smallest easy quantum group containing $G$. 
\end{exercise} 

Obviously, this is somehing of Tannakian nature. The problem is that of finding the precise Tannakian formulation of the exercise, and then solving it.

\begin{exercise}
Prove that if $H,K$ are easy then we have inclusions as follows,
$$<H,K>\subset\widetilde{<H,K>}\subset\{H,K\}$$
where the middle object is an easy envelope, as constructed above.
\end{exercise} 

As before, this is something of Tannakian nature. As a comment here, this improves the results that we have so far, and refines the questions which remain to be solved. Indeed, it is not clear that either of the above inclusions must be an isomorphism.

\begin{exercise}
Prove that the symmetric group, regarded as group of permutation matrices,
$$S_N\subset O_N$$
is easy, with the corresponding category of partitions being $P$ itself.
\end{exercise}

This is something quite fundamental, that we will discuss in detail later on. However, the proof is not that difficult, and can be certainly worked out.

\begin{exercise}
Prove that the hyperoctahedral group, which is the symmetry group of the $N$-hypercube, when regarded as group of orthogonal matrices,
$$H_N\subset O_N$$
is easy, with the corresponding category of partitions being $P_{even}$.
\end{exercise}

As before with $S_N$, this is something that we will discuss in detail later on. We will discuss as well later free analogues of these results, involving $NC,NC_{even}$.

\begin{exercise}
Find a general formula connecting the linear maps
$$T_\pi\quad,\quad\bar{T}_\pi$$
involving the M\"obius function of the partitions.
\end{exercise}

To be more precise, we have already seen in the above a number of  formulae for the maps $\bar{T}_\pi$, expressed as linear combinations of maps $T_\pi$. The problem is that of understanding how the correspondence between the maps $\bar{T}_\pi$ and the maps $T_\pi$ works, and since all this is about partitions, the answer can only be a M\"obius type formula.

\begin{exercise}
Work out the missing details in the proof of Theorem 7.32, by taking the untwisted computations, and adding signs where needed.
\end{exercise}

Obviously, this is something self-explanatory, which can only work, with a bit of patience, and care for the details. If lost, you can of course look it up.

\begin{exercise}
Work out what happens at $N=2$, in connection with all the easy quantum groups introduced so far, and with their twists as well.
\end{exercise}

To be more precise, we have met many examples of easy quantum groups $G_N$, and the problem is that of understanding, for each of these quantum groups, if $G_2$ is something well-known. Generally speaking, the answer here is yes, but all this is worth to be worked out in detail. After this, the question regarding the twists makes sense as well.

\chapter{Probabilistic aspects}

\section*{8a. Free probability}

We have seen in chapters 5-6 that the easiness property of $O_N,U_N$ and $O_N^+,U_N^+$ leads to a number of interesting probabilistic consequences, notably in what concerns the computation of the law of the main character $\chi$, in the $N\to\infty$ limit. Our purpose here will be two-fold. On one hand, we would like to have similar results for the various quantum groups introduced in chapter 7. And on the other hand, we would like to upgrade our results about characters $\chi$ into something more advanced.

\bigskip

In order to do this, we are in need of more probability knowledge. We have certainly met some classical and free probability in chapters 5-6, but the computations and results there were a bit ad-hoc, adapted to what we wanted to prove about $O_N,U_N$ and $O_N^+,U_N^+$. And this kind of ad-hoc point of view will not do it, for what we want to do here.

\bigskip

In short, time for a crash course on probability. We will be following the book of Voiculescu, Dykema, Nica \cite{vdn}, and we will be quite brief, because at the core of classical and free probability are the Gaussian laws, real and complex, and the semicircular and circular laws, that we already know about, from chapters 5-6. Thus, our job will be basically that of putting what we know in a more conceptual framework.

\bigskip

Let us first talk about classical probability. The starting point here is:

\index{random variable}
\index{moment}
\index{law}
\index{distribution}

\begin{definition}
Let $X$ be a probability space.
\begin{enumerate}
\item The real functions $f\in L^\infty(X)$ are called random variables.

\item The moments of such a variable are the numbers $M_k(f)=\mathbb E(f^k)$.

\item The law of such a variable is the measure given by $M_k(f)=\int_\mathbb Rx^kd\mu_f(x)$.
\end{enumerate}
\end{definition}

Here the fact that $\mu_f$ exists indeed is not trivial. By linearity, we would like to have a real probability measure making hold the following formula, for any $P\in\mathbb R[X]$:
$$\mathbb E(P(f))=\int_\mathbb RP(x)d\mu_f(x)$$

By using a continuity argument, it is enough to have this for the characteristic functions $\chi_I$ of the measurable sets $I\subset\mathbb R$. Thus, we would like to have $\mu_f$ such that:
$$\mathbb P(f\in I)=\mu_f(I)$$

But this latter formula can serve as a definition for $\mu_f$, so we are done. Next in line, we need to talk about independence. Once again with the idea of doing things a bit abstractly, and most adapted to what we want to do here, the definition is as follows:

\index{independence}

\begin{definition}
Two variables $f,g\in L^\infty(X)$ are called independent when
$$\mathbb E(f^kg^l)=\mathbb E(f^k)\cdot\mathbb E(g^l)$$
happens, for any $k,l\in\mathbb N$.
\end{definition}

Again, this definition hides some non-trivial things. Indeed, by linearity, we would like to have a formula as follows, valid for any polynomials $P,Q\in\mathbb R[X]$:
$$\mathbb E(P(f)Q(g))=\mathbb E(P(f))\cdot\mathbb E(Q(g))$$

By continuity, it is enough to have this for characteristic functions of type $\chi_I,\chi_J$, with $I,J\subset\mathbb R$. Thus, we are led to the usual definition of independence, namely:
$$\mathbb P(f\in I,g\in J)=\mathbb P(f\in I)\cdot\mathbb P(g\in J)$$

Here is now our first result, providing tools for the study of the independence:

\index{independence}
\index{Fourier transform}

\begin{theorem}
Assume that $f,g\in L^\infty(X)$ are independent.
\begin{enumerate}
\item We have $\mu_{f+g}=\mu_f*\mu_g$, where $*$ is the convolution of measures.

\item We have $F_{f+g}=F_fF_g$, where $F_f(x)=\mathbb E(e^{ixf})$ is the Fourier transform.
\end{enumerate}
\end{theorem}

\begin{proof}
This is something very standard, the idea being as follows:

\medskip

(1) We have the following computation, using the independence of $f,g$:
\begin{eqnarray*}
M_k(f+g)
&=&\mathbb E((f+g)^k)\\
&=&\sum_l\binom{k}{l}\mathbb E(f^lg^{k-l})\\
&=&\sum_l\binom{k}{l}M_l(f)M_{k-l}(g)
\end{eqnarray*}

On the other hand, by using the Fubini theorem, we have as well:
\begin{eqnarray*}
\int_\mathbb Rx^kd(\mu_f*\mu_g)(x)
&=&\int_{\mathbb R\times\mathbb R}(x+y)^kd\mu_f(x)d\mu_g(y)\\
&=&\sum_l\binom{k}{l}\int_\mathbb Rx^kd\mu_f(x)\int_\mathbb Ry^ld\mu_g(y)\\
&=&\sum_l\binom{k}{l}M_l(f)M_{k-l}(g)
\end{eqnarray*}

Thus the measures $\mu_{f+g}$ and $\mu_f*\mu_g$ have the same moments, and so coincide.

\medskip

(2) We have indeed the following computation, using (1) and Fubini:
\begin{eqnarray*}
F_{f+g}(x)
&=&\int_\mathbb Re^{ixy}d\mu_{f+g}(y)\\
&=&\int_{\mathbb R\times\mathbb R}e^{ix(y+z)}d\mu_f(y)d\mu_g(z)\\
&=&\int_\mathbb Re^{ixy}d\mu_f(y)\int_\mathbb Re^{ixz}d\mu_g(z)\\
&=&F_f(x)F_g(x)
\end{eqnarray*}

Thus, we are led to the conclusion in the statement.
\end{proof}

Let us discuss now the normal distributions. We have here:

\index{normal law}
\index{Gaussian law}

\begin{definition}
The normal law of parameter $t>0$ is the following measure:
$$g_t=\frac{1}{\sqrt{2\pi t}}\,e^{-x^2/2t}dx$$
This is also called Gaussian distribution, with ``g'' standing for Gauss.
\end{definition}

As a first remark, the above law has indeed mass 1, as it should. This follows indeed from the Gauss formula, which gives, with $x=y/\sqrt{2t}$:
$$\int_\mathbb R e^{-y^2/2t}dy=\sqrt{2\pi t}$$

Generally speaking, the normal laws appear as bit everywhere, in real life. The reasons behind this phenomenon come from the Central Limit Theorem (CLT), that we will explain in a moment, after developing the needed general theory. We first have:

\index{Fourier transform}
\index{convolution semigroup}

\begin{proposition}
We have the following formula, for any $t>0$:
$$F_{g_t}(x)=e^{-tx^2/2}$$
In particular, the normal laws satisfy $g_s*g_t=g_{s+t}$, for any $s,t>0$.
\end{proposition}

\begin{proof}
The Fourier transform formula can be established as follows:
\begin{eqnarray*}
F_{g_t}(x)
&=&\frac{1}{\sqrt{2\pi t}}\int_\mathbb Re^{-y^2/2t+ixy}dy\\
&=&\frac{1}{\sqrt{2\pi t}}\int_\mathbb Re^{-(y/\sqrt{2t}-\sqrt{t/2}ix)^2-tx^2/2}dy\\
&=&\frac{1}{\sqrt{\pi}}\int_\mathbb Re^{-z^2-tx^2/2}dz
\end{eqnarray*}

As for the last assertion, this follows from the linearization result from Theorem 8.3 (2) above, because $\log F_{g_t}$ is linear in $t$.
\end{proof}

We are now ready to state and prove the CLT, as follows:

\index{CLT}
\index{Central Limit Theorem}
\index{normal law}
\index{Gaussian law}

\begin{theorem}[CLT]
Given random variables $f_1,f_2,f_3,\ldots\in L^\infty(X)$ which are i.i.d., centered, and with variance $t>0$, we have, with $n\to\infty$, in moments,
$$\frac{1}{\sqrt{n}}\sum_{i=1}^nf_i\sim g_t$$
where $g_t$ is the Gaussian law of parameter $t$.
\end{theorem}

\begin{proof}
We have the following formula for $F_f(x)=\mathbb E(e^{ixf})$, in terms of moments:
$$F_f(x)=\sum_{k=0}^\infty\frac{i^kM_k(f)}{k!}\,x^k$$

Thus, the Fourier transform of the variable in the statement is:
\begin{eqnarray*}
F(x)
&=&\left[F_f\left(\frac{x}{\sqrt{n}}\right)\right]^n\\
&=&\left[1-\frac{tx^2}{2n}+O(n^{-2})\right]^n\\
&\simeq&e^{-tx^2/2}
\end{eqnarray*}

But this latter function being the Fourier transform of $g_t$, we obtain the result.
\end{proof}

Let us record as well the complex version of the CLT. This is as follows:

\index{CCLT}
\index{Complex CLT}
\index{Complex Central Limit Theorem}
\index{complex normal law}
\index{complex Gaussian law}

\begin{theorem}[Complex CLT]
Given variables $f_1,f_2,f_3,\ldots\in L^\infty(X)$ whose real and imaginary parts are i.i.d., centered, and with variance $t>0$, we have, with $n\to\infty$,
$$\frac{1}{\sqrt{n}}\sum_{i=1}^nf_i\sim G_t$$
where $G_t$ is the complex Gaussian law of parameter $t$, appearing as the law of $\frac{1}{\sqrt{2}}(a+ib)$, where $a,b$ are real and independent, each following the law $g_t$.
\end{theorem}

\begin{proof}
This is clear from Theorem 8.6, by taking real and imaginary parts.
\end{proof}

In the noncommutative setting now, the starting definition is as follows:

\index{random variable}
\index{moment}
\index{law}
\index{distribution}

\begin{definition}
Let $A$ be a $C^*$-algebra, given with a trace $tr$.
\begin{enumerate}
\item The elements $a\in A$ are called random variables.

\item The moments of such a variable are the numbers $M_k(a)=tr(a^k)$.

\item The law of such a variable is the functional $\mu:P\to tr(P(a))$.
\end{enumerate}
\end{definition}

Here $k=\circ\bullet\bullet\circ\ldots$ is as usual a colored integer, and the powers $a^k$ are defined by multiplicativity and the usual formulae, namely:
$$a^\emptyset=1\quad,\quad
a^\circ=a\quad,\quad 
a^\bullet=a^*$$

As for the polynomial $P$, this is a noncommuting $*$-polynomial in one variable: 
$$P\in\mathbb C<X,X^*>$$

Observe that the law is uniquely determined by the moments, because:
$$P(X)=\sum_k\lambda_kX^k\implies\mu(P)=\sum_k\lambda_kM_k(a)$$

Generally speaking, the above definition is something quite abstract, but there is no other way of doing things, at least at this level of generality. We have indeed:

\begin{theorem}
Given a $C^*$-algebra with a faithful trace $(A,tr)$, any normal variable, 
$$aa^*=a^*a$$
has a usual law, namely a complex probability measure $\mu\in\mathcal P(\mathbb C)$ satisfying:
$$tr(a^k)=\int_\mathbb Cz^kd\mu(z)$$
This law is unique, and is supported by the spectrum $\sigma(a)\subset\mathbb C$. In the non-normal case, $aa^*\neq a^*a$, such a usual law does not exist.
\end{theorem}

\begin{proof}
This is something that we know from chapter 6, coming from the Gelfand theorem, which gives $<a>=C(\sigma(a))$, and the Riesz theorem. As for the last assertion, we know this too from chapter 6, coming via $tr(aa^*aa^*)>tr(aaa^*a^*)$ for $aa^*\neq a^*a$.
\end{proof}

Let us discuss now the independence, and its noncommutative versions, in the above setting. As a starting point here, we have the following notion:

\index{independence}

\begin{definition}
Two subalgebras $B,C\subset A$ are called independent when the following condition is satisfied, for any $b\in B$ and $c\in C$: 
$$tr(bc)=tr(b)tr(c)$$
Equivalently, the following condition must be satisfied, for any $b\in B$ and $c\in C$: 
$$tr(b)=tr(c)=0\implies tr(bc)=0$$
Also, two variables $b,c\in A$ are called independent when the algebras that they generate, 
$$B=<b>\quad,\quad 
C=<c>$$
are independent inside $A$, in the above sense.
\end{definition}

Observe that the above two conditions are indeed equivalent. In one sense this is clear, and in the other sense, with $a'=a-tr(a)$, this follows from:
\begin{eqnarray*}
tr(bc)
&=&tr[(b'+tr(b))(c'+tr(c))]\\
&=&tr(b'c')+t(b')tr(c)+tr(b)tr(c')+tr(b)tr(c)\\
&=&tr(b'c')+tr(b)tr(c)\\
&=&tr(b)tr(c)
\end{eqnarray*}

The other remark is that the above notion generalizes indeed the usual notion of independence, from the classical case, the result here being as follows:

\begin{theorem}
Given two compact measured spaces $Y,Z$, the algebras
$$C(Y)\subset C(Y\times Z)\quad,\quad 
C(Z)\subset C(Y\times Z)$$
are independent in the above sense, and a converse of this fact holds too.
\end{theorem}

\begin{proof}
We have two assertions here, the idea being as follows:

\medskip

(1) First of all, given two arbitrary compact spaces $Y,Z$, we have embeddings of algebras as in the statement, defined by the following formulae:
$$f\to[(y,z)\to f(y)]\quad,\quad 
g\to[(y,z)\to g(z)]$$

In the measured space case now, the Fubini theorems tells us that:
$$\int_{Y\times Z}f(y)g(z)=\int_Yf(y)\int_Zg(z)$$

Thus, the algebras $C(Y),C(Z)$ are independent in the sense of Definition 8.3.

\medskip

(2) Conversely now, assume that $B,C\subset A$ are independent, with $A$ being commutative. Let us write our algebras as follows, with $X,Y,Z$ being certain compact spaces:
$$A=C(X)\quad,\quad 
B=C(Y)\quad,\quad
C=C(Z)$$ 

In this picture, the inclusions $B,C\subset A$ must come from quotient maps, as follows:
$$p:Z\to X\quad,\quad 
q:Z\to Y$$

Regarding now the independence condition from Definition 8.3, in the above picture, this tells us that the folowing equality must happen:
$$\int_Xf(p(x))g(q(x))=\int_Xf(p(x))\int_Xg(q(x))$$

Thus we are in a Fubini type situation, and we obtain from this $Y\times Z\subset X$. Thus, the independence of $B,C\subset A$ appears as in (1) above.
\end{proof}

It is possible to develop some theory here, but this is ultimately not very interesting. As a much more interesting notion now, we have Voiculescu's freeness \cite{vdn}:

\index{freeness}

\begin{definition}
Two subalgebras $B,C\subset A$ are called free when the following condition is satisfied, for any $b_i\in B$ and $c_i\in C$:
$$tr(b_i)=tr(c_i)=0\implies tr(b_1c_1b_2c_2\ldots)=0$$
Also, two variables $b,c\in A$ are called free when the algebras that they generate,
$$B=<b>\quad,\quad 
C=<c>$$
are free inside $A$, in the above sense.
\end{definition}

As a first observation, of theoretical nature, there is actually a certain lack of symmetry between Definition 8.10 and Definition 8.12, because in contrast to the former, the latter does not include an explicit formula for the quantities of the following type:
$$tr(b_1c_1b_2c_2\ldots)$$

However, this is not an issue, and is simply due to the fact that the formula in the free case is something more complicated, the result being as follows:

\begin{proposition}
Assuming that $B,C\subset A$ are free, the restriction of $tr$ to $<B,C>$ can be computed in terms of the restrictions of $tr$ to $B,C$. To be more precise,
$$tr(b_1c_1b_2c_2\ldots)=P\Big(\{tr(b_{i_1}b_{i_2}\ldots)\}_i,\{tr(c_{j_1}c_{j_2}\ldots)\}_j\Big)$$
where $P$ is certain polynomial in several variables, depending on the length of the word $b_1c_1b_2c_2\ldots$, and having as variables the traces of products of type
$$b_{i_1}b_{i_2}\ldots\quad,\quad
c_{j_1}c_{j_2}\ldots$$
with the indices being chosen increasing, $i_1<i_2<\ldots$ and $j_1<j_2<\ldots$
\end{proposition}

\begin{proof}
This is something that we know from chapter 6, which is based on a computation which is similar to that made after Definition 8.10.
\end{proof}

Let us discuss now some models for independence and freeness. We first have:

\index{tensor product}
\index{free product}

\begin{theorem}
Given two algebras $(B,tr)$ and $(C,tr)$, the following hold:
\begin{enumerate}
\item $B,C$ are independent inside their tensor product $B\otimes C$, endowed with its canonical tensor product trace, given on basic tensors by $tr(b\otimes c)=tr(b)tr(c)$.

\item $B,C$ are free inside their free product $B*C$, endowed with its canonical free product trace, given by the formulae in Proposition 8.13.
\end{enumerate}
\end{theorem}

\begin{proof}
Both the assertions are clear from definitions, as follows:

\medskip

(1) This is clear with either of the definitions of the independence, from Definition 8.10 above, because we have by construction of the trace:
$$tr(bc)
=tr[(b\otimes1)(1\otimes c)]
=tr(b\otimes c)
=tr(b)tr(c)$$

(2) This is clear from definitions, the only point being that of showing that the notion of freeness, or the recurrence formulae in Proposition 8.13, can be used in order to construct a canonical free product trace, on the free product of the two algebras involved:
$$tr:B*C\to\mathbb C$$

But this can be checked for instance by using a GNS construction. Indeed, consider the GNS constructions for the algebras $(B,tr)$ and $(C,tr)$:
$$B\to B(l^2(B))\quad,\quad 
C\to B(l^2(C))$$

By taking the free product of these representations, we obtain a representation as follows, with the $*$ symbol on the right being a free product of pointed Hilbert spaces:
$$B*C\to B(l^2(B)*l^2(C))$$

Now by composing with the linear form $T\to<T\xi,\xi>$, where $\xi=1_B=1_C$ is the common distinguished vector of $l^2(B)$ and $l^2(C)$, we obtain a linear form, as follows:
$$tr:B*C\to\mathbb C$$

It is routine then to check that $tr$ is indeed a trace, and this is the ``canonical free product trace'' from the statement. Then, an elementary computation shows that $B,C$ are indeed free inside $B*C$, with respect to this trace, and this finishes the proof. 
\end{proof}

As a concrete application of the above results, still following \cite{vdn}, we have:

\index{free convolution}

\begin{theorem}
We have a free convolution operation $\boxplus$ for the distributions
$$\mu:\mathbb C<X,X^*>\to\mathbb C$$
which is well-defined by the following formula, with $b,c$ taken to be free:
$$\mu_b\boxplus\mu_c=\mu_{b+c}$$
This restricts to an operation, still denoted $\boxplus$, on the real probability measures.
\end{theorem}

\begin{proof}
We have several verifications to be performed here, as follows:

\medskip

(1) We first have to check that given two variables $b,c$ which live respectively in certain $C^*$-algebras $B,C$, we can recover them inside some $C^*$-algebra $A$, with exactly the same distributions $\mu_b,\mu_c$, as to be able to sum them and then talk about $\mu_{b+c}$. But this comes from Theorem 8.14, because we can set $A=B*C$, as explained there.

\medskip

(2) The other verification which is needed is that of the fact that if $b,c$ are free, then the distribution $\mu_{b+c}$ depends only on the distributions $\mu_b,\mu_c$. But for this purpose, we can use the general formula from Proposition 8.13, namely:
$$tr(b_1c_1b_2c_2\ldots)=P\Big(\{tr(b_{i_1}b_{i_2}\ldots)\}_i,\{tr(c_{j_1}c_{j_2}\ldots)\}_j\Big)$$

Here $P$ is certain polynomial, depending on the length of $b_1c_1b_2c_2\ldots$, having as variables the traces of products $b_{i_1}b_{i_2}\ldots$ and $c_{j_1}c_{j_2}\ldots$, with $i_1<i_2<\ldots$ and $j_1<j_2<\ldots$

Now by plugging in arbitrary powers of $b,c$ as variables $b_i,c_j$, we obtain a family of formulae of the following type, with $Q$ being certain polynomials:
$$tr(b^{k_1}c^{l_1}b^{k_2}c^{l_2}\ldots)=P\Big(\{tr(b^k)\}_k,\{tr(c^l)\}_l\Big)$$

Thus the moments of $b+c$ depend only on the moments of $b,c$, with of course colored exponents in all this, according to our moment conventions, and this gives the result.

\medskip

(3) Finally, in what regards the last assertion, regarding the real measures, this is clear from the fact that if $b,c$ are self-adjoint, then so is their sum $b+c$.
\end{proof}

We would like now to have a linearization result for $\boxplus$, in the spirit of the previous result for $*$. We will do this slowly, in several steps. As a first observation, both independence and freeness are nicely modelled inside group algebras, as follows:

\index{group algebra}
\index{independence}
\index{freeness}

\begin{theorem}
We have the following results, valid for group algebras:
\begin{enumerate}
\item $C^*(\Gamma),C^*(\Lambda)$ are independent inside $C^*(\Gamma\times\Lambda)$.

\item $C^*(\Gamma),C^*(\Lambda)$ are free inside $C^*(\Gamma*\Lambda)$.
\end{enumerate}
\end{theorem}

\begin{proof}
In order to prove these results, we have two possible methods:

\medskip

(1) We can use here the general results in Theorem 8.14 above, along with the following two isomorphisms, which are both standard:
$$C^*(\Gamma\times\Lambda)=C^*(\Lambda)\otimes C^*(\Gamma)$$
$$C^*(\Gamma*\Lambda)=C^*(\Lambda)*C^*(\Gamma)$$

\medskip

(2) We can prove this directly as well, by using the fact that each group algebra is spanned by the corresponding group elements. Indeed, it is enough to check the independence and freeness formulae on group elements, which is in turn trivial.
\end{proof}

Regarding now the linearization problem for $\boxplus$, the situation here is quite tricky. We need good models for the pairs of free random variables $(b,c)$, and the problem is that the models that we have will basically lead us into the combinatorics from Proposition 8.13 and its proof, that cannot be solved with bare hands, and that we want to avoid.

\bigskip

The idea will be that of temporarily lifting the self-adjointness assumption on our variables $b,c$, and looking instead for arbitrary random variables $\beta,\gamma$, not necessarily self-adjoint, modelling in integer moments our given laws $\mu,\nu\in\mathcal P(\mathbb R)$, as follows:
$$tr(\beta^k)=M_k(\mu),\ \forall k\in\mathbb N$$
$$tr(\gamma^k)=M_k(\nu),\ \forall k\in\mathbb N$$

To be more precise, assuming that $\beta,\gamma$ are indeed not self-adjoint, the above formulae are not the general formulae for $\beta,\gamma$, simply because these latter formulae involve colored integers $k=\circ\bullet\bullet\circ\ldots$ as exponents. Thus, in the context of the above formulae, $\mu,\nu$ are not the distributions of $\beta,\gamma$, but just some ``pieces'' of these distributions.

\bigskip

Now with this tricky idea in mind, due to Voiculescu \cite{vdn}, the solution to our law modelling problem comes in a quite straightforward way, involving the good old Hilbert space $H=l^2(\mathbb N)$ and the good old shift operator $S\in B(H)$, as follows:

\index{shift}

\begin{theorem}
Consider the shift operator on the space $H=l^2(\mathbb N)$, given by:
$$S(e_i)=e_{i+1}$$
The variables of the following type, with $f\in\mathbb C[X]$ being a polynomial, 
$$S^*+f(S)$$
model then in moments, up to finite order, all the distributions $\mu:\mathbb C[X]\to\mathbb C$.
\end{theorem}

\begin{proof}
The adjoint of the shift is given by the following formula:
$$S^*(e_i)=\begin{cases}
e_{i-1}&(i>0)\\
0&(i=0)
\end{cases}$$

Consider now a variable as in the statement, namely:
$$T=S^*+a_0+a_1S+a_2S^2+\ldots+a_nS^n$$

We have then $tr(T)=a_0$, then $tr(T^2)$ will involve $a_1$, then $tr(T^3)$ will involve $a_2$, and so on. Thus, we are led to a certain recurrence, that we will not attempt to solve now, with bare hands, but which definitely gives the conclusion in the statement.
\end{proof}

Before getting further, let us point out the following fundamental fact:

\index{shift}
\index{semicircle law}
\index{Wigner semicircle}

\begin{proposition}
In the context of the above correspondence, the variable
$$T=S+S^*$$
follows the Wigner semicircle law on $[-2,2]$.
\end{proposition}

\begin{proof}
This is something that we know from chapter 6, the idea being that the combinatorics of $(S+S^*)^k$ leads us into paths on $\mathbb N$, and to the Catalan numbers.
\end{proof}

Getting back now to our linearization program for $\boxplus$, the next step is that of taking a free product of the model found in Theorem 8.17 with itself. We have here:

\index{Fock space}
\index{free Fock space}
\index{shift}
\index{creation operator}

\begin{proposition}
We can define the algebra of creation operators
$$S_x:v\to x\otimes v$$
on the free Fock space associated to a real Hilbert space $H$, given by 
$$F(H)=\mathbb C\Omega\oplus H\oplus H^{\otimes2}\oplus\ldots$$
and at the level of examples, we have:
\begin{enumerate}
\item With $H=\mathbb C$ we recover the shift algebra $A=<S>$ on $H=l^2(\mathbb N)$.

\item With $H=\mathbb C^2$, we obtain the algebra $A=<S_1,S_2>$ on $H=l^2(\mathbb N*\mathbb N)$.
\end{enumerate}
\end{proposition}

\begin{proof}
We can talk indeed about the algebra $A(H)$ of creation operators on the free Fock space $F(H)$ associated to a real Hilbert space $H$, with the remark that, in terms of the abstract semigroup notions from chapter 6 above, we have:
$$A(\mathbb C^k)=C^*(\mathbb N^{*k})\quad,\quad 
F(\mathbb C^k)=l^2(\mathbb N^{*k})$$

Thus, we are led to the conclusions in the statement.
\end{proof}

With the above notions in hand, we have the following key freeness result:

\index{vacuum vector}

\begin{proposition}
Given a real Hilbert space $H$, and two orthogonal vectors $x\perp y$, the corresponding creation operators $S_x$ and $S_y$ are free with respect to
$$tr(T)=<T\Omega,\Omega>$$
called trace associated to the vacuum vector.
\end{proposition}

\begin{proof}
This is something that we know from chapter 6, coming from the formula $S_x^*S_y=<x,y>id$, valid for any two vectors $x,y\in H$, which itself is elementary.
\end{proof}

With this technology in hand, let us go back to our linearization program for $\boxplus$. We have the following key result, further building on Proposition 8.20:

\begin{theorem}
Given two polynomials $f,g\in\mathbb C[X]$, consider the variables 
$$R^*+f(R)\quad,\quad 
S^*+g(S)$$
where $R,S$ are two creation operators, or shifts, associated to a pair of  orthogonal norm $1$ vectors. These variables are then free, and their sum has the same law as
$$T^*+(f+g)(T)$$
with $T$ being the usual shift on $l^2(\mathbb N)$.
\end{theorem}

\begin{proof}
Again, this is something that we know from chapter 6, the idea being that this comes from Proposition 8.20, by using a $45^\circ$ rotation trick.
\end{proof}

We can now solve the linearization problem. Following Voiculescu \cite{vdn}, we have:

\index{R-transform}
\index{Voiculescu R-transform}
\index{Cauchy transform}

\begin{theorem}
Given a real probability measure $\mu$, define its $R$-transform as follows:
$$G_\mu(\xi)=\int_\mathbb R\frac{d\mu(t)}{\xi-t}\implies G_\mu\left(R_\mu(\xi)+\frac{1}{\xi}\right)=\xi$$
The free convolution operation is then linearized by this $R$-transform.
\end{theorem}

\begin{proof}
This can be done by using the above results, in several steps, as follows:

\medskip

(1) According to Theorem 8.21, the operation $\mu\to f$ from Theorem 8.10 linearizes the free convolution operation $\boxplus$. We are therefore left with a computation inside $C^*(\mathbb N)$. To be more precise, consider a variable as in Theorem 8.21:
$$X=S^*+f(X)$$

In order to establish the result, we must prove that the $R$-transform of $X$, constructed according to the procedure in the statement, is the function $f$ itself.

\medskip

(2) In order to do so, fix $|z|<1$ in the complex plane, and let us set:
$$w_z=\delta_0+\sum_{k=1}^\infty z_k\delta_k$$

The shift and its adjoint act then as follows, on this vector:
$$Sw_z=z^{-1}(w_z-\delta_0)\quad,\quad
S^*w_z=zw_z$$

It follows that the adjoint of our operator $X$ acts as follows on this vector:
\begin{eqnarray*}
X^*w_z
&=&(S+f(S^*))w_z\\
&=&z^{-1}(w_z-\delta_0)+f(z)w_z\\
&=&(z^{-1}+f(z))w_z-z^{-1}\delta_0
\end{eqnarray*}

Now observe that this formula can be written as follows:
$$z^{-1}\delta_0=(z^{-1}+f(z)-X^*)w_z$$

The point now is that when $|z|$ is small, the operator appearing on the right is invertible. Thus, we can rewrite this formula as follows:
$$(z^{-1}+f(z)-X^*)^{-1}\delta_0=zw_z$$

Now by applying the trace, we are led to the following formula:
\begin{eqnarray*}
tr\left[(z^{-1}+f(z)-X^*)^{-1}\right]
&=&\left<(z^{-1}+f(z)-X^*)^{-1}\delta_0,\delta_0\right>\\
&=&<zw_z,\delta_0>\\
&=&z
\end{eqnarray*}

(3) Let us apply now the complex function procedure in the statement to the real probability measure $\mu$ modelled by $X$. The Cauchy transform $G_\mu$ is given by:
\begin{eqnarray*}
G_\mu(\xi)
&=&tr((\xi-X)^{-1})\\
&=&\overline{tr\big((\bar{\xi}-X^*)^{-1}\big)}\\
&=&tr((\xi-X^*)^{-1})
\end{eqnarray*}

Now observe that, with the choice $\xi=z^{-1}+f(z)$ for our complex variable, the trace formula found in (2) above tells us precisely that we have:
$$G_\mu\big(z^{-1}+f(z)\big)=z$$

Thus, we have $R_\mu(z)=f(z)$, which finishes the proof, as explained in step (1).
\end{proof}

With the above linearization technology in hand, we can now establish the following free analogue of the CLT, also due to Voiculescu \cite{vdn}: 

\index{FCLT}
\index{Free CLT}
\index{Free Central Limit Theorem}
\index{R-transform}
\index{Voiculescu R-transform}
\index{Wigner law}
\index{semicircle law}

\begin{theorem}[Free CLT]
Given self-adjoint variables $x_1,x_2,x_3,\ldots$ which are f.i.d., centered, with variance $t>0$, we have, with $n\to\infty$, in moments,
$$\frac{1}{\sqrt{n}}\sum_{i=1}^nx_i\sim\gamma_t$$\
where $\gamma_t$ is the Wigner semicircle law of parameter $t$, having density:
$$\gamma_t=\frac{1}{2\pi t}\sqrt{4t^2-x^2}dx$$
\end{theorem}

\begin{proof}
We follow the same idea as in the proof of the CLT:

\medskip 

(1) At $t=1$, the $R$-transform of the variable in the statement on the left can be computed by using the linearization property from Theorem 8.22, and is given by:
$$R(\xi)
=nR_x\left(\frac{\xi}{\sqrt{n}}\right)
\simeq\xi$$

(2) Regarding now the right term, also at $t=1$, our claim is that the $R$-transform of the Wigner semicircle law $\gamma_1$ is given by the following formula:
$$R_{\gamma_1}(\xi)=\xi$$

But this follows via some calculus, or directly from the following formula, coming from Proposition 8.18, and from the technical details of the $R$-transform:
$$S+S^*\sim\gamma_1$$

Thus, the laws in the statement have the same $R$-transforms, and so they are equal.

\medskip
 
(4) Summarizing, we have proved the free CLT at $t=1$. The passage to the general case, $t>0$, is routine, by some standard dilation computations.
\end{proof}

Similarly, in the complex case, we have the following result, also from \cite{vdn}:

\index{FCCLT}
\index{Free complex CLT}
\index{Voiculescu law}
\index{circular law}
\index{Voiculescu circular law}

\begin{theorem}[Free complex CLT]
Given variables $x_1,x_2,x_3,\ldots,$ whose real and imaginary parts are f.i.d., centered, and with variance $t>0$, we have, with $n\to\infty$,
$$\frac{1}{\sqrt{n}}\sum_{i=1}^nx_i\sim\Gamma_t$$
where $\Gamma_t$ is the Voiculescu circular law of parameter $t$, appearing as the law of $\frac{1}{\sqrt{2}}(a+ib)$, where $a,b$ are self-adjoint and free, each following the law $\gamma_t$.
\end{theorem}

\begin{proof}
This is clear from Theorem 8.23, by taking real and imaginary parts.
\end{proof}

There are of course many other things that can be said about $g_t,\gamma_t,G_t,\Gamma_t$, but for the moment, this is all we need. We will be back later to these laws, with more details. 

\section*{8b. Laws of characters}

Now back to our quantum group questions, let us start with the following general result, which provides us with motivations for the study of the main character:

\index{main character}
\index{Kesten measure}
\index{Cayley graph}
\index{amenability}

\begin{theorem}
Given a Woronowicz algebra $(A,u)$, the law of the main character
$$\chi=\sum_{i=1}^Nu_{ii}$$
with respect to the Haar integration has the following properties:
\begin{enumerate}
\item The moments of $\chi$ are the numbers $M_k=\dim(Fix(u^{\otimes k}))$.

\item $M_k$ counts as well the lenght $p$ loops at $1$, on the Cayley graph of $A$.

\item $law(\chi)$ is the Kesten measure of the associated discrete quantum group.

\item When $u\sim\bar{u}$ the law of $\chi$ is a usual measure, supported on $[-N,N]$.

\item The algebra $A$ is amenable precisely when $N\in supp(law(Re(\chi)))$.

\item Any morphism $f:(A,u)\to (B,v)$ must increase the numbers $M_k$.

\item Such a morphism $f$ is an isomorphism when $law(\chi_u)=law(\chi_v)$.
\end{enumerate}
\end{theorem}

\begin{proof}
These are things that we already know, the idea being as follows:

\medskip

(1) This comes from the Peter-Weyl theory, which tells us the number of fixed points of $v=u^{\otimes k}$ can be recovered by integrating the character $\chi_v=\chi_u^k$.

\medskip

(2) This is something true, and well-known, for $A=C^*(\Gamma)$, with $\Gamma=<g_1,\ldots,g_N>$ being a discrete group. In general, the proof is quite similar.

\medskip

(3) This is actually the definition of the Kesten measure, in the case $A=C^*(\Gamma)$, with $\Gamma=<g_1,\ldots,g_N>$ being a discrete group. In general, this follows from (2).

\medskip

(4) The equivalence $u\sim\bar{u}$ translates into $\chi_u=\chi_u^*$, and this gives the first assertion. As for the support claim, this follows from $uu^*=1\implies||u_{ii}||\leq1$, for any $i$.

\medskip

(5) This is the Kesten amenability criterion, which can be established as in the classical case, $A=C^*(\Gamma)$, with $\Gamma=<g_1,\ldots,g_N>$ being a discrete group.

\medskip

(6) This is something elementary, which follows from (1) above, and from the fact that the morphisms of Woronowicz algebras increase the spaces of fixed points.

\medskip

(7) This follows by using (6), and the Peter-Weyl theory, the idea being that if $f$ is not injective, then it must strictly increase one of the spaces $Fix(u^{\otimes k})$.
\end{proof}

As a conclusion, computing $\mu=law(\chi)$ is the main question to be solved, from a mathematical viewpoint. The same goes for physics too, although this is rather folklore. In what follows we will be interested in computing such laws, for the main examples of quantum groups that we have. In the easy quantum group case, we have:

\index{main character}

\begin{theorem}
For an easy quantum group $G=(G_N)$, coming from a category of partitions $D=(D(k,l))$, the asymptotic moments of the main character are given by
$$\lim_{N\to\infty}\int_{G_N}\chi^k=|D(k)|$$
where $D(k)=D(\emptyset,k)$, with the limiting sequence on the left consisting of certain integers, and being stationary at least starting from the $k$-th term.
\end{theorem}

\begin{proof}
This follows indeed from the general formula from Theorem 8.25 (1), by using the linear independence result for partitions from chapter 5.
\end{proof}

Our next purpose will be that of understanding what happens for the basic classes of easy quantum groups. In the orthogonal case, we have:

\index{orthogonal group}
\index{Gaussian law}
\index{Wigner law}

\begin{theorem}
In the $N\to\infty$ limit, the law of the main character $\chi_u$ is as follows:
\begin{enumerate}
\item For $O_N$ we obtain a Gaussian law, namely:
$$g_1=\frac{1}{\sqrt{2\pi}}e^{-x^2/2}dx$$

\item For $O_N^+$ we obtain a Wigner semicircle law, namely:
$$\gamma_1=\frac{1}{2\pi}\sqrt{4-x^2}dx$$
\end{enumerate}
\end{theorem}

\begin{proof}
These are results that we both know, from chapter 5.
\end{proof}

In the unitary case now, we have:

\index{unitary group}
\index{complex Gaussian law}
\index{circular law}

\begin{theorem}
In the $N\to\infty$ limit, the law of the main character $\chi_u$ is as follows:
\begin{enumerate}
\item For $U_N$ we obtain the complex Gaussian law $G_1$.

\item For $U_N^+$ we obtain the Voiculescu circular law $\Gamma_1$.
\end{enumerate}
\end{theorem}

\begin{proof}
These are once again results that we know, from chapter 6.
\end{proof}

Summarizing, for $O_N,O_N^+,U_N,U_N^+$ the asymptotic laws of the main characters are the laws $g_1,\gamma_1,G_1,\Gamma_1$ coming from the various CLT in classical and free probability. This is certainly nice, but there is still one conceptual problem, coming from:

\begin{proposition}
The above convergences $law(\chi_u)\to g_1,\gamma_1,G_1,\Gamma_1$ are as follows: 
\begin{enumerate}
\item They are non-stationary in the classical case.

\item They are stationary in the free case, starting from $N=2$.
\end{enumerate}
\end{proposition}

\begin{proof}
This is something quite subtle, which can be proved as follows:

\medskip

(1) Here we can use an amenability argument, based on the Kesten criterion. Indeed, $O_N,U_N$ being coamenable, the upper bound of the support of the law of $Re(\chi_u)$ is precisely $N$, and we obtain from this that the law of $\chi_u$ itself depends on $N\in\mathbb N$.

\medskip

(2) Here the result follows from the computations in chapter 4 above, performed when working out the representation theory of $O_N^+,U_N^+$, which show that the linear maps $T_\pi$ associated to the noncrossing pairings are linearly independent, at any $N\geq2$.
\end{proof}

\section*{8c. Truncated characters}

In short, we are not over with our study, which seems to open more questions than it solves. Fortunately, the solution to the question raised by Proposition 8.29 is quite simple. The idea indeed will be that of improving our $g_1,\gamma_1,G_1,\Gamma_1$ results above with $g_t,\gamma_t,G_t,\Gamma_t$ results, which will require $N\to\infty$ in both the classical and free cases, in order to hold at any $t$. In practice, the definition that we will need is as follows:

\index{truncated character}

\begin{definition}
Given a Woronowicz algebra $(A,u)$, the variable
$$\chi_t=\sum_{i=1}^{[tN]}u_{ii}$$
is called truncation of the main character, with parameter $t\in(0,1]$.
\end{definition}

Our purpose in what follows will be that of proving that for $O_N,O_N^+,U_N,U_N^+$, the asymptotic laws of the truncated characters $\chi_t$ with $t\in(0,1]$ are the laws $g_t,\gamma_t,G_t,\Gamma_t$. This is something quite technical, motivated by the findings in Proposition 8.29, and also by a number of more advanced considerations, to become clear later on.

\bigskip

In order to start now, the formula in Theorem 8.26 is not useful in the general $t\in(0,1]$ setting, and we must use instead general integration methods. We first have:

\begin{theorem}
The Haar integration of a Woronowicz algebra is given, on the coefficients of the Peter-Weyl corepresentations, by the Weingarten formula
$$\int_Gu_{i_1j_1}^{e_1}\ldots u_{i_kj_k}^{e_k}=\sum_{\pi,\sigma\in D_k}\delta_\pi(i)\delta_\sigma(j)W_k(\pi,\sigma)$$
valid for any colored integer $k=e_1\ldots e_k$ and any multi-indices $i,j$, where:
\begin{enumerate}
\item  $D_k$ is a linear basis of $Fix(u^{\otimes k})$.

\item $\delta_\pi(i)=<\pi,e_{i_1}\otimes\ldots\otimes e_{i_k}>$. 

\item $W_k=G_k^{-1}$, with $G_k(\pi,\sigma)=<\pi,\sigma>$.
\end{enumerate}
\end{theorem}

\begin{proof}
This is something that we know from chapter 3, coming from the fact that integrals in the statement form altogether the orthogonal projection onto $Fix(u^{\otimes k})$.
\end{proof}

In the easy case, this gives the following formula, from \cite{bc1}, \cite{bsp}:

\index{Weingarten formula}
\index{Haar integration}
\index{easiness}

\begin{theorem}
For an easy quantum group $G\subset U_N^+$, coming from a category of partitions $D=(D(k,l))$, we have the Weingarten integration formula
$$\int_Gu_{i_1j_1}^{e_1}\ldots u_{i_kj_k}^{e_k}=\sum_{\pi,\sigma\in D(k)}\delta_\pi(i)\delta_\sigma(j)W_{kN}(\pi,\sigma)$$
for any colored integer $k=e_1\ldots e_k$ and any multi-indices $i,j$, where $D(k)=D(\emptyset,k)$, $\delta$ are usual Kronecker symbols, and 
$$W_{kN}=G_{kN}^{-1}$$
with $G_{kN}(\pi,\sigma)=N^{|\pi\vee\sigma|}$, where $|.|$ is the number of blocks. 
\end{theorem}

\begin{proof}
With notations from Theorem 8.31, the Kronecker symbols are given by:
\begin{eqnarray*}
\delta_{\xi_\pi}(i)
&=&<\xi_\pi,e_{i_1}\otimes\ldots\otimes e_{i_k}>\\
&=&\delta_\pi(i_1,\ldots,i_k)
\end{eqnarray*}

The Gram matrix being as well the correct one, we obtain the result.
\end{proof}

We can use this for truncated characters, and following \cite{bc1}, we obtain:

\begin{proposition}
The moments of truncated characters are given by the formula
$$\int_G(u_{11}+\ldots +u_{ss})^k=Tr(W_{kN}G_{ks})$$
and with $N\to\infty$ this quantity equals $(s/N)^k|D(k)|$.
\end{proposition}

\begin{proof}
The first assertion follows from the following computation:
\begin{eqnarray*}
\int_G(u_{11}+\ldots +u_{ss})^k
&=&\sum_{i_1=1}^{s}\ldots\sum_{i_k=1}^s\int u_{i_1i_1}\ldots u_{i_ki_k}\\
&=&\sum_{\pi,\sigma\in D(k)}W_{kN}(\pi,\sigma)\sum_{i_1=1}^{s}\ldots\sum_{i_k=1}^s\delta_\pi(i)\delta_\sigma(i)\\
&=&\sum_{\pi,\sigma\in D(k)}W_{kN}(\pi,\sigma)G_{ks}(\sigma,\pi)\\
&=&Tr(W_{kN}G_{ks})
\end{eqnarray*}

The point now is that we have the following trivial estimates:
$$G_{kN}(\pi,\sigma):
\begin{cases}
=N^k&(\pi=\sigma)\\
\leq N^{k-1}&(\pi\neq\sigma)
\end{cases}$$

Thus with $N\to\infty$ we have the following estimate:
$$G_{kN}\sim N^k1$$

But this gives the following estimate, for our moment:
\begin{eqnarray*}
\int_G(u_{11}+\ldots +u_{ss})^k
&=&Tr(G_{kN}^{-1}G_{ks})\\
&\sim&Tr((N^k1)^{-1} G_{ks})\\
&=&N^{-k}Tr(G_{ks})\\
&=&N^{-k}s^k|D(k)|
\end{eqnarray*}

Thus, we have obtained the formula in the statement. See \cite{bc1}.
\end{proof}

\index{cumulant}
\index{free cumulant}

In order to process the above formula, we will need some more free probability theory. Following Nica-Speicher \cite{nsp}, given a random variable $a$, we write:
$$\log F_a(\xi)=\sum_nk_n(a)\xi^n\quad,\quad 
R_a(\xi)=\sum_n\kappa_n(a)\xi^n$$

We call the coefficients $k_n(a),\kappa_n(a)$ cumulants, respectively free cumulants of $a$. With this notion in hand, we can define then more general quantities $k_\pi(a),\kappa_\pi(a)$, depending on partitions $\pi\in P(k)$, by multiplicativity over the blocks. We have then:

\index{moment-cumulant formula}
\index{M\"obius inversion}

\begin{theorem}
We have the classical and free moment-cumulant formulae
$$M_k(a)=\sum_{\pi\in P(k)}k_\pi(a)\quad,\quad 
M_k(a)=\sum_{\pi\in NC(k)}\kappa_\pi(a)$$
where $k_\pi(a),\kappa_\pi(a)$ are the generalized cumulants and free cumulants of $a$.
\end{theorem}

\begin{proof}
These formulae, due to Rota in the classical case, and to Speicher in the free case, are something very standard, obtained by using the formulae of $F_a,R_a$, or by doing some direct combinatorics, based on the M\"obius inversion formula. See \cite{nsp}.
\end{proof}

Following \cite{bc1}, we can now improve our results about characters, as follows:

\index{Gaussian law}
\index{complex Gaussian law}
\index{circular law}

\begin{theorem}
With $N\to\infty$, the laws of truncated characters are as follows:
\begin{enumerate}
\item For $O_N$ we obtain the Gaussian law $g_t$.

\item For $O_N^+$ we obtain the Wigner semicircle law $\gamma_t$.

\item For $U_N$ we obtain the complex Gaussian law $G_t$.

\item For $U_N^+$ we obtain the Voiculescu circular law $\Gamma_t$.
\end{enumerate}
\end{theorem}

\begin{proof}
With $s=[tN]$ and $N\to\infty$, the formula in Proposition 8.33 gives:
$$\lim_{N\to\infty}\int_{G_N}\chi_t^k=\sum_{\pi\in D(k)}t^{|\pi|}$$

By using now the formulae in Theorem 8.34, this gives the results. See \cite{bc1}.
\end{proof}

In relation with the above, let us recall now that the Bercovici-Pata bijection \cite{bep} is the bijection $\{m_t\}\to\{\mu_t\}$ between the semigroups $\{m_t\}$ of infinitely divisible measures and the semigroups $\{\mu_t\}$ of freely infinitely divisible measures, given by the fact that the classical cumulants of $m_t$ equal the free cumulants of $\mu_t$. Following \cite{bsp}, we have:

\index{free convolution semigroup}
\index{Bercovici-Pata bijection}

\begin{theorem}
The asymptotic laws of truncated characters for the operations 
$$O_N\to O_N^+$$
$$U_N\to U_N^+$$
are in Bercovici-Pata bijection.
\end{theorem}

\begin{proof}
This follows indeed from the computations in the proof of Theorem 8.35.
\end{proof}

Let us discuss now the other easy quantum groups that we have. Regarding the half-liberations $O_N^*,U_N^*$, the situation here is a bit complicated, and we will discuss this later on. But we have the following result that, we can formulate here, at $t=1$:

\index{Rayleigh variable}
\index{half-liberation}

\begin{proposition}
The asymptotic laws of characters for $O_N^*,U_N^*$ are as follows:
\begin{enumerate}
\item For $O_N^*$ we obtain a symmetrized Rayleigh variable.

\item For $U_N^*$ we obtain a complexification of this variable.
\end{enumerate}
\end{proposition}

\begin{proof}
The idea is to use a projective version trick. Indeed, assuming that $G=(G_N)$ is easy, coming from a category of pairings $D$, we have:
$$\lim_{N\to\infty}\int_{PG_N}(\chi\chi^*)^k=\# D((\circ\bullet)^k)$$

In our case, where $G_N=O_N^*,U_N^*$, we can therefore use Theorem 8.35 above at $t=1$, and we are led to the conclusions in the statement. See \cite{ez1}.
\end{proof}

The above result is of course something quite modest. We will be back to the quantum groups $O_N^*,U_N^*$ in chapter 16 below, with some better techniques for dealing with them, and more specifically with explicit modelling results using $2\times2$ matrices, which virtually allow to prove anything that you want, probabilistically, about them. 

\bigskip

Next in our lineup, we have the bistochastic quantum groups. We have here:

\index{bistochastic group}
\index{bistochastic quantum group}

\begin{proposition}
For the bistochastic quantum groups, namely 
$$B_N,B_N^+,C_N,C_N^+$$
the asymptotic laws of truncated characters appear as modified versions of 
$$g_t,\gamma_t,G_t,\Gamma_t$$
and $B_N\to B_N^+$ and $C_N\to C_N^+$ are compatible with the Bercovici-Pata bijection.  
\end{proposition}

\begin{proof}
This follows indeed by using the same methods as for $O_N,O_N^+,U_N,U_N^+$, with the verification of the Bercovici-Pata bijection being elementary, and with the computation of the corresponding laws being routine as well. See \cite{bsp}, \cite{twe}.
\end{proof}

Regarding now the twists, we first have here the following general result:

\index{twisted integration}
\index{twisted Weingarten formula}

\begin{proposition}
The integration over $\bar{G}_N$ is given by the Weingarten type formula
$$\int_{\bar{G}_N}u_{i_1j_1}\ldots u_{i_kj_k}=\sum_{\pi,\sigma\in D(k)}\bar{\delta}_\pi(i)\bar{\delta}_\sigma(j)W_{kN}(\pi,\sigma)$$
where $W_{kN}$ is the Weingarten matrix of $G_N$. 
\end{proposition}

\begin{proof}
This follows exactly as in the untwisted case, the idea being that the signs will cancel. Let us recall indeed from the general twisting theory from chapter 7 that the twisted vectors $\bar{\xi}_\pi$ associated to the partitions $\pi\in P_{even}(k)$ are as follows: 
$$\bar{\xi}_\pi=\sum_{\tau\geq\pi}\varepsilon(\tau)\sum_{i:\ker(i)=\tau}e_{i_1}\otimes\ldots\otimes e_{i_k}$$

Thus, the Gram matrix of these vectors is given by:
\begin{eqnarray*}
<\xi_\pi,\xi_\sigma>
&=&\sum_{\tau\geq\pi\vee\sigma}\varepsilon(\tau)^2\left|\left\{(i_1,\ldots,i_k)\Big|\ker i=\tau\right\}\right|\\
&=&\sum_{\tau\geq\pi\vee\sigma}\left|\left\{(i_1,\ldots,i_k)\Big|\ker i=\tau\right\}\right|\\
&=&N^{|\pi\vee\sigma|}
\end{eqnarray*}

Thus the Gram matrix is the same as in the untwisted case, and so the Weingarten matrix is the same as well as in the untwisted case, and this gives the result.
\end{proof}

As a consequence of the above result, we have another general result, as follows:

\index{Schur-Weyl twisting}

\begin{theorem}
The Schur-Weyl twisting operation $G_N\leftrightarrow\bar{G}_N$ leaves invariant:
\begin{enumerate}
\item The law of the main character.

\item The coamenability property.

\item The asymptotic laws of truncated characters.
\end{enumerate}
\end{theorem}

\begin{proof}
This follows from Proposition 8.39, as follows:

\medskip

(1) This is clear indeed from the integration formula.

\medskip

(2) This follows from (1), and from the Kesten criterion.

\medskip

(3) This follows once again from the integration formula.
\end{proof}

To summarize, we have asymptotic character results for all the easy quantum groups introduced so far, and in each case we obtain Gaussian laws, and their versions. We will see in chapters 9-12 below that pretty much the same happens in the discrete setting, where we will obtain Poisson laws, and their versions.

\section*{8d. Gram determinants}

One interesting question regarding the Weingarten calculus, where $W_{kN}=G_{kN}^{-1}$, is the computation of the determinant of $G_{kN}$. Following Di Francesco \cite{dif}, we discuss here this key question, for $O_N^+$. Let us begin with something that we know, namely:

\begin{theorem}
The determinant of the Gram matrix of $P(k)$ is given by
$$\det(G_{kN})=\prod_{\pi\in P(k)}\frac{N!}{(N-|\pi|)!}$$
with the convention that in the case $N<k$ we obtain $0$.
\end{theorem}

\begin{proof}
This is indeed the Lindst\"om formula that we established in chapter 5 above, by using a decomposition of type $G_{kN}=A_{kN}L_{kN}$.
\end{proof}

In what regards $O_N^+$, the set of partitions is here $NC_2(2k)$, and things are far more complicated than for $P(k)$. However, we can use the Lindst\"om formula at $k=1,2,3$, where $P(k)=NC(k)$, via the following fattening/shrinking trick:

\index{fattening of partitions}
\index{shrinking of partitions}

\begin{proposition}
We have a bijection $NC(k)\simeq NC_2(2k)$, constructed by fattening and shrinking, as follows:
\begin{enumerate}
\item The application $NC(k)\to NC_2(2k)$ is the ``fattening'' one, obtained by doubling all the legs, and doubling all the strings as well.

\item Its inverse $NC_2(2k)\to NC(k)$ is the ``shrinking'' application, obtained by collapsing pairs of consecutive neighbors.
\end{enumerate}
\end{proposition}

\begin{proof}
The fact that the two operations in the statement are indeed inverse to each other is clear, by computing the corresponding two compositions.
\end{proof}

At the level of the associated Gram matrices, the result is as follows:

\index{Gram matrix}

\begin{proposition}
The Gram matrices of the sets of partitions
$$NC_2(2k)\simeq NC(k)$$
are related by the following formula, where $\pi\to\pi'$ is the shrinking operation,
$$G_{2k,n}(\pi,\sigma)=n^k(\Delta_{kn}^{-1}G_{k,n^2}\Delta_{kn}^{-1})(\pi',\sigma')$$
and where $\Delta_{kn}$ is the diagonal of $G_{kn}$.
\end{proposition}

\begin{proof}
It is elementary to see that we have the following formula:
$$|\pi\vee\sigma|=k+2|\pi'\vee\sigma'|-|\pi'|-|\sigma'|$$

We therefore have the following formula, valid for any $n\in\mathbb N$:
$$n^{|\pi\vee\sigma|}=n^{k+2|\pi'\vee\sigma'|-|\pi'|-|\sigma'|}$$

Thus, we are led to the formula in the statement.
\end{proof}

We can do now the computations for $O_N^+$ at $k=2,4,6$, as follows:

\begin{proposition}
The Gram matrices and determinants for $O_N^+$ are as follows,
$$\det(N)=N$$
$$\det\begin{pmatrix}N^2&N\\N&N^2\end{pmatrix}=N^2(N^2-1)$$
$$\det\begin{pmatrix}
N^3&N^2&N^2&N^2&N\\
N^2&N^3&N&N&N^2\\
N^2&N&N^3&N&N^2\\
N^2&N&N&N^3&N^2\\
N&N^2&N^2&N^2&N^3
\end{pmatrix}=N^5(N^2-1)^4(N^2-2)$$
at $k=2,4,6$, with the matrices written by using the lexicographic order on $NC_2(2k)$.
\end{proposition}

\begin{proof}
The formula at $k=2$, where $NC_2(2)=\{\sqcap\}$, is clear. The same goes for the formula at $k=4$, where $NC_2(4)=\{\sqcap\sqcap,\bigcap\hskip-4.9mm{\ }_\cap\,\}$. At $k=6$ however, things are tricky, and we must use the Lindst\"om formula. We have $NC(3)=\{|||,\sqcap|,\sqcap\hskip-3.2mm{\ }_|\,,|\sqcap,\sqcap\hskip-0.7mm\sqcap\}$, and the corresponding Gram matrix and its determinant are, according to Theorem 8.41:
$$\det\begin{pmatrix}
N^3&N^2&N^2&N^2&N\\
N^2&N^2&N&N&N\\
N^2&N&N^2&N&N\\
N^2&N&N&N^2&N\\
N&N&N&N&N
\end{pmatrix}=N^5(N-1)^4(N-2)$$

By using Proposition 8.43, the Gram determinant of $NC_2(6)$ is given by:
\begin{eqnarray*}
\det(G_{6N})
&=&\frac{1}{N^2\sqrt{N}}\times N^{10}(N^2-1)^4(N^2-2)\times\frac{1}{N^2\sqrt{N}}\\
&=&N^5(N^2-1)^4(N^2-2)
\end{eqnarray*}

Thus, we have obtained the formula in the statement.
\end{proof}

In general, such tricks won't work, because $NC(k)$ is strictly smaller than $P(k)$ at $k\geq4$. However, following Di Francesco \cite{dif}, we have the following result:

\index{meander determinant}
\index{Gram determinant}

\begin{theorem}
The determinant of the Gram matrix for $O_N^+$ is given by
$$\det(G_{kN})=\prod_{r=1}^{[k/2]}P_r(N)^{d_{k/2,r}}$$
where $P_r$ are the Chebycheff polynomials, given by
$$P_0=1\quad,\quad 
P_1=X\quad,\quad 
P_{r+1}=XP_r-P_{r-1}$$
and $d_{kr}=f_{kr}-f_{k,r+1}$, with $f_{kr}$ being the following numbers, depending on $k,r\in\mathbb Z$,
$$f_{kr}=\binom{2k}{k-r}-\binom{2k}{k-r-1}$$
with the convention $f_{kr}=0$ for $k\notin\mathbb Z$. 
\end{theorem}

\begin{proof}
This is something fairly heavy, obtained by using a decomposition as follows of the Gram matrix $G_{kN}$, with the matrix $T_{kN}$ being lower triangular:
$$G_{kN}=T_{kN}T_{kN}^t$$

Thus, a bit as in the proof of the Lindst\"om formula, we obtain the result, but the problem lies however in the construction of $T_{kN}$, which is non-trivial. See \cite{dif}.
\end{proof}

The above result is interesting for countless reasons, and we refer here to Di Francesco \cite{dif}, and also to \cite{bcu}, where a systematic study of the Gram determinants for the easy quantum groups was performed, following Lindst\"om, Di Francesco, Zinn-Justin and others. Getting into all this, which is advanced mathematical physics, would be well beyond the purposes of the present book, so let us just mention here that:

\bigskip

(1) We will see in chapter 9 that $S_N,S_N^+$ are easy, coming from $P,NC$. Thus the Lindst\"om formula computes the determinant for $S_N$, and in what regards $S_N^+$, here the determinant can be computed by using Proposition 8.43 and Theorem 8.45.

\bigskip

(2) In what regards $O_N$, here the determinant is as follows, where $f^\lambda$ is  the number of standard Young tableaux of shape $\lambda$, and $c_N(\lambda)=\prod_{(i,j)\in\lambda}(N+2j-i-1)$:
$$\det(G_{kN})=\prod_{|\lambda|=k/2}c_N(\lambda)^{f^{2\lambda}}$$

Obviously, some interesting mathematics is going on here. For more, you can check \cite{bcu}, \cite{dif}, as suggested above, as well as the related literature, in mathematics and physics, which is considerable. And also, why not trying to compute one of these beasts yourself. There is no better introduction to advanced combinatorics than this.

\section*{8e. Exercises} 

There are many interesting theoretical questions regarding the laws of the main character, and as a first exercise here, we have:

\begin{exercise}
Prove that any morphism of Woronowicz algebras
$$f:(A,u)\to(B,v)$$
increases the moments of the main character, and that such a morphism is an isomorphism precisely when all these moments, and so the character laws, are the same. 
\end{exercise}

This is something that we already discussed in the above, when introducing the main characters, but very briefly, with the comment that all this basically comes from Peter-Weyl. The problem is that of working out carefully all the details.

\begin{exercise}
Consider the symmetric group $S_N$, regarded as symmetry group of the $N$ coordinate axes of $\mathbb R^N$, and so as group of orthogonal matrices:
$$S_N\subset O_N$$
Compute the main character for this group, then the law of this main character, and work out the $N\to\infty$ asymptotics.
\end{exercise}

As a comment here, since the permutation matrices have 0-1 entries, the law of the main character is supported by $\mathbb N$. Thus, with a bit of luck, the asymptotic spectral measure can only be some basic measure in discrete probability. 

\begin{exercise}
Work out the formulae of the Gram and Weingarten matrices for all the easy quantum groups introduced so far, up to the size $5\times5$.
\end{exercise}

There are many computations here, and all of them are very instructive.

\begin{exercise}
Work out explicitely the asymptotic laws of the main characters for the half-classical quantum groups $O_N^*,U_N^*$.
\end{exercise}

This is something that we briefly discussed in the above, by indicating what the final result should be like, involving Rayleigh variables, along with a strategy for the proof. The problem is that of working out all this, with full details.

\begin{exercise}
Work out explicitely the asymptotic laws of the main characters for the bistochastic quantum groups $B_N,B_N^+,C_N,C_N^+$.
\end{exercise}

As with the previous exercise, this is something that we briefly discussed in the above, and the problem now is that of working out all the details.

\part{Quantum permutations}

\ \vskip50mm

\begin{center}
{\em Hey talk about a-rambling

She's the fastest train on the line

It's that Orange Blossom Special

Rolling down the seaboard line}
\end{center}

\chapter{Quantum permutations}

\section*{9a. Magic matrices}

The quantum groups that we considered so far, namely $O_N,U_N$ and their liberations and twists, are of ``continuous'' nature. In this third part of the present book we discuss the ``discrete'' examples, such as the symmetric group $S_N$, the hyperoctahedral group $H_N$, and more complicated reflection groups, and their liberations. We will see that most of these groups are easy, have indeed liberations, and equal their own twists.

\bigskip

In this chapter, to start with, we discuss the symmetric group $S_N$, and its free version $S_N^+$. All this is very standard, following the paper of Wang \cite{wa2} and the subsequent paper \cite{ba2}, both from the late 90s, and the paper \cite{bc2}, from the mid 00s. At the level of more advanced results, we will have as well a look into \cite{bbs}, from the late 00s, dealing more generally with quantum permutation groups $S_F^+$ of finite quantum spaces $F$, which are subject to some interesting twisting results, bringing us back to $O_N,U_N$.

\bigskip

In order to get started, we need to have a matrix group look at the symmetric group $S_N$. But this is something obvious, obtained by looking at $S_N$ as being the permutation group of the $N$ coordinate axes of $\mathbb R^N$. Indeed, we obtain in this way an embedding $S_N\subset O_N$, which is given by the standard permutation matrices, $\sigma(e_j)=e_{\sigma(j)}$.

\bigskip

Based on this, we have the following functional analytic description of $S_N$:

\index{symmetric group}
\index{magic matrix}
\index{magic unitary}

\begin{proposition}
Consider the symmetric group $S_N$.
\begin{enumerate}
\item The standard coordinates $v_{ij}\in C(S_N)$, coming from the embedding $S_N\subset O_N$ given by the permutation matrices, are given by:
$$v_{ij}=\chi\left(\sigma\Big|\sigma(j)=i\right)$$

\item The matrix $v=(v_{ij})$ is magic, in the sense that its entries are orthogonal projections, summing up to $1$ on each row and each column.

\item The algebra $C(S_N)$ is isomorphic to the universal commutative $C^*$-algebra generated by the entries of a $N\times N$ magic matrix.\end{enumerate}
\end{proposition}

\begin{proof}
These results are all elementary, as follows:

\medskip

(1) We recall that the canonical embedding $S_N\subset O_N$, coming from the standard permutation matrices, is given by $\sigma(e_j)=e_{\sigma(j)}$. Thus, we have $\sigma=\sum_je_{\sigma(j)j}$, and it follows that the standard coordinates on $S_N\subset O_N$ are given by:
$$v_{ij}(\sigma)=\delta_{i,\sigma(j)}$$

(2) Any characteristic function $\chi\in\{0,1\}$ being a projection in the operator algebra sense ($\chi^2=\chi^*=\chi$), we have indeed a matrix of projections. As for the sum 1 condition on rows and columns, this is clear from the formula of the elements $v_{ij}$.

\medskip

(3) Consider the universal algebra in the statement, namely:
$$A=C^*_{comm}\left((w_{ij})_{i,j=1,\ldots,N}\Big|w={\rm magic}\right)$$

We have a quotient map $A\to C(S_N)$, given by $w_{ij}\to v_{ij}$. On the other hand, by using the Gelfand theorem we can write $A=C(X)$, with $X$ being a compact space, and by using the coordinates $w_{ij}$ we have $X\subset O_N$, and then $X\subset S_N$. Thus we have as well a quotient map $C(S_N)\to A$ given by $v_{ij}\to w_{ij}$, and this gives (3).
\end{proof}

With the above result in hand, we can now formulate, following Wang \cite{wa2}:

\index{magic matrix}
\index{magic unitary}
\index{quantum permutation}
\index{quantum permutation group}

\begin{theorem}
The following is a Woronowicz algebra, with magic meaning formed of projections, which sum up to $1$ on each row and each column,
$$C(S_N^+)=C^*\left((u_{ij})_{i,j=1,\ldots,N}\Big|u={\rm magic}\right)$$
and the underlying compact quantum group $S_N^+$ is called quantum permutation group.
\end{theorem}

\begin{proof}
The algebra $C(S_N^+)$ is indeed well-defined, because the magic condition forces $||u_{ij}||\leq1$, for any $C^*$-norm. Our claim now is that, by using the universal property of this algebra, we can define maps $\Delta,\varepsilon,S$. Consider indeed the following matrix: 
$$U_{ij}=\sum_ku_{ik}\otimes u_{kj}$$

As a first observation, we have $U_{ij}=U_{ij}^*$. In fact the entries $U_{ij}$ are orthogonal projections, because we have as well:
\begin{eqnarray*}
U_{ij}^2
&=&\sum_{kl}u_{ik}u_{il}\otimes u_{kj}u_{lj}\\
&=&\sum_ku_{ik}\otimes u_{kj}\\
&=&U_{ij}
\end{eqnarray*}

In order to prove now that the matrix $U=(U_{ij})$ is magic, it remains to verify that the sums on the rows and columns are 1. For the rows, this can be checked as follows:
\begin{eqnarray*}
\sum_jU_{ij}
&=&\sum_{jk}u_{ik}\otimes u_{kj}\\
&=&\sum_ku_{ik}\otimes1\\
&=&1\otimes1
\end{eqnarray*}

For the columns the computation is similar, as follows:
\begin{eqnarray*}
\sum_iU_{ij}
&=&\sum_{ik}u_{ik}\otimes u_{kj}\\
&=&\sum_k1\otimes u_{kj}\\
&=&1\otimes1
\end{eqnarray*}

Thus the matrix $U=(U_{ij})$ is magic indeed, and so we can define a comultiplication map by setting $\Delta(u_{ij})=U_{ij}$. By using a similar reasoning, we can define as well a counit map by $\varepsilon(u_{ij})=\delta_{ij}$, and an antipode map by $S(u_{ij})=u_{ji}$. Thus the Woronowicz algebra axioms from chapter 2 are satisfied, and this finishes the proof.
\end{proof}

The terminology in the above result comes from the comparison with Proposition 9.1 (3), which tells us that we have an inclusion $S_N\subset S_N^+$, and that this inclusion is a liberation, in the sense that the classical version of $S_N^+$, obtained at the algebra level by dividing by the commutator ideal, is the usual symmetric group $S_N$. The terminology is further motivated by the following result, also from Wang's paper \cite{wa2}:

\index{quantum permutation}

\begin{proposition}
The quantum permutation group $S_N^+$ acts on $X=\{1,\ldots,N\}$, the corresponding coaction map $\Phi:C(X)\to C(X)\otimes C(S_N^+)$ being given by:
$$\Phi(\delta_i)=\sum_j\delta_j\otimes u_{ji}$$
In fact, $S_N^+$ is the biggest compact quantum group acting on $X$, by leaving the counting measure invariant, in the sense that 
$$(tr\otimes id)\Phi=tr(.)1$$
where $tr$ is the standard trace, given by $tr(\delta_i)=\frac{1}{N},\forall i$.
\end{proposition}

\begin{proof}
Our claim is that given a compact matrix quantum group $G$, the formula $\Phi(\delta_i)=\sum_j\delta_j\otimes u_{ji}$ defines a morphism of algebras, which is a coaction map, leaving the trace invariant, precisely when the matrix $u=(u_{ij})$ is a magic corepresentation of $C(G)$. 

Indeed, let us first determine when $\Phi$ is multiplicative. We have:
$$\Phi(\delta_i)\Phi(\delta_k)
=\sum_{jl}\delta_j\delta_l\otimes u_{ji}u_{lk}
=\sum_j\delta_j\otimes u_{ji}u_{jk}$$

On the other hand, we have as well the following formula:
$$\Phi(\delta_i\delta_k)
=\delta_{ik}\Phi(\delta_i)
=\delta_{ik}\sum_j\delta_j\otimes u_{ji}$$

Thus, the multiplicativity of $\Phi$ is equivalent to the following conditions:
$$u_{ji}u_{jk}=\delta_{ik}u_{ji}\quad,\quad\forall i,j,k$$

Regarding now the unitality of $\Phi$, we have the following formula:
\begin{eqnarray*}
\Phi(1)
&=&\sum_i\Phi(\delta_i)\\
&=&\sum_{ij}\delta_j\otimes u_{ji}\\
&=&\sum_j\delta_j\otimes\left(\sum_iu_{ji}\right)
\end{eqnarray*}

Thus $\Phi$ is unital when the following conditions are satisfied:
$$\sum_iu_{ji}=1\quad,\quad\forall i$$

Finally, the fact that $\Phi$ is a $*$-morphism translates into:
$$u_{ij}=u_{ij}^*\quad,\quad\forall i,j$$

Summing up, in order for $\Phi(\delta_i)=\sum_j\delta_j\otimes u_{ji}$ to be a morphism of $C^*$-algebras, the elements $u_{ij}$ must be projections, summing up to 1 on each row of $u$. Regarding now the preservation of the trace condition, observe that we have:
$$(tr\otimes id)\Phi(\delta_i)=\frac{1}{N}\sum_ju_{ji}$$

Thus the trace is preserved precisely when the elements $u_{ij}$ sum up to 1 on each of the columns of $u$. We conclude from this that $\Phi(\delta_i)=\sum_j\delta_j\otimes u_{ji}$ is a morphism of $C^*$-algebras preserving the trace precisely when $u$ is magic, and since the coaction conditions on $\Phi$ are equivalent to the fact that $u$ must be a corepresentation, this finishes the proof of our claim. But this claim proves all the assertions in the statement.
\end{proof}

As a perhaps quite surprising result now, also from Wang \cite{wa2}, we have:

\begin{theorem}
We have an embedding of compact quantum groups 
$$S_N\subset S_N^+$$
given at the algebra level, $C(S_N^+)\to C(S_N)$, by the formula
$$u_{ij}\to\chi\left(\sigma\Big|\sigma(j)=i\right)$$
and this embedding is an isomorphism at $N\leq3$, but not at $N\geq4$, where $S_N^+$ is non-classical, infinite compact quantum group.
\end{theorem} 

\begin{proof}
The fact that we have indeed an embedding as above is clear from Proposition 9.1 and Theorem 9.2. Note that this follows as well from Proposition 9.3. Regarding now the second assertion, we can prove this in four steps, as follows:

\medskip

\underline{Case $N=2$}. The result here is trivial, the $2\times2$ magic matrices being by definition as follows, with $p$ being a projection:
$$U=\begin{pmatrix}p&1-p\\1-p&p\end{pmatrix}$$

Indeed, this shows that the entries of a $2\times2$ magic matrix must pairwise commute, and so the algebra $C(S_2^+)$ follows to be commutative, which gives the result.

\medskip

\underline{Case $N=3$}. This is more tricky, and we present here a simple, recent proof, from \cite{lmr}. By using the same abstract argument as in the $N=2$ case, and by permuting rows and columns, it is enough to check that $u_{11},u_{22}$ commute. But this follows from:
\begin{eqnarray*}
u_{11}u_{22}
&=&u_{11}u_{22}(u_{11}+u_{12}+u_{13})\\
&=&u_{11}u_{22}u_{11}+u_{11}u_{22}u_{13}\\
&=&u_{11}u_{22}u_{11}+u_{11}(1-u_{21}-u_{23})u_{13}\\
&=&u_{11}u_{22}u_{11}
\end{eqnarray*}

Indeed, by applying the involution to this formula, we obtain from this that we have as well $u_{22}u_{11}=u_{11}u_{22}u_{11}$. Thus we get $u_{11}u_{22}=u_{22}u_{11}$, as desired.

\medskip

\underline{Case $N=4$}. In order to prove our various claims about $S_4^+$, consider the following matrix, with $p,q$ being projections, on some infinite dimensional Hilbert space:
$$U=\begin{pmatrix}
p&1-p&0&0\\
1-p&p&0&0\\
0&0&q&1-q\\
0&0&1-q&q
\end{pmatrix}$$ 

This matrix is magic, and if we choose $p,q$ as for the algebra $<p,q>$ to be not commutative, and infinite dimensional, we conclude that $C(S_4^+)$ is not commutative and infinite dimensional as well, and in particular is not isomorphic to $C(S_4)$.

\medskip

\underline{Case $N\geq5$}. Here we can use the standard embedding $S_4^+\subset S_N^+$, obtained at the level of the corresponding magic matrices in the following way:
$$u\to\begin{pmatrix}u&0\\ 0&1_{N-4}\end{pmatrix}$$

Indeed, with this embedding in hand, the fact that $S_4^+$ is a non-classical, infinite compact quantum group implies that $S_N^+$ with $N\geq5$ has these two properties as well.
\end{proof}

\section*{9b. Representations}

In order to study now $S_N^+$, we can use our various methods developed in chapters 1-4 above. Let us begin with some basic algebraic results, as follows:

\index{infinite dihedral group}

\begin{proposition}
The quantum groups $S_N^+$ have the following properties:
\begin{enumerate}
\item We have $S_N^+\,\hat{*}\,S_M^+\subset S_{N+M}^+$, for any $N,M$.

\item In particular, we have an embedding $\widehat{D_\infty}\subset S_4^+$. 

\item $S_4\subset S_4^+$ are distinguished by their spinned diagonal tori.

\item The half-classical version $S_N^*=S_N^+\cap O_N^*$ collapses to $S_N$.
\end{enumerate}
\end{proposition}

\begin{proof}
These results are all elementary, the proofs being as follows:

\medskip

(1) If we denote by $u,v$ the fundamental corepresentations of $C(S_N^+),C(S_M^+)$, the fundamental corepresentation of $C(S_N^+\,\hat{*}\,S_M^+)$ is by definition:
$$w=\begin{pmatrix}u&0\\0&v\end{pmatrix}$$

But this matrix is magic, because both $u,v$ are magic. Thus by universality of $C(S_{N+M}^+)$ we obtain a quotient map as follows, as desired:
$$C(S_{N+M}^+)\to C(S_N^+\,\hat{*}\,S_M^+)$$

(2) This result, which refines our $N=4$ trick from the proof of Theorem 9.4, follows from (1) with $N=M=2$. Indeed, we have the following computation:
\begin{eqnarray*}
S_2^+\,\hat{*}\,S_2^+
&=&S_2\,\hat{*}\, S_2\\
&=&\mathbb Z_2\,\hat{*}\, \mathbb Z_2\\
&\simeq&\widehat{\mathbb Z_2}\,\hat{*}\, \widehat{\mathbb Z_2}\\
&=&\widehat{\mathbb Z_2*\mathbb Z_2}\\
&=&\widehat{D_\infty}
\end{eqnarray*}

(3) As a first observation here, the quantum groups $S_4\subset S_4^+$ are not distinguished by their diagonal torus, which is $\{1\}$ for both of them. However, according to the general results of Woronowicz in \cite{wo1}, the group dual $\widehat{D_\infty}\subset S_4^+$ that we found in (2) must be a subgroup of the diagonal torus of the following compact quantum group, with the standard unitary representations being spinned by a certain unitary $F\in U_4$:
$$(S_4^+,FuF^*)$$

Now since this group dual $\widehat{D_\infty}$ is not classical, it cannot be a subgroup of the diagonal torus of $(S_4,FuF^*)$. Thus, the diagonal torus spinned by $F$ distinguishes $S_4\subset S_4^+$.

\medskip

(4) Consider the following compact quantum group, with the intersection operation being taken inside $U_N^+$, whose coordinates satisfy $abc=cba$:
$$S_N^*=S_N^+\cap O_N^*$$

In order to prove that we have $S_N^*=S_N$, it is enough to prove that $S_N^*$ is classical. And here, we can use the fact that for a magic matrix, the entries in each row sum up to 1. Indeed, by making $c$ vary over a full row of $u$, we obtain $abc=cba\implies ab=ba$.
\end{proof}

Summarizing, we have some advances on the quantum permutations, including a more conceptual explanation for our main observation so far, namely $S_4^+\neq S_4$. At the representation theory level now, we have the following result, from \cite{bc2}:

\index{easiness}
\index{easy quantum group}
\index{quantum permutation group}

\begin{theorem}
For the quantum groups $S_N,S_N^+$, the intertwining spaces for the tensor powers of the fundamental corepresentation $u=(u_{ij})$ are given by
$$Hom(u^{\otimes k},u^{\otimes l})=span\left(T_\pi\Big|\pi\in D(k,l)\right)$$
with $D=P,NC$. In other words, $S_N,S_N^+$ are easy, coming from the categories $P,NC$.
\end{theorem}

\begin{proof}
We use the Tannakian duality results from chapter 4 above:

\medskip

(1) $S_N^+$. According to Theorem 9.2, the algebra $C(S_N^+)$ appears as follows:
$$C(S_N^+)=C(O_N^+)\Big\slash\Big<u={\rm magic}\Big>$$

Consider the one-block partition $\mu\in P(2,1)$. The linear map associated to it is:
$$T_\mu(e_i\otimes e_j)=\delta_{ij}e_i$$

We have $T_\mu=(\delta_{ijk})_{i,jk}$, and we obtain the following formula:
$$(T_\mu u^{\otimes 2})_{i,jk}
=\sum_{lm}(T_\mu)_{i,lm}(u^{\otimes 2})_{lm,jk}
=u_{ij}u_{ik}$$

On the hand, we have as well the following formula:
$$(uT_\mu)_{i,jk}
=\sum_lu_{il}(T_\mu)_{l,jk}
=\delta_{jk}u_{ij}$$

Thus, the relation defining $S_N^+\subset O_N^+$ reformulates as follows:
$$T_\mu\in Hom(u^{\otimes 2},u)\iff u_{ij}u_{ik}=\delta_{jk}u_{ij},\forall i,j,k$$

The condition on the right being equivalent to the magic condition, we obtain:
$$C(S_N^+)=C(O_N^+)\Big\slash\Big<T_\mu\in Hom(u^{\otimes 2},u)\Big>$$

By using now the general easiness theory from chapter 7, we conclude that the quantum group $S_N^+$ is indeed easy, with the corresponding category of partitions being:
$$D=<\mu>$$

But this latter category is $NC$, as one can see by ``chopping'' the  noncrossing partitions into $\mu$-shaped components. Thus, we are led to the conclusion in the statement.

\medskip

(2) $S_N$. Here the first part of the proof is similar, leading to the following formula:
$$C(S_N)=C(O_N)\Big\slash\Big<T_\mu\in Hom(u^{\otimes 2},u)\Big>$$

But this shows that $S_N$ is easy, the corresponding category of partitions being:
$$D
=<\mu,P_2>
=<NC,P_2>
=P$$

Alternatively, this latter formula follows directly for the result for $S_N^+$ proved above, via $S_N=S_N^+\cap O_N$, and the functoriality results explained in chapter 7.
\end{proof}

In order to discuss the representations of $S_N^+$, we will need linear independence results for the vectors $\xi_\pi$ associated to the partitions $\pi\in NC$. This is something which is more technical than the previous results for pairings. Let us start with:

\index{fattening of partitions}
\index{shrinking partitions}

\begin{proposition}
We have a bijection $NC(k)\simeq NC_2(2k)$, constructed as follows:
\begin{enumerate}
\item The application $NC(k)\to NC_2(2k)$ is the ``fattening'' one, obtained by doubling all the legs, and doubling all the strings as well.

\item Its inverse $NC_2(2k)\to NC(k)$ is the ``shrinking'' application, obtained by collapsing pairs of consecutive neighbors.
\end{enumerate}
\end{proposition}

\begin{proof}
This is something elementary, coming from the fact that both compositions of the operations in the statement are the identity, that we know from chapter 8.
\end{proof}

The point now is that we have the following result, due to Jones \cite{jo1}:

\index{Temperley-Lieb algebra}

\begin{theorem}
Consider the Temperley-Lieb algebra of index $N\geq4$, defined as 
$$TL_N(k)=span(NC_2(k,k))$$
with product given by the rule $\bigcirc=N$, when concatenating. 
\begin{enumerate}
\item We have a representation $i:TL_N(k)\to B((\mathbb C^N)^{\otimes k})$, given by $\pi\to T_\pi$.

\item $Tr(T_\pi)=N^{loops(<\pi>)}$, where $\pi\to<\pi>$ is the closing operation.

\item The linear form $\tau=Tr\circ i:TL_N(k)\to\mathbb C$ is a faithful positive trace.

\item The representation $i:TL_N(k)\to B((\mathbb C^N)^{\otimes k})$ is faithful.
\end{enumerate}
In particular, the vectors $\left\{\xi_\pi|\pi\in NC(k)\right\}\subset(\mathbb C^N)^{\otimes k}$ are linearly independent.
\end{theorem}

\begin{proof}
All this is quite standard, but advanced, the idea being as follows:

\medskip

(1) This is clear from the categorical properties of $\pi\to T_\pi$.

\medskip

(2) This follows indeed from the following computation:
\begin{eqnarray*}
Tr(T_\pi)
&=&\sum_{i_1\ldots i_k}\delta_\pi\binom{i_1\ldots i_k}{i_1\ldots i_k}\\
&=&\#\left\{i_1,\ldots,i_k\in\{1,\ldots,N\}\Big|\ker\binom{i_1\ldots i_k}{i_1\ldots i_k}\geq\pi\right\}\\
&=&N^{loops(<\pi>)}
\end{eqnarray*}

(3) The traciality of $\tau$ is clear. Regarding now the faithfulness, this is something well-known, and we refer here to Jones' paper \cite{jo1}.

\medskip

(4) This follows from (3) above, via a standard positivity argument. As for the last assertion, this follows from (4), by fattening the partitions.
\end{proof}

Alternatively, the linear independence of the vectors $\xi_\pi$ with $\pi\in NC(k)$, that we will need in what follows, comes the following result of Di Francesco \cite{dif}:

\index{meander determinant}
\index{Gram determinant}

\begin{theorem}
The determinant of the Gram matrix for $S_N^+$ is given by
$$\det(G_{kN})=(\sqrt{N})^{a_k}\prod_{r=1}^kP_r(\sqrt{N})^{d_{kr}}$$
where $P_r$ are the Chebycheff polynomials, given by
$$P_0=1\quad,\quad 
P_1=X\quad,\quad 
P_{r+1}=XP_r-P_{r-1}$$
and $d_{kr}=f_{kr}-f_{k,r+1}$, with $f_{kr}$ being the following numbers, depending on $k,r\in\mathbb Z$,
$$f_{kr}=\binom{2k}{k-r}-\binom{2k}{k-r-1}$$
with the convention $f_{kr}=0$ for $k\notin\mathbb Z$, and where $a_k$ are the following numbers:
$$a_k=\sum_{\pi\in P(k)}(2|\pi|-k)$$
In particular, $\left\{\xi_\pi|\pi\in NC(k)\right\}\subset(\mathbb C^N)^{\otimes k}$ are linearly independent at $N\geq4$.
\end{theorem}

\begin{proof}
This looks very similar to the result for $O_N^+$ from chapter 8, also due to Di Francesco, and can be in fact deduced from that result, by shrinking the partitions, according to the shrinking formulae explained in chapter 8. See \cite{dif}.
\end{proof}

Long story short, we need to know that the vectors $\xi_\pi$ with $\pi\in NC(k)$ are linearly independent at $N\geq4$, and this is something non-trivial, but proofs of this fact do exist, coming from the work of Jones \cite{jo1} and Di Francesco \cite{dif}, so the result holds. There are some other proofs as well, but none is any simpler. Be said in passing, I can feel that you might say at this point that you would like a full, simple proof, we're not mathematical physicists, right. To which I would answer, welcome to mathematical physics.

\bigskip

We can work out now the representation theory of $S_N^+$, as follows:

\index{fusion rules}
\index{Clebsch-Gordan rules}
\index{quantum permutation group}
\index{main character}

\begin{theorem}
The quantum groups $S_N^+$ with $N\geq4$ have the following properties:
\begin{enumerate}
\item The moments of the main character are the Catalan numbers:
$$\int_{S_N^+}\chi^k=C_k$$

\item The fusion rules for representations are as follows, exactly as for $SO_3$:
$$r_k\otimes r_l=r_{|k-l|}+r_{|k-l|+1}+\ldots+r_{k+l}$$

\item The dimensions of the irreducible representations are given by
$$\dim(r_k)=\frac{q^{k+1}-q^{-k}}{q-1}$$
where $q,q^{-1}$ are the roots of $X^2-(N-2)X+1=0$.
\end{enumerate}
\end{theorem}

\begin{proof}
The proof, from \cite{ba2}, based on Theorems 9.8 or 9.9, goes as follows:

\medskip

(1) We have indeed the following computation, coming from the $SU_2$ computations from chapter 5, and from Theorem 9.6, Proposition 9.7, and Theorems 9.8 or 9.9:
\begin{eqnarray*}
\int_{S_N^+}\chi^k
&=&\dim(Fix(u^{\otimes k}))\\
&=&|NC(k)|\\
&=&|NC_2(2k)|\\
&=&C_k
\end{eqnarray*}

(2) This is standard, by using the formula in (1), and the known theory of $SO_3$. Let $A=span(\chi_k|k\in\mathbb N)$ be the algebra of characters of $SO_3$. We can define a morphism as follows, where $f$ is the character of the fundamental representation of $S_N^+$:
$$\Psi:A\to C(S_N^+)\quad,\quad 
\chi_1\to f-1$$

The elements $f_k=\Psi(\chi_k)$ verify then the following formulae: 
$$f_kf_l=f_{|k-l|}+f_{|k-l|+1}+\ldots+f_{k+l}$$

We prove now by recurrence that each $f_k$ is the character of an irreducible corepresentation $r_k$ of $C(S_N^+)$, non-equivalent to $r_0,\ldots,r_{k-1}$. At $k=0,1$ this is clear, so assume that the result holds at $k-1$. By integrating characters we have, exactly as for $SO_3$:
$$r_{k-2},r_{k-1}\subset r_{k-1}\otimes r_1$$

Thus there exists a certain corepresentation $r_k$ such that:
$$r_{k-1}\otimes r_1=r_{k-2}+r_{k-1}+r_k$$

Once again by integrating characters, we conclude that $r_k$ is irreducible, and non-equivalent to $r_1,\ldots,r_{k-1}$, as for $SO_3$, which proves our claim. Finally, since any irreducible representation of $S_N^+$ must appear in some tensor power of $u$, and we have a formula for decomposing each $u^{\otimes k}$ into sums of representations $r_l$, we conclude that these representations $r_l$ are all the irreducible representations of $S_N^+$.

\medskip

(3) From the Clebsch-Gordan rules we have, in particular:
$$r_kr_1=r_{k-1}+r_k+r_{k+1}$$

We are therefore led to a recurrence, and the initial data being $\dim(r_0)=1$ and $\dim(r_1)=N-1=q+1+q^{-1}$, we are led to the following formula:
$$\dim(r_k)=q^k+q^{k-1}+\ldots+q^{1-k}+q^{-k}$$

In more compact form, this gives the formula in the statement.
\end{proof}

The above result is quite surprising, and raises a massive number of questions. We would like to better understand the relation with $SO_3$, and more generally see what happens at values $N=n^2$ with $n\geq2$, and also compute the law of $\chi$, and so on.

\section*{9c. Twisted extension}

As a first topic to be discussed, one way of understanding the relation with $SO_3$ comes from noncommutative geometry considerations. We recall that, according to the general theory from chapter 1, each finite dimensional $C^*$-algebra $A$ can be written as $A=C(F)$, with $F$ being a ``finite quantum space''. To be more precise, we have:

\index{finite quantum space}

\begin{definition}
A finite quantum space $F$ is the abstract dual of a finite dimensional $C^*$-algebra $A$, according to the following formula:
$$C(F)=A$$
The number of elements of such a space is by definition the number $|F|=\dim A$. By decomposing the algebra $A$, we have a formula of the following type:
$$C(F)=M_{n_1}(\mathbb C)\oplus\ldots\oplus M_{n_k}(\mathbb C)$$
With $n_1=\ldots=n_k=1$ we obtain in this way the space $F=\{1,\ldots,k\}$. Also, when $k=1$ the equation is $C(F)=M_n(\mathbb C)$, and the solution will be denoted $F=M_n$.
\end{definition}

In order to talk about the quantum symmetry group $S_F^+$, we must use universal coactions. As in Proposition 9.3, we must endow our space $F$ with its counting measure:

\begin{definition}
We endow each finite quantum space $F$ with its counting measure, corresponding as the algebraic level to the integration functional
$$tr:C(F)\to B(l^2(F))\to\mathbb C$$
obtained by applying the regular representation, and then the normalized matrix trace.
\end{definition}

To be more precise, consider the algebra $A=C(F)$, which is by definition finite dimensional. We can make act $A$ on itself, by left multiplication:
$$\pi:A\to\mathcal L(A)\quad,\quad 
a\to(b\to ab)$$

The target of $\pi$ being a matrix algebra, $\mathcal L(A)\simeq M_N(\mathbb C)$ with $N=\dim A$, we can further compose with the normalized matrix trace, and we obtain $tr$:
$$tr=\frac{1}{N}\,Tr\circ\pi$$

Let us study the quantum group actions $G\curvearrowright F$. We denote by $\mu,\eta$ the multiplication and unit map of the algebra $C(F)$. Following \cite{ba2}, \cite{wa2}, we first have:

\begin{proposition}
Consider a linear map $\Phi:C(F)\to C(F)\otimes C(G)$, written as
$$\Phi(e_i)=\sum_je_j\otimes u_{ji}$$
with $\{e_i\}$ being a linear space basis of $C(F)$, orthonormal with respect to $tr$.
\begin{enumerate}
\item $\Phi$ is a linear space coaction $\iff$ $u$ is a corepresentation.

\item $\Phi$ is multiplicative $\iff$ $\mu\in Hom(u^{\otimes 2},u)$.

\item $\Phi$ is unital $\iff$ $\eta\in Hom(1,u)$.

\item $\Phi$ leaves invariant $tr$ $\iff$ $\eta\in Hom(1,u^*)$.

\item If these conditions hold, $\Phi$ is involutive $\iff$ $u$ is unitary.
\end{enumerate}
\end{proposition}

\begin{proof}
This is a bit similar to the proof of Proposition 9.3, as follows:

\medskip

(1) There are two axioms to be processed here. First, we have:
\begin{eqnarray*}
(id\otimes\Delta)\Phi=(\Phi\otimes id)\Phi
&\iff&\sum_je_j\otimes\Delta(u_{ji})=\sum_k\Phi(e_k)\otimes u_{ki}\\
&\iff&\sum_je_j\otimes\Delta(u_{ji})=\sum_{jk}e_j\otimes u_{jk}\otimes u_{ki}\\
&\iff&\Delta(u_{ji})=\sum_ku_{jk}\otimes u_{ki}
\end{eqnarray*}

As for the axiom involving the counit, here we have as well, as desired:
\begin{eqnarray*}
(id\otimes\varepsilon)\Phi=id
&\iff&\sum_j\varepsilon(u_{ji})e_j=e_i\\
&\iff&\varepsilon(u_{ji})=\delta_{ji}
\end{eqnarray*}

(2) We have the following formula:
\begin{eqnarray*}
\Phi(e_i)
&=&\sum_je_j\otimes u_{ji}\\
&=&\left(\sum_{ij}e_{ji}\otimes u_{ji}\right)(e_i\otimes 1)\\
&=&u(e_i\otimes 1)
\end{eqnarray*}

By using this formula, we obtain the following identity:
\begin{eqnarray*}
\Phi(e_ie_k)
&=&u(e_ie_k\otimes 1)\\
&=&u(\mu\otimes id)(e_i\otimes e_k\otimes 1)
\end{eqnarray*}

On the other hand, we have as well the following identity, as desired:
\begin{eqnarray*}
\Phi(e_i)\Phi(e_k)
&=&\sum_{jl}e_je_l\otimes u_{ji}u_{lk}\\
&=&(\mu\otimes id)\sum_{jl}e_j\otimes e_l\otimes u_{ji}u_{lk}\\\
&=&(\mu\otimes id)\left(\sum_{ijkl}e_{ji}\otimes e_{lk}\otimes u_{ji}u_{lk}\right)(e_i\otimes e_k\otimes 1)\\
&=&(\mu\otimes id)u^{\otimes 2}(e_i\otimes e_k\otimes 1)
\end{eqnarray*}

\smallskip

(3) The formula $\Phi(e_i)=u(e_i\otimes1)$ found above gives by linearity $\Phi(1)=u(1\otimes1)$, which shows that $\Phi$ is unital precisely when $u(1\otimes1)=1\otimes1$, as desired.

\medskip

(4) This follows from the following computation, by applying the involution:
\begin{eqnarray*}
(tr\otimes id)\Phi(e_i)=tr(e_i)1
&\iff&\sum_jtr(e_j)u_{ji}=tr(e_i)1\\
&\iff&\sum_ju_{ji}^*1_j=1_i\\
&\iff&(u^*1)_i=1_i\\
&\iff&u^*1=1
\end{eqnarray*}

(5) Assuming that (1-4) are satisfied, and that $\Phi$ is involutive, we have:
\begin{eqnarray*}
(u^*u)_{ik}
&=&\sum_lu_{li}^*u_{lk}\\
&=&\sum_{jl}tr(e_j^*e_l)u_{ji}^*u_{lk}\\
&=&(tr\otimes id)\sum_{jl}e_j^*e_l\otimes u_{ji}^*u_{lk}\\
&=&(tr\otimes id)(\Phi(e_i)^*\Phi(e_k))\\
&=&(tr\otimes id)\Phi(e_i^*e_k)\\
&=&tr(e_i^*e_k)1\\
&=&\delta_{ik}
\end{eqnarray*}

Thus $u^*u=1$, and since we know from (1) that $u$ is a corepresentation, it follows that $u$ is unitary. The proof of the converse is standard too, by using similar tricks.
\end{proof}

Still following \cite{ba2}, \cite{wa2}, we have the following result, extending the basic theory of $S_N^+$ to the present finite noncommutative space setting:

\index{finite quantum space}

\begin{theorem}
Given a finite quantum space $F$, there is a universal compact quantum group $S_F^+$ acting on $F$, leaving the counting measure invariant. We have
$$C(S_F^+)=C(U_N^+)\Big/\Big<\mu\in Hom(u^{\otimes2},u),\eta\in Fix(u)\Big>$$
where $N=|F|$ and where $\mu,\eta$ are the multiplication and unit maps of $C(F)$. For $F=\{1,\ldots,N\}$ we have $S_F^+=S_N^+$. Also, for the space $F=M_2$ we have $S_F^+=SO_3$.
\end{theorem}

\begin{proof}
This result is from \cite{ba2}, the idea being as follows:

\medskip

(1) This follows from Proposition 9.13 above, by using the standard fact that the complex conjugate of a corepresentation is a corepresentation too.

\medskip

(2) Regarding now the main example, for $F=\{1,\ldots,N\}$ we obtain indeed the quantum permutation group $S_N^+$, due to the abstract result in Proposition 9.3 above.

\medskip

(3) In order to do now the computation for $F=M_2$, we use some standard facts about $SU_2,SO_3$. We have an action by conjugation $SU_2\curvearrowright M_2(\mathbb C)$, and this action produces, via the canonical quotient map $SU_2\to SO_3$, an action $SO_3\curvearrowright M_2(\mathbb C)$. On the other hand, it is routine to check, by using arguments like those in the proof of Theorem 9.4 at $N=2,3$, that any action $G\curvearrowright M_2(\mathbb C)$ must come from a classical group. We conclude that the action $SO_3\curvearrowright M_2(\mathbb C)$ is universal, as claimed.
\end{proof}

Regarding now the representation theory of these generalized quantum permutation groups $S_F^+$, the result here, from \cite{ba2}, is very similar to the one for $S_N^+$, as follows:

\index{Temperley-Lieb algebra}
\index{fusion rules}
\index{Marchenko-Pastur law}

\begin{theorem}
The quantum groups $S_F^+$ have the following properties: 
\begin{enumerate}
\item The associated Tannakian categories are $TL_N$, with $N=|F|$.

\item The main character follows the Marchenko-Pastur law $\pi_1$, when $N\geq4$.

\item The fusion rules for $S_F^+$ with $|F|\geq4$ are the same as for $SO_3$.
\end{enumerate}
\end{theorem}

\begin{proof}
Once again this result is from \cite{ba2}, the idea being as follows:

\medskip

(1) Our first claim is that the fundamental representation is equivalent to its adjoint, $u\sim\bar{u}$. Indeed, let us go back to the coaction formula from Proposition 9.13:
$$\Phi(e_i)=\sum_je_j\otimes u_{ji}$$

We can pick our orthogonal basis $\{e_i\}$ to be the stadard multimatrix basis of $C(F)$, so that we have $e_i^*=e_{i^*}$, for a certain involution $i\to i^*$ on the index set. With this convention made, by conjugating the above formula of $\Phi(e_i)$, we obtain:
$$\Phi(e_{i^*})=\sum_je_{j^*}\otimes u_{ji}^*$$

Now by interchanging $i\leftrightarrow i^*$ and $j\leftrightarrow j^*$, this latter formula reads:
$$\Phi(e_i)=\sum_je_j\otimes u_{j^*i^*}^*$$

We therefore conclude, by comparing with the original formula, that we have:
$$u_{ji}^*=u_{j^*i^*}$$

But this shows that we have an equivalence $u\sim\bar{u}$, as claimed. Now with this result in hand, the proof goes as for the proof for $S_N^+$. To be more precise, the result follows from the fact that the multiplication and unit of any complex algebra, and in particular of $C(F)$, can be modelled by the following two diagrams:
$$m=|\cup|\quad,\quad u=\cap$$

Indeed, this is certainly true algebrically, and this is something well-known. As in what regards the $*$-structure, things here are fine too, because our choice for the trace leads to the following formula, which must be satisfied as well:
$$\mu\mu^*=N\cdot id$$

But the above diagrams $m,u$ generate the Temperley-Lieb algebra $TL_N$, as stated.

\medskip

(2) The proof here is exactly as for $S_N^+$, by using moments. To be more precise, according to (1) these moments are the Catalan numbers, which are the moments of $\pi_1$.

\medskip

(3) Once again same proof as for $S_N^+$, by using the fact that the moments of $\chi$ are the Catalan numbers, which naturally leads to the Clebsch-Gordan rules.
\end{proof}

It is quite clear now that our present formalism, and the above results, provide altogether a good and conceptual explanation for our $SO_3$ result regarding $S_N^+$. To be more precise, we can merge and reformulate the above results in the following way:

\begin{theorem}
The quantun groups $S_F^+$ have the following properties: 
\begin{enumerate}
\item For $F=\{1,\ldots,N\}$ we have $S_F^+=S_N^+$.

\item For the space $F=M_N$ we have $S_F^+=PO_N^+=PU_N^+$.

\item In particular, for the space $F=M_2$ we have $S_F^+=SO_3$.

\item The fusion rules for $S_F^+$ with $|F|\geq4$ are independent of $F$.

\item Thus, the fusion rules for $S_F^+$ with $|F|\geq4$ are the same as for $SO_3$.
\end{enumerate}
\end{theorem}

\begin{proof}
This is basically a compact form of what has been said above, with a new result added, and with some technicalities left aside:

\medskip

(1) This is something that we know from Theorem 9.14.

\medskip

(2) We know from chapter 4 that the inclusion $PO_N^+\subset PU_N^+$ is an isomorphism, with this coming from the formula $\widetilde{O}_N^+=U_N^+$, but we will actually reprove this result. Consider indeed the standard vector space action $U_N^+\curvearrowright\mathbb C^N$, and then its adjoint action $PU_N^+\curvearrowright M_N(\mathbb C)$. By universality of $S_{M_N}^+$, we have inclusions as follows:
$$PO_N^+\subset PU_N^+\subset S_{M_N}^+$$

On the other hand, the main character of $O_N^+$ with $N\geq2$ being semicircular, the main character of $PO_N^+$ must be Marchenko-Pastur. Thus the inclusion $PO_N^+\subset S_{M_N}^+$ has the property that it keeps fixed the law of main character, and by Peter-Weyl theory we conclude that this inclusion must be an isomorphism, as desired.

\medskip

(3) This is something that we know from Theorem 9.14, and that can be deduced as well from (2), by using the formula $PO_2^+=SO_3$, which is something elementary.

\medskip

(4) This is something that we know from Theorem 9.15.

\medskip

(5) This follows from (3,4), as already pointed out in Theorem 9.15.
\end{proof}

All this is certainly conceptual, but perhaps a bit too abstract. At $N=4$ we can formulate a more concrete result on the subject, by using the following construction:

\index{twisting}

\begin{definition}
$C(SO_3^{-1})$ is the universal $C^*$-algebra generated by the entries of a $3\times 3$ orthogonal matrix $a=(a_{ij})$, with the following relations:
\begin{enumerate}
\item Skew-commutation: $a_{ij}a_{kl}=\pm a_{kl}a_{ij}$, with sign $+$ if $i\neq k,j\neq l$, and $-$ otherwise.

\item Twisted determinant condition: $\Sigma_{\sigma\in S_3}a_{1\sigma(1)}a_{2\sigma(2)}a_{3\sigma(3)}=1$.
\end{enumerate}
\end{definition}

Observe the similarity with the twisting constructions from chapter 7. However, $SO_3$ being not easy, we are not exactly in the Schur-Weyl twisting framework from there. 

\bigskip

Our first task would be to prove that $C(SO_3^{-1})$ is a Woronowicz algebra. This is of course possible, by doing some computations, but we will not need to do these computations, because the result follows from the following theorem, from \cite{bb1}:

\begin{theorem}
We have an isomorphism of compact quantum groups
$$S_4^+=SO_3^{-1}$$
given by the Fourier transform over the Klein group $K=\mathbb Z_2\times\mathbb Z_2$.
\end{theorem}

\begin{proof}
Consider indeed the matrix $a^+=diag(1,a)$, corresponding to the action of $SO_3^{-1}$ on $\mathbb C^4$, and apply to it the Fourier transform over the Klein group $K=\mathbb Z_2\times\mathbb Z_2$: 
$$u=
\frac{1}{4}
\begin{pmatrix}
1&1&1&1\\
1&-1&-1&1\\
1&-1&1&-1\\
1&1&-1&-1
\end{pmatrix}
\begin{pmatrix}
1&0&0&0\\
0&a_{11}&a_{12}&a_{13}\\
0&a_{21}&a_{22}&a_{23}\\
0&a_{31}&a_{32}&a_{33}
\end{pmatrix}
\begin{pmatrix}
1&1&1&1\\
1&-1&-1&1\\
1&-1&1&-1\\
1&1&-1&-1
\end{pmatrix}$$

It is routine to check that this matrix is magic, and vice versa, i.e. that the Fourier transform over $K$ converts the relations in Definition 9.17 into the magic relations in Definition 9.1. Thus, we obtain the identification from the statement.
\end{proof}

Yet another extension of Theorem 9.10, which is however quite technical, comes by looking at the general case $N=n^2$, with $n\geq2$. It is possible indeed to complement Theorem 9.16 above with a general twisting result of the following type:
$$G^+(\widehat{F}_\sigma)=G^+(\widehat{F})^\sigma$$

To be more precise, this formula is valid indeed, for any finite group $F$ and any $2$-cocycle $\sigma$ on it. In the case $F=\mathbb Z_n^2$ with Fourier cocycle on it, this leads to the conclusion that $PO_n^+$ appears as a cocycle twist of $S_{n^2}^+$. For details here, we refer to \cite{bbs}. 

\bigskip

We have the following interesting probabilistic fact, from \cite{bbs} as well:

\begin{theorem}
The following families of variables have the same joint law,
\begin{enumerate}
\item $\{u_{ij}^2\}\in C(O_n^+)$,

\item $\{X_{ij}=\frac{1}{n}\sum_{ab}p_{ia,jb}\}\in C(S_{n^2}^+)$,
\end{enumerate}
where $u=(u_{ij})$ and $p=(p_{ia,jb})$ are the corresponding fundamental corepresentations.
\end{theorem}

\begin{proof}
As explained in \cite{bbs}, this result can be obtained via the above twisting methods. An alternative approach, also from \cite{bbs}, is by using the Weingarten formula for our two quantum groups, and the shrinking operation $\pi\to\pi'$. Indeed, we have:
\begin{eqnarray*}
\int_{O_n^+}u_{ij}^{2k}&=&\sum_{\pi,\sigma\in NC_2(2k)}W_{2k,n}(\pi,\sigma)\\
\int_{S_{n^2}^+}X_{ij}^k&=&\sum_{\pi,\sigma\in NC_2(2k)}n^{|\pi'|+|\sigma'|-k}W_{k,n^2}(\pi',\sigma')
\end{eqnarray*}

The point now is that, via the shrinking operation $\pi\to\pi'$, the Gram matrices of $NC_2(2k),NC(k)$ are related by the following formula, $\Delta_{kn}$ being the diagonal of $G_{kn}$:
$$G_{2k,n}(\pi,\sigma)=n^k(\Delta_{kn}^{-1}G_{k,n^2}\Delta_{kn}^{-1})(\pi',\sigma')$$

Thus in our moment formulae above the summands coincide, and so the moments are equal, as desired. The proof in general, dealing with joint moments, is similar.
\end{proof}

\index{hypergeometric law}
\index{hyperspherical law}

The above result is quite interesting, because it makes a connection between free hyperspherical and free hypergeometric laws. We refer here to \cite{bbs}, \cite{bcz}. 

\section*{9d. Poisson laws}

Let us go back now to our main result so far, namely Theorem 9.10, and further build on that, with probabilistic results. We have the following result:

\index{Marchenko-Pastur law}
\index{coamenability}

\begin{theorem}
The spectral measure of the main character of $S_N^+$ with $N\geq4$ is the Marchenko-Pastur law of parameter $1$, having the following density:
$$\pi_1=\frac{1}{2\pi}\sqrt{4x^{-1}-1}dx$$
Also, $S_4^+$ is coamenable, and $S_N^+$ with $N\geq5$ is not coamenable.
\end{theorem}

\begin{proof}
Here the first assertion follows from the following formula, which can be established by doing some calculus, and more specifically by setting $x=4\sin^2t$:
$$\frac{1}{2\pi}\int_0^4\sqrt{1-4x^{-1}}x^kdx=C_k$$

As for the second assertion, this follows from this, which shows that the spectrum of the main character is $[0,4]$, and from the Kesten criterion from chapter 3. 
\end{proof}

Our next purpose will be that of understanding, probabilistically speaking, the liberation operation $S_N\to S_N^+$. In what regards $S_N$, we have the following basic result:

\index{derangement}
\index{Poisson law}
\index{main character}
\index{inclusion-exclusion}

\begin{theorem}
Consider the symmetric group $S_N$, regarded as a compact group of matrices, $S_N\subset O_N$, via the standard permutation matrices.
\begin{enumerate}
\item The main character $\chi\in C(S_N)$, defined as usual as $\chi=\sum_iu_{ii}$, counts the number of fixed points, $\chi(\sigma)=\#\{i|\sigma(i)=i\}$.

\item The probability for a permutation $\sigma\in S_N$ to be a derangement, meaning to have no fixed points at all, becomes, with $N\to\infty$, equal to $1/e$.

\item The law of the main character $\chi\in C(S_N)$ becomes, with $N\to\infty$, a Poisson law of parameter $1$, with respect to the counting measure.
\end{enumerate}
\end{theorem}

\begin{proof}
This is something very classical, and beautiful, as follows:

\medskip

(1) We have indeed the following computation:
$$\chi(\sigma)
=\sum_iu_{ii}(\sigma)
=\sum_i\delta_{\sigma(i)i}
=\#\left\{i\big|\sigma(i)=i\right\}$$

(2) This is best viewed by using the inclusion-exclusion principle. Let us set:
$$S_N^{i_1\ldots i_k}=\left\{\sigma\in S_N\Big|\sigma(i_1)=i_1,\ldots,\sigma(i_k)=i_k\right\}$$

By using the inclusion-exclusion principle, we have, as desired:
\begin{eqnarray*}
\mathbb P(\chi=0)
&=&\frac{1}{N!}|(S_1\cup\ldots\cup S_N)^c|\\
&=&\frac{1}{N!}\left(|S_N|-\sum_i|S_N^i|+\sum_{i<j}|S_N^{ij}|-\ldots+(-1)^N\sum_{i_1<\ldots<i_N}|S_N^{i_1\ldots i_N}|\right)\\
&=&\frac{1}{N!}\sum_{k=0}^N(-1)^k\binom{N}{k}(N-k)!\\
&=&1-\frac{1}{1!}+\frac{1}{2!}-\ldots+(-1)^{N-1}\frac{1}{(N-1)!}+(-1)^N\frac{1}{N!}
\end{eqnarray*}

(3) This follows by generalizing the computation in (2). To be more precise, a similar application of the inclusion-exclusion principle gives the following formula:
$$\lim_{N\to\infty}\mathbb P(\chi=k)=\frac{1}{k!e}$$

Thus, we obtain in the limit a Poisson law, as stated.
\end{proof}

In order to talk about free analogues of this, we will need some theory:

\index{PLT}
\index{Poisson limit theorem}
\index{Poisson law}
\index{Marchenko-Pastur law}

\begin{theorem}
The following Poisson type limits converge, for any $t>0$,
$$p_t=\lim_{n\to\infty}\left(\left(1-\frac{t}{n}\right)\delta_0+\frac{t}{n}\delta_1\right)^{*n}\quad,\quad 
\pi_t=\lim_{n\to\infty}\left(\left(1-\frac{t}{n}\right)\delta_0+\frac{t}{n}\delta_1\right)^{\boxplus n}$$
the limiting measures being the Poisson law $p_t$, and the Marchenko-Pastur law $\pi_t$, 
$$p_t=\frac{1}{e^t}\sum_{k=0}^\infty\frac{t^k\delta_k}{k!}\quad,\quad 
\pi_t=\max(1-t,0)\delta_0+\frac{\sqrt{4t-(x-1-t)^2}}{2\pi x}\,dx$$
whose moments are given by $M_k(p_t)=\sum_{\pi\in D(k)}t^{|\pi|}$, with $D=P,NC$.
\end{theorem}

\begin{proof}
This is something standard, which follows by using either $\log F,R$ and calculus, or classical and free cumulants. The point indeed is that the limiting measures must be those having classical and free cumulants $t,t,t,\ldots$ But this gives all the assertions, the density computations being standard. We refer here to \cite{mpa}, \cite{nsp}, \cite{vdn}, \cite{wig}, but we will be back to this in chapter 10 below, with full details, directly in a more general setting.
\end{proof}

We can now formulate a conceptual result about $S_N\to S_N^+$, as follows:

\index{Bercovici-Pata bijection}
\index{Poisson law}

\begin{theorem}
The law of the main character $\chi_u$ is as follows:
\begin{enumerate}
\item For $S_N$ with $N\to\infty$ we obtain a Poisson law $p_1$.

\item For $S_N^+$ with $N\geq4$ we obtain a Marchenko-Pastur, or free Poisson law $\pi_1$.
\end{enumerate}
In addition, these laws are related by the Bercovici-Pata correspondence.
\end{theorem}

\begin{proof}
This follows indeed from the computations that we have, from Theorem 9.20 and Theorem 9.21, by using the various theoretical results from Theorem 9.22.
\end{proof}

As in the continuous case, our purpose now will be that of extending this result to the truncated characters. In order to discuss the classical case, we first have:

\begin{proposition}
Consider the symmetric group $S_N$, together with its standard matrix coordinates $u_{ij}=\chi(\sigma\in S_N|\sigma(j)=i)$. We have the formula
$$\int_{S_N}u_{i_1j_1}\ldots u_{i_kj_k}=\begin{cases}
\frac{(N-|\ker i|)!}{N!}&{\rm if}\ \ker i=\ker j\\
0&{\rm otherwise}
\end{cases}$$
where $\ker i$ denotes as usual the partition of $\{1,\ldots,k\}$ whose blocks collect the equal indices of $i$, and where $|.|$ denotes the number of blocks.
\end{proposition}

\begin{proof}
According to the definition of $u_{ij}$, the integrals in the statement are given by:
$$\int_{S_N}u_{i_1j_1}\ldots u_{i_kj_k}=\frac{1}{N!}\#\left\{\sigma\in S_N\Big|\sigma(j_1)=i_1,\ldots,\sigma(j_k)=i_k\right\}$$

The existence of $\sigma\in S_N$ as above requires $i_m=i_n\iff j_m=j_n$. Thus, the integral vanishes when $\ker i\neq\ker j$. As for the case $\ker i=\ker j$, if we denote by $b\in\{1,\ldots,k\}$ the number of blocks of this partition, we have $N-b$ points to be sent bijectively to $N-b$ points, and so $(N-b)!$ solutions, and the integral is $\frac{(N-b)!}{N!}$, as claimed.
\end{proof}

We can now compute the laws of truncated characters, and we obtain:

\begin{proposition}
For the symmetric group $S_N\subset O_N$, regarded as a compact group of matrices, $S_N\subset O_N$, via the standard permutation matrices, the truncated character
$$\chi_t=\sum_{i=1}^{[tN]}u_{ii}$$
counts the number of fixed points among $\{1,\ldots,[tN]\}$, and its law with respect to the counting measure becomes, with $N\to\infty$, a Poisson law of parameter $t$. 
\end{proposition}

\index{Stirling numbers}

\begin{proof}
With $S_{k,b}$ being the Stirling numbers, we have:
\begin{eqnarray*}
\int_{S_N}\chi_t^k
&=&\sum_{i_1\ldots i_k=1}^{[tN]}\int_{S_N}u_{i_1i_1}\ldots u_{i_ki_k}\\
&=&\sum_{\pi\in P_k}\frac{[tN]!}{([tN]-|\pi|!)}\cdot\frac{(N-|\pi|!)}{N!}\\
&=&\sum_{b=1}^{[tN]}\frac{[tN]!}{([tN]-b)!}\cdot\frac{(N-b)!}{N!}\cdot S_{k,b}
\end{eqnarray*}

In particular with $N\to\infty$ we obtain the following formula:
$$\lim_{N\to\infty}\int_{S_N}\chi_t^k=\sum_{b=1}^kS_{k,b}t^b$$

But this is a Poisson($t$) moment, and so we are done.
\end{proof}

We can now finish our computations, and generalize Theorem 9.23, as follows:

\index{truncated character}
\index{Bercovici-Pata bijection}

\begin{theorem}
The laws of truncated characters $\chi_t=\sum_{i=1}^{[tN]}u_{ii}$ are as follows:
\begin{enumerate}
\item For $S_N$ with $N\to\infty$ we obtain a Poisson law $p_t$.

\item For $S_N^+$ with $N\to\infty$ we obtain a free Poisson law $\pi_t$.
\end{enumerate}
In addition, these laws are related by the Bercovici-Pata correspondence.
\end{theorem}

\begin{proof}
This follows from the above results:

\medskip

(1) This is something that we already know, from Proposition 9.25.

\medskip

(2) This is something that we know so far only at $t=1$, from Theorem 9.23. In order to deal with the general $t\in(0,1]$ case, we can use the same method as for the orthogonal and unitary quantum groups, from chapter 8, and we obtain the following moments:
$$M_k=\sum_{\pi\in NC(k)}t^{|\pi|}$$

But these numbers being the moments of the free Poisson law of parameter $t$, as explained in Theorem 9.22 above, we obtain the result. See \cite{bc2}.
\end{proof}

Summarizing, the liberation operation $S_N\to S_N^+$ has many common features with the liberation operations $O_N\to O_N^+$ and $U_N\to U_N^+$, studied in chapter 8 above.

\section*{9e. Exercises} 

There has been a lot of material in this chapter, with this second part of the present book being at a more advanced level than the first part, and most of our exercises here will be about better understanding what has been said above. Let us start with:

\index{Fourier transform}

\begin{exercise}
Prove that we have $S_3=S_3^+$ by looking at the coaction
$$\Phi:\mathbb C^3\to\mathbb C^3\otimes C(S_3^+)$$
written in terms of the Fourier basis of $\mathbb C^3$.
\end{exercise}

To be more precise, the question here is that of changing the basis of $\mathbb C^3$, by using the Fourier transform over the group $\mathbb Z_3$, then reformulating the magic condition at $N=3$ in terms of this new basis, and then deducing that the coefficients must commute.

\begin{exercise}
Prove directly, without Kesten, that the discrete quantum group
$$\Gamma=\widehat{S_5^+}$$
is not amenable, in the discrete quantum group sense.
\end{exercise}

This requires of course some knowledge of the notion of amenability. As a hint here, try finding a quantum subgroup $G\subset S_5^+$ whose dual is not amenable.

\begin{exercise}
Consider a discrete group generated by elements of finite order, written as a quotient group, as follows:
$$\mathbb Z_{N_1}*\ldots*\mathbb Z_{N_k}\to\Gamma$$
Prove that we have an embedding $\widehat{\Gamma}\subset S_N^+$, where $N=N_1+\ldots+N_k$.
\end{exercise}

This should be normally not very difficult. What is difficult, however, is to prove that any group dual subgroup $\widehat{\Gamma}\subset S_N^+$ appears as above. We will be back to this.

\begin{exercise}
Prove that we have the following equality:
$$S_{M_2}^+=SO_3$$
\end{exercise}

This is something that was already discussed in the above, but quite briefly. The problem now is that of working out all the details.

\begin{exercise}
Check out all the details for Theorem 9.22, regarding the Poisson and free Poisson limiting theorems.
\end{exercise}

There is quite some work to be done here, but everything is quite routine. As an alternative approach, we will discuss later in this book a generalization of this, regarding the compound Poisson and compound free Poisson limits, so the problem is to go there, and to work out in detail the particular case of the Poisson and free Poisson limits.

\chapter{Quantum reflections}

\section*{10a. Finite graphs}

Many interesting examples of quantum permutation groups appear as particular cases of the following general construction from \cite{ba3}, involving finite graphs:

\index{quantum automorphism group}

\begin{proposition}
Given a finite graph $X$, with adjacency matrix $d\in M_N(0,1)$, the following construction produces a quantum permutation group, 
$$C(G^+(X))=C(S_N^+)\Big/\Big<du=ud\Big>$$
whose classical version $G(X)$ is the usual  automorphism group of $X$.
\end{proposition}

\begin{proof}
The fact that we have a quantum group comes from the fact that $du=ud$ reformulates as $d\in End(u)$, which makes it clear that we are dividing by a Hopf ideal. Regarding the second assertion, we must establish here the following equality:
$$C(G(X))=C(S_N)\Big/\Big<du=ud\Big>$$

For this purpose, recall that $u_{ij}(\sigma)=\delta_{\sigma(j)i}$. By using this formula, we have:
\begin{eqnarray*}
(du)_{ij}(\sigma)
&=&\sum_kd_{ik}u_{kj}(\sigma)\\
&=&\sum_kd_{ik}\delta_{\sigma(j)k}\\
&=&d_{i\sigma(j)}
\end{eqnarray*}

On the other hand, we have as well the following formula:
\begin{eqnarray*}
(ud)_{ij}(\sigma)
&=&\sum_ku_{ik}(\sigma)d_{kj}\\
&=&\sum_k\delta_{\sigma(k)i}d_{kj}\\
&=&d_{\sigma^{-1}(i)j}
\end{eqnarray*}

Thus the condition $du=ud$ reformulates as $d_{ij}=d_{\sigma(i)\sigma(j)}$, and we are led to the usual notion of an action of a permutation group on $X$, as claimed.
\end{proof}

Let us work out some basic examples. We have the following result:

\index{simplex}
\index{complementary graph}
\index{square graph}

\begin{theorem}
The construction $X\to G^+(X)$ has the following properties:
\begin{enumerate}
\item For the $N$-point graph, having no edges at all, we obtain $S_N^+$.

\item For the $N$-simplex, having edges everywhere, we obtain as well $S_N^+$.

\item We have $G^+(X)=G^+(X^c)$, where $X^c$ is the complementary graph.

\item For a disconnected union, we have $G^+(X)\,\hat{*}\,G^+(Y)\subset G^+(X\sqcup Y)$.

\item For the square, we obtain a non-classical, proper subgroup of $S_4^+$.
\end{enumerate}
\end{theorem}

\begin{proof}
All these results are elementary, the proofs being as follows:

\medskip

(1) This follows from definitions, because here we have $d=0$.

\medskip

(2) Here $d=\mathbb I$ is the all-one matrix, and the magic condition gives $u\mathbb I=\mathbb Iu=N\mathbb I$. We conclude that $du=ud$ is automatic in this case, and so $G^+(X)=S_N^+$.

\medskip

(3) The adjacency matrices of $X,X^c$ being related by the formula $d_X+d_{X^c}=\mathbb I$. We can use here the above formula $u\mathbb I=\mathbb Iu=N\mathbb I$, and we conclude that $d_Xu=ud_X$ is equivalent to $d_{X^c}u=ud_{X^c}$. Thus, we obtain, as claimed, $G^+(X)=G^+(X^c)$.

\medskip

(4) The adjacency matrix of a disconnected union is given by $d_{X\sqcup Y}=diag(d_X,d_Y)$. Now let $w=diag(u,v)$ be the fundamental corepresentation of $G^+(X)\,\hat{*}\,G^+(Y)$. Then $d_Xu=ud_X$ and $d_Yv=vd_Y$ imply, as desired, $d_{X\sqcup Y}w=wd_{X\sqcup Y}$.

\medskip

(5) We know from (3) that we have $G^+(\square)=G^+(|\ |)$. We know as well from (4) that we have $\mathbb Z_2\,\hat{*}\,\mathbb Z_2\subset G^+(|\ |)$. It follows that $G^+(\square)$ is non-classical. Finally, the inclusion $G^+(\square)\subset S_4^+$ is indeed proper, because $S_4\subset S_4^+$ does not act on the square.
\end{proof}

In order to further advance, and to explicitely compute various quantum automorphism groups, we can use the spectral decomposition of $d$, as follows:

\index{spectral decomposition}

\begin{proposition}
A closed subgroup $G\subset S_N^+$ acts on a graph $X$ precisely when
$$P_\lambda u=uP_\lambda\quad,\quad\forall\lambda\in\mathbb R$$
where $d=\sum_\lambda\lambda\cdot P_\lambda$ is the spectral decomposition of the adjacency matrix of $X$.
\end{proposition}

\begin{proof}
Since $d\in M_N(0,1)$ is a symmetric matrix, we can consider indeed its spectral decomposition, $d=\sum_\lambda\lambda\cdot P_\lambda$. We have then the following formula:
$$<d>=span\left\{P_\lambda\Big|\lambda\in\mathbb R\right\}$$

But this shows that we have the following equivalence:
$$d\in End(u)\iff P_\lambda\in End(u),\forall\lambda\in\mathbb R$$

Thus, we are led to the conclusion in the statement.
\end{proof}

In order to exploit this, we will often combine it with the following standard fact:

\begin{proposition}
Consider a closed subgroup $G\subset S_N^+$, with associated coaction map 
$$\Phi:\mathbb C^N\to \mathbb C^N\otimes C(G)$$
For a linear subspace $V\subset\mathbb C^N$, the following are equivalent:
\begin{enumerate}
\item The magic matrix $u=(u_{ij})$ commutes with $P_V$.

\item $V$ is invariant, in the sense that $\Phi(V)\subset V\otimes C(G)$.
\end{enumerate}
\end{proposition}

\begin{proof}
Let $P=P_V$. For any $i\in\{1,\ldots,N\}$ we have the following formula:
\begin{eqnarray*}
\Phi(P(e_i))
&=&\Phi\left(\sum_kP_{ki}e_k\right)\\ 
&=&\sum_{jk}P_{ki}e_j\otimes u_{jk}\\
&=&\sum_je_j\otimes (uP)_{ji}
\end{eqnarray*}

On the other hand the linear map $(P\otimes id)\Phi$ is given by a similar formula:
\begin{eqnarray*}
(P\otimes id)(\Phi(e_i))
&=&\sum_kP(e_k)\otimes u_{ki}\\
&=&\sum_{jk}P_{jk}e_j\otimes u_{ki}\\
&=&\sum_je_j\otimes (Pu)_{ji}
\end{eqnarray*}

Thus $uP=Pu$ is equivalent to $\Phi P=(P\otimes id)\Phi$, and the conclusion follows.
\end{proof}

We have as well the following useful complementary result, from \cite{ba3}:

\index{color decomposition}

\begin{proposition}
Let $p\in M_N(\mathbb C)$ be a matrix, and consider its ``color'' decomposition, obtained by setting $(p_c)_{ij}=1$ if $p_{ij}=c$ and $(p_c)_{ij}=0$ otherwise:
$$p=\sum_{c\in\mathbb C}c\cdot p_c$$
Then $u=(u_{ij})$ commutes with $p$ if and only if it commutes with all matrices $p_c$.
\end{proposition}

\begin{proof}
Consider the multiplication and counit maps of the algebra $\mathbb C^N$:
$$M:e_i\otimes e_j\to e_ie_j$$
$$C:e_i\to e_i\otimes e_i$$

Since $M,C$ intertwine $u,u^{\otimes 2}$, their iterations $M^{(k)},C^{(k)}$ intertwine $u,u^{\otimes k}$, and so:
\begin{eqnarray*}
p^{(k)}
&=&M^{(k)}p^{\otimes k}C^{(k)}\\
&=&\sum_{c\in\mathbb C}c^kp_c\\
&\in&End(u)
\end{eqnarray*}

Let $S=\{c\in\mathbb C|p_c\neq 0\}$, and $f(c)=c$. By Stone-Weierstrass we have $S=<f>$, and so for any $e\in S$ the Dirac mass at $e$ is a linear combination of powers of $f$:
\begin{eqnarray*}
\delta_e
&=&\sum_{k}\lambda_kf^k\\
&=&\sum_{k}\lambda_k \left(\sum_{c\in S}c^k\delta_c\right)\\
&=&\sum_{c\in S}\left(\sum_{k}\lambda_kc^k\right)\delta_c
\end{eqnarray*}

The corresponding linear combination of matrices $p^{(k)}$ is given by:
\begin{eqnarray*}
\sum_k\lambda_kp^{(k)}
&=&\sum_k\lambda_k \left(\sum_{c\in S}c^kp_c\right)\\
&=&\sum_{c\in S}\left(\sum_{k}\lambda_kc^k\right)p_c
\end{eqnarray*}

The Dirac masses being linearly independent, in the first formula all coefficients in the right term are 0, except for the coefficient of $\delta_e$, which is 1. Thus the right term in the second formula is $p_e$, and it follows that we have $p_e\in End(u)$, as claimed.
\end{proof}

The above results can be combined, and we are led into a ``color-spectral'' decomposition method for $d$, which can lead to a number of nontrivial results. In fact, all this is best understood in terms of Jones' planar algebras \cite{jo3}. We refer here to \cite{ba3}.

\bigskip

As a basic application of this, we can further study $G^+(\square)$, as follows:

\index{cycle graph}
\index{dihedral group}

\begin{proposition}
The quantum automorphism group of the $N$-cycle is as follows:
\begin{enumerate}
\item At $N\neq 4$ we have $G^+(X)=D_N$.
 
\item At $N=4$ we have $D_4\subset G^+(X)\subset S_4^+$, with proper inclusions.
\end{enumerate}
\end{proposition}

\begin{proof}
We already know that the results hold at $N\leq4$, so let us assume $N\geq5$. Given a $N$-th root of unity, $w^N=1$, consider the following vector:
$$\xi=(w^i)$$ 

This is an eigenvector of $d$, with eigenvalue $w+w^{N-1}$. With $w=e^{2\pi i/N}$, it follows that $1,f,f^2,\ldots ,f^{N-1}$ are eigenvectors of $d$. More precisely, the invariant subspaces of $d$ are as follows, with the last subspace having dimension 1 or 2, depending on $N$:
$$\mathbb C 1,\, \mathbb C f\oplus\mathbb C f^{N-1},\, \mathbb C f^2\oplus\mathbb C f^{N-2},\ldots$$

Consider now the associated coaction $\Phi:\mathbb C^N\to \mathbb C^N\otimes C(G)$, and write:
$$\Phi(f)=f\otimes a+f^{N-1}\otimes b$$

By taking the square of this equality we obtain:
$$\Phi(f^2)=f^2\otimes a^2+f^{N-2}\otimes b^2+1\otimes(ab+ba)$$

It follows that $ab=-ba$, and that $\Phi(f^2)$ is given by the following formula:
$$\Phi(f^2)=f^2\otimes a^2+f^{N-2}\otimes b^2$$

By multiplying this with $\Phi(f)$ we obtain:
$$\Phi(f^3)=f^3\otimes a^3+f^{N-3}\otimes b^3+f^{N-1}\otimes ab^2+f\otimes ba^2$$

Now since $N\geq 5$ implies that $1,N-1$ are different from $3,N-3$, we must have $ab^2=ba^2=0$. By using this and $ab=-ba$, we obtain by recurrence on $k$ that:
$$\Phi(f^k)=f^k\otimes a^k+f^{N-k}\otimes b^k$$

In particular at $k=N-1$ we obtain:
$$\Phi(f^{N-1})=f^{N-1}\otimes a^{N-1}+f\otimes b^{N-1}$$

On the other hand we have $f^*=f^{N-1}$, so by applying $*$ to $\Phi(f)$ we get:
$$\Phi(f^{N-1})=f^{N-1}\otimes a^*+f\otimes b^*$$

Thus $a^*=a^{N-1}$ and $b^*=b^{N-1}$. Together with $ab^2=0$ this gives:
\begin{eqnarray*}
(ab)(ab)^*
&=&abb^*a^*\\
&=&ab^Na^{N-1}\\
&=&(ab^2)b^{N-2}a^{N-1}\\
&=&0
\end{eqnarray*}

From positivity we get from this $ab=0$, and together with $ab=-ba$, this shows that $a,b$ commute. On the other hand $C(G)$ is generated by the coefficients of $\Phi$, which are powers of $a,b$, and so $C(G)$ must be commutative, and we obtain the result.
\end{proof}

Summarizing, this was a bad attempt in understanding $G^+(\square)$, which appears to be ``exceptional'' among the quantum symmetry groups of the $N$-cycles. 

\bigskip

An alternative approach to $G^+(\square)$ comes by regarding the square as the $N=2$ particular case of the $N$-hypercube $\square_N$. Indeed, the usual symmetry group of $\square_N$ is the hyperoctahedral group $H_N$, so we should have a formula of the following type:
$$G(\square)=H_2^+$$

In order to clarify this, let us start with the following simple fact:

\index{square graph}
\index{hypercube}

\begin{proposition}
We have an embedding as follows, $g_i$ being the generators of $\mathbb Z_2^N$,
$$\widehat{\mathbb Z_2^N}\subset S^{N-1}_{\mathbb R,+}\quad,\quad x_i=\frac{g_i}{\sqrt{N}}$$
whose image is the geometric hypercube:
$$\square_N=\left\{x\in\mathbb R^N\Big|x_i=\pm\frac{1}{\sqrt{N}},\forall i\right\}$$ 
\end{proposition}

\begin{proof}
This is something that we already know, from chapter 1 above. Consider indeed the following standard group algebra generators:
$$g_i\in C^*(\mathbb Z_2^N)=C(\widehat{\mathbb Z_2^N})$$

These generators satisfy satisfy then $g_i=g_i^*$, $g_i^2=1$, and when rescaling by $1/\sqrt{N}$, we obtain the relations defining $\square_N$.
\end{proof}

We can now study the quantum groups $G^+(\square_N)$, and we are led to the quite surprising conclusion, from \cite{bbc}, that these are the twisted orthogonal groups $\bar{O}_N$:

\begin{theorem}
With $\mathbb Z_2^N=<g_1,\ldots,g_N>$ we have a coaction map
$$\Phi:C^*(\mathbb Z_2^N)\to C^*(\mathbb Z_2^N)\otimes C(\bar{O}_N)\quad,\quad g_i\to\sum_jg_j\otimes u_{ji}$$
which makes $\bar{O}_N$ the quantum isometry group of the hypercube $\square_N=\widehat{\mathbb Z_2^N}$, as follows:
\begin{enumerate}
\item With $\square_N$ viewed as an algebraic manifold, $\square_N\subset S^{N-1}_\mathbb R\subset S^{N-1}_{\mathbb R,+}$.

\item With $\square_N$ viewed as a graph, with $2^N$ vertices and $2^{N-1}N$ edges.

\item With $\square_N$ viewed as a metric space, with metric coming from $\mathbb R^N$.
\end{enumerate}
\end{theorem} 

\begin{proof}
Observe first that $\square_N$ is indeed an algebraic manifold, so (1) as formulated above makes sense, in the general framework of chapter 2. The cube $\square_N$ is also a graph, as indicated, and so (2) makes sense as well, in the framework of Proposition 10.1. Finally, (3) makes sense as well, because we can define the quantum isometry group of a finite metric space exactly as for graphs, but with $d$ being this time the distance matrix.

\medskip

(1) In order for $G\subset O_N^+$ to act affinely on $\square_N$, the variables $G_i=\sum_jg_j\otimes u_{ji}$ must satisfy the same relations as the generators $g_i\in \mathbb Z_2^N$. The self-adjointness condition being automatic, the relations to be checked are therefore:
$$G_i^2=1\quad,\quad G_iG_j=G_jG_i$$

We have the following computation:
\begin{eqnarray*}
G_i^2
&=&\sum_{kl}g_kg_l\otimes u_{ik}u_{il}\\
&=&1+\sum_{k<l}g_kg_l\otimes(u_{ik}u_{il}+u_{il}u_{ik})
\end{eqnarray*}

As for the commutators, these are given by:
\begin{eqnarray*}
\left[G_i,G_j\right]
&=&\sum_{k<l}g_kg_l\otimes(u_{ik}u_{jl}-u_{jk}u_{il}+u_{il}u_{jk}-u_{jl}u_{ik})
\end{eqnarray*}

From the first relation we obtain $ab=0$ for $a\neq b$ on the same row of $u$, and by using the antipode, the same happens for the columns. From the second relation we obtain:
$$[u_{ik},u_{jl}]=[u_{jk},u_{il}]\quad,\quad\forall k\neq l$$

We use the Bhowmick-Goswami trick \cite{bhg}. By applying the antipode we obtain:
$$[u_{lj},u_{ki}]=[u_{li},u_{kj}]$$

By relabelling, this gives the following formula:
$$[u_{ik},u_{jl}]=[u_{il},u_{jk}]\quad,\quad j\neq i$$

Thus for $i\neq j,k\neq l$ we must have:
$$[u_{ik},u_{jl}]=[u_{jk},u_{il}]=0$$

We are therefore led to $G\subset\bar{O}_N$, as claimed.

\medskip

(2) We can use here the fact that the cube $\square_N$, when regarded as a graph, is the Cayley graph of the group $\mathbb Z_2^N$. The eigenvectors and eigenvalues of $\square_N$ are as follows:
\begin{eqnarray*}
v_{i_1\ldots i_N}&=&\sum_{j_1\ldots j_N} (-1)^{i_1j_1
+\ldots+i_Nj_N}g_1^{j_1}\ldots g_N^{j_N}\\
\lambda_{i_1\ldots i_N}&=&(-1)^{i_1}+\ldots +(-1)^{i_N}
\end{eqnarray*}

With this picture in hand, and by using Proposition 10.3 and Proposition 10.4 above, the result follows from the same computations as in the proof of (1). See \cite{bbc}.

\medskip

(3) Our claim here is that we obtain the same symmetry group as in (2). Indeed, observe that distance matrix of the cube has a color decomposition as follows:
$$d=d_1+\sqrt{2}d_2+\sqrt{3}d_3+\ldots+
\sqrt{N}d_N$$

Since the powers of $d_1$ can be computed by counting loops on the cube, we have formulae as follows, with $x_{ij}\in\mathbb N$ being certain positive integers:
\begin{eqnarray*}
d_1^2&=&x_{21}1_N+x_{22}d_2\\
d_1^3&=&x_{31}1_N+x_{32}d_2+ x_{33}d_3\\
&\ldots&\\
d_1^N&=&x_{N1}1_N+x_{N2}d_2+x_{N3}d_3+\ldots+x_{NN}d_N
\end{eqnarray*}

But this shows that we have $<d>=<d_1>$. Now since $d_1$ is the adjacency matrix of $\square_N$, viewed as graph, this proves our claim, and we obtain the result via (2).
\end{proof}

Now back to our questions regarding the square, we have $G^+(\square)=\bar{O}_2$, and this formula appears as the $N=2$ particular case of a general formula, namely $G^+(\square_N)=\bar{O}_N$. This is quite conceptual, but still not ok. The problem is that we have $G(\square_N)=H_N$, and so for our theory to be complete, we would need a formula of type $H_N^+=\bar{O}_N$. And this latter formula is obviously wrong, because for $\bar{O}_N$ the character computations lead to Gaussian laws, who cannot appear as liberations of the character laws for $H_N$, that we have not computed yet, but which can only be something Poisson-related.

\section*{10b. Reflection groups}

Summarizing, the problem of conceptually understanding $G(\square)$ remains open. In order to present now the correct, final solution, the idea will be that to look at the quantum group $G^+(|\ |)$ instead, which is equal to it, according to Proposition 10.2 (3). We first have the following result, extending Proposition 10.2 (4) above:

\index{disconnected union}
\index{dual free product}

\begin{proposition}
For a disconnected union of graphs we have
$$G^+(X_1)\;\hat{*}\;\ldots\;\hat{*}\;G^+(X_k)\subset G^+(X_1\sqcup\ldots\sqcup X_k)$$
and this inclusion is in general not an isomorphism.
\end{proposition}

\begin{proof}
The proof of the first assertion is nearly identical to the proof of Proposition 10.2 (4) above. Indeed, the adjacency matrix of the disconnected union is given by:
$$d_{X_1\sqcup\ldots\sqcup X_k}=diag(d_{X_1},\ldots,d_{X_k})$$
$$w=diag(u_1,\ldots,u_k)$$

We have then $d_{X_i}u_i=u_id_{X_i}$, and this implies $dw=wd$, which gives the result. As for the last assertion, this is something that we already know, from Proposition 10.6 (2).
\end{proof}

In the case where the graphs $X_1,\ldots,X_k$ are identical, which is the one that we are truly interested in, we can further build on this. Following Bichon \cite{bi1}, we have:

\index{free wreath product}

\begin{proposition}
Given closed subgroups $G\subset U_N^+$, $H\subset S_k^+$, with fundamental corepresentations $u,v$, the following construction produces a closed subgroup of $U_{Nk}^+$:
$$C(G\wr_*H)=(C(G)^{*k}*C(H))/<[u_{ij}^{(a)},v_{ab}]=0>$$
In the case where $G,H$ are classical, the classical version of $G\wr_*H$ is the usual wreath product $G\wr H$. Also, when $G$ is a quantum permutation group, so is $G\wr_*H$.
\end{proposition}

\begin{proof}
Consider indeed the matrix $w_{ia,jb}=u_{ij}^{(a)}v_{ab}$, over the quotient algebra in the statement. It is routine to check that $w$ is unitary, and in the case $G\subset S_N^+$, our claim is that this matrix is magic. Indeed, the entries are projections, because they appear as products of commuting projections, and the row sums are as follows:
\begin{eqnarray*}
\sum_{jb}w_{ia,jb}
&=&\sum_{jb}u_{ij}^{(a)}v_{ab}\\
&=&\sum_bv_{ab}\sum_ju_{ij}^{(a)}\\
&=&1
\end{eqnarray*}

As for the column sums, these are as follows:
\begin{eqnarray*}
\sum_{ia}w_{ia,jb}
&=&\sum_{ia}u_{ij}^{(a)}v_{ab}\\
&=&\sum_av_{ab}\sum_iu_{ij}^{(a)}\\
&=&1
\end{eqnarray*}

With these observations in hand, it is routine to check that $G\wr_*H$ is indeed a quantum group, with fundamental corepresentation $w$, by constructing maps $\Delta,\varepsilon,S$ as in section 1, and in the case $G\subset S_N^+$, we obtain in this way a closed subgroup of $S_{Nk}^+$. Finally, the assertion regarding the classical version is standard as well. See \cite{bi1}.
\end{proof}

We refer to Bichon \cite{bi1} and to Tarrago-Wahl \cite{twa} for more details regarding the above construction. Now with this notion in hand, we have the following result:

\begin{theorem}
Given a connected graph $X$, and $k\in\mathbb N$, we have the formulae
$$G(kX)=G(X)\wr S_k$$
$$G^+(kX)=G^+(X)\wr_*S_k^+$$
where $kX=X\sqcup\ldots\sqcup X$ is the $k$-fold disjoint union of $X$ with itself.
\end{theorem}

\begin{proof}
The first formula is something well-known, which follows as well from the second formula, by taking the classical version. Regarding now the second formula, it is elementary to check that we have an inclusion as follows, for any finite graph $X$:
$$G^+(X)\wr_*S_k^+\subset G^+(kX)$$

Indeed, we want to construct an action $G^+(X)\wr_*S_k^+\curvearrowright kX$, and this amounts in proving that we have $[w,d]=0$. But, the matrices $w,d$ are given by:
$$w_{ia,jb}=u_{ij}^{(a)}v_{ab}\quad,\quad d_{ia,jb}=\delta_{ij}d_{ab}$$

With these formulae in hand, we have the following computation:
\begin{eqnarray*}
(dw)_{ia,jb}
&=&\sum_kd_{ik}w_{ka,jb}\\
&=&\sum_kd_{ik}u_{kj}^{(a)}v_{ab}\\
&=&(du^{(a)})_{ij}v_{ab}
\end{eqnarray*}

On the other hand, we have as well the following computation:
\begin{eqnarray*}
(wd)_{ia,jb}
&=&\sum_kw_{ia,kb}d_{kj}\\
&=&\sum_ku_{ik}^{(a)}v_{ab}d_{kj}\\
&=&(u^{(a)}d)_{ij}v_{ab}
\end{eqnarray*}

Thus we have $[w,d]=0$, and from this we obtain:
$$G^+(X)\wr_*S_k^+\subset G^+(kX)$$

Regarding now the reverse inclusion, which requires $X$ to be connected, this follows by doing some matrix analysis, by using the commutation with $u$. To be more precise, let us denote by $w$ the fundamental corepresentation of $G^+(kX)$, and set:
$$u_{ij}^{(a)}=\sum_bw_{ia,jb}\quad,\quad 
v_{ab}=\sum_iv_{ab}$$

It is then routine to check, by using the fact that $X$ is indeed connected, that we have here magic unitaries, as in the definition of the free wreath products. Thus we obtain the reverse inclusion, that we were looking for, namely:
$$G^+(kX)\subset G^+(X)\wr_*S_k^+$$

To be more precise, the key ingredient is the fact that when $X$ is connected, the $*$-algebra generated by $d_X$ contains a matrix having nonzero entries.
\end{proof}

We are led in this way to the following result, from \cite{bbc}:

\index{hyperoctahedral group}
\index{hyperoctahedral quantum group}

\begin{theorem}
Consider the graph consisting of $N$ segments.
\begin{enumerate}
\item Its symmetry group is the hyperoctahedral group $H_N=\mathbb Z_2\wr S_N$.

\item Its quantum symmetry group is the quantum group $H_N^+=\mathbb Z_2\wr_*S_N^+$.
\end{enumerate}
\end{theorem}

\begin{proof}
This comes from the above results, as follows:

\medskip

(1) This is clear from definitions, with the remark that the relation with the formula $H_N=G(\square_N)$ comes by viewing the $N$ segments as being the $[-1,1]$ segments on each of the $N$ coordinate axes of $\mathbb R^N$. Indeed, a symmetry of the $N$-cube is the same as a symmetry of the $N$ segments, and so we obtain, as desired: 
$$G(\square_N)=\mathbb Z_2\wr S_N$$

(2) This follows from Theorem 10.11 above, applied to the segment graph. Observe also that (2) implies (1), by taking the classical version.
\end{proof}

Now back to the square, we have $G^+(\square)=H_2^+$, and our claim is that this is the ``good'' and final formula. In order to prove this, we must work out the easiness theory for $H_N,H_N^+$, and prove that $H_N\to H_N^+$ is an easy quantum group liberation. 

\bigskip

Following \cite{bbc}, we first have the following result:

\index{cubic unitary}
\index{sudoku unitary}

\begin{proposition}
The algebra $C(H_N^+)$ can be presented in two ways, as follows:
\begin{enumerate}
\item As the universal algebra generated by the entries of a $2N\times 2N$ magic unitary having the following ``sudoku'' pattern, with $a,b$ being square matrices:
$$w=\begin{pmatrix}a&b\\b&a\end{pmatrix}$$

\item As the universal algebra generated by the entries of a $N\times N$ orthogonal matrix which is ``cubic'', in the sense that, for any $j\neq k$:
$$u_{ij}u_{ik}=u_{ji}u_{ki}=0$$
\end{enumerate}
As for $C(H_N)$, this has similar presentations, among the commutative algebras.
\end{proposition}

\begin{proof}
Here the first assertion follows from Theorem 10.12, via Proposition 10.10, and the last assertion is clear as well, because $C(H_N)$ is the abelianization of $C(H_N^+)$. Thus, we are left with proving that the algebras $A_s,A_c$ coming from (1,2) coincide.

We construct first the arrow $A_c\to A_s$. The elements $a_{ij},b_{ij}$ being self-adjoint, their differences are self-adjoint as well. Thus $a-b$ is a matrix of self-adjoint elements. We have the following formula for the products on the columns of $a-b$:
\begin{eqnarray*}
(a-b)_{ik}(a-b)_{jk}
&=&a_{ik}a_{jk}-a_{ik}b_{jk}-b_{ik}a_{jk}+b_{ik}b_{jk}\\
&=&\begin{cases}
0&{\rm for }\ i\neq j\\
a_{ik}+b_{ik}&{\rm for}\ i=j
\end{cases}
\end{eqnarray*}

In the $i=j$ case the elements $a_{ik}+b_{ik}$ sum up to $1$, so the columns of $a-b$ are orthogonal. A similar computation works for rows, so $a-b$ is orthogonal.

Now by using the $i\neq j$ computation, along with its row analogue, we conclude that $a-b$ is cubic. Thus we can define a morphism $A_c\to A_s$ by the following formula:
$$\varphi(u_{ij})=a_{ij}-b_{ij}$$

We construct now the inverse morphism. Consider the following elements:
$$\alpha_{ij}=\frac{u_{ij}^2+u_{ij}}{2}\quad,\quad 
\beta_{ij}=\frac{u_{ij}^2-u_{ij}}{2}$$

These are projections, and the following matrix is a sudoku unitary:
$$M=\begin{pmatrix}
(\alpha_{ij})&(\beta_{ij})\\
(\beta_{ij})&(\alpha_{ij})
\end{pmatrix}$$

Thus we can define a morphism $A_s\to A_c$ by the following formulae:
$$\psi(a_{ij})=\frac{u_{ij}^2+u_{ij}}{2}\quad,\quad 
\psi(b_{ij})=\frac{u_{ij}^2-u_{ij}}{2}$$

We check now the fact that $\psi,\varphi$ are indeed inverse
morphisms:
\begin{eqnarray*}
\psi\varphi(u_{ij})
&=&\psi(a_{ij}-b_{ij})\\
&=&\frac{u_{ij}^2+u_{ij}}{2}-\frac{u_{ij}^2-u_{ij}}{2}\\
&=&u_{ij}
\end{eqnarray*}

As for the other composition, we have the following computation:
\begin{eqnarray*}
\varphi\psi(a_{ij})
&=&\varphi\left(\frac{u_{ij}^2+u_{ij}}{2}\right)\\
&=&\frac{(a_{ij}-b_{ij})^2+(a_{ij}-b_{ij})}{2}\\
&=&a_{ij}
\end{eqnarray*}

A similar computation gives $\varphi\psi(b_{ij})=b_{ij}$, which completes the proof.
\end{proof}

We can now work out the easiness property of $H_N,H_N^+$, with respect to the cubic representations, and we are led to the following result, which is fully satisfactory:

\index{Brauer theorem}
\index{hyperoctahedral group}
\index{hyperoctahedral quantum group}

\begin{theorem}
The quantum groups $H_N,H_N^+$ are both easy, as follows:
\begin{enumerate}
\item $H_N$ corresponds to the category $P_{even}$.

\item $H_N^+$ corresponds to the category $NC_{even}$.
\end{enumerate}
\end{theorem}

\begin{proof}
This follows indeed from the fact that the cubic relations are implemented by the one-block partition in $P(2,2)$, which generates $NC_{even}$. See \cite{bbc}.
\end{proof}

\section*{10c. Complex reflections}

Following \cite{bb+}, the basic algebraic results regarding $S_N,S_N^+$ and $H_N,H_N^+$ appear in fact as the $s=1,2$ particular cases of the following result:

\index{complex reflection}
\index{easiness}

\begin{theorem}
The complex reflection groups $H_N^s=\mathbb Z_s\wr S_N$ and their free analogues $H_N^{s+}=\mathbb Z_s\wr_*S_N^+$, defined for any $s\in\mathbb N$, have the following properties:
\begin{enumerate}
\item They have $N$-dimensional coordinates $u=(u_{ij})$, subject to the relations:
$$u_{ij}u_{ij}^*=u_{ij}^*u_{ij}$$
$$p_{ij}=u_{ij}u_{ij}^*={\rm magic}$$
$$u_{ij}^s=p_{ij}$$

\item They are easy, the corresponding categories $P^s\subset P,NC^s\subset NC$ being given by the fact that we have $\#\circ-\#\bullet=0(s)$, as a weighted sum, in each block.
\end{enumerate}
\end{theorem}

\begin{proof}
We already know that the results hold at $s=1,2$, and the proof in general is similar. With respect to the above proof at $s=2$, the situation is as follows:

\medskip

(1) Observe first that the result holds at $s=1$, where we obtain the magic condition, and at $s=2$ as well, where we obtain something equivalent to the cubic condition. In general, this follows from a $\mathbb Z_s$-analogue of Proposition 10.13. See \cite{bv1}.

\medskip

(2) Once again, the result holds at $s=1$, trivially, and at $s=2$ as well, where our condition is equivalent to $\#\circ+\#\bullet=0(2)$, in each block. In general, this follows as in the proof of Theorem 10.14, by using the one-block partition in $P(s,s)$. See \cite{bb+}.
\end{proof}

We have as well a result at $s=\infty$, which is of particular interest, as follows:

\begin{theorem}
The pure complex reflection groups $K_N=\mathbb T\wr S_N$ and their free analogues $K_N^+=\mathbb T\wr_*S_N^+$ have the following properties:
\begin{enumerate}
\item They have $N$-dimensional coordinates $u=(u_{ij})$, subject to the relations:
$$u_{ij}u_{ij}^*=u_{ij}^*u_{ij}$$
$$p_{ij}=u_{ij}u_{ij}^*={\rm magic}$$

\item They are easy, the corresponding categories $\mathcal P_{even}\subset P,\mathcal{NC}_{even}\subset NC$ being given by the fact that we have $\#\circ=\#\bullet$, as a weighted equality, in each block.
\end{enumerate}
\end{theorem}

\begin{proof}
The assertions here appear as an $s=\infty$ extension of (1,2) in Theorem 10.15 above, and their proof can be obtained along the same lines, as follows:

\medskip

(1) This follows indeed by working out a $\mathbb T$-analogue of the computations in the proof of Proposition 10.13 above. Again, for details we refer here to \cite{bv1}.

\medskip

(2) This result appears too as a $s=\infty$ extension of the results that we already have, and for details here, we refer once again to \cite{bb+}.
\end{proof}

We can now focus on $H_N,H_N^+,K_N,K_N^+$, with the idea in mind of completing the continuous quantum group picture from chapter 7. Before doing this, however, we have two more quantum groups to be introduced, namely $H_N^*,K_N^*$. We have here:

\begin{theorem}
We have quantum groups $H_N^*,K_N^*$, both easy, as follows,
\begin{enumerate}
\item $H_N^*=H_N^+\cap O_N^*$, corresponding to the category $P_{even}^*$, 

\item $K_N^*=K_N^+\cap U_N^*$, corresponding to the category $\mathcal P_{even}^*$,
\end{enumerate}
with the symbol $*$ standing for the fact that the corresponding partitions, when relabelled clockwise $\circ\bullet\circ\bullet\ldots$, must contain the same number of $\circ,\bullet$, in each block.
\end{theorem}

\begin{proof}
This is standard, from the results that we already have, regarding the various quantum groups involved, because the intersection operations at the quantum group level correspond to generation operations, at the category of partitions level.
\end{proof}

We can now complete the ``continuous'' picture from chapter 7 above, as follows:

\begin{theorem}
The basic orthogonal and unitary quantum groups are related to the basic real and complex quantum reflection groups as follows,
$$\xymatrix@R=15mm@C=17mm{
U_N\ar[r]&U_N^*\ar[r]&U_N^+\\
O_N\ar[r]\ar[u]&O_N^*\ar[r]\ar[u]&O_N^+\ar[u]}
\ \ \ \ \ \xymatrix@R=8mm@C=5mm{\\ \leftrightarrow&\\&\\}\ \ 
\xymatrix@R=15mm@C=17mm{
K_N\ar[r]&K_N^*\ar[r]&K_N^+\\
H_N\ar[r]\ar[u]&H_N^*\ar[r]\ar[u]&H_N^+\ar[u]}$$
the connecting operations $U\leftrightarrow K$ being given by $K=U\cap K_N^+$ and $U=\{K,O_N\}$.
\end{theorem}

\begin{proof}
According to the general results in chapter 7, in terms of categories of partitions, the operations introduced in the statement reformulate as follows:
$$D_K=<D_U,\mathcal{NC}_{even}>\quad,\quad 
D_U=D_K\cap P_2$$

On the other hand, by putting together the various easiness results that we have, the categories of partitions for the quantum groups in the statement are as follows:
$$\xymatrix@R=16mm@C=15mm{
\mathcal P_2\ar[d]&\mathcal P_2^*\ar[l]\ar[d]&\mathcal{NC}_2\ar[l]\ar[d]\\
P_2&P_2^*\ar[l]&NC_2\ar[l]}
\ \ \ \ \ \xymatrix@R=8mm@C=5mm{\\ :&\\&\\}\ \ 
\xymatrix@R=16mm@C=12mm{
\mathcal P_{even}\ar[d]&\mathcal P_{even}^*\ar[l]\ar[d]&\mathcal{NC}_{even}\ar[l]\ar[d]\\
P_{even}&P_{even}^*\ar[l]&NC_{even}\ar[l]}$$

It is elementary to check that these categories are related by the above intersection and generation operations, and we conclude that the correspondence holds indeed.
\end{proof}

Our purpose now will be that of showing that a twisted analogue of the above result holds. It is convenient to include in our discussion two more quantum groups, coming from \cite{ez1}, \cite{rwe} and denoted $H_N^{[\infty]},K_N^{[\infty]}$, which are constructed as follows:

\index{intermediate liberation}

\begin{theorem}
We have intermediate liberations $H_N^{[\infty]},K_N^{[\infty]}$ as follows, constructed by using the relations $\alpha\beta\gamma=0$, for any $a\neq c$ on the same row or column of $u$,
$$\xymatrix@R=15mm@C=17mm{
K_N\ar[r]&K_N^*\ar[r]&K_N^{[\infty]}\ar[r]&K_N^+\\
H_N\ar[r]\ar[u]&H_N^*\ar[r]\ar[u]&H_N^{[\infty]}\ar[r]\ar[u]&H_N^+\ar[u]}$$
with the convention $\alpha=a,a^*$, and so on. These quantum groups are easy, the corresponding categories $P_{even}^{[\infty]}\subset P_{even}$ and $\mathcal P_{even}^{[\infty]}\subset\mathcal P_{even}$ being generated by $\eta=\ker(^{iij}_{jii})$.
\end{theorem}

\begin{proof}
This is routine, by using the fact that the relations $\alpha\beta\gamma=0$ in the statement are equivalent to the following condition, with $|k|=3$:
$$\eta\in End(u^{\otimes k})$$ 

For further details on these quantum groups, we refer to \cite{ez1}, \cite{rwe}.
\end{proof}

In order to discuss the twisting, we will need the following technical result:

\begin{proposition}
We have the following equalities,
\begin{eqnarray*}
P_{even}^*&=&\left\{\pi\in P_{even}\Big|\varepsilon(\tau)=1,\forall\tau\leq\pi,|\tau|=2\right\}
\\
P_{even}^{[\infty]}&=&\left\{\pi\in P_{even}\Big|\sigma\in P_{even}^*,\forall\sigma\subset\pi\right\}\\
P_{even}^{[\infty]}&=&\left\{\pi\in P_{even}\Big|\varepsilon(\tau)=1,\forall\tau\leq\pi\right\}
\end{eqnarray*}
where $\varepsilon:P_{even}\to\{\pm1\}$ is the signature of even permutations.
\end{proposition}

\begin{proof}
This is routine combinatorics, the idea being as follows:

\medskip

(1) Given $\pi\in P_{even}$, we have $\tau\leq\pi,|\tau|=2$ precisely when $\tau=\pi^\beta$ is the partition obtained from $\pi$ by merging all the legs of a certain subpartition $\beta\subset\pi$, and by merging as well all the other blocks. Now observe that $\pi^\beta$ does not depend on $\pi$, but only on $\beta$, and that the number of switches required for making $\pi^\beta$ noncrossing is $c=N_\bullet-N_\circ$ modulo 2, where $N_\bullet/N_\circ$ is the number of black/white legs of $\beta$, when labelling the legs of $\pi$ counterclockwise $\circ\bullet\circ\bullet\ldots$ Thus $\varepsilon(\pi^\beta)=1$ holds precisely when $\beta\in\pi$ has the same number of black and white legs, and this gives the result.

\medskip

(2) This simply follows from the equality $P_{even}^{[\infty]}=<\eta>$ coming from Theorem 10.19, by computing $<\eta>$, and for the complete proof here we refer to Raum-Weber \cite{rwe}.

\medskip

(3) We use here the fact, also from \cite{rwe}, that the relations $g_ig_ig_j=g_jg_ig_i$ are trivially satisfied for real reflections. This leads to the following conclusion:
$$P_{even}^{[\infty]}(k,l)=\left\{\ker\begin{pmatrix}i_1&\ldots&i_k\\ j_1&\ldots&j_l\end{pmatrix}\Big|g_{i_1}\ldots g_{i_k}=g_{j_1}\ldots g_{j_l}\ {\rm inside}\ \mathbb Z_2^{*N}\right\}$$

In other words, the partitions in $P_{even}^{[\infty]}$ are those describing the relations between free variables, subject to the conditions $g_i^2=1$. We conclude that $P_{even}^{[\infty]}$ appears from $NC_{even}$ by ``inflating blocks'', in the sense that each $\pi\in P_{even}^{[\infty]}$ can be transformed into a partition $\pi'\in NC_{even}$ by deleting pairs of consecutive legs, belonging to the same block. 

Now since this inflation operation leaves invariant modulo 2 the number $c\in\mathbb N$ of switches in the definition of the signature, it leaves invariant the signature $\varepsilon=(-1)^c$ itself, and we obtain in this way the inclusion ``$\subset$'' in the statement. 

Conversely, given $\pi\in P_{even}$ satisfying $\varepsilon(\tau)=1$, $\forall\tau\leq\pi$, our claim is that:
$$\rho\leq\sigma\subset\pi,|\rho|=2\implies\varepsilon(\rho)=1$$

Indeed, let us denote by $\alpha,\beta$ the two blocks of $\rho$, and by $\gamma$ the remaining blocks of $\pi$, merged altogether. We know that the partitions $\tau_1=(\alpha\wedge\gamma,\beta)$, $\tau_2=(\beta\wedge\gamma,\alpha)$, $\tau_3=(\alpha,\beta,\gamma)$ are all even. On the other hand, putting these partitions in noncrossing form requires respectively $s+t,s'+t,s+s'+t$ switches, where $t$ is the number of switches needed for putting $\rho=(\alpha,\beta)$ in noncrossing form. Thus $t$ is even, and we are done.

With the above claim in hand, we conclude, by using the second equality in the statement, that we have $\sigma\in P_{even}^*$. Thus $\pi\in P_{even}^{[\infty]}$, which ends the proof of ``$\supset$''.
\end{proof}

With the above result in hand, we can now prove:

\index{twisting}

\begin{theorem}
We have the following results:
\begin{enumerate}
\item The quantum groups from Theorem 10.19 are equal to their own twists.

\item With input coming from this, a twisted version of Theorem 10.18 holds.
\end{enumerate}
\end{theorem}

\begin{proof}
This result basically comes from the results that we have.

\medskip

(1) In the real case, the verifications are as follows:

\medskip

-- $H_N^+$. We know from chapter 7 above that for $\pi\in NC_{even}$ we have $\bar{T}_\pi=T_\pi$, and since we are in the situation $D\subset NC_{even}$, the definitions of $G,\bar{G}$ coincide.

\medskip

-- $H_N^{[\infty]}$. Here we can use the same argument as in (1), based this time on the description of $P_{even}^{[\infty]}$ involving the signature found in Proposition 10.20.

\medskip

-- $H_N^*$. We have $H_N^*=H_N^{[\infty]}\cap O_N^*$, so $\bar{H}_N^*\subset H_N^{[\infty]}$ is the subgroup obtained via the defining relations for $\bar{O}_N^*$. But all the $abc=-cba$ relations defining $\bar{H}_N^*$ are automatic, of type $0=0$, and it follows that $\bar{H}_N^*\subset H_N^{[\infty]}$ is the subgroup obtained via the relations $abc=cba$, for any $a,b,c\in\{u_{ij}\}$. Thus we have $\bar{H}_N^*=H_N^{[\infty]}\cap O_N^*=H_N^*$, as claimed.

\medskip

-- $H_N$. We have $H_N=H_N^*\cap O_N$, and by functoriality, $\bar{H}_N=\bar{H}_N^*\cap\bar{O}_N=H_N^*\cap\bar{O}_N$. But this latter intersection is easily seen to be equal to $H_N$, as claimed.

\medskip

In the complex case the proof is similar, by using the same arguments.

\medskip

(2) This can be proved by proceeding as in the proof of Theorem 10.18 above, with of course some care when formulating the result.
\end{proof}

\section*{10d. Bessel laws}

Let us go back to $H_N^+,K_N^+$, or rather to the whole series $H_N^{s+}$, with $s\in\{1,2,\ldots,\infty\}$ and work out the fusion rules, and probabilistic aspects. We first have:

\begin{proposition}
The algebra $C(H_N^{s+})$ has a family of $N$-dimensional corepresentations $\{u_k|k\in\mathbb Z\}$, satisfying the following conditions:
\begin{enumerate}
\item $u_k=(u_{ij}^k)$ for any $k\geq 0$.

\item $u_k=u_{k+s}$ for any $k\in\mathbb Z$.

\item $\bar{u}_k=u_{-k}$ for any $k\in\mathbb Z$.
\end{enumerate}
\end{proposition}

\begin{proof}
Our claim is that all the above holds, with $u_k=(u_{ij}^k)$. Indeed, all these results follow from the definition of $H_N^{s+}$. See \cite{bv1}.
\end{proof}

Next, we have the following result, also from \cite{bv1}:

\begin{theorem}
With the convention $u_{i_1\ldots i_k}=u_{i_1}\otimes\ldots\otimes u_{i_k}$, for any $i_1,\ldots,i_k\in\mathbb Z$, we have the following equality of linear spaces,
$$Hom(u_{i_1\ldots i_k},u_{j_1\ldots j_l})=span\left\{T_p\Big|p\in NC_s(i_1\ldots i_k,j_1\ldots j_l)\right\}$$
where the set on the right consists of elements of $NC(k,l)$ having the property that in each block, the sum of $i$ indices equals the sum of $j$ indices, modulo $s$.
\end{theorem}

\begin{proof}
This result is from \cite{bv1}, the idea of the proof being as follows:

\medskip

(1) Our first claim is that, in order to prove $\supset$, we may restrict attention to the case $k=0$. This follows indeed from the Frobenius duality isomorphism.

\medskip

(2) Our second claim is that, in order to prove $\supset$ in the case $k=0$, we may restrict attention to the one-block partitions. Indeed, this follows once again from a standard trick. Consider the following disjoint union:
$$NC_s=\bigcup_{k=0}^\infty\bigcup_{i_1\ldots i_k} NC_s(0,i_1\ldots i_k)$$

This is a set of labeled partitions, having property that each $p\in NC_s$ is noncrossing, and that for $p\in NC_s$, any block of $p$ is in $NC_s$. But it is well-known that under these assumptions, the global algebraic properties of $NC_s$ can be checked on blocks.

\medskip

(3) Proof of $\supset$. According to the above considerations, we just have to prove that the vector associated to the one-block partition in $NC(l)$ is fixed by $u_{j_1\ldots j_l}$, when:
$$s|j_1+\ldots+j_l$$

Consider the standard generators $e_{ab}\in M_N(\mathbb C)$, acting on the basis vectors by $e_{ab}(e_c)=\delta_{bc}e_a$. The corepresentation $u_{j_1\ldots j_l}$ is given by the following formula:
$$u_{j_1\ldots j_l}=\sum_{a_1\ldots a_l}\sum_{b_1\ldots b_l}u_{a_1b_1}^{j_1}\ldots u_{a_lb_l}^{j_l}\otimes e_{a_1b_1}\otimes\ldots\otimes e_{a_lb_l}$$

As for the vector associated to the one-block partition, this is $\xi_l=\sum_be_b^{\otimes l}$. By using now several times the relations in Proposition 10.22, we obtain, as claimed: 
\begin{eqnarray*}
u_{j_1\ldots j_l}(1\otimes\xi_l)
&=&\sum_{a_1\ldots a_l}\sum_bu_{a_1b}^{j_1}\ldots u_{a_lb}^{j_l}\otimes e_{a_1}\otimes\ldots\otimes e_{a_l}\\
&=&\sum_{ab}u_{ab}^{j_1+\ldots+j_l}\otimes e_a^{\otimes l}\\
&=&1\otimes\xi_l
\end{eqnarray*}

(4) Proof of $\subset$. The spaces on the right in the statement form a Tannakian category in the sense of Woronowicz \cite{wo2}, so they correspond to a certain Woronowicz algebra $A$, which is by definition the maximal model for the Tannakian category. In other words, $A$ comes with a family of corepresentations $\{v_i\}$, such that:
$$Hom(v_{i_1\ldots i_k},v_{j_1\ldots j_l})={\rm span}\left\{T_p\Big|p\in NC_s(i_1\ldots i_k,j_1\ldots j_l)\right\}$$

On the other hand, the inclusion $\supset$ that we just proved shows that $C(H_N^{s+})$ is a model for the category. Thus we have a quotient map $A\to C(H_N^{s+})$, mapping $v_i\to u_i$. But this latter map can be shown to be an isomorphism, by suitably adapting the proof from the $s=1$ case, for the quantum permutation group $S_N^+$. See \cite{bb+}, \cite{bv1}.
\end{proof}

Still following \cite{bv1}, we have the following result:

\index{fusion rules}
\index{quantum reflection group}

\begin{theorem}
Let $F=<\mathbb Z_s>$ be the monoid formed by the words over $\mathbb Z_s$, with involution $(i_1\ldots i_k)^-=(-i_k)\ldots(-i_1)$, and with fusion product given by:
$$(i_1\ldots i_k)\cdot (j_1\ldots j_l)=i_1\ldots i_{k-1}(i_k+j_1)j_2\ldots j_l$$
The irreducible representations of $H_N^{s+}$ can then be labeled $r_x$ with $x\in F$, such that
$$r_x\otimes r_y=\sum_{x=vz,y=\bar{z}w}r_{vw}+r_{v\cdot w}$$
and $\bar{r}_x=r_{\bar{x}}$, and such that $r_i=u_i-\delta_{i0}1$ for any $i\in\mathbb Z_s$.
\end{theorem}

\begin{proof}
This basically follows from Theorem 10.23, the idea being as follows:

\medskip

(1) Consider the monoid $A=\{a_x|x\in F\}$, with multiplication $a_xa_y=a_{xy}$. We endow $\mathbb NA$ with fusion rules as in the statement, namely:
$$a_x\otimes a_y=\sum_{x=vz,y=\bar{z}w}a_{vw}+a_{v\cdot w}$$

(2) The fusion rules on $\mathbb ZA$ can be then uniquely described by conversion formulae as follows, with $C$ being positive integers, and $D$ being integers:
$$a_{i_1}\otimes\ldots\otimes a_{i_k}=\sum_l\sum_{j_1\ldots j_l}C_{i_1\ldots i_k}^{j_1\ldots j_l}a_{j_1\ldots j_l}$$
$$a_{i_1\ldots i_k}=\sum_l\sum_{j_1\ldots j_l}D_{i_1\ldots i_k}^{j_1\ldots j_l}a_{j_1}\otimes\ldots\otimes a_{j_l}$$

(3) Now observe that there is a unique morphism of rings $\Phi:\mathbb ZA\to R$, such that $\Phi(a_i)=r_i$ for any $i$. Indeed, consider the following elements of $R$:
$$r_{i_1\ldots i_k}=\sum_l\sum_{j_1\ldots j_l}D_{i_1\ldots i_k}^{j_1\ldots j_l}r_{j_1}\otimes\ldots\otimes r_{j_l}$$

In case we have a morphism as claimed, we must have $\Phi(a_x)=r_x$ for any $x\in F$. Thus our morphism is uniquely determined on  $A$, so it is uniquely determined on $\mathbb ZA$.

\medskip

(4) Our claim is that $\Phi$ commutes with the linear forms $x\to\#(1\in x)$. Indeed, by linearity we just have to check the following equality:
$$\#(1\in a_{i_1}\otimes\ldots\otimes a_{i_k})=\#(1\in r_{i_1}\otimes\ldots\otimes r_{i_k})$$

Now remember that the elements $r_i$ are defined as $r_i=u_i-\delta_{i0}1$. So, consider the elements $c_i=a_i+\delta_{i0}1$. Since the operations $r_i\to u_i$ and $a_i\to c_i$ are of the same nature, by linearity the above formula is equivalent to:
$$\#(1\in c_{i_1}\otimes\ldots\otimes c_{i_k})=\#(1\in u_{i_1}\otimes\ldots\otimes u_{i_k})$$

Now by using Theorem 10.23, what we have to prove is:
$$\#(1\in c_{i_1}\otimes\ldots\otimes c_{i_k})=\#NC_s(i_1\ldots i_k)$$

In order to prove this formula, consider the product on the left:
$$P=(a_{i_1}+\delta_{i_10}1)\otimes(a_{i_2}+\delta_{i_20}1)\otimes\ldots\otimes (a_{i_k}+\delta_{i_k0}1)$$

But this quantity can be computed by using the fusion rules on $A$, and the combinatorics leads to the conclusion that we have $\#(1\in P)=\# NC_s(i_1\ldots i_k)$, as claimed.

\medskip

(5) Our claim now is that $\Phi$ is injective. Indeed, this follows from the result in the previous step, by using a standard positivity argument.

\medskip

(6) Our claim is that we have $\Phi(A)\subset R_{irr}$. This is the same as saying that $r_x\in R_{irr}$ for any $x\in F$, and we will prove it by recurrence. Assume that the assertion is true for all the words of length $<k$, and consider a length $k$ word, $x=i_1\ldots i_k$. We have:
$$a_{i_1}\otimes a_{i_2\ldots i_k}=a_x+a_{i_1+i_2,i_3\ldots i_k}+\delta_{i_1+i_2,0}a_{i_3\ldots i_k}$$

By applying $\Phi$ to this decomposition, we obtain:
$$r_{i_1}\otimes r_{i_2\ldots i_k}=r_x+r_{i_1+i_2,i_3\ldots i_k}+\delta_{i_1+i_2,0}r_{i_3\ldots i_k}$$

We have the following computation, which is valid for $y=i_1+i_2,i_3\ldots i_k$, as well as for $y=i_3\ldots i_k$ in the case $i_1+i_2=0$:
\begin{eqnarray*}
\#(r_y\in r_{i_1}\otimes r_{i_2\ldots i_k})
&=&\#(1,r_{\bar{y}}\otimes r_{i_1}\otimes r_{i_2\ldots i_k})\\
&=&\#(1,a_{\bar{y}}\otimes a_{i_1}\otimes a_{i_2\ldots i_k})\\
&=&\#(a_y\in a_{i_1}\otimes a_{i_2\ldots i_k})\\
&=&1  
\end{eqnarray*}

Moreover, we know from the previous step that we have $r_{i_1+i_2,i_3\ldots i_k}\neq r_{i_3\ldots i_k}$, so we conclude that the following formula defines an element of $R^+$:
$$\alpha=r_{i_1}\otimes r_{i_2\ldots i_k}-r_{i_1+i_2,i_3\ldots i_k}-\delta_{i_1+i_2,0}r_{i_3\ldots i_k}$$

On the other hand, we have $\alpha=r_x$, so we conclude that we have $r_x\in R^+$. Finally, the irreducibility of $r_x$ follows from $\#(1\in r_x\otimes\bar{r}_x)=1$.

\medskip

(7) Summarizing, we have constructed an injective ring morphism $\Phi:\mathbb ZA\to R$, having the property $\Phi(A)\subset R_{irr}$. The remaining fact to be proved, namely that we have $\Phi(A)=R_{irr}$, is something of abstract nature, which is clear. Thus, we are done.
\end{proof}

Regarding the probabilistic aspects, we will need some general theory. We have the following definition, extending the Poisson limit theory from chapter 9 above:

\index{compound Poisson limit}
\index{compond PLT}

\begin{definition}
Associated to any compactly supported positive measure $\rho$, not necessarily of mass $1$, are the probability measures
$$p_\rho=\lim_{n\to\infty}\left(\left(1-\frac{c}{n}\right)\delta_0+\frac{1}{n}\rho\right)^{*n}$$
$$\pi_\rho=\lim_{n\to\infty}\left(\left(1-\frac{c}{n}\right)\delta_0+\frac{1}{n}\rho\right)^{\boxplus n}$$
where $c=mass(\rho)$, called compound Poisson and compound free Poisson laws.
\end{definition}

In what follows we will be interested in the case where $\rho$ is discrete, as is for instance the case for $\rho=t\delta_1$ with $t>0$, which produces the Poisson and free Poisson laws. The following result allows one to detect compound Poisson/free Poisson laws:

\index{Fourier transform}
\index{R-transform}

\begin{proposition}
For a discrete measure, written as 
$$\rho=\sum_{i=1}^sc_i\delta_{z_i}$$
with $c_i>0$ and $z_i\in\mathbb R$, we have the formulae
$$F_{p_\rho}(y)=\exp\left(\sum_{i=1}^sc_i(e^{iyz_i}-1)\right)$$
$$R_{\pi_\rho}(y)=\sum_{i=1}^s\frac{c_iz_i}{1-yz_i}$$
where $F,R$ are respectively the Fourier transform, and Voiculescu's $R$-transform.
\end{proposition}

\begin{proof}
Let $\mu_n$ be the measure appearing in Definition 10.25, under the convolution signs. In the classical case, we have the following computation:
\begin{eqnarray*}
&&F_{\mu_n}(y)=\left(1-\frac{c}{n}\right)+\frac{1}{n}\sum_{i=1}^sc_ie^{iyz_i}\\
&\implies&F_{\mu_n^{*n}}(y)=\left(\left(1-\frac{c}{n}\right)+\frac{1}{n}\sum_{i=1}^sc_ie^{iyz_i}\right)^n\\
&\implies&F_{p_\rho}(y)=\exp\left(\sum_{i=1}^sc_i(e^{iyz_i}-1)\right)
\end{eqnarray*}

In the free case now, we use a similar method. The Cauchy transform of $\mu_n$ is:
$$G_{\mu_n}(\xi)=\left(1-\frac{c}{n}\right)\frac{1}{\xi}+\frac{1}{n}\sum_{i=1}^s\frac{c_i}{\xi-z_i}$$

Consider now the $R$-transform of the measure $\mu_n^{\boxplus n}$, which is given by:
$$R_{\mu_n^{\boxplus n}}(y)=nR_{\mu_n}(y)$$

The above formula of $G_{\mu_n}$ shows that the equation for $R=R_{\mu_n^{\boxplus n}}$ is as follows:
\begin{eqnarray*}
&&\left(1-\frac{c}{n}\right)\frac{1}{y^{-1}+R/n}+\frac{1}{n}\sum_{i=1}^s\frac{c_i}{y^{-1}+R/n-z_i}=y\\
&\implies&\left(1-\frac{c}{n}\right)\frac{1}{1+yR/n}+\frac{1}{n}\sum_{i=1}^s\frac{c_i}{1+yR/n-yz_i}=1
\end{eqnarray*}

Now multiplying by $n$, rearranging the terms, and letting $n\to\infty$, we get:
\begin{eqnarray*}
&&\frac{c+yR}{1+yR/n}=\sum_{i=1}^s\frac{c_i}{1+yR/n-yz_i}\\
&\implies&c+yR_{\pi_\rho}(y)=\sum_{i=1}^s\frac{c_i}{1-yz_i}\\
&\implies&R_{\pi_\rho}(y)=\sum_{i=1}^s\frac{c_iz_i}{1-yz_i}
\end{eqnarray*}

This finishes the proof in the free case, and we are done.
\end{proof}

We also have the following result, providing an alternative to Definition 10.25, and which is an extension of the classical and free Poisson limiting theorems (PLT) that we know from chapter 9, called Compound Poisson Limiting Theorem (CPLT):

\begin{theorem}
For a discrete measure, written as
$$\rho=\sum_{i=1}^sc_i\delta_{z_i}$$
with $c_i>0$ and $z_i\in\mathbb R$, we have the formulae
$$p_\rho/\pi_\rho={\rm law}\left(\sum_{i=1}^sz_i\alpha_i\right)$$
where the variables $\alpha_i$ are Poisson/free Poisson$(c_i)$, independent/free.
\end{theorem}

\begin{proof}
Let $\alpha$ be the sum of Poisson/free Poisson variables in the statement. We will show that the Fourier/$R$-transform of $\alpha$ is given by the formulae in Proposition 10.26. Indeed, by using some well-known Fourier transform formulae, we have:
\begin{eqnarray*}
F_{\alpha_i}(y)=\exp(c_i(e^{iy}-1))
&\implies&F_{z_i\alpha_i}(y)=\exp(c_i(e^{iyz_i}-1))\\
&\implies&F_\alpha(y)=\exp\left(\sum_{i=1}^sc_i(e^{iyz_i}-1)\right)
\end{eqnarray*}

Also, by using some well-known $R$-transform formulae, we have:
\begin{eqnarray*}
R_{\alpha_i}(y)=\frac{c_i}{1-y}
&\implies&R_{z_i\alpha_i}(y)=\frac{c_iz_i}{1-yz_i}\\
&\implies&R_\alpha(y)=\sum_{i=1}^s\frac{c_iz_i}{1-yz_i}
\end{eqnarray*}

Thus we have indeed the same formulae as those in Proposition 10.26.
\end{proof}

We can go back now to quantum reflection groups, and we have:

\index{truncated character}
\index{Bercovici-Pata bijection}

\begin{theorem}
The asymptotic laws of truncated characters are as follows, where $\varepsilon_s$ with $s\in\{1,2,\ldots,\infty\}$ is the uniform measure on the $s$-th roots of unity:
\begin{enumerate}
\item For $H_N^s$ we obtain the compound Poisson law $b_t^s=p_{t\varepsilon_s}$.

\item For $H_N^{s+}$ we obtain the compound free Poisson law $\beta_t^s=\pi_{t\varepsilon_s}$.
\end{enumerate}
These measures are in Bercovici-Pata bijection.
\end{theorem}

\begin{proof}
This follows from easiness, and from the Weingarten formula. To be more precise, at $t=1$ this follows by counting the partitions, and at $t\in(0,1]$ general, this follows in the usual way, for instance by using cumulants. See \cite{bb+}.
\end{proof}

\index{Bessel law}
\index{free Bessel law}
\index{free convolution}

The above measures are called Bessel and free Bessel laws. This is because at $s=2$ we have $b_t^2=e^{-t}\sum_{k=-\infty}^\infty f_k(t/2)\delta_k$, with $f_k$ being the Bessel function of the first kind:
$$f_k(t)=\sum_{p=0}^\infty \frac{t^{|k|+2p}}{(|k|+p)!p!}$$

The Bessel and free Bessel laws have particularly interesting properties at the parameter values $s=2,\infty$. So, let us record the precise statement here:

\begin{theorem}
The asymptotic laws of truncated characters are as follows:
\begin{enumerate}
\item For $H_N$ we obtain the real Bessel law $b_t=p_{t\varepsilon_2}$.

\item For $K_N$ we obtain the complex Bessel law $B_t=p_{t\varepsilon_\infty}$.

\item For $H_N^+$ we obtain the free real Bessel law $\beta_t=\pi_{t\varepsilon_2}$.

\item For $K_N^+$ we obtain the free complex Bessel law $\mathfrak B_t=\pi_{t\varepsilon_\infty}$.
\end{enumerate}
\end{theorem}

\begin{proof}
This follows indeed from Theorem 10.28 above, at $s=2,\infty$.
\end{proof}

In addition to what has been said above, there are as well some interesting results about the Bessel and free Bessel laws involving the multiplicative convolution $\times$, and the multiplicative free convolution $\boxtimes$. For details, we refer here to \cite{bb+}.

\bigskip

As a conclusion to all this, work that we did in chapter 9 and here, things in the discrete setting are often more complicated than in the continuous setting, although when restricting the attention to $H_N,H_N^+$ and $K_N,K_N^+$, everything is after all quite similar to what we knew from chapters 5-6, regarding $O_N,O_N^+$ and $U_N,U_N^+$. We will keep building in chapter 11 below, with this kind of philosophy, with the idea in mind of unifying the theory of $O_N,O_N^+$ and $U_N,U_N^+$ with the theory of $H_N,H_N^+$ and $K_N,K_N^+$.

\section*{10e. Exercises}

As before with the quantum permutations, there has been a lot of material in this section, and most of our exercises will be about what has been said above. To start with, in relation with the quantum automorphisms of the finite graphs, we have:

\begin{exercise}
Extract, from the computation of the quantum symmetry group of the $N$-cycle with $N\geq4$, a simple proof for the equality $S_3^+=S_3$.
\end{exercise}

To be more precise, that computation shows at $N=3$ that we have $S_3^+=S_3$, and the problem is that of writing down a short proof for this latter equality.

\begin{exercise}
Work out all the details regarding the easiness property of $H_N,H_N^+$, involving the categories $P_{even},NC_{even}$.
\end{exercise}

This is something that was already discussed in the above, but just briefly. The idea is to proceed a bit in the same way as we did for $S_N,S_N^+$, in chapter 9.

\begin{exercise}
Work out all the details regarding the easiness property of $H_N^s,H_N^{s+}$, involving the categories $P^s,NC^s$.
\end{exercise}

As before with $H_N,H_N^+$, the idea here is that of proceeding a bit in the same way as we did for $S_N,S_N^+$, in chapter 9. 

\begin{exercise}
Work out the structure of the complex reflection groups $H_N^s,H_N^{s+}$ at $N=2$, and at various values of the parameter $s$.
\end{exercise}

To be more precise, the problem here is that of studying the groups and quantum groups $H_2^s,H_2^{s+}$ at various values of the parameter $s$, with the various methods developed so far, and see if we have here previously known groups and quantum groups.

\begin{exercise}
Deduce the Clebsch-Gordan rules for the irreducible representations of $S_N^+$, from the general result regarding $H_N^{s+}$, taken at $s=1$.
\end{exercise}

This might seem quite trivial, but in practice, there is some work to be done here.

\begin{exercise}
Prove that at $s=2$ the Bessel law is given by
$$b_t^2=e^{-t}\sum_{k=-\infty}^\infty f_k(t/2)\delta_k$$
with $f_k$ being the Bessel function of the first kind, namely:
$$f_k(t)=\sum_{p=0}^\infty \frac{t^{|k|+2p}}{(|k|+p)!p!}$$
\end{exercise}

As mentioned above, there are many other interesting things that can be said about the Bessel and free Bessel laws, and as a final and supplementary exercise, we recommend exploring the subject, by reading some related literature.

\chapter{Classification results}

\section*{11a. Uniform groups}

We discuss in this chapter and in the next one various classification questions for the closed subgroups $G\subset U_N^+$, in the easy case, and beyond. There has been a lot of work on the subject, and our objective here will be quite modest, namely presenting a few basic such classification results, along with some discussion. The idea is as follows:

\bigskip

(1) Technically speaking, the simplest question is that of classifying the easy subgroups $G\subset O_N^+$, and the work here goes back to my paper with Speicher \cite{bsp}, with some basic results on the subject, and then to my paper with Curran-Speicher \cite{ez1}, with a number of finer results. A few years later, Raum-Weber managed to find the correct techniques for dealing with the general case, and did the full classification in \cite{rwe}.

\bigskip

(2) In the general easy unitary case $G\subset U_N^+$, it is possible, to start with, to construct all sorts of ``complexifications'' of the quantum groups from \cite{ez1}, \cite{bsp}, \cite{rwe}. However, classification remains a complicated business, due to the jungle formed by these complexification operations. Full results here include those by Tarrago-Weber \cite{twe}, regarding the case $G\subset U_N$, and those by Mang-Weber \cite{mwe}, regarding the case $U_N\subset G\subset U_N^+$.

\bigskip

(3) All this suggests adding some extra axioms, in order to deal with the general easy case $G\subset U_N^+$, and although there are several reasonable candidates here, no one really knows what the ``miracle axiom'' is, which will not exclude any interesting example, while allowing finishing the classification. But this is more of a physics question, because you need to know what ``interesting'' exactly means. And this latter question is open.

\bigskip

(4) As a somewhat original result here, inspired from the work on the noncommutative geometry and free probability applications of the compact quantum groups, the Ground Zero theorem in \cite{ba5} states that when putting altogether all the above-mentioned ``reasonable axioms'', coming on top of easiness, only 8 quantum groups survive, namely $O_N,U_N,H_N,K_N$ and $O_N^+,U_N^+,H_N^+,K_N^+$. Which is something conceptual and nice.

\bigskip

(5) In an opposite direction now, purely mathematical, the classification of all easy quantum groups $G\subset U_N$, without extra assumption, is a beautiful problem which makes sense, and which is probably as important to ``noncommutative mathematics'' as the classification of all classical Lie groups, or of all complex reflection groups, is important to classical mathematics. The modern trend here is to go towards computer usage.

\bigskip

(6) When going beyond easiness, there is a whole jungle of known results, and the questions abound. A wise goal here would be probably that of upgrading easiness, theory and classification, into a ``super-easiness'' theory, covering all the classical Lie groups, ABCDEFG, and their liberations and twists. But this is something quite difficult, volunteers needed, with even the answer in the regular, ABCD case, being not known.

\bigskip

(7) And this is not the end of the story, because we will see in chapter 13 below that, while Lie theory is definitely not available for the arbitrary closed subgroups $G\subset U_N^+$, a notion of ``maximal torus'' for such quantum groups, based on the work in \cite{bbd}, \cite{bpa}, \cite{bv2}, does exist, and can potentially lead to interesting, powerful results, and with all this having little to do with easiness, and with what has been said above. 

\bigskip

As you can see, many questions here, and finding your way through all this jungle looks like a quite complicated task. We will explain in this chapter the Ground Zero theorem from \cite{ba5}, mentioned in (4) above, which is something conceptual and nice, and whose proof heavily relies on all sorts of classification results from \cite{ez1}, \cite{bsp}, \cite{rwe},  \cite{twe} mentioned in (1-3), and with this whole chapter being an introduction to all this. Then, in chapter 12 below we will go back to (1-6), and present more results, regarded from a Ground Zero perspective. And then in chapter 13 below we will talk about (7).

\bigskip

Getting to work now, we will be interested, to start with, in easiness in a general sense. We have already met a number of easy quantum groups, as follows:

\index{easiness}
\index{Brauer theorem}

\begin{theorem}
We have the following examples of easy quantum groups:
\begin{enumerate}
\item Orthogonal quantum groups: $O_N,O_N^*,O_N^+$.

\item Unitary quantum groups: $U_N,U_N^*,U_N^+$.

\item Bistochastic versions: $B_N,B_N^+,C_N,C_N^+$.

\item Quantum permutation groups: $S_N,S_N^+$.

\item Hyperoctahedral quantum groups: $H_N,H_N^*,H_N^+$.

\item Quantum reflection groups: $K_N,K_N^*,K_N^+$.
\end{enumerate} 
\end{theorem}

\begin{proof}
This is something that we already know, the partitions being as follows:

\medskip

(1) For $O_N$ we obtain the category of pairings $P_2$. For $O_N^+$ we obtain the category of noncrossing pairings $NC_2$. For $O_N^*$ we obtain the category $P_2^*$ of pairings having the property that when labelling the legs clockwise $\circ\bullet\circ\bullet\ldots$\,, each string connects $\circ-\bullet$.

\medskip

(2) For $U_N$ we obtain the category $\mathcal P_2$ of pairings which are matching, in the sense that the horizontal strings connect $\circ-\circ$ or $\bullet-\bullet$, and the vertical strings connect $\circ-\bullet$. For $U_N^+$ we obtain the category $\mathcal{NC}_2=NC_2\cap\mathcal P_2$. For $U_N^*$ we obtain $\mathcal P_2^*=P_2^*\cap\mathcal P_2$.

\medskip

(3) For $B_N,C_N$ we obtain the categories $P_{12},\mathcal P_{12}$ of singletons and pairings, and matching singletons and pairings. For $B_N^+,C_N^+$ we obtain the categories $NC_{12},\mathcal{NC}_{12}$ of singletons and noncrossing pairings, and matching singletons and noncrossing pairings. 

\medskip

(4) For $S_N$ we obtain the category of all partitions $P$, and for $S_N^+$ we obtain the category of all noncrossing partitions $NC$.

\medskip

(5) For $H_N$ we obtain the category $P_{even}$ or partitions having even blocks. For $H_N^+$ we obtain the category $NC_{even}=NC\cap P_{even}$ of noncrossing partitions having even blocks. For $H_N^*$ we obtain the category $P_{even}^*\subset P_{even}$ of partitions having the property that when labelling the legs clockwise $\circ\bullet\circ\bullet\ldots$\,, in each block we have $\#\circ=\#\bullet$.

\medskip

(6) For $K_N$ we obtain the category $\mathcal P_{even}$ of partitions having the property that we have $\#\circ=\#\bullet$, as a weighted equality, in each block. For $K_N^+$ we obtain the category $\mathcal{NC}_{even}=\mathcal P_{even}\cap NC$. For $K_N^*$ we obtain the category $\mathcal P_{even}^*=\mathcal P_{even}\cap P_{even}^*$.
\end{proof}

In the above list the examples (4,5,6) appear as the $s=1,2,\infty$ particular cases of the quantum groups $H_N^s,H_N^{s*},H_N^{s+}$, so we have as extra examples these latter quantum groups at $3\leq s<\infty$. Further examples can be constructed via free complexification, or via operations of type $G_N\to\mathbb Z_r\times G_N$, or $G_N\to\mathbb Z_rG_N$, with $r\in\{2,3,\ldots,\infty\}$. 

\bigskip

There are as well ``exotic'' intermediate liberation procedures, involving relations which are more complicated than the half-commutation ones $abc=cba$, which can produce new examples, in the unitary and reflection group cases. We will be back to this.

\bigskip

All this makes the classification question particularly difficult. So, our first task in what follows will be that of cutting a bit from complexity, by adding some extra axioms, chosen as ``natural'' as possible. A first such axiom, very natural, is as follows:

\index{uniform quantum group}
\index{removing blocks}

\begin{proposition}
For an easy quantum group $G=(G_N)$, coming from a category of partitions $D\subset P$, the following conditions are equivalent:
\begin{enumerate}
\item $G_{N-1}=G_N\cap U_{N-1}^+$, via the embedding $U_{N-1}^+\subset U_N^+$ given by $u\to diag(u,1)$.

\item $G_{N-1}=G_N\cap U_{N-1}^+$, via the $N$ possible diagonal embeddings $U_{N-1}^+\subset U_N^+$.

\item $D$ is stable under the operation which consists in removing blocks.
\end{enumerate}
If these conditions are satisfied, we say that $G=(G_N)$ is ``uniform''.
\end{proposition}

\begin{proof}
We use here the general easiness theory from chapter 7 above:

\medskip

$(1)\iff(2)$ This is something standard, coming from the inclusion $S_N\subset G_N$, which makes everything $S_N$-invariant. The result follows as well from the proof of $(1)\iff(3)$ below, which can be converted into a proof of $(2)\iff(3)$, in the obvious way.

\medskip

$(1)\iff(3)$ Given a subgroup $K\subset U_{N-1}^+$, with fundamental corepresentation $u$, consider the $N\times N$ matrix $v=diag(u,1)$. Our claim is that for any $\pi\in P(k)$ we have:
$$\xi_\pi\in Fix(v^{\otimes k})\iff\xi_{\pi'}\in Fix(v^{\otimes k'}),\,\forall\pi'\in P(k'),\pi'\subset\pi$$

In order to prove this, we must study the condition on the left. We have:
\begin{eqnarray*}
\xi_\pi\in Fix(v^{\otimes k})
&\iff&(v^{\otimes k}\xi_\pi)_{i_1\ldots i_k}=(\xi_\pi)_{i_1\ldots i_k},\forall i\\
&\iff&\sum_j(v^{\otimes k})_{i_1\ldots i_k,j_1\ldots j_k}(\xi_\pi)_{j_1\ldots j_k}=(\xi_\pi)_{i_1\ldots i_k},\forall i\\
&\iff&\sum_j\delta_\pi(j_1,\ldots,j_k)v_{i_1j_1}\ldots v_{i_kj_k}=\delta_\pi(i_1,\ldots,i_k),\forall i
\end{eqnarray*}

Now let us recall that our corepresentation has the special form $v=diag(u,1)$. We conclude from this that for any index $a\in\{1,\ldots,k\}$, we must have:
$$i_a=N\implies j_a=N$$

With this observation in hand, if we denote by $i',j'$ the multi-indices obtained from $i,j$ obtained by erasing all the above $i_a=j_a=N$ values, and by $k'\leq k$ the common length of these new multi-indices, our condition becomes:
$$\sum_{j'}\delta_\pi(j_1,\ldots,j_k)(v^{\otimes k'})_{i'j'}=\delta_\pi(i_1,\ldots,i_k),\forall i$$

Here the index $j$ is by definition obtained from $j'$ by filling with $N$ values. In order to finish now, we have two cases, depending on $i$, as follows:

\medskip

\underline{Case 1}. Assume that the index set $\{a|i_a=N\}$ corresponds to a certain subpartition $\pi'\subset\pi$. In this case, the $N$ values will not matter, and our formula becomes:
$$\sum_{j'}\delta_\pi(j'_1,\ldots,j'_{k'})(v^{\otimes k'})_{i'j'}=\delta_\pi(i'_1,\ldots,i'_{k'})$$

\underline{Case 2}. Assume now the opposite, namely that the set $\{a|i_a=N\}$ does not correspond to a subpartition $\pi'\subset\pi$. In this case the indices mix, and our formula reads:
$$0=0$$

Thus, we are led to $\xi_{\pi'}\in Fix(v^{\otimes k'})$, for any subpartition $\pi'\subset\pi$, as claimed.

\medskip

Now with this claim in hand, the result follows from Tannakian duality.
\end{proof}

At the level of the basic examples, from Theorem 11.1, the classical and free quantum groups are uniform, while the half-liberations are not. Indeed, this can be seen either with categories of partitions, or with intersections, the point in the half-classical case being that the relations $abc=cba$, when applied to the coefficients of a matrix of type $v=diag(u,1)$, collapse with $c=1$ to the usual commutation relations $ab=ba$. 

\bigskip

For classification purposes the uniformity axiom is something very natural and useful, substantially cutting from complexity, and we have the following result, from \cite{bsp}:

\index{intersection and generation diagram}
\index{uniform quantum group}
\index{easy group}
\index{free quantum group}

\begin{theorem}
The classical and free uniform orthogonal easy quantum groups, with inclusions between them, are as follows:
$$
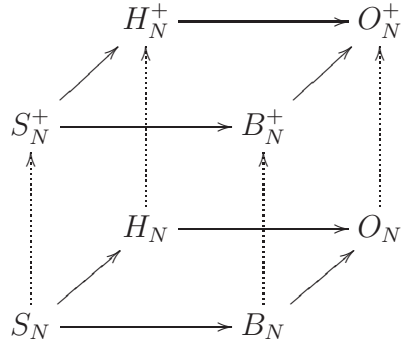
\xymatrix@R=20pt@C=20pt{
&H_N^+\ar[rr]&&O_N^+\\
S_N^+\ar[rr]\ar[ur]&&B_N^+\ar[ur]\\
&H_N\ar[rr]\ar@.[uu]&&O_N\ar@.[uu]\\
S_N\ar@.[uu]\ar[ur]\ar[rr]&&B_N\ar@.[uu]\ar[ur]
}$$
Moreover, this is an intersection/easy generation diagram, in the sense that for any of its square subdiagrams $P\subset Q,R\subset S$ we have $P=Q\cap R$ and $\{Q,R\}=S$.
\end{theorem}

\begin{proof}
We know that the quantum groups in the statement are indeed easy and uniform, the corresponding categories of partitions being as follows:
$$
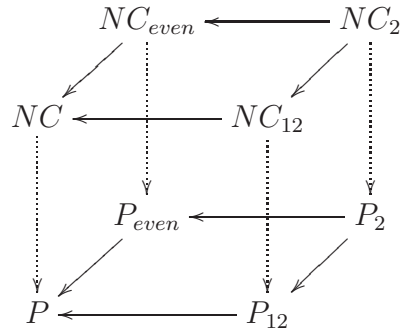
\xymatrix@R=20pt@C6pt{
&NC_{even}\ar[dl]\ar@.[dd]&&NC_2\ar[dl]\ar[ll]\ar@.[dd]\\
NC\ar@.[dd]&&NC_{12}\ar@.[dd]\ar[ll]\\
&P_{even}\ar[dl]&&P_2\ar[dl]\ar[ll]\\
P&&P_{12}\ar[ll]
}$$

Since this latter diagram is an intersection and generation diagram, we conclude that we have an intersection and easy generation diagram of quantum groups, as stated.

\medskip

Regarding now the classification, consider an easy quantum group $S_N\subset G_N\subset O_N$. This most come from a category $P_2\subset D\subset P$, and if we assume $G=(G_N)$ to be uniform, then $D$ is uniquely determined by the subset $L\subset\mathbb N$ consisting of the sizes of the blocks of the partitions in $D$. Our claim is that the admissible sets are as follows:

\begin{enumerate}
\item $L=\{2\}$, producing $O_N$.

\item $L=\{1,2\}$, producing $B_N$.

\item $L=\{2,4,6,\ldots\}$, producing $H_N$.

\item $L=\{1,2,3,\ldots\}$, producing $S_N$.
\end{enumerate}

In one sense, this follows from our easiness results for $O_N,B_N,H_N,S_N$. In the other sense now, assume that $L\subset\mathbb N$ is such that the set $P_L$ consisting of partitions whose sizes of the blocks belong to $L$ is a category of partitions. We know from the axioms of the categories of partitions that the semicircle $\cap$ must be in the category, so we have $2\in L$. We claim that the following conditions must be satisfied as well:
$$k,l\in L,\,k>l\implies k-l\in L$$
$$k\in L,\,k\geq 2\implies 2k-2\in L$$

Indeed, we will prove that both conditions follow from the axioms of the categories of
partitions. Let us denote by $b_k\in P(0,k)$ the one-block partition:
$$b_k=\left\{\begin{matrix}\sqcap\hskip-0.7mm \sqcap&\ldots&\sqcap\\
1\hskip2mm 2&\ldots&k\end{matrix} \right\}$$

For $k>l$, we can write $b_{k-l}$ in the following way:
$$b_{k-l}=\left\{\begin{matrix}\sqcap\hskip-0.7mm
\sqcap&\ldots&\ldots&\ldots&\ldots&\sqcap\\ 1\hskip2mm 2&\ldots&l&l+1&\ldots&k\\
\sqcup\hskip-0.7mm \sqcup&\ldots&\sqcup&|&\ldots&|\\ &&&1&\ldots&k-l\end{matrix}\right\}$$

In other words, we have the following formula:
$$b_{k-l}=(b_l^*\otimes |^{\otimes k-l})b_k$$

Since all the terms of this composition are in $P_L$, we have $b_{k-l}\in P_L$, and this proves our first claim. As for the second claim, this can be proved in a similar way, by capping two adjacent $k$-blocks with a $2$-block, in the middle.

\medskip

With these conditions in hand, we can conclude in the following way:

\medskip

\underline{Case 1}. Assume $1\in L$. By using the first condition with $l=1$ we get:
$$k\in L\implies k-1\in L$$

This condition shows that we must have $L=\{1,2,\ldots,m\}$, for a certain number $m\in\{1,2,\ldots,\infty\}$. On the other hand, by using the second condition we get:
\begin{eqnarray*}
m\in L
&\implies&2m-2\in L\\
&\implies&2m-2\leq m\\
&\implies&m\in\{1,2,\infty\}
\end{eqnarray*}

The case $m=1$ being excluded by the condition $2\in L$, we reach to one of the two sets producing the groups $S_N,B_N$.

\medskip

\underline{Case 2}. Assume $1\notin L$. By using the first condition with $l=2$ we get:
$$k\in L\implies k-2\in L$$

This condition shows that we must have $L=\{2,4,\ldots,2p\}$, for a certain number $p\in\{1,2,\ldots,\infty\}$. On the other hand, by using the second condition we get:
\begin{eqnarray*}
2p\in L
&\implies&4p-2\in L\\
&\implies&4p-2\leq 2p\\
&\implies&p\in\{1,\infty\}
\end{eqnarray*}

Thus $L$ must be one of the two sets producing $O_N,H_N$, and we are done. In the free case, $S_N^+\subset G_N\subset O_N^+$, the situation is quite similar, the admissible sets being once again the above ones, producing this time $O_N^+,B_N^+,H_N^+,S_N^+$. See \cite{bsp}. 
\end{proof}

As already mentioned, when removing the uniformity axiom things become more complicated, and the classification result here, from \cite{bsp}, \cite{rwe}, is as follows:

\index{easy group}
\index{free quantum group}
\index{bistochastic group}
\index{bistochastic quantum group}

\begin{theorem}
The classical and free orthogonal easy quantum groups are
$$
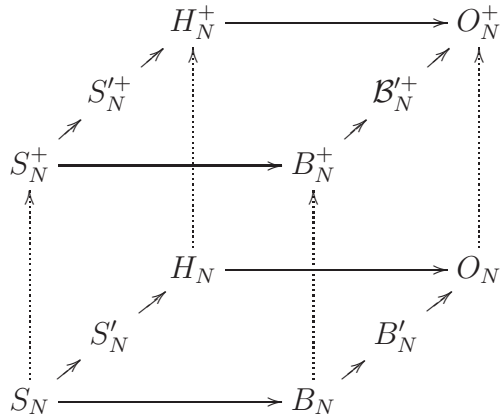
\xymatrix@R=7pt@C=7pt{
&&H_N^+\ar[rrrr]&&&&O_N^+\\
&S_N'^+\ar[ur]&&&&\mathcal B_N'^+\ar[ur]\\
S_N^+\ar[rrrr]\ar[ur]&&&&B_N^+\ar[ur]\\
\\
&&H_N\ar[rrrr]\ar@.[uuuu]&&&&O_N\ar@.[uuuu]\\
&S_N'\ar[ur]&&&&B_N'\ar[ur]\\
S_N\ar@.[uuuu]\ar[ur]\ar[rrrr]&&&&B_N\ar@.[uuuu]\ar[ur]
\\
}$$
with $S_N'=S_N\times\mathbb Z_2$, $B_N'=B_N\times\mathbb Z_2$, and with $S_N'^+,\mathcal B_N'^+$ being their liberations, where $\mathcal B_N'^+$ stands for the two possible such liberations, $B_N'^+\subset B_N''^+$.
\end{theorem}

\begin{proof}
The idea here is that of jointly classifying the ``classical'' categories of partitions $P_2\subset D\subset P$, and the ``free'' ones $NC_2\subset D\subset NC$:

\medskip

(1) At the classical level this leads, via a study which is quite similar to that from the proof of Theorem 11.3, to 2 more groups, namely $S_N',B_N'$. See \cite{bsp}. 

\medskip

(2) At the free level we obtain 3 more quantum groups, $S_N'^+,B_N'^+,B_N''^+$, with the inclusion $B_N'^+\subset B_N''^+$, which is something a bit surprising, being best thought of as coming from an inclusion $B_N'\subset B_N''$, which happens to be an isomorphism. See \cite{bsp}, \cite{rwe}.
\end{proof}

\section*{11b. Twistability}

Now back to the easy uniform case, the classification here remains a quite technical topic. The problem comes from the following negative result:

\begin{proposition}
The cubic diagram from Theorem 11.3, and its unitary analogue,
$$
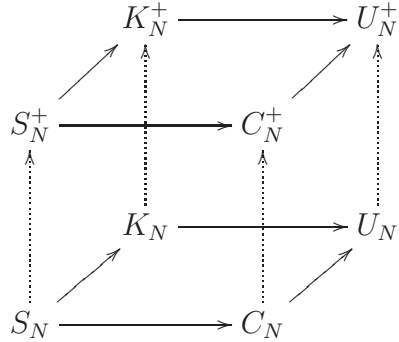
\xymatrix@R=20pt@C=20pt{
&K_N^+\ar[rr]&&U_N^+\\
S_N^+\ar[rr]\ar[ur]&&C_N^+\ar[ur]\\
&K_N\ar[rr]\ar@.[uu]&&U_N\ar@.[uu]\\
S_N\ar@.[uu]\ar[ur]\ar[rr]&&C_N\ar@.[uu]\ar[ur]
}$$
cannot be merged, without degeneration, into a $4$-dimensional cubic diagram.
\end{proposition}

\begin{proof}
All this is a bit philosophical, with the problem coming from the ``taking the bistochastic version'' operation, and more specifically, from the following equalities:
$$H_N\cap C_N=K_N\cap C_N=S_N$$

Indeed, these equalities do hold, and so the 3D cube obtained by merging the classical faces of the orthogonal and unitary cubes is something degenerate, as follows:
$$
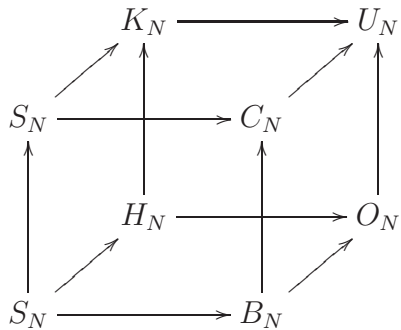
\xymatrix@R=20pt@C=20pt{
&K_N\ar[rr]&&U_N\\
S_N\ar[rr]\ar[ur]&&C_N\ar[ur]\\
&H_N\ar[rr]\ar[uu]&&O_N\ar[uu]\\
S_N\ar[uu]\ar[ur]\ar[rr]&&B_N\ar[uu]\ar[ur]
}$$

Thus, the 4D cube, having this 3D cube as one of its faces, is degenerate too.
\end{proof}

Summarizing, when positioning ourselves at $U_N^+$, we have 4 natural directions to be followed, namely taking the classical, discrete, real and bistochastic versions. And the problem is that, while the first three operations are ``good'', the fourth one is ``bad''. 

\bigskip

In order to fix this problem, in a useful and efficient way, the natural choice is that of slashing the bistochastic quantum groups $B_N,B_N^+,C_N,C_N^+$, which are rather secondary objects anyway, as well the quantum permutation groups $S_N,S_N^+$. 

\bigskip

In order to formulate now our second general axiom, doing the job, consider the cube $T_N=\mathbb Z_2^N$, regarded as diagonal torus of $O_N$. We have then:

\index{twisting}
\index{twistability}

\begin{proposition}
For an easy quantum group $G=(G_N)$, coming from a category of partitions $D\subset P$, the following conditions are equivalent:
\begin{enumerate}
\item $T_N\subset G_N$.

\item $H_N\subset G_N$.

\item $D\subset P_{even}$.
\end{enumerate}
If these conditions are satisfied, we say that $G_N$ is ``twistable''.
\end{proposition}

\begin{proof}
We use the general easiness theory from chapter 7 above:

\medskip

$(1)\iff(2)$ Here it is enough to check that the easy envelope $T_N'$ of the cube equals the hyperoctahedral group $H_N$. But this follows from:
$$T_N'
=<T_N,S_N>'
=H_N'
=H_N$$

$(2)\iff(3)$ This follows by functoriality, from the fact that $H_N$ comes from the category of partitions $P_{even}$, that we know from chapter 10 above.
\end{proof}

The teminology in the above result comes from the fact that, assuming $D\subset P_{even}$, we can indeed twist $G_N$, into a certain quizzy quantum group $\bar{G}_N$. We refer to chapter 7 above to full details regarding the construction $G_N\to\bar{G}_N$. In what follows we will not need this twisting procedure, and we will just use Proposition 11.6 as it is, as a statement providing us with a simple and natural condition to be imposed on $G_N$. In practice now, imposing this second axiom leads to something nice, namely:

\index{uniform quantum group}

\begin{theorem}
The basic quantum unitary and quantum reflection groups, from Proposition 11.1 above, which are uniform and twistable, are as follows,
$$
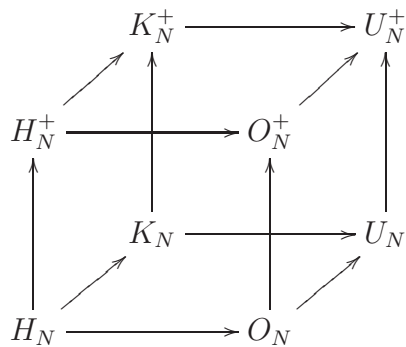
\xymatrix@R=20pt@C=20pt{
&K_N^+\ar[rr]&&U_N^+\\
H_N^+\ar[rr]\ar[ur]&&O_N^+\ar[ur]\\
&K_N\ar[rr]\ar[uu]&&U_N\ar[uu]\\
H_N\ar[uu]\ar[ur]\ar[rr]&&O_N\ar[uu]\ar[ur]
}$$
and this is an intersection and easy generation diagram.
\end{theorem}

\begin{proof}
The first assertion comes from discussion after Proposition 11.2, telling us that the uniformity condition eliminates $O_N^*,U_N^*,H_N^*,K_N^*$. Also, the twistability condition eliminates $B_N,B_N^+,C_N,C_N^+$ and $S_N,S_N^+$. Thus, we are left with the 8 quantum groups in the statement, which are indeed easy, coming from the following categories:
$$
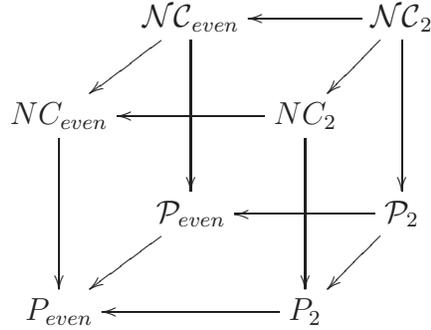
\xymatrix@R=20pt@C6pt{
&\mathcal{NC}_{even}\ar[dl]\ar[dd]&&\mathcal {NC}_2\ar[dl]\ar[ll]\ar[dd]\\
NC_{even}\ar[dd]&&NC_2\ar[dd]\ar[ll]\\
&\mathcal P_{even}\ar[dl]&&\mathcal P_2\ar[dl]\ar[ll]\\
P_{even}&&P_2\ar[ll]
}$$

Since this latter diagram is an intersection and generation diagram, we conclude that we have an intersection and easy generation diagram of quantum groups, as stated.
\end{proof}

As explained above, we will not really need in what follows the twists of the twistable quantum groups that we consider, our plan being that of using the twistability condition as a natural condition to be imposed on our quantum groups, for classification purposes. However, let us record as well the following result, in relation with the twists:

\index{Schur-Weyl twisting}
\index{twisting}

\begin{theorem}
The Schur-Weyl twists of the basic twistable quantum groups are
$$\xymatrix@R=20pt@C=20pt{
&K_N^+\ar[rr]&&U_N^+\\
H_N^+\ar[rr]\ar[ur]&&O_N^+\ar[ur]\\
&K_N\ar[rr]\ar[uu]&&\bar{U}_N\ar[uu]\\
H_N\ar[uu]\ar[ur]\ar[rr]&&\bar{O}_N\ar[uu]\ar[ur]
}$$
and this is an intersection and quizzy generation diagram.
\end{theorem}

\begin{proof}
Here the formulae of the twists are something that we already know, coming from the computations in chapter 7 above, and the last assertion is clear as well, coming from the definition of the various quantum groups involved.
\end{proof}

\section*{11c. Orientability}

In the general case now, where we have an arbitrary uniform and twistable easy quantum group, this quantum group appears by definition as follows:
$$H_N\subset G_N\subset U_N^+$$

Thus, we can imagine our quantum group $G_N$ as sitting inside the standard cube, from Theorem 11.7 above:
$$\xymatrix@R=20pt@C=20pt{
&K_N^+\ar[rr]&&U_N^+\\
H_N^+\ar[rr]\ar[ur]&&O_N^+\ar[ur]\\
&K_N\ar[rr]\ar[uu]&&U_N\ar[uu]\\
H_N\ar[uu]\ar[ur]\ar[rr]&&O_N\ar[uu]\ar[ur]
}$$

The point now is that, by using the operations $\cap$ and $\{\,,\}$, we can in principle ``project'' $G_N$ on the faces and edges of the cube, and then use some kind of 3D orientation coming from this, in order to deduce some structure and classification results. 

\bigskip

In order to do this, let us start with the following definition:

\index{classical version}
\index{discrete version}
\index{real version}
\index{free version}
\index{smooth version}
\index{unitary version}

\begin{definition}
Associated to any twistable easy quantum group 
$$H_N\subset G_N\subset U_N^+$$
are its classical, discrete and real versions, given by the following formulae,
$$G_N^c=G_N\cap U_N$$
$$G_N^d=G_N\cap K_N^+$$
$$G_N^r=G_N\cap O_N^+$$
as well as its free, smooth and unitary versions, given by the following formulae,
$$G_N^f=\{G_N,H_N^+\}$$
$$G_N^s=\{G_N,O_N\}$$
$$G_N^u=\{G_N,K_N\}$$
where $\cap$ and $\{\,,\}$ are respectively the intersection and easy generation operations.
\end{definition}

In this definition the classical, real and unitary versions are something quite standard. Regarding the discrete and smooth versions, here we have no abstract justification for our terminology, due to the fact that easy quantum groups do not have known differential geometry. However, in the classical case, where $G_N\subset U_N$, our constructions produce indeed discrete and smooth versions, and this is where our terminology comes from. Finally, regarding the free version, this comes once again from the known examples. 

\bigskip

To be more precise, regarding the free version, the various results that we have show that the liberation operation $G_N\to G_N^+$ usually appears via the formula:
$$G_N^+=\{G_N,S_N^+\}$$

This formula expresses the fact that the category of partitions of $G_N^+$ is obtained from the one of $G_N$ by removing the crossings. But in the twistable setting, where we have by definition $H_N\subset G_N$, this is the same as setting:
$$G_N^+=\{G_N,H_N^+\}$$

All this is of course a bit theoretical, and this is why we use the symbol $f$ for free versions in the above sense, and keep $+$ for well-known, studied liberations.

\bigskip

In relation now with our questions, and our 3D plan, we can now formulate:

\begin{proposition}
Given an intermediate quantum group $H_N\subset G_N\subset U_N^+$, we have a diagram of closed subgroups of $U_N^+$, obtained by inserting
$$
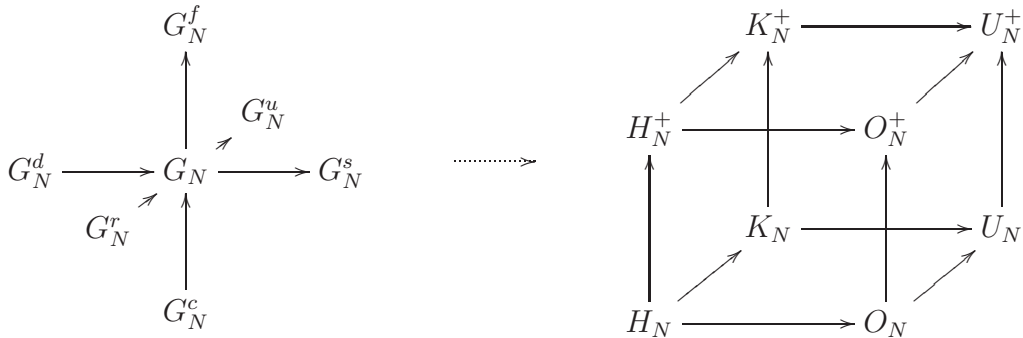
\xymatrix@R=3pt@C=5pt{
&&G_N^f&&\\
\\
&&&G_N^u&\\
G_N^d\ar[rr]&&G_N\ar[rr]\ar[uuu]\ar[ur]&&G_N^s\\
&G_N^r\ar[ur]&&&\\
\\
&&G_N^c\ar[uuu]&&}
\qquad\xymatrix@R=10pt@C=30pt{\\ \\ \\ \ar@.[r]&}\qquad
\xymatrix@R=19pt@C=20pt{
&K_N^+\ar[rr]&&U_N^+\\
H_N^+\ar[rr]\ar[ur]&&O_N^+\ar[ur]\\
&K_N\ar[rr]\ar[uu]&&U_N\ar[uu]\\
H_N\ar[uu]\ar[ur]\ar[rr]&&O_N\ar[uu]\ar[ur]
}$$
in the obvious way, with each $G_N^x$ belonging to the main diagonal of each face.
\end{proposition}

\begin{proof}
The fact that we have indeed the diagram of inclusions on the left is clear from the constructions of the quantum groups involved, from Definition 11.9. Regarding the insertion procedure, consider any of the faces of the cube, denoted as follows:
$$P\subset Q,R\subset S$$

Our claim is that the corresponding quantum group $G=G_N^x$ can be inserted on the corresponding main diagonal $P\subset S$, as follows:
$$
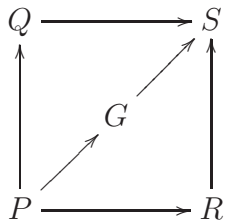
\xymatrix@R=20pt@C=20pt{
Q\ar[rr]&&S\\
&G\ar[ur]\\
P\ar[rr]\ar[uu]\ar[ur]&&R\ar[uu]}$$

We have to check here a total of $6\times 2=12$ inclusions. But, according to Definition 11.9, these inclusions that must be checked are as follows:

\medskip

(1) $H_N\subset G_N^c\subset U_N$, where $G_N^c=G_N\cap U_N$.

\medskip

(2) $H_N\subset G_N^d\subset K_N^+$, where $G_N^d=G_N\cap K_N^+$.

\medskip

(3) $H_N\subset G_N^r\subset O_N^+$, where $G_N^r=G_N\cap O_N^+$.

\medskip

(4) $H_N^+\subset G_N^f\subset U_N^+$, where $G_N^f=\{G_N,H_N^+\}$.

\medskip

(5) $O_N\subset G_N^s\subset U_N^+$, where $G_N^s=\{G_N,O_N\}$.

\medskip

(6) $K_N\subset G_N^u\subset U_N^+$, where $G_N^u=\{G_N,K_N\}$.

\medskip

All these statements being trivial from the definition of $\cap$ and $\{\,,\}$, and from our assumption $H_N\subset G_N\subset U_N^+$, our insertion procedure works indeed, and we are done.
\end{proof}

In order now to complete the diagram, we have to project as well $G_N$ on the edges of the cube. For this purpose we can basically assume, by replacing $G_N$ with each of its 6 projections on the faces, that $G_N$ actually lies on one of the six faces. The technical result that we will need here is as follows:

\index{slicing}

\begin{proposition}
Given an intersection and easy generation diagram $P\subset Q,R\subset S$ and an intermediate easy quantum group $P\subset G\subset S$, as follows,
$$
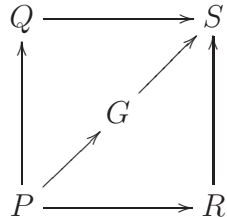
\xymatrix@R=20pt@C=20pt{
Q\ar[rr]&&S\\
&G\ar[ur]\\
P\ar[rr]\ar[uu]\ar[ur]&&R\ar[uu]}$$
we can extend this diagram into a diagram as follows:
$$
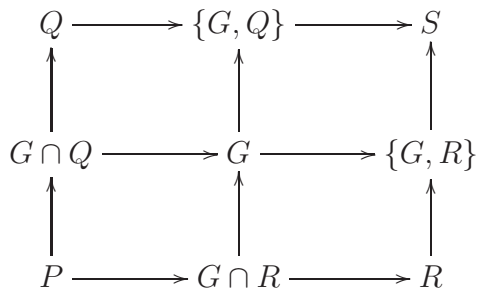
\xymatrix@R=30pt@C=30pt{
Q\ar[r]&\{G,Q\}\ar[r]&S\\
G\cap Q\ar[u]\ar[r]&G\ar[r]\ar[u]&\{G,R\}\ar[u]\\
P\ar[r]\ar[u]&G\cap R\ar[u]\ar[r]&R\ar[u]}$$
In addition, $G$ ``slices the square'', in the sense that this is an intersection and easy generation diagram, precisely when $G=\{G\cap Q,G\cap R\}$ and $G=\{G,Q\}\cap\{G,R\}$.
\end{proposition}

\begin{proof}
This is indeed clear from definitions, because the intersection and easy generation conditions are automatic for the upper left and lower right squares, and so are half of the intersection and easy generation conditions for the lower left and upper right squares. Thus, we are left with two conditions only, which are those in the statement.
\end{proof}

Now back to 3 dimensions, and to the cube, we have the following result:

\begin{proposition}
Assuming that $H_N\subset G_N\subset U_N^+$ satisfies the conditions
$$G_N^{cs}=G_N^{sc}\quad,\quad
G_N^{cu}=G_N^{uc}\quad,\quad
G_N^{df}=G_N^{fd}$$
$$G_N^{du}=G_N^{ud}\quad,\quad 
G_N^{rf}=G_N^{fr}\quad,\quad
G_N^{rs}=G_N^{sr}$$
the diagram in Proposition 11.10 can be completed, via the construction in Proposition 11.11, into a diagram dividing the cube along the $3$ coordinates axes, into $8$ small cubes.
\end{proposition}

\begin{proof}
We have to prove that the 12 projections on the edges are well-defined, with the problem coming from the fact that each of these projections can be defined in 2 possible ways, depending on the face that we choose first. 

\medskip

The verification goes as follows:

\medskip

(1) Regarding the $3$ edges emanating from $H_N$, the result here follows from:
$$G_N^{cd}=G_N^{dc}=G_N\cap K_N$$ 
$$G_N^{cr}=G_N^{rc}=G_N\cap O_N$$
$$G_N^{dr}=G_N^{rd}=G_N\cap H_N^+$$

These formulae are indeed all trivial, of type:
$$(G\cap Q)\cap R=(G\cap R)\cap Q=G\cap P$$

(2) Regarding the $3$ edges landing into $U_N^+$, the result here follows from:
$$G_N^{fs}=G_N^{sf}=\{G_N,O_N^+\}$$
$$G_N^{fu}=G_N^{uf}=\{G_N,K_N^+\}$$
$$G_N^{su}=G_N^{us}=\{G_N,U_N\}$$

These formulae are once again trivial, of type:
$$\{\{G,Q\},R\}=\{\{G,R\},Q\}=\{G,S\}$$

(3) Finally, regarding the remaining $6$ edges, not emanating from $H_N$ or landing into $U_N^+$, here the result follows from our assumptions in the statement.
\end{proof}

Unfortunately, we are not done yet, because nothing guarantees that we obtain in this way an intersection and easy generation diagram. 

\bigskip

Thus, we must add more axioms, as follows:

\index{slicing}

\begin{theorem}
Assume that $H_N\subset G_N\subset U_N^+$ satisfies the following conditions, where by ``intermediate'' we mean in each case ``parallel to its neighbors'':
\begin{enumerate}
\item The $6$ compatibility conditions in Proposition 11.12 above,

\item $G_N^c,G_N,G_N^f$ slice the classical/intermediate/free faces,

\item $G_N^d,G_N,G_N^s$ slice the discrete/intermediate/smooth faces,

\item $G_N^r,G_N,G_N^u$ slice the real/intermediate/unitary faces,
\end{enumerate}
Then $G_N$ ``slices the cube'', in the sense that the diagram obtained in Proposition 11.12 above is an intersection and easy generation diagram.
\end{theorem}

\begin{proof}
This follows indeed from Proposition 11.11 and Proposition 11.12 above.
\end{proof}

Summarizing, we are done now with our geometric program, and we have a whole collection of natural geometric conditions that can be imposed to $G_N$.

\section*{11d. Ground zero}

It is quite clear that $G_N$ can be reconstructed from its edge projections, so in order to do the classification, we first need a ``coordinate system''. Common sense would suggest to use the one emanating from $H_N$, or perhaps the one landing into $U_N^+$. However, technically speaking, best is to use the coordinate system based at $O_N$, highlighted below:
$$
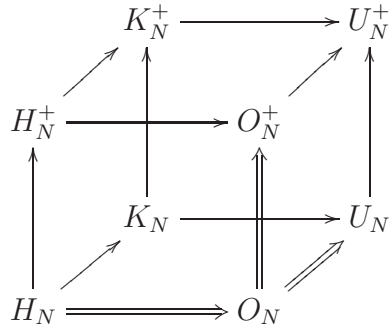
\xymatrix@R=18pt@C=18pt{
&K_N^+\ar[rr]&&U_N^+\\
H_N^+\ar[rr]\ar[ur]&&O_N^+\ar[ur]\\
&K_N\ar[rr]\ar[uu]&&U_N\ar[uu]\\
H_N\ar[uu]\ar[ur]\ar@=[rr]&&O_N\ar@=[uu]\ar@=[ur]
}$$

This choice comes from the fact that the classification result for $O_N\subset O_N^+$, explained below, is something very simple. And this is not the case with the results for $H_N\subset H_N^+$ and for $U_N\subset U_N^+$, from \cite{mwe}, \cite{rwe} which are quite complicated, with uncountably many solutions, in the general non-uniform case. As for the result for $K_N\subset K_N^+$, this is not available yet, but it is known that there are uncountably many solutions here as well.

\bigskip

So, here is now the key result, from \cite{bv2}, dealing with the vertical direction:

\begin{theorem}
There is only one proper intermediate easy quantum group
$$O_N\subset G_N\subset O_N^+$$
namely the quantum group $O_N^*$, which is not uniform.
\end{theorem}

\begin{proof}
We must compute here the categories of pairings $NC_2\subset D\subset P_2$, and this can be done via some standard combinatorics, in three steps, as follows:

\medskip

(1) Let $\pi\in P_2-NC_2$, having $s\geq 4$ strings. Our claim is that:

\medskip

-- If $\pi\in P_2-P_2^*$, there exists a semicircle capping $\pi'\in P_2-P_2^*$.

-- If $\pi\in P_2^*-NC_2$, there exists a semicircle capping $\pi'\in P_2^*-NC_2$.

\medskip

Indeed, both these assertions can be easily proved, by drawing pictures.

\medskip

(2) Consider now a partition $\pi\in P_2(k,l)-NC_2(k,l)$. Our claim is that:

\medskip

-- If $\pi\in P_2(k, l)-P_2^*(k,l)$ then $<\pi>=P_2$.

-- If $\pi\in P_2^*(k,l)-NC_2(k,l)$ then $<\pi>=P_2^*$.

\medskip

This can be indeed proved by recurrence on the number of strings, $s=(k+l)/2$, by using (1), which provides us with a descent procedure $s\to s-1$, at any $s\geq4$.

\medskip

(3) Finally, assume that we are given an easy quantum group $O_N\subset G\subset O_N^+$, coming from certain sets of pairings $D(k,l)\subset P_2(k,l)$. We have three cases:

\medskip

-- If $D\not\subset P_2^*$, we obtain $G=O_N$.

-- If $D\subset P_2,D\not\subset NC_2$, we obtain $G=O_N^*$.

-- If $D\subset NC_2$, we obtain $G=O_N^+$.

\medskip

Thus, we have proved the uniquess result. As for the non-uniformity of the unique solution, $O_N^*$, this is something that we already know, from Theorem 11.7 above.
\end{proof}

The above result is something quite remarkable, and it is actually believed that the result could still hold, without the easiness assumption. We refer here to \cite{bc+}. 

\bigskip

As already mentioned, the related inclusions $H_N\subset H_N^+$ and $U_N\subset U_N^+$, studied in \cite{mwe} and \cite{rwe}, are far from being maximal, having uncountably many intermediate objects, and the same is known to hold for $K_N\subset K_N^+$. There are many interesting open questions here. It is conjectured for instance that there should be a contravariant duality $H_N^\times\leftrightarrow U_N^\times$, mapping the family and series from \cite{rwe} to the series and family from \cite{twe}.

\bigskip

Here is now another basic result that we will need, in order to perform our classification work here, dealing this time with the ``discrete vs. continuous'' direction:

\index{easy group}

\begin{theorem}
There are no proper intermediate easy groups
$$H_N\subset G_N\subset O_N$$
except for $H_N,O_N$ themselves.
\end{theorem}

\begin{proof}
We must prove that there are no proper intermediate categories as follows:
$$P_2\subset D\subset P_{even}$$

But this can done via some combinatorics, in the spirit of the proof of Theorem 11.3, and with the result itself coming from Theorem 11.4. For full details here, see \cite{bsp}.
\end{proof}

As a comment here, the inclusion $H_N^+\subset O_N^+$ is maximal as well, as explained once again in \cite{bsp}. As for the complex versions of these results, regarding the inclusions $K_N\subset U_N$ and $K_N^+\subset U_N^+$, here the classification, in the non-uniform case, is available from \cite{twe}. Summarizing, we have here once again something very basic and fundamental, providing some evidence for a kind of general ``discrete vs. continuous'' dichotomy.

\bigskip

Finally, here is a third and last result that we will need, for our classification work here, regarding the missing direction, namely the  ``real vs. complex'' one:

\index{easy group}

\begin{theorem}
The proper intermediate easy groups
$$O_N\subset G_N\subset U_N$$
are the groups $\mathbb Z_rO_N$ with $r\in\{2,3,\ldots,\infty\}$, which are not uniform.
\end{theorem}

\begin{proof}
This is standard and well-known, from \cite{twe}, the proof being as follows:

\medskip

(1) Our first claim is that the group $\mathbb TO_N\subset U_N$ is easy, the corresponding category of partitions being the subcategory $\bar{P}_2\subset P_2$ consisting of the pairings having the property that when flatenning, we have the global formula $\#\circ=\#\bullet$. 

\medskip

(2) Indeed, if we denote the standard corepresentation by $u=zv$, with $z\in\mathbb T$ and with $v=\bar{v}$, then in order to have $Hom(u^{\otimes k},u^{\otimes l})\neq\emptyset$, the $z$ variabes must cancel, and in the case where they cancel, we obtain the same Hom-space as for $O_N$. 

Now since the cancelling property for the $z$ variables corresponds precisely to the fact that $k,l$ must have the same numbers of $\circ$ symbols minus $\bullet$ symbols, the associated Tannakian category must come from the category of pairings $\bar{P}_2\subset P_2$, as claimed.

\medskip

(3) Our second claim is that, more generally, the group $\mathbb Z_rO_N\subset U_N$ is easy, with the corresponding category $P_2^r\subset P_2$ consisting of the pairings having the property that when flatenning, we have the global formula $\#\circ=\#\bullet(r)$. 

\medskip

(4) Indeed, this is something that we already know at $r=1,\infty$, where the group in question is $O_N,\mathbb TO_N$. The proof in general is similar, by writing $u=zv$ as above.

\medskip

(5) Let us prove now the converse, stating that the above groups $O_N\subset\mathbb Z_rO_N\subset U_N$ are the only intermediate easy groups $O_N\subset G\subset U_N$. According to our conventions for the easy quantum groups, which apply of course to the classical case, we must compute the following intermediate categories of pairings:
$$\mathcal P_2\subset D\subset P_2$$

(6) So, assume that we have such a category, $D\neq\mathcal P_2$, and pick an element $\pi\in D-\mathcal P_2$, assumed to be flat. We can modify $\pi$, by performing the following operations:

\medskip

-- First, we can compose with the basic crossing, in order to assume that $\pi$ is a partition of type $\cap\ldots\ldots\cap$, consisting of consecutive semicircles. Our assumption $\pi\notin\mathcal P_2$ means that at least one semicircle is colored black, or white.

\medskip

-- Second, we can use the basic mixed-colored semicircles, and cap with them all the mixed-colored semicircles. Thus, we can assume that $\pi$ is a nonzero partition of type $\cap\ldots\ldots\cap$, consisting of consecutive black or white semicircles.

\medskip

-- Third, we can rotate, as to assume that $\pi$ is a partition consisting of an upper row of white semicircles, $\cup\ldots\ldots\cup$, and a lower row of white semicircles, $\cap\ldots\ldots\cap$. Our assumption $\pi\notin\mathcal P_2$ means that this latter partition is nonzero.

\medskip

(7) For $a,b\in\mathbb N$ consider the partition consisting of an upper row of $a$ white semicircles, and a lower row of $b$ white semicircles, and set:
$$\mathcal C=\left\{\pi_{ab}\Big|a,b\in\mathbb N\right\}\cap D$$

According to the above we have $\pi\in<\mathcal C>$. The point now is that we have:

\medskip

-- There exists $r\in\mathbb N\cup\{\infty\}$ such that $\mathcal C$ equals the following set:
$$\mathcal C_r=\left\{\pi_{ab}\Big|a=b(r)\right\}$$

This is indeed standard, by using the categorical axioms.

\medskip

-- We have the following formula, with $P_2^r$ being as above:
$$<\mathcal C_r>=P_2^r$$

This is standard as well, by doing some diagrammatic work.

\medskip

(8) With these results in hand, the conclusion now follows. Indeed, with $r\in\mathbb N\cup\{\infty\}$ being as above, we know from the beginning of the proof that any $\pi\in D$ satisfies:
$$\pi
\in<\mathcal C>
=<\mathcal C_r>
=P_2^r$$

Thus we have an inclusion $D\subset P_2^r$. Conversely, we have as well:
$$P_2^r
=<\mathcal C_r>
=<\mathcal C>
\subset<D>
=D$$

Thus we have $D=P_2^r$, and this finishes the proof. See Tarrago-Weber \cite{twe}.
\end{proof}

Once again, there are many comments that can be made here, with the whole subject in the easy case being generally covered by the classification results in \cite{twe}. As for the non-easy case, there are many interesting things here as well, as for instance the results in \cite{bc+}, stating that $PO_N\subset PU_N$, and $\mathbb TO_N\subset U_N$ as well, are maximal.

\bigskip

We can now formulate a classification result, from \cite{ba5}, as follows:

\index{slicing}
\index{Ground Zero theorem}

\begin{theorem}[Ground zero]
There are exactly eight closed subgroups $G_N\subset U_N^+$ having the following properties,
\begin{enumerate}
\item Easiness,

\item Uniformity,

\item Twistability,

\item Slicing property,
\end{enumerate}
namely the quantum groups $O_N,U_N,H_N,K_N$ and $O_N^+,U_N^+,H_N^+,K_N^+$.
\end{theorem}

\begin{proof}
We already know, from Theorem 11.7 above, that the 8 quantum groups in the statement have indeed the properties (1-4), and form a cube, as follows:
$$
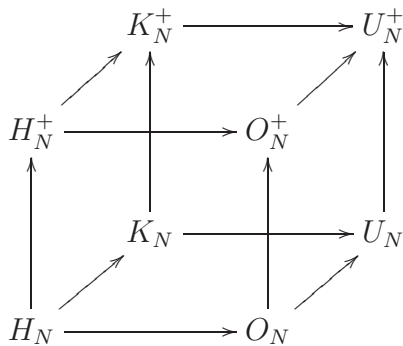
\xymatrix@R=20pt@C=20pt{
&K_N^+\ar[rr]&&U_N^+\\
H_N^+\ar[rr]\ar[ur]&&O_N^+\ar[ur]\\
&K_N\ar[rr]\ar[uu]&&U_N\ar[uu]\\
H_N\ar[uu]\ar[ur]\ar[rr]&&O_N\ar[uu]\ar[ur]
}$$

Conversely now, assuming that an easy quantum group $G=(G_N)$ has the above properties (2-4), the twistability property, (3), tells us that we have:
$$H_N\subset G_N\subset U_N^+$$

Thus $G_N$ sits inside the cube, and the above discussion applies. To be more precise, let us project $G$ on the faces of the cube, as in Proposition 11.10 above:
$$\xymatrix@R=3pt@C=5pt{
&&G_N^f&&\\
\\
&&&G_N^u&\\
G_N^d\ar[rr]&&G_N\ar[rr]\ar[uuu]\ar[ur]&&G_N^s\\
&G_N^r\ar[ur]&&&\\
\\
&&G_N^c\ar[uuu]&&}
\qquad\xymatrix@R=10pt@C=30pt{\\ \\ \\ \ar@.[r]&}\qquad
\xymatrix@R=19pt@C=20pt{
&K_N^+\ar[rr]&&U_N^+\\
H_N^+\ar[rr]\ar[ur]&&O_N^+\ar[ur]\\
&K_N\ar[rr]\ar[uu]&&U_N\ar[uu]\\
H_N\ar[uu]\ar[ur]\ar[rr]&&O_N\ar[uu]\ar[ur]
}$$

In order to compute these projections, and eventually prove that $G_N$ is one of the vertices of the cube, we can use use the coordinate system based at $O_N$:
$$\xymatrix@R=18pt@C=18pt{
&K_N^+\ar[rr]&&U_N^+\\
H_N^+\ar[rr]\ar[ur]&&O_N^+\ar[ur]\\
&K_N\ar[rr]\ar[uu]&&U_N\ar[uu]\\
H_N\ar[uu]\ar[ur]\ar@=[rr]&&O_N\ar@=[uu]\ar@=[ur]
}$$

Now by using Theorem 11.14, Theorem 11.15 and Theorem 11.16, along with the uniformity condition, (2), we conclude that the edge projections of $G_N$ must be among the vertices of the cube. Moreover, by using the slicing axiom, (4), we deduce from this that $G_N$ itself must be a vertex of the cube. Thus, we have exactly 8 solutions to our problem, namely the vertices of the cube, as claimed.
\end{proof}

All this is quite philosophical. Bluntly put, by piling up a number of very natural axioms, namely those of Woronowicz from \cite{wo1}, then our assumption $S^2=id$, and then the easiness, uniformity, twistability, and slicing properties, we have managed to destroy everything, or almost. The casualities include lots of interesting finite and compact Lie groups, the duals of all finitely generated discrete groups, plus of course lots of interesting quantum groups, which appear not to be strong enough to survive our axioms. 

\bigskip

We should mention that the above result is in tune with free probability, and with noncommutative geometry, where the most important quantum groups which appear are precisely the above 8 ones. In what regards free probability, this comes from the various character computations performed in chapters 8 and 10 above, which give:

\index{truncated character}
\index{CLT}

\begin{theorem}
The asymptotic character laws for the $8$ main quantum groups are
$$
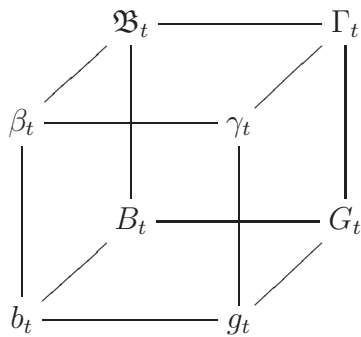
\xymatrix@R=20pt@C=22pt{
&\mathfrak B_t\ar@{-}[rr]\ar@{-}[dd]&&\Gamma_t\ar@{-}[dd]\\
\beta_t\ar@{-}[rr]\ar@{-}[dd]\ar@{-}[ur]&&\gamma_t\ar@{-}[dd]\ar@{-}[ur]\\
&B_t\ar@{-}[rr]\ar@{-}[uu]&&G_t\ar@{-}[uu]\\
b_t\ar@{-}[uu]\ar@{-}[ur]\ar@{-}[rr]&&g_t\ar@{-}[uu]\ar@{-}[ur]
}$$
which are exactly the $8$ main limiting laws in classical and free probability.
\end{theorem}

\begin{proof}
This is something that we already know, explained in chapters 8 and 10, and which comes from easiness. Consider indeed our 8 main quantum groups:
$$\xymatrix@R=20pt@C=20pt{
&K_N^+\ar[rr]&&U_N^+\\
H_N^+\ar[rr]\ar[ur]&&O_N^+\ar[ur]\\
&K_N\ar[rr]\ar[uu]&&U_N\ar[uu]\\
H_N\ar[uu]\ar[ur]\ar[rr]&&O_N\ar[uu]\ar[ur]
}$$

Accoring to our various Brauer type results, all these quantum groups are easy, the corresponding categories of partitions being as follows:
$$\xymatrix@R=20pt@C6pt{
&\mathcal{NC}_{even}\ar[dl]\ar[dd]&&\mathcal {NC}_2\ar[dl]\ar[ll]\ar[dd]\\
NC_{even}\ar[dd]&&NC_2\ar[dd]\ar[ll]\\
&\mathcal P_{even}\ar[dl]&&\mathcal P_2\ar[dl]\ar[ll]\\
P_{even}&&P_2\ar[ll]
}$$

But this shows, via the Weingarten computations from chapters 8 and 10 above, that the laws of asymptotic characters for our quantum groups are:
$$\xymatrix@R=20pt@C=22pt{
&\mathfrak B_t\ar@{-}[rr]\ar@{-}[dd]&&\Gamma_t\ar@{-}[dd]\\
\beta_t\ar@{-}[rr]\ar@{-}[dd]\ar@{-}[ur]&&\gamma_t\ar@{-}[dd]\ar@{-}[ur]\\
&B_t\ar@{-}[rr]\ar@{-}[uu]&&G_t\ar@{-}[uu]\\
b_t\ar@{-}[uu]\ar@{-}[ur]\ar@{-}[rr]&&g_t\ar@{-}[uu]\ar@{-}[ur]
}$$

Regarding now the last assertion, consider the main central limiting theorems in classical and free probability, which are as follows, with $R,C$ standing for real and complex, $CP$ standing for compound Poisson, and $F$ standing for free:
$$\xymatrix@R=20pt@C=1pt{
&FCCPLT\ar@{-}[rr]\ar@{-}[dd]&&FCCLT\ar@{-}[dd]\\
FRCPLT\ar@{-}[rr]\ar@{-}[dd]\ar@{-}[ur]&&FCLT\ar@{-}[dd]\ar@{-}[ur]\\
&CCPLT\ar@{-}[rr]\ar@{-}[uu]&&CCLT\ar@{-}[uu]\\
RCPLT\ar@{-}[uu]\ar@{-}[ur]\ar@{-}[rr]&&CLT\ar@{-}[uu]\ar@{-}[ur]
}$$

Once again as explained in chapters 8 and 10 above, the limiting characters come from the categories of partitions given above, and so are the laws given above.
\end{proof}

\index{noncommutative geometry}
\index{free torus}

In what regards now noncommutative geometry, the idea is that our 8 main quantum groups correspond to the 4 possible ``abstract noncommutative geometries'', in the strongest possible sense, which are the real/complex, classical/free ones. 

\bigskip

In order to explain this, consider the following diagram, consisting of main quantum spheres, that we know from before, and of the corresponding tori: 
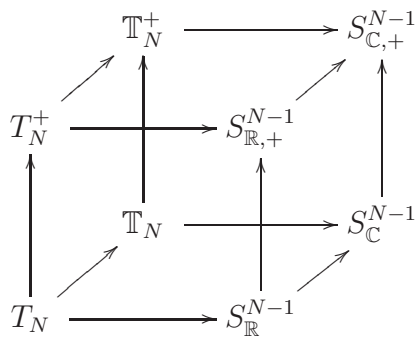
\smallskip
$$\xymatrix@R=15pt@C=12pt{
&\ \mathbb T_N^+\ \ar[rr]&&S^{N-1}_{\mathbb C,+}\\
\ T_N^+\ \ar[rr]\ar[ur]&&S^{N-1}_{\mathbb R,+}\ar[ur]\\
&\ \mathbb T_N\ \ar[rr]\ar[uu]&&S^{N-1}_\mathbb C\ar[uu]\\
\ T_N\ \ar[uu]\ar[ur]\ar[rr]&&S^{N-1}_\mathbb R\ar[uu]\ar[ur]
}$$

These 4+4 spheres and tori add to the 4+4 unitary and reflection groups that we have, which form as well a cubic diagram, as follows:
$$\xymatrix@R=20pt@C=20pt{
&K_N^+\ar[rr]&&U_N^+\\
H_N^+\ar[rr]\ar[ur]&&O_N^+\ar[ur]\\
&K_N\ar[rr]\ar[uu]&&U_N\ar[uu]\\
H_N\ar[uu]\ar[ur]\ar[rr]&&O_N\ar[uu]\ar[ur]
}$$

Thus, we have a total of 16 basic geometric objects. But these objects can be arranged, in an obvious way, into 4 quadruplets of type $(S,T,U,K)$, consisting a sphere $S$, a torus $T$, a unitary group $U$, and a reflection group $K$, with relations between them, as follows:
$$
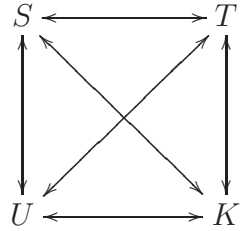
\xymatrix@R=60pt@C=60pt{
S\ar[r]\ar[d]\ar[dr]&T\ar[l]\ar[d]\ar[dl]\\
U\ar[u]\ar[ur]\ar[r]&K\ar[l]\ar[ul]\ar[u]
}$$

To be more precise, we obtain in this way the quadruplets $(S,T,U,K)$
corresponding to the real/complex, classical/free geometries. As mentioned above, it is possible to do some axiomatization and classification work here, with the conclusion that, under strong combinatorial axioms, including easiness, these 4 geometries are the only ones.

\bigskip

Summarizing, our Ground Zero classification theorem for the compact quantum groups is compatible with both probability theory, and noncommutative geometry.

\section*{11e. Exercises} 

There has been a lot of theory in this chapter, often explained quite briefly, and our exercises here will be mostly about details on all this. First, we have:

\begin{exercise}
Prove that the orthogonal easy groups are
$$\xymatrix@R=30pt@C=80pt{
H_N\ar[r]&O_N\\
S_N'\ar[u]&B_N'\ar[u]\\
S_N\ar[r]\ar[u]&B_N\ar[u]}$$
where $S_N'=S_N\times\mathbb Z_2$ and $B_N'=B_N\times\mathbb Z_2$.
\end{exercise}

In the uniform case the classification was explained in the above, leading to the 4 corners of the square, as the unique solutions. The problem is that of understanding what happens to this classification when lifting the uniformity assumption.

\begin{exercise}
Find two distinct easy liberations 
$$B_N'^+\subset B_N''^+$$
of the group $B_N'=B_N\times\mathbb Z_2$.
\end{exercise}

The problem here is that of reformulating the question in terms of categories of partitions, and then producing 2 distinct categories of partitions which do the job.

\begin{exercise}
Prove that the orthogonal easy free quantum groups are
$$\xymatrix@R=1pt@C=100pt{
H_N^+\ar[r]&O_N^+\\
\ &\\
\ &\\
&B_N''^+\ar[uuu]\\
S_N'^+\ar[uuuu]&\\
&B_N'^+\ar[uu]\\
\ &\\
\ &\\
S_N^+\ar[r]\ar[uuuu]&B_N^+\ar[uuu]}$$
where $S_N'^+=S_N^+\times\mathbb Z_2$, and where $B_N'^+\subset B_N''^+$ are easy liberations of $B_N'=B_N\times\mathbb Z_2$.
\end{exercise}

As before, in the uniform case the classification was explained in the above, leading to the 4 corners of the square, as the unique solutions. The problem is that of understanding what happens to this classification when lifting the uniformity assumption.

\chapter{The standard cube}

\section*{12a. Face results}

We discuss here a number of more specialized classification results, for the twistable easy quantum groups, and for more general intermediate quantum groups as follows:
$$H_N\subset G\subset U_N^+$$

The general idea will be as before, namely that of viewing our quantum group as sitting inside the standard cube, discussed in chapter 11:
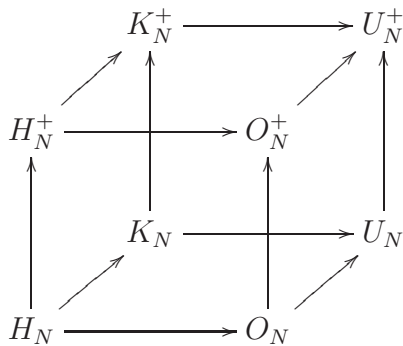
$$\xymatrix@R=20pt@C=20pt{
&K_N^+\ar[rr]&&U_N^+\\
H_N^+\ar[rr]\ar[ur]&&O_N^+\ar[ur]\\
&K_N\ar[rr]\ar[uu]&&U_N\ar[uu]\\
H_N\ar[uu]\ar[ur]\ar[rr]&&O_N\ar[uu]\ar[ur]
}$$

\index{standard cube}

We will be interested in several questions, as follows:

\bigskip

(1) Face results, in the easy case. The problem here is that of classifying the easy quantum groups lying on each of the 6 faces of the cube. Thus, we would like to solve the following intermediate easy quantum group problems:
$$H_N\subset G\subset U_N\quad,\quad 
H_N\subset G\subset O_N^+$$
$$H_N\subset G\subset K_N^+\quad,\quad 
H_N^+\subset G\subset U_N^+$$
$$K_N\subset G\subset U_N^+\quad,\quad
U_N\subset G\subset U_N^+$$

\smallskip

(2) Edge results, in the easy case. This is a question which is easier, amounting in solving 12 intermediate easy quantum group problems, one for each edge of the cube.

\bigskip

(3) Face and edge results, in the general non-easy case. Here the problems are quite difficult, but we will discuss some strategies, in order to deal with them.

\bigskip

Let us first discuss the classification in the easy case, for the lower and upper faces of the cube. Following Tarrago-Weber \cite{twe}, in the uniform case, the result is as follows:

\index{easy group}
\index{free quantum group}

\begin{theorem}
The classical and free uniform twistable easy quantum groups are
$$
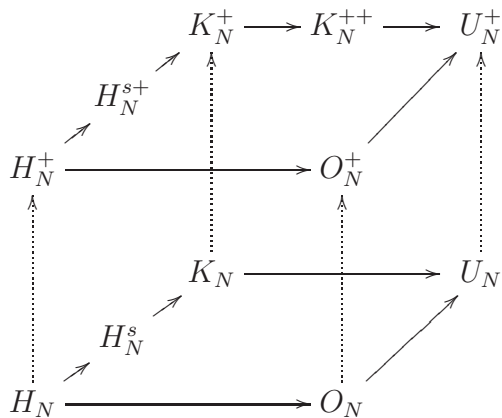
\xymatrix@R=7pt@C=7pt{
&&K_N^+\ar[rr]&&K_N^{++}\ar[rr]&&\ U_N^+\ \\
&H_N^{s+}\ar[ur]&&&&\\
H_N^+\ar[rrrr]\ar[ur]&&&&O_N^+\ar[uurr]\\
\\
&&K_N\ar[rrrr]\ar@.[uuuu]&&&&\ U_N\ \ar@.[uuuu]\\
&H_N^s\ar[ur]&&&&\\
H_N\ar@.[uuuu]\ar[ur]\ar[rrrr]&&&&O_N\ar@.[uuuu]\ar[uurr]
\\
}$$
where $H_s=\mathbb Z_s\wr S_N$, $H_N^{s+}=\mathbb Z_s\wr_*S_N^+$ with $s=4,6,8\ldots$\,, and where $K_N^+=\widetilde{K_N^+}$. 
\end{theorem}

\begin{proof}
The idea here is that of jointly classifying the ``classical'' categories of partitions $\mathcal P_2\subset D\subset P_{even}$, and the ``free'' ones $\mathcal{NC}_2\subset D\subset NC_{even}$, under the assumption that the category is stable under the operation which consists in removing blocks:

\medskip

(1) In the classical case, the new solutions appear on the edge $H_N\subset K_N$, and are the complex reflection groups $H_s=\mathbb Z_s\wr S_N$ with $s=4,6,8\ldots$\,, the cases $s=2,\infty$ corresponding respectively to $H_N,K_N$. 

\medskip

(2) In the free case we obtain as new solutions the standard liberattions of these groups, namely the quantum groups $H_N^{s+}=\mathbb Z_s\wr_*S_N^+$ with $s=4,6,8\ldots$\,, and we have as well an extra solution, appearing on the edge $K_N^+\subset U_N^+$, which is the free complexification $\widetilde{K_N^+}$ of the quantum group $K_N^+$, which is easy, and bigger than $K_N^+$.
\end{proof}

The above result can be generalized, by lifting both the uniformity and twistablility assumptions, and the result here, which is more technical, is explained in \cite{twe}.

\bigskip

We will be back to this at the end of the present chapter, with an extension of the above result, and with some classification results as well for the twists.

\bigskip

Another key result is the one of Raum-Weber \cite{rwe}, dealing with the front face of the standard cube, the orthogonal one. We first have the folowing result:

\index{real quantum group}

\begin{proposition}
The easy quantum groups $H_N\subset G\subset O_N^+$ are as follows,
$$\xymatrix@R=10mm@C=35mm{
H_N^+\ar[r]&O_N^+\\
H_N^{[\infty]}\ar@.[u]&O_N^*\ar[u]\\
H_N\ar@.[u]\ar[r]&O_N\ar[u]}$$
with the dotted arrows indicating that we have intermediate quantum groups there.
\end{proposition}

\begin{proof}
This is a key result in the classification of easy quantum groups, whose proof is quite technical, the idea being as follows:

\medskip

(1) We have a first dichotomy concerning the quantum groups in the statement, namely $H_N\subset G\subset O_N^+$, which must fall into one of the following two classes:
$$O_N\subset G\subset O_N^+$$
$$H_N\subset G\subset H_N^+$$

This dichotomy comes indeed from the early classification results for the easy quantum groups, from \cite{bc1}, \cite{bsp}, \cite{bv1}, whose proofs are quite elementary. 

\medskip

(2) In addition to this, these early classification results solve as well the first problem, namely $O_N\subset G\subset O_N^+$, with $G=O_N^*$ being the unique non-trivial solution.

\medskip

(3) We have then a second dichotomy, concerning the quantum groups which are left, namely $H_N\subset G\subset H_N^+$, which must fall into one of the following two classes:
$$H_N\subset G\subset H_N^{[\infty]}$$
$$H_N^{[\infty]}\subset G\subset H_N^+$$

This comes indeed from various papers, and more specifically from the final classification paper of Raum and Weber \cite{rwe}, where the quantum groups $S_N\subset G\subset H_N^+$ with $G\not\subset H_N^{[\infty]}$ were classified, and shown to contain $H_N^{[\infty]}$. For full details, we refer to \cite{rwe}.
\end{proof}

Summarizing, in order to deal with the front face of the main cube, we are left with classifying the following intermediate easy quantum groups:
$$H_N\subset G\subset H_N^{[\infty]}$$
$$H_N^{[\infty]}\subset G\subset H_N^+$$

Regarding the second case, namely $H_N^{[\infty]}\subset G\subset H_N^+$, the result here, by Raum-Weber \cite{rwe}, which is quite technical, but has a simple formulation, is as follows:

\begin{proposition}
Let $H_N^{[r]}\subset H_N^+$ be the easy quantum group coming from:
$$\pi_r=\ker\begin{pmatrix}1&\ldots&r&r&\ldots&1\\1&\ldots&r&r&\ldots&1\end{pmatrix}$$
We have then inclusions of quantum groups as follows,
$$H_N^+=H_N^{[1]}\supset H_N^{[2]}\supset H_N^{[3]}\supset\ldots\ldots\supset H_N^{[\infty]}$$
and we obtain in this way all the intermediate easy quantum groups 
$$H_N^{[\infty]}\subset G\subset H_N^+$$
satisfying the assumption $G\neq H_N^{[\infty]}$.
\end{proposition}

\begin{proof}
Once again, this is something technical, and we refer here to \cite{rwe}.
\end{proof}

It remains to discuss the easy quantum groups $H_N\subset G\subset H_N^{[\infty]}$, with the endpoints $G=H_N,H_N^{[\infty]}$ included. Once again, we follow here \cite{rwe}. First, we have:

\begin{definition}
A discrete group generated by real reflections, $g_i^2=1$,
$$\Gamma=<g_1,\ldots,g_N>$$
is called uniform if each $\sigma\in S_N$ produces a group automorphism, $g_i\to g_{\sigma(i)}$.
\end{definition}

Consider now a uniform reflection group, as follows:
$$\mathbb Z_2^{*N}\to\Gamma\to\mathbb Z_2^N$$

As explained by Raum-Weber in \cite{rwe}, we can associate to this group a family of subsets $D(k,l)\subset P(k,l)$, which form a category of partitions, as follows:
$$D(k,l)=\left\{\pi\in P(k,l)\Big|\ker\binom{i}{j}\leq\pi\implies g_{i_1}\ldots g_{i_k}=g_{j_1}\ldots g_{j_l}\right\}$$

Observe that we have inclusions of categories of partitions as follows, coming respectively from $\eta\in D$, and from the quotient map $\Gamma\to\mathbb Z_2^N$:
$$P_{even}^{[\infty]}\subset D\subset P_{even}$$

Conversely, consider a category of partitions as follows:
$$P_{even}^{[\infty]}\subset D\subset P_{even}$$

We can associate to it a uniform reflection group $\mathbb Z_2^{*N}\to\Gamma\to\mathbb Z_2^N$, as follows:
$$\Gamma=\left\langle g_1,\ldots g_N\Big|g_{i_1}\ldots g_{i_k}=g_{j_1}\ldots g_{j_l},\forall i,j,k,l,\ker\binom{i}{j}\in D(k,l)\right\rangle$$

As explained by Raum-Weber in \cite{rwe}, the correspondences $\Gamma\to D$ and $D\to\Gamma$ constructed above are bijective, and inverse to each other, at $N=\infty$. 

\bigskip

We have in fact the following result, from \cite{rwe}:

\begin{proposition}
We have correspondences between:
\begin{enumerate}
\item Uniform reflection groups $\mathbb Z_2^{*\infty}\to\Gamma\to\mathbb Z_2^\infty$.

\item Categories of partitions $P_{even}^{[\infty]}\subset D\subset P_{even}$.

\item Easy quantum groups $G=(G_N)$, with $H_N^{[\infty]}\supset G_N\supset H_N$.
\end{enumerate}
\end{proposition}

\begin{proof}
This is something quite technical, which follows along the lines of the above discussion. As an illustration, if we denote by $\mathbb Z_2^{\circ N}$ the quotient of $\mathbb Z_2^{*N}$ by the relations of type $abc=cba$ between the generators, we have the following correspondences:
$$\xymatrix@R=15mm@C=15mm{
\mathbb Z_2^N\ar@{~}[d]&\mathbb Z_2^{\circ N}\ar[l]\ar@{~}[d]&\mathbb Z_2^{*N}\ar[l]\ar@{~}[d]\\
H_N\ar[r]&H_N^*\ar[r]&H_N^{[\infty]}}$$

More generally, for any $s\in\{2,4,\ldots,\infty\}$, the quantum groups $H_N^{(s)}\subset H_N^{[s]}$ constructed in \cite{bc1} come from the quotients of $\mathbb Z_2^{\circ N}\leftarrow\mathbb Z_2^{*N}$ by the relations $(ab)^s=1$. See \cite{rwe}.
\end{proof}

We can now formulate a final classification result, due to Raum-Weber \cite{rwe}, as follows:

\index{easy quantum group}
\index{real quantum group}
\index{quantum reflection group}

\begin{theorem}
The easy quantum groups $H_N\subset G\subset O_N^+$ are as follows,
$$\xymatrix@R=4mm@C=50mm{
H_N^+\ar[r]&O_N^+\\
H_N^{[r]}\ar[u]\\
H_N^{[\infty]}\ar[u]&O_N^*\ar[uu]\\
H_N^\Gamma\ar[u]\\
H_N\ar[u]\ar[r]&O_N\ar[uu]}$$
with the family $H_N^\Gamma$ covering $H_N,H_N^{[\infty]}$, and with the series $H_N^{[r]}$ covering $H_N^+$.
\end{theorem}

\begin{proof}
This follows indeed from Proposition 12.2, Proposition 12.3 and Proposition 12.5 above. For further details, we refer to the paper of Raum and Weber \cite{rwe}.
\end{proof}

All the above is quite technical, and can be extended as well, as to cover all the orthogonal easy quantum groups, $S_N\subset G\subset O_N^+$. For details here, we refer to \cite{rwe}.

\section*{12b. Edge results}

Another interesting result, dealing this time with the unitary edge of the standard cube, is the one obtained by Mang-Weber in \cite{mwe}. To be more precise, the problem here is that of classifying the intermediate easy quantum groups as follows:
$$U_N\subset G\subset U_N^+$$

A first construction of such quantum groups is as follows:

\index{cyclic matrix model}

\begin{proposition}
Associated to any $r\in\mathbb N$ is the easy quantum group 
$$U_N\subset U_N^{(r)}\subset U_N^+$$
coming from the category $\mathcal P_2^{(r)}$ of matching pairings having the property that 
$$\#\circ=\#\bullet(r)$$
holds between the legs of each string. These quantum groups have the following properties:
\begin{enumerate}
\item At $r=1$ we obtain the usual unitary group, $U_N^{(1)}=U_N$.

\item At $r=2$ we obtain the half-classical unitary group, $U_N^{(2)}=U_N^*$.

\item For any $r|s$ we have an embedding $U_N^{(r)}\subset U_N^{(s)}$.

\item In general, we have an embedding $U_N^{(r)}\subset U_N^r\rtimes\mathbb Z_r$.

\item We have as well a cyclic matrix model $C(U_N^{(r)})\subset M_r(C(U_N^r))$.

\item In this latter model, $\int_{U_N^{(r)}}$ appears as the restriction of $tr_r\otimes\int_{U_N^r}$.
\end{enumerate}
\end{proposition}

\begin{proof}
This is something quite compact, summarizing various findings from \cite{bb4}, and from \cite{mwe}. Here are a few brief explanations on all this:

\medskip

(1) This is clear from $\mathcal P_2^{(1)}=\mathcal P_2$, and from the Brauer theorem \cite{bra}.

\medskip

(2) This is because $\mathcal P_2^{(2)}$ is generated by the partitions with implement the relations $abc=cba$ between the variables $\{u_{ij},u_{ij}^*\}$, used in \cite{bdu} for constructing $U_N^*$.

\medskip

(3) This simply follows from $\mathcal P_2^{(s)}\subset\mathcal P_2^{(r)}$, by functoriality.

\medskip

(4) This is the original definition of $U_N^{(r)}$, from \cite{bb4}. We refer to \cite{bb4} for the formula of the embedding, and to \cite{mwe} for the compatibility with the Tannakian definition.

\medskip

(5) This is also from \cite{bb4}, more specifically it is an alternative definition for $U_N^{(r)}$.

\medskip

(6) Once again, this is something from \cite{bb4}, and we will be back to it.
\end{proof}

Let us discuss now the second known construction of unitary quantum groups, from \cite{mwe}. This construction uses an additive semigroup $D\subset\mathbb N$, but as pointed out there, using instead the complementary set $C=\mathbb N-D$ leads to several simplifications. So, let us call ``cosemigroup'' any subset $C\subset\mathbb N$ which is complementary to an additive semigroup, $x,y\notin C\implies x+y\notin C$. The construction from \cite{mwe} is then:

\begin{proposition}
Associated to any cosemigroup $C\subset\mathbb N$ is the easy quantum group 
$$U_N\subset U_N^C\subset U_N^+$$
coming from the category $\mathcal P_2^C\subset P_2^{(\infty)}$ of pairings having the property 
$$\#\circ-\#\bullet\in C$$
between each two legs colored $\circ,\bullet$ of two strings which cross. We have:
\begin{enumerate}
\item For $C=\emptyset$ we obtain the quantum group $U_N^+$.

\item For $C=\{0\}$ we obtain the quantum group $U_N^\times$.

\item For $C=\{0,1\}$ we obtain the quantum group $U_N^{**}$.

\item For $C=\mathbb N$ we obtain the quantum group $U_N^{(\infty)}$.

\item For $C\subset C'$ we have an inclusion $U_N^{C'}\subset U_N^C$.

\item Each quantum group $U_N^C$ contains each quantum group $U_N^{(r)}$.
\end{enumerate}
\end{proposition}

\begin{proof}
Once again this is something very compact, coming from recent work in \cite{mwe}, with our convention that the semigroup $D\subset\mathbb N$ which is used there is replaced here by its complement $C=\mathbb N-D$. Here are a few explanations on all this:

\medskip

(1) The assumption $C=\emptyset$ means that the condition $\#\circ-\#\bullet\in C$ can never be applied. Thus, the strings cannot cross, we have $\mathcal P_2^\emptyset=\mathcal{NC}_2$, and so $U_N^\emptyset=U_N^+$.

\medskip

(2) As explained in \cite{mwe}, here we obtain indeed the quantum group $U_N^\times$, constructed by using the relations $ab^*c=cb^*a$, with $a,b,c\in\{u_{ij}\}$. 

\medskip

(3) This is also explained in \cite{mwe}, with $U_N^{**}$ being the quantum group from \cite{bb4}, which is the biggest whose full projective version, in the sense there, is classical. 

\medskip

(4) Here the assumption $C=\mathbb N$ simply tells us that the condition $\#\circ-\#\bullet\in C$ in the statement is irrelevant. Thus, we have $\mathcal P_2^\mathbb N=\mathcal P_2^{(\infty)}$, and so $U_N^\mathbb N=U_N^{(\infty)}$.

\medskip

(5) This is clear by functoriality, because $C\subset C'$ implies $\mathcal P_2^{C}\subset\mathcal P_2^{C'}$.

\medskip

(6) This is clear from definitions, and from Proposition 12.7 above.
\end{proof}

We have the following key result, from Mang-Weber \cite{mwe}:

\begin{theorem}
The easy quantum groups $U_N\subset G\subset U_N^+$ are as follows,
$$U_N\subset\{U_N^{(r)}\}\subset\{U_N^C\}\subset U_N^+$$
with the series covering $U_N$, and the family covering $U_N^+$.
\end{theorem}

\begin{proof}
This is something non-trivial, and we refer here to \cite{mwe}. The general idea is that $U_N^{(\infty)}$ produces a dichotomy for the quantum groups in the statement, and this leads, via combinatorial computations, to the series and the family. See \cite{mwe}.
\end{proof}

Observe that there is an obvious similarity here with the dichotomy for the liberations of $H_N$, coming from the work of Raum-Weber \cite{rwe}, explained in the above. To be more precise, the above-mentioned classification results for the liberations of $H_N$ and the liberations of $U_N$ have some obvious similarity between them. We have indeed a family followed by a series, and a series followed by a family.

\bigskip

\index{contravariant duality}

All this suggests the existence of a general ``contravariant duality'' between these quantum groups, as follows:
$$\xymatrix@R=50pt@C=50pt{
U_N\ar[r]\ar@.[d]&U_N^{(r)}\ar[r]\ar@.[d]&U_N^C\ar[r]\ar@.[d]&U_N^+\ar@.[d]\\
H_N^+\ar@.[u]&H_N^{[r]}\ar[l]\ar@.[u]&H_N^\Gamma\ar[l]\ar@.[u]&H_N\ar[l]\ar@.[u]
}$$

At the first glance, this might sound a bit strange. Indeed, we have some natural and well-established correspondences $H_N\leftrightarrow U_N$ and $H_N^+\leftrightarrow U_N^+$, obtained in one sense by taking the real reflection subgroup, $H=U\cap H_N^+$, and in the other sense by setting $U=<H,U_N>$. Thus, our proposal of duality seems to go the wrong way.

\bigskip

On the other hand, obvious as well is the fact that these correspondences $H_N\leftrightarrow U_N$ and $H_N^+\leftrightarrow U_N^+$ cannot be extended as to map the series to the series, and the family to the family, because the series/families would have to be ``inverted'', in order to do so. 

\bigskip

Thus, we are led to the above contravariant duality conjecture. In practice, the idea would be that of constructing the duality by a clever use of the interesection and generation operations $\cap$ and $<\,,>$, but it is not clear so far on how to do this.

\bigskip

Following \cite{ba5}, let us discuss now what happens inside the standard cube, first in the easy case, and then in general. The idea here will be that of carefully looking at the Ground Zero theorem from chapter 11 above, and removing the easiness axiom there.

\bigskip

This is something quite technical, and in order to do so, let us start with a study of the easy case, with the goal of improving the Ground Zero theorem, by relaxing a bit the orientability axiom there. Let us start with the following definition:

\begin{definition}
A twistable easy quantum group $H_N\subset G_N\subset U_N^+$ is called ``bi-oriented'' if the diagram
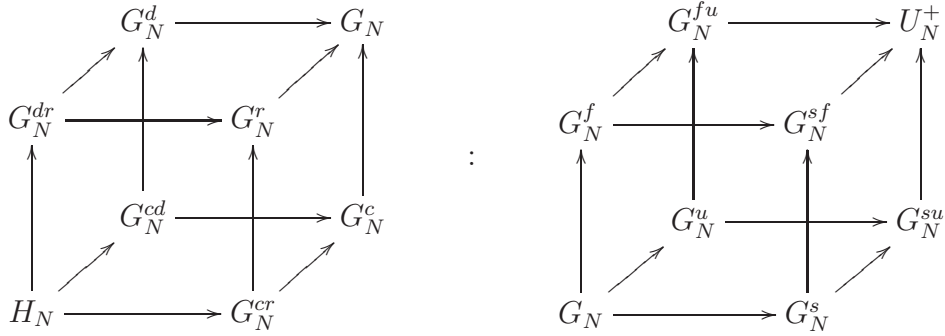
$$\xymatrix@R=17pt@C=17pt{
&G_N^d\ar[rr]&&G_N\\
G_N^{dr}\ar[rr]\ar[ur]&&G_N^r\ar[ur]\\
&G_N^{cd}\ar[rr]\ar[uu]&&G_N^c\ar[uu]\\
H_N\ar[uu]\ar[ur]\ar[rr]&&G_N^{cr}\ar[uu]\ar[ur]}$$
as well as the diagram
$$\xymatrix@R=17pt@C=17pt{
&G_N^{fu}\ar[rr]&&U_N^+\\
G_N^f\ar[rr]\ar[ur]&&G_N^{sf}\ar[ur]\\
&G_N^u\ar[rr]\ar[uu]&&G_N^{su}\ar[uu]\\
G_N\ar[uu]\ar[ur]\ar[rr]&&G_N^s\ar[uu]\ar[ur]
}$$
are intersection and easy generation diagrams.
\end{definition}

Observe that the first diagram is automatically an intersection diagram, and that the second diagram is automatically an easy generation diagram.

\bigskip

The question of replacing the slicing axiom with the bi-orientability condition makes sense. In fact, we can even talk about weaker axioms, as follows:

\index{orientability}

\begin{definition}
An easy quantum group $H_N\subset G_N\subset U_N^+$ is called ``oriented'' if
$$G_N=\{G_N^{cd},G_N^{cr},G_N^{dr}\}$$
$$G_N=G_N^{fs}\cap G_N^{fu}\cap G_N^{su}$$
and ``weakly oriented'' if the following weaker conditions hold,
$$G_N=\{G_N^c,G_N^d,G_N^r\}$$
$$G_N=G_N^f\cap G_N^s\cap G_N^u$$
where the various versions are those in chapter 11 above.
\end{definition}

In order to prove now the uniqueness of the main 8 easy quantum groups, in the bi-orientable case, we can still proceed as in the proof of the Ground Zero theorem, but we are no longer allowed to use the coordinate system there, based at $O_N$. 

\bigskip

To be more precise, we must use the 2 coordinate systems highlighted below, both taken in some weak sense, weaker than the slicing:
$$
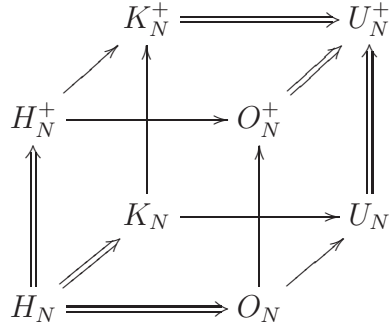
\xymatrix@R=18pt@C=18pt{
&K_N^+\ar@=[rr]&&U_N^+\\
H_N^+\ar[rr]\ar[ur]&&O_N^+\ar@=[ur]\\
&K_N\ar[rr]\ar[uu]&&U_N\ar@=[uu]\\
H_N\ar@=[uu]\ar@=[ur]\ar@=[rr]&&O_N\ar[uu]\ar[ur]
}$$

Skipping some details here, all this is viable, by using the known ``edge results'' surveyed above, along with the key fact, coming also from the above edge results, that the quantum group $H_N^{[\infty]}$ from \cite{rwe} has no orthogonal counterpart. 

\bigskip

\index{Ground Zero theorem}

Thus, we obtain in principle some improvements of the Ground Zero theorem, under the bi-orientability assumption, and more generally under the orientability assumption. As for the weak orientability assumption, the situation here is more tricky, because we would need full ``face results'', which are not available yet.

\section*{12c. Beyond easiness}

Let us discuss now the general, non-easy case. In order to do so, we must find extensions of the notions of uniformity, twistability and orientability. Regarding the notion of uniformity, the definition here is straightforward, with only some minor changes with respect to the easy quantum group case, as follows:

\index{uniformity}

\begin{definition}
A series $G=(G_N)$ of closed subgroups $G_N\subset U_N^+$ is called:
\begin{enumerate}
\item Weakly uniform, if  for any $N\in\mathbb N$ we have $G_{N-1}=G_N\cap U_{N-1}^+$, with respect to the embedding $U_{N-1}^+\subset U_N^+$ given by $u\to diag(u,1)$.

\item Uniform, if for any $N\in\mathbb N$ we have $G_{N-1}=G_N\cap U_{N-1}^+$, with respect to the $N$ possible embeddings $U_{N-1}^+\subset U_N^+$, of type $u\to diag(u,1)$.
\end{enumerate}
\end{definition}

In the easy quantum group case these two notions coincide, due to the presence of the symmetric group $S_N\subset G_N$, which acts on everything, and allows one to pass from one embedding $U_{N-1}^+\subset U_N^+$ to another. In general, these two notions do not coincide.

\bigskip

Regarding the examples, in the classical case we have substantially more examples than in the easy case, obtained by using the determinant, and its powers:

\index{complex reflection group}

\begin{proposition}
The following compact groups are uniform,
\begin{enumerate}
\item The complex reflection groups 
$$H_N^{sd}=\left\{g\in\mathbb Z_s\wr S_N\Big|(\det g)^d=1\right\}$$
for any values of the parameters $s\in\{1,2,\ldots,\infty\}$ and $d\in\mathbb N$, $d|s$,

\item The orthogonal group $O_N$, the special orthogonal group $SO_N$, and the series 
$$U_N^d=\left\{g\in U_N\Big|(\det g)^d=1\right\}$$
of modified unitary groups, with $s\in\{1,2,\ldots,\infty\}$,
\end{enumerate}
and so are the bistochastic versions of these groups.
\end{proposition}

\begin{proof}
Both these assertions are clear from definitions, the idea being as follows:

\medskip

(1) These groups are well-known objects in finite group theory, and more precisely form the series of complex reflection groups, and generalize the groups $H_N^s$ from chapter 10 above, which appear at $d=s$. See Shephard-Todd \cite{sto}. 

\medskip

(2) These groups are well-known as well, in compact Lie group theory, with $U_N^1$ being equal to $SU_N$, and with $U_N^\infty$ being by definition $U_N$ itself.
\end{proof}

In the free case now, corresponding to the condition $S_N^+\subset G_N\subset U_N^+$, it is widely believed that the only examples are the easy ones. A precise conjecture in this sense, which is a bit more general, valid for any $G_N\subset U_N^+$, states that we should have:
$$<G_N,S_N^+>=\{G_N',S_N^+\}$$

Here $G_N'$ denotes as usual the easy envelope of $G_N$, and $\{\,,\}$ is an easy generation operation. This conjecture is probably something quite difficult.

\bigskip

Now back to our questions, we have definitely no new examples in the free case. So, the basic examples will be those that we previously met, namely:

\begin{proposition}
The following free quantum groups are uniform,
\begin{enumerate}
\item Liberations $H_N^{s+}=\mathbb Z_s\wr_*S_N^+$ of the complex reflection groups $H_N^s=\mathbb Z_s\wr S_N$,

\item Liberations $O_N^+,U_N^+$ of the continuous groups $O_N,U_N$,
\end{enumerate}
and so are the bistochastic versions of these quantum groups.
\end{proposition}

\begin{proof}
This is something that we basically know, with the uniformity check for $H_N^{s+}$ being the same as for $S_N^+,H_N^+,K_N^+$, which appear at $s=1,2,\infty$.
\end{proof}

We would need now a second axiom, such as the twistability condition $T_N\subset G_N$ used in chapter 11. However, if we carefully look at Proposition 12.13, and we want to have as examples the groups there, a condition of type $A_N\subset G_N$ would be more appropriate. 

\bigskip

In order to comment on this dillema, let us recall from chapter 11 that, in view of the considerations there, ``taking the bistochastic version'' is a bad direction, geometrically speaking. But the operations ``taking the diagonal torus'' and ``taking the special version'', that we are currently discussing, are bad too. Thus, we have 3 bad directions, and so we end up with a cube formed by these bad 3 directions, as follows:

\begin{proposition}
We have the following diagram of finite groups,
$$
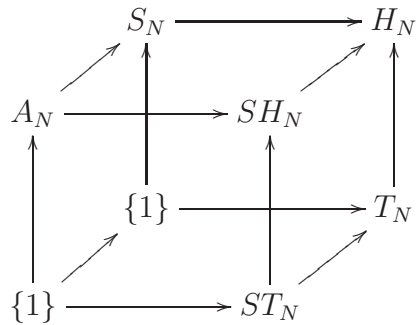
\xymatrix@R=18pt@C=18pt{
&S_N\ar[rr]&&H_N\\
A_N\ar[rr]\ar[ur]&&SH_N\ar[ur]\\
&\{1\}\ar[rr]\ar[uu]&&T_N\ar[uu]\\
\{1\}\ar[uu]\ar[ur]\ar[rr]&&ST_N\ar[uu]\ar[ur]
}$$
obtained from $H_N$ by taking bistochastic, special and diagonal versions.
\end{proposition}

\begin{proof}
This is clear from definitions, with the operations of taking bistochastic versions, special versions and diagonal subgroups corresponding respectively to going left, backwards, and downwards, with respect to the coordinates in the statement.
\end{proof}

Now back to our classification questions, the vertices of the above cube are all interesting groups, and assuming that the quantum groups $G_N\subset U_N^+$ that we want to classify contain any of them is something quite natural. 

\bigskip

Let us just select here three such conditions, as follows:

\index{twistable quantum group}
\index{homogeneous quantum group}
\index{half-homogeneous quantum group}

\begin{definition}
A closed subgroup $G_N\subset U_N^+$ is called:
\begin{enumerate}
\item Twistable, if $T_N\subset G_N$.

\item Homogeneous, if $S_N\subset G_N$.

\item Half-homogeneous, if $A_N\subset G_N$.
\end{enumerate}
\end{definition}

As before with the notion of uniformity, things simplify in the easy case. To be more precise, any easy quantum group is automatically homogeneous, and half-homogeneous as well. As for the notion of twistability, this coincides with the old one.

\bigskip

Let us go ahead now, and formulate our third and last definition, regarding the orientability axiom. Things are quite tricky here, and we must start as follows:

\index{classical version}
\index{discrete version}
\index{real version}
\index{free version}
\index{smooth version}
\index{unitary version}

\begin{definition}
Associated to any closed subgroup $G_N\subset U_N^+$ are its classical, discrete and real versions, given by
$$G_N^c=G_N\cap U_N$$
$$G_N^d=G_N\cap K_N^+$$
$$G_N^r=G_N\cap O_N^+$$
as well as its free, smooth and unitary versions, given by
$$G_N^f=<G_N,H_N^+>$$
$$G_N^s=<G_N,O_N>$$
$$G_N^u=<G_N,K_N>$$
where $<\,,>$ is the usual, non-easy topological generation operation.
\end{definition}

Observe the difference, and notational clash, with some of the notions used in chapter 11. To be more precise, as explained in chapter 7, it is believed that we should have $\{\,,\}=<\,,>$, but this is not clear at all, and the problem comes from this.

\bigskip

A second issue comes when composing the above operations, and more specifically those involving the generation operation, once again due to the conjectural status of the formula $\{\,,\}=<\,,>$. Due to this fact, instead of formulating a result here, we have to formulate a second definition, complementary to Definition 12.7, as follows:

\begin{definition}
Associated to any closed subgroup $G_N\subset U_N^+$ are the mixes of its classical, discrete and real versions, given by
$$G_N^{cd}=G_N\cap K_N$$
$$G_N^{cr}=G_N\cap O_N^+$$
$$G_N^{dr}=G_N\cap H_N^+$$
as well as the mixes of its free, smooth and unitary versions, given by
$$G_N^{fs}=<G_N,O_N^+>$$
$$G_N^{fu}=<G_N,K_N^+>$$
$$G_N^{us}=<G_N,U_N>$$
where $<\,,>$ is the usual, non-easy topological generation operation.
\end{definition}

Now back to our orientation questions, the slicing and bi-orientability conditions lead us again into $\{\,,\}$ vs. $<\,,>$ troubles, and are therefore rather to be ignored. The orientability conditions from Definition 12.11, however, have the following analogue:

\index{orientability}

\begin{definition}
A closed subgroup $G_N\subset U_N^+$ is called ``oriented'' if
$$G_N=<G_N^{cd},G_N^{cr},G_N^{dr}>$$
$$G_N=G_N^{fs}\cap G_N^{fu}\cap G_N^{su}$$
and ``weakly oriented'' if the following conditions hold,
$$G_N=<G_N^c,G_N^d,G_N^r>$$
$$G_N=G_N^f\cap G_N^s\cap G_N^u$$
where the various versions are those in Definition 12.17 and Definition 12.18.
\end{definition}

With these notions, our claim is that some classification results are possible:

\bigskip

(1) In the classical case, we believe that the uniform, half-homogeneous, oriented groups are those in Proposition 12.13, with some bistochastic versions excluded. This is of course something quite heavy, well beyond easiness, with the potential tools available for proving such things coming from advanced finite group theory and Lie algebra theory.  Our uniformity axiom could play a key role here, when combined with \cite{sto}, in order to exclude all the exceptional objects which might appear on the way.

\bigskip

(2) In the free case, under similar assumptions, we believe that the solutions should be those in Proposition 12.14, once again with some bistochastic versions excluded. This is something heavy, too, related to the above-mentioned well-known conjecture $<G_N,S_N^+>=\{G_N',S_N^+\}$. Indeed, assuming that we would have such a formula, and perhaps some more formulae of the same type as well, we can in principle work out our way inside the cube, from the edge and face projections to $G_N$ itself, and in this process $G_N$ would become easy. This would be the straightforward strategy here.

\bigskip

(3) In the group dual case, the orientability axiom simplifies, because the group duals are discrete in our sense. We believe that the uniform, twistable, oriented group duals should appear as combinations of certain abelian groups, which appear in the classical case, with duals of varieties of real reflection groups, which appear in the real case. This is probably the easiest question in the present series, and the most reasonable one, to start with. However, there are no concrete results so far, in this direction.

\bigskip

We refer to \cite{ba5} and related papers for further comments, on all the above.

\section*{12d. Maximality questions}

Let us go back now to the standard cube, and to edge problems, but without the easiness assumption, this time. An interesting family of questions here is that of proving that the easy solutions of various edge problems are in fact the only ones, even in the non-easy case. We will see that these are several results and conjectures here.

\bigskip

We have the following result from \cite{bc+}, to start with:

\index{maximality}
\index{projective version}

\begin{theorem}
The following inclusions are maximal:
\begin{enumerate}
\item $\mathbb TO_N\subset U_N$.

\item $PO_N\subset PU_N$.
\end{enumerate}
\end{theorem}

\begin{proof}
In order to prove these results, consider as well the group $\mathbb TSO_N$. Observe that we have $\mathbb TSO_N=\mathbb TO_N$ if $N$ is odd. If $N$ is even the group $\mathbb TO_N$ has two connected components, with $\mathbb TSO_N$ being the component containing the identity. 

\medskip

Let us denote as well by $\mathfrak{so}_N,\mathfrak u_N$ the Lie algebras of $SO_N,U_N$.  It is well-known that $\mathfrak u_N$ consists of the matrices $M\in M_N(\mathbb C)$ satisfying $M^*=-M$, and that:
$$\mathfrak{so}_N=\mathfrak u_N\cap M_N(\mathbb R)$$

Also, it is easy to see that the Lie algebra of $\mathbb TSO_N$ is $\mathfrak{so}_N\oplus i\mathbb R$.

\medskip

\underline{Step 1}. Our first claim is that if $N\geq 2$, the adjoint representation of $SO_N$ on the space of real symmetric matrices of trace zero is irreducible.

\medskip

Let indeed $X \in M_N(\mathbb R)$ be symmetric with trace zero. We must prove that the following space consists of all the real symmetric matrices of trace zero:
$$V=span\left\{UXU^t\Big|U \in SO_N\right\}$$

We first prove that $V$ contains all the diagonal matrices of trace zero. Since we may diagonalize $X$ by conjugating with an element of $SO_N$, our space $V$ contains a nonzero diagonal matrix of trace zero. Consider such a matrix:
$$D=\begin{pmatrix}
d_1\\
&\ddots\\
&&d_N
\end{pmatrix}$$

We can conjugate this matrix by the following matrix:
$$\begin{pmatrix}
0&-1&0\\
1&0&0\\
0&0&I_{N-2}
\end{pmatrix}\in SO_N$$

We conclude that our space $V$ contains as well the following matrix: 
$$D'=\begin{pmatrix}
d_2\\
&d_1\\
&&d_3\\
&&&\ddots\\
&&&&d_N
\end{pmatrix}$$

More generally, we see that for any $1\leq i,j\leq N$ the diagonal matrix obtained from $D$ by interchanging $d_i$ and $d_j$ lies in $V$. Now since $S_N$ is generated by transpositions, it follows that $V$ contains any diagonal matrix obtained by permuting the entries of $D$.  But it is well-known that this representation of $S_N$ on the diagonal matrices of trace zero is irreducible, and hence $V$ contains all such diagonal matrices, as claimed.

\medskip

In order to conclude now, assume that $Y$ is an arbitrary real symmetric matrix of trace zero. We can find then an element $U\in SO_N$ such that $UYU^t$ is a diagonal matrix of trace zero.  But we then have $UYU^t \in V$, and hence also $Y\in V$, as desired.

\medskip

\underline{Step 2}. Our claim is that the inclusion $\mathbb TSO_N\subset U_N$ is maximal in the category of connected compact groups. 

\medskip

Let indeed $G$ be a connected compact group satisfying $\mathbb TSO_N\subset G\subset U_N$. Then $G$ is a Lie group. Let $\mathfrak g$ denote its Lie algebra, which satisfies:
$$\mathfrak{so}_N\oplus i\mathbb R\subset\mathfrak g\subset\mathfrak u_N$$

Let $ad_{G}$ be the action of $G$ on $\mathfrak g$ obtained by differentiating the adjoint action of $G$ on itself. This action turns $\mathfrak g$ into a $G$-module. Since $SO_N \subset G$, $\mathfrak g$ is also a $SO_N$-module.

Now if $G\neq\mathbb TSO_N$, then since $G$ is connected we must have $\mathfrak{so}_N\oplus i\mathbb{R}\neq\mathfrak g$. It follows from the real vector space structure of the Lie algebras $\mathfrak u_N$ and $\mathfrak{so}_N$ that
there exists a nonzero symmetric real matrix of trace zero $X$ such that:
$$iX\in\mathfrak g$$

We know that the space of symmetric real matrices of trace zero is an irreducible representation of $SO_N$ under the adjoint action. Thus $\mathfrak g$ must contain all such $X$, and hence $\mathfrak g=\mathfrak u_N$.  But since $U_N$ is connected, it follows that $G=U_N$.  

\medskip

\underline{Step 3}. Our claim is that the commutant of $SO_N$ in $ M_N(\mathbb C)$ is as follows:
\begin{enumerate}
\item $SO_2'
=\left\{\begin{pmatrix}
\alpha&\beta\\
-\beta&\alpha
\end{pmatrix}\Big|\alpha,\beta\in\mathbb C\right\}$.

\smallskip

\item If $N\geq 3$, $SO_N'=\{\alpha I_N|\alpha\in\mathbb C\}$.
\end{enumerate}  
 
\medskip

Indeed, at $N=2$ this is a direct computation. At $N\geq 3$ now, an element in $X\in SO_N'$ commutes with any diagonal matrix having exactly $N-2$ entries equal to $1$ and two entries equal to $-1$. Hence $X$ is a diagonal matrix. Now since $X$ commutes with any even permutation matrix and $N\geq 3$, it commutes in particular with the permutation matrix associated with the cycle $(i,j,k)$ for any $1<i<j<k$, and hence all the entries of $X$ are the same. We conclude that $X$ is a scalar matrix, as claimed.

\medskip 

\underline{Step 4}. Our claim is that the set of matrices with nonzero trace is dense in $SO_N$.

\medskip  

At $N=2$ this is clear, since the set of elements in $SO_2$ having a given trace is finite. So assume $N>2$, and let:
$$T\in SO_N\simeq SO(\mathbb R^N)\quad,\quad 
Tr(T)=0$$

Let $E\subset\mathbb R^N$ be a 2-dimensional subspace preserved by $T$, such that: 
$$T_{|E} \in SO(E)$$ 
 
Let $\varepsilon >0$ and let $S_\varepsilon \in SO(E)$ with $||T_{|E}-S_\varepsilon|| <\varepsilon$, and with $Tr(T_{|E}) \not= Tr(S_\varepsilon)$, in the $N=2$ case. Now define $T_\varepsilon\in SO(\mathbb R^N)=SO_N$ by:
$$T_{\varepsilon|E}=S_\varepsilon\quad,\quad
T_{\varepsilon|E^\perp}=T_{|E^\perp}$$

It is clear that we have the following estimate:
$$||T-T_\varepsilon|| \leq ||T_{|E}-S_\varepsilon||<\varepsilon$$

Also, we have the following computation:
$$Tr(T_\varepsilon)=Tr(S_\varepsilon)+Tr(T_{|E^\perp})\neq0$$

Thus, we have proved our claim.

\medskip

\underline{Step 5}. Our claim is that $\mathbb TO_N$ is the normalizer of $\mathbb TSO_N$ in $U_N$, i.e. is the subgroup of $U_N$ consisting of the unitaries $U$ for which, for all $X\in\mathbb TSO_N$:
$$U^{-1}XU \in\mathbb TSO_N$$

It is clear that $\mathbb TO_N$ normalizes $\mathbb TSO_N$, so we must show that if $U\in U_N$ normalizes $\mathbb TSO_N$ then $U\in\mathbb TO_N$. First note that $U$ normalizes $SO_N$. Indeed if $X \in SO_N$ then:
$$U^{-1}XU \in \mathbb TSO_N$$

Thus $U^{-1}XU= \lambda Y$ with $\lambda\in\mathbb T$, $Y\in SO_N$. If $Tr(X)\neq0$, we have $\lambda\in\mathbb R$ and so:
$$\lambda Y=U^{-1}XU \in SO_N$$

The set of matrices having nonzero trace being dense in $SO_N$, we conclude that $U^{-1}XU \in SO_N$ for all $X\in SO_N$. Thus, we have:
\begin{eqnarray*}
X \in SO_N
&\implies&(UXU^{-1})^t(UXU^{-1})=I_N\\
&\implies&X^tU^tUX= U^tU\\
&\implies&U^tU \in SO_N'
\end{eqnarray*}

It follows that at $N\geq 3$ we have $U^tU=\alpha I_N$, with $\alpha \in \mathbb T$, since $U$ is unitary. Hence we have $U=\alpha^{1/2}(\alpha^{-1/2}U)$ with:
$$\alpha^{-1/2}U\in O_N\quad,\quad 
U\in\mathbb TO_N$$

If $N=2$, $(U^tU)^t=U^tU$ gives again that $U^tU=\alpha I_2$, and we conclude as in the previous case.

\medskip

\underline{Step 6}. Our claim is that the inclusion $\mathbb TO_N\subset U_N$ is maximal in the category of compact groups.

\medskip

Suppose indeed that $\mathbb TO_N\subset G\subset U_N$ is a compact group such that $G\neq U_N$. It is a well-known fact that the connected component of the identity in $G$ is a normal subgroup, denoted $G_0$. Since we have $\mathbb TSO_N\subset G_0 \subset U_N$, we must have:
$$G_0=\mathbb TSO_N$$

But since $G_0$ is normal in $G$, the group $G$ normalizes $\mathbb TSO_N$, and hence $G\subset\mathbb TO_N$.

\medskip

\underline{Step 7}. Our claim is that the inclusion $PO_N\subset PU_N$ is maximal in the category of compact groups.

\medskip

This follows from the above result. Indeed, if $PO_N\subset G \subset PU_N$ is a proper intermediate subgroup, then its preimage under the quotient map $U_N\to PU_N$ would be a proper intermediate subgroup of $\mathbb TO_N\subset U_N$, which is a contradiction.
\end{proof}

In connection now with the ``edge question'' of classifying the intermediate groups $O_N\subset G\subset U_N$, the above result leads to a dichotomy, coming from:
$$PG\in\{PO_N,PU_N\}$$

Here are some basic examples of such intermediate groups:

\begin{proposition}
We have compact groups $O_N\subset G\subset U_N$ as follows:
\begin{enumerate}
\item The following groups, depending on a parameter $r\in\mathbb N\cup\{\infty\}$, 
$$\mathbb Z_r=O_N\left\{wU\Big|w\in\mathbb Z_r,U\in O_N\right\}$$ 
whose projective versions equal $PO_N$, and the biggest of which is the group $\mathbb TO_N$, which appears as affine lift of $PO_N$. 

\item The following groups, depending on a parameter $d\in 2\mathbb N\cup\{\infty\}$,
$$U_N^d=\left\{U\in U_N\Big|\det U\in\mathbb Z_d\right\}$$
interpolating between $U_N^2$ and $U_N^\infty=U_N$, whose projective versions equal $PU_N$.
\end{enumerate}
\end{proposition}

\begin{proof}
All the assertions are elementary, and well-known.
\end{proof}

The above results suggest that the solutions of $O_N\subset G\subset U_N$ should come from $O_N,U_N$, by succesively applying the constructions $G\to\mathbb Z_rG$ and $G\to G\cap U_N^d$. These operations do not exactly commute, but normally we should be led in this way to a 2-parameter series, unifying the two 1-parameter series from (1,2) above. However, some other groups like $\mathbb Z_NSO_N$ work too, so all this is probably a bit more complicated.

\bigskip

We have as well the following related result, also from \cite{bc+}:

\index{maximality}

\begin{theorem}
The inclusion of compact quantum groups
$$O_N\subset O_N^*$$
is maximal in the category of compact quantum groups.
\end{theorem}

\begin{proof}
The idea is that this follows from the result regarding $PO_N\subset PU_N$, by taking affine lifts, and using algebraic techniques. Consider indeed a sequence of surjective Hopf $*$-algebra maps as follows, whose  composition is the canonical surjection:
$$C(O_N^*)\overset{f}\longrightarrow A\overset{g}\longrightarrow C(O_N)$$

This produces a diagram of Hopf algebra maps with pre-exact rows, as follows:
$$\xymatrix@R=40pt@C=40pt{
\mathbb C\ar[r]&C(PO_N^*)\ar[d]^{f_|}\ar[r]&C(O_N^*)\ar[d]^f\ar[r]&C(\mathbb Z_2)\ar[r]\ar@{=}[d]&\mathbb C\\
\mathbb C\ar[r]&PA\ar[d]^{g_|}\ar[r]&A\ar[d]^g\ar[r]&C(\mathbb Z_2)\ar[r]\ar@{=}[d]&\mathbb C\\
\mathbb C\ar[r]&PC(O_N)\ar[r]&C(O_N)\ar[r]&C(\mathbb Z_2)\ar[r]&\mathbb C}$$
 
Consider now the following composition, with the isomorphism on the left being something well-known, coming from \cite{bdu}, that we will explain in chapter 16 below: 
$$C(PU_N)\simeq C(PO_N^*)\overset{f_|}\longrightarrow PA\overset{g_|}\longrightarrow PC(O_N)\simeq C(PO_N)$$

This induces, at the group level, the embedding $PO_N\subset PU_N$. Thus $f_|$ or $g_|$ is an isomorphism. If $f_|$ is an isomorphism we get a commutative diagram of Hopf algebra morphisms with pre-exact rows, as follows:
$$\xymatrix@R=40pt@C=40pt{
\mathbb C\ar[r]&C(PO_N^*)\ar@{=}[d]\ar[r]&C(O_N^*)\ar[d]^f\ar[r]&C(\mathbb Z_2)\ar[r]\ar@{=}[d]&\mathbb C\\
\mathbb C\ar[r]&C(PO_N^*)\ar[r]&A\ar[r]&C(\mathbb Z_2)\ar[r]&\mathbb C}$$
 
Then $f$ is an isomorphism. Similarly if $g_|$ is an isomorphism, then $g$ is an isomorphism. For further details on all this, we refer to \cite{bc+}.
\end{proof}

Summarizing, we are reaching to the conclusion formulated in the beginning of this chapter, namely that some of the easy solutions of the easy edge problems for the standard cube stay unique, even when lifting the easiness assumption.

\bigskip

In relation with these questions, we have as well the well-known and popular question of proving that the quantum group inclusion $S_N\subset S_N^+$ is maximal, in the sense that there is no intermediate quantum group, as follows:
$$S_N\subset G\subset S_N^+$$

\index{quantum permutation}
\index{maximality conjecture}

As evidence for this latter conjecture, the inclusions $S_4\subset S_4^+$ and $S_5\subset S_5^+$ can be both shown to be maximal, by using advanced quantum algebra techniques, including some recent planar algebra work. However, there is no good idea so far in order to deal with the general case. We refer to \cite{bb1}, \cite{bc+} and related papers for a discussion here.

\bigskip

Finally, for our discussion to be complete, let us discuss as well twisting results. For this purpose, let us go back to the standard cube, namely:
$$
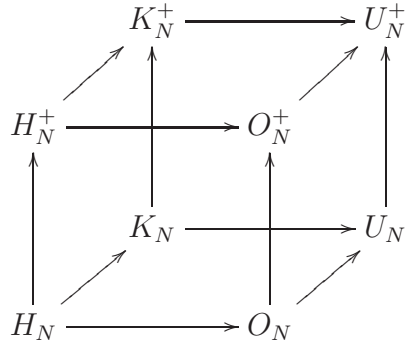
\xymatrix@R=20pt@C=20pt{
&K_N^+\ar[rr]&&U_N^+\\
H_N^+\ar[rr]\ar[ur]&&O_N^+\ar[ur]\\
&K_N\ar[rr]\ar[uu]&&U_N\ar[uu]\\
H_N\ar[uu]\ar[ur]\ar[rr]&&O_N\ar[uu]\ar[ur]
}$$

According to the general Schur-Weyl twisting method from chapter 7, all these quantum groups can be twisted. In addition, the continuous twists were computed in chapter 7, and the discrete objects were shown in chapter 10 to be equal to their own twists. Thus, we are led to the following conclusion, in relation with the standard cube:

\index{twisted standard cube}
\index{twisting}

\begin{theorem}
The Schur-Weyl twists of the main quantum groups are
$$
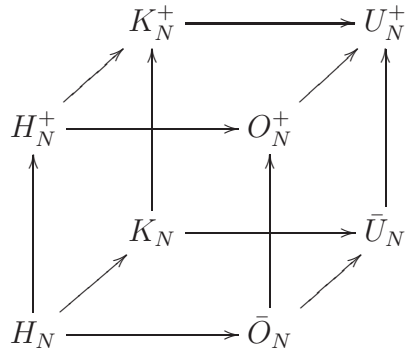
\xymatrix@R=20pt@C=20pt{
&K_N^+\ar[rr]&&U_N^+\\
H_N^+\ar[rr]\ar[ur]&&O_N^+\ar[ur]\\
&K_N\ar[rr]\ar[uu]&&\bar{U}_N\ar[uu]\\
H_N\ar[uu]\ar[ur]\ar[rr]&&\bar{O}_N\ar[uu]\ar[ur]
}$$
and we will call this diagram ``twisted standard cube''.
\end{theorem}

\begin{proof}
This follows indeed from the above discussion.
\end{proof}

This construction raises the perspective of finding the twisted versions of the above classification results. Following \cite{ba4}, in the uniform case, the result here is as follows:

\index{uniform quantum group}

\begin{theorem}
The classical and free uniform twisted easy quantum groups are
$$
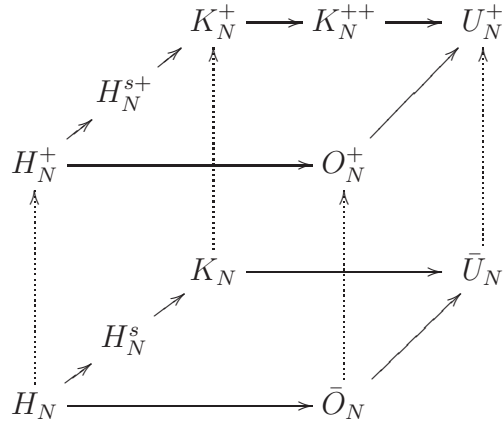
\xymatrix@R=7pt@C=7pt{
&&K_N^+\ar[rr]&&K_N^{++}\ar[rr]&&\ U_N^+\ \\
&H_N^{s+}\ar[ur]&&&&\\
H_N^+\ar[rrrr]\ar[ur]&&&&O_N^+\ar[uurr]\\
\\
&&K_N\ar[rrrr]\ar@.[uuuu]&&&&\ \bar{U}_N\ \ar@.[uuuu]\\
&H_N^s\ar[ur]&&&&\\
H_N\ar@.[uuuu]\ar[ur]\ar[rrrr]&&&&\bar{O}_N\ar@.[uuuu]\ar[uurr]
\\
}$$
where $H_s=\mathbb Z_s\wr S_N$, $H_N^{s+}=\mathbb Z_s\wr_*S_N^+$ with $s=4,6,8\ldots$\,, and where $K_N^+=\widetilde{K_N^+}$. 
\end{theorem}

\begin{proof}
This follows indeed from Theorem 12.1 above, dealing with the untwisted case, and from the above discussion, regarding the twists.
\end{proof}

We can merge the above result with the untwisted result, and we are led to the following statement, also from \cite{ba4}:

\begin{theorem}
The uniform classical/twisted and free quantum groups are
$$\xymatrix@R=2mm@C=12mm{
&&U_N,\bar{U}_N\ar@/^/[drr]\\
K_N\ar[rr]\ar@/^/[urr]&&K_N^+\ar[r]&K_N^{++}\ar[r]&U_N^+\\
\\
H_N^s\ar[rr]\ar[uu]&&H_N^{s+}\ar[uu]\\
\\
H_N\ar[rr]\ar[uu]\ar@/_/[drr]&&H_N^+\ar[rr]\ar[uu]&&O_N^+\ar[uuuu]\\
&&O_N,\bar{O}_N\ar@/_/[urr]}$$
where $H_N^s=\mathbb Z_s\wr S_N$, $H_N^{s+}=\mathbb Z_s\wr_*S_N^+$, with $s\in\{2,4,\ldots,\infty\}$, and $K_N^{++}=\widetilde{K}_N^+$. 
\end{theorem}

\begin{proof}
This is a slight extension of Theorem 12.1, the idea being as follows:

\medskip

(1) All the above quantum groups are quizzy, and the uniformity condition is clear as well, for each of the quantum groups under consideration. Finally, all these quantum groups are either classical/twisted or free.

\medskip

(2) In order to prove now the converse, in view of our Schur-Weyl twisting method, which only needs a category of partitions as input, it is enough to deal with the $q=1$ case. So, consider a uniform category of partitions $D\subset P_{even}$. We must prove that in the classical/free cases, the solutions are:
$$\xymatrix@R=2mm@C=11mm{
&&\mathcal P_2\ar@/_/[dll]\\
\mathcal P_{even}\ar[dd]&&\mathcal{NC}_{even}\ar[ll]\ar[dd]&\mathcal{NC}_{even}^-\ar[l]&\mathcal{NC}_2\ar[l]\ar@/_/[ull]\ar[dddd]\\
\\
P_{even}^s\ar[dd]&&NC_{even}^s\ar[ll]\ar[dd]\\
\\
P_{even}&&NC_{even}\ar[ll]&&NC_2\ar@/^/[dll]\ar[ll]\\
&&P_2\ar@/^/[ull]}$$

To be more precise, in the classical case, where $\backslash\hskip-2.1mm/\in D$, we must prove that the only solutions are the categories $P_2,\mathcal P_2,P_{even}^s$, and that in the free case, where $D\subset NC_{even}$, we must prove that the only solutions are the categories $NC_2,\mathcal{NC}_2,\mathcal{NC}_{even}^-,NC_{even}^s$.

\medskip

(3) We jointly investigate these two problems. Let $B$ be the set of all possible labelled blocks in $D$, having no upper legs. Observe that $B$ is stable under the switching of colors operation, $\circ\leftrightarrow\bullet$. We have two possible situations, as follows:

\medskip

\underline{Case 1}. The set $B$ consists of pairings only. Here the pairings in question can be either all labelled pairings, namely $\circ-\circ$, $\circ-\bullet$, $\bullet-\circ$, $\bullet-\bullet$, or just the matching ones, namely $\circ-\bullet$, $\bullet-\circ$, and we obtain here the categories $P_2,\mathcal P_2$ in the classical case, and the categories $NC_2,\mathcal{NC}_2$ in the free case.

\medskip

\underline{Case 2}. $B$ has at least one block of size $\geq 4$. In this case we can let $s\in\{2,4,\ldots,\infty\}$ to be the length of the smallest $\circ\ldots\circ$ block, and we obtain in this way the category $P_{even}^s$ in the classical case, and the categories $\mathcal{NC}_{even}^-,NC_{even}^s$ in the free case.
\end{proof}

Finally, we have as well the following classification result:

\index{twist}
\index{easy quantum group}

\begin{theorem}
The easy quantum groups $H_N\subset G\subset O_N^+$ and their twists are
$$\xymatrix@R=7mm@C=20mm{
&O_N\ar[r]&O_N^*\ar[rd]\\
H_N\ar[r]\ar[ur]\ar[rd]&H_N^\Gamma\ar[r]&H_N^{\diamond  k}\ar[r]&O_N^+\\
&\bar{O}_N\ar[r]&\bar{O}_N^*\ar[ru]}$$
and the set formed by these quantum groups is stable by intersections.
\end{theorem}

\begin{proof}
There are several things to be proved here, the idea being as follows:

\medskip

(1) According to the various classification and twisting results presented so far in this book, and to some straightforward extensions of them, the easy quantum groups $H_N\subset G\subset O_N^+$ and their twists are the quantum groups in the above diagram. 

\medskip

(2) Regarding now the intersection assertion, some straightforward computations show that we have the following intersection diagram:
$$\xymatrix@R=7mm@C=20mm{
&O_N\ar[r]&O_N^*\ar[rd]\\
H_N\ar[r]\ar[ur]\ar[rd]&H_N^*\ar[r]\ar[ur]\ar[dr]&H_N^+\ar[r]&O_N^+\\
&\bar{O}_N\ar[r]&\bar{O}_N^*\ar[ru]}$$

But with this diagram in hand, the assertion follows. Indeed, the intersections between the quantum groups $O_N^\times$ are their twists are all on this diagram, and hence on the diagram in the statement as well. Regarding now the intersections of an easy quantum group $H_N\subset G\subset H_N^+$ with the twists $\bar{O}_N,\bar{O}_N^*$, we can use again the above diagram. Indeed, from $H_N^+\cap\bar{O}_N^*=H_N^*$ we deduce that both $K=G\cap\bar{O}_N,K'=G\cap\bar{O}_N^*$ appear as intermediate easy quantum groups $H_N\subset K^\times\subset H_N^*$, and we are done.
\end{proof}

For more details on the above, we refer to \cite{ba4} and related papers, for the most dealing with noncommutative geometry. And with the remark that noncommutative geometry has something to do with all this because groups and quantum groups are not everything, let's not forget about spheres, tori and other homogeneous spaces and manifolds, which are equally in need of structure and classification results, often in relation with structure and classification results for some related groups or quantum groups.

\bigskip

In relation with this, we have already talked a bit about noncommutative geometry at the end of chapter 11, and we will be back to this in chapter 15 below.

\bigskip

As a conclusion now to all this, the classification of the compact quantum groups is a very interesting topic, and the general idea is that of stating and proving the results in the easy case first, and then trying to lift the easiness assumption. But this is of course just a general idea, and we will in chapters 13-14 below that many things can be done in a completely different way, using maximal tori, without reference to easiness.

\section*{12e. Exercises} 

As before with the previous chapter, there has been a lot of theory here, and our exercises will be basically about details on all this. First, we have:

\begin{exercise}
Prove that the easy quantum groups 
$$H_N\subset G\subset O_N^+$$
must fall into one of the following two classes:
$$O_N\subset G\subset O_N^+$$
$$H_N\subset G\subset H_N^+$$
\end{exercise}

This is something that we mentioned in the above, and whose proof is normally not that difficult. Of course, if stuck with something, you can look it up.

\begin{exercise}
Prove that the easy quantum groups 
$$H_N\subset G\subset H_N^+$$
must fall into one of the following two classes:
$$H_N\subset G\subset H_N^{[\infty]}$$
$$H_N^{[\infty]}\subset G\subset H_N^+$$
\end{exercise}

As before, in case this turns to be too difficult, the exercise is that of finding the relevant results in the relevant literature, and writing down a brief account of that.

\begin{exercise}
Prove that the easy quantum groups 
$$U_N\subset G\subset U_N^+$$
must fall into one of the following two classes:
$$U_N\subset G\subset U_N^{(\infty)}$$
$$U_N^{(\infty)}\subset G\subset U_N^+$$
\end{exercise}

Again, as with the previous two exercises, up to you to decide if this is something that you want to fully understand by yourself, or if you are in need of some help.

\part{Advanced topics}

\ \vskip50mm

\begin{center}
{\em And we'll talk of trails we walked up

Far above the timber line

There are nights I only feel right

With Carolina in the pines}
\end{center}

\chapter{Toral subgroups}

\section*{13a. Diagonal tori}

In this final part of the present book we discuss a number of more advanced questions. There are plenty of topics here, on which serious work has gone into, and where there are interesting things to be said, to choose from. We have opted here to talk about:

\bigskip

(1) Toral subgroups. This is a fundamental question, coming from the fact that, while Lie theory is definitely not available for the arbitrary closed subgroups $G\subset U_N$, a notion of ``maximal torus'' is available. In addition, while all this is very interesting, things are quie recent, and there are just a handful of things known here. Which makes this topic an ideal one to start with, providing a glimpse at beautiful lands, not explored yet.

\bigskip

(2) Amenability, growth. This is something far more classical, and we definitely owe you some more explanations here, besides what was quickly said on these subjects in chapter 3 above. Passed the basics, that we will explain in detail, there are many things that have been done on these classical topics, to choose from. And we will choose here to talk about amenability and growth in relation with maximal tori, a hot topic.

\bigskip

(3) Homogeneous spaces. This is yet another thing that we have to talk about, imperatively, as a continuation of the various noncommutative geometry considerations scattered all across the present book, starting from chapter 1. Again, wide subject, a  choice to be made here, and we will discuss, as a main topic, the construction and main properties of the ``simplest'' possible homogeneous spaces, over the easy quantum groups.

\bigskip

(4) Matrix models. This is perhaps the most beautiful of all ``advanced topics'' that can be discussed, the idea here, which is extremely simple, being that of looking for matrix models $U_{ij}\in M_K(C(T))$ for the standard coodinates $u_{ij}\in C(G)$ of a given closed subgroup $G\subset U_N$. And not only this is mathematically natural, but physically speaking, this is expected to be of great use, in connection with statistical mechanics.

\bigskip

So, this will be our plan, for the present chapter and for the next 3 ones, introduction to (1-4). Regarding other topics, unfortunately left aside, it was particularly heartbreaking not to talk more about quantum permutation groups, although we will still meet such quantum groups on numerous occasions, when discussing (1-4). For even more about quantum permutations, I recommend my advanced book \cite{ba6}.

\bigskip

Getting started now, in relation with tori, we have seen on various occasions that the group duals $G=\widehat{\Gamma}$ can be thought of as being ``tori'', in the compact quantum group framework. Also, given a closed subgroup $G\subset U_N^+$, the group dual subgroups $\widehat{\Lambda}\subset G$ can be thought of as being tori of $G$, and play a potentially important role. 

\bigskip

Our purpose here will be that of understanding how the structure of a closed subgroup $G\subset U_N^+$ can be recovered from the knowledge of these tori $\widehat{\Lambda}\subset G$. Let us start with a basic statement, regarding the classical and group dual cases:

\index{torus}

\begin{proposition}
Let $G\subset U_N^+$ be a compact quantum group, and consider the group dual subgroups $\widehat{\Lambda}\subset G$, also called toral subgroups, or simply ``tori''.
\begin{enumerate}
\item In the classical case, where $G\subset U_N$ is a compact Lie group, these are the usual tori, that is, the closed abelian subgroups of $G$.

\item In the group dual case, $G=\widehat{\Gamma}$ with $\Gamma=<g_1,\ldots,g_N>$ being a discrete group, these are the duals of the various quotients $\Gamma\to\Lambda$.
\end{enumerate}
\end{proposition}

\begin{proof}
Both these assertions are clear, as follows:

\medskip

(1) This follows indeed from the fact that a closed subgroup $H\subset U_N^+$ is at the same time classical, and a group dual, precisely when it is classical and abelian.

\medskip

(2) This follows from the general propreties of the Pontrjagin duality, and more precisely from the fact that the subgroups $\widehat{\Lambda}\subset\widehat{\Gamma}$ correspond to the quotients $\Gamma\to\Lambda$.
\end{proof}

Based on the above simple facts, regarding the simplest compact quantum groups that we know, namely the compact groups and the discrete group duals, we can see that there are two potential motivations for the study of toral subgroups $\widehat{\Lambda}\subset G$, as follows:

\bigskip

(1) First, it is well-known that the fine structure of a compact Lie group $G\subset U_N$ is partly encoded by its maximal torus. Thus, in view of Proposition 13.1, the various tori $\widehat{\Lambda}\subset G$ encode interesting information about a quantum group $G\subset U_N^+$, both in the classical and the group dual case. We can expect this to hold in general.

\bigskip

(2) Also, any action $G\curvearrowright X$ on some geometric object, such as a manifold, will produce actions of its tori on the same object, $\widehat{\Lambda}\curvearrowright X$. And, due to the fact that $\Lambda$ are familiar objects, namely discrete groups, these latter actions are easier to study, and this can ultimately lead to results about the action $G\curvearrowright X$ itself.

\bigskip

At a more concrete level now, most of the tori that we met appear as diagonal tori, in the sense of chapter 2 above. Let us first review this material. We first have:

\index{group dual}
\index{diagonal torus}

\begin{theorem}
Given a closed subgroup $G\subset U_N^+$, consider its ``diagonal torus'', which is the closed subgroup $T\subset G$ constructed as follows:
$$C(T)=C(G)\Big/\left<u_{ij}=0\Big|\forall i\neq j\right>$$
This torus is then a group dual, $T=\widehat{\Lambda}$, where $\Lambda=<g_1,\ldots,g_N>$ is the discrete group generated by the elements $g_i=u_{ii}$, which are unitaries inside $C(T)$.
\end{theorem}

\begin{proof}
This is something going back to \cite{bv2}, that we know from chapter 2. The idea indeed is that since $u$ is unitary, its diagonal entries $g_i=u_{ii}$ are unitaries inside $C(T)$. Moreover, from $\Delta(u_{ij})=\sum_ku_{ik}\otimes u_{kj}$ we obtain, when passing inside the quotient:
$$\Delta(g_i)=g_i\otimes g_i$$

It follows that we have $C(T)=C^*(\Lambda)$, modulo identifying as usual the $C^*$-completions of the various group algebras, and so that we have $T=\widehat{\Lambda}$, as claimed.
\end{proof}

Alternatively, we have the following construction for the diagonal torus:

\index{free group}

\begin{proposition}
The diagonal torus $T\subset G$ can be defined as well by
$$T=G\cap\mathbb T_N^+$$
where $\mathbb T_N^+\subset U_N^+$ is the free complex torus, appearing as
$$\mathbb T_N^+=\widehat{F_N}$$
with $F_N=<g_1,\ldots,g_N>$ being the free group on $N$ generators. 
\end{proposition}

\begin{proof}
As a main particular case of Theorem 13.2, that we know as well from chapter 2, the biggest quantum group produces the biggest torus, and so we have:
$$C(\mathbb T_N^+)=C(U_N^+)\Big/\left<u_{ij}=0\Big|\forall i\neq j\right>$$

Thus, by intersecting with $G$ we obtain the diagonal torus of $G$.
\end{proof}

Most of our computations so far of diagonal tori, that we will recall in a moment, concern various classes of easy quantum groups. In the general easy case, we have:

\begin{proposition}
For an easy quantum group $G\subset U_N^+$, coming from a category of partitions $D\subset P$, the associated diagonal torus is $T=\widehat{\Gamma}$, with:
$$\Gamma=F_N\Big/\left<g_{i_1}\ldots g_{i_k}=g_{j_1}\ldots g_{j_l}\Big|\forall i,j,k,l,\exists\pi\in D(k,l),\delta_\pi\begin{pmatrix}i\\ j\end{pmatrix}\neq0\right>$$
Moreover, we can just use partitions $\pi$ which generate the category $D$.
\end{proposition}

\begin{proof}
Let $g_i=u_{ii}$ be the standard coordinates on the diagonal torus $T$, and set $g=diag(g_1,\ldots,g_N)$. We have then the following computation:
\begin{eqnarray*}
C(T)
&=&\left[C(U_N^+)\Big/\left<T_\pi\in Hom(u^{\otimes k},u^{\otimes l})\Big|\forall\pi\in D\right>\right]\Big/\left<u_{ij}=0\Big|\forall i\neq j\right>\\
&=&\left[C(U_N^+)\Big/\left<u_{ij}=0\Big|\forall i\neq j\right>\right]\Big/\left<T_\pi\in Hom(u^{\otimes k},u^{\otimes l})\Big|\forall\pi\in D\right>\\
&=&C^*(F_N)\Big/\left<T_\pi\in Hom(g^{\otimes k},g^{\otimes l})\Big|\forall\pi\in D\right>
\end{eqnarray*}

The associated discrete group, $\Gamma=\widehat{T}$, is therefore given by:
$$\Gamma=F_N\Big/\left<T_\pi\in Hom(g^{\otimes k},g^{\otimes l})\Big|\forall\pi\in D\right>$$

Now observe that, with $g=diag(g_1,\ldots,g_N)$ as above, we have:
$$T_\pi g^{\otimes k}(e_{i_1}\otimes\ldots\otimes e_{i_k})=\sum_{j_1\ldots j_l}\delta_\pi\begin{pmatrix}i_1&\ldots&i_k\\ j_1&\ldots&j_l\end{pmatrix}e_{j_1}\otimes\ldots\otimes e_{j_l}\cdot g_{i_1}\ldots g_{i_k}$$

On the other hand, we have as well the following formula:
$$g^{\otimes l}T_\pi(e_{i_1}\otimes\ldots\otimes e_{i_k})=\sum_{j_1\ldots j_l}\delta_\pi\begin{pmatrix}i_1&\ldots&i_k\\ j_1&\ldots&j_l\end{pmatrix}e_{j_1}\otimes\ldots\otimes e_{j_l}\cdot g_{j_1}\ldots g_{j_l}$$

We conclude that the relation $T_\pi\in Hom(g^{\otimes k},g^{\otimes l})$ reformulates as follows:
\begin{eqnarray*}
&&\sum_{j_1\ldots j_l}\delta_\pi\begin{pmatrix}i_1&\ldots&i_k\\ j_1&\ldots&j_l\end{pmatrix}e_{j_1}\otimes\ldots\otimes e_{j_l}\cdot g_{i_1}\ldots g_{i_k}\\
&=&\sum_{j_1\ldots j_l}\delta_\pi\begin{pmatrix}i_1&\ldots&i_k\\ j_1&\ldots&j_l\end{pmatrix}e_{j_1}\otimes\ldots\otimes e_{j_l}\cdot g_{j_1}\ldots g_{j_l}
\end{eqnarray*}

Thus, the following condition must be satisfied:
$$\delta_\pi\begin{pmatrix}i_1&\ldots&i_k\\ j_1&\ldots&j_l\end{pmatrix}\neq0\implies g_{i_1}\ldots g_{i_k}=g_{j_1}\ldots g_{j_l}$$

Thus, we obtain the formula in the statement. Finally, the last assertion follows from Tannakian duality, because we can replace everywhere $D$ by a generating subset.
\end{proof}

In practice now, in the continuous case we have the following result:

\index{diagonal torus}

\begin{theorem}
The diagonal tori of the basic unitary quantum groups, namely
$$\xymatrix@R=15mm@C=17mm{
U_N\ar[r]&U_N^*\ar[r]&U_N^+\\
O_N\ar[r]\ar[u]&O_N^*\ar[r]\ar[u]&O_N^+\ar[u]}$$
and of their $q=-1$ twists as well, are the standard cube and torus, namely $T_N=\mathbb Z_2^N$ and $\mathbb T_N=\mathbb T^N$ in the classical case, and their liberations in general, which are as follows:
$$\xymatrix@R=15mm@C=17mm{
\mathbb T_N\ar[r]&\mathbb T_N^*\ar[r]&\mathbb T_N^+\\
T_N\ar[r]\ar[u]&T_N^*\ar[r]\ar[u]&T_N^+\ar[u]}$$
Also, for the quantum groups $B_N,B_N^+,C_N,C_N^+$, the diagonal torus collapses to $\{1\}$.
\end{theorem}

\begin{proof}
We have several assertions here, the idea being as follows:

\medskip

(1) The main assertion, regarding the basic unitary quantum groups, is something that we already know, from chapter 2 above, with the various liberations $T_N^\times,\mathbb T_N^\times$ of the basic tori $T_N,\mathbb T_N$ in the statement being by definition those appearing there. 

\medskip

(2) Regarding the invariance under twisting, this is best seen by using Proposition 13.4. Indeed, the computation in the proof there applies in the same way to the general quizzy case, and shows that the diagonal torus is invariant under twisting.

\medskip

(3) In the bistochastic case the fundamental corepresentation $g=diag(g_1,\ldots,g_N)$ of the diagonal torus must be bistochastic, and so $g_1=\ldots=g_N=1$, as desired.
\end{proof}

Regarding now the discrete case, the result here is as follows:

\index{diagonal torus}
\index{quantum reflection group}

\begin{theorem}
The diagonal tori of the basic quantum reflection groups, namely
$$\xymatrix@R=15mm@C=17mm{
K_N\ar[r]&K_N^*\ar[r]&K_N^+\\
H_N\ar[r]\ar[u]&H_N^*\ar[r]\ar[u]&H_N^+\ar[u]}$$
are the same as those for $O_N^\times,U_N^\times$ described above. Also, for $S_N,S_N^+$ we have $T=\{1\}$.
\end{theorem}

\begin{proof}
The first assertion follows from the general fact that the diagonal torus of $G_N\subset U_N^+$ equals the diagonal torus of the discrete version, namely:
$$G_N^d=G_N\cap K_N^+$$

Indeed, this fact follows from definitions, for instance via Proposition 13.3. As for the second assertion, this follows from the following inclusions:
$$S_N\subset B_N\quad,\quad
S_N^+\subset B_N^+$$

Indeed, by using the last assertion in Theorem 13.5, we obtain the result.
\end{proof}

As a conclusion, the diagonal torus $T\subset G$ is usually a quite interesting object, but for certain quantum groups like the bistochastic ones, or the quantum permutation ones, this torus collapses to $\{1\}$, and so it cannot be of use in the study of $G$.

\section*{13b. The skeleton}

In order to deal with the above issue, regarding the diagonal torus, the idea, from \cite{bbd}, \cite{bpa}, will be that of using the following generalization of Theorem 13.2:

\begin{theorem}
Given a closed subgroup $G\subset U_N^+$ and a matrix $Q\in U_N$, we let $T_Q\subset G$ be the diagonal torus of $G$, with fundamental representation spinned by $Q$:
$$C(T_Q)=C(G)\Big/\left<(QuQ^*)_{ij}=0\Big|\forall i\neq j\right>$$
This torus is then a group dual, given by $T_Q=\widehat{\Lambda}_Q$, where $\Lambda_Q=<g_1,\ldots,g_N>$ is the discrete group generated by the elements 
$$g_i=(QuQ^*)_{ii}$$
which are unitaries inside the quotient algebra $C(T_Q)$.
\end{theorem}

\begin{proof}
This follows from Theorem 13.2, because, as said in the statement, $T_Q$ is by definition a diagonal torus. Equivalently, since $v=QuQ^*$ is a unitary corepresentation, its diagonal entries $g_i=v_{ii}$, when regarded inside $C(T_Q)$, are unitaries, and satisfy:
$$\Delta(g_i)=g_i\otimes g_i$$

Thus $C(T_Q)$ is a group algebra, and more specifically we have $C(T_Q)=C^*(\Lambda_Q)$, where $\Lambda_Q=<g_1,\ldots,g_N>$ is the group in the statement, and this gives the result.
\end{proof}

Summarizing, associated to any closed subgroup $G\subset U_N^+$ is a whole family of tori, indexed by the unitaries $U\in U_N$. We use the following terminology:

\index{skeleton}
\index{standard tori}

\begin{definition}
Let $G\subset U_N^+$ be a closed subgroup.
\begin{enumerate}
\item The tori $T_Q\subset G$ constructed above are called standard tori of $G$.

\item The collection of tori $T=\left\{T_Q\subset G\big|Q\in U_N\right\}$ is called skeleton of $G$.
\end{enumerate}
\end{definition}

This might seem a bit awkward, but in view of various results, examples and counterexamples, to be presented below, this is perhaps the best terminology.  As a first general result now regarding these tori, coming from Woronowicz \cite{wo1}, we have:

\begin{theorem}
Any torus $T\subset G$ appears as follows, for a certain $Q\in U_N$:
$$T\subset T_Q\subset G$$
In other words, any torus appears inside a standard torus.
\end{theorem}

\begin{proof}
Given a torus $T\subset G$, we have an inclusion as follows:
$$T\subset G\subset U_N^+$$

On the other hand, we know from chapter 3 above that each torus $T\subset U_N^+$ has a fundamental corepresentation as follows, with $Q\in U_N$:
$$u=Q
\begin{pmatrix}g_1\\&\ddots\\&&g_N\end{pmatrix}
Q^*$$

But this shows that we have $T\subset T_Q$, and this gives the result.
\end{proof}

Let us do now some computations, following \cite{bbd}, where the standard tori were introduced. In the classical case, the result is as follows:

\begin{proposition}
For a closed subgroup $G\subset U_N$ we have
$$T_Q=G\cap(Q^*\mathbb T^NQ)$$
where $\mathbb T^N\subset U_N$ is the group of diagonal unitary matrices.
\end{proposition}

\begin{proof}
This is indeed clear at $Q=1$, where $\Gamma_1$ appears by definition as the dual of the compact abelian group $G\cap\mathbb T^N$. In general, this follows by conjugating by $Q$.
\end{proof}

In the group dual case now, still following \cite{bbd}, we have the following result:

\index{group dual}

\begin{proposition}
Given a finitely generated discrete group 
$$\Gamma=<g_1,\ldots,g_N>$$
consider its dual compact quantum group $G=\widehat{\Gamma}$, diagonally embedded into $U_N^+$. We have
$$\Lambda_Q=\Gamma/\left<g_i=g_j\Big|\exists k,Q_{ki}\neq0,Q_{kj}\neq0\right>$$
with the embedding $T_Q\subset G=\widehat{\Gamma}$ coming from the quotient map $\Gamma\to\Lambda_Q$.
\end{proposition}

\begin{proof}
Assume indeed that $\Gamma=<g_1,\ldots,g_N>$ is a discrete group, with dual $\widehat{\Gamma}\subset U_N^+$ coming via $u=diag(g_1,\ldots,g_N)$. With $v=QuQ^*$, we have the following computation:
\begin{eqnarray*}
\sum_s\bar{Q}_{si}v_{sk}
&=&\sum_{st}\bar{Q}_{si}Q_{st}\bar{Q}_{kt}g_t\\
&=&\sum_t\delta_{it}\bar{Q}_{kt}g_t\\
&=&\bar{Q}_{ki}g_i
\end{eqnarray*}

Thus the condition $v_{ij}=0$ for $i\neq j$ gives $\bar{Q}_{ki}v_{kk}=\bar{Q}_{ki}g_i$, which tells us that:
$$Q_{ki}\neq0\implies g_i=v_{kk}$$

Now observe that this latter equality reads:
$$g_i=\sum_j|Q_{kj}|^2g_j$$

We conclude from this that we have, as desired:
$$Q_{ki}\neq0,Q_{kj}\neq0\implies g_i=g_j$$

As for the converse, this is elementary to establish as well. See \cite{bbd}.
\end{proof}

According to the above results, we can expect the skeleton $T$ to encode various algebraic and analytic properties of $G$. We will discuss this in what follows, with a number of results and conjectures, following \cite{bpa}. We first have the following result:

\index{generation property}

\begin{theorem}
The following hold, both over the category of compact Lie groups, and over the category of duals of finitely generated discrete groups:
\begin{enumerate}
\item Injectivity: the construction $G\to T$ is injective, in the sense that $G\neq H$ implies, for some $Q\in U_N$:
$$T_Q(G)\neq T_Q(H)$$

\item Monotony: the construction $G\to T$ is increasing, in the sense that passing to a subgroup $H\subset G$ decreases at least one of the tori $T_Q$:
$$T_Q(H)\neq T_Q(G)$$

\item Generation: any closed quantum subgroup $G\subset U_N^+$ is generated by its tori, or, equivalently, has the following generation property: 
$$G=<T_Q|Q\in U_N>$$
\end{enumerate}
\end{theorem}

\begin{proof}
We have two cases to be investigated, as follows:

\medskip

(1) Assume first that we are in the classical case, $G\subset U_N$. In order to prove the generation property we use the following formula, established above: 
$$T_Q=G\cap Q^*\mathbb T^NQ$$

Now since any group element $U\in G$ is unitary, and so diagonalizable by basic linear algebra, we can write, for certain matrices $Q\in U_N$ and $D\in\mathbb T^N$:
$$U=Q^*DQ$$

But this shows that we have $U\in T_Q$, for this precise value of the spinning matrix $Q\in U_N$, used in the construction of the standard torus $T_Q$. Thus we have proved the generation property, and the injectivity and monotony properties follow from this.

\medskip

(2) Regarding now the group duals, here everything is trivial. Indeed, when the group duals are diagonally embedded we can take $Q=1$, and when the group duals are embedded by using a spinning matrix $Q\in U_N$, we can use precisely this matrix $Q$.
\end{proof}

As explained in \cite{bpa}, it is possible to go beyond the above verifications, notably with some results regarding the half-classical and the free cases. We will be back to this in chapter 14 below, with a number of more specialized statements, also from \cite{bpa}, and which are for the moment conjectural as well, on the question of recovering the fine analytic properties of $G$ out of the fine analytic properties of its tori.

\section*{13c. Generation questions}

Let us focus now on the generation property, from Theorem 13.12 (3), which is perhaps the most important, in view of the various potential applications. In order to discuss the general case, we will need some abstract theory. Let us start with:

\begin{proposition}
Given a closed subgroup $G\subset U_N^+$ and a matrix $Q\in U_N$, the corresponding standard torus and its Tannakian category are given by
$$T_Q=G\cap\mathbb T_Q\quad,\quad 
C_{T_Q}=<C_G,C_{\mathbb T_Q}>$$
where $\mathbb T_Q\subset U_N^+$ is the dual of the free group $F_N=<g_1,\ldots,g_N>$, with the fundamental corepresentation of $C(\mathbb T_Q)$ being the matrix $u=Qdiag(g_1,\ldots,g_N)Q^*$.
\end{proposition}

\begin{proof}
The first assertion comes from the fact, that we know from chapter 7, that given two closed subgroups $G,H\subset U_N^+$, the corresponding quotient algebra $C(U_N^+)\to C(G\cap H)$ appears by dividing by the kernels of the following quotient maps:
$$C(U_N^+)\to C(G)\quad,\quad 
C(U_N^+)\to C(H)$$

Indeed, the construction of $T_Q$ from Theorem 13.7 amounts precisely in performing this operation, with $H=\mathbb T_Q$, and so we obtain, as claimed:
$$T_Q=G\cap\mathbb T_Q$$

As for the Tannakian category formula, this follows from this, and from the following general Tannakian duality formula, that we know as well from chapter 7:
$$C_{G\cap H}=<C_G,C_H>$$

Thus, we are led to the conclusion in the statement.
\end{proof}

We have the following Tannakian reformulation of the toral generation property:

\index{generation property}

\begin{theorem}
Given a closed subgroup $G\subset U_N^+$, the subgroup 
$$G'=<T_Q|Q\in U_N>$$
generated by its standard tori has the following Tannakian category:
$$C_{G'}=\bigcap_{Q\in U_N}<C_G,C_{\mathbb T_Q}>$$
In particular we have $G=G'$ when this intersection reduces to $C_G$.
\end{theorem}

\begin{proof}
Consider the subgroup $G'\subset G$ constructed in the statement. We have:
$$C_{G'}=\bigcap_{Q\in U_N}C_{T_Q}$$

Together with the formula in Proposition 13.13, this gives the result.
\end{proof}

Let us further discuss now the toral generation property, with some modest results, regarding its behaviour with respect to product operations. We first have:

\begin{proposition}
Given two closed subgroups $G,H\subset U_N^+$, and $Q\in U_N$, we have:
$$<T_Q(G),T_Q(H)>\subset T_Q(<G,H>)$$
Also, the toral generation property is stable under the operation $<\,,>$.
\end{proposition}

\begin{proof}
The first assertion can be proved either by using Theorem 13.14, or directly. For the direct proof, which is perhaps the simplest, we have:
\begin{eqnarray*}
T_Q(G)
&=&G\cap\mathbb T_Q\subset<G,H>\cap\mathbb T_Q\\
&=&T_Q(<G,H>)
\end{eqnarray*}

On the other hand, we have as well the following computation:
\begin{eqnarray*}
T_Q(H)
&=&H\cap\mathbb T_Q\subset<G,H>\cap\mathbb T_Q\\
&=&T_Q(<G,H>)
\end{eqnarray*}

Now since $A,B\subset C$ implies $<A,B>\subset C$, this gives the result. Regarding now the second assertion, we have the following computation:
\begin{eqnarray*}
<G,H>
&=&<<T_Q(G)|Q\in U_N>,<T_Q(H)|Q\in U_N>>\\
&=&<T_Q(G),T_Q(H)|Q\in U_N>\\
&=&<<T_Q(G),T_Q(H)>|Q\in U_N>\\
&\subset&<T_Q(<G,H>)|Q\in U_N>
\end{eqnarray*}

Thus the quantum group $<G,H>$ is generated by its tori, as claimed.
\end{proof}

Along the same lines, we have as well the following result:

\begin{proposition}
We have the following formula, for any $G,H$ and $R,S$:
$$T_{R\otimes S}(G\times H)=T_R(G)\times T_S(H)$$
Also, the toral generation property is stable under usual products $\times$.
\end{proposition}

\begin{proof}
The product formula in the statement is clear from definitions. Regarding now the second assertion, we have the following computation:
\begin{eqnarray*}
<T_Q(G\times H)|Q\in U_{MN}>
&\supset&<T_{R\otimes S}(G\times H)|R\in U_M,S\in U_N>\\
&=&<T_R(G)\times T_S(H)|R\in U_M,S\in U_N>\\
&=&<T_R(G)\times\{1\},\{1\}\times T_S(H)|R\in U_M,S\in U_N>\\
&=&<T_R(G)|R\in U_M>\times<T_G(H)|H\in U_N>\\
&=&G\times H
\end{eqnarray*}

Thus the quantum group $G\times H$ is generated by its tori, as claimed.
\end{proof}

In order to get beyond these results, let us discuss now some weaker versions of the generation property, related to the classification program for the compact quantum groups, explained in chapters 11-12. We have here the following technical definition:

\index{diagonal liberation}
\index{liberation}

\begin{definition}
A closed subgroup $G_N\subset U_N^+$, with classical version $G_N^c=G_N\cap U_N$, is said to be weakly generated by its tori when:
$$G_N=<G_N^c,(T_Q)_{Q\in U_N}>$$
When the following even weaker condition, involving $T_1$ only, is satisfied,
$$G_N=<G_N^c,T_1>$$
we say that $G_N$ appears as a diagonal liberation of its classical version $G_N^c$.
\end{definition}

According to our results above, the first property is satisfied for the compact groups, for the discrete group duals, and is stable under generation, and direct products. Regarding the second property, this is something quite interesting, related to many things. The idea here, from Chirvasitu \cite{chi} and subsequent papers, is that such formulae can be proved by recurrence on $N\in\mathbb N$. In order to discuss this, let us start with:

\begin{proposition}
Assume that $G=(G_N)$ is weakly uniform, let $n\in\{2,3,\ldots,\infty\}$ be minimal such that $G_n$ is not classical, and consider the following conditions:
\begin{enumerate}
\item Strong generation: $G_N=<G_N^c,G_n>$, for any $N>n$.

\item Usual generation: $G_N=<G_N^c,G_{N-1}>$, for any $N>n$.

\item Initial step generation: $G_{n+1}=<G_{n+1}^c,G_n>$.
\end{enumerate}
We have then $(1)\iff(2)\implies(3)$, and $(3)$ is in general strictly weaker.
\end{proposition}

\begin{proof}
All the implications and non-implications are elementary, as follows:

\medskip

$(1)\implies(2)$ This follows from $G_n\subset G_{N-1}$ for $N>n$, coming from uniformity.

\medskip

$(2)\implies(1)$ By using twice the usual generation, and then the uniformity, we have:
\begin{eqnarray*}
G_N
&=&<G_N^c,G_{N-1}>\\
&=&<G_N^c,G_{N-1}^c,G_{N-2}>\\
&=&<G_N^c,G_{N-2}>
\end{eqnarray*}

Thus we have a descent method, and we end up with the strong generation condition.

\medskip

$(2)\implies(3)$ This is clear, because (2) at $N=n+1$ is precisely (3).

\medskip

$(3)\hskip2.3mm\not\hskip-2.3mm\implies(2)$ In order to construct counterexamples here, the simplest is to use group duals. Indeed, with $G_N=\widehat{\Gamma_N}$ and $\Gamma_N=<g_1,\ldots,g_N>$, the uniformity condition tells us that we must be in a projective limit situation, as follows:
$$\Gamma_1\leftarrow\Gamma_2\leftarrow\Gamma_3\leftarrow\Gamma_4\leftarrow\ldots$$
$$\Gamma_{N-1}=\Gamma_N/<g_N=1>$$

Now by assuming for instance that $\Gamma_2$ is given and not abelian, there are many ways of completing the sequence, and so the uniqueness coming from (2) can only fail.
\end{proof}

Let us introduce now a few more notions, as follows:

\begin{proposition}
Assume that $G=(G_N)$ is weakly uniform, let $n\in\{2,3,\ldots,\infty\}$ be as above, and consider the following conditions, where $I_N\subset G_N$ is the diagonal torus:
\begin{enumerate}
\item Strong diagonal liberation: $G_N=<G_N^c,I_n>$, for any $N\geq n$.

\item Technical condition: $G_N=<G_N^c,I_{N-1}>$ for any $N>n$, and $G_n=<G_n^c,I_n>$.

\item Diagonal liberation: $G_N=<G_N^c,I_N>$, for any $N$.

\item Initial step diagonal liberation: $G_n=<G_n^c,I_n>$.
\end{enumerate}
We have then $(1)\implies(2)\implies(3)\implies(4)$.
\end{proposition}

\begin{proof}
Our claim is that when assuming that $G=(G_N)$ is weakly uniform, so is the family of diagonal tori $I=(I_N)$. Indeed, we have the following computation:
\begin{eqnarray*}
I_N\cap U_{N-1}^+
&=&(G_N\cap\mathbb T_N^+)\cap U_{N-1}^+\\
&=&(G_N\cap U_{N-1}^+)\cap(\mathbb T_N^+\cap U_{N-1}^+)\\
&=&G_{N-1}\cap\mathbb T_{N-1}^+\\
&=&I_{N-1}
\end{eqnarray*}

Thus our claim is proved, and this gives the various implications in the statement. 
\end{proof}

Based on the above technical results, we can now formulate a key theoretical observation, in relation with the various generation properties that we have, as follows:

\begin{theorem}
If $G=(G_N)$ is weakly uniform, and with $n\in\{2,3,\ldots,\infty\}$ being as above, the following conditions are equivalent, modulo their initial steps:
\begin{enumerate}
\item Generation: $G_N=<G_N^c,G_{N-1}>$, for any $N>n$.

\item Strong generation: $G_N=<G_N^c,G_n>$, for any $N>n$.

\item Diagonal liberation: $G_N=<G_N^c,I_N>$, for any $N\geq n$.

\item Strong diagonal liberation: $G_N=<G_N^c,I_n>$, for any $N\geq n$.
\end{enumerate}
\end{theorem}

\begin{proof}
Our first claim is that generation plus initial step diagonal liberation imply the technical diagonal liberation condition. Indeed, the recurrence step goes as follows:
\begin{eqnarray*}
G_N
&=&<G_N^c,G_{N-1}>\\
&=&<G_N^c,G_{N-1}^c,I_{N-1}>\\
&=&<G_N^c,I_{N-1}>
\end{eqnarray*}

In order to pass now from the technical diagonal liberation condition to the strong diagonal liberation condition itself, observe that we have:
\begin{eqnarray*}
G_N
&=&<G_N^c,G_{N-1}>\\
&=&<G_N^c,G_{N-1}^c,I_{N-1}>\\
&=&<G_N^c,I_{N-1}>
\end{eqnarray*}

With this condition in hand, we have then as well:
\begin{eqnarray*}
G_N
&=&<G_N^c,G_{N-1}>\\
&=&<G_N^c,G_{N-1}^c,I_{N-2}>\\
&=&<G_N^c,I_{N-2}>
\end{eqnarray*}

This procedure can be of course be continued. Thus we have a descent method, and we end up with the strong diagonal liberation condition, as desired. In the other sense now, we want to prove that we have the following formula, at any $N\geq n$:
$$G_N=<G_N^c,G_{N-1}>$$

At $N=n+1$ this is something that we already know. At $N=n+2$ now, we have:
\begin{eqnarray*}
G_{n+2}
&=&<G_{n+2}^c,I_n>\\
&=&<G_{n+2}^c,G_{n+1}^c,I_n>\\
&=&<G_{n+2}^c,G_{n+1}>
\end{eqnarray*}

This procedure can be of course be continued. Thus, we have a descent method, and we end up with the strong generation condition, as desired.
\end{proof}

It is possible to prove that many interesting quantum groups have the above properties, and hence appear as diagonal liberations, but the whole subject is quite technical. Here is however a statement, collecting most of the known results on the subject: 

\index{diagonal liberation}
\index{liberation}

\begin{theorem}
The basic quantum unitary and reflection groups are as follows:
\begin{enumerate}
\item $O_N^*,U_N^*$ appear via diagonal liberation.

\item $O_N^+,U_N^+$ appear via diagonal liberation.

\item $H_N^*,K_N^*$ appear via diagonal liberation.

\item $H_N^+,K_N^+$ do not appear via diagonal liberation.

\end{enumerate}
In addition, $B_N^+,C_N^+,S_N^+$ do not appear either via diagonal liberation.
\end{theorem}

\begin{proof}
All this is quite technical, the idea being as follows:

\medskip

(1) The half-classical quantum groups $O_N^*,U_N^*$ are not uniform, and so cannot be investigated with the above techniques. However, these quantum groups can be studied by using the matrix model technology in \cite{bb4}, \cite{bdu}, which will be briefly discussed in chapter 16 below, and this leads to the following generation formulae:
$$O_N^*=<O_N,T_N^*>$$
$$U_N^*=<U_N,T_N^*>$$

But these two formulae imply the following generation formula, as desired:
$$U_N^*=<U_N,\mathbb T_N^*>$$

(2) The quantum groups $O_N^+,U_N^+$ are uniform, and a quite technical computation, from Chirvasitu et al. \cite{bcf}, \cite{chi}, shows that the generation conditions from Theorem 13.20 are satisfied for $O_N^+$. Thus we obtain the following generation formula:
$$O_N^+=<O_N,T_N^+>$$

From this we can deduce via the maximality results in \cite{bc+} that we have:
$$U_N^+=<U_N,T_N^+>$$

But this implies the following generation formula, as desired:
$$U_N^+=<U_N,\mathbb T_N^+>$$

(3) The situation for $H_N^*,K_N^*$ is quite similar to the one for $O_N^*,U_N^*$, explained above. Indeed, the technology in \cite{bb4}, \cite{bdu} applies, and this leads to:
$$H_N^*=<H_N,T_N^*>$$
$$K_N^*=<K_N,T_N^*>$$

Thus, we have as well the following formula, as desired:
$$K_N^*=<K_N,\mathbb T_N^*>$$

As a comment here, in fact these results are stronger than the above-mentioned ones for the quantum groups $O_N^*,U_N^*$, via some standard generation formulae.

\medskip

(4) This is something subtle as well, coming from the quantum groups $H_N^{[\infty]},K_N^{[\infty]}$ from Raum-Weber \cite{rwe}, discussed before. The idea here is that the following relations, related to the defining relations for $H_N^{[\infty]},K_N^{[\infty]}$, are trivially satisfied for real reflections:
$$g_ig_ig_j=g_jg_ig_i$$

Thus, the diagonal tori of these quantum groups coincide with those for $H_N^+,K_N^+$, and so the diagonal liberation procedure ``stops'' at $H_N^{[\infty]},K_N^{[\infty]}$.

\medskip

Finally, regarding the last assertion, here $B_N^+,C_N^+,S_N^+$ do not appear indeed via diagonal liberation, and this because of a trivial reason, namely $T=\{1\}$.
\end{proof}

\section*{13d. Fourier liberation}

All the above was quite technical, but regardless of the difficulties involved there, and also of the various positive results on the subject, the notion of diagonal liberation is obviously not the good one, in general. As a conjectural solution to these difficulties, we have the notion of Fourier liberation, that we will discuss now. 

\bigskip

Let us start with the following basic fact, which generalizes the construction of the embedding $\widehat{D}_\infty\subset S_4^+$, that we met in chapter 9, when proving that we have $S_4^+\neq S_4$:

\index{group dual}

\begin{proposition}
Consider a discrete group generated by elements of finite order, written as a quotient group, as follows:
$$\mathbb Z_{N_1}*\ldots*\mathbb Z_{N_k}\to\Gamma$$
We have then an embedding of quantum groups $\widehat{\Gamma}\subset S_N^+$, where $N=N_1+\ldots+N_k$.
\end{proposition}

\begin{proof}
We have a sequence of embeddings and isomorphisms as follows:
\begin{eqnarray*}
\widehat{\Gamma}
&\subset&\widehat{\mathbb Z_{N_1}*\ldots*\mathbb Z_{N_k}}\\
&=&\widehat{\mathbb Z_{N_1}}\,\hat{*}\,\ldots\,\hat{*}\,\widehat{\mathbb Z_{N_k}}\\
&\simeq&\mathbb Z_{N_1}\,\hat{*}\,\ldots\,\hat{*}\,\mathbb Z_{N_k}\\
&\subset&S_{N_1}\,\hat{*}\,\ldots\,\hat{*}\,S_{N_k}\\
&\subset&S_{N_1}^+\,\hat{*}\,\ldots\,\hat{*}\,S_{N_k}^+\\
&\subset&S_N^+
\end{eqnarray*}

Thus, we are led to the conclusion in the statement.
\end{proof}

The above result is quite abstract, and it is worth working out the details, with an explicit formula for the associated magic matrix. Let us start with a study of the simplest situation, where $k=1$, and where $\Gamma=\mathbb Z_{N_1}$. The result here is as follows:

\index{Fourier transform}

\begin{proposition}
The magic matrix for the quantum permutation group
$$\widehat{\mathbb Z}_N\simeq\mathbb Z_N\subset S_N\subset S_N^+$$
with standard Fourier isomorphism on the left, is given by the formula
$$u=FIF^*$$
where $F=\frac{1}{\sqrt{N}}(w^{ij})$ with $w=e^{2\pi i/N}$ is the Fourier matrix, and where
$$I=\begin{pmatrix}
1\\
&g\\
&&\ddots\\
&&&g^{N-1}
\end{pmatrix}$$
is the diagonal matrix formed by the elements of $\mathbb Z_N$, regarded as elements of $C^*(\mathbb Z_N)$.
\end{proposition}

\begin{proof}
The magic matrix for the quantum group $\mathbb Z_N\subset S_N\subset S_N^+$ is given by:
$$v_{ij}
=\chi\left(\sigma\in\mathbb Z_N\Big|\sigma(j)=i\right)
=\delta_{i-j}$$

Let us apply now the Fourier transform. According to our Pontrjagin duality conventions from chapter 1 above, in one sense this is given by the following formula:
$$\Phi:C(\mathbb Z_N)\to C^*(\mathbb Z_N)\quad,\quad 
\delta_i\to\frac{1}{N}\sum_kw^{ik}g^k$$

As for the inverse isomorphism, this is given by the following formula:
$$\Psi:C^*(\mathbb Z_N)\to C(\mathbb Z_N)\quad,\quad 
g^i\to\sum_kw^{-ik}\delta_k$$

Here $w=e^{2\pi i/N}$, and we use the standard Fourier analysis convention that the indices are $0,1,\ldots,N-1$. With $F=\frac{1}{\sqrt{N}}(w^{ij})$ and $I=diag(g^i)$ as above, we have:
\begin{eqnarray*}
u_{ij}
&=&\Phi(v_{ij})\\
&=&\frac{1}{N}\sum_kw^{(i-j)k}g^k\\
&=&\frac{1}{N}\sum_kw^{ik}g^kw^{-jk}\\
&=&\sum_kF_{ik}I_{kk}(F^*)_{kj}\\
&=&(FIF^*)_{ij}
\end{eqnarray*}

Thus, the magic matrix that we are looking for is $u=FIF^*$, as claimed.
\end{proof}

With the above result in hand, we can complement Proposition 13.22 with:

\begin{proposition}
Given a quotient group $\mathbb Z_{N_1}*\ldots*\mathbb Z_{N_k}\to\Gamma$, the magic matrix for the subgroup $\widehat{\Gamma}\subset S_N^+$ found in Proposition 13.22, with $N=N_1+\ldots+N_k$, is given by
$$u=\begin{pmatrix}
F_{N_1}I_1F_{N_1}^*\\
&\ddots\\
&&F_{N_k}I_kF_{N_k}^*
\end{pmatrix}$$
where $F_N=\frac{1}{\sqrt{N}}(w_N^{ij})$ with $w_N=e^{2\pi i/N}$ are Fourier matrices, and where 
$$I_r=\begin{pmatrix}
1\\
&g_r\\
&&\ddots\\
&&&g_r^{N_r-1}
\end{pmatrix}$$
with $g_1,\ldots,g_k$ being the standard generators of $\Gamma$.
\end{proposition}

\begin{proof}
This follows indeed from Proposition 13.22 and Proposition 13.23.
\end{proof}

Following Bichon \cite{bi2}, let us prove now that this construction provides us with all the group duals $\widehat{\Gamma}\subset S_N^+$. The idea will be that of using orbit theory, which is as follows:

\begin{theorem}
Given a closed subgroup $G\subset S_N^+$, with standard coordinates denoted $u_{ij}\in C(G)$, the following defines an equivalence relation on $\{1,\ldots,N\}$,
$$i\sim j\iff u_{ij}\neq0$$
that we call orbit decomposition associated to the action $G\curvearrowright\{1,\ldots,N\}$. In the classical case, $G\subset S_N$, this is the usual orbit equivalence coming from the action of $G$.
\end{theorem}

\begin{proof}
We first check the fact that we have indeed an equivalence relation:

\medskip

(1) The condition $i\sim i$ follows indeed from $\varepsilon(u_{ij})=\delta_{ij}$, which gives:
$$\varepsilon(u_{ii})=1$$

(2) The condition $i\sim j\implies j\sim i$ follows from $S(u_{ij})=u_{ji}$, which gives:
$$u_{ij}\neq0\implies u_{ji}\neq0$$

(3) As for the condition $i\sim j,j\sim k\implies i\sim k$, this follows from:
$$\Delta(u_{ik})=\sum_ju_{ij}\otimes u_{jk}$$ 

Indeed, in this formula, the right-hand side is by definition a sum of projections, so assuming that we have $u_{ij}\neq0,u_{jk}\neq0$ for a certain index $j$, we obtain:
$$u_{ij}\otimes u_{jk}>0$$

Thus we have $\Delta(u_{ik})>0$, which gives $u_{ik}\neq0$, as desired. Finally, in the classical case, $G\subset S_N$, the standard coordinates are the following characteristic functions:
$$u_{ij}=\chi\left(\sigma\in G\Big|\sigma(j)=i\right)$$

Thus $u_{ij}\neq0$ is equivalent to the existence of an element $\sigma\in G$ such that $\sigma(j)=i$. But this means precisely that $i,j$ must be in the same orbit of $G$, as claimed.
\end{proof}

Generally speaking, the theory from the classical case extends well to the quantum group setting, and we have in particular the following result, also from Bichon \cite{bi2}:

\index{orbits}

\begin{theorem}
Given a closed subgroup $G\subset S_N^+$, with magic matrix $u=(u_{ij})$, consider the associated coaction map, on the space $X=\{1,\ldots,N\}$:
$$\Phi:C(X)\to C(X)\otimes C(G)\quad,\quad e_i\to\sum_je_j\otimes u_{ji}$$
The following three subalgebras of $C(X)$ are then equal,
$$Fix(u)=\left\{\xi\in C(X)\Big|u\xi=\xi\right\}$$
$$Fix(\Phi)=\left\{\xi\in C(X)\Big|\Phi(\xi)=\xi\otimes1\right\}$$
$$F=\left\{\xi\in C(X)\Big|i\sim j\implies \xi(i)=\xi(j)\right\}$$
where $\sim$ is the orbit equivalence relation constructed in Theorem 13.25.
\end{theorem}

\begin{proof}
There are several assertions here, the idea being as follows:

\medskip

(1) The fact that we have the equality $Fix(u)=Fix(\Phi)$ is standard, with this being valid for any corepresentation of a compact quantum group $u=(u_{ij})$. 

\medskip

(2) Regarding now the equality with the algebra $F$, we know from Theorem 13.25 that the magic unitary $u=(u_{ij})$ is block-diagonal, with respect to the orbit decomposition there. But this shows that the algebra $Fix(u)=Fix(\Phi)$ decomposes as well with respect to the orbit decomposition, and so in order to prove the result, we are left with a study in the transitive case, where the result is clear. For details here, see \cite{bi2}.
\end{proof}

We have as well the following result, of analytic flavor:

\index{transitivity}

\begin{proposition}
For a closed subgroup $G\subset S_N^+$, the following are equivalent:
\begin{enumerate}
\item $G$ is transitive.

\item $Fix(u)=\mathbb C\xi$, where $\xi$ is the all-one vector.

\item $\int_Gu_{ij}=\frac{1}{N}$, for any $i,j$.
\end{enumerate}
\end{proposition}

\begin{proof}
This is well-known in the classical case. In general, the proof is as follows:

\medskip

$(1)\iff(2)$ This follows from the identifications in Theorem 13.26.

\medskip

$(2)\iff(3)$ This is clear from the general properties of the Haar integration.
\end{proof}

As a final comment here, the theory of quantum group orbits and transitivity, originally developed by Bichon in \cite{bi2}, has an interesting extension into a theory of quantum group orbitals and 2-transitivity, developed by Lupini-Man\v cinska-Roberson in \cite{lmr}.

\bigskip

Now back to the tori, we have the following key result, from \cite{bi2}:

\begin{theorem}
Consider a quotient group as follows, with $N=N_1+\ldots+N_k$:
$$\mathbb Z_{N_1}*\ldots*\mathbb Z_{N_k}\to\Gamma$$
We have then $\widehat{\Gamma}\subset S_N^+$, and any group dual subgroup of $S_N^+$ appears in this way.
\end{theorem}

\begin{proof}
This result, from \cite{bi2}, can be proved in two steps, as follows:

\medskip

(1) The fact that we have a subgroup as in the statement is something that we already know. Conversely, assume that we have a group dual subgroup $\widehat{\Gamma}\subset S_N^+$. The corresponding magic unitary must be of the following form, with $U\in U_N$:
$$u=U
\begin{pmatrix}g_1\\&\ddots\\&&g_N\end{pmatrix}
U^*$$

Consider now the orbit decomposition for $\widehat{\Gamma}\subset S_N^+$, coming from Theorem 13.25:
$$N=N_1+\ldots+N_k$$

We conclude that $u$ has a $N=N_1+\ldots+N_k$ block-diagonal pattern, and so that $U$ has as well this $N=N_1+\ldots+N_k$ block-diagonal pattern.

\medskip

(2) But this discussion reduces our problem to its $k=1$ particular case, with the statement here being that the cyclic group $\mathbb Z_N$ is the only transitive group dual $\widehat{\Gamma}\subset S_N^+$. The proof of this latter fact being elementary, we obtain the result. See \cite{bi2}.
\end{proof}

Here is a related result, from \cite{bbd}, which is useful for our purposes:

\index{Fourier matrix}

\begin{theorem}
For the quantum permutation group $S_N^+$, we have:
\begin{enumerate}
\item Given $Q\in U_N$, the quotient $F_N\to\Lambda_Q$ comes from the following relations:
$$\begin{cases}
g_i=1&{\rm if}\ \sum_lQ_{il}\neq 0\\
g_ig_j=1&{\rm if}\ \sum_lQ_{il}Q_{jl}\neq 0\\ 
g_ig_jg_k=1&{\rm if}\ \sum_lQ_{il}Q_{jl}Q_{kl}\neq 0
\end{cases}$$

\item Given a decomposition $N=N_1+\ldots+N_k$, for the matrix $Q=diag(F_{N_1},\ldots,F_{N_k})$, where $F_N=\frac{1}{\sqrt{N}}(\xi^{ij})_{ij}$ with $\xi=e^{2\pi i/N}$ is the Fourier matrix, we obtain: 
$$\Lambda_Q=\mathbb Z_{N_1}*\ldots*\mathbb Z_{N_k}$$

\item Given a matrix $Q\in U_N$, there exists a decomposition $N=N_1+\ldots+N_k$, such that $\Lambda_Q$ appears as quotient of $\mathbb Z_{N_1}*\ldots*\mathbb Z_{N_k}$.
\end{enumerate}
\end{theorem}

\begin{proof}
This is more or less equivalent to Theorem 13.28, and the proof can be deduced either from Theorem 13.28, or from some direct computations, as follows:

\medskip

(1) Fix a unitary matrix $Q\in U_N$, and consider the following quantities:
$$\begin{cases}
c_i=\sum_lQ_{il}\\
c_{ij}=\sum_lQ_{il}Q_{jl}\\
d_{ijk}=\sum_l\bar{Q}_{il}\bar{Q}_{jl}Q_{kl}
\end{cases}$$

We write $w=QvQ^*$, where $v$ is the fundamental corepresentation of $C(S_N^+)$. Assume $X\simeq\{1,\ldots,N\}$, and let $\alpha$ be the coaction of $C(S_N^+)$ on $C(X)$. Let us set:
$$\varphi_i=\sum_l\bar{Q}_{il}\delta_l\in C(X)$$

Also, let $g_i=(QvQ^*)_{ii}\in C^*(\Lambda_Q)$. If $\beta$ is the restriction of $\alpha$ to $C^*(\Lambda_Q)$, then:
$$\beta(\varphi_i)=\varphi_i\otimes g_i$$

Now recall that $C(X)$ is the universal $C^*$-algebra generated by elements $\delta_1,\ldots,\delta_N$ which are pairwise orthogonal projections. Writing these conditions in terms of the linearly independent elements $\varphi_i$ by means of the formulae $\delta_i=\sum_lQ_{il}\varphi_l$, we find that the universal relations for $C(X)$ in terms of the elements $\varphi_i$ are as follows:
$$\begin{cases}
\sum_ic_i\varphi_i=1\\
\varphi_i^*=\sum_jc_{ij}\varphi_j\\
\varphi_i\varphi_j=\sum_kd_{ijk}\varphi_k
\end{cases}$$

Let $\tilde{\Lambda}_Q$ be the group in the statement. Since $\beta$ preserves these relations, we get:
$$\begin{cases}
c_i(g_i-1)=0\\
c_{ij}(g_ig_j-1)=0\\
d_{ijk}(g_ig_j-g_k)=0
\end{cases}$$

We conclude from this that $\Lambda_Q$ is a quotient of $\tilde{\Lambda}_Q$. On the other hand, it is immediate that we have a coaction map as follows:
$$C(X)\to C(X)\otimes C^*(\tilde{\Lambda}_Q)$$

Thus $C(\tilde{\Lambda}_Q)$ is a quotient of $C(S_N^+)$. Since $w$ is the fundamental corepresentation of $S_N^+$ with respect to the basis $\{\varphi_i\}$, it follows that the generator $w_{ii}$ is sent to $\tilde{g}_i\in\tilde{\Lambda}_Q$, while $w_{ij}$ is sent to zero. We conclude that $\tilde{\Lambda}_Q$ is a quotient of $\Lambda_Q$. Since the above quotient maps send generators on generators, we conclude that $\Lambda_Q=\tilde{\Lambda}_Q$, as desired.

\medskip

(2) We apply the result found in (1), with the $N$-element set $X$ used in the proof there chosen to be the following set:
$$X=\mathbb Z_{N_1}\sqcup\ldots\sqcup\mathbb Z_{N_k}$$

With this choice, we have $c_i=\delta_{i0}$ for any $i$. Also, we have $c_{ij}=0$, unless $i,j,k$ belong to the same block to $Q$, in which case $c_{ij}=\delta_{i+j,0}$, and also $d_{ijk} =0$, unless $i,j,k$ belong to the same block of $Q$, in which case $d_{ijk}=\delta_{i+j,k}$. We conclude from this that $\Lambda_Q$ is the free product of $k$ groups which have generating relations as follows:
$$g_ig_j=g_{i+j}\quad,\quad g_i^{-1}=g_{-i}$$

But this shows that our group is $\Lambda_Q=\mathbb Z_{N_1}*\ldots*\mathbb Z_{N_k}$, as stated.

\medskip

(3) This follows indeed from (2). See \cite{bbd}.
\end{proof}

Summarizing, for quantum permutation groups, the standard tori parametrized by Fourier matrices play a special role. This suggests formulating the following definition:

\index{Fourier liberation}

\begin{definition}
Consider a closed subgroup $G\subset U_N^+$.
\begin{enumerate}
\item Its standard tori $T_F$, with $F=F_{N_1}\otimes\ldots\otimes F_{N_k}$, and $N=N_1+\ldots+N_k$ being regarded as a partition, are called Fourier tori.

\item In the case where we have $G_N=<G_N^c,(T_F)_F>$, we say that $G_N$ appears as a Fourier liberation of its classical version $G_N^c$.
\end{enumerate}
\end{definition}

We believe that the easy quantum groups should appear as Fourier liberations. With respect to Theorem 13.21 above, the situation in the free case is as follows:

\bigskip

(1) $O_N^+,U_N^+$ are diagonal liberations, so they are Fourier liberations as well.

\bigskip

(2) $B_N^+,C_N^+$ are Fourier liberations too, by using the results in chapter 7.

\bigskip

(3) $S_N^+$ is a Fourier liberation too, being generated by its tori \cite{bcf}.

\bigskip

(4) $H_N^+,K_N^+$ remain to be investigated, by using the general theory in \cite{rwe}.

\bigskip

Finally, as a word of warning here, observe that an arbitrary classical group $G_N\subset U_N$ is not necessarily generated by its Fourier tori, and nor is an arbitrary discrete group dual, with spinned embedding. Thus, the Fourier tori, and the related notion of Fourier liberation, remain something quite technical, in connection with the easy case.

\section*{13e. Exercises} 

There are many interesting computations in relation with the above, and as a first exercise on the subject, we have:

\begin{exercise}
Find the spinned tori of the dual of a discrete group
$$\Gamma=<g_1,\ldots,g_N>$$
arbitrarily embedded into $U_N^+$.
\end{exercise}

To be more precise, the computation was given in the above, in the case where the embedding $\widehat{\Gamma}\subset U_N^+$ is diagonal. The problem is that of adapting this computation, as to work for the twisted embeddings $\widehat{\Gamma}\subset U_N^+$ as well.

\begin{exercise}
Find more evidence for the injectivity and monotony conjectures formulated in the above.
\end{exercise}

There are many things that can be done here, for instance in relation with the various product operations for the compact quantum groups, in the spirit of the study that we did in the above for the generation conjecture. The more things you find, the better.

\begin{exercise}
Find the group dual subgroups
$$\widehat{\Gamma}\subset S_N^+$$
at $N=4,5$, with explicit formulae for the embeddings.
\end{exercise}

To be more precise here, the problem was solved in the above, abstractly, for any $N\in\mathbb N$. The problem is that of working out the $N=4,5$ particular cases, explicitely.

\chapter{Amenability, growth}

\section*{14a. Functional analysis}

We have seen so far that the theory of the compact quantum Lie groups, $G\subset U_N^+$, can be developed with some inspiration from the theory of compact Lie groups, $G\subset U_N$. In this chapter we discuss an alternative approach to this, by looking at the finitely generated discrete quantum groups $\Gamma=\widehat{G}$ which are dual to our objects. Thus, the idea will be that of developing the theory of the finitely generated discrete quantum groups, $\widehat{U_N^+}\to\Gamma$, with inspiration from the theory of finitely generated discrete groups, $F_N\to\Gamma$.

\bigskip

Normally the theory is already there, as developed in the previous chapters, which equally concern the compact quantum group $G$ and its discrete dual $\Gamma=\widehat{G}$. However, from the discrete group viewpoint, what has been worked out so far looks more like specialized mathematics, and there are still a lot of basic things, to be developed. In short, what we will be doing here will be a complement to the material from the previous chapters, obtained by using a different, and somehow opposite, philosophy.

\bigskip

Let us begin with a reminder regarding the cocommutative Woronowicz algebras, which will be our main objects in this chapter, coming before the commutative ones, that we are so used to have in the $\#1$ spot. As explained in chapter 3 above, we have:

\index{cocommutative algebra}

\begin{theorem}
For a Woronowicz algebra $A$, the following are equivalent:
\begin{enumerate}
\item $A$ is cocommutative, $\Sigma\Delta=\Delta$.

\item The irreducible corepresentations of $A$ are all $1$-dimensional.

\item $A=C^*(\Gamma)$, for some group $\Gamma=<g_1,\ldots,g_N>$, up to equivalence.
\end{enumerate}
\end{theorem}

\begin{proof}
This follows from the Peter-Weyl theory, as follows:

\medskip

$(1)\implies(2)$ The assumption $\Sigma\Delta=\Delta$ tells us that the inclusion $\mathcal A_{central}\subset\mathcal A$ is an isomorphism, and by using Peter-Weyl theory we conclude that any irreducible corepresentation of $A$ must be equal to its character, and so must be 1-dimensional.

\medskip

$(2)\implies(3)$ This follows once again from Peter-Weyl, because if we denote by $\Gamma$ the group formed by the 1-dimensional corepresentations, then we have $\mathcal A=\mathbb C[\Gamma]$, and so $A=C^*(\Gamma)$ up to the standard equivalence relation for Woronowicz algebras.

\medskip

$(3)\implies(1)$ This is something trivial, that we already know from chapter 2.
\end{proof}

The above result is not the end of the story, because one can still ask what are the cocommutative Woronowicz algebras, without reference to the equivalence relation. More generally, we are led in this way into the question, that we have usually avoided so far, as being not part of the ``compact'' philosophy, of computing the equivalence class of a given Woronowicz algebra $A$. We first have here the following construction:

\begin{theorem}
Given a Woronowicz algebra $(A,u)$, the enveloping $C^*$-algebra $A_{full}$ of the algebra of ``smooth functions'' $\mathcal A=<u_{ij}>$ has morphisms
$$\Delta:A_{full}\to A_{full}\otimes A_{full}$$
$$\varepsilon:A_{full}\to\mathbb C$$
$$S:A_{full}\to A_{full}^{opp}$$
which make it a Woronowicz algebra, which is equivalent to $A$. In the cocommutative case, where $A\sim C^*(\Gamma)$, we obtain in this way the full group algebra $C^*(\Gamma)$.
\end{theorem}

\begin{proof}
There are several assertions here, the idea being as follows:

\medskip

(1) Consider indeed the algebra $A_{full}$, obtained by completing the $*$-algebra $\mathcal A\subset A$ with respect to its maximal $C^*$-norm. We have then a quotient map, as follows:
$$\pi:A_{full}\to A$$

By universality of $A_{full}$, the comultiplication, counit and antipode of $A$ lift into morphisms $\Delta,\varepsilon,S$ as in the statement, and the Woronowicz algebra axioms are satisfied.

\medskip

(2) The fact that we have an equivalence $A_{full}\sim A$ is clear from definitions, because at the level of $*$-algebras of coefficients, the above quotient map $\pi$ is an isomorphism.

\medskip

(3) Finally, in the cocommutative case, where $A\sim C^*(\Gamma)$, the coefficient algebra is $\mathcal A=\mathbb C[\Gamma]$, and the corresponding enveloping $C^*$-algebra is $A_{full}=C^*(\Gamma)$.
\end{proof}

Summarizing, in connection with our equivalence class question, we already have an advance, with the construction of a biggest object in each equivalence class:
$$A_{full}\to A$$

We can of course stop our study here, by formulating the following statement, which apparently terminates any further discussion about equivalence classes:

\index{full Woronowicz algebra}

\begin{proposition}
Let us call a Woronowicz algebra ``full'' when the following canonical quotient map is an isomorphism:
$$\pi:A_{full}\to A$$
Then any Woronowicz algebra is equivalent to a full Woronowicz algebra, and when restricting the attention to the full algebras, we have $1$ object per equivalence class.
\end{proposition}

\begin{proof}
The first assertion is clear from Theorem 14.2, which tells us that we have $A\sim A_{full}$, and the second assertion holds as well, for exactly the same reason.
\end{proof}

As a first observation, restricting the attention to the full Woronowicz algebras is more or less what we have being doing so far in this book, with all the algebras that we introduced and studied being full by definition. However, there are several good reasons for not leaving things like this, and for further getting into the subject, one problem for instance coming from the fact that for the non-amenable groups $\Gamma$, we have:
$$C^*(\Gamma)\not\subset L(\Gamma)$$

To be more precise, on the right we have the group von Neumann algebra $L(\Gamma)$, appearing by definition as the weak closure of $\mathbb C[\Gamma]$, in the left regular representation. It is known that the above non-inclusion happens indeed in the non-amenable case, and in terms of the quantum group $G=\widehat{\Gamma}$, we are led to the following bizarre conclusion:
$$C(G)\not\subset L^\infty(G)$$

In other words, we have noncommutative continuous functions which are not measurable. This is something that we must clarify. Welcome to functional analysis.

\bigskip

Before anything, we must warn the reader that a lot of modesty and faith is needed, in order to deal with such questions. We are basically doing quantum mechanics here, where the moving objects don't have clear positions or speeds, and where the precise laws of motion are not known, and where any piece of extra data costs a few billion dollars. Thus, the fact that we have $C(G)\not\subset L^\infty(G)$ is just one problem, among many other.

\bigskip

With this discussion made, let us go back now to Theorem 14.2. As a next step in our study, we can attempt to construct a smallest object $A_{red}$ in each equivalence class. The situation here is more tricky, and we have the following statement:

\begin{theorem}
Given a Woronowicz algebra $(A,u)$, its quotient $A\to A_{red}$ by the null ideal of the Haar integration $tr:A\to\mathbb C$ has morphisms as follows,
$$\Delta:A_{red}\to A_{red}\otimes_{min}A_{red}$$
$$\varepsilon:\mathcal A_{red}\to\mathbb C$$
$$S:A_{red}\to A_{red}^{opp}$$
where $\otimes_{min}$ is the spatial tensor product of $C^*$-algebras, and where $\mathcal A_{red}=<u_{ij}>$. In the case where these morphisms lift into morphisms
$$\Delta:A_{red}\to A_{red}\otimes A_{red}$$
$$\varepsilon:A_{red}\to\mathbb C$$
$$S:A_{red}\to A_{red}^{opp}$$
we have a Woronowicz algebra, which is equivalent to $A$. Also, in the cocommutative case, where $A\sim C^*(\Gamma)$, we obtain in this way the reduced group algebra $C^*_{red}(\Gamma)$.
\end{theorem}

\begin{proof}
We have several assertions here, the idea being as follows:

\medskip

(1) Consider indeed the algebra $A_{red}$, obtained by dividing $A$ by the null ideal of the Haar integration $tr:A\to\mathbb C$. We have then a quotient map, as follows:
$$\pi:A\to A_{red}$$

Also, by GNS construction, we have an embedding as follows:
$$i:A_{red}\subset B(L^2(A))$$

By using these morphisms $\pi,i$, we can see that the comultiplication, counit and antipode of the $*$-algebra $\mathcal A$ lift into morphisms $\Delta,\varepsilon,S$ as in the statement, or, equivalently, that the comultiplication, counit and antipode of the $C^*$-algebra $A$ factorize into morphisms $\Delta,\varepsilon,S$ as in the statement. Thus, we have our morphisms, as claimed.

\medskip

(2) In the case where the morphisms $\Delta,\varepsilon,S$ that we just constructed lift, as indicated in the statement, the Woronowicz algebra axioms are clearly satisfied, and so the algebra $A_{red}$, together with the matrix $u=(u_{ij})$, is a Woronowicz algebra, in our sense.

\medskip

(3) The fact that we have an equivalence $A_{red}\sim A$ is clear from definitions, because at the level of $*$-algebras of coefficients, the above quotient map $\pi$ is an isomorphism.

\medskip

(4) Finally, in the cocommutative case, where $A\sim C^*(\Gamma)$, the above embedding $i$ is the left regular representation, and so we have $A_{red}=C^*_{red}(\Gamma)$, as claimed.
\end{proof}

With the above result in hand, which is complementary to Theorem 14.2, we can now answer some of our philosophical questions, the idea being as follows:

\medskip

\begin{enumerate}

\item In the group dual case we have $C^*_{red}(\Gamma)\subset L(\Gamma)$, as subalgebras of $B(l^2(\Gamma))$, and so in terms of the compact quantum group $G=\widehat{\Gamma}$, the conclusion is that we have $C(G)\subset L^\infty(G)$, as we should, with the convention $C(G)=C^*_{red}(\Gamma)$.

\medskip

\item In view of this, it is tempting to modify our Woronowicz algebra axioms, with $\Delta,\varepsilon,S$ being redefined as in the first part of Theorem 14.4, as to include the reduced group algebras $C^*_{red}(\Gamma)$, and more generally, all the algebras $A_{red}$.

\medskip

\item With such a modification done, we could call then a Woronowicz algebra ``reduced'' when the quotient map $A\to A_{red}$ is an isomorphism. This would lead to a nice situation like in Proposition 14.3, with 1 object per equivalence class.

\medskip

\item However, we will not do this, simply because the bulk of the present book, which is behind us, is full of interesting examples of Woronowicz algebras constructed with generators and relations, which are full by definition.
\end{enumerate}

\medskip

In short, nevermind for the philosophy, we will keep our axioms which are nice, simple and powerful, keeping however in mind the fact that the full picture is as follows:

\begin{theorem}
Given a Woronowicz algebra $A$, we have morphisms
$$A_{full}\to A\to A_{red}\subset A_{red}''$$
which in terms of the associated compact quantum group $G$ read
$$C_{full}(G)\to A\to C_{red}(G)\subset L^\infty(G)$$
and in terms of the associated discrete quantum group $\Gamma$ read
$$C^*(\Gamma)\to A\to C^*_{red}(\Gamma)\subset L(\Gamma)$$
with Woronowicz algebras at left, and with von Neumann algebras at right.
\end{theorem}

\begin{proof}
This is something rather philosophical, coming by putting together the results that we have, namely Theorem 14.2 and Theorem 14.4.
\end{proof}

\section*{14b. Amenability}

With this discussion made, and with the reiterated warning that a lot of modesty and basic common sense is needed, in order to deal with such questions, let us get now into the real thing, namely the understanding of the following projection map:
$$\pi:A_{full}\to A_{red}$$

As already mentioned before, on numerous occasions, when the algebra $A$ is cocommutative, $A\sim C^*(\Gamma)$, and with the underlying group $\Gamma$ being assumed amenable, this projection map is an isomorphism. And the contrary happens when $\Gamma$ is not amenable.

\bigskip

This leads us into the amenability question for the general Woronowicz algebras $A$. We have seen the basic theory here in chapter 3 above, in the form of a list of equivalent conditions, which altogether are called amenability. The theory presented there, worked out now in more detail, and with a few items added, is as follows:

\index{amenability}
\index{Kesten amenability}

\begin{theorem}
Let $A_{full}$ be the enveloping $C^*$-algebra of $\mathcal A$, and let $A_{red}$ be the quotient of $A$ by the null ideal of the Haar integration. The following are then equivalent:
\begin{enumerate}
\item The Haar functional of $A_{full}$ is faithful.

\item The projection map $A_{full}\to A_{red}$ is an isomorphism.

\item The counit map $\varepsilon:A\to\mathbb C$ factorizes through $A_{red}$.

\item We have $N\in\sigma(Re(\chi_u))$, the spectrum being taken inside $A_{red}$.

\item $||ax_k-\varepsilon(a)x_k||\to0$ for any $a\in\mathcal A$, for certain norm $1$ vectors $x_k\in L^2(A)$.
\end{enumerate}
If this is the case, we say that the underlying discrete quantum group $\Gamma$ is amenable.
\end{theorem}

\begin{proof}
Before starting, we should mention that amenability and the present result are a bit like the Spectral Theorem, in the sense that knowing that the result formally holds does not help much, and in practice, one needs to remember the proof as well. For this reason, we will work out explicitely all the possible implications between (1-5), whenever possible, adding to the global formal proof, which will be linear, as follows:
$$(1)\implies(2)\implies(3)\implies(4)\implies(5)\implies(1)$$

In order to prove these implications, and the other ones too, the general idea is that this is is well-known in the group dual case, $A=C^*(\Gamma)$, with $\Gamma$ being a usual discrete group, and in general, the result follows by adapting the group dual case proof. 

\medskip

$(1)\iff(2)$ This follows from the fact that the GNS construction for the algebra $A_{full}$ with respect to the Haar functional produces the algebra $A_{red}$.

\medskip

$(2)\implies(3)$ This is trivial, because we have quotient maps $A_{full}\to A\to A_{red}$, and so our assumption $A_{full}=A_{red}$ implies that we have $A=A_{red}$. 

\medskip 

$(3)\implies(2)$ Assume indeed that we have a counit map $\varepsilon:A_{red}\to\mathbb C$. In order to prove $A_{full}=A_{red}$, we can use the right regular corepresentation. Indeed, we can define such a corepresentation by the following formula:
$$W(a\otimes x)=\Delta(a)(1\otimes x)$$

This corepresentation is unitary, so we can define a morphism as follows: 
$$\Delta':A_{red}\to A_{red}\otimes A_{full}$$
$$a\to W(a\otimes1)W^*$$

Now by composing with $\varepsilon\otimes id$, we obtain a morphism as follows:
$$(\varepsilon\otimes id)\Delta':A_{red}\to A_{full}$$
$$u_{ij}\to u_{ij}$$

Thus, we have our inverse for the canonical projection $A_{full}\to A_{red}$, as desired.

\medskip

$(3)\implies(4)$ This implication is clear, because we have:
\begin{eqnarray*}
\varepsilon(Re(\chi_u))
&=&\frac{1}{2}\left(\sum_{i=1}^N\varepsilon(u_{ii})+\sum_{i=1}^N\varepsilon(u_{ii}^*)\right)\\
&=&\frac{1}{2}(N+N)\\
&=&N
\end{eqnarray*}

Thus the element $N-Re(\chi_u)$ is not invertible in $A_{red}$, as claimed.

\medskip

$(4)\implies(3)$ In terms of the corepresentation $v=u+\bar{u}$, whose dimension is $2N$ and whose character is $2Re(\chi_u)$, our assumption $N\in\sigma(Re(\chi_u))$ reads:
$$\dim v\in\sigma(\chi_v)$$

By functional calculus the same must hold for $w=v+1$, and then once again by functional calculus, the same must hold for any tensor power of $w$:
$$w_k=w^{\otimes k}$$ 

Now choose for each $k\in\mathbb N$ a state $\varepsilon_k\in A_{red}^*$ having the following property:
$$\varepsilon_k(w_k)=\dim w_k$$

By Peter-Weyl we must have $\varepsilon_k(r)=\dim r$ for any $r\leq w_k$, and since any irreducible corepresentation appears in this way, the sequence $\varepsilon_k$ converges to a counit map: 
$$\varepsilon:A_{red}\to\mathbb C$$

$(4)\implies(5)$ Consider the following elements of $A_{red}$, which are positive:
$$a_i=1-Re(u_{ii})$$
 
Our assumption $N\in\sigma(Re(\chi_u))$ tells us that $a=\sum a_i$ is not invertible, and so there exists a sequence $x_k$ of norm one vectors in $L^2(A)$ such that: 
$$<ax_k,x_k>\to 0$$

Since the summands $<a_ix_k,x_k>$ are all positive, we must have, for any $i$:
$$<a_ix_k,x_k>\to0$$

We can go back to the variables $u_{ii}$ by using the following general formula:
$$||vx-x||^2=||vx||^2 +2<(1-Re(v))x,x>-1$$

Indeed, with $v=u_{ii}$ and $x=x_k$ the middle term on the right goes to 0, and so the whole term on the right becomes asymptotically negative, and so we must have:
$$||u_{ii}x_k-x_k||\to0$$

Now let $M_n(A_{red})$ act on $\mathbb C^n\otimes L^2(A)$. Since $u$ is unitary we have:
$$\sum_i||u_{ij}x_k||^2
=||u(e_j\otimes x_k)||
=1$$

From $||u_{ii}x_k||\to1$ we obtain $||u_{ij}x_k||\to0$ for $i\neq j$. Thus we have, for any $i,j$:
$$||u_{ij}x_k-\delta_{ij}x_k||\to0$$

Now by remembering that we have $\varepsilon(u_{ij})=\delta_{ij}$, this formula reads:
$$||u_{ij}x_k-\varepsilon(u_{ij})x_k||\to0$$

By linearity, multiplicativity and continuity, we must have, for any $a\in\mathcal  A$, as desired:
$$||ax_k-\varepsilon(a)x_k||\to0$$

$(5)\implies(1)$ This is something well-known, which follows via some standard functional analysis arguments, worked out in Blanchard's paper \cite{bla}.

\medskip

$(1)\implies(5)$ Once again this is something well-known, which follows via some standard functional analysis arguments, worked out in Blanchard's paper \cite{bla}.
\end{proof}

This was for the basic amenability theory. We will be back to this on several occasions, with more specialized amenability conditions, which will add to the above list. As a first application of the above result, we can now advance on a problem left before, in chapter 3 above, and then in the beginning of the present chapter as well:

\index{cocommutative algebra}
\index{maximal tensor product}
\index{minimal tensor product}

\begin{theorem}
The cocommutative Woronowicz algebras are the intermediate quotients of the following type, with $\Gamma=<g_1,\ldots,g_N>$ being a discrete group,
$$C^*(\Gamma)\to C^*_\pi(\Gamma)\to C^*_{red}(\Gamma)$$
and with $\pi$ being a unitary representation of $\Gamma$, subject to weak containment conditions of type $\pi\otimes\pi\subset\pi$ and $1\subset\pi$, which guarantee the existence of $\Delta,\varepsilon$.
\end{theorem}

\begin{proof}
We use Theorem 14.1 above, combined with Theorem 14.5 and then with Theorem 14.6, the idea being to proceed in several steps, as follows:

\medskip

(1) Theorem 14.1 and standard functional analysis arguments show that the cocommutative Woronowicz algebras should appear as intermediate quotients, as follows:
$$C^*(\Gamma)\to A\to C^*_{red}(\Gamma)$$

(2) The existence of $\Delta:A\to A\otimes A$ requires our intermediate quotient to appear as follows, with $\pi$ being a unitary representation of $\Gamma$, satisfying the condition $\pi\otimes\pi\subset\pi$, taken in a weak containment sense, and with the tensor product $\otimes$ being taken here to be compatible with our usual maximal tensor product $\otimes$ for the $C^*$-algebras:
$$C^*(\Gamma)\to C^*_\pi(\Gamma)\to C^*_{red}(\Gamma)$$

(3) With this condition imposed, the existence of the antipode $S:A\to A^{opp}$ is then automatic, coming from the group antirepresentation $g\to g^{-1}$. 

\medskip

(4) The existence of the counit $\varepsilon:A\to\mathbb C$, however, is something non-trivial, related to amenability, and leading to a condition of type $1\subset\pi$, as in the statement.
\end{proof}  

The above result is of course not the end of the story, because as formulated, with the above highly abstract conditions on $\pi$, it comes along with 0 non-trivial examples. We refer to Woronowicz \cite{wo1} and related papers for more on these topics.

\bigskip

Let us get back now to real life, and concrete mathematics, and focus on the Kesten amenability criterion, from Theorem 14.6 (4), which brings connections with interesting mathematics and physics, and which in practice will be our main amenability criterion. In order to discuss this, we will need the following standard fact:

\index{main character}

\begin{proposition}
Given a Woronowicz algebra $(A,u)$, with $u\in M_N(A)$, the moments of the main character $\chi=\sum_iu_{ii}$ are given by:
$$\int_G\chi^k=\dim\left(Fix(u^{\otimes k})\right)$$
In the case $u\sim\bar{u}$ the law of $\chi$ is a usual probability measure, supported on $[-N,N]$.
\end{proposition}

\begin{proof}
There are two assertions here, the proof being as follows:

\medskip

(1) The first assertion follows from the Peter-Weyl theory, which tells us that we have the following formula, valid for any corepresentation $v\in M_n(A)$:
$$\int_G\chi_v=\dim(Fix(v))$$

Indeed, for the corepresentation $v=u^{\otimes k}$, the corresponding character is:
$$\chi_v=\chi^k$$

Thus, we obtain the result, as a consequence of the above formula.

\medskip

(2) As for the second assertion, if we assume $u\sim\bar{u}$ then we have $\chi=\chi^*$, and so the general theory, explained above, tells us that $law(\chi)$ is in this case a real probability measure, supported by the spectrum of $\chi$. But, since $u\in M_N(A)$ is unitary, we have:
\begin{eqnarray*}
uu^*=1
&\implies&||u_{ij}||\leq 1,\forall i,j\\
&\implies&||\chi||\leq N
\end{eqnarray*}

Thus the spectrum of the character satisfies the following condition:
$$\sigma(\chi)\subset [-N,N]$$

Thus, we are led to the conclusion in the statement. 
\end{proof}

In relation now with the notion of amenability, we have:

\index{Kesten amenability}
\index{amenability}

\begin{theorem}
A Woronowicz algebra $(A,u)$, with $u\in M_N(A)$, is amenable when:
$$N\in supp\Big(law\left[Re(\chi)\right]\Big)$$
Also, the support on the right depends only on $law(\chi)$.
\end{theorem}

\begin{proof}
There are two assertions here, the proof being as follows:

\medskip

(1) According to the Kesten amenability criterion, from Theorem 14.6 (4), the algebra $A$ is amenable when the following condition is satisfied:
$$N\in\sigma(Re(\chi))$$

Now since $Re(\chi)$ is self-adjoint, we know from spectral theory that the support of its spectral measure $law(Re(\chi))$ is precisely its spectrum $\sigma(Re(\chi))$, as desired:
$$supp(law(Re(\chi)))=\sigma(Re(\chi))$$

(2) Regarding the second assertion, once again the variable $Re(\chi)$ being self-adjoint, its law depends only on the moments $\int_GRe(\chi)^p$, with $p\in\mathbb N$. But, we have:
\begin{eqnarray*}
\int_GRe(\chi)^p
&=&\int_G\left(\frac{\chi+\chi^*}{2}\right)^p\\
&=&\frac{1}{2^p}\sum_{|k|=p}\int_G\chi^k
\end{eqnarray*}

Thus $law(Re(\chi))$ depends only on $law(\chi)$, and this gives the result.
\end{proof}

Let us work out now in detail the group dual case. Here we obtain a very interesting measure, called Kesten measure of the group, as follows:

\index{Cayley graph}

\begin{proposition}
In the case $A=C^*(\Gamma)$ and $u=diag(g_1,\ldots,g_N)$, and by enlarging if necessary the generating set $g_1,\ldots,g_N$, as for the following to happen,
$$1\in u=\bar{u}$$
the moments of the main character are given by the formula
$$\int_{\widehat{\Gamma}}\chi^p=\#\left\{i_1,\ldots,i_p\Big|g_{i_1}\ldots g_{i_p}=1\right\}$$
counting the loops based at $1$, having lenght $p$, on the corresponding Cayley graph.
\end{proposition}

\begin{proof}
Consider indeed a discrete group $\Gamma=<g_1,\ldots,g_N>$. The main character of $A=C^*(\Gamma)$, with fundamental corepresentation $u=diag(g_1,\ldots,g_N)$, is then:
$$\chi=g_1+\ldots+g_N$$

Given a colored integer $k=e_1\ldots e_p$, the corresponding moment is given by:
\begin{eqnarray*}
\int_{\widehat{\Gamma}}\chi^k
&=&\int_{\widehat{\Gamma}}(g_1+\ldots+g_N)^k\\
&=&\#\left\{i_1,\ldots,i_p\Big|g_{i_1}^{e_1}\ldots g_{i_p}^{e_p}=1\right\}
\end{eqnarray*}

In the self-adjoint case, $u\sim\bar{u}$, we are only interested in the moments with respect to usual integers, $p\in\mathbb N$, and the above formula becomes:
$$\int_{\widehat{\Gamma}}\chi^p=\#\left\{i_1,\ldots,i_p\Big|g_{i_1}\ldots g_{i_p}=1\right\}$$

Assume now that we have in addition $1\in u$, so that the condition $1\in u=\bar{u}$ in the statement is satisfied. At the level of the generating set $S=\{g_1,\ldots,g_N\}$ this means:
$$1\in S=S^{-1}$$

Thus the corresponding Cayley graph is well-defined, with the elements of $\Gamma$ as vertices, and with the edges $g-h$ appearing when the following condition is satisfied:
$$gh^{-1}\in S$$

A loop on this graph based at 1, having lenght $p$, is then a sequence as follows:
$$(1)-(g_{i_1})-(g_{i_1}g_{i_2})-\ldots-(g_{i_1}\ldots g_{i_{p-1}})-(g_{i_1}\ldots g_{i_p}=1)$$

Thus the moments of $\chi$ count indeed such loops, as claimed.
\end{proof}

In order to generalize the above result to arbitrary Woronowicz algebras, we can use the discrete quantum group philosophy. The fundamental result here is as follows:

\index{Cayley graph}

\begin{theorem}
Let $(A,u)$ be a Woronowicz algebra, and assume, by enlarging if necessary $u$, that we have $1\in u=\bar{u}$. The following formula
$$d(v,w)=\min\left\{k\in\mathbb N\Big|1\subset\bar{v}\otimes w\otimes u^{\otimes k}\right\}$$
defines then a distance on $Irr(A)$, which coincides with the geodesic distance on the associated Cayley graph. In the group dual case we obtain the usual distance.
\end{theorem}

\begin{proof}
The fact that the lengths are finite follows from Woronowicz's analogue of Peter-Weyl theory, and the other verifications are as follows:

\medskip

(1) The symmetry axiom is clear from definitions.

\medskip

(2) The triangle inequality is elementary to establish as well. 

\medskip

(3) Finally, the last assertion is elementary as well.

\medskip

In the group dual case now, where our Woronowicz algebra is of the form $A=C^*(\Gamma)$, with $\Gamma=<S>$ being a finitely generated discrete group, our normalization condition $1\in u=\bar{u}$ means that the generating set must satisfy the following condition:
$$1\in S=S^{-1}$$

But this is precisely the normalization condition made before for the discrete groups, and the fact that we obtain the same metric space is clear.
\end{proof}

Summarizing, we have a good understanding of what a discrete quantum group is. We can now formulate a generalization of Proposition 14.10, as follows:

\begin{theorem}
Let $(A,u)$ be a Woronowicz algebra, with the normalization assumption $1\in u=\bar{u}$ made. The moments of the main character, 
$$\int_G\chi^p=\dim\left(Fix(u^{\otimes p})\right)$$
count then the loops based at $1$, having lenght $p$, on the corresponding Cayley graph.
\end{theorem}

\begin{proof}
Here the formula of the moments, with $p\in\mathbb N$, is the one coming from Proposition 14.8 above, and the Cayley graph interpretation comes from Theorem 14.11.
\end{proof}

Here is a related useful result, in relation with the notion of amenability:

\begin{theorem}
A Woronowicz algebra $(A,u)$ is amenable precisely when
$$||X||=N$$
where $X$ is the principal graph of the associated planar algebra 
$$P_k=End(u^{\otimes k})$$
obtained by deleting the reflections in the Bratteli diagram of $P=(P_k)$.
\end{theorem}

\begin{proof}
This is something which might look quite complicated, but the idea is very simple. First of all, it is well-known that the spaces in the statement form indeed a planar algebra in the sense of Jones \cite{jo3}, and we refer here to \cite{ba2}, \cite{ba3} and related papers, but we will not really need this here. What we need to know, which is something quite elementary, and for which we refer again to \cite{ba2}, \cite{ba3} and related papers, is that in the following sequence of inclusions, we have a copy of Jones' basic construction \cite{jo1} at every step, so that we can delete the corresponding reflections, as in the statement:
$$P_0\subset P_1\subset P_2\subset\ldots$$

With this done, via some standard identifications and rescalings, we have:
\begin{eqnarray*}
||X||
&=&||M_{\chi_u}||_{A_{central}}\\
&=&||\chi_u||_{A_{central}}\\
&=&||\chi_u||_{A_{red}}
\end{eqnarray*}

Thus, the result follows from the Kesten amenability criterion.
\end{proof}

There are many concrete illustrations for the above results, and we will be back to this, on several occasions, in what follows. 

\section*{14c. Growth}

As an application of all this, corepresentation theory used for ``discrete'' questions, we can introduce the notion of growth for the discrete quantum groups, as follows:

\index{growth}

\begin{definition}
Given a closed subgroup $G\subset U_N^+$, with $1\in u=\bar{u}$, consider the series  whose coefficients are the ball volumes on the corresponding Cayley graph,
$$f(z)=\sum_kb_kz^k\quad,\quad 
b_k=\sum_{l(v)\leq k}\dim(v)^2$$
and call it growth series of the discrete quantum group $\widehat{G}$. In the group dual case, $G=\widehat{\Gamma}$, we obtain in this way the usual growth series of $\Gamma$. 
\end{definition}

There are many things that can be said about the growth, and we will be back to this in a moment, with explicit examples, and some general theory as well.

\bigskip

As a first result here, in relation with the notion of amenability, we have:

\index{polynomial growth}
\index{amenability}

\begin{theorem}
Polynomial growth implies amenability.
\end{theorem}

\begin{proof}
We recall from Theorem 14.11 above that the Cayley graph of $\widehat{G}$ has by definition the elements of $Irr(G)$ as vertices, and the distance is as follows:
$$d(v,w)=\min\left\{k\in\mathbb N\Big|1\subset\bar{v}\otimes w\otimes u^{\otimes k}\right\}$$

By taking $w=1$ and by using Frobenius reciprocity, the lenghts are given by:
$$l(v)=\min\left\{k\in\mathbb N\Big|v\subset u^{\otimes k}\right\}$$

By Peter-Weyl we have a decomposition as follows, where $B_k$ is the ball of radius $k$, and $m_k(v)\in\mathbb N$ are certain multiplicities:
$$u^{\otimes k}=\sum_{v\in B_k}m_k(v)\cdot v$$

By using now Cauchy-Schwarz, we obtain the following inequality:
\begin{eqnarray*}
m_{2k}(1)b_k
&=&\sum_{v\in B_k}m_k(v)^2\sum_{v\in B_k}\dim(v)^2\\
&\geq&\left(\sum_{v\in B_k}m_k(v)\dim(v)\right)^2\\
&=&N^{2k}
\end{eqnarray*}

But shows that if $b_k$ has polynomial growth, then the following happens:
$$\limsup_{k\to\infty}\, m_{2k}(1)^{1/2k}\geq N$$

Thus, the Kesten type criterion applies, and gives the result.
\end{proof}

Let us discuss now as well, as a continuation of all this, the notions of connectedness for $G$, and no torsion for $\widehat{\Gamma}$. These two notions are in fact related, as follows:

\index{connected quantum group}
\index{torsion-free}

\begin{theorem}
For a closed subgroup $G\subset U_N^+$ the following conditions are equivalent, and if they are satisfied, we call $G$ connected:
\begin{enumerate}
\item There is no finite quantum group quotient, as follows:
$$G\to F\neq\{1\}$$

\item The following algebra is infinite dimensional, for any corepresentation $v\neq1$:
$$A_v=<v_{ij}>$$
\end{enumerate}
In the classical case, $G\subset U_N$, we recover in this way the usual notion of connectedness. For the group duals, $G=\widehat{\Gamma}$, this is the same as asking for $\Gamma$ to have no torsion.
\end{theorem}

\begin{proof}
The above equivalence comes from the fact that a quotient $G\to F$ must correspond to an embedding $C(F)\subset C(G)$, which must be of the form:
$$C(F)=<v_{ij}>$$

Regarding now the last two assertions, the situation here is as follows:

\medskip

(1) In the classical case, $G\subset U_N$, it is well-known that $F=G/G_1$ is a finite group, where $G_1$ is the connected component of the identity $1\in G$, and this gives the result.

\medskip

(2) As for the group dual case, $G=\widehat{\Gamma}$, here the irreducible corepresentations are 1-dimensional, corresponding to the group elements $g\in\Gamma$, and this gives the result.
\end{proof}

Along the same lines, and at a more specialized level, we can talk as well about the connected component of the identity $G_0\subset G$, obtained at the algebra level by dividing the Woronowicz algebra $C(G)$ by a suitable Hopf ideal, as to make dissapear the corepresentations $v$ such that $A_v$ is finite dimensional. See Pinzari et al. \cite{cdp}, \cite{dpr}.

\bigskip

Finally, once again in connection with all the above, we can talk as well about normal subgroups, and about simple compact quantum groups, as follows:

\index{normal subgroup}

\begin{definition}
Given a quantum subgroup $H\subset G$, coming from a quotient map $\pi:C(G)\to C(H)$, the following are equivalent:
\begin{enumerate}
\item The following algebra satisfies $\Delta(A)\subset A\otimes A$:
$$A=\left\{a\in C(G)\Big|(id\otimes\pi)\Delta(a)=a\otimes1\right\}$$

\item The following algebra satisfies $\Delta(B)\subset B\otimes B$:
$$B=\left\{a\in C(G)\Big|(\pi\otimes id)\Delta(a)=1\otimes a\right\}$$

\item We have $A=B$, as subalgebras of $C(G)$.
\end{enumerate}
If these conditions are satisfied, we say that $H\subset G$ is a normal subgroup.
\end{definition}

\begin{proof}
This is something well-known, the idea being as follows:

\medskip

(1) The conditions in the statement are indeed equivalent, and in the classical case we obtain the usual normality notion for the subgroups.

\medskip

(2) In the group dual case the normality of any subgroup, which must be a group dual subgroup, is then automatic, with this being something trivial. 

\medskip

(3) For more on these topics, and on the basic compact group theory in general, extended to the present quantum group setting, we refer to \cite{cdp}, \cite{dpr}.
\end{proof}

Summarizing, we have a quite complete theory for the notion of amenability, and for other related notions, coming either from discrete group theory, or from Lie theory.

\section*{14d. Toral conjectures}

Let us discuss now some further questions, in relation with the theory of toral subgroups, developed in chapter 13. We recall from there that associated to any closed subgroup $G\subset U_N^+$ is its diagonal torus, given by the following formula:
$$C(T_1)=C(G)\Big/\left<u_{ij}=0\Big|\forall i\neq j\right>$$

More generally, given a closed subgroup $G\subset U_N^+$ and a matrix $Q\in U_N$, we let $T_Q\subset G$ be the diagonal torus of $G$, with fundamental representation spinned by $Q$:
$$C(T_Q)=C(G)\Big/\left<(QuQ^*)_{ij}=0\Big|\forall i\neq j\right>$$

This torus is then a group dual, given by the formula $T_Q=\widehat{\Lambda}_Q$, as usual up to the standard equivalence relation for the compact quantum groups, in order to avoid amenability issues, where $\Lambda_Q=<g_1,\ldots,g_N>$ is the discrete group generated by the following elements, which are unitaries inside the quotient algebra $C(T_Q)$: 
$$g_i=(QuQ^*)_{ii}$$

As explained in chapter 13, the correct analogue of the maximal torus for $G\subset U_N^+$ is the collection of these spinned tori, called skeleton of $G$:
$$T=\left\{T_Q\subset G\big|Q\in U_N\right\}$$

Finally, let us recall from chapter 13 that several properties of $G$ are conjecturally encoded by the skeleton $T$, and with the conjectures being usually verified for the compact Lie groups, for the duals of the  finitely generated discrete groups, and in a few other cases. We have the following result, from \cite{bpa}, complementing the material in chapter 13:

\index{toral conjectures}
\index{amenability}
\index{growth}

\begin{theorem}
The following hold, both over the category of compact Lie groups, and over the category of duals of finitely generated discrete groups:
\begin{enumerate}
\item Characters: if $G$ is connected, for any nonzero $P\in C(G)_{central}$ there exists $Q\in U_N$ such that $P$ becomes nonzero, when mapped into $C(T_Q)$.

\item Amenability: a closed subgroup $G\subset U_N^+$ is coamenable if and only if each of the tori $T_Q$ is coamenable, in the usual discrete group sense.

\item Growth: assuming $G\subset U_N^+$, the discrete quantum group $\widehat{G}$ has polynomial growth if and only if each the discrete groups $\widehat{T_Q}$ has polynomial growth.
\end{enumerate}
\end{theorem}

\begin{proof}
In the classical case, where $G\subset U_N$, the proof goes as follows:

\medskip

(1) Characters. We can take here $Q\in U_N$ to be such that $QTQ^*\subset\mathbb T^N$, where $T\subset U_N$ is a maximal torus for $G$, and this gives the result.

\medskip

(2) Amenability. This conjecture holds trivially in the classical case, $G\subset U_N$, due to the fact that these latter quantum groups are all coamenable.

\medskip

(3) Growth. This is something nontrivial, well-known from the theory of compact Lie groups, and we refer here for instance to D'Andrea-Pinzari-Rossi \cite{dpr}.

\medskip

Regarding now the group duals, here everything is trivial. Indeed, when the group duals are diagonally embedded we can take $Q=1$, and when the group duals are embedded by using a spinning matrix $Q\in U_N$, we can use precisely this matrix $Q$.
\end{proof}

As in the previous chapter with the general results regarding the tori there, it is conjectured that the properties in Theorem 14.18 should hold in general. However, proving such things in general is probably something quite difficult, because Tannakian duality, which is basically our only tool, leads into fairly complicated combinatorial questions.

\bigskip

Following \cite{bpa}, as a first solid piece of evidence for the above conjectures, we have the following result, regarding the main examples of free quantum groups:

\index{free quantum group}

\begin{theorem}
The character, amenability and growth conjectures hold for the free quantum groups $G=O_N^+,U_N^+,S_N^+,H_N^+$.
\end{theorem}

\begin{proof}
We have $3\times4=12$ assertions to be proved, and the idea in each case will be that of using certain special group dual subgroups. We will mostly use the group dual subgroups coming at $Q=1$, which are well-known to be as follows:
$$G=O_N^+,U_N^+,S_N^+,H_N^+\implies\Gamma_1=\mathbb Z_2^{*N},F_N,\{1\},\mathbb Z_2^{*N}$$ 

However, for some of our 12 questions, using these subgroups will not be enough, and we will use as well some carefully chosen subgroups of type $\Gamma_Q$, with $Q\neq1$.

\medskip

As a last ingredient, we will need some specialized structure results for $G$, in the cases where $G$ is coamenable. Once again, the theory here is well-known, and the situations where $G=O_N^+,U_N^+,S_N^+,H_N^+$ is coamenable, along with the values of $G$, are as follows:
$$\begin{cases}
O_2^+=SU_2^{-1}\\
S_2^+=S_2,S_3^+=S_3,S_4^+=SO_3^{-1}\\
H_2^+=O_2^{-1}
\end{cases}$$

To be more precise, the equalities $S_N^+=S_N$ at $N\leq3$ are known since Wang's paper \cite{wa2}, and the twisting results are all well-known, and we refer here to \cite{bbd}.

\medskip

With these ingredients in hand, we can now go ahead with the proof. It is technically convenient to split the discussion over the 3 conjectures, as follows:

\medskip

(1) Characters. For $G=O_N^+,U_N^+$, it is known that the algebra $C(G)_{central}$ is polynomial, respectively $*$-polynomial, on the following variable:
$$\chi=\sum_iu_{ii}$$

Thus, it is enough to show that the following variable generates a polynomial, respectively $*$-polynomial algebra, inside the group algebra of $\mathbb Z_2^{*N},F_N$:
$$\rho=\sum_ig_i$$
 
But for the group $\mathbb Z_2^{*N}$ this is clear, and by using a multiplication by a unitary free from $\mathbb Z_2^{*N}$, the result holds as well for $F_N$.

\medskip

Regarding now $G=S_N^+$, we have three cases to be discussed, as follows:

\medskip

-- At $N=2,3$ this quantum group collapses to the usual permutation group $S_N$, and the character conjecture holds indeed. 

\medskip

-- At $N=4$ we have $S_4^+=SO_3^{-1}$, the fusion rules are the Clebsch-Gordan ones, and the algebra $C(G)_{central}$ is therefore polynomial on $\chi=\sum_iu_{ii}$. Now observe that the spinned torus, with $Q=diag(F_2,F_2)$, is the following discrete group: 
$$\Gamma_Q=\mathbb Z_2*\mathbb Z_2=D_\infty$$

Since $Tr(u)=Tr(Q^*uQ)$, the image of $\chi=\sum_iu_{ii}$ in the quotient $C^*(\Gamma_Q)$ is the variable $\rho=2+g+h$, where $g,h$ are the generators of the two copies of $\mathbb Z_2$. Now since this latter variable generates a polynomial algebra, we obtain the result. 

\medskip

-- At $N\geq5$ now, the fusion rules are once again the Clebsch-Gordan ones, the algebra $C(G)_{central}$ is, as before, polynomial on $\chi=\sum_iu_{ii}$, and the result follows by functoriality from the result at $N=4$, by using the embedding $S_4^+\subset S_N^+$.

\medskip

Regarding now $G=H_N^+$, here it is known, from the computations in \cite{bv1}, that the algebra $C(G)_{central}$ is polynomial on the following two variables:
$$\chi=\sum_iu_{ii}\quad,\quad 
\chi'=\sum_iu_{ii}^2$$

We have two cases to be discussed, as follows:

\medskip

-- At $N=2$ we have the following formula, which is well-known, and elementary:
$$H_2^+=O_2^{-1}$$

Also, as explained in \cite{bbd}, with $Q=F_2$ we have:
$$\Gamma_Q=D_\infty$$

Let us compute now the images $\rho,\rho'$ of the above variables $\chi,\chi'$ in the group algebra of $D_\infty$. As before, from $Tr(u)=Tr(Q^*uQ)$ we obtain the following formula, where $g,h$ are the generators of the two copies of $\mathbb Z_2$: 
$$\rho=g+h$$

Regarding now $\rho'$, let us first recall that the quotient map $C(H_2^+)\to C^*(D_\infty)$ is constructed as follows: 
$$\frac{1}{2}\begin{pmatrix}1&1\\1&-1\end{pmatrix}\begin{pmatrix}u_{11}&u_{12}\\u_{21}&u_{22}\end{pmatrix}\begin{pmatrix}1&1\\1&-1\end{pmatrix}\to\begin{pmatrix}g&0\\0&h\end{pmatrix}$$

Equivalently, this quotient map is constructed as follows:
\begin{eqnarray*}
\begin{pmatrix}u_{11}&u_{12}\\u_{21}&u_{22}\end{pmatrix}
&\to&\frac{1}{2}\begin{pmatrix}1&1\\1&-1\end{pmatrix}\begin{pmatrix}g&0\\0&h\end{pmatrix}\begin{pmatrix}1&1\\1&-1\end{pmatrix}\\
&=&\frac{1}{2}\begin{pmatrix}g+h&g-h\\g-h&g+h\end{pmatrix}
\end{eqnarray*}

We can now compute the image of our character, as follows:
\begin{eqnarray*}
\rho'
&=&\frac{1}{2}(g+h)^2\\
&=&\frac{1}{2}(2+2gh)\\
&=&1+gh
\end{eqnarray*}

By using now the elementary fact that the variables $\rho=g+h$ and $\rho'=1+gh$ generate a polynomial algebra inside $C^*(D_\infty)$, this gives the result. 

\medskip

-- Finally, at $N\geq3$ the result follows by functoriality, via the standard diagonal inclusion $H_2^+\subset H_N^+$, from the result at $N=2$, that we established above. 

\medskip

(2) Amenability. Here the cases where $G$ is not coamenable are those of $O_N^+$ with $N\geq3$, $U_N^+$ with $N\geq2$, $S_N^+$ with $N\geq5$, and $H_N^+$ with $N\geq3$. 

\medskip

-- For $G=O_N^+,H_N^+$ with $N\geq3$ the result is clear, because the discrete group $\Gamma_1=\mathbb Z_2^{*N}$ is not amenable. 

\medskip

-- Clear as well is the result for $U_N^+$ with $N\geq2$, because the discrete group $\Gamma_1=F_N$ is not amenable either. 

\medskip

-- Finally, for $S_N^+$ with $N\geq5$ the result holds as well, because of the presence of Bichon's group dual subgroup $\widehat{\mathbb Z_2*\mathbb Z_3}$.

\medskip

(3) Growth. Here the growth is polynomial precisely in the situations where $G$ is infinite and coamenable, the precise cases being:
$$O_2^+=SU_2^{-1}\quad,\quad 
S_4^+=SO_3^{-1}\quad,\quad 
H_2^+=O_2^{-1}$$

With these formulae in hand, the result follows from the well-known fact that the growth invariants are stable under twisting. 
\end{proof}

With a bit more work, the above result from \cite{bpa} can be extended to general quantum reflection groups $H_N^{s+}$ as well, and in particular to the quantum group $K_N^+$, and we conclude that our conjectures hold for the main easy quantum groups, namely:
$$\xymatrix@R=20pt@C=20pt{
&K_N^+\ar[rr]&&U_N^+\\
H_N^+\ar[rr]\ar[ur]&&O_N^+\ar[ur]\\
&K_N\ar[rr]\ar[uu]&&U_N\ar[uu]\\
H_N\ar[uu]\ar[ur]\ar[rr]&&O_N\ar[uu]\ar[ur]
}$$

As a second piece of evidence now for our conjectures, of different nature, we will prove that these conjectures hold for any half-classical quantum group. In order to do so, we can use the modern approach to half-liberation, from Bichon and Dubois-Violette \cite{bdu}, based on crossed products and related $2\times2$ matrix models, as follows:

\index{half-liberation}
\index{matrix model}

\begin{theorem}
Given a conjugation-stable closed subgroup $H\subset U_N$, consider the algebra $C([H])\subset M_2(C(H))$ generated by the following variables:
$$u_{ij}=\begin{pmatrix}0&v_{ij}\\ \bar{v}_{ij}&0\end{pmatrix}$$
Then $[H]$ is a compact quantum group, we have $[H]\subset O_N^*$, and any non-classical subgroup $G\subset O_N^*$ appears in this way, with $G=O_N^*$ itself appearing from $H=U_N$.
\end{theorem}

\begin{proof}
We have several things to be proved, the idea being as follows:

\medskip

(1) As a first observation, the matrices in the statement are self-adjoint. Let us prove now that these matrices are orthogonal. We have:
\begin{eqnarray*}
\sum_ku_{ik}u_{jk}
&=&\sum_k\begin{pmatrix}0&v_{ik}\\ \bar{v}_{ik}&0\end{pmatrix}
\begin{pmatrix}0&v_{jk}\\ \bar{v}_{jk}&0\end{pmatrix}\\
&=&\sum_k\begin{pmatrix}v_{ik}\bar{v}_{jk}&0\\ 0&\bar{v}_{ik}v_{jk}\end{pmatrix}\\
&=&\begin{pmatrix}1&0\\0&1\end{pmatrix}
\end{eqnarray*}

In the other sense, the computation is similar, as follows:
\begin{eqnarray*}
\sum_ku_{ki}u_{kj}
&=&\sum_k\begin{pmatrix}0&v_{ki}\\ \bar{v}_{ki}&0\end{pmatrix}
\begin{pmatrix}0&v_{kj}\\ \bar{v}_{kj}&0\end{pmatrix}\\
&=&\sum_k\begin{pmatrix}v_{ki}\bar{v}_{kj}&0\\ 0&\bar{v}_{ki}v_{kj}\end{pmatrix}\\
&=&\begin{pmatrix}1&0\\0&1\end{pmatrix}
\end{eqnarray*}

(2) Our second claim is that the matrices in the statement half-commute. Consider indeed arbitrary antidiagonal $2\times2$ matrices, with commuting entries, as follows:
$$X_i=\begin{pmatrix}0&x_i\\ y_i&0\end{pmatrix}$$

We have then the following computation:
\begin{eqnarray*}
X_iX_jX_k
&=&\begin{pmatrix}0&x_i\\ y_i&0\end{pmatrix}\begin{pmatrix}0&x_j\\ y_j&0\end{pmatrix}\begin{pmatrix}0&x_k\\ y_k&0\end{pmatrix}\\
&=&\begin{pmatrix}0&x_iy_jx_k\\ y_ix_jy_k&0\end{pmatrix}
\end{eqnarray*}

Since this quantity is symmetric in $i,k$, we obtain, as desired:
$$X_iX_jX_k=X_kX_jX_i$$

(3) According now to the definition of the quantum group $O_N^*$, we have a representation of algebras, as follows where $w$ is the fundamental corepresentation of $C(O_N^*)$:
$$\pi:C(O_N^*)\to M_2(C(H))\quad,\quad 
w_{ij}\to u_{ij}$$

Thus, with the compact quantum space $[H]$ being constructed as in the statement, we have a representation of algebras, as follows:
$$\rho:C(O_N^*)\to C([H])\quad,\quad 
w_{ij}\to u_{ij}$$

(4) With this in hand, it is routine to check that the compact quantum space $[H]$ constructed in the statement is indeed a compact quantum group, with this being best viewed via an equivalent construction, with a quantum group embedding as follows:
$$C([H])\subset C(H)\rtimes\mathbb Z_2$$

(5) As for the proof of the converse, stating that any non-classical subgroup $G\subset O_N^*$ appears in this way, this is something more tricky, and we refer here to \cite{bdu}. 

\medskip

(6) Finally, for the fact that we have indeed $O_N^*=[U_N]$, we refer here as well to \cite{bdu}. We will be back to this as well in chapter 16 below, with a direct analytic proof of this, based on the fact that the representation $\rho$ constructed above, with $H=U_N$, commutes with the respective Haar functionals, and so must be faithful.
\end{proof}

In relation with the above, we will need as well the following result, regarding the irreducible corepresentations, also from Bichon-Dubois-Violette \cite{bdu}:

\begin{theorem}
In the context of the correspondence $H\to[H]$ we have a bijection 
$$Irr([H])\simeq Irr_0(H)\coprod Irr_1(H)$$
where the sets on the right are given by
$$Irr_k(H)=\left\{r\in Irr(H)\Big|\exists l\in\mathbb N,r\in u^{\otimes k}\otimes(u\otimes\bar{u})^{\otimes l}\right\}$$
induced by the canonical identification $Irr(H\rtimes\mathbb Z_2)\simeq Irr(H)\coprod Irr(H)$.
\end{theorem}

\begin{proof}
This is something more technical, also from \cite{bdu}. It is easy to see that we have an equality of projective versions $P[H]=PH$, which gives an inclusion as follows:
$$Irr_0(H)=Irr(PH)\subset Irr([H])$$

As for the remaining irreducible representations of $[H]$, these must come from an inclusion $Irr_1(H)\subset Irr([H])$, appearing as above. See \cite{bdu}.
\end{proof}

Now back the maximal tori, the situation here is very simple, as follows:

\begin{proposition}
The group dual subgroups $\widehat{[\Gamma]}_Q\subset[H]$ appear via
$$[\Gamma]_Q=[\Gamma_Q]$$
from the group dual subgroups $\widehat{\Gamma}_Q\subset H$ associated to $H\subset U_N$.
\end{proposition}

\begin{proof}
Let us first discuss the case $Q=1$. Consider the diagonal subgroup $\widehat{\Gamma}_1\subset H$, with the associated quotient map $C(H)\to C(\widehat{\Gamma}_1)$ denoted:
$$v_{ij}\to\delta_{ij}h_i$$

At the level of the algebras of $2\times2$ matrices, this map induces a quotient map:
$$M_2(C(H))\to M_2(C(\widehat{\Gamma}_1))$$

Our claim is that we have a factorization, as follows:
$$\begin{matrix}
C([H])&\subset&M_2(C(H))\\
\\
\downarrow&&\downarrow\\
\\
C([\widehat{\Gamma}_1])&\subset&M_2(C(\widehat{\Gamma}_1))
\end{matrix}$$

Indeed, it is enough to show that the standard generators of $C([H])$ and of $ C([\widehat{\Gamma}_1])$ map to the same elements of $M_2(C(\widehat{\Gamma}_1))$. But these generators map indeed as follows:
$$\begin{matrix}
u_{ij}&\to&\begin{pmatrix}0&v_{ij}\\ \bar{v}_{ij}&0\end{pmatrix}\\
\\
&&\downarrow\\
\\
\delta_{ij}v_{ij}&\to&\begin{pmatrix}0&\delta_{ij}h_i\\ \delta_{ij}h_i^{-1}&0\end{pmatrix}
\end{matrix}$$

Thus we have the above factorization, and since the map on the left is obtained by imposing the relations $u_{ij}=0$ with $i\neq j$, we obtain, as desired:
$$[\Gamma]_1=[\Gamma_1]$$

In the general case now, $Q\in U_N$, the result follows by applying the above $Q=1$ result to the quantum group $[H]$, with fundamental corepresentation $w=QuQ^*$.
\end{proof}

Now back to our conjectures, we have the following result, from \cite{bpa}:

\index{half-liberation}
\index{toral conjectures}

\begin{theorem}
The $3$ toral conjectures, regarding the characters, amenability and growth, hold for any half-classical quantum group of the form 
$$[H]\subset O_N^*$$
with $H\subset U_N$ being connected.
\end{theorem}

\begin{proof}
We know that the conjectures hold for $H\subset U_N$. The idea will be that of ``transporting'' these results, via $H\to [H]$:

\medskip

(1) Characters. We can pick here a maximal torus $T=\Gamma_Q$ for the compact group $H\subset U_N$, and by using the formula $[\Gamma]_Q=[\Gamma_Q]=[T]$ from Proposition 14.21 above, we obtain the result, via the identification in Theorem 14.20.

\medskip

(2) Amenability. There is nothing to be proved here, because $O_N^*$ is coamenable, and so are all its quantum subgroups. Note however, in relation with the various comments made in chapter 3 above, that in the connected case, the Kesten measures of $G,[T]$ are intimately related. For some explicit formulae here, for $G=O_N^*$ itself, see \cite{ez1}.

\medskip

(3) Growth. Here the situation is similar to the one for the amenability conjecture, because the quantum group  $[H]$ has polynomial growth.
\end{proof}

Let us mention that the above results can be extended to the general, unitary half-classical case, by using some suitable variations of the $2\times2$ matrix models used in the above. We refer here to \cite{bb4} and related papers, for the most dealing with noncommutative geometry in general, where such questions were investigated.

\bigskip

As a conclusion now, the theory of maximal tori developed in this chapter and in the previous one looks like something quite promising, worth some further investigation. Unfortunately, and as usual in mathematics and physics with new topics, things going on slowly here, with the various communities hesitating of getting into the subject, and with the unaffiliated individuals being a dying breed, in these modern times.

\section*{14e. Exercises} 

There are many interesting computations in relation with the material from the present chapter, which was quite varied, and some more general theory to be learned as well. As a first exercise, regarding the general theory, we have:

\begin{exercise}
Learn about topological tensor products, and approximation properties for $C^*$-algebras, such as nuclearity and exactness, and their relation with amenability, and write down a brief account of what you learned.
\end{exercise}

There are of course many things that you can learn here. The more, the better.

\begin{exercise}
Read the Murray-von Neumann and Connes papers about amenability and hyperfiniteness in the von Neumann algebra setting, and then the fact that a discrete quantum group $\Gamma$ is amenable precisely when $L(\Gamma)$ is hyperfinite.
\end{exercise}

Here the first part of the exercise is really tough, to the point that, while many people talk as experts about amenability, tensor products and so on, just a handful of them have read Connes' main paper on the subject. Want to join the club? Read that paper.

\begin{exercise}
Clarify the fact that a discrete quantum group $\Gamma$ is amenable precisely when the associated planar algebra, or subfactor, is amenable.
\end{exercise}

This is something that we briefly talked about, in the above, at the planar algebra level, and the question now is that of understanding all this. As for the previous exercise, some tough von Neumann algebra readings ahead, this time from papers of Jones.

\begin{exercise}
Draw the Cayley graphs of the duals of the main quantum groups,
$$\xymatrix@R=20pt@C=20pt{
&K_N^+\ar[rr]&&U_N^+\\
H_N^+\ar[rr]\ar[ur]&&O_N^+\ar[ur]\\
&K_N\ar[rr]\ar[uu]&&U_N\ar[uu]\\
H_N\ar[uu]\ar[ur]\ar[rr]&&O_N\ar[uu]\ar[ur]
}$$
with respect to suitably chosen fundamental representations, and compute the growth.
\end{exercise}

To be more precise, in order to draw the Cayley graph you need to know the fusion rules for the representations, and this is something that we already know, for the quantum groups in the statement. The problem, however, is that we need the fundamental corepresentation to be suitably modified, as to satisfy $1\in u\sim\bar{u}$, and this is what the first part of the exercise is about, namely performing this modification, where needed, and with the simplest possible solution to this, and then computing the Cayley graph. As for the second question, all sorts of non-trivial computations to be done here.

\begin{exercise}
Find examples and counterexamples for the notion of connectedness, for the compact quantum groups.
\end{exercise}

There are many things that can be done here, for instance in conenction with the various product operations for the compact quantum groups.

\begin{exercise}
Find examples and counterexamples for the notion of normality of subgroups, for the compact quantum groups.
\end{exercise}

Again, many things that can be done here, for instance in connection with the various product operations for the compact quantum groups.

\begin{exercise}
Work out explicitely the helf-liberation formulae
$$O_N^*=[U_N]$$
$$H_N^*=[K_N]$$
appearing as particular cases of the theory developed above.
\end{exercise}

Here the problem is that of working out what happens to the half-liberation theory explained in the above, in the particular cases of the quantum groups $O_N^*$ and $H_N^*$.

\chapter{Homogeneous spaces}

\section*{15a. Quotient spaces}

We have seen that the closed subgroups $G\subset U_N^+$ can be investigated with a variety of techniques, for the most belonging to algebraic geometry and probability theory, and with most of our new, original results concerning the free case, where $S_N^+\subset G\subset U_N^+$. All this suggests developing, more generally, a theory of ``free geometry'', again of algebraic geometry and probability flavor. And also, why not developing as well, along the same lines, theories like ``half-classical geometry'', or ``twisted geometry'', and so on. 

\bigskip

This is certainly possible, but quite time-consuming, and going well beyond the purposes of the present book. Instead, we will provide in this chapter an introduction to all this. Our purpose, quite modest, will be that of extending some of our quantum group results to certain classes of ``quantum homogeneous spaces''. With this being the first step towards constructing the above-mentioned noncommutative geometry theories.

\bigskip

Before starting, a few words on motivations. These come from physics, and more specifically from quantum mechanics, of course. Classical mechanics is described by classical geometry, and so quantum mechanics should be described by some kind of quantum geometry, it's as simple as that. In practice however, all this is quite new, 100 years old, and no one really knows how to do this. In addition, the experts are bitterly split, with Connes \cite{con} and his group believing in differential geometry and smoothness, and with us, meaning me and you, dear reader, and our friends, believing instead in algebraic geometry, or call that Riemannian geometry a la Nash \cite{nas}, and probability.

\bigskip

But probably enough talking, let's have something started, and more comments later. Let us begin with some generalities regarding the quotient spaces. We have:

\index{quotient space}
\index{homogeneous space}

\begin{proposition}
Given a quantum subgroup $H\subset G$, with associated quotient map $\rho:C(G)\to C(H)$, if we define the quotient space $X=G/H$ by setting
$$C(X)=\left\{f\in C(G)\Big|(\rho\otimes id)\Delta f=1\otimes f\right\}$$
then we have a coaction $\Phi:C(X)\to C(X)\otimes C(G)$, obtained as the restriction of the comultiplication of $C(G)$. In the classical case, we obtain the usual space $X=G/H$.
\end{proposition}

\begin{proof}
Observe that $C(X)\subset C(G)$ is indeed a subalgebra, because it is defined via a relation of type $\varphi(f)=\psi(f)$, with $\varphi,\psi$ morphisms. Observe also that in the classical case we obtain the algebra of continuous functions on $X=G/H$, because:
\begin{eqnarray*}
(\rho\otimes id)\Delta f=1\otimes f
&\iff&(\rho\otimes id)\Delta f(h,g)=(1\otimes f)(h,g),\forall h\in H,\forall g\in G\\
&\iff&f(hg)=f(g),\forall h\in H,\forall g\in G\\
&\iff&f(hg)=f(kg),\forall h,k\in H,\forall g\in G
\end{eqnarray*}

Regarding now the construction of $\Phi$, observe that for $f\in C(X)$ we have: 
\begin{eqnarray*}
(\rho\otimes id\otimes id)(\Delta\otimes id)\Delta f
&=&(\rho\otimes id\otimes id)(id\otimes\Delta)\Delta f\\
&=&(id\otimes\Delta)(\rho\otimes id)\Delta f\\
&=&(id\otimes\Delta)(1\otimes f)\\
&=&1\otimes\Delta f
\end{eqnarray*}

Thus $f\in C(X)$ implies $\Delta f\in C(X)\otimes C(G)$, and this gives the existence of $\Phi$, as in the statement. Finally, all the other assertions are clear.
\end{proof}

As an illustration, in the group dual case we have:

\begin{proposition}
Assume that $G=\widehat{\Gamma}$ is a discrete group dual.
\begin{enumerate}
\item The quantum subgroups of $G$ are $H=\widehat{\Lambda}$, with $\Gamma\to\Lambda$ being a quotient group.

\item For such a quantum subgroup $\widehat{\Lambda}\subset\widehat{\Gamma}$, we have $\widehat{\Gamma}/\widehat{\Lambda}=\widehat{\Theta}$, where $\Theta=\ker(\Gamma\to\Lambda)$.
\end{enumerate}
\end{proposition}

\begin{proof}
This is well-known, the idea being as follows:

\medskip

(1) In one sense, this is clear. Conversely, since the algebra $C(G)=C^*(\Gamma)$ is cocommutative, so are all its quotients, and this gives the result.

\medskip

(2) Consider a quotient map $r:\Gamma\to\Lambda$, and denote by $\rho:C^*(\Gamma)\to C^*(\Lambda)$ its extension. With $f=\sum_{g\in\Gamma}\lambda_g\cdot g\in C^*(\Gamma)$ we have:
\begin{eqnarray*}
f\in C(\widehat{\Gamma}/\widehat{\Lambda})
&\iff&(\rho\otimes id)\Delta(f)=1\otimes f\\
&\iff&\sum_{g\in\Gamma}\lambda_g\cdot r(g)\otimes g=\sum_{g\in\Gamma}\lambda_g\cdot 1\otimes g\\
&\iff&\lambda_g\cdot r(g)=\lambda_g\cdot 1,\forall g\in\Gamma\\
&\iff&supp(f)\subset\ker(r)
\end{eqnarray*}

But this means $\widehat{\Gamma}/\widehat{\Lambda}=\widehat{\Theta}$, with $\Theta=\ker(\Gamma\to\Lambda)$, as claimed.
\end{proof}

Given two quantum spaces $X,Y$, we say that $X$ is a quotient space of $Y$ when we have an embedding of algebras $\alpha:C(X)\subset C(Y)$. With this convention, we have:

\begin{definition}
We call a quotient space $G\to X$ homogeneous when
$$\Delta(C(X))\subset C(X)\otimes C(G)$$
where $\Delta:C(G)\to C(G)\otimes C(G)$ is the comultiplication map.
\end{definition}

In other words, an homogeneous quotient space $G\to X$ is a quantum space coming from a subalgebra $C(X)\subset C(G)$, which is stable under the comultiplication. The relation with the quotient spaces from Proposition 15.1 is as follows:

\begin{theorem}
The following results hold:
\begin{enumerate}
\item The quotient spaces $X=G/H$ are homogeneous.

\item In the classical case, any homogeneous space is of type $G/H$.

\item In general, there are homogeneous spaces which are not of type $G/H$.
\end{enumerate}
\end{theorem}

\begin{proof}
Once again these results are well-known, the proof being as follows:

\medskip

(1) This is clear from Proposition 15.1 above.

\medskip

(2) Consider a quotient map $p:G\to X$. The invariance condition in the statement tells us that we must have an action $G\curvearrowright X$, given by $g(p(g'))=p(gg')$. Thus:
$$p(g')=p(g'')\implies p(gg')=p(gg''),\ \forall g\in G$$

Now observe that the following subset $H\subset G$ is a subgroup:
$$H=\left\{g\in G\Big|p(g)=p(1)\right\}$$

Indeed, $g,h\in H$ implies $p(gh)=p(g)=p(1)$, so $gh\in H$, and the other axioms are satisfied as well. Our claim now, finishing the proof here, is that we have $X=G/H$, via $p(g)\to Hg$. Indeed, the map $p(g)\to Hg$ is well-defined and bijective, because $p(g)=p(g')$ is equivalent to $p(g^{-1}g')=p(1)$, and so to $Hg=Hg'$, as desired. 

\medskip

(3) Given a discrete group $\Gamma$ and an arbitrary subgroup $\Theta\subset\Gamma$, the quotient space $\widehat{\Gamma}\to\widehat{\Theta}$ is homogeneous. Now by using Proposition 15.2 above, we can see that if $\Theta\subset\Gamma$ is not normal, the quotient space $\widehat{\Gamma}\to\widehat{\Theta}$ is not of the form $G/H$.
\end{proof}

Let us try now to understand the properties of the homogeneous spaces $G\to X$, in the above sense. We have the following result, which is once again well-known:

\begin{proposition}
Assume that a quotient space $G\to X$ is homogeneous.
\begin{enumerate}
\item The restriction $\Phi:C(X)\to C(X)\otimes C(G)$ of $\Delta$ is a coaction.

\item We have $\Phi(f)=f\otimes 1\implies f\in\mathbb C1$, and $(id\otimes\int)\Phi f=\int f$.

\item The restriction of $\int$ is the unique unital form satisfying $(\tau\otimes id)\Phi=\tau(.)1$.
\end{enumerate}
\end{proposition}

\begin{proof}
These results are all elementary, the proof being as follows:

\medskip

(1) This is clear from definitions, because $\Delta$ itself is a coaction.

\medskip

(2) If $f\in C(G)$ is such that $\Delta(f)=f\otimes 1$, then by applying the counit we obtain:
$$(\varepsilon\otimes id)\Delta f=(\varepsilon\otimes id)(f\otimes 1)$$

Thus $f=\varepsilon(f)1$, as desired. As for the second assertion, this follows from the left invariance property $(id\otimes\int)\Delta f=\int f$ of the Haar functional, by restriction to $C(X)$.

\medskip

(3) By using the right invariance property $(\int\otimes id)\Delta f=\int f$ of the Haar functional of $C(G)$, we obtain that $tr=\int_{|C(X)}$ is $G$-invariant, in the sense that:
$$(tr\otimes id)\Phi f=tr(f)1$$

Conversely, assuming that $\tau:C(X)\to\mathbb C$ satisfies $(\tau\otimes id)\Phi f=\tau(f)1$, we have:
$$\left(\tau\otimes\int\right)\Phi(f)
=\int(\tau\otimes id)\Phi(f)
=\int(\tau(f)1)
=\tau(f)$$

On the other hand, we can compute the same quantity as follows:
$$\left(\tau\otimes\int\right)\Phi(f)
=\tau\left(id\otimes\int\right)\Phi(f)
=\tau(tr(f)1)
=tr(f)$$

Thus we have $\tau(f)=tr(f)$ for any $f\in C(X)$, and this finishes the proof.
\end{proof}

Let us discuss now an extra issue, of analytic nature. The point is that for one of the most basic examples of actions, $O_N^+\curvearrowright S^{N-1}_{\mathbb R,+}$, the associated morphism $\alpha:C(X)\to C(G)$ is not injective. In order to include such examples, we must relax our axioms:

\index{extended homogeneous space}

\begin{definition}
An extended homogeneous space consists of a morphism of algebras $\alpha:C(X)\to C(G)$, and a coaction map $\Phi:C(X)\to C(X)\otimes C(G)$, such that
$$\xymatrix@R=16mm@C=20mm{
C(X)\ar[r]^\Phi\ar[d]_\alpha&C(X)\otimes C(G)\ar[d]^{\alpha\otimes id}\\
C(G)\ar[r]^\Delta&C(G)\otimes C(G)
}$$
commutes, and such that
$$\xymatrix@R=16mm@C=20mm{
C(X)\ar[r]^\Phi\ar[d]_\alpha&C(X)\otimes C(G)\ar[d]^{id\otimes\int}\\
C(G)\ar[r]^{\int(.)1}&C(X)
}$$
commutes as well, where $\int$ is the Haar integration over $G$. We write then $G\to X$.
\end{definition}

When $\alpha$ is injective we obtain an homogeneous space in the previous sense. The examples with $\alpha$ being not injective, which motivate the above formalism, include the standard action $O_N^+\curvearrowright S^{N-1}_{\mathbb R,+}$, and the standard action $U_N^+\curvearrowright S^{N-1}_{\mathbb C,+}$. 

\bigskip

Here are a few general remarks on the above axioms:

\begin{proposition}
Assume that we have morphisms of algebras $\alpha:C(X)\to C(G)$ and $\Phi:C(X)\to C(X)\otimes C(G)$, satisfying $(\alpha\otimes id)\Phi=\Delta\alpha$.
\begin{enumerate}
\item If $\alpha$ is injective on a dense $*$-subalgebra $A\subset C(X)$, and $\Phi(A)\subset A\otimes C(G)$, then $\Phi$ is automatically a coaction map, and is unique.

\item The ergodicity type condition $(id\otimes\int)\Phi=\int\alpha(.)1$ is equivalent to the existence of a linear form $\lambda:C(X)\to\mathbb C$ such that $(id\otimes\int)\Phi=\lambda(.)1$.
\end{enumerate}
\end{proposition}

\begin{proof}
This is something elementary, the idea being as follows:

\medskip

(1) Assuming that we have a dense $*$-subalgebra $A\subset C(X)$ as in the statement, satisying $\Phi(A)\subset A\otimes C(G)$, the restriction $\Phi_{|A}$ is given by:
$$\Phi_{|A}=(\alpha_{|A}\otimes id)^{-1}\Delta\alpha_{|A}$$

This restriction and is therefore coassociative, and unique. By continuity, $\Phi$ itself follows to be coassociative and unique, as desired.

\medskip

(2) Assuming $(id\otimes\int)\Phi=\lambda(.)1$, we have $(\alpha\otimes\int)\Phi=\lambda(.)1$. On the other hand, we have as well the following formula:
$$\left(\alpha\otimes\int\right)\Phi=\left(id\otimes\int\right)\Delta\alpha=\int\alpha(.)1$$

Thus we obtain $\lambda=\int\alpha$, as claimed.
\end{proof}

Given an extended homogeneous space $G\to X$ as above, with associated map $\alpha:C(X)\to C(G)$, we can consider the image of this latter map:
$$\alpha:C(X)\to C(Y)\subset C(G)$$

Equivalently, at the level of the associated noncommutative spaces, we can factorize the corresponding quotient map $G\to Y\subset X$. With these conventions, we have:

\begin{proposition}
Consider an extended homogeneous space $G\to X$.
\begin{enumerate}
\item $\Phi(f)=f\otimes 1\implies f\in\mathbb C1$.

\item $tr=\int\alpha$ is the unique unital $G$-invariant form on $C(X)$.

\item The image space obtained by factorizing, $G\to Y$, is homogeneous.
\end{enumerate}
\end{proposition}

\begin{proof}
We have several assertions to be proved, the idea being as follows:

\medskip

(1) This follows indeed from $(id\otimes\int)\Phi(f)=\int\alpha(f)1$, which gives:
$$f=\int\alpha(f)1$$

(2) The fact that $tr=\int\alpha$ is indeed $G$-invariant can be checked as follows:
\begin{eqnarray*}
(tr\otimes id)\Phi f
&=&(\smallint\alpha\otimes id)\Phi f\\
&=&(\smallint\otimes id)\Delta\alpha f\\
&=&\smallint\alpha(f)1\\
&=&tr(f)1
\end{eqnarray*}

As for the uniqueness assertion, this follows as before.

\medskip

(3) The condition $(\alpha\otimes id)\Phi=\Delta\alpha$, together with the fact that $i$ is injective, allows us to factorize $\Delta$ into a morphism $\Psi$, as follows:
$$\xymatrix@R=12mm@C=30mm{
C(X)\ar[r]^\Phi\ar[d]_\alpha&C(X)\otimes C(G)\ar[d]^{\alpha\otimes id}\\
C(Y)\ar@.[r]^\Psi\ar[d]_i&C(Y)\otimes C(G)\ar[d]^{i\otimes id}\\
C(G)\ar[r]^\Delta&C(G)\otimes C(G)
}$$

Thus the image space $G\to Y$ is indeed homogeneous, and we are done.
\end{proof}

Finally, we have the following result, which further clarifies our formalism:

\begin{theorem}
Let $G\to X$ be an extended homogeneous space, and construct quotients $X\to X'$, $G\to G'$ by performing the GNS construction with respect to $\int\alpha,\int$. Then $\alpha$ factorizes into an inclusion $\alpha':C(X')\to C(G')$, and we have an homogeneous space.
\end{theorem}

\begin{proof}
We factorize $G\to Y\subset X$ as above. By performing the GNS construction with respect to $\int i\alpha,\int i,\int$, we obtain a diagram as follows:
$$\xymatrix@R=12mm@C=30mm{
C(X)\ar[r]^p\ar[d]_\alpha&C(X')\ar[d]^{\alpha'}\ar[dr]^{tr'}\\
C(Y)\ar[r]^q\ar[d]_i&C(Y')\ar[d]^{i'}&\mathbb C\\
C(G)\ar[r]^r&C(G')\ar[ur]_{\int'}
}$$

Indeed, with $tr=\int\alpha$, the GNS quotient maps $p,q,r$ are defined respectively by:
\begin{eqnarray*}
\ker p&=&\left\{f\in C(X)\Big|tr(f^*f)=0\right\}\\
\ker q&=&\left\{f\in C(Y)\Big|\smallint(f^*f)=0\right\}\\
\ker r&=&\left\{f\in C(G)\Big|\smallint(f^*f)=0\right\}
\end{eqnarray*}

Next, we can define factorizations $i',\alpha'$ as above. Observe that $i'$ is injective, and that $\alpha'$ is surjective. Our claim now is that $\alpha'$ is injective as well. Indeed:
\begin{eqnarray*}
\alpha'p(f)=0
&\implies&q\alpha(f)=0\\
&\implies&\int\alpha(f^*f)=0\\
&\implies&tr(f^*f)=0\\
&\implies&p(f)=0
\end{eqnarray*}

We conclude that we have $X'=Y'$, and this gives the result.
\end{proof}

Summarizing, the basic homogeneous space theory from the classical case extends to the quantum group setting, with a few twists, both of algebraic and analytic nature.

\section*{15b. Partial isometries}

We discuss now some explicit examples of homogeneous spaces. This can be done at several levels of generality, and there has been quite some work here, starting with \cite{bgo}. In what follows we discuss the formalism in \cite{ba4}, which is quite broad, while remaining not very abstract. We will study the spaces of the following type:
$$X=(G_M\times G_N)\big/(G_L\times G_{M-L}\times G_{N-L})$$

These spaces cover indeed the quantum groups and the spheres. And also, they are quite concrete and useful objects, consisting of certain classes of ``partial isometries''. Our main result will be a verification of the Bercovici-Pata liberation criterion, for certain variables associated $\chi\in C(X)$, in a suitable $L,M,N\to\infty$ limit.

\bigskip

We begin with a study in the classical case. Our starting point will be:

\index{partial isometry}

\begin{definition}
Associated to any integers $L\leq M,N$ are the spaces
$$O_{MN}^L=\left\{T:E\to F\ {\rm isometry}\Big|E\subset\mathbb R^N,F\subset\mathbb R^M,\dim_\mathbb RE=L\right\}$$
$$U_{MN}^L=\left\{T:E\to F\ {\rm isometry}\Big|E\subset\mathbb C^N,F\subset\mathbb C^M,\dim_\mathbb CE=L\right\}$$
where the notion of isometry is with respect to the usual real/complex scalar products.
\end{definition}

As a first observation, at $L=M=N$ we obtain the groups $O_N,U_N$:
$$O_{NN}^N=O_N$$
$$U_{NN}^N=U_N$$ 

Another interesting specialization is $L=M=1$. Here the elements of $O_{1N}^1$ are the isometries $T:E\to\mathbb R$, with $E\subset\mathbb R^N$ one-dimensional. But such an isometry is uniquely determined by $T^{-1}(1)\in\mathbb R^N$, which must belong to $S^{N-1}_\mathbb R$. Thus, we have $O_{1N}^1=S^{N-1}_\mathbb R$. Similarly, in the complex case we have $U_{1N}^1=S^{N-1}_\mathbb C$, and so our results here are:
$$O_{1N}^1=S^{N-1}_\mathbb R$$
$$U_{1N}^1=S^{N-1}_\mathbb C$$

Yet another interesting specialization is $L=N=1$. Here the elements of $O_{1N}^1$ are the isometries $T:\mathbb R\to F$, with $F\subset\mathbb R^M$ one-dimensional. But such an isometry is uniquely determined by $T(1)\in\mathbb R^M$, which must belong to $S^{M-1}_\mathbb R$. Thus, we have $O_{M1}^1=S^{M-1}_\mathbb R$. Similarly, in the complex case we have $U_{M1}^1=S^{M-1}_\mathbb C$, and so our results here are:
$$O_{M1}^1=S^{M-1}_\mathbb R$$
$$U_{M1}^1=S^{M-1}_\mathbb C$$

In general, the most convenient is to view the elements of $O_{MN}^L,U_{MN}^L$ as rectangular matrices, and to use matrix calculus for their study. We have indeed:

\begin{proposition}
We have identifications of compact spaces
$$O_{MN}^L\simeq\left\{U\in M_{M\times N}(\mathbb R)\Big|UU^t={\rm projection\ of\ trace}\ L\right\}$$
$$U_{MN}^L\simeq\left\{U\in M_{M\times N}(\mathbb C)\Big|UU^*={\rm projection\ of\ trace}\ L\right\}$$
with each partial isometry being identified with the corresponding rectangular matrix.
\end{proposition}

\begin{proof}
We can indeed identify the partial isometries $T:E\to F$ with their corresponding extensions $U:\mathbb R^N\to\mathbb R^M$, $U:\mathbb C^N\to\mathbb C^M$, obtained by setting $U_{E^\perp}=0$. Then, we can identify these latter maps $U$ with the corresponding rectangular matrices.
\end{proof}

As an illustration, at $L=M=N$ we recover in this way the usual matrix description of $O_N,U_N$. Also, at $L=M=1$ we obtain the usual description of $S^{N-1}_\mathbb R,S^{N-1}_\mathbb C$, as row spaces over the corresponding groups $O_N,U_N$. Finally, at $L=N=1$ we obtain the usual description of $S^{N-1}_\mathbb R,S^{N-1}_\mathbb C$, as column spaces over the corresponding groups $O_N,U_N$. 

\bigskip

Now back to the general case, observe that the isometries $T:E\to F$, or rather their extensions $U:\mathbb K^N\to\mathbb K^M$, with $\mathbb K=\mathbb R,\mathbb C$, obtained by setting $U_{E^\perp}=0$, can be composed with the isometries of $\mathbb K^M,\mathbb K^N$, according to the following scheme:
$$
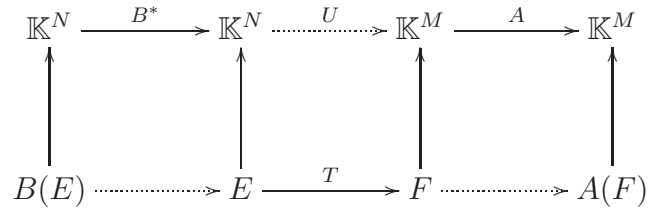
\xymatrix@R=15mm@C=15mm{
\mathbb K^N\ar[r]^{B^*}&\mathbb K^N\ar@.[r]^U&\mathbb K^M\ar[r]^A&\mathbb K^M\\
B(E)\ar@.[r]\ar[u]&E\ar[r]^T\ar[u]&F\ar@.[r]\ar[u]&A(F)\ar[u]
}$$

With the identifications in Proposition 15.11 made, the precise statement here is:

\begin{proposition}
We have action maps as follows, which are both transitive,
$$O_M\times O_N\curvearrowright O_{MN}^L\quad,\quad 
(A,B)U=AUB^t$$
$$U_M\times U_N\curvearrowright U_{MN}^L\quad,\quad 
(A,B)U=AUB^*$$
whose stabilizers are respectively the following groups:
$$O_L\times O_{M-L}\times O_{N-L}$$
$$U_L\times U_{M-L}\times U_{N-L}$$
\end{proposition}

\begin{proof}
We have indeed action maps as in the statement, which are transitive. Let us compute now the stabilizer $G$ of the following point:
$$U=\begin{pmatrix}1&0\\0&0\end{pmatrix}$$

Since $(A,B)\in G$ satisfy $AU=UB$, their components must be of the following form:
$$A=\begin{pmatrix}x&*\\0&a\end{pmatrix}\quad,\quad 
B=\begin{pmatrix}x&0\\ *&b\end{pmatrix}$$

Now since $A,B$ are both unitaries, these matrices follow to be block-diagonal, and so:
$$G=\left\{(A,B)\Big|A=\begin{pmatrix}x&0\\0&a\end{pmatrix},B=\begin{pmatrix}x&0\\ 0&b\end{pmatrix}\right\}$$

The stabilizer of $U$ is then parametrized by triples $(x,a,b)$ belonging respectively to:$$O_L\times O_{M-L}\times O_{N-L}$$
$$U_L\times U_{M-L}\times U_{N-L}$$

Thus, we are led to the conclusion in the statement.
\end{proof}

Finally, let us work out the quotient space description of $O_{MN}^L,U_{MN}^L$. We have here:

\begin{theorem}
We have isomorphisms of homogeneous spaces as follows,
\begin{eqnarray*}
O_{MN}^L&=&(O_M\times O_N)/(O_L\times O_{M-L}\times O_{N-L})\\
U_{MN}^L&=&(U_M\times U_N)/(U_L\times U_{M-L}\times U_{N-L})
\end{eqnarray*}
with the quotient maps being given by $(A,B)\to AUB^*$, where:
$$U=\begin{pmatrix}1&0\\0&0\end{pmatrix}$$
\end{theorem}

\begin{proof}
This is just a reformulation of Proposition 15.12 above, by taking into account the fact that the fixed point used in the proof there was $U=(^1_0{\ }^0_0)$.
\end{proof}

Once again, the basic examples here come from the cases $L=M=N$ and $L=M=1$. At $L=M=N$ the quotient spaces at right are respectively:
$$O_N,U_N$$

At $L=M=1$  the quotient spaces at right are respectively:
$$O_N/O_{N-1}\quad,\quad U_N/U_{N-1}$$

In fact, in the general orthogonal $L=M$ case we obtain the following spaces:
\begin{eqnarray*}
O_{MN}^M
&=&(O_M\times O_N)/(O_M\times O_{N-M})\\
&=&O_N/O_{N-M}
\end{eqnarray*}

Also, in the general unitary $L=M$ case we obtain the following spaces:
\begin{eqnarray*}
U_{MN}^M
&=&(U_M\times U_N)/(U_M\times U_{N-M})\\
&=&U_N/U_{N-M}
\end{eqnarray*}

Similarly, the examples coming from the cases $L=M=N$ and $L=N=1$ are particular cases of the general $L=N$ case, where we obtain the following spaces:
\begin{eqnarray*}
O_{MN}^N
&=&(O_M\times O_N)/(O_M\times O_{M-N})\\
&=&O_N/O_{M-N}
\end{eqnarray*}

In the unitary case, we obtain the following spaces:
\begin{eqnarray*}
U_{MN}^N
&=&(U_M\times U_N)/(U_M\times U_{M-N})\\
&=&U_N/U_{M-N}
\end{eqnarray*}

Summarizing, we have here some basic homogeneous spaces, unifying the real and complex spheres with the orthogonal and unitary groups.

\section*{15c. Free isometries}

We can liberate the spaces $O_{MN}^L,U_{MN}^L$, as follows:

\index{free partial isometry}

\begin{definition}
Associated to any integers $L\leq M,N$ are the algebras
\begin{eqnarray*}
C(O_{MN}^{L+})&=&C^*\left((u_{ij})_{i=1,\ldots,M,j=1,\ldots,N}\Big|u=\bar{u},uu^t={\rm projection\ of\ trace}\ L\right)\\
C(U_{MN}^{L+})&=&C^*\left((u_{ij})_{i=1,\ldots,M,j=1,\ldots,N}\Big|uu^*,\bar{u}u^t={\rm projections\ of\ trace}\ L\right)
\end{eqnarray*}
with the trace being by definition the sum of the diagonal entries.
\end{definition}

Observe that the above universal algebras are indeed well-defined, as it was previously  the case for the free spheres, and this due to the trace conditions, which read: 
$$\sum_{ij}u_{ij}u_{ij}^*
=\sum_{ij}u_{ij}^*u_{ij}
=L$$

We have inclusions between the various spaces constructed so far, as follows:
$$\xymatrix@R=15mm@C=15mm{
O_{MN}^{L+}\ar[r]&U_{MN}^{L+}\\
O_{MN}^L\ar[r]\ar[u]&U_{MN}^L\ar[u]}$$

At the level of basic examples now, we first have the following result:

\begin{proposition}
At $L=M=1$ we obtain the following diagram,
$$\xymatrix@R=15mm@C=15mm{
S^{N-1}_{\mathbb R,+}\ar[r]&S^{N-1}_{\mathbb C,+}\\
S^{N-1}_\mathbb R\ar[r]\ar[u]&S^{N-1}_\mathbb C\ar[u]}$$
and at $L=N=1$ we obtain the following diagram:
$$\xymatrix@R=15mm@C=15mm{
S^{M-1}_{\mathbb R,+}\ar[r]&S^{M-1}_{\mathbb C,+}\\
S^{M-1}_\mathbb R\ar[r]\ar[u]&S^{M-1}_\mathbb C\ar[u]}$$
\end{proposition}

\begin{proof}
Both the assertions are clear from definitions.
\end{proof}

We have as well the following result:

\begin{proposition}
At $L=M=N$ we obtain the diagram
$$\xymatrix@R=15mm@C=15mm{
O_N^+\ar[r]&U_N^+\\
O_N\ar[r]\ar[u]&U_N\ar[u]}$$
consisting of the groups $O_N,U_N$, and their liberations.
\end{proposition}

\begin{proof}
We recall that the various quantum groups in the statement are constructed as follows, with the symbol $\times$ standing once again for ``commutative'' and ``free'':
\begin{eqnarray*}
C(O_N^\times)&=&C^*_\times\left((u_{ij})_{i,j=1,\ldots,N}\Big|u=\bar{u},uu^t=u^tu=1\right)\\
C(U_N^\times)&=&C^*_\times\left((u_{ij})_{i,j=1,\ldots,N}\Big|uu^*=u^*u=1,\bar{u}u^t=u^t\bar{u}=1\right)
\end{eqnarray*}

On the other hand, according to Proposition 15.11 and to Definition 15.14 above, we have the following presentation results:
\begin{eqnarray*}
C(O_{NN}^{N\times})&=&C^*_\times\left((u_{ij})_{i,j=1,\ldots,N}\Big|u=\bar{u},uu^t={\rm projection\ of\ trace}\ N\right)\\
C(U_{NN}^{N\times})&=&C^*_\times\left((u_{ij})_{i,j=1,\ldots,N}\Big|uu^*,\bar{u}u^t={\rm projections\ of\ trace}\ N\right)
\end{eqnarray*}

We use now the standard fact that if $p=aa^*$ is a projection then $q=a^*a$ is a projection too. We use as well the following formulae:
$$Tr(uu^*)=Tr(u^t\bar{u})$$
$$Tr(\bar{u}u^t)=Tr(u^*u)$$

We therefore obtain the following formulae:
\begin{eqnarray*}
C(O_{NN}^{N\times})&=&C^*_\times\left((u_{ij})_{i,j=1,\ldots,N}\Big|u=\bar{u},\ uu^t,u^tu={\rm projections\ of\ trace}\ N\right)\\
C(U_{NN}^{N\times})&=&C^*_\times\left((u_{ij})_{i,j=1,\ldots,N}\Big|uu^*,u^*u,\bar{u}u^t,u^t\bar{u}={\rm projections\ of\ trace}\ N\right)
\end{eqnarray*}

Now observe that, in tensor product notation, the conditions at right are all of the form $(tr\otimes id)p=1$. Thus, $p$ must be follows, for the above conditions:
$$p=uu^*,u^*u,\bar{u}u^t,u^t\bar{u}$$

We therefore obtain that, for any faithful state $\varphi$, we have:
$$(tr\otimes\varphi)(1-p)=0$$  

It follows from this that the following projections must be all equal to the identity:
$$p=uu^*,u^*u,\bar{u}u^t,u^t\bar{u}$$

But this leads to the conclusion in the statement.
\end{proof}

Regarding now the homogeneous space structure of $O_{MN}^{L\times},U_{MN}^{L\times}$, the situation here is a bit more complicated in the free case than in the classical case, due to a number of algebraic and analytic issues. We first have the following result:

\begin{proposition}
The spaces $U_{MN}^{L\times}$ have the following properties:
\begin{enumerate}
\item We have an action $U_M^\times\times U_N^\times\curvearrowright U_{MN}^{L\times}$, given by:
$$u_{ij}\to\sum_{kl}u_{kl}\otimes a_{ki}\otimes b_{lj}^*$$

\item We have a map $U_M^\times\times U_N^\times\to U_{MN}^{L\times}$, given by: 
$$u_{ij}\to\sum_{r\leq L}a_{ri}\otimes b_{rj}^*$$
\end{enumerate}
Similar results hold for the spaces $O_{MN}^{L\times}$, with all the $*$ exponents removed.
\end{proposition}

\begin{proof}
In the classical case, consider the following action and quotient maps:
$$U_M\times U_N\curvearrowright U_{MN}^L$$
$$U_M\times U_N\to U_{MN}^L$$

The transposes of these two maps are as follows, where $J=(^1_0{\ }^0_0)$:
\begin{eqnarray*}
\varphi&\to&((U,A,B)\to\varphi(AUB^*))\\
\varphi&\to&((A,B)\to\varphi(AJB^*))
\end{eqnarray*}

But with $\varphi=u_{ij}$ we obtain precisely the formulae in the statement. The proof in the orthogonal case is similar. Regarding now the free case, the proof goes as follows:

\medskip

(1) Assuming $uu^*u=u$, let us set:
$$U_{ij}=\sum_{kl}u_{kl}\otimes a_{ki}\otimes b_{lj}^*$$

We have then the following computation:
\begin{eqnarray*}
(UU^*U)_{ij}
&=&\sum_{pq}\sum_{klmnst}u_{kl}u_{mn}^*u_{st}\otimes a_{ki}a_{mq}^*a_{sq}\otimes b_{lp}^*b_{np}b_{tj}^*\\
&=&\sum_{klmt}u_{kl}u_{ml}^*u_{mt}\otimes a_{ki}\otimes b_{tj}^*\\
&=&\sum_{kt}u_{kt}\otimes a_{ki}\otimes b_{tj}^*\\
&=&U_{ij}
\end{eqnarray*}

Also, assuming that we have $\sum_{ij}u_{ij}u_{ij}^*=L$, we obtain:
\begin{eqnarray*}
\sum_{ij}U_{ij}U_{ij}^*
&=&\sum_{ij}\sum_{klst}u_{kl}u_{st}^*\otimes a_{ki}a_{si}^*\otimes b_{lj}^*b_{tj}\\
&=&\sum_{kl}u_{kl}u_{kl}^*\otimes1\otimes1\\
&=&L
\end{eqnarray*}

(2) Assuming $uu^*u=u$, let us set:
$$V_{ij}=\sum_{r\leq L}a_{ri}\otimes b_{rj}^*$$

We have then the following computation:
\begin{eqnarray*}
(VV^*V)_{ij}
&=&\sum_{pq}\sum_{x,y,z\leq L}a_{xi}a_{yq}^*a_{zq}\otimes b_{xp}^*b_{yp}b_{zj}^*\\
&=&\sum_{x\leq L}a_{xi}\otimes b_{xj}^*\\
&=&V_{ij}
\end{eqnarray*}

Also, assuming that we have $\sum_{ij}u_{ij}u_{ij}^*=L$, we obtain:
\begin{eqnarray*}
\sum_{ij}V_{ij}V_{ij}^*
&=&\sum_{ij}\sum_{r,s\leq L}a_{ri}a_{si}^*\otimes b_{rj}^*b_{sj}\\
&=&\sum_{l\leq L}1\\
&=&L
\end{eqnarray*}

By removing all the $*$ exponents, we obtain as well the orthogonal results.
\end{proof}

Let us examine now the relation between the above maps. In the classical case, given a quotient space $X=G/H$, the associated action and quotient maps are given by:
$$\begin{cases}
a:X\times G\to X&:\quad (Hg,h)\to Hgh\\
p:G\to X&:\quad g\to Hg
\end{cases}$$

Thus we have $a(p(g),h)=p(gh)$. In our context, a similar result holds: 

\index{extended homogeneous space}

\begin{theorem}
With $G=G_M\times G_N$ and $X=G_{MN}^L$, where $G_N=O_N^\times,U_N^\times$, we have
$$\xymatrix@R=15mm@C=30mm{
G\times G\ar[r]^m\ar[d]_{p\times id}&G\ar[d]^p\\
X\times G\ar[r]^a&X
}$$
where $a,p$ are the action map and the map constructed in Proposition 15.17.
\end{theorem}

\begin{proof}
At the level of the associated algebras of functions, we must prove that the following diagram commutes, where $\Phi,\alpha$ are morphisms of algebras induced by $a,p$:
$$\xymatrix@R=15mm@C=25mm{
C(X)\ar[r]^\Phi\ar[d]_\alpha&C(X\times G)\ar[d]^{\alpha\otimes id}\\
C(G)\ar[r]^\Delta&C(G\times G)
}$$

When going right, and then down, the composition is as follows:
\begin{eqnarray*}
(\alpha\otimes id)\Phi(u_{ij})
&=&(\alpha\otimes id)\sum_{kl}u_{kl}\otimes a_{ki}\otimes b_{lj}^*\\
&=&\sum_{kl}\sum_{r\leq L}a_{rk}\otimes b_{rl}^*\otimes a_{ki}\otimes b_{lj}^*
\end{eqnarray*}

On the other hand, when going down, and then right, the composition is as follows, where $F_{23}$ is the flip between the second and the third components:
\begin{eqnarray*}
\Delta\pi(u_{ij})
&=&F_{23}(\Delta\otimes\Delta)\sum_{r\leq L}a_{ri}\otimes b_{rj}^*\\
&=&F_{23}\left(\sum_{r\leq L}\sum_{kl}a_{rk}\otimes a_{ki}\otimes b_{rl}^*\otimes b_{lj}^*\right)
\end{eqnarray*}

Thus the above diagram commutes indeed, and this gives the result.
\end{proof}

Let us discuss now some discrete extensions of the above constructions:

\index{partial permutation}
\index{free partial permutation}

\begin{definition}
Associated to any partial permutation, $\sigma:I\simeq J$ with $I\subset\{1,\ldots,N\}$ and $J\subset\{1,\ldots,M\}$, is the real/complex partial isometry
$$T_\sigma:span\left(e_i\Big|i\in I\right)\to span\left(e_j\Big|j\in J\right)$$
given on the standard basis elements by $T_\sigma(e_i)=e_{\sigma(i)}$.
\end{definition}

Let $S_{MN}^L$ be the set of partial permutations $\sigma:I\simeq J$ as above, with range $I\subset\{1,\ldots,N\}$ and target $J\subset\{1,\ldots,M\}$, and with $L=|I|=|J|$. We have:

\begin{proposition}
The space of partial permutations signed by elements of $\mathbb Z_s$,
$$H_{MN}^{sL}=\left\{T(e_i)=w_ie_{\sigma(i)}\Big|\sigma\in S_{MN}^L,w_i\in\mathbb Z_s\right\}$$
is isomorphic to the quotient space 
$$(H_M^s\times H_N^s)/(H_L^s\times H_{M-L}^s\times H_{N-L}^s)$$
via a standard isomorphism.
\end{proposition}

\begin{proof}
This follows by adapting the computations in the proof of Proposition 15.12 above. Indeed, we have an action map as follows, which is transitive:
$$H_M^s\times H_N^s\to H_{MN}^{sL}\quad,\quad 
(A,B)U=AUB^*$$

Consider now the following point:
$$U=\begin{pmatrix}1&0\\0&0\end{pmatrix}$$

The stabilizer of this point follows to be the following group:
$$H_L^s\times H_{M-L}^s\times H_{N-L}^s$$

To be more precise, this group is embedded via:
$$(x,a,b)\to\left[\begin{pmatrix}x&0\\0&a\end{pmatrix},\begin{pmatrix}x&0\\0&b\end{pmatrix}\right]$$

But this gives the result.
\end{proof}

In the free case now, the idea is similar, by using inspiration from the construction of the quantum group $H_N^{s+}=\mathbb Z_s\wr_*S_N^+$ in \cite{bb+}. The result here is as follows:

\begin{proposition}
The compact quantum space $H_{MN}^{sL+}$ associated to the algebra
$$C(H_{MN}^{sL+})=C(U_{MN}^{L+})\Big/\left<u_{ij}u_{ij}^*=u_{ij}^*u_{ij}=p_{ij}={\rm projections},u_{ij}^s=p_{ij}\right>$$
has an action map, and is the target of a quotient map, as in Theorem 15.18 above.
\end{proposition}

\begin{proof}
We must show that if the variables $u_{ij}$ satisfy the relations in the statement, then these relations are satisfied as well for the following variables: 
$$U_{ij}=\sum_{kl}u_{kl}\otimes a_{ki}\otimes b_{lj}^*\quad,\quad 
V_{ij}=\sum_{r\leq L}a_{ri}\otimes b_{rj}^*$$

We use the fact that the standard coordinates $a_{ij},b_{ij}$ on the quantum groups $H_M^{s+},H_N^{s+}$ satisfy the following relations, for any $x\neq y$ on the same row or column of $a,b$:
$$xy=xy^*=0$$
 
We obtain, by using these relations:
\begin{eqnarray*}
U_{ij}U_{ij}^*
&=&\sum_{klmn}u_{kl}u_{mn}^*\otimes a_{ki}a_{mi}^*\otimes b_{lj}^*b_{mj}\\
&=&\sum_{kl}u_{kl}u_{kl}^*\otimes a_{ki}a_{ki}^*\otimes b_{lj}^*b_{lj}
\end{eqnarray*}

We have as well the following formula:
\begin{eqnarray*}
V_{ij}V_{ij}^*
&=&\sum_{r,t\leq L}a_{ri}a_{ti}^*\otimes b_{rj}^*b_{tj}\\
&=&\sum_{r\leq L}a_{ri}a_{ri}^*\otimes b_{rj}^*b_{rj}
\end{eqnarray*}

In terms of the projections $x_{ij}=a_{ij}a_{ij}^*$, $y_{ij}=b_{ij}b_{ij}^*$, $p_{ij}=u_{ij}u_{ij}^*$, we have:
$$U_{ij}U_{ij}^*=\sum_{kl}p_{kl}\otimes x_{ki}\otimes y_{lj}$$
$$V_{ij}V_{ij}^*=\sum_{r\leq L}x_{ri}\otimes y_{rj}$$

By repeating the computation, we conclude that these elements are projections. Also, a similar computation shows that $U_{ij}^*U_{ij},V_{ij}^*V_{ij}$ are given by the same formulae. Finally, once again by using the relations of type $xy=xy^*=0$, we have:
\begin{eqnarray*}
U_{ij}^s
&=&\sum_{k_rl_r}u_{k_1l_1}\ldots u_{k_sl_s}\otimes a_{k_1i}\ldots a_{k_si}\otimes b_{l_1j}^*\ldots b_{l_sj}^*\\
&=&\sum_{kl}u_{kl}^s\otimes a_{ki}^s\otimes(b_{lj}^*)^s
\end{eqnarray*}

We have as well the following formula:
\begin{eqnarray*}
V_{ij}^s
&=&\sum_{r_l\leq L}a_{r_1i}\ldots a_{r_si}\otimes b_{r_1j}^*\ldots b_{r_sj}^*\\
&=&\sum_{r\leq L}a_{ri}^s\otimes(b_{rj}^*)^s
\end{eqnarray*}

Thus the conditions of type $u_{ij}^s=p_{ij}$ are satisfied as well, and we are done.
\end{proof}

Let us discuss now the general case. We have the following result:

\begin{proposition}
The various spaces $G_{MN}^L$ constructed so far appear by imposing to the standard coordinates of $U_{MN}^{L+}$ the relations
$$\sum_{i_1\ldots i_s}\sum_{j_1\ldots j_s}\delta_\pi(i)\delta_\sigma(j)u_{i_1j_1}^{e_1}\ldots u_{i_sj_s}^{e_s}=L^{|\pi\vee\sigma|}$$
with $s=(e_1,\ldots,e_s)$ ranging over all the colored integers, and with $\pi,\sigma\in D(0,s)$.
\end{proposition}

\begin{proof}
According to the various constructions above, the relations defining $G_{MN}^L$ can be written as follows, with $\sigma$ ranging over a family of generators, with no upper legs, of the corresponding category of partitions $D$:
$$\sum_{j_1\ldots j_s}\delta_\sigma(j)u_{i_1j_1}^{e_1}\ldots u_{i_sj_s}^{e_s}=\delta_\sigma(i)$$

We therefore obtain the relations in the statement, as follows:
\begin{eqnarray*}
\sum_{i_1\ldots i_s}\sum_{j_1\ldots j_s}\delta_\pi(i)\delta_\sigma(j)u_{i_1j_1}^{e_1}\ldots u_{i_sj_s}^{e_s}
&=&\sum_{i_1\ldots i_s}\delta_\pi(i)\sum_{j_1\ldots j_s}\delta_\sigma(j)u_{i_1j_1}^{e_1}\ldots u_{i_sj_s}^{e_s}\\
&=&\sum_{i_1\ldots i_s}\delta_\pi(i)\delta_\sigma(i)\\
&=&L^{|\pi\vee\sigma|}
\end{eqnarray*}

As for the converse, this follows by using the relations in the statement, by keeping $\pi$ fixed, and by making $\sigma$ vary over all the partitions in the category.
\end{proof}

In the general case now, where $G=(G_N)$ is an arbitary uniform easy quantum group, we can construct spaces $G_{MN}^L$ by using the above relations, and we have:

\index{extended homogeneous space}

\begin{theorem}
The spaces $G_{MN}^L\subset U_{MN}^{L+}$ constructed by imposing the relations 
$$\sum_{i_1\ldots i_s}\sum_{j_1\ldots j_s}\delta_\pi(i)\delta_\sigma(j)u_{i_1j_1}^{e_1}\ldots u_{i_sj_s}^{e_s}=L^{|\pi\vee\sigma|}$$
with $\pi,\sigma$ ranging over all the partitions in the associated category, having no upper legs, are subject to an action map/quotient map diagram, as in Theorem 15.18.
\end{theorem}

\begin{proof}
We proceed as in the proof of Proposition 15.17. We must prove that, if the variables $u_{ij}$ satisfy the relations in the statement, then so do the following variables:
$$U_{ij}=\sum_{kl}u_{kl}\otimes a_{ki}\otimes b_{lj}^*\quad,\quad 
V_{ij}=\sum_{r\leq L}a_{ri}\otimes b_{rj}^*$$

Regarding the variables $U_{ij}$, the computation here goes as follows:
\begin{eqnarray*}
&&\sum_{i_1\ldots i_s}\sum_{j_1\ldots j_s}\delta_\pi(i)\delta_\sigma(j)U_{i_1j_1}^{e_1}\ldots U_{i_sj_s}^{e_s}\\
&=&\sum_{i_1\ldots i_s}\sum_{j_1\ldots j_s}\sum_{k_1\ldots k_s}\sum_{l_1\ldots l_s}u_{k_1l_1}^{e_1}\ldots u_{k_sl_s}^{e_s}\otimes \delta_\pi(i)\delta_\sigma(j)a_{k_1i_1}^{e_1}\ldots a_{k_si_s}^{e_s}\otimes(b_{l_sj_s}^{e_s}\ldots b_{l_1j_1}^{e_1})^*\\
&=&\sum_{k_1\ldots k_s}\sum_{l_1\ldots l_s}\delta_\pi(k)\delta_\sigma(l)u_{k_1l_1}^{e_1}\ldots u_{k_sl_s}^{e_s}\\
&=&L^{|\pi\vee\sigma|}
\end{eqnarray*}

For the variables $V_{ij}$ the proof is similar, as follows:
\begin{eqnarray*}
&&\sum_{i_1\ldots i_s}\sum_{j_1\ldots j_s}\delta_\pi(i)\delta_\sigma(j)V_{i_1j_1}^{e_1}\ldots V_{i_sj_s}^{e_s}\\
&=&\sum_{i_1\ldots i_s}\sum_{j_1\ldots j_s}\sum_{l_1,\ldots,l_s\leq L}\delta_\pi(i)\delta_\sigma(j)a_{l_1i_1}^{e_1}\ldots a_{l_si_s}^{e_s}\otimes(b_{l_sj_s}^{e_s}\ldots b_{l_1j_1}^{e_1})^*\\
&=&\sum_{l_1,\ldots,l_s\leq L}\delta_\pi(l)\delta_\sigma(l)\\
&=&L^{|\pi\vee\sigma|}
\end{eqnarray*}

Thus we have constructed an action map, and a quotient map, as in Proposition 15.17 above, and the commutation of the diagram in Theorem 15.18 is then trivial.
\end{proof}

\section*{15d. Integration theory}

Still following \cite{ba4}, let us discuss now the integration over $G_{MN}^L$. We have:

\begin{definition}
The integration functional of $G_{MN}^L$ is the composition
$$\int_{G_{MN}^L}:C(G_{MN}^L)\to C(G_M\times G_N)\to\mathbb C$$
of the representation $u_{ij}\to\sum_{r\leq L}a_{ri}\otimes b_{rj}^*$ with the Haar functional of $G_M\times G_N$.
\end{definition}

Observe that in the case $L=M=N$ we obtain the integration over $G_N$. Also, at $L=M=1$, or at $L=N=1$, we obtain the integration over the sphere. In the general case now, we first have the following result:

\begin{proposition}
The integration functional of $G_{MN}^L$ has the invariance property 
$$\left(\int_{G_{MN}^L}\!\otimes\ id\right)\Phi(x)=\int_{G_{MN}^L}x$$
with respect to the coaction map:
$$\Phi(u_{ij})=\sum_{kl}u_{kl}\otimes a_{ki}\otimes b_{lj}^*$$
\end{proposition}

\begin{proof}
We restrict the attention to the orthogonal case, the proof in the unitary case being similar. We must check the following formula:
$$\left(\int_{G_{MN}^L}\!\otimes\ id\right)\Phi(u_{i_1j_1}\ldots u_{i_sj_s})=\int_{G_{MN}^L}u_{i_1j_1}\ldots u_{i_sj_s}$$

Let us compute the left term. This is given by:
\begin{eqnarray*}
X
&=&\left(\int_{G_{MN}^L}\!\otimes\ id\right)\sum_{k_xl_x}u_{k_1l_1}\ldots u_{k_sl_s}\otimes a_{k_1i_1}\ldots a_{k_si_s}\otimes b_{l_1j_1}^*\ldots b_{l_sj_s}^*\\
&=&\sum_{k_xl_x}\sum_{r_x\leq L}a_{k_1i_1}\ldots a_{k_si_s}\otimes b_{l_1j_1}^*\ldots b_{l_sj_s}^*\int_{G_M}a_{r_1k_1}\ldots a_{r_sk_s}\int_{G_N}b_{r_1l_1}^*\ldots b_{r_sl_s}^*\\
&=&\sum_{r_x\leq L}\sum_{k_x}a_{k_1i_1}\ldots a_{k_si_s}\int_{G_M}a_{r_1k_1}\ldots a_{r_sk_s}
\otimes\sum_{l_x}b_{l_1j_1}^*\ldots b_{l_sj_s}^*\int_{G_N}b_{r_1l_1}^*\ldots b_{r_sl_s}^*
\end{eqnarray*}

By using now the invariance property of the Haar functionals of $G_M,G_N$, we obtain:
\begin{eqnarray*}
X
&=&\sum_{r_x\leq L}\left(\int_{G_M}\!\otimes\ id\right)\Delta(a_{r_1i_1}\ldots a_{r_si_s})
\otimes\left(\int_{G_N}\!\otimes\ id\right)\Delta(b_{r_1j_1}^*\ldots b_{r_sj_s}^*)\\
&=&\sum_{r_x\leq L}\int_{G_M}a_{r_1i_1}\ldots a_{r_si_s}\int_{G_N}b_{r_1j_1}^*\ldots b_{r_sj_s}^*\\
&=&\left(\int_{G_M}\otimes\int_{G_N}\right)\sum_{r_x\leq L}a_{r_1i_1}\ldots a_{r_si_s}\otimes b_{r_1j_1}^*\ldots b_{r_sj_s}^*
\end{eqnarray*}

But this gives the formula in the statement, and we are done.
\end{proof}

We prove now that the above functional is in fact the unique positive unital invariant trace on $C(G_{MN}^L)$. For this purpose, we will need the Weingarten formula:

\index{Weingarten formula}

\begin{theorem}
We have the Weingarten type formula
$$\int_{G_{MN}^L}u_{i_1j_1}\ldots u_{i_sj_s}=\sum_{\pi\sigma\tau\nu}L^{|\pi\vee\tau|}\delta_\sigma(i)\delta_\nu(j)W_{sM}(\pi,\sigma)W_{sN}(\tau,\nu)$$
where the matrices on the right are given by $W_{sM}=G_{sM}^{-1}$, with $G_{sM}(\pi,\sigma)=M^{|\pi\vee\sigma|}$.
\end{theorem}

\begin{proof}
By using the Weingarten formula for $G_M,G_N$, we obtain:
\begin{eqnarray*}
\int_{G_{MN}^L}u_{i_1j_1}\ldots u_{i_sj_s}
&=&\sum_{l_1\ldots l_s\leq L}\int_{G_M}a_{l_1i_1}\ldots a_{l_si_s}\int_{G_N}b_{l_1j_1}^*\ldots b_{l_sj_s}^*\\
&=&\sum_{l_1\ldots l_s\leq L}\sum_{\pi\sigma}\delta_\pi(l)\delta_\sigma(i)W_{sM}(\pi,\sigma)\sum_{\tau\nu}\delta_\tau(l)\delta_\nu(j)W_{sN}(\tau,\nu)\\
&=&\sum_{\pi\sigma\tau\nu}\left(\sum_{l_1\ldots l_s\leq L}\delta_\pi(l)\delta_\tau(l)\right)\delta_\sigma(i)\delta_\nu(j)W_{sM}(\pi,\sigma)W_{sN}(\tau,\nu)
\end{eqnarray*}

The coefficient being $L^{|\pi\vee\tau|}$, we obtain the formula in the statement.
\end{proof}

We can now derive an abstract characterization of the integration, as follows:

\begin{theorem}
The integration of $G_{MN}^L$ is the unique positive unital trace 
$$C(G_{MN}^L)\to\mathbb C$$
which is invariant under the action of the quantum group $G_M\times G_N$.
\end{theorem}

\begin{proof}
The idea, from \cite{bgo}, will be that of proving the following ergodicity formula: 
$$\left(id\otimes\int_{G_M}\otimes\int_{G_N}\right)\Phi(x)=\int_{G_{MN}^L}x$$

We restrict the attention to the orthogonal case, the proof in the unitary case being similar. We must verify that the following holds:
$$\left(id\otimes\int_{G_M}\otimes\int_{G_N}\right)\Phi(u_{i_1j_1}\ldots u_{i_sj_s})=\int_{G_{MN}^L}u_{i_1j_1}\ldots u_{i_sj_s}$$

By using the Weingarten formula, the left term can be written as follows:
\begin{eqnarray*}
X
&=&\sum_{k_1\ldots k_s}\sum_{l_1\ldots l_s}u_{k_1l_1}\ldots u_{k_sl_s}\int_{G_M}a_{k_1i_1}\ldots a_{k_si_s}\int_{G_N}b_{l_1j_1}^*\ldots b_{l_sj_s}^*\\
&=&\sum_{k_1\ldots k_s}\sum_{l_1\ldots l_s}u_{k_1l_1}\ldots u_{k_sl_s}\sum_{\pi\sigma}\delta_\pi(k)\delta_\sigma(i)W_{sM}(\pi,\sigma)\sum_{\tau\nu}\delta_\tau(l)\delta_\nu(j)W_{sN}(\tau,\nu)\\
&=&\sum_{\pi\sigma\tau\nu}\delta_\sigma(i)\delta_\nu(j)W_{sM}(\pi,\sigma)W_{sN}(\tau,\nu)\sum_{k_1\ldots k_s}\sum_{l_1\ldots l_s}\delta_\pi(k)\delta_\tau(l)u_{k_1l_1}\ldots u_{k_sl_s}
\end{eqnarray*}

By using now the summation formula in Theorem 15.23, we obtain:
$$X=\sum_{\pi\sigma\tau\nu}L^{|\pi\vee\tau|}\delta_\sigma(i)\delta_\nu(j)W_{sM}(\pi,\sigma)W_{sN}(\tau,\nu)$$

Now by comparing with the Weingarten formula for $G_{MN}^L$, this proves our claim. Assume now that $\tau:C(G_{MN}^L)\to\mathbb C$ satisfies the invariance condition. We have:
\begin{eqnarray*}
\tau\left(id\otimes\int_{G_M}\otimes\int_{G_N}\right)\Phi(x)
&=&\left(\tau\otimes\int_{G_M}\otimes\int_{G_N}\right)\Phi(x)\\
&=&\left(\int_{G_M}\otimes\int_{G_N}\right)(\tau\otimes id)\Phi(x)\\
&=&\left(\int_{G_M}\otimes\int_{G_N}\right)(\tau(x)1)\\
&=&\tau(x)
\end{eqnarray*}

On the other hand, according to the formula established above, we have as well:
$$\tau\left(id\otimes\int_{G_M}\otimes\int_{G_N}\right)\Phi(x)
=\tau(tr(x)1)
=tr(x)$$

Thus we obtain $\tau=tr$, and this finishes the proof.
\end{proof}

As a main application of the above integration technology, still following \cite{ba4}, we have the following result, extending previous computations for quantum group characters:

\begin{proposition}
For a sum of coordinates
$$\chi_E=\sum_{(ij)\in E}u_{ij}$$
which do not overlap on rows and columns we have
$$\int_{G_{MN}^L}\chi_E^s=\sum_{\pi\sigma\tau\nu}K^{|\pi\vee\tau|}L^{|\sigma\vee\nu|}W_{sM}(\pi,\sigma)W_{sN}(\tau,\nu)$$
where $K=|E|$ is the cardinality of the indexing set.
\end{proposition}

\begin{proof}
With $K=|E|$, we can write $E=\{(\alpha(i),\beta(i))\}$, for certain embeddings:
$$\alpha:\{1,\ldots,K\}\subset\{1,\ldots,M\}$$
$$\beta:\{1,\ldots,K\}\subset\{1,\ldots,N\}$$

In terms of these maps $\alpha,\beta$, the moment in the statement is given by:
$$M_s=\int_{G_{MN}^L}\left(\sum_{i\leq K}u_{\alpha(i)\beta(i)}\right)^s$$

By using the Weingarten formula, we can write this quantity as follows:
\begin{eqnarray*}
&&M_s\\
&=&\int_{G_{MN}^L}\sum_{i_1\ldots i_s\leq K}u_{\alpha(i_1)\beta(i_1)}\ldots u_{\alpha(i_s)\beta(i_s)}\\
&=&\sum_{i_1\ldots i_s\leq K}\sum_{\pi\sigma\tau\nu}L^{|\sigma\vee\nu|}\delta_\pi(\alpha(i_1),\ldots,\alpha(i_s))\delta_\tau(\beta(i_1),\ldots,\beta(i_s))W_{sM}(\pi,\sigma)W_{sN}(\tau,\nu)\\
&=&\sum_{\pi\sigma\tau\nu}\left(\sum_{i_1\ldots i_s\leq K}\delta_\pi(i)\delta_\tau(i)\right)L^{|\sigma\vee\nu|}W_{sM}(\pi,\sigma)W_{sN}(\tau,\nu)
\end{eqnarray*}

But, as explained before, the coefficient on the left in the last formula is:
$$C=K^{|\pi\vee\tau|}$$

We therefore obtain the formula in the statement.
\end{proof}

We can further advance in the classical/twisted and free cases, where the Weingarten theory for the corresponding quantum groups is available from \cite{bb+}, \cite{bc1}, \cite{bsp}, \cite{csn}:

\index{Bercovici-Pata bijection}

\begin{theorem}
In the context of the liberation operations 
$$O_{MN}^L\to O_{MN}^{L+}\quad,\quad 
U_{MN}^L\to U_{MN}^{L+}\quad,\quad 
H_{MN}^{sL}\to H_{MN}^{sL+}$$ 
the laws of the sums of non-overlapping coordinates,
$$\chi_E=\sum_{(ij)\in E}u_{ij}$$
are in Bercovici-Pata bijection, in the 
$$|E|=\kappa N,L=\lambda N,M=\mu N$$
regime and $N\to\infty$ limit.
\end{theorem}

\begin{proof}
We use general theory from \cite{bb+}, \cite{bc1}, \cite{bsp}, \cite{csn}. According to Proposition 15.28, in terms of $K=|E|$, the moments of the variables in the statement are:
$$M_s=\sum_{\pi\sigma\tau\nu}K^{|\pi\vee\tau|}L^{|\sigma\vee\nu|}W_{sM}(\pi,\sigma)W_{sN}(\tau,\nu)$$

We use now two standard facts, namely:

\medskip

(1) The fact that in the $N\to\infty$ limit the Weingarten matrix $W_{sN}$ is concentrated on the diagonal. This is indeed something very standard.

\medskip

(2) The fact that we have an inequality as follows, with equality when $\pi=\sigma$:
$$|\pi\vee\sigma|\leq\frac{|\pi|+|\sigma|}{2}$$

Again, this is standard, and for details on all this, we refer to \cite{bc1}. 

\medskip

Let us discuss now what happens in the regime from the statement, namely:
$$K=\kappa N,L=\lambda N,M=\mu N,N\to\infty$$

In this regime, we obtain:
\begin{eqnarray*}
M_s
&\simeq&\sum_{\pi\tau}K^{|\pi\vee\tau|}L^{|\pi\vee\tau|}M^{-|\pi|}N^{-|\tau|}\\
&\simeq&\sum_\pi K^{|\pi|}L^{|\pi|}M^{-|\pi|}N^{-|\pi|}\\
&=&\sum_\pi\left(\frac{\kappa\lambda}{\mu}\right)^{|\pi|}
\end{eqnarray*}

In order to interpret this formula, we use general theory from \cite{bb+}, \cite{bc1}, \cite{bsp}, \cite{csn}:

\medskip

(1) For $G_N=O_N,\bar{O}_N/O_N^+$, the above variables $\chi_E$ follow to be asymptotically Gaussian/semicircular, of parameter $\frac{\kappa\lambda}{\mu}$, and hence in Bercovici-Pata bijection.

\medskip

(2) For $G_N=U_N,\bar{U}_N/U_N^+$ the situation is similar, with $\chi_E$ being asymptotically complex Gaussian/circular, of parameter $\frac{\kappa\lambda}{\mu}$, and in Bercovici-Pata bijection. 

\medskip

(3) Finally, for $G_N=H_N^s/H_N^{s+}$, the variables $\chi_E$ are asymptotically Bessel/free Bessel of parameter $\frac{\kappa\lambda}{\mu}$, and once again in Bercovici-Pata bijection.  
\end{proof}

There are several possible extensions of the above result, for instance by using quantum reflection groups instead of unitary quantum groups, and by using twisting operations as well. Finally, there are many interesting questions in relation with Connes' noncommutative geometry \cite{con}, and more specifically with the quantum extension of the Nash embedding theorem \cite{nas}. We refer here to \cite{dfw}, \cite{dgo}, \cite{gos} and related papers.

\section*{15e. Exercises} 

There are several interesting questions, which appear as a continuation of the material from this chapter. As a first exercise about this, we have:

\begin{exercise}
Work out explicitely the algebraic and probabilistic theory of the spaces $G_{MN}^L$, in the case $G=S_N,S_N^+$.
\end{exercise}

To be more precise here, the general case $G=H_N^s,H_N^{s+}$ was discussed in the above, and the problem is that of working out the particular case $s=1$ of all this.

\begin{exercise}
Work out the algebraic and probabilistic theory of the spaces $G_{MN}^L$, in the particular cases $L=M$ and $L=N$.
\end{exercise}

The point here is that the case $L=M=N$ corresponds to the quantum groups themselves, and so the cases $L=M$ and $L=N$ correspond to a ``minimal'' extension of the usual theory of the quantum groups, which is worth to be worked out in detail.

\chapter{Modelling questions}

\section*{16a. Matrix models}

In this final chapter we discuss one more ``advanced topic'', namely the use of matrix models for the study of the closed subgroups $G\subset U_N^+$. The idea here, that we have not tried yet in this book, is something extremely simple, namely that of modelling the standard coordinates $u_{ij}\in C(G)$ by certain concrete variables $U_{ij}\in B$. 

\bigskip

Indeed, assuming that the model is faithful in some suitable sense, that the algebra $B$ is something quite familiar, and that the variables $U_{ij}$ are not too complicated, all the questions about $G$ would correspond in this way to routine questions inside $B$.

\bigskip

Regarding the choice of the target algebra $B$, some very convenient algebras are the random matrix ones, $B=M_K(C(T))$, with $K\in\mathbb N$, and with $T$ being a compact space. These algebras generalize indeed the most familiar algebras that we know, namely the matrix ones $M_K(\mathbb C)$, and the commutative ones $C(T)$. We are led in this way to:

\index{random matrix}
\index{matrix model}

\begin{definition}
A matrix model for $G\subset U_N^+$ is a morphism of $C^*$-algebras
$$\pi:C(G)\to M_K(C(T))$$
where $T$ is a compact space, and $K\geq1$ is an integer.
\end{definition}

And isn't this amazingly simple, as an idea. In fact, most likely, we are not here into ``advanced topics'', but rather into the basics. We could have well started the present book with chapter 1 containing random matrices instead of operator algebras, say following Anderson-Guionnet-Zeitouni \cite{agz}, and with statistical mechanics instead of quantum mechanics, as a main motivation, and then with chapter 2 containing some kind of axioms for the quantum groups, directly modelled by using random matrix algebras.

\bigskip

Of course, things are a bit more complicated than this, and we will see details in a moment. However, for the philosophy, Definition 16.1 remains something extremely simple, and bright. And motivating too. The point indeed is that the matrix models for the quantum groups make some interesting links with the work of Jones \cite{jo1}, \cite{jo2}, \cite{jo3}, on the mathematics of quantum and statistical mechanics and related topics, and this has been known since the late 90s, and has served as a main motivation for the development of the theory of closed subgroups $G\subset U_N^+$, all over the 00s and 10s. But more on this later, once we'll get more familiar with Definition 16.1, and its consequences.

\bigskip

Getting back now to Definition 16.1, more generally, we can model in this way the standard coordinates $x_i\in C(X)$ of the various algebraic manifolds $X\subset S^{N-1}_{\mathbb C,+}$. Indeed, these manifolds generalize the compact matrix quantum groups, which appear as:
$$G\subset U_N^+\subset S^{N^2-1}_{\mathbb C,+}$$

Thus, we have many other interesting examples of such manifolds, such as the homogeneous spaces discussed in chapter 15. However, at this level of generality, not much general theory is available. It is elementary to show that, under the technical assumption $X^c\neq\emptyset$, there exists a universal $K\times K$ model for the algebra $C(X)$, which factorizes as follows, with $X^{(K)}\subset X$ being a certain algebraic submanifold: 
$$\pi_K:C(X)\to C(X^{(K)})\subset M_K(C(T_K))$$

To be more precise, the universal $K\times K$ model space $T_K$ appears by imposing to the complex $K\times K$ matrices the relations defining $X$, and the algebra $C(X^{(K)})$ is then by definition the image of $\pi_K$. In relation with this, we can set as well:
$$X^{(\infty)}=\bigcup_{K\in\mathbb N}X^{(K)}$$

We are led in this way to a filtration of $X$, as follows:
$$X^c= X^{(1)}\subset X^{(2)}\subset X^{(3)}\subset\ldots\ldots\subset X^{(\infty)}\subset X$$

It is possible to say a few non-trivial things about these manifolds $X^{(K)}$, by using algebraic and functional analytic techniques, and we refer here to \cite{bb4}. 

\bigskip

In the compact quantum group case, however, that we are mainly interested in here, the matrix truncations $G^{(K)}\subset G$ are generically not subgroups at $K\geq2$, and so this theory is a priori not very useful, at least in its basic form presented here.

\bigskip

In order to reach, however, to some results, let us introduce as well:

\index{stationary model}

\begin{definition}
A matrix model $\pi:C(G)\to M_K(C(T))$ is called stationary when
$$\int_G=\left(tr\otimes\int_T\right)\pi$$
where $\int_T$ is the integration with respect to a given probability measure on $T$.
\end{definition}

Observe that this definition can be extended as well to the algebraic manifold case, $X\subset S^{N-1}_{\mathbb C,+}$, provided that our manifolds have certain integration functionals $\int_X$. This is the case for instance with the homogeneous spaces discussed in chapter 15, where $\int_X$ appears as the unique $G$-invariant trace, with respect to the underlying quantum group $G$. However, the axiomatization of such manifolds being not available yet, we will keep this as a remark, and get back in what follows, until the end, to the quantum groups.

\bigskip

So, back to Definition 16.2, as it is, our first comment concerns the terminology. The term ``stationary'' comes from a functional analytic interpretation of all this, with a certain Ces\`aro limit being needed to be stationary, and this will be explained later on. Yet another explanation comes from a certain relation with the lattice models, but this relation is rather something folklore, not axiomatized yet. More on this later.

\bigskip

As a first result now, the stationarity property implies the faithfulness:

\index{amenability}
\index{coamenability}

\begin{theorem}
Assuming that $G\subset U_N^+$ has a stationary model,
$$\pi:C(G)\to M_K(C(T))$$
$$\int_G=\left(tr\otimes\int_T\right)\pi$$
it follows that $G$ must be coamenable, and that the model is faithful.
\end{theorem}

\begin{proof}
We have two assertions to be proved, the idea being as follows:

\medskip

(1) Assume that we have a stationary model, as in the statement. By performing the GNS construction with respect to $\int_G$, we obtain a factorization as follows, which commutes with the respective canonical integration functionals:
$$\pi:C(G)\to C(G)_{red}\subset M_K(C(T))$$

Thus, in what regards the coamenability question, we can assume that $\pi$ is faithful. With this assumption made, observe that we have embeddings as follows:
$$C^\infty(G)\subset C(G)\subset M_K(C(T))$$

The point now is that the GNS construction gives a better embedding, as follows:
$$L^\infty(G)\subset M_K(L^\infty(T))$$

Now since the von Neumann algebra on the right is of type I, so must be its subalgebra $A=L^\infty(G)$. This means that, when writing the center of this latter algebra as  $Z(A)=L^\infty(X)$, the whole algebra decomposes over $X$, as an integral of type I factors:
$$L^\infty(G)=\int_XM_{K_x}(\mathbb C)\,dx$$

In particular, we can see from this that $C^\infty(G)\subset L^\infty(G)$ has a unique $C^*$-norm, and so $G$ is coamenable. Thus we have proved our first assertion.

\medskip

(2) The second assertion follows as well from the above, because our factorization of $\pi$ consists of the identity, and of an inclusion.
\end{proof}

Regarding now the examples of stationary models, we first have:

\begin{proposition}
The following have stationary models:
\begin{enumerate}
\item The compact Lie groups.

\item The finite quantum groups.
\end{enumerate}
\end{proposition}

\begin{proof}
Both these assertions are elementary, with the proofs being as follows:

\medskip

(1) This is clear, because we can use the identity $id:C(G)\to M_1(C(G))$.

\medskip

(2) Here we can use the regular representation $\lambda:C(G)\to M_{|G|}(\mathbb C)$. Indeed, let us endow the linear space $H=C(G)$ with the scalar product $<a,b>=\int_Gab^*$. We have then a representation of $*$-algebras, as follows:
$$\lambda:C(G)\to B(H)\quad,\quad 
a\to[b\to ab]$$

Now since we have $H\simeq\mathbb C^{|G|}$ with $|G|=\dim A$, we can view $\lambda$ as a matrix model map, as above, and the stationarity axiom $\int_G=tr\circ\lambda$ is satisfied, as desired. 
\end{proof}

In order to discuss now the group duals, consider a model as follows:
$$\pi:C^*(\Gamma)\to M_K(C(T))$$

Since a representation of a group algebra must come from a representation of the group, such a matrix model must come from a group representation, as follows:
$$\rho:\Gamma\to C(T,U_K)$$

With this identification made, we have the following result:

\begin{proposition}
An matrix model $\rho:\Gamma\subset C(T,U_K)$ is stationary when:
$$\int_Ttr(g^x)dx=0,\forall g\neq1$$
Moreover, the examples include all the abelian groups, and all finite groups.
\end{proposition}

\begin{proof}
Consider indeed a group embedding $\rho:\Gamma\subset C(T,U_K)$, which produces by linearity a matrix model, as follows:
$$\pi:C^*(\Gamma)\to M_K(C(T))$$

It is enough to formulate the stationarity condition on the group elements $g\in C^*(\Gamma)$. Let us set $\rho(g)=(x\to g^x)$. With this notation, the stationarity condition reads:
$$\int_Ttr(g^x)dx=\delta_{g,1}$$

Since this equality is trivially satisfied at $g=1$, where by unitality of our representation we must have $g^x=1$ for any $x\in T$, we are led to the condition in the statement. Regarding now the examples, these are both clear. More precisely:

\medskip

(1) When $\Gamma$ is abelian we can use the following trivial embedding:
$$\Gamma\subset C(\widehat{\Gamma},U_1)\quad,\quad 
g\to[\chi\to\chi(g)]$$

(2) When $\Gamma$ is finite we can use the left regular representation:
$$\Gamma\subset\mathcal L(\mathbb C\Gamma)\quad,\quad 
g\to[h\to gh]$$

Indeed, in both cases, the stationarity condition is trivially satisfied.
\end{proof}

In order to further advance, and to come up with tools for discussing the non-stationary case as well, let us keep looking at the group duals $G=\widehat{\Gamma}$.  We know that a matrix model $\pi:C^*(\Gamma)\to M_K(C(T))$ must come from a group representation, as follows:
$$\rho:\Gamma\to C(T,U_K)$$

Now observe that when $\rho$ is faithful, the representation $\pi$ is in general not faithful, for instance because when $T=\{.\}$ its target algebra is finite dimensional. On the other hand, this representation obviously ``reminds'' $\Gamma$, and so can be used in order to fully understand $\Gamma$. Thus, we have a new idea here, basically saying that, for practical purposes, the faithfuless property can be replaced with something much weaker. 

\bigskip

This weaker notion is called ``inner faithfulness''. The theory here, going back to the late 90s, and in its modern formulation, from the late 00s paper \cite{bb2}, is as follows:

\index{Hopf image}
\index{inner faithfulness}

\begin{definition}
Let $\pi:C(G)\to M_K(C(T))$ be a matrix model. 
\begin{enumerate}
\item The Hopf image of $\pi$ is the smallest quotient Hopf $C^*$-algebra $C(G)\to C(H)$ producing a factorization of type $\pi:C(G)\to C(H)\to M_K(C(T))$.

\item When the inclusion $H\subset G$ is an isomorphism, i.e. when there is no non-trivial factorization as above, we say that $\pi$ is inner faithful.
\end{enumerate}
\end{definition}

As a basic illustration for these notions, in the case where $G=\widehat{\Gamma}$ is a group dual, $\pi$ must come, as above, from a group representation, as follows:
$$\rho:\Gamma\to C(T,U_K)$$

We conclude that in this case, the minimal factorization constructed in Definition 16.6 is simply the one obtained by taking the image:
$$\rho:\Gamma\to\Lambda\subset C(T,U_K)$$

Thus $\pi$ is inner faithful when our group satisfies the following condition:
$$\Gamma\subset C(T,U_K)$$

As a second illustration now, given a compact group $G$, and elements $g_1,\ldots,g_K\in G$, we have a representation $\pi:C(G)\to\mathbb C^K$, given by the following formula:
$$f\to(f(g_1),\ldots,f(g_K))$$

The minimal factorization of $\pi$ is then via the algebra $C(H)$, with:
$$H=\overline{<g_1,\ldots,g_K>}$$

Thus $\pi$ is inner faithful precisely when our group satisfies:
$$G=H$$

In general, the existence and uniqueness of the Hopf image comes from dividing $C(G)$ by a suitable ideal, as explained in \cite{bb2}. In Tannakian terms, we have:

\index{tensor category}
\index{Tannakian category}

\begin{theorem}
Consider a closed subgroup $G\subset U_N^+$, with fundamental corepresentation denoted $u=(u_{ij})$. The Hopf image of a matrix model
$$\pi:C(G)\to M_K(C(T))$$
comes then from the Tannakian category
$$C_{kl}=Hom(U^{\otimes k},U^{\otimes l})$$
where $U_{ij}=\pi(u_{ij})$, and where the spaces on the right are taken in a formal sense.
\end{theorem}

\begin{proof}
Since the morphisms increase the intertwining spaces, when defined either in a representation theory sense, or just formally, we have inclusions as follows:
$$Hom(u^{\otimes k},u^{\otimes l})\subset Hom(U^{\otimes k},U^{\otimes l})$$

More generally, we have such inclusions when replacing $(G,u)$ with any pair producing a factorization of $\pi$. Thus, by Tannakian duality, the Hopf image must be given by the fact that the intertwining spaces must be the biggest, subject to the above inclusions.

On the other hand, since $u$ is biunitary, so is $U$, and it follows that the spaces on the right form a Tannakian category. Thus, we have a quantum group $(H,v)$ given by:
$$Hom(v^{\otimes k},v^{\otimes l})=Hom(U^{\otimes k},U^{\otimes l})$$

By the above discussion, $C(H)$ follows to be the Hopf image of $\pi$, as claimed.
\end{proof}

The inner faithful models $\pi:C(G)\to M_K(C(T))$ are a very interesting notion, because they are not subject to the coamenability condition on $G$, as it was the case with the stationary models, as explained in Theorem 16.3. In fact, there are no known restrictions on the class of subgroups $G\subset U_N^+$ which can be modelled in an inner faithful way. Thus, our modelling theory applies a priori to any compact quantum group. 

\bigskip

Regarding now the study of the inner faithful models, a key problem is that of computing the Haar functional. The result here, from Wang \cite{wa3}, is as follows:

\index{Haar integration}

\begin{theorem}
Given an inner faithful model $\pi:C(G)\to M_K(C(T))$, we have
$$\int_G=\lim_{k\to\infty}\frac{1}{k}\sum_{r=1}^k\int_G^r$$
where $\int_G^r=(\varphi\circ\pi)^{*r}$, with $\varphi=tr\otimes\int_T$ being the random matrix trace.
\end{theorem}

\begin{proof}
As a first observation, there is an obvious similarity here with the Woronowicz construction of the Haar measure from \cite{wo1}, explained in chapter 3. In fact, the above result holds more generally for any model $\pi:C(G)\to B$, with $\varphi\in B^*$ being a faithful trace. With this picture in hand, the Woronowicz construction simply corresponds to the case $\pi=id$, and the result itself is therefore a generalization of Woronowicz's result. 

\medskip

In order to prove now the result, we can proceed as in chapter 3. If we denote by $\int_G'$ the limit in the statement, we must prove that this limit converges, and that we have:
$$\int_G'=\int_G$$

It is enough to check this on the coefficients of corepresentations, and if we let $v=u^{\otimes k}$ be one of the Peter-Weyl corepresentations, we must prove that we have:
$$\left(id\otimes\int_G'\right)v=\left(id\otimes\int_G\right)v$$

We know from chapter 3 that the matrix on the right is the orthogonal projection onto $Fix(v)$. Regarding now the matrix on the left, this is the orthogonal projection onto the $1$-eigenspace of $(id\otimes\varphi\pi)v$. Now observe that, if we set $V_{ij}=\pi(v_{ij})$, we have:
$$(id\otimes\varphi\pi)v=(id\otimes\varphi)V$$

Thus, exactly as in chapter 3, we conclude that the $1$-eigenspace that we are interested in equals $Fix(V)$. But, according to Theorem 16.7, we have:
$$Fix(V)=Fix(v)$$

Thus, we have proved that we have $\int_G'=\int_G$, as desired.
\end{proof}

\section*{16b. Stationarity}

Before getting into more about inner faithfulness, let us first go back to the stationary models. These models are quite restrictive, because $G$ must be coamenable. However, there are many interesting examples of coamenable compact quantum groups, and in order to better understand these examples, and also in order to construct some new examples, our idea will be that of looking for stationary models for them. We first have:

\index{stationary on its image}

\begin{theorem}
For $\pi:C(G)\to M_K(C(T))$, the following are equivalent:
\begin{enumerate}
\item $Im(\pi)$ is a Hopf algebra, and $(tr\otimes\int_T)\pi$ is the Haar integration on it.

\item $\psi=(tr\otimes\int_X)\pi$ satisfies the idempotent state property $\psi*\psi=\psi$.

\item $T_e^2=T_e$, $\forall p\in\mathbb N$, $\forall e\in\{1,*\}^p$, where:
$$(T_e)_{i_1\ldots i_p,j_1\ldots j_p}=\left(tr\otimes\int_T\right)(U_{i_1j_1}^{e_1}\ldots U_{i_pj_p}^{e_p})$$
\end{enumerate}
If these conditions are satisfied, we say that $\pi$ is stationary on its image.
\end{theorem}

\begin{proof}
Given a matrix model $\pi:C(G)\to M_K(C(T))$ as in the statement, we can factorize it via its Hopf image, as in Definition 16.6 above:
$$\pi:C(G)\to C(H)\to M_K(C(T))$$

Now observe that the conditions (1,2,3) in the statement depend only on the factorized representation:
$$\nu:C(H)\to M_K(C(T))$$

Thus, we can assume in practice that we have $G=H$, which means that we can assume that $\pi$ is inner faithful. With this assumption made, the general integration formula from Theorem 16.8 applies to our situation, and the proof of the equivalences goes as follows:

\medskip

$(1)\implies(2)$ This is clear from definitions, because the Haar integration on any compact quantum group satisfies the idempotent state equation, namely:
$$\psi*\psi=\psi$$

$(2)\implies(1)$ Assuming $\psi*\psi=\psi$, we have, for any $r\in\mathbb N$:
$$\psi^{*r}=\psi$$

Thus Theorem 16.8 gives $\int_G=\psi$, and by using Theorem 16.3, we obtain the result.

\medskip

In order to establish now $(2)\Longleftrightarrow(3)$, we use the following elementary formula, which comes from the definition of the convolution operation:
$$\psi^{*r}(u_{i_1j_1}^{e_1}\ldots u_{i_pj_p}^{e_p})=(T_e^r)_{i_1\ldots i_p,j_1\ldots j_p}$$

$(2)\implies(3)$ Assuming $\psi*\psi=\psi$, by using the above formula at $r=1,2$ we obtain that the matrices $T_e$ and $T_e^2$ have the same coefficients, and so they are equal.

\medskip

$(3)\implies(2)$ Assuming $T_e^2=T_e$, by using the above formula at $r=1,2$ we obtain that the linear forms $\psi$ and $\psi*\psi$ coincide on any product of coefficients $u_{i_1j_1}^{e_1}\ldots u_{i_pj_p}^{e_p}$. Now since these coefficients span a dense subalgebra of $C(G)$, this gives the result.
\end{proof}

As a first illustration, we will apply this criterion to certain models for the quantum groups $O_N^*,U_N^*$. We first have the following result:

\index{half-liberation}

\begin{proposition}
We have a matrix model as follows, 
$$C(O_N^*)\to M_2(C(U_N))\quad,\quad 
u_{ij}\to\begin{pmatrix}0&v_{ij}\\ \bar{v}_{ij}&0\end{pmatrix}$$
where $v$ is the fundamental corepresentation of $C(U_N)$, as well as a model as follows,
$$C(U_N^*)\to M_2(C(U_N\times U_N))\quad,\quad 
u_{ij}\to\begin{pmatrix}0&v_{ij}\\ w_{ij}&0\end{pmatrix}$$
where $v,w$ are the fundamental corepresentations of the two copies of $C(U_N)$.
\end{proposition}

\begin{proof}
It is routine to check that the matrices on the right are indeed biunitaries, and since the first matrix is also self-adjoint, we obtain models as follows:
$$C(O_N^+)\to M_2(C(U_N))$$
$$C(U_N^+)\to M_2(C(U_N\times U_N))$$

Consider now antidiagonal $2\times2$ matrices, with commuting entries, as follows:
$$X_i=\begin{pmatrix}0&x_i\\ y_i&0\end{pmatrix}$$

We have then the following computation:
\begin{eqnarray*}
X_iX_jX_k
&=&\begin{pmatrix}0&x_i\\ y_i&0\end{pmatrix}
\begin{pmatrix}0&x_j\\ y_j&0\end{pmatrix}
\begin{pmatrix}0&x_k\\ y_k&0\end{pmatrix}\\
&=&\begin{pmatrix}0&x_iy_jx_k\\ y_ix_jy_k&0\end{pmatrix}
\end{eqnarray*}

Since this quantity is symmetric in $i,k$, we obtain from this:
$$X_iX_jX_k=X_kX_jX_i$$

Thus, our models above factorize as claimed.
\end{proof}

We can now formulate our first concrete modelling theorem, as folllows:

\begin{theorem}
The above antidiagonal models, namely
$$C(O_N^*)\to M_2(C(U_N))$$
$$C(U_N^*)\to M_2(C(U_N\times U_N))$$
are both stationary.
\end{theorem}

\begin{proof}
We first discuss the case of $O_N^*$. We use Theorem 16.9 (3). Since the fundamental representation is self-adjoint, the matrices $T_e$ with $e\in\{1,*\}^p$ are all equal. We denote this common matrix by $T_p$. According to the definition of $T_p$, we have:
$$(T_p)_{i_1\ldots i_p,j_1\ldots j_p}
=\left(tr\otimes\int_H\right)\left[\begin{pmatrix}0&v_{i_1j_1}\\\bar{v}_{i_1j_1}&0\end{pmatrix}\ldots\ldots\begin{pmatrix}0&v_{i_pj_p}\\\bar{v}_{i_pj_p}&0\end{pmatrix}\right]$$

Since when multipliying an odd number of antidiagonal matrices we obtain an atidiagonal matrix, we have $T_p=0$ for $p$ odd. Also, when $p$ is even, we have:
\begin{eqnarray*}
(T_p)_{i_1\ldots i_p,j_1\ldots j_p}
&=&\left(tr\otimes\int_H\right)\begin{pmatrix}v_{i_1j_1}\ldots\bar{v}_{i_pj_p}&0\\0&\bar{v}_{i_1j_1}\ldots v_{i_pj_p}\end{pmatrix}\\
&=&\frac{1}{2}\left(\int_Hv_{i_1j_1}\ldots\bar{v}_{i_pj_p}+\int_H\bar{v}_{i_1j_1}\ldots v_{i_pj_p}\right)\\
&=&\int_HRe(v_{i_1j_1}\ldots\bar{v}_{i_pj_p})
\end{eqnarray*}

We have $T_p^2=T_p=0$ when $p$ is odd, so we are left with proving that we have $T_p^2=T_p$, when $p$ is even. For this purpose, we use the following formula:
$$Re(x)Re(y)=\frac{1}{2}\left(Re(xy)+Re(x\bar{y})\right)$$

By using this identity for each of the terms which appear in the product, and multi-index notations in order to simplify the writing, we obtain:
\begin{eqnarray*}
&&(T_p^2)_{ij}\\
&=&\sum_{k_1\ldots k_p}(T_p)_{i_1\ldots i_p,k_1\ldots k_p}(T_p)_{k_1\ldots k_p,j_1\ldots j_p}\\
&=&\int_H\int_H\sum_{k_1\ldots k_p}Re(v_{i_1k_1}\ldots\bar{v}_{i_pk_p})Re(w_{k_1j_1}\ldots\bar{w}_{k_pj_p})dvdw\\
&=&\frac{1}{2}\int_H\int_H\sum_{k_1\ldots k_p}Re(v_{i_1k_1}w_{k_1j_1}\ldots\bar{v}_{i_pk_p}\bar{w}_{k_pj_p})+Re(v_{i_1k_1}\bar{w}_{k_1j_1}\ldots\bar{v}_{i_pk_p}w_{k_pj_p})dvdw\\
&=&\frac{1}{2}\int_H\int_HRe((vw)_{i_1j_1}\ldots(\bar{v}\bar{w})_{i_pj_p})+Re((v\bar{w})_{i_1j_1}\ldots(\bar{v}w)_{i_pj_p})dvdw
\end{eqnarray*}

Now since $vw\in H$ is uniformly distributed when $v,w\in H$ are uniformly distributed, the quantity on the left integrates up to $(T_p)_{ij}$. Also, since $H$ is conjugation-stable, $\bar{w}\in H$ is uniformly distributed when $w\in H$ is uniformly distributed, so the quantity on the right integrates up to the same quantity, namely $(T_p)_{ij}$. Thus, we have:
\begin{eqnarray*}
(T_p^2)_{ij}
&=&\frac{1}{2}\Big((T_p)_{ij}+(T_p)_{ij}\Big)\\
&=&(T_p)_{ij}
\end{eqnarray*}

Summarizing, we have obtained that for any $p$, the condition $T_p^2=T_p$ is satisfied. Thus Theorem 16.9 applies, and shows that our model is stationary, as claimed. As for the proof of the stationarity for the model for $U_N^*$, this is similar.
\end{proof}

As a second illustration, regarding $H_N^*,K_N^*$, we have:

\begin{theorem}
We have a stationary matrix model as follows, 
$$C(H_N^*)\to M_2(C(K_N))\quad,\quad 
u_{ij}\to\begin{pmatrix}0&v_{ij}\\ \bar{v}_{ij}&0\end{pmatrix}$$
where $v$ is the fundamental corepresentation of $C(K_N)$, as well as a stationary model
$$C(K_N^*)\to M_2(C(K_N\times K_N))\quad,\quad 
u_{ij}\to\begin{pmatrix}0&v_{ij}\\ w_{ij}&0\end{pmatrix}$$
where $v,w$ are the fundamental corepresentations of the two copies of $C(K_N)$.
\end{theorem}

\begin{proof}
This follows by adapting the proof of Proposition 16.10 and Theorem 16.11 above, by adding there the $H_N^+,K_N^+$ relations. All this is in fact part of a more general phenomenon, concerning half-liberation in general, and we refer here to \cite{bb4}, \cite{bdu}.
\end{proof}

Summarizing, we have some interesting theory and examples for both the stationary models, and for the general inner faithful models.

\section*{16c. Weyl matrices}

\index{Pauli matrix}
\index{Weyl matrix}

Following \cite{bne}, let us discuss now some more subtle examples of stationary models, related to the Pauli matrices, and Weyl matrices, and physics. We first have:

\begin{definition}
Given a finite abelian group $H$, the associated Weyl matrices are
$$W_{ia}:e_b\to<i,b>e_{a+b}$$
where $i\in H$, $a,b\in\widehat{H}$, and where $(i,b)\to<i,b>$ is the Fourier coupling $H\times\widehat{H}\to\mathbb T$.
\end{definition}

As a basic example, consider the simplest cyclic group, namely:
$$H=\mathbb Z_2=\{0,1\}$$

Here the Fourier coupling is $<i,b>=(-1)^{ib}$, and the Weyl matrices act as follows:
$$W_{00}:e_b\to e_b\qquad,\qquad
W_{10}:e_b\to(-1)^be_b$$
$$W_{11}:e_b\to(-1)^be_{b+1}\qquad,\qquad 
W_{01}:e_b\to e_{b+1}$$

Thus, we have the following formulae for the Weyl matrices:
$$W_{00}=\begin{pmatrix}1&0\\0&1\end{pmatrix}\quad,\quad
W_{10}=\begin{pmatrix}1&0\\0&-1\end{pmatrix}$$
$$W_{11}=\begin{pmatrix}0&-1\\1&0\end{pmatrix}\quad,\quad  
W_{01}=\begin{pmatrix}0&1\\1&0\end{pmatrix}$$

We recognize here, up to some multiplicative factors, the four Pauli matrices. Now back to the general case, we have the following well-known result:

\begin{proposition}
The Weyl matrices are unitaries, and satisfy:
\begin{enumerate}
\item $W_{ia}^*=<i,a>W_{-i,-a}$.

\item $W_{ia}W_{jb}=<i,b>W_{i+j,a+b}$.

\item $W_{ia}W_{jb}^*=<j-i,b>W_{i-j,a-b}$.

\item $W_{ia}^*W_{jb}=<i,a-b>W_{j-i,b-a}$.
\end{enumerate}
\end{proposition}

\begin{proof}
The unitary follows from (3,4), and the rest of the proof goes as follows:

\medskip

(1) We have indeed the following computation:
\begin{eqnarray*}
W_{ia}^*
&=&\left(\sum_b<i,b>E_{a+b,b}\right)^*\\
&=&\sum_b<-i,b>E_{b,a+b}\\
&=&\sum_b<-i,b-a>E_{b-a,b}\\
&=&<i,a>W_{-i,-a}
\end{eqnarray*}

(2) Here the verification goes as follows:
\begin{eqnarray*}
W_{ia}W_{jb}
&=&\left(\sum_d<i,b+d>E_{a+b+d,b+d}\right)\left(\sum_d<j,d>E_{b+d,d}\right)\\
&=&\sum_d<i,b><i+j,d>E_{a+b+d,d}\\
&=&<i,b>W_{i+j,a+b}
\end{eqnarray*}

(3,4) By combining the above two formulae, we obtain:
\begin{eqnarray*}
W_{ia}W_{jb}^*
&=&<j,b>W_{ia}W_{-j,-b}\\
&=&<j,b><i,-b>W_{i-j,a-b}
\end{eqnarray*}

We obtain as well the following formula:
\begin{eqnarray*}
W_{ia}^*W_{jb}
&=&<i,a>W_{-i,-a}W_{jb}\\
&=&<i,a><-i,b>W_{j-i,b-a}
\end{eqnarray*}

But this gives the formulae in the statement, and we are done.
\end{proof}

Observe that, with $n=|H|$, we can use an isomorphism $l^2(\widehat{H})\simeq\mathbb C^n$ as to view each $W_{ia}$ as a usual matrix, $W_{ia}\in M_n(\mathbb C)$, and hence as a usual unitary, $W_{ia}\in U_n$. 

\bigskip

Given a vector $\xi$, we denote by $Proj(\xi)$ the orthogonal projection onto $\mathbb C\xi$. We have:

\begin{proposition}
Given a closed subgroup $E\subset U_n$, we have a representation
$$\pi_H:C(S_N^+)\to M_N(C(E))$$
$$w_{ia,jb}\to[U\to Proj(W_{ia}UW_{jb}^*)]$$
where $n=|H|,N=n^2$, and where $W_{ia}$ are the Weyl matrices associated to $H$.
\end{proposition}

\begin{proof}
The Weyl matrices being given by $W_{ia}:e_b\to<i,b>e_{a+b}$, we have:
$$tr(W_{ia})=\begin{cases}
1&{\rm if}\ (i,a)=(0,0)\\
0&{\rm if}\ (i,a)\neq(0,0)
\end{cases}$$

Together with the formulae in Proposition 16.14, this shows that the Weyl matrices are pairwise orthogonal with respect to the following scalar product on $M_n(\mathbb C)$:
$$<x,y>=tr(x^*y)$$

Thus, these matrices form an orthogonal basis of $M_n(\mathbb C)$, consisting of unitaries:
$$W=\left\{W_{ia}\Big|i\in H,a\in\widehat{H}\right\}$$

Thus, each row and each column of the matrix $\xi_{ia,jb}=W_{ia}UW_{jb}^*$ is an orthogonal basis of $M_n(\mathbb C)$, and so the corresponding projections form a magic unitary, as claimed.
\end{proof}

We will need the following well-known result:

\begin{proposition}
With $T=Proj(x_1)\ldots Proj(x_p)$ and $||x_i||=1$ we have 
$$<T\xi,\eta>=<\xi,x_p><x_p,x_{p-1}>\ldots<x_2,x_1><x_1,\eta>$$
for any $\xi,\eta$. In particular, we have:
$$Tr(T)=<x_1,x_p><x_p,x_{p-1}>\ldots<x_2,x_1>$$
\end{proposition}

\begin{proof}
For $||x||=1$ we have $Proj(x)\xi=<\xi,x>x$. This gives:
\begin{eqnarray*}
T\xi
&=&Proj(x_1)\ldots Proj(x_p)\xi\\
&=&Proj(x_1)\ldots Proj(x_{p-1})<\xi,x_p>x_p\\
&=&Proj(x_1)\ldots Proj(x_{p-2})<\xi,x_p><x_p,x_{p-1}>x_{p-1}\\
&=&\ldots\\
&=&<\xi,x_p><x_p,x_{p-1}>\ldots<x_2,x_1>x_1
\end{eqnarray*}

Now by taking the scalar product with $\eta$, this gives the first assertion. As for the second assertion, this follows from the first assertion, by summing over $\xi=\eta=e_i$.
\end{proof}

Now back to the Weyl matrix models, let us first compute $T_p$. We have:

\begin{proposition}
We have the formula
\begin{eqnarray*}
&&(T_p)_{ia,jb}\\
&=&\frac{1}{N}<i_1,a_1-a_p>\ldots<i_p,a_p-a_{p-1}><j_1,b_1-b_2>\ldots<j_p,b_p-b_1>\\
&&\int_Etr(W_{i_1-i_2,a_1-a_2}UW_{j_2-j_1,b_2-b_1}U^*)\ldots tr(W_{i_p-i_1,a_p-a_1}UW_{j_1-j_p,b_1-b_p}U^*)dU
\end{eqnarray*}
with all the indices varying in a cyclic way.
\end{proposition}

\begin{proof}
By using the trace formula in Proposition 16.16 above, we obtain:
\begin{eqnarray*}
&&(T_p)_{ia,jb}\\
&=&\left(tr\otimes\int_E\right)\left(Proj(W_{i_1a_1}UW_{j_1b_1}^*)\ldots Proj(W_{i_pa_p}UW_{j_pb_p}^*)\right)\\
&=&\frac{1}{N}\int_E<W_{i_1a_1}UW_{j_1b_1}^*,W_{i_pa_p}UW_{j_pb_p}^*>\ldots<W_{i_2a_2}UW_{j_2b_2}^*,W_{i_1a_1}UW_{j_1b_1}^*>dU
\end{eqnarray*}

In order to compute now the scalar products, observe that we have:
\begin{eqnarray*}
<W_{ia}UW_{jb}^*,W_{kc}UW_{ld}^*>
&=&tr(W_{jb}U^*W_{ia}^*W_{kc}UW_{ld}^*)\\
&=&tr(W_{ia}^*W_{kc}UW_{ld}^*W_{jb}U^*)\\
&=&<i,a-c><l,d-b>tr(W_{k-i,c-a}UW_{j-l,b-d}U^*)
\end{eqnarray*}

By plugging these quantities into the formula of $T_p$, we obtain the result.
\end{proof}

Consider now the Weyl group $W=\{W_{ia}\}\subset U_n$, that we already met in the proof of Proposition 16.15 above. We have the following result, from \cite{bne}:

\index{stationary on its image}

\begin{theorem}
For any compact group $W\subset E\subset U_n$, the model
$$\pi_H:C(S_N^+)\to M_N(C(E))$$
$$w_{ia,jb}\to[U\to Proj(W_{ia}UW_{jb}^*)]$$
constructed above is stationary on its image.
\end{theorem}

\begin{proof}
We must prove that we have $T_p^2=T_p$. We have:
\begin{eqnarray*}
&&(T_p^2)_{ia,jb}\\
&=&\sum_{kc}(T_p)_{ia,kc}(T_p)_{kc,jb}\\
&=&\frac{1}{N^2}\sum_{kc}<i_1,a_1-a_p>\ldots<i_p,a_p-a_{p-1}><k_1,c_1-c_2>\ldots<k_p,c_p-c_1>\\
&&<k_1,c_1-c_p>\ldots<k_p,c_p-c_{p-1}><j_1,b_1-b_2>\ldots<j_p,b_p-b_1>\\
&&\int_Etr(W_{i_1-i_2,a_1-a_2}UW_{k_2-k_1,c_2-c_1}U^*)\ldots tr(W_{i_p-i_1,a_p-a_1}UW_{k_1-k_p,c_1-c_p}U^*)dU\\
&&\int_Etr(W_{k_1-k_2,c_1-c_2}VW_{j_2-j_1,b_2-b_1}V^*)\ldots tr(W_{k_p-k_1,c_p-c_1}VW_{j_1-j_p,b_1-b_p}V^*)dV
\end{eqnarray*}

By rearranging the terms, this formula becomes:
\begin{eqnarray*}
&&(T_p^2)_{ia,jb}\\
&=&\frac{1}{N^2}<i_1,a_1-a_p>\ldots<i_p,a_p-a_{p-1}><j_1,b_1-b_2>\ldots<j_p,b_p-b_1>\\
&&\int_E\int_E\sum_{kc}<k_1-k_p,c_1-c_p>\ldots<k_p-k_{p-1},c_p-c_{p-1}>\\
&&tr(W_{i_1-i_2,a_1-a_2}UW_{k_2-k_1,c_2-c_1}U^*)tr(W_{k_1-k_2,c_1-c_2}VW_{j_2-j_1,b_2-b_1}V^*)\\
&&\hskip50mm\ldots\ldots\\
&&tr(W_{i_p-i_1,a_p-a_1}UW_{k_1-k_p,c_1-c_p}U^*)tr(W_{k_p-k_1,c_p-c_1}VW_{j_1-j_p,b_1-b_p}V^*)dUdV
\end{eqnarray*}

Let us denote by $I$ the above double integral. By using $W_{kc}^*=<k,c>W_{-k,-c}$ for each of the couplings, and by moving as well all the $U^*$ variables to the left, we obtain:
\begin{eqnarray*}
I
&=&\int_E\int_E\sum_{kc}tr(U^*W_{i_1-i_2,a_1-a_2}UW_{k_2-k_1,c_2-c_1})tr(W_{k_2-k_1,c_2-c_1}^*VW_{j_2-j_1,b_2-b_1}V^*)\\
&&\hskip50mm\ldots\ldots\\
&&tr(U^*W_{i_p-i_1,a_p-a_1}UW_{k_1-k_p,c_1-c_p})tr(W_{k_1-k_p,c_1-c_p}^*VW_{j_1-j_p,b_1-b_p}V^*)dUdV
\end{eqnarray*}

In order to perform now the sums, we use the following formula:
\begin{eqnarray*}
tr(AW_{kc})tr(W_{kc}^*B)
&=&\frac{1}{N}\sum_{qrst}A_{qr}(W_{kc})_{rq}(W^*_{kc})_{st}B_{ts}\\
&=&\frac{1}{N}\sum_{qrst}A_{qr}<k,q>\delta_{r-q,c}<k,-s>\delta_{t-s,c}B_{ts}\\
&=&\frac{1}{N}\sum_{qs}<k,q-s>A_{q,q+c}B_{s+c,s}
\end{eqnarray*}

If we denote by $A_x,B_x$ the variables which appear in the formula of $I$, we have:
\begin{eqnarray*}
&&I\\
&=&\frac{1}{N^p}\int_E\int_E\sum_{kcqs}<k_2-k_1,q_1-s_1>\ldots<k_1-k_p,q_p-s_p>\\
&&(A_1)_{q_1,q_1+c_2-c_1}(B_1)_{s_1+c_2-c_1,s_1}\ldots (A_p)_{q_p,q_p+c_1-c_p}(B_p)_{s_p+c_1-c_p,s_p}\\
&=&\frac{1}{N^p}\int_E\int_E\sum_{kcqs}<k_1,q_p-s_p-q_1+s_1>\ldots<k_p,q_{p-1}-s_{p-1}-q_p+s_p>\\
&&(A_1)_{q_1,q_1+c_2-c_1}(B_1)_{s_1+c_2-c_1,s_1}\ldots (A_p)_{q_p,q_p+c_1-c_p}(B_p)_{s_p+c_1-c_p,s_p}
\end{eqnarray*}

Now observe that we can perform the sums over $k_1,\ldots,k_p$. We obtain in this way a multiplicative factor $n^p$, along with the condition:
$$q_1-s_1=\ldots=q_p-s_p$$

Thus we must have $q_x=s_x+a$ for a certain $a$, and the above formula becomes:
$$I=\frac{1}{n^p}\int_E\int_E\sum_{csa}(A_1)_{s_1+a,s_1+c_2-c_1+a}(B_1)_{s_1+c_2-c_1,s_1}\ldots(A_p)_{s_p+a,s_p+c_1-c_p+a}(B_p)_{s_p+c_1-c_p,s_p}$$

Consider now the variables $r_x=c_{x+1}-c_x$, which altogether range over the set $Z$ of multi-indices having sum 0. By replacing the sum over $c_x$ with the sum over $r_x$, which creates a multiplicative $n$ factor, we obtain the following formula:
$$I=\frac{1}{n^{p-1}}\int_E\int_E\sum_{r\in Z}\sum_{sa}(A_1)_{s_1+a,s_1+r_1+a}(B_1)_{s_1+r_1,s_1}\ldots(A_p)_{s_p+a,s_p+r_p+a}(B_p)_{s_p+r_p,s_p}$$

For an arbitrary multi-index $r$ we have:
$$\delta_{\sum_ir_i,0}=\frac{1}{n}\sum_i<i,r_1>\ldots<i,r_p>$$

Thus, we can replace the sum over $r\in Z$ by a full sum, as follows:
\begin{eqnarray*}
I
&=&\frac{1}{n^p}\int_E\int_E\sum_{rsia}<i,r_1>(A_1)_{s_1+a,s_1+r_1+a}(B_1)_{s_1+r_1,s_1}\\
&&\hskip40mm\ldots\ldots\\
&&\hskip20mm<i,r_p>(A_p)_{s_p+a,s_p+r_p+a}(B_p)_{s_p+r_p,s_p}
\end{eqnarray*}

In order to ``absorb'' now the indices $i,a$, we can use the following formula:
\begin{eqnarray*}
&&W_{ia}^*AW_{ia}\\
&=&\left(\sum_b<i,-b>E_{b,a+b}\right)\left(\sum_{bc}E_{a+b,a+c}A_{a+b,a+c}\right)\left(\sum_c<i,c>E_{a+c,c}\right)\\
&=&\sum_{bc}<i,c-b>E_{bc}A_{a+b,a+c}
\end{eqnarray*}

Thus we have:
$$(W_{ia}^*AW_{ia})_{bc}=<i,c-b>A_{a+b,a+c}$$

Our formula becomes:
\begin{eqnarray*}
&&I\\
&=&\frac{1}{n^p}\int_E\int_E\sum_{rsia}(W_{ia}^*A_1W_{ia})_{s_1,s_1+r_1}(B_1)_{s_1+r_1,s_1}\ldots(W_{ia}^*A_pW_{ia})_{s_p,s_p+r_p}(B_p)_{s_p+r_p,s_p}\\
&=&\int_E\int_E\sum_{ia}tr(W_{ia}^*A_1W_{ia}B_1)\ldots\ldots tr(W_{ia}^*A_pW_{ia}B_p)
\end{eqnarray*}

Now by replacing $A_x,B_x$ with their respective values, we obtain:
\begin{eqnarray*}
I
&=&\int_E\int_E\sum_{ia}tr(W_{ia}^*U^*W_{i_1-i_2,a_1-a_2}UW_{ia}VW_{j_2-j_1,b_2-b_1}V^*)\\
&&\hskip30mm\ldots\ldots\\
&&tr(W_{ia}^*U^*W_{i_p-i_1,a_p-a_1}UW_{ia}VW_{j_1-j_p,b_1-b_p}V^*)dUdV
\end{eqnarray*}

By moving the $W_{ia}^*U^*$ variables at right, we obtain, with $S_{ia}=UW_{ia}V$:
\begin{eqnarray*}
I
&=&\sum_{ia}\int_E\int_Etr(W_{i_1-i_2,a_1-a_2}S_{ia}W_{j_2-j_1,b_2-b_1}S_{ia}^*)\\
&&\hskip30mm\ldots\ldots\\
&&tr(W_{i_p-i_1,a_p-a_1}S_{ia}W_{j_1-j_p,b_1-b_p}S_{ia}^*)dUdV
\end{eqnarray*}

Now since $S_{ia}$ is Haar distributed when $U,V$ are Haar distributed, we obtain:
$$I=N\int_E\int_Etr(W_{i_1-i_2,a_1-a_2}UW_{j_2-j_1,b_2-b_1}U^*)\ldots tr(W_{i_p-i_1,a_p-a_1}UW_{j_1-j_p,b_1-b_p}U^*)dU$$

But this is exactly $N$ times the integral in the formula of $(T_p)_{ia,jb}$, from Proposition 16.17 above. Since the $N$ factor cancels with one of the two $N$ factors that we found in the beginning of the proof, when first computing $(T_p^2)_{ia,jb}$, we are done.
\end{proof}

The above computation was of course quite tricky, and there are several possible generalizations of it, and some open questions as well, of both algebraic and analytic nature, all quite interesting. We refer to \cite{bne} are related papers for more on all this.

\bigskip

As an illustration for the above result, which is something known for a long time, and quite fundamental, going back to the paper \cite{bc3}, we have:

\index{Pauli matrix}
\index{Pauli model}

\begin{theorem}
We have a stationary matrix model
$$\pi:C(S_4^+)\subset M_4(C(SU_2))$$
given on the standard coordinates by the formula
$$\pi(u_{ij})=[x\to Proj(c_ixc_j)]$$
where $x\in SU_2$, and $c_1,c_2,c_3,c_4$ are the Pauli matrices.
\end{theorem}

\begin{proof}
As already explained in the comments following Definition 16.13, the Pauli matrices appear as particular cases of the Weyl matrices. To be more precise, consider the group $H=\mathbb Z_2=\{0,1\}$, with standard Fourier coupling, as follows: 
$$<i,b>=(-1)^{ib}$$

The Weyl matrices, as defined in the above, act then as follows:
$$W_{00}:e_b\to e_b\qquad,\qquad
W_{10}:e_b\to(-1)^be_b$$
$$W_{11}:e_b\to(-1)^be_{b+1}\qquad,\qquad 
W_{01}:e_b\to e_{b+1}$$

Thus, we have the following formulae for the Weyl matrices:
$$W_{00}=\begin{pmatrix}1&0\\0&1\end{pmatrix}\quad,\quad
W_{10}=\begin{pmatrix}1&0\\0&-1\end{pmatrix}$$
$$W_{11}=\begin{pmatrix}0&-1\\1&0\end{pmatrix}\quad,\quad  
W_{01}=\begin{pmatrix}0&1\\1&0\end{pmatrix}$$

We recognize here, up to some multiplicative factors, the four Pauli matrices. By working out now the details of the various constructions above, we conclude that Theorem 16.18 produces in this case the model in the statement.
\end{proof}

Observe that, since the matrix $Proj(c_ixc_j)$ depends only on the image of $x$ in the quotient group $SU_2\to SO_3$, we can replace the model space $SU_2$ by the smaller space $SO_3$, if we want to, and so we have a matrix model as follows:
$$\pi:C(S_4^+)\subset M_4(C(SO_3))$$

This is something that can be used in conjunction with the isomorphism $S_4^+\simeq SO_3^{-1}$ from chapter 9 above, and as explained in \cite{bb1}, our model becomes in this way something quite conceptual, algebrically speaking, as follows:
$$\pi:C(SO_3^{-1})\subset M_4(C(SO_3))$$

In general, going beyond stationarity is a difficult task, and among the results here, let us mention the universal modelling questions for quantum permutations and quantum reflections, and various results on the flat models for the discrete groups, from \cite{bne}, \cite{bc3} and related papers, questions regarding the Hadamard matrix models \cite{bb3}, and the related fine analytic study on the compact and discrete quantum groups \cite{bcf}, \cite{vve}.

\section*{16d. Fourier models}

\index{Hadamard matrix}

In what follows we discuss the Hadamard models, which are of particular importance. Let us start with the following well-known definition:

\begin{definition}
A complex Hadamard matrix is a square matrix 
$$H\in M_N(\mathbb C)$$
whose entries are on the unit circle, and whose rows are pairwise orthogonal.
\end{definition}

Observe that the orthogonality condition tells us that the rescaled matrix $U=H/\sqrt{N}$ must be unitary. Thus, these matrices form a real algebraic manifold, given by:
$$X_N=M_N(\mathbb T)\cap\sqrt{N}U_N$$

The basic example is the Fourier matrix, $F_N=(w^{ij})$ with $w=e^{2\pi i/N}$. In standard matrix form, and with indices $i,j=0,1,\ldots,N-1$, this matrix is as follows:
$$F_N=\begin{pmatrix}
1&1&1&\ldots&1\\
1&w&w^2&\ldots&w^{N-1}\\
1&w^2&w^4&\ldots&w^{2(N-1)}\\
\vdots&\vdots&\vdots&&\vdots\\
1&w^{N-1}&w^{2(N-1)}&\ldots&w^{(N-1)^2}
\end{pmatrix}$$

\index{Fourier matrix}

More generally, we have as example the Fourier coupling  of any finite abelian group $G$, regarded via the isomorphism $G\simeq\widehat{G}$ as a square matrix, $F_G\in M_G(\mathbb C)$: 
$$F_G=<i,j>_{i\in G,j\in\widehat{G}}$$

Observe that for the cyclic group $G=\mathbb Z_N$ we obtain in this way the above standard Fourier matrix $F_N$. In general, we obtain a tensor product of Fourier matrices $F_N$.

\bigskip

To be more precise here, we have the following result:

\begin{theorem}
Given a finite abelian group $G$, with dual group $\widehat{G}=\{\chi:G\to\mathbb T\}$, consider the Fourier coupling $\mathcal F_G:G\times\widehat{G}\to\mathbb T$, given by $(i,\chi)\to\chi(i)$.
\begin{enumerate}
\item Via the standard isomorphism $G\simeq\widehat{G}$, this Fourier coupling can be regarded as a square matrix, $F_G\in M_G(\mathbb T)$, which is a complex Hadamard matrix.

\item In the case of the cyclic group $G=\mathbb Z_N$ we obtain in this way, via the standard identification $\mathbb Z_N=\{1,\ldots,N\}$, the Fourier matrix $F_N$.

\item In general, when using a decomposition $G=\mathbb Z_{N_1}\times\ldots\times\mathbb Z_{N_k}$, the corresponding Fourier matrix is given by $F_G=F_{N_1}\otimes\ldots\otimes F_{N_k}$.
\end{enumerate}
\end{theorem}

\begin{proof}
This follows indeed from some basic facts from group theory:

\medskip

(1) With the identification $G\simeq\widehat{G}$ made our matrix is given by $(F_G)_{i\chi}=\chi(i)$, and the scalar products between the rows are computed as follows:
\begin{eqnarray*}
<R_i,R_j>
&=&\sum_\chi\chi(i)\overline{\chi(j)}\\
&=&\sum_\chi\chi(i-j)\\
&=&|G|\cdot\delta_{ij}
\end{eqnarray*}

Thus, we obtain indeed a complex Hadamard matrix.

\medskip

(2) This follows from the well-known and elementary fact that, via the identifications $\mathbb Z_N=\widehat{\mathbb Z_N}=\{1,\ldots,N\}$, the Fourier coupling here is as follows, with $w=e^{2\pi i/N}$:
$$(i,j)\to w^{ij}$$

(3) We use here the following well-known formula, for the duals of products: 
$$\widehat{H\times K}=\widehat{H}\times\widehat{K}$$

At the level of the corresponding Fourier couplings, we obtain from this:
$$F_{H\times K}=F_H\otimes F_K$$

Now by decomposing $G$ into cyclic groups, as in the statement, and by using (2) for the cyclic components, we obtain the formula in the statement.
\end{proof}

There are many other examples of Hadamard matrices, with some being fairly exotic, appearing in various branches of mathematics and physics. The idea is that the complex Hadamard matrices can be though of as being ``generalized Fourier matrices'', and this is where the interest in these matrices comes from.

\bigskip

In relation with the quantum groups, the starting observation is as follows:

\index{Hadamard matrix}
\index{magic matrix}

\begin{proposition}
If $H\in M_N(\mathbb C)$ is Hadamard, the rank one projections 
$$P_{ij}=Proj\left(\frac{H_i}{H_j}\right)$$
where $H_1,\ldots,H_N\in\mathbb T^N$ are the rows of $H$, form a magic unitary.
\end{proposition}

\begin{proof}
This is clear, the verification for the rows being as follows:
\begin{eqnarray*}
\left<\frac{H_i}{H_j},\frac{H_i}{H_k}\right>
&=&\sum_l\frac{H_{il}}{H_{jl}}\cdot\frac{H_{kl}}{H_{il}}\\
&=&\sum_l\frac{H_{kl}}{H_{jl}}\\
&=&N\delta_{jk}
\end{eqnarray*}

The verification for the columns is similar, as follows:
\begin{eqnarray*}
\left<\frac{H_i}{H_j},\frac{H_k}{H_j}\right>
&=&\sum_l\frac{H_{il}}{H_{jl}}\cdot\frac{H_{jl}}{H_{kl}}\\
&=&\sum_l\frac{H_{il}}{H_{kl}}\\
&=&N\delta_{ik}
\end{eqnarray*}

Thus, we obtain the result.
\end{proof}

We can proceed now exactly in the same way as we did with the Weyl matrices, namely by constructing a model of $C(S_N^+)$, and performing the Hopf image construction. We are led in this way to the following definition:

\index{Hopf image}

\begin{definition}
To any Hadamard matrix $H\in M_N(\mathbb C)$ we associate the quantum permutation group $G\subset S_N^+$ given by the fact that $C(G)$ is the Hopf image of
$$\pi:C(S_N^+)\to M_N(\mathbb C)\quad,\quad 
u_{ij}\to Proj\left(\frac{H_i}{H_j}\right)$$
where $H_1,\ldots,H_N\in\mathbb T^N$ are the rows of $H$.
\end{definition}

Summarizing, we have a construction $H\to G$, and our claim is that this construction is something really useful, with $G$ encoding the combinatorics of $H$. To be more precise, our claim is that ``$H$ can be thought of as being a kind of Fourier matrix for $G$''.

\bigskip

There are several results supporting this claim, with the main evidence coming from the following result, which collects the basic results regarding the construction $H\to G$:

\index{Fourier matrix}

\begin{theorem}
The construction $H\to G$ has the following properties:
\begin{enumerate}
\item For a Fourier matrix $H=F_G$ we obtain the group $G$ itself, acting on itself.

\item For $H\not\in\{F_G\}$, the quantum group $G$ is not classical, nor a group dual.

\item For a tensor product $H=H'\otimes H''$ we obtain a product, $G=G'\times G''$.
\end{enumerate}
\end{theorem}

\begin{proof}
All this material is standard, and elementary, as follows:

\medskip

(1) Let us first discuss the cyclic group case, where our Hadamard matrix is a usual Fourier matrix, $H=F_N$. Here the rows of $H$ are given by $H_i=\rho^i$, where:
$$\rho=(1,w,w^2,\ldots,w^{N-1})$$

Thus, we have the following formula, for the associated magic basis:
$$\frac{H_i}{H_j}=\rho^{i-j}$$

It follows that the corresponding rank 1 projections $P_{ij}=Proj(H_i/H_j)$ form a circulant matrix, all whose entries commute. Since the entries commute, the corresponding quantum group must satisfy $G\subset S_N$. Now by taking into account the circulant property of $P=(P_{ij})$ as well, we are led to the conclusion that we have:
$$G=\mathbb Z_N$$

In the general case now, where $H=F_G$, with $G$ being an arbitrary finite abelian group, the result can be proved either by extending the above proof, of by decomposing $G=\mathbb Z_{N_1}\times\ldots\times\mathbb Z_{N_k}$ and using (3) below, whose proof is independent from the rest.

\medskip

(2) This is something more tricky, needing some general study of the representations whose Hopf images are commutative, or cocommutative.

\medskip

(3) Assume that we have a tensor product $H=H'\otimes H''$, and let $G,G',G''$ be the associated quantum permutation groups. We have then a diagram as follows:
$$\xymatrix@R=45pt@C25pt{
C(S_{N'}^+)\otimes C(S_{N''}^+)\ar[r]&C(G')\otimes C(G'')\ar[r]&M_{N'}(\mathbb C)\otimes M_{N''}(\mathbb C)\ar[d]\\
C(S_N^+)\ar[u]\ar[r]&C(G)\ar[r]&M_N(\mathbb C)
}$$

Here all the maps are the canonical ones, with those on the left and on the right coming from $N=N'N''$. At the level of standard generators, the diagram is as follows:
$$\xymatrix@R=45pt@C65pt{
u_{ij}'\otimes u_{ab}''\ar[r]&w_{ij}'\otimes w_{ab}''\ar[r]&P_{ij}'\otimes P_{ab}''\ar[d]\\
u_{ia,jb}\ar[u]\ar[r]&w_{ia,jb}\ar[r]&P_{ia,jb}
}$$

Now observe that this diagram commutes. We conclude that the representation associated to $H$ factorizes indeed through $C(G')\otimes C(G'')$, and this gives the result.
\end{proof}

Going beyond the above result is an interesting question, and we refer here to \cite{bb3}, and follow-up papers. There are several computations available here, for the most regarding the deformations of the Fourier models. We believe that the unification of all this with the Weyl matrix models is a very good question, related to many interesting things.

\bigskip

And this is all. In the hope that you liked the present book, and that we will see you soon doing some research on the quantum groups. With things here being however a bit tricky, and here is some advice on this, research matters, to finish with:

\bigskip

(1) Generally speaking, quantum groups have been around since the late 70s, and the work by Faddeev and others \cite{fad}, and so, many things are known about them. The whole area is quite advanced, and if you want to come up with some truly original, interesting new things, you need to know well mathematics and physics. No less than that.

\bigskip

(2) So this would be my advice, learn some mathematics and physics. And be aware that you'll have to do that alone, with your love for mathematics and physics being the only thing that you can rely upon. Of course, some things can be learned from various communities, but community basically means specialization, so wrong way.

\bigskip

(3) Getting to mathematics, besides Rudin \cite{rud} which is the Bible, you can learn all sorts of useful things from Arnold \cite{arn}, Atiyah \cite{ati}, Connes \cite{con}, Drinfeld \cite{dri}, Jones \cite{jo3}, Voiculescu \cite{vdn}, von Neumann \cite{von}, Witten \cite{wit}. These are all people knowing well both mathematics and physics, and reading their writings is certainly a good idea.

\bigskip

(4) As for physics, for some general learning here, rather quantum mechanics oriented, you have Feynman \cite{fe1}, \cite{fe2}, \cite{fe3}, or Griffiths \cite{gr1}, \cite{gr2}, \cite{gr3}, or Weinberg \cite{we1}, \cite{we2}. But, and importantly, if needed complete with some classical mechanics, say from Kibble \cite{kbe}, and some thermodynamics, say from Schroeder \cite{dsc} or Huang \cite{hua}. 

\bigskip

Finally, in what concerns the closed subgroups $G\subset U_N^+$ from this book, as a good continuation, you can read various standard papers on easiness, such as \cite{bb+}, \cite{bbc}, \cite{bc1}, \cite{bc2}, \cite{ez1}, \cite{ez2}, \cite{bsp}, \cite{bv1}, \cite{bv2}, all from the 00s, and also various standard papers on quantum permutations, such as \cite{ba3}, \cite{bb1}, \cite{bb3}, \cite{bbs}, \cite{bne}, \cite{lmr}, \cite{sch}, for the most from the late 00s and early 10s. This is certainly something quite time-consuming, but with this, you can virtually read afterwards anything that you want to, on quantum groups.

\section*{16e. Exercises} 

The matrix modelling problematics from this chapter is something quite exciting, and we have several exercises here. To start with, we have the following question:

\begin{exercise}
Given a real algebraic manifold of the free complex sphere,
$$X\subset S^{N-1}_{\mathbb C,+}$$
and an integer $K\in\mathbb N$, construct a universal $K\times K$ model for $C(X)$,
$$\pi_K:C(X)\to M_K(C(T_K))$$
with $T_K$ being the space of all $K\times K$ models for $C(X)$.
\end{exercise}

This is something quite theoretical, the problem being that of proving that the universal model space $T_K$ in the above is indeed compact.

\begin{exercise}
Given $X\subset S^{N-1}_{\mathbb C,+}$ and $K\in\mathbb N$ as above, consider the submanifold $X^{(K)}\subset X$ obtained by factorizing the universal $K\times K$ model:
$$\pi_K:C(X)\to C(X^{(K)})\subset M_K(C(T_K))$$
Prove that at $K=1$ we obtain in this way the classical version of $X$,
$$X^{(1)}=X_{class}$$
and that at $K\geq2$, assuming that $X$ is a compact quantum group, $X=G\subset U_n^+$ with $N=n^2$, the space $X^{(K)}$ is not necessarily a compact quantum group.
\end{exercise}

Here the first question is something more or less trivial, and so the exercise is about finding counterexamples at $K\geq2$, in the quantum group case.

\begin{exercise}
Work out the details for the fact that the stationarity of a model
$$\pi:C(G)\to M_K(C(T))$$
implies its faithfulness.
\end{exercise}

This is something that we already discussed in the above, but with some standard functional analysis details missing. The problem is that of working out these details.

\begin{exercise}
Find an example of an inner faithful model 
$$\pi:C(G)\to M_K(C(T))$$
which is not faithful, not coming from a classical group, or a group dual.
\end{exercise}

This is something quite tricky, and it is of course possible to cheat a bit here, by using product operations. The exercise asks for a high-quality counterexample.

\begin{exercise}
Extract from the general theory developed above a concise proof for the fact that the Pauli matrix model
$$\pi:C(S_4^+)\subset M_4(C(SU_2))$$
$$\pi(u_{ij})=[x\to Proj(c_ixc_j)]$$
where $x\in SU_2$, and $c_1,c_2,c_3,c_4$ are the Pauli matrices, is faithful.
\end{exercise}

This is something that we discussed above, but the problem now is that of doing the thing, and writing down a concise, self-contained proof for the faithfulness of $\pi$.

\printindex

\end{document}